\documentclass[10pt]{book}

\usepackage[left=1.2in,top=1.0in,right=1.2in,bottom=1.0in]{geometry}

\title{\textbf{A Berkovich Approach to Perfectoid Spaces}}
\author{Attilio Castano}
\date{}

\usepackage[whole]{bxcjkjatype}
\usepackage{amsfonts, amsmath, amssymb, amsthm, enumerate,stmaryrd}
\usepackage{hyperref}	%allows urls
\usepackage{comment}
\hypersetup{colorlinks=true, linkcolor=teal, citecolor=teal}

\usepackage{tikz-cd}

\theoremstyle{definition}
\newtheorem{thm}{Theorem}[section]
\newtheorem{lemma}[thm]{Lemma}
\newtheorem{prop}[thm]{Proposition}
\newtheorem{corollary}[thm]{Corollary}

\newtheorem{defn}[thm]{Definition}
\newtheorem{example}[thm]{Example}
\newtheorem{rem}[thm]{Remark}
\newtheorem{notation}[thm]{Notation}
\newtheorem{const}[thm]{Construction}
\newtheorem{warning}[thm]{Warning}
\newtheorem{convention}[thm]{Convention}

\newtheorem{thmx}{Theorem}

\newenvironment{cd}[0]{\begin{equation*}\begin{tikzcd}}{\end{tikzcd}\end{equation*}}

\newcommand{\Spec}{\operatorname{Spec}}
\newcommand{\Spc}{\operatorname{Spc}}

\newcommand{\Cone}{\operatorname{Cone}}
\newcommand{\perf}{\operatorname{perf}}
\newcommand{\perfd}{\operatorname{perfd}}
\newcommand{\Spf}{\operatorname{Spf}}
\newcommand{\colim}{\operatorname{colim}}

\newcommand{\arc}{\operatorname{arc}}
\newcommand{\cotimes}{\widehat{\otimes}}
\newcommand{\rat}{\operatorname{rat}}
\newcommand{\op}{\operatorname{op}}
\newcommand{\im}{\operatorname{Im}}
\newcommand{\CAlg}{\operatorname{CAlg}}
\newcommand{\Kos}{\operatorname{Kos}}
\newcommand{\Perfd}{\operatorname{Perfd}}
\newcommand{\Perf}{\operatorname{Perf}}
\newcommand{\Shv}{\operatorname{Shv}}
\newcommand{\PreShv}{\operatorname{PreShv}}
\newcommand{\Set}{\operatorname{Set}}
\newcommand{\Hom}{\operatorname{Hom}}
\newcommand{\Maps}{\operatorname{Maps}}

\newcommand{\Comp}{\operatorname{Comp}}
\newcommand{\Mod}{\operatorname{Mod}}
\newcommand{\Cond}{\operatorname{Cond}}

\newcommand{\Ban}{\operatorname{Ban}}
\newcommand{\uBan}{\operatorname{uBan}}
\newcommand{\pt}{\operatorname{pt}}
\newcommand{\coeq}{\operatorname{coeq}}
\newcommand{\eq}{\operatorname{eq}}
\newcommand{\Aff}{\operatorname{Aff}}
\newcommand{\eff}{\operatorname{eff}}
\newcommand{\Top}{\operatorname{Top}}
\newcommand{\Frac}{\operatorname{Frac}}
\newcommand{\PerfdSpc}{\operatorname{Perfd-Spc}}
\newcommand{\AnSpc}{\operatorname{Analytic-Spc}}

\newcommand{\sep}{\operatorname{sep}}

\newcommand{\Int}{\operatorname{Int}}
\newcommand{\tic}{\operatorname{tic}}
\newcommand{\picomp}{{}_{\varpi}^{\wedge}}

\newcommand{\Id}{\operatorname{Id}}
\newcommand{\Fil}{\operatorname{Fil}}
\newcommand{\tf}{\operatorname{tf}}
\newcommand{\contr}{\operatorname{contr}}
\newcommand{\pic}{\operatorname{pic}}
\newcommand{\ic}{\operatorname{ic}}
\newcommand{\init}{\operatorname{init}}
\newcommand{\Sch}{\operatorname{Sch}}
\newcommand{\qcqs}{\operatorname{qcqs}}
\newcommand{\Tot}{\operatorname{Tot}}
\newcommand{\cosk}{\operatorname{cosk}}
\newcommand{\Val}{\operatorname{Val}}
\newcommand{\ProFin}{\operatorname{ProFin}}
\newcommand{\Sub}{\operatorname{Sub}}
\newcommand{\Zar}{\operatorname{Zar}}
\newcommand{\cg}{\operatorname{cg}}
\newcommand{\cgwh}{\operatorname{cgwh}}
\newcommand{\Open}{\operatorname{Open}}

\newcommand{\Yo}{\operatorname{\text{よ}}}

\newcommand{\blang}{\Big \langle}
\newcommand{\brang}{\Big \rangle}

\newcommand{\cC}{\mathcal{C}}
\newcommand{\cB}{\mathcal{B}}
\newcommand{\cD}{\mathcal{D}}
\newcommand{\cI}{\mathcal{I}}
\newcommand{\cO}{\mathcal{O}}
\newcommand{\cH}{\mathcal{H}}
\newcommand{\cU}{\mathcal{U}}
\newcommand{\cT}{\mathcal{T}}

\newcommand{\cX}{\mathcal{X}}
\newcommand{\cW}{\mathcal{W}}
\newcommand{\frakp}{\mathfrak{p}}
\newcommand{\frakm}{\mathfrak{m}}

\newcommand{\CC}{\mathbf{C}}
\newcommand{\A}{\mathbf{A}}
\newcommand{\FF}{\mathbf{F}}
\newcommand{\ZZ}{\mathbf{Z}}
\newcommand{\RR}{\mathbf{R}}

\newcommand{\cM}{\mathcal{M}}
\newcommand{\uHom}{\underline{\Hom}}

\usepackage{relsize} 
\usepackage[bbgreekl]{mathbbol} 
\DeclareSymbolFontAlphabet{\mathbb}{AMSb} %to ensure that the meaning of \mathbb does not change
\DeclareSymbolFontAlphabet{\mathbbl}{bbold} 
\newcommand{\Prism}{\mathbbl{\Delta}}

\begin{document}

\maketitle
\hbadness=99999
\tolerance=9999

\chapter*{Abstract}
\addcontentsline{toc}{chapter}{Abstract}
Since their inception perfectoid spaces have catalyzed a revolution in $p$-adic geometry. We redevelop the foundations of perfectoid spaces from the point of view of Berkovich Spaces, where the underlying topological space of an affinoid perfectoid space is a compact Hausdorff space -- closely resembling the situation in complex geometry. The key technical ingredient in our construction is arc$_\varpi$-descent for perfectoid Banach algebras. Along the way, we establish various foundational results for arc$_\varpi$-sheaves, notably a form of the Gerritzen-Grauert theorem.

\setcounter{tocdepth}{2}

\tableofcontents

\newpage

\chapter{Introduction}

\section{Affinoid Perfectoid Spaces}

The work of Ostrowski on the classification of non-trivial valuations on $\mathbf{Q}$ -- up to equivalence there is one valuation for each prime number $p$, the $p$-adic valuation, and the archimedean valuation given by the usual absolute value -- illustrate an old analogy in mathematics that the real numbers $\mathbf{R}$ (or the complex numbers $\mathbf{C}$) should be thought as the prime at infinity, given that the completion of $\mathbf{Q}$ with respect to the archimedean valuation is the real numbers $\mathbf{R}$. Thus, it is natural to ask whether there is a theory of $p$-adic geometry, parallel to the rich theory of complex geometry. This vision was first realized by Tate in the early 1960's, and then systematically developed by a number of other mathematicians, with the resulting $p$-adic geometric objects being dubbed rigid analytic spaces by Tate -- with some variants due mainly to Raynaud, Berkovich and Huber.

During the last decade there has been a revolution in $p$-adic geometry catalyzed by the introduction of Perfectoid Spaces by Scholze in his thesis, with one of its first achievements being an extension of the main theorems of $p$-adic Hodge theory to (proper) rigid-analytic varieties, where perfectoid techniques allowed for proofs that closely resemble those of the main theorems of complex Hodge theory. For his work on perfectoid spaces Scholze used Huber's foundational work on rigid-analytic geometry to build upon, it had the advantage that work in more general situations -- needed in the context of perfectoid spaces -- than Tate's or Berkovich's work, but at the price of replacing the more traditional category of non-archimedean Banach algebras with the more subtle category of Huber pairs.

Before diving into the specifics of our work, let us introduce some preliminary definitions. For the rest of this document fix a prime number $p$. A \emph{perfectoid field} $K$ is a non-archimedean field for which exists a $\varpi \in K$ which satisfies $1 > |\varpi^p| \ge |p|$, and such that the Frobenius morphism $K_{\le 1}/\varpi^p \rightarrow K_{\le 1}/\varpi^p$ is surjective -- in what follows $K$ will always denote a perfectoid field. More generally, a \emph{perfectoid Banach $K$-algebra} is a non-archimedean uniform Banach $K$-algebra $R$ such that the Frobenius morphism $\varphi: R_{\le 1}/\varpi^p \rightarrow R_{\le 1}/\varpi^p$ is surjective; we will denote the full-subcategory of $\Ban_K$\footnote{All Banach $K$-algebras considered in this text will be non-archimedean, and so we will often drop the ``non-archimedean'' adjective and just call them Banach $K$-algebras. This is reflected in our choice of notation where we choose to denote the category of non-archimedean Banach $K$-algebras by $\Ban_K$.} spanned by the perfectoid Banach $K$-algebras by $\Perfd_K^{\Ban}$\footnote{There is a subtle distinction between our definition of perfectoid Banach $K$-algebras and Scholze's original definition \cite[Definition 5.1.]{scholze2012perfectoid}; we insist that the norm on $R$ to be power-multiplicative, while Scholze only requires it to be power-multiplicative up to equivalence. This has some minor practical consequences for us, as we often work with the non-full subcategory $\Ban_K^{\contr} \subset \Ban_K$ of non-archimedean Banach $K$-algebras and contractive morphisms, where different norms give rise to non-isomorphic objects.}. We will often call an object of $\Perfd_{K}^{\Ban, \op}$ an \emph{affinoid perfectoid space}, and denote the affinoid perfectoid space corresponding to the perfectoid Banach $K$-algebra $A$ by $\cM(A)$; and for an non-archimedean Banach $K$-algebra $A$ we will denote the corresponding object of $\Ban_K^{\op}$ by $\cM(A)$.

Motivated by the work of Tate on rigid analytic geometry and the geometric intuition from complex geometry Berkovich introduced what is now called the \emph{Berkovich spectrum} of a Banach $K$-algebra in \cite{berkovich_spectral}. One critical feature of Berkovich's construction is that -- unlike Tate's or Huber's approach to non-archimedean geometry --  the spectrum of a Banach $K$-algebra is a compact Hausdorff space which allows for direct application of the geometric intuition that is so valuable over the complex numbers. We refer the reader to Definition \ref{defn_berko_spectrum} for a formal definition, but it suffices to say that Berkovich's construction produces a functor
\begin{align*}
	|-|: \Ban_K^{\op} \longrightarrow \Comp && \cM(A) \mapsto |\cM(A)|
\end{align*}
which we will often call the \emph{Berkovich functor}, and where $\Comp$ is the category of compact Hausdorff spaces.

We use recent advances of the perfectoid theory -- notably the development of the $\arc_\varpi$-topology -- to rebuild the theory of perfectoid spaces from the ground up, using Berkovich's work as our foundations. The $\arc_\varpi$-topology was introduced by Bhatt and Mathew in \cite{arc_topology}, it is a Grothendieck topology in the category of qcqs $K_{\le 1}$-schemes where covers are tested by valuation rings of rank $1$ where $\varpi \not= 0$. Following Raynaud's approach to non-archimedean geometry, where Banach $K$-algebra $R$ are studied via their formal model $R_{\le 1}$\footnote{The notation $R_{\le 1}$ means the objects of $R$ which have norm $\le 1$, and we regard $R_{\le 1}$ as as a $\varpi$-complete $K_{\le 1}$-algebra. Throughout much of this work we restrict our attention to the subcategory $\Ban_K^{\contr} \subset \Ban_K$ of Banach $K$-algebras with contractive morphisms, making the construction $(-)_{\le 1}$ functorial.}, we translate the purely algebraic definition of the $\arc_\varpi$-topology of Bhatt and Mathew to a definition that can be entirely formulated in terms of the Berkovich spectrum.

\begin{defn}[$\arc_\varpi$-covers] Let $\{A \rightarrow B_i\}_{i \in I}$ be a finite collection of contractive morphisms of non-archimedean Banach $K$-algebras. We say that $\{\cM(B_i) \rightarrow \cM(A)\}_{i \in I}$ is an \emph{$\arc_\varpi$-cover} if the induced map of compact Hausdorff spaces $\sqcup_{i \in I} |\cM(B_i)| \rightarrow |\cM(A)|$ is surjective -- this determines a (finitary) Grothendieck topology on the category $\Ban_K^{\contr, \op}$.
\end{defn}
 
An important feature of the $\arc_\varpi$-topology is that it reflects the topology of $|\cM(A)|$ much more than than the topology generated by affinoid domains, used by Tate and Berkovich. To illustrate this point recall the following fact from point-set topology: if $\{X_i \rightarrow Y\}_{i \in I}$ is a finite collection of maps of compact Hausdorff spaces such that $\sqcup X_i \rightarrow Y$ is surjective, then the following canonical map is an isomorphism
\begin{align*}
	\coeq \Big( (\sqcup_{i \in I} X_i) \times_Y (\sqcup_{i \in I} X_i) \rightrightarrows (\sqcup_{i \in I} X_i) \Big) \overset{\simeq}{\longrightarrow} Y
\end{align*}
We call the Grothendieck topology on $\Comp$ generated by finite collections of maps $\{X_i \rightarrow Y\}_{i \in I}$ which induce a surjective map $\sqcup X_i \rightarrow Y$ the \emph{effective topology} (often written in symbols as $\eff$) on $\Comp$, and the above statement can be interpreted as saying that compact Hausdorff spaces are sheaves with respect to the $\eff$-topology. While its not true that all Banach $K$-algebras are sheaves with respect to the $\arc_\varpi$-topology, it becomes true if we restrict ourselves to the full subcategory of perfectoid Banach $K$-algebras. In fact, in many ways the results of this work can be summarized as a way of justifying the following sentence
\begin{center}
	\emph{``Via the Berkovich spectrum, affinoid perfectoid spaces behave as if they were compact Hausdorff spaces''}
\end{center}
Or in other words, we show to what extent an affinoid perfectoid space is determined by its underlying compact Hausdorff space. Let us illustrate this point with some of the results from this work.

\begin{thmx}[Tate acyclicity - Corollary \ref{arc_pi_descent_for_bananach_perfectoids}]\label{intro_tate_acyclicity_perfd} Let $\cM(R)$ be a perfectoid affinoid space, and $\{ \cM(S_i) \rightarrow \cM(R) \}_{i \in I}$ a finite collection of morphisms which determine an $\arc_\varpi$-cover. Then, the canonical map
\begin{equation*}
	\coeq \Big ( \sqcup_{i,j \in I} \cM(S_i) \times_{\cM(R)} \cM(S_j) \rightrightarrows \sqcup_{i \in I} \cM(S_i)    \Big ) \overset{\simeq}{\longrightarrow} \cM(R)
\end{equation*}
is an isomorphism\footnote{Furthermore, the Berkovich functor $|-|: \Perfd_K^{\Ban, \op} \rightarrow \Comp$ will preserve the coequalizer (cf. Propositions \ref{general_coeq_comp} and \ref{general_fiber_prod_berko_sp}).}, where the coequalizer is computed in $\Ban_K^{\contr, \op}$. This is in fact a small consequence of a much more fundamental result, which says that the following sequence
\begin{equation*}
	0 \rightarrow R \rightarrow \prod_{i \in I} S_i \rightarrow \prod_{i,j \in I} S_i \cotimes_R S_j \rightarrow \cdots 
\end{equation*}
is exact and admissible\footnote{In fact, the above result admits an analog after applying the functor $(-)_{\le 1}$, which says that the following sequence
\begin{align*}
	0 \rightarrow R_{\le 1} \rightarrow \prod_{i \in I} S_{i, \le 1} \rightarrow \prod_{i,j \in I} S_{i, \le 1} \cotimes_{R_{\le 1}}^a S_{j, \le 1} \rightarrow \cdots 
\end{align*}
is almost exact, in the almost category of $K_{\le 1}$-modules.} \footnote{In Theorem \ref{struct_presheaf_perfectoid} we show that every affinoid perfectoid space $\cM(A)$ admits a structure presheaf with the expected properties, and the above discussion shows that it is in fact a structure sheaf. See also Theorem \ref{stalks_perfectoid} for discussion on the basic properties of stalks on affinoid perfectoid spaces.}.
\end{thmx}

This result has an important antecedent in \cite[Proposition 8.8]{diamonds_scholze}, where an analogous result is proved by replacing the Berkovich spectrum, by the adic space spectrum of Huber; since the Berkovich spectrum lies inside the adic spectrum our results can be viewed as a slight strengthening of Scholze's $v$-descent for perfectoid affinoid algebras -- from $v$-descent to $\arc_\varpi$-descent. Let us also contrast our result with the original form of Tate acyclicity where $R$ and $S_i$ are assumed to be non-archimedean Banach $K$-algebras which are topologically of finite type, and where the morphisms $\cM(S_i) \rightarrow \cM(R)$ have the form of an affinoid domain.

\begin{example} Let $\cM(R)$ be a perfectoid affinoid space, and consider the Gelfand transform $\prod_{x \in \cM(R)} \cH(x)$ of $R$ -- the most relevant property of the Gelfand transform for this example is that the underlying topological space of $\cM(\prod_{x \in \cM(R)} \cH(x))$ can be canonically identified with the Stone-Cech compactification of $\cM(R)$ considered as a discrete set. Then, the following sequence is exact and admissible
\begin{equation*}
	0 \rightarrow R \rightarrow \prod_{x \in \cM(R)} \cH(x) \rightarrow \Big (\prod_{x \in \cM(R)} \cH(x) \Big ) \cotimes_R  \Big ( \prod_{x \in \cM(R)} \cH(x) \Big )  \rightarrow \cdots 
\end{equation*}
In particular, this implies that the canonically induced map is an isomorphism\footnote{This result admits the following topological analog: if $X$ is a compact Hausdorff space, and $\beta(X^{\delta})$ is the Stone-Cech compactification of $X$ as a discrete set, then we have the following isomorphism $\coeq( \beta(X^{\delta}) \times_X \beta(X^{\delta}) \rightrightarrows \beta(X^{\delta})) \simeq X$.}
\begin{align*}
	\coeq \Big ( \cM(\prod_{x \in \cM(R)} \cH(x) ) \times_{\cM(R)} \cM(\prod_{x \in \cM(R)} \cH(x))  \rightrightarrows \cM(\prod_{x \in \cM(R)} \cH(x) )   \Big) \overset{\simeq}{\longrightarrow} \cM(R).
\end{align*}
\end{example}

At the crux of the definition of rigid analytic spaces is the notion of affinoid domains, which play a role analogous to the one that open subsets play in topology. If $A \rightarrow B$ is a bounded morphism of non-archimedean Banach $K$-algebras of topologically finite type, then we say that $\cM(B) \rightarrow \cM(A)$ is an affinoid domain if the map $\cM(B) \rightarrow \cM(A)$ is universal in the sense that for any other map $\cM(C) \rightarrow \cM(A)$ such that $\im(\cM(C) \rightarrow \cM(A))$ is contained in $\im(\cM(B) \rightarrow \cM(A))$, there exists a unique morphism $\cM(C) \rightarrow \cM(B)$ making the following diagram commute
\begin{cd}
	& \cM(C) \ar[d] \ar[ld, dashed] \\
	\cM(B) \ar[r] & \cM(A)
\end{cd}
The Gerritzen-Grauert theorem \cite[7.3.5]{BGR} is one of the most important foundational results in rigid analytic geometry, in particular it implies that  affinoid domains are finite unions of rational domains, allowing the proof of Tate's acyclicity for finite covers by affinoid domains by bootstrapping the same result from rational domains. In contrast, the analogous result in the category of compact Hausdorff spaces admits a much cleaner answer, if $X \rightarrow Y$ is an injective morphism of compact Hausdorff spaces then $X$ is endowed with the subspace topology of $Y$ and so any map $Z \rightarrow Y$ whose image is contained in $X$ admits a unique factorization as $Z \rightarrow X \rightarrow Y$. In the category of affinoid perfectoid spaces we have the following result, which parallels \cite[Proposition 5.3]{diamonds_scholze}.

\begin{thmx}[Affinoid Domains - Theorem \ref{perfd_spc_mono}]\label{intro_mono_perfd} Let $\cM(S) \rightarrow \cM(R)$ be a morphism of affinoid perfectoid spaces, then the following are equivalent
\begin{enumerate}[(1)]
	\item The morphism $\cM(S) \rightarrow \cM(R)$ is a monomorphism in the category of affinoid perfectoid spaces.
	\item The induced map of underlying sets $\cM(S) \rightarrow \cM(R)$ is injective, and for each $x \in \cM(R)$ with inverse image $y \in \cM(S)$ the induced map of completed residue fields $\cH(x) \rightarrow \cH(y)$ is an isomorphism.
	\item The map $\cM(S) \rightarrow \cM(R)$ is an affinoid domain: for any morphism $\cM(T) \rightarrow \cM(R)$ such that $\im(\cM(T) \rightarrow \cM(R)) \subset \im(\cM(S) \rightarrow \cM(R))$, there exists a unique morphism $\cM(T) \rightarrow \cM(S)$ of affinoid perfectoid spaces, making the following diagram commute
	\begin{cd}
		& \cM(T) \ar[d] \ar[ld, dashed] \\
		\cM(S) \ar[r] & \cM(R)
	\end{cd}
\end{enumerate}
\end{thmx}
The above result really is a special feature of affinoid perfectoid spaces, and does not hold for more general non-archimedean Banach $K$-algebras. Indeed, we can consider the canonical surjection $K\langle T \rangle \rightarrow K \langle T \rangle/ (T^2)$ and then the induced map $\cM(K \langle T \rangle/ (T^2)) \rightarrow \cM(K\langle T \rangle)$ is a monomorphism in $\Ban_K^{\op}$; however, it is not an affinoid domain as the canonical inclusion $\cM(K) \rightarrow \cM(K\langle T \rangle)$ at the origin will not factor through $\cM(K \langle T \rangle/ (T^2)) \rightarrow \cM(K\langle T \rangle)$. The proof of Theorem \ref{intro_mono_perfd} relies on Theorem \ref{intro_iso_perfd}, which was inspired from the fact that a map $X \rightarrow Y$ of compact Hausdorff spaces is an homeomorphism if and only if the underlying map of sets is a bijection. This also explains the failure of Theorem \ref{intro_mono_perfd} for general Banach $K$-algebras, as maps from non-archimedean fields cannot distinguish between a Banach $K$-algebra $A$ and its uniformization $A^u$.

\begin{thmx}[Isomorphisms - Theorem \ref{perfd_spc_maps}]\label{intro_iso_perfd} Let $\cM(S) \rightarrow \cM(R)$ be a morphism of affinoid perfectoid spaces, then the following are equivalent
\begin{enumerate}[(1)]
	\item The map $\cM(S) \rightarrow \cM(R)$ is an isomorphism of affinoid perfectoid spaces.
	\item The induced map of underlying sets $\cM(S) \rightarrow \cM(R)$ is bijective, and for each $x \in \cM(R)$ with inverse image $y \in \cM(S)$ the induced map of completed residue fields $\cH(x) \rightarrow \cH(y)$ is an isomorphism.
	\item For all perfectoid non-archimedean field $L$, the induced map $\Maps(\cM(L), \cM(S)) \rightarrow \Maps(\cM(L), \cM(R))$ is bijective.
\end{enumerate}
\end{thmx}

\begin{example} The following are all monomorphisms in the category of affinoid perfectoid spaces.
\begin{enumerate}[(1)]
	\item (Residue Fields) For every point $x \in \cM(R)$ of an affinoid perfectoid space, the canonical map $\cM(\cH(x)) \rightarrow \cM(R)$ is a monomorphism of perfectoid affinoid spaces\footnote{For general Banach $K$-algebras it is not generally true that the map $\cM(\cH(x)) \rightarrow \cM(R)$ is a monomorphism in $\Ban_K^{\contr, \op}$, but its close to being so (cf. Proposition \ref{properties_stalk}). Furthermore, in general $\cM(\cH(x)) \hookrightarrow \cM(R)$ fails to be an affinoid domain, for example due to the presence of nilpotents.}.
	\item (Rational Domains) Let $X = \cM(R)$ and $V = \{x \in X \text{ such that } |f_i(x)| \le |g(x)|\}$ for a collection of elements $\{g, f_1, \dots, f_n\} \subset R$ which generate the unit ideal. There exists a unique perfectoid algebra $\cO_X(V)$ together with a monomorphism $\cM(\cO_X(V)) \rightarrow X$, such that at the level of underlying sets its image is given exactly by $V \subset X$\footnote{In Proposition \ref{rat_domains_are_mono} we show that rational domains are monomorphisms for general Banach $K$-algebras and in Proposition \ref{rational_dom_univ_prop} that they satisfy the universal property of affinoid domains with respect to uniform Banach $K$-algebras.}.
	\item (Zariski Closed Subsets) Let $R$ is an perfectoid Banach $K$-algebra and $I$ and ideal of $R$. Even though the Banach $K$-algebra $R/I$ is not perfectoid, we know from \cite[Remark 7.5]{prisms} that there exists a initial perfectoid Banach $K$-algebra $S$ receiving a map from $R/I$, such that the induced map $R \rightarrow S$ is surjective. Thus, we learn that the map $\cM(S) \rightarrow \cM(R)$ is a monomorphism of affinoid perfectoid spaces, and at the level of underlying sets has the same image as $\cM(R/I) \rightarrow \cM(R)$\footnote{If $R$ is a Banach $K$-algebra and $I \subset R$ a closed ideal, the induced map $\cM(R/I) \rightarrow \cM(R)$ is a monomorphism in $\Ban_K^{\contr, \op}$, but fails to be an affinoid domain, for example due to the presence of nilpotents.}.
\end{enumerate}
\end{example}

Let us conclude this section with a word on what is involved in the proofs of this statements. The proof of the form of Tate's acyclicity for affinoid perfectoid spaces stated above, ultimately relies on \cite[Proposition 8.10]{prisms} which show that integral perfectoid algebras satisfy $\arc$-descent. However, this result is not directly applicable to our situation as integral perfectoid algebras are never perfectoid Banach $K$-algebras, thus there is some translation needed in order to use loc. cit. to prove statements about perfectoid Banach $K$-algebras. In order to achieve this goal, in Theorem \ref{intro_dictionary} we establish an equivalence of categories $(-)_{\le 1}: \Ban_K^{\contr} \simeq \CAlg_{K_{\le 1}}^{\wedge a \tf}: (-)[\frac{1}{\varpi}]$, which provides a way to translate statements form integral perfectoid algebras to statements about perfectoid Banach algebras\footnote{While the need to restrict to $\varpi$-complete $\varpi$-torsion free $K_{\le 1}$-algebras is clear, the need for almost mathematics may not be so transparent. The following heuristic attempts at explaining the need for almost mathematics: if $R$ is a Banach $K$-algebra it is clear that if $x \in R$ satisfies $|x| \le 1 + \varepsilon$ for all $\varepsilon > 0$ then $|x| \le 1$; on the other hand if $R \subset R[\frac{1}{\varpi}]$ is a $\varpi$-torsion free $K_{\le 1}$-algebra and $x \in R[\frac{1}{\varpi}]$ it is not generally true that if $\varpi^{1/p^n} x \in R$ for all $n \in \ZZ_{\ge 0}$ that $x \in R$ -- almost mathematics fixes this issue.}. We regard the following dictionary as a categorical version of \cite[2.3.1]{andrelemme}.

\begin{thmx}[Dictionary - Proposition \ref{equiv_almost_tf_banach}]\label{intro_dictionary} Let $K$ be a perfectoid non-archimedean field and $\varpi \in K$ a topological nilpotent unit admitting a compatible system of $p$-power roots. We denote by $\Ban_K^{\contr}$ the category of non-archimedean Banach $K$-algebras with contractive morphisms, and $\CAlg_{K_{\le 1}}^{\wedge a \tf}$ the category of $\varpi$-complete $\varpi$-torsion-free almost $K_{\le 1}$-algebras, where almost mathematics is perform with respect to the ideal $(\varpi^{1/p^\infty}) \subset K_{\le 1}$. Then, there is an equivalence of categories
\begin{align*}
	(-)_{\le 1}: \Ban_K^{\contr} \simeq \CAlg_{K_{\le 1}}^{\wedge a \tf}: (-)[\frac{1}{\varpi}].
\end{align*}
Furthermore, this equivalence induces equivalences between the following categories:
\begin{enumerate}[(1)]
	\item The category $\uBan_K \subset \Ban_K^{\contr}$ of uniform Banach $K$-algebras, and the category $\CAlg_{K_{\le 1}}^{\wedge \tic} \subset \CAlg_{K_{\le 1}}^{\wedge a \tf}$ of $\varpi$-complete $\varpi$-torsion free algebras $A$ which are also totally integrally closed with respect to $A \subset A[\frac{1}{\varpi}]$ (cf. Proposition \ref{uBan_tic_dict}).
	\item The category $\Perfd_K^{\Ban} \subset \Ban_K^{\contr}$ of perfectoid Banach $K$-algebras, and the category $\Perfd_{K_{\le 1}}^{\Prism a} \subset \CAlg_{K_{\le 1}}^{\wedge a \tf}$ of almost integral perfectoid $K_{\le 1}$-algebras (cf. Propositions \ref{equiv_perfd_tic_almost} and \ref{equiv_perfd_ban_tic}).
	\item The category of non-archimedean fields (resp. perfectoid non-archimedean fields) over $K$, and the category of $\varpi$-complete rank one (resp. integral perfectoid) valuation rings $V$ with faithfully flat structure map $K_{\le 1} \rightarrow V$ (cf. Proposition \ref{na_field_rk_one_dict}).
	\item Let $R \rightarrow S$ be a contractive morphism of Banach $K$-algebras. Then, $\cM(S) \rightarrow \cM(R)$ is surjective if and only if the map $R_{\le 1} \rightarrow S_{\le 1}$ is an $\arc_\varpi$-cover in the sense of \cite[Definition 6.14]{arc_topology} (cf. Proposition \ref{arc_pi_cover_for_banach}).
\end{enumerate}
\end{thmx}

In order to effectively work with the category $\CAlg_{K_{\le 1}}^{\wedge a \tf}$ in Section \ref{sect_int_closure} we show that the canonical inclusion $\CAlg_{K_{\le 1}}^{\wedge a \tf} \subset \CAlg_{K_{\le 1}}$ admits a left adjoint
\begin{align*}
	(-)^{\wedge a \tf}: \CAlg_{K_{\le 1}} \rightarrow \CAlg_{K_{\le 1}}^{\wedge a \tf}
\end{align*}
which can be described as: first passing to the $\varpi$-torsion free quotient, followed by $\varpi$-completing and finally passing to the almost $\varpi$-category\footnote{When restricted to the category of integral perfectoid algebras $\Perfd_{K_{\le 1}}^{\Prism} \subset \CAlg_{K_{\le 1}}$ the functor $(-)^{\wedge a \tf}$ identifies with the much simpler functor $(-)^a$.}. As a result, we learn that for a pair of contractive morphisms $A \leftarrow C \rightarrow B$ of Banach $K$-algebras, the completed tensor product admits the following identity $(A \cotimes_C B)_{\le 1} \simeq (A_{\le 1} \otimes_{C_{\le 1}} B_{\le 1})^{\wedge a \tf}$.

On the other hand, the proof of Theorem \ref{intro_mono_perfd} rely on some basic topos theory performed in the category of $\arc_\varpi$-sheaves which we will discuss in the next section.

\section{The Berkovich Functor}

In developing a global theory of perfectoid spaces we follow Grothendieck's functor of points approach to algebraic geometry and regard the functor the geometric object represents as fundamental. This is not to say that the traditional geometric ingredients are lost, like the Zariski spectrum of a ring, but rather they take an auxiliary role as they can be extracted from the functor the geometric object represents. This approach was also taken up to some extent by Tate in the original definitions of rigid analytic spaces. Before diving into the more sophisticated approach to analytic geometry taken in this paper, let us do a whirlwind tour of the functor of points approach to the definition of the category of schemes, we refer the reader to \cite[Chapter VI]{eisenbud_harris_geometry} for more on this perspective. As we see it, there are two basic ingredients needed to get the theory off the ground, first it is the Zariski spectrum of a ring, which provides us with an underlying topological space associated to our geometric object in a functorial way
\begin{align*}
	|\Spec(-)|: \CAlg^{\op} \rightarrow \Top && \Spec(R) \mapsto |\Spec(R)|
\end{align*}
The second main ingredient is the ability to work locally in the Zariski spectrum $\Spec(R)$ to answer questions about $R$ itself; for example if we have a finite collection of morphisms $\{\Spec(R[\frac{1}{f_i}]) \rightarrow \Spec(R)\}_{i \in I}$ -- which are called Zariski open sets and form a basis for the topology of $\Spec(R)$ -- which induce a surjective map at the level of underlying topological spaces $\sqcup_{i \in I} |\Spec(R[\frac{1}{f_i}])| \rightarrow |\Spec(R)|$, then the following sequence is exact\footnote{Analogous to Theorem \ref{intro_tate_acyclicity_perfd}.}
\begin{align*}
	0 \rightarrow R \rightarrow \prod_{i \in I} R[\frac{1}{f_i}] \rightarrow \prod_{i,j \in I} R[\frac{1}{f_i}] \otimes_R R[\frac{1}{f_j}] \rightarrow \cdots.
\end{align*}
Informally, this is saying that one can recover $R$ from $\{R[\frac{1}{f_i}] \}_{i \in I}$ together with some gluing instructions along Zariski open subsets; furthermore, this is compatible with the Zariski spectrum in the sense that the topological space $|\Spec(R)|$ can be obtained by gluing $\{|\Spec(R[\frac{1}{f_i}]) |\}_{i \in I}$ along the intersections
\begin{align*}
	\Big\{ \Big|\Spec \Big(R \Big[\frac{1}{f_i}\Big] \otimes_R R \Big[\frac{1}{f_j}  \Big] \Big) \Big| = \Big|\Spec \Big(R \Big[\frac{1}{f_i} \Big] \Big)\Big| \cap \Big|\Spec \Big(R \Big[\frac{1}{f_j} \Big] \Big)\Big|  \Big\}_{i,j \in I}.
\end{align*}
Intuitively, the category of schemes is then the category of geometric objects which are obtained by gluing a collection $\{\Spec(S_j)\}$ along Zariski open subsets. In order to make this definition precise, the language of sheaves on a site (and thus the language of topoi) provide a powerful framework to perform local-to-global constructions; indeed, the category of schemes can be realized as the full subcategory of the category of Set-valued Zariski sheaves $\Shv_{\Zar}(\CAlg^{\op})$ spanned by objects $X$ satisfying the following conditions: there is a collection of open subfunctors $\{\Spec(S_j) \hookrightarrow X\}_{j \in I}$ such that the induced map $\sqcup_{j \in J} \Spec(S_j) \rightarrow X$ is a surjective map of Zariski sheaves.

Our approach to defining a global theory of perfectoid spaces is formally quiet similar to the discussion in the previous paragraph, we will first define a category of sheaves and then we will isolate the category of perfectoid spaces as a full-subcategory satisfying certain properties. The analog of the category $\Shv_{\Zar}(\CAlg^{\op})$ will be the $\arc_\varpi$-topos $\cX_{\arc_\varpi}$ which is defined as
\begin{align*}
	\cX_{\arc_\varpi} := \Shv_{\arc_\varpi}(\Ban_K^{\contr, \op})
\end{align*}
the category of $\Set$-valued $\arc_\varpi$-sheaves on $\Ban_K^{\contr,\op}$. However, instead of associating a ``underlying topological space'' to each $\arc_\varpi$-sheaf we associate a condensed set. The category of condensed sets, denoted by $\Cond$, was first introduced by Clausen and Scholze \cite{condensedlectures} and its defined as
\begin{align*}
	\Cond := \Shv_{\eff}(\Comp)
\end{align*}
the category of $\Set$-valued sheaves on $\Comp$ with respect to the effective topology\footnote{A word of warning is in order: the categories $\Ban_K^{\contr}$ and $\Comp$ are large categories, and so considering functors defined on them presents set-theoretic difficulties. In order to avoid this problems we implicitly impose a cardinal bound $< \kappa$ by some uncountable strong limit cardinal to the categories $\Ban_K^{\contr}$ and $\Comp$. Thus what we call a condensed set is called a $\kappa$-small condensed set in \cite{condensedlectures}. Our constructions do not depend on the choice of cardinal bound, thus we will not mention it throughout most of this work.} \footnote{Clausen and Scholze also define the category of $\kappa$-small condensed sets as $\Shv_{\eff}(\ProFin)$, where $\ProFin \subset \Comp$ is the category of profinite sets which forms a basis for the effective topology on $\Comp$, thus giving rise to an equivalent category.}. One reason we prefer to work with the category of condensed sets as opposed to the category of topological spaces, is that it mirrors the construction of $\cX_{\arc_\varpi}$ better, for instance, the category of condensed sets is a topos while the category of topological spaces is not. Furthermore, we can informally think of the $\arc_\varpi$-topology on $\Ban_K^{\contr}$ as the ``inverse image'' of the effective topology under the Berkovich functor $|-|: \Ban_K^{\contr} \rightarrow \Comp$. The following result shows that we can extend the Berkovich functor to all $\arc_\varpi$-sheaves.

\begin{thmx}[Berkovich Functor - Construction \ref{berko_funct}]\label{intro_const_berko_funct} There exists a unique colimit preserving functor, which we call the \emph{Berkovich functor},
\begin{align*}
	|-|: \cX_{\arc_\varpi} \longrightarrow \Cond
\end{align*}
making the following diagram commute
\begin{cd}
	\Ban_K^{\contr, \op} \ar[r, "\vert - \vert"] \ar[d, "\Yo_{\arc_\varpi}", swap] & \Comp \ar[d, "\Yo_{\eff}"] \\
	\cX_{\arc_\varpi} \ar[r, "\vert - \vert"] & \Cond
\end{cd}
Where $\Yo_{\tau}$ is the sheafified Yoneda functor with respect to the topology $\tau$. Recall that since compact Hausdorff spaces are sheaves with respect to the effective topology the Yoneda functor $\Yo_{\eff}$ is fully faithful, and by Theorem \ref{intro_tate_acyclicity_perfd} we learn that the restriction of $\Yo_{\arc_\varpi}$ to $\Perfd_K^{\Ban,\op}$ is fully faithful, but not in general\footnote{For example, the uniformization $\cM(A^u) \rightarrow \cM(A)$ becomes an isomorphism after applying $\Yo_{\arc_\varpi}$.} \footnote{Even though technically $\Yo_{\arc_\varpi}$ is defined on $\Ban_K^{\contr, \op}$ and not on $\Ban_K^{\op}$, since $\Yo_{\arc_\varpi}$ identifies $\cM(A)$ with its uniformization $\cM(A^u)$ and morphisms between uniform Banach $K$-algebras are contractive, by precomposing $\Yo_{\arc_\varpi}$ with the uniformization functor $\Ban_K^{\op} \overset{(-)^u}{\longrightarrow} \uBan_K \overset{\Yo_{\arc_\varpi}}{\longrightarrow} \cX_{\arc_\varpi}$ we obtain a natural extension of $\Yo_{\arc_\varpi}$ to $\Ban_K^{\op}$.}.
\end{thmx}

In what follows we will often not make a distinction between a compact Hausdorff space $X$ and it associated condensed set $\Yo_{\eff}(X)$, and we will just denote both by $X$. Moreover, for a condensed set $X: \Comp^{\op} \rightarrow \Set$ we will often consider the set $X(*)$, the value of $X$ at $* \in \Comp$, and refer to it as the \emph{underlying set} of $X$. On the other hand, if $\cM(A)$ is an affinoid perfectoid space we will not make a distinction between $\cM(A)$ and its associated $\arc_\varpi$-sheaf $\Yo_{\arc_\varpi}(\cM(A))$ and we will just denote both by $\cM(A)$. However, if $A$ is a general Banach $K$-algebra, since $\Yo_{\arc_\varpi}$ is not fully faithful we will write $\cM(A)_{\arc_{\varpi}}$ for $\Yo_{\arc_\varpi}(\cM(A))$.

\begin{example}[Perfectoid Torus]\label{intro_perfd_torus} Let $\CC_p$ be an algebraically closed perfectoid field obtained as the completion of an algebraic closure of $\mathbf{Q}_p$ and set $X = \cM(\CC_p \langle T^{\pm 1} \rangle)_{\arc_\varpi}$. Consider the family of maps
\begin{align*}
	\cdots \rightarrow X_n \rightarrow X_{n-1} \rightarrow \cdots \rightarrow X && X_{n} = \cM(\CC_p \langle T^{\pm 1/p^n} \rangle)_{\arc_\varpi}
\end{align*}
where the transition maps $X_n \rightarrow X_{n-1}$ are induced by the canonical inclusions $\CC_p \langle T^{\pm 1/p^{n-1}} \rangle \hookrightarrow \CC_p \langle T^{\pm 1/p^{n}} \rangle$, which are finite etale maps and in particular $\arc_\varpi$-covers. Then, the following identity
\begin{align*}
	\CC_p \langle T^{\pm 1/p^n} \rangle \cotimes_{\CC_p \langle T^{\pm 1} \rangle} \CC_p \langle T^{\pm 1/p^n} \rangle \overset{\simeq}{\longrightarrow} \prod_{g \in \mu_{p^n}(\CC_p)} \CC_p \langle T^{\pm 1/p^n} \rangle && x \otimes y \mapsto (x \epsilon_n(y))_{\epsilon_n \in \mu_{p^n}(\CC_p)}
\end{align*}
where an element $\epsilon_n \in \mu_{p^n}(\CC_p)$ acts on $\CC_p \langle T^{\pm 1/p^n} \rangle$ via the map $T^{1/p^n} \mapsto \epsilon_n T^{1/p^n}$, allows us to rewrite the isomorphism $\coeq \Big( X_n \times_X X_n \rightrightarrows X_n \Big) \overset{\simeq}{\longrightarrow} X$ as
\begin{align*}
	\coeq \Big( \sqcup_{\epsilon_n \in \mu_{p^n}(\CC_p)} X_n \rightrightarrows X_n \Big) \overset{\simeq}{\longrightarrow} X && \text{equivalently } X_n/\mu_{p^n}(\CC_p) \simeq X
\end{align*}
Hence, we learn that the map $X_n \rightarrow X$ presents $X_n$ as a $\mu_{p^n}(\CC_p)$-torsor, and since the Berkovich functor $\cX_{\arc_\varpi} \rightarrow \Cond$ preserves all colimits we get an induced isomorphism of compact Hausdorff spaces $|X_n|/\mu_{p^n}(\CC_p) \simeq |X|$. Next, define the object $X_{\infty} = \lim_{n} X_n$ where the limit is computed in $\cX_{\arc_\varpi}$, we see that $X_{\infty}$ is represented by $\cM(\CC_p \langle T^{1/p^\infty} \rangle)$ showing that it is an affinoid perfectoid space. By definition, the map $\CC_p \langle T^{\pm 1} \rangle \hookrightarrow \CC_p \langle T^{\pm 1/p^\infty} \rangle$ is a pro-finite etale map, and since the the Berkovich functor $|-|: \Ban_K^{\contr, \op} \rightarrow \Comp$ preserves all limits, the compact Hausdorff space $|X_{\infty}|$ identifies with $\lim_n |X_n|$, showing that $X_\infty \rightarrow X$ is an $\arc_\varpi$-cover -- Tate acyclicity for perfectoids (Theorem \ref{intro_tate_acyclicity_perfd}) suggests that from the point of view of the $\arc_\varpi$-topology we may think of $X_{\infty}$ as a universal cover of $X$. By the above discussion, we learn that $X_{\infty} \rightarrow X$ presents $X_\infty$ as a $\mu_{p^\infty}(\CC_p) = \lim_n \mu_{p^n} (\CC_p)$ torsor, where an element $(\epsilon_n) \in \lim_n \mu_{p^n}(\CC_p)$ acts on $X_\infty$ via the map $T^{1/p^n} \mapsto \epsilon_n T^{1/p^n}$ and where we regard the group $\mu_{p^\infty}(\CC_p)$ as a profinite group. Thus, we get the identity $X_\infty/\mu_{p^\infty}(\CC_p) \simeq X$, and since the Berkovich functor preserves all colimits we get an isomorphism $|X_{\infty}|/\mu_{p^\infty}(\CC_p) \simeq |X|$ of compact Hausdorff spaces.

Let us now explain a purely topological analog of the above example. Let $S^1$ be the unit circle, which we regard as an analog of $\cM(\CC_p \langle T^{\pm 1} \rangle)_{\arc_\varpi}$ and the map $p^n$-fold covering map $S^1 \rightarrow S^1$ as an analog of the morphism $X_n \rightarrow X$. The $p^n$-fold covering map $S^1 \rightarrow S^1$ admits an action of $\mu_{p^n}(\CC_p)$ by deck-transformations and induces an isomorphism $S^1/\mu_{p^n}(\CC_p) \simeq S^1$. Taking the limit $S_{\infty}^1 = \lim_n S^1$, where the transition maps are the $p$-fold covering space map $S^1 \rightarrow S^1$, we obtain a solenoid-like compact Hausdorff space $S^1_{\infty}$ -- analogous to the affinoid perfectoid space $\cM(\CC_p \langle T^{\pm 1/p^\infty} \rangle)$. By construction, the map $S^1_{\infty} \rightarrow S^1$ presents $S^1_{\infty}$ as a $\mu_{p^\infty}(\CC_p)$-torsor, and induces the identity $S^1_{\infty}/\mu_{p^\infty}(\CC_p) \simeq S^1$. The construction of $S^1_{\infty} \rightarrow S^1$ is analogous to the universal covering space map $\mathbf{R} \rightarrow S^1$, with the important distinction that $S^1_{\infty}$ is a compact Hausdorff space.
\end{example}

\begin{example}[Covers by Perfectoids]\label{intro_cover_perfd} Every Banach $K$-algebra $A$ admits a non-canonical $\arc_\varpi$-cover by a perfectoid Banach $K$-algebra, which implies that $\Perfd_K^{\Ban, \op} \subset \Ban_K^{\contr, \op}$ form a basis for the $\arc_\varpi$-topology. The Gelfand transform of $A$ -- defined as $A \rightarrow \prod_{x \in \cM(A)} \cH(x)$ -- is an $\arc_\varpi$-cover, and let $\overline{\cH(x)}^{\wedge}$ be an algebraic closure of $\cH(x)$ followed by the completion of the non-archimedean field $\overline{\cH(x)}$. The resulting object $\prod_{x \in \cM(A)} \overline{\cH(x)}^{\wedge}$ is a perfectoid Banach $K$-algebra, and the induced contractive morphism
\begin{align*}
	A \rightarrow \prod_{x \in \cM(A)} \overline{\cH(x)}^{\wedge}
\end{align*}
is an $\arc_\varpi$-cover of $A$ by the perfectoid Banach $K$-algebra $\prod_{x \in \cM(A)} \overline{\cH(x)}^{\wedge}$. For completeness sake let us mention again that the underlying topological space of $\cM(\prod_{x \in \cM(A)} \overline{\cH(x)}^{\wedge})$ can be canonically identified with $\beta(\cM(A)^{\delta})$ -- the Stone-Cech compactification of $\cM(A)$ regarded as a discrete set.
\end{example}

Following Grothendieck we regard a topos as a natural place to do geometry, and our work aims to show that we can do non-archimedean geometry over a perfectoid field $K$ in the topos $\cX_{\arc_\varpi}$ -- this is parallels Scholze's approach to non-archimedean geometry via $v$-sheaves \cite{diamonds_scholze}. In fact, in order to prove the classification of affinoid domains for affinoid perfectoid spaces (Theorem \ref{intro_mono_perfd}) we are forced to consider the totality of the category of $\arc_\varpi$-sheaves. Along the way we prove a much more general version of the aforementioned theorem -- we will often need to impose mild finiteness hypothesis to the $\arc_\varpi$-sheaves we consider, like being quasicompact or quasiseparated, we follow \cite[Expose VI]{SGA4} for the relevant background. In fact, our original slogan now admits the following more general version
\begin{center}
	\emph{``Via the Berkovich functor, quasicompact quasiseparated $\arc_\varpi$-sheaves behave as if they were compact Hausdorff spaces''}\footnote{In Proposition \ref{berko_funct_stability} we show that the Berkovich functor preserves quasicompact and quasiseparated objects, and in Proposition \ref{epi_condensed_sets} we showed that subcategory of quasicompact quasiseparated objects of $\Cond$ is canonically equivalent to the category of compact Hausdorff spaces.}.
\end{center}

To get the theory off the ground, let us provide some examples of $\qcqs$ $\arc_\varpi$-sheaves and mention some of the basic information of $X \in \cX_{\arc_\varpi}$ we can read off from the set $|X|(*)$.

\begin{example}[Proposition \ref{basic_arc_pi_properties}] For every object $\cM(A) \in \Ban_K^{\contr, \op}$, its image under the sheafified Yoneda functor $\cM(A)_{\arc_\varpi}$ is a quasicompact quasiseparated $\arc_\varpi$-sheaf.
\end{example}

\begin{thmx}[Section \ref{section_arc_pi_topos}]\label{intro_generalities_arc_topos} The $\arc_\varpi$-topos $\cX_{\arc_\varpi}$ has the following basic properties
\begin{enumerate}[(1)]
	\item(Points) Let $X \in \cX_{\arc_\varpi}$, then for each $x \in |X|(*)$ there exists a perfectoid non-archimedean field $L$ and a morphism $\cM(L) \rightarrow X$ such that under the Berkovich functor it gets mapped to $x: * \rightarrow |X|$. Furthermore, if we have a pair of morphisms $Y \rightarrow X \leftarrow Z$ such that $|Y|(*) \times_{|X|(*)} |Z|(*) \not= \emptyset$ then $Y \times_X Z \not= \emptyset$.
	\item(Epimorphisms) Let $X \rightarrow Y$ be a morphism in $\cX_{\arc_\varpi}$, and assume that $Y$ is qcqs and $X$ is quasicompact, then $X \rightarrow Y$ is an epimorphism if and only if $|X|(*) \rightarrow |Y|(*)$ is a surjective map of sets.
	\item(Residue Fields) Let $Y \in \cX_{\arc_\varpi}$ be a quasiseparated object, then for each $x \in |Y|(*)$ there exists a unique qcqs object $Y_x$ together with a monomorphism $Y_x \hookrightarrow Y$ such that it gets mapped to $x: * \rightarrow |Y|$ under the Berkovich functor. We call the resulting map $Y_x \hookrightarrow Y$ the \emph{completed residue field} of $Y$ at $x \in |Y|(*)$.
\end{enumerate}
\end{thmx}

\begin{thmx}[Isomorphisms\footnote{Compare with \cite[Lemma 12.5]{diamonds_scholze}.} - Proposition \ref{iso_arc_topos}]\label{intro_iso_arc_topos} Let $X \rightarrow Y$ be a morphism of qcqs objects of $\cX_{\arc_\varpi}$. Then, the following are equivalent
\begin{enumerate}[(1)]
	\item The morphism $X \rightarrow Y$ is an isomorphism.
	\item The morphism $X \rightarrow Y$ is an \emph{$\arc_\varpi$-equivalence}: there exists a cofinal collection of perfectoid non-archimedean fields such that the induced map $X(L) \rightarrow Y(L)$ is a bijection for every object in this cofinal system.
	\item The induced map $|X|(*) \rightarrow |Y|(*)$ is a bijection, and for each $x \in |X|(*) \simeq |Y|(*) \ni y$ the induced map of completed residue fields $X_x \rightarrow Y_y$ is an $\arc_\varpi$-equivalence.
\end{enumerate}
\end{thmx}

The previous result allows us to understand to what extent the sheafified Yoneda functor $\Yo_{\arc_\varpi}: \Ban_K^{\contr, \op} \rightarrow \cX_{\arc_\varpi}$ is not fully faithful. We say that a morphism $f: \cM(A) \rightarrow \cM(B)$ in $\Ban_K^{\contr, \op}$ if part of the collection $\cW$ if the map $|f|(*): |\cM(A)|(*) \rightarrow |\cM(B)|(*)$ is bijective and each $x \in |\cM(A)|(*) \simeq |\cM(B)|(*) \ni y$ the induced map $\cM(\cH(x)) \rightarrow \cM(\cH(y))$ of completed residue fields is an $\arc_\varpi$-equivalence. Then, the functor $\Yo_{\arc_\varpi}: \Ban_K^{\contr, \op} \rightarrow \cX_{\arc_\varpi}$ factors as
\begin{align*}
	\Yo_{\arc_\varpi}: \Ban_K^{\contr, \op} \rightarrow \Ban_K^{\contr, \op}[\cW^{-1}] \hookrightarrow \cX_{\arc_\varpi}
\end{align*}
where the first functor is a localization of $\Ban_K^{\contr, \op}$ with respect to $\cW$, and the following functor is fully faithful.

The following result is a version of the Gerritzen-Grauert theorem for rigid analytic geometry, and extends Theorem \ref{intro_mono_perfd} to more general $\arc_\varpi$-sheaves.

\begin{thmx}[Monomorphisms\footnote{Compare with \cite[Proposition 12.15]{diamonds_scholze}.} - Proposition \ref{mono_arc_topos}]\label{intro_mono_arc_topos} Let $X \rightarrow Y$ be a morphism in $\cX_{\arc_\varpi}$, and assume that $X$ is qcqs and $Y$ is quasiseparated. Then, the following are equivalent
\begin{enumerate}[(1)]
	\item The morphism $X \rightarrow Y$ is a monomorphism.
	\item The induced map $|X|(*) \rightarrow |Y|(*)$ is an injective map of sets, and for each $x \in |X|(*)$ with image $y \in |Y|(*)$ the induced map of completed residue fields $X_x \rightarrow Y_y$ is an $\arc_\varpi$-equivalence.
	\item The morphism $X \rightarrow Y$ is an \emph{analytic domain}: for any object $Z \in \cX_{\arc_\varpi}$ and any morphism $Z \rightarrow Y$ satisfying $\im(|Z|(*) \rightarrow |Y|(*)) \subset \im(|X|(*) \rightarrow |Y|(*))$, there exists a unique morphism $Z \rightarrow X$ making the following diagram commute
	\begin{cd}
		& Z \ar[ld, dashed] \ar[d] \\
		X \ar[r] & Y
	\end{cd}
\end{enumerate}
\end{thmx}

\begin{example} Let $A$ be a Banach $K$-algebra, and $X = \cM(A)_{\arc_\varpi}$ its associated $\arc_\varpi$-sheaf. Then, the following are examples of analytic domains
\begin{enumerate}[(1)]
	\item (Residue Fields) For each $x \in |X|(*)$ there exists a non-archimedean field $\cH(x)$ together with a monomorphism $\cM(\cH(x))_{\arc_\varpi} \rightarrow X$.
	\item (Rational Domains) For any subset $|\cM(A)| \supset V := \{ x \in \cM(A) \text{ such that } |f_i(x)| \le |g(x)| \}$, where $\{g, f_1, \dots, f_n\} \subset A$ generate the unit ideal, there exists a Banach $K$-algebra $B$ and a monomorphism of $\arc_\varpi$-sheaves $\cM(B)_{\arc_\varpi} \rightarrow \cM(A)_{\arc_\varpi}$ with image $V \subset |\cM(A)| = |\cM(A)_{\arc_\varpi}|$.
	\item (Zariski Closed Subsets) For any closed ideal $I \in A$, the induced map $\cM(A/I)_{\arc_\varpi} \rightarrow \cM(A)_{\arc_\varpi}$ is a monomorphism of $\arc_\varpi$-sheaves.
\end{enumerate}
\end{example}

\section{Analytic Geometry}

In the first section we introduced and studied the category of affinoid perfectoid spaces, in this section we will introduce and study global analogs called perfectoid spaces, and we will do so by isolating them as a full-subcategory of $\cX_{\arc_\varpi}$\footnote{Parallel to what happens in algebraic geometry, where we can isolate the category of schemes as a full-subcategory of $\Shv_{\Zar}(\CAlg^{\op})$.}. Moreover, we wield the full power of the $\arc_\varpi$-topos $\cX_{\arc_\varpi}$ to develop a theory of non-archimedean geometry that works without any finiteness hypothesis; to the best of our knowledge this is the first instance in the literature where Berkovich geometry has been developed beyond the topologically of finite type case. One of the major obstructions in doing so is that there is no structure sheaf on Banach $K$-algebras in this generality\footnote{See \cite[4.1]{buzzard2014stably} for a counterexample.} -- we remedy the this by replacing the structure presheaf by its perfectoidization.

Let us introduce the basic objects of study. We say that an $\arc_\varpi$-sheaf $X$ is a \emph{perfectoid space} if there exists a collection of monomorphisms $\{\cM(A_i) \hookrightarrow X\}_{i \in I}$ from affinoid perfectoid spaces\footnote{By Theorem \ref{intro_tate_acyclicity_perfd} we know that $\Perfd_K^{\Ban, \op}$ embeds fully faithfully in $\cX_{\arc_\varpi}$ via the Yoneda functor $\Yo_{\arc_\varpi}$. Hence, from this point onwards we regard the category of affinoid perfectoid spaces as a full subcategory of $\cX_{\arc_\varpi}$.}, such that the induced map $\sqcup_{i \in I} \cM(A_i) \twoheadrightarrow X$ is an epimorphism of $\arc_\varpi$-sheaves (cf. Definition \ref{defn_perfectoid_space}). More generally, we say that an $\arc_\varpi$-sheaf $X$ is a \emph{$\arc_\varpi$-analytic space} if there exists a collection of monomorphisms $\{\cM(A_i)_{\arc_\varpi} \hookrightarrow X\}_{i \in I}$, where each $A_i$ is a Banach $K$-algebra, such that the induced map $\sqcup_{i \in I} \cM(A_i)_{\arc_\varpi} \twoheadrightarrow X$ is an epimorphism of $\arc_\varpi$-sheaves (cf. Definition \ref{defn_analytic_space}). In particular, if $X$ is of the form $\cM(A)_{\arc_\varpi}$ for some Banach $K$-algebra $A$ then we say that $X$ is an \emph{affinoid $\arc_\varpi$-analytic space}. Furthermore, we say that a perfectoid space (resp $\arc_\varpi$-analytic space) $X$ is quasicompact (resp. quasiseparated) if it is so as an object of $\cX_{\arc_\varpi}$.

Since the category of perfectoid spaces is a full subcategory of the category of $\arc_\varpi$-analytic spaces we will formulate our results using the language of $\arc_\varpi$-analytic spaces.

\begin{defn}[Maximal Atlas]\label{into_defn_max_atlas} Let $X$ be an $\arc_\varpi$-analytic space (resp. condensed set). Define $\Sub(X)_{\qcqs}$ as the category whose objects are monomorphisms $Y \hookrightarrow X$ from a qcqs $\arc_\varpi$-sheaves (resp. compact Hausdorff spaces) and morphisms are maps $Y_1 \rightarrow Y_2$ making the following diagram commute
\begin{cd}
	Y_1 \ar[rd, hook] \ar[rr] && Y_2 \ar[dl, hook] \\
	& X
\end{cd}
In particular we see that any map $Y_1 \rightarrow Y_2$ in $\Sub(X)_{\qcqs}$ must be a monomorphism. We call $\Sub(X)_{\qcqs}$ the category of \emph{qcqs subobjects} of $X$. Furthermore, if $X$ is an $\arc_\varpi$-analytic space we can also consider the full-subcategory $\Sub(X)_{\Aff} \subset \Sub(X)_{\qcqs}$ spanned by affinoid $\arc_\varpi$-analytic spaces $Y = \cM(A)_{\arc_\varpi}$, we regard $\Sub(X)_{\Aff}$ as an analog of Berkovich's \emph{maximal affinoid atlas} (cf. \cite[Proposition 1.2.15]{berkovichetale}).
\end{defn}

\begin{prop}[Proposition \ref{subobjects_arc_topos}] Let $X$ be a quasiseparated $\arc_\varpi$-analytic space, then the Berkovich functor $|-|: \cX_{\arc_\varpi} \rightarrow \Cond$ induces an equivalence of categories
\begin{align*}
	\Sub(X)_{\qcqs} \overset{\simeq}{\longrightarrow} \Sub(|X|)_{\qcqs} && (Y \hookrightarrow X) \mapsto (|Y| \hookrightarrow |X|)
\end{align*}
In particular for any pair of morphisms $Y_1 \hookrightarrow X \hookleftarrow Y_2$, there is a canonical isomorphism $|Y_1 \times_X Y_2| \simeq |Y_1| \times_{|X|} |Y_2| = |Y_1| \cap |Y_2|$.
\end{prop}

The following result greatly extends Theorem \ref{intro_tate_acyclicity_perfd} beyond the perfectoid case; however, it comes at the price that we can no longer work with a sheaf of modules, but rather have to work with sheaves of complexes and thus we have to work in the almost derived $\infty$-category of $K_{\le 1}$, where almost mathematics is performed with respect to the ideal $(\varpi^{1/p^\infty}) \subset K_{\le 1}$. In order to formulate this result, we need the notion of perfectoidization of a $\varpi$-complete $K_{\le 1}$-algebra introduced in \cite[Section 8]{prisms}, we refer the reader to loc. cit. for the definitions.

\begin{thmx}[Structure Sheaf - Proposition \ref{perfectoidization_banach}]\label{intro_struct_sheaf_perfd} Let $X = \cM(A)_{\arc_\varpi}$ be an affinoid $\arc_\varpi$-analytic space. Then, the functor\footnote{By Proposition \ref{perfectoidization_banach} we learn that the construction $(\cM(B)_{\arc_\varpi} \hookrightarrow \cM(A)_{\arc_\varpi}) \mapsto (B_{\le 1, \perfd})^a$ only depends on $\cM(B)_{\arc_\varpi}$ and not on the choice of representative $B$.}
\begin{align*}
	\cO_{X, \perfd}(-)_{\le 1}^a: \Sub(X)_{\Aff}^{\op} \longrightarrow \cD(K_{\le 1})^{\wedge a}_{\varpi}  && (\cM(B)_{\arc_\varpi} \hookrightarrow \cM(A)_{\arc_\varpi}) \mapsto (B_{\le 1, \perfd})^a
\end{align*}
is an $\arc_\varpi$-sheaf of complexes in the almost derived $\infty$-category of $K_{\le 1}$, where almost mathematics is performed with respect to the ideal $(\varpi^{1/p^\infty}) \subset K_{\le 1}$.
\end{thmx}

\begin{example}[Proposition \ref{properties_perfd}] The following examples compute $\cO_{X, \perfd}(-)_{\le 1}^a$ in a couple of special cases
\begin{enumerate}[(1)]
	\item If $A$ is a Banach $K$-algebra of characteristic $p$, then $(A_{\le 1, \perfd})^a = (A_{\le 1, \perf})^a$, where the latter is just the $\varpi$-completed colimit perfection $A_{\le 1, \perf} = \colim_{x \mapsto x^p} A_{\le 1}$.
	\item If $A$ is a perfectoid Banach $K$-algebra, we may restrict to the full-subcategory $\Sub(X)_{\Perfd} \subset \Sub(X)_{\Aff}$ spanned by perfectoid Banach $K$-algebras, then the structure sheaf $\cO_{X, \perfd}(-)_{\le 1}^a$ admits a much simpler description as
	\begin{align*}
	\cO_{X}(-)_{\le 1}^a: \Sub(X)_{\Perfd}^{\op} \longrightarrow \cD(K_{\le 1})^{\wedge a}_{\varpi}  && (\cM(B) \hookrightarrow \cM(A)) \mapsto (B_{\le 1})^a
	\end{align*}
	\item Let $X = \cM(\CC_p \langle T^{\pm 1} \rangle)_{\arc_\varpi}$, we follow the discussion on Example \ref{intro_perfd_torus} to compute $\cO_{X, \perfd} (X)_{\le 1}^a$. Since $X_{\infty} \rightarrow X$ is a $\mu_{p^{\infty}}(\CC_p)$-torsor we can use the presentation
	\begin{align*}
		\CC_p \langle T^{\pm 1/p^\infty} \rangle = \widehat{\bigoplus}_{i \in \ZZ[\frac{1}{p}]} \CC_p \cdot T^i
	\end{align*}
	and $\arc_\varpi$-descent for $(-)_{\le 1, \perfd}^a$ to identify $(\CC_p\langle T^{\pm 1} \rangle)_{\le 1, \perfd}^a$ as the complex in the category $\cD(K_{\le 1})^{\wedge a}_{\varpi}$
	\begin{align*}
		(\CC_p\langle T^{\pm 1} \rangle)_{\le 1, \perfd}^a \simeq  
		\widehat{\bigoplus}_{i \in \ZZ[\frac{1}{p}]} \Big ( \CC_{p, \le 1} \cdot T^i \overset{T^i \mapsto (1 - \epsilon^i)T^i}{\xrightarrow{\hspace{2cm}}}  \CC_{p, \le 1} \cdot T^i \Big)
	\end{align*}
	where $\epsilon = (\epsilon_n) \in \lim_n \mu_{p^n}(\CC_p) = \mu_{p^\infty}(\CC_p)$ is a generator. In particular, we see that when $i \in \ZZ$ the map $T^{i} \mapsto (1 - \epsilon^i)T^i$ is the zero map, while when $i \not\in \ZZ$ it is an isomorphism. Thus, we obtain the identity
	\begin{align*}
		(\CC_p \langle T^{\pm 1} \rangle)_{\le 1, \perfd}^a \simeq \widehat{\bigoplus}_{i \in \ZZ} \Big ( \CC_{p, \le 1} \cdot T^i \overset{0}{\longrightarrow}  \CC_{p, \le 1} \cdot T^i \Big)
	\end{align*}
	providing an example where $(\CC_p\langle T^{\pm 1} \rangle)_{\le 1, \perfd}^a$ is not concentrated in degree zero\footnote{Compare with \cite[Section 4]{scholze_padic_hodge}.}.
\end{enumerate}
\end{example}

In order to make closer contact with Berkovich's theory of $K$-analytic spaces, we would like to show that $\arc_\varpi$-analytic spaces admit a theory of strong morphisms (cf. \cite[Definition 1.2.7]{berkovichetale}) -- informally, a theory of strong morphisms says that given an a morphism $X \rightarrow Y$ of $K$-analytic spaces and an affinoid cover $\{\cM(A_i) \hookrightarrow Y\}$, there exists an affinoid cover of $\{\cM(B_j) \hookrightarrow X\}$ refining the cover on $Y$. We show that separated $\arc_\varpi$-analytic spaces admit a theory of strong morphisms.

\begin{defn}[Definition \ref{defn_separated}] Let $Y \rightarrow X$ be a morphism of $\arc_\varpi$-analytic spaces. Then,
\begin{enumerate}[(1)]
	\item We say that $Y \rightarrow X$ is \emph{affine} if for every morphisms $\cM(A)_{\arc_\varpi} \rightarrow X$, where $A$ is a Banach $K$-algebra, the fiber product $Y \times_X \cM(A)_{\arc_\varpi}$ is represented by some Banach $K$-algebra $B$, in other words we have an identification $\cM(B)_{\arc_\varpi} = Y \times_X \cM(A)_{\arc_\varpi}$.
	\item We say that $Y \rightarrow X$ is a \emph{closed immersion} if it is affine, and for every morphisms $\cM(P) \rightarrow X$, where $P$ is a perfectoid Banach $K$-algebra, the induced morphism $Y \times_X \cM(P) \rightarrow \cM(P)$ is represented by a surjective map $P \twoheadrightarrow R$ of perfectoid Banach $K$-algebras. In other words, there exists a perfectoid Banach $K$-algebra $R$ and an isomorphism $\cM(R) = Y \times_X \cM(P)$ such that the induced map $P \rightarrow R$ is surjective.
	\item We say that $Y \rightarrow X$ is \emph{separated} if the diagonal map $\Delta: Y \rightarrow Y \times_X Y$ is a closed immersion. In particular, we say that $X$ is separated if the map $X \rightarrow \cM(K)$ is separated.
\end{enumerate}
\end{defn}

\begin{example} The following are some examples of the classes of morphisms of $\arc_\varpi$-analytic spaces we just introduced
\begin{enumerate}[(1)]
	\item A morphism $\cM(B)_{\arc_\varpi} \rightarrow \cM(A)_{\arc_\varpi}$ of affinoid $\arc_\varpi$-analytic spaces is affine (cf. Example \ref{affine_morphism_banach}).
	\item If $A \rightarrow B$ is a contractive surjective morphism of Banach $K$-algebras, the induced map $\cM(B)_{\arc_\varpi} \rightarrow \cM(A)_{\arc_\varpi}$ is a closed immersion\footnote{This result relies critically on the fact that Zariski closed sets are strongly Zariski closed \cite[Theorem 7.4]{prisms}.} (cf. Example \ref{closed_immersions_banach}). 
	\item All affinoid $\arc_\varpi$-analytic spaces $\cM(A)_{\arc_\varpi}$ are separated (cf. Example \ref{separated_banach}).
\end{enumerate}
\end{example}

\begin{rem}[Atlases and strong morphisms - Remark \ref{an_spc_atlas_strong_morphism}] Let us use the technology developed so far to make contact with Berkovich's theory of $K$-analytic spaces as developed in \cite[Section 1]{berkovichetale}. Let $X$ be a separated $\arc_\varpi$-analytic space, then the category $\Sub(X)_{\Aff}$ is closed under fiber products and satisfies most of the conditions for a net of compact hausdorff spaces on a topological space in the sense of Berkovich\footnote{With the exception that for each $x \in |X|(*)$ there need not exists a finite collection of objects $\{Y_i \hookrightarrow X\}$ in $\Sub(X)_{\Aff}$ whose union contains a neighborhood of $x$. Later we will introduce the notion of a ``locally compact'' $\arc_\varpi$-analytic space which is meant to fill in this gap.}. Furthermore, one can show that the natural functor $\Sub(X)_{\Aff} \rightarrow \cX_{\arc_\varpi}$ defined by $(Y \hookrightarrow X) \mapsto Y$ satisfies $\colim_{\Sub(X)_{\Aff}} Y = X$.

Finally, let us argue that any morphism $f: X \rightarrow Y$ of separated $\arc_\varpi$-analytic spaces comes from a ``strong morphism'' in the sense of Berkovich \cite[Definition 1.2.7]{berkovichetale}. Indeed, for each monomorphism $\cM(A)_{\arc_\varpi} \hookrightarrow X$ there exists a finite collection of monomorphisms $\{\cM(A_i)_{\arc_\varpi} \hookrightarrow \cM(A)_{\arc_\varpi}\}$ inducing an $\arc_\varpi$-cover, such that for each $\cM(A_i)_{\arc_\varpi}$ there exists a monomorphism $\cM(B_i) \hookrightarrow Y$ such that $|f|(*) \Big( |\cM(A_i)|(*) \Big) \subset |\cM(B_i)|(*)$. In particular, there exists an essentially unique morphism $\cM(A_i)_{\arc_\varpi} \rightarrow \cM(B_i)_{\arc_\varpi}$ making the following diagram commute
\begin{cd}
	\cM(A_i)_{\arc_\varpi} \ar[r] \ar[d, hook] & \cM(B_i)_{\arc_\varpi} \ar[d, hook] \\
	X \ar[r] & Y
\end{cd}
\end{rem}

Before diving deeper into the theory, let us show a recipe for producing separated $\arc_\varpi$-analytic spaces as the ``$\varpi$-complete generic fiber'' of quasicompact separated schemes over $K_{\le 1}$.

\begin{thmx}[The Generic Fiber Functor - Construction \ref{const_generic_fiber_funct}] There exists a functor
\begin{align*}
	(-)_{\eta, \arc_\varpi}: \Sch_{K_{\le 1}, \qcqs} \rightarrow \cX_{\arc_\varpi} && X \mapsto X_{\eta, \arc_\varpi}
\end{align*}
which we call the \emph{generic fiber functor} satisfying the following properties:
\begin{enumerate}[(1)]
	\item If $X$ is qcqs $K_{\le 1}$-scheme, then $X_{\eta, \arc_\varpi}$ is a qcqs $\arc_\varpi$-analytic space.
	\item If $Y \rightarrow X$ is a closed immersion of qcqs $K_{\le 1}$-schemes, then $Y_{\eta, \arc_\varpi} \rightarrow X_{\eta, \arc_\varpi}$ is a closed immersion of $\arc_\varpi$-sheaves.
	\item If $X$ is a quasicompact separated $K_{\le 1}$-scheme, then $X_{\eta, \arc_\varpi}$ is a quasicompact separated $\arc_\varpi$-analytic space.
\end{enumerate} 
\end{thmx}

\begin{example} The affine line $\mathbf{A}^1_{K_{\le 1}}$ is a quasicompact separated scheme over $K_{\le 1}$, and its image under the generic fiber functor is given by the unit disk $\mathbf{D}^1_{K, \arc_\varpi} = \cM(K\langle T \rangle)_{\arc_\varpi}$ which is a quasicompact separated $\arc_\varpi$-analytic space. Similarly, the projective line $\mathbf{P}^1_{K_{\le 1}}$ is a quasicompact separated $K_{\le 1}$ scheme, and its image under the generic fiber functor gives rise to the quasicompact separated $\arc_\varpi$-analytic space $\mathbf{P}^1_{K, \arc_\varpi}$ defined as the colimit of the following diagram
\begin{cd}
	& \cM(K \langle T^{\pm 1} \rangle)_{\arc_\varpi} \ar[ld, "p_1", swap] \ar[rd, "p_2"] \\
	\cM(K \langle T_1 \rangle)_{\arc_\varpi} && \cM(K \langle T_2 \rangle)_{\arc_\varpi}
\end{cd}
where $p_1(T) = T_1$ and $p_2(T) = T_2^{-1}$. There are also perfectoid analogs of the previous examples, for instance the affine scheme $\mathbf{A}_{K_{\le 1}}^{1, \infty} = \Spec(K_{\le 1}[T^{1/p^\infty}])$ is a quasicompact separated scheme over $K_{\le 1}$, and its image under the generic fiber functor is the perfectoid unit disk $\mathbf{D}_{K}^{1, \infty} = \cM(K \langle T^{1/p^\infty} \rangle)$ which is a quasicompact separated perfectoid space. Similarly we can define the quasicompact separated scheme $\mathbf{P}_{K_{\le 1}}^{1, \infty}$ as the scheme obtained by gluing two copies of $\mathbf{A}^{1, \infty}_{K_{\le 1}}$ along $\Spec(K_{\le 1}[T^{\pm 1/p^\infty}])$, analogously to how $\mathbf{P}^{1}_{K_{\le 1}}$ is defined. We denote the image of $\mathbf{P}^{1, \infty}_{K_{\le 1}}$ under the generic fiber functor by $\mathbf{P}^{1, \infty}_{K}$, which is a quasicompact separated perfectoid space, which can be defined as the colimit of the diagram
\begin{cd}
	& \cM(K \langle T^{\pm 1/p^\infty} \rangle) \ar[ld, "p_1", swap] \ar[rd, "p_2"] \\
	\cM(K \langle T_1^{1/p\infty} \rangle) && \cM(K \langle T_2^{1/p^\infty} \rangle)
\end{cd}
where $p_1(T^{1/p^n}) = T_1^{1/p^n}$ and $p_2(T^{1/p^n}) = T_2^{-1/p^n}$.
\end{example}

Its rather surprising that we have been able to developed much of this theory without relying on a notion of ``open subsets'', we conclude this introduction by explaining that $\arc_\varpi$-analytic spaces come equipped with a good theory of open subsets. Our theory of $\arc_\varpi$-analytic spaces (and more generally $\arc_\varpi$-sheaves) differ the most from Huber's theory of adic spaces (or Scholze's theory of $v$-sheaves) when we study their open subsets. To illustrate the differences recall that in Scholze's theory of perfectoid spaces for an affinoid perfectoid space $X$ and a rational domain $V$ the canonical map $V \hookrightarrow X$ is an open immersion, while for us the inclusion of a rational domain into an affinoid perfectoid space $V \hookrightarrow X$ is not an open immersion generally\footnote{We expect this distinction to lead to small differences in the etale topology of $\arc_\varpi$-analytic spaces, when compared to Scholze's theory of $v$-sheaves, but we do not pursue this line of work.}.

\begin{defn}[Open Immersions - Definition \ref{defn_open_immersion_arc}] A monomorphism $U \hookrightarrow X$ of $\arc_\varpi$-sheaves is an \emph{open immersion} if the induced map $|U| \hookrightarrow |X|$ is an open immersion of condensed sets (cf. Definition \ref{defn_open_immersion_cond}).
\end{defn}

\begin{thmx}[Propositions \ref{open_subobject_arc} and \ref{open_of_spc_is_spc}] The collection of open immersions have the following properties
\begin{enumerate}[(1)]
	\item If $X$ is an $\arc_\varpi$-analytic space (resp. a perfectoid space), and $U \hookrightarrow X$ is an open immersion of $\arc_\varpi$-sheaves then $U$ is an $\arc_\varpi$-analytic space (resp. a perfectoid space).
	\item If $X$ is a quasiseparated $\arc_\varpi$-analytic space and $U \hookrightarrow X$ an open immersion, then it is an \emph{analytic domain}: for any morphism $Z \rightarrow X$ such that $\im(|Z|(*) \rightarrow |X|(*)) \subset |U|(*)$ there exists a unique morphism $Z \rightarrow U$ making the following diagram commute
	\begin{cd}
		& Z \ar[ld, dashed] \ar[d] \\
		U \ar[r, hook] & X
	\end{cd}
	\item For a quasiseparated $\arc_\varpi$-analytic space (resp. condensed set) $X$ denote by $\Open(X)$ the category of open immersions from $\arc_\varpi$-analytic spaces (resp. condensed sets) into $X$. Then, the Berkovich functor induces an equivalence of categories
	\begin{align*}
		\Open(X) \overset{\simeq}{\longrightarrow} \Open(|X|)
	\end{align*}
\end{enumerate}
\end{thmx}

\begin{example}[Zariski Opens] Let $X = \cM(A)_{\arc_\varpi}$ be an affinoid $\arc_\varpi$-analytic space, $f \in A$ an object of $A$, and $|X|_{f \not= 0} \subset |X|$ the open subset of $|X|$ where the function $f$ does not vanish. Then, there exists an essentially unique $\arc_\varpi$-analytic space $X_{f \not= 0}$ together with a monomorphism $X_{f \not= 0} \hookrightarrow X$ such that under the Berkovich functor it induces the open subset $|X|_{f \not= 0} \subset |X|$.
\end{example}

With this definitions at hand we are able to isolate a full subcategory of quasiseparated $\arc_\varpi$-analytic spaces which are ``locally compact'' (cf. Definition \ref{defn_locally_compact_arc}); this is analogous to Berkovich's definition of ``good'' $K$-analytic space. Instead of giving more definitions, let us just state some consequences.

\begin{thmx}[Proposition \ref{properties_loc_compact_arc}] Let $X$ be a quasiseparated locally compact $\arc_\varpi$-analytic space, and $\{U_i \hookrightarrow X\}_{i \in I}$ be a (possibly infinite) collection of open immersions, then the induced map
\begin{align*}
	\sqcup_{i \in I} U_i \longrightarrow X
\end{align*}
is an epimorphism of $\arc_\varpi$-sheaves. In particular, this implies that the following map is an isomorphism
\begin{align*}
	\coeq \Big( \sqcup_{i,j \in I} U_i \times_X U_j \rightrightarrows \sqcup_{i \in I} U_i   \Big) \overset{\simeq}{\longrightarrow} X
\end{align*}
\end{thmx}

\begin{example} Let $\{f_1, \dots, f_n\} \subset A$ be a collection of objects of the Banach $K$-algebra $A$ generating the unit ideal. Set $X = \cM(A)_{\arc_\varpi}$ and $X_{f_i \not= 0} \hookrightarrow X$ the open immersion corresponding to the open set $|X|_{f_i \not= 0} \subset |X|$. Then, the following map is an isomorphism
\begin{align*}
	\coeq \Big( \sqcup_{i,j \in I} X_{f_i \not= 0} \times_X X_{f_j \not= 0} \rightrightarrows \sqcup_{i \in I} X_{f_i \not= 0}   \Big) \overset{\simeq}{\longrightarrow} X
\end{align*}
Furthermore, we have the identification $X_{f_i \not= 0} \times_X X_{f_j \not= 0} = X_{f_i f_j \not= 0}$. Showing that affinoid $\arc_\varpi$-analytic spaces can be glued along ``Zariski open subsets'' just like schemes.
\end{example}

\section{Organization}

In Chapter \ref{chapt_comm_alg} we lay the foundations for the work ahead of us. The main goal is to establish the dictionary (Theorem \ref{intro_dictionary}), which allows us to translate questions about non-archimedean functional analysis on the category $\Ban_K^{\contr}$ to more algebraic questions on the category $\CAlg_{K_{\le 1}}^{\wedge a \tf}$. As a by-product we also make contact between modern definition of integral perfectoid algebras with the original definition of perfectoid Banach $K$-algebras. In Chapter \ref{chapt_berk_sp} we study the geometry of Banach $K$-algebras via their Berkovich spectrum and establish some of their basic results beyond the topologically of finite type case. Furthermore, leveraging the dictionary we establish the existence of a structure sheaf and $\arc_\varpi$-descent for affinoid perfectoid spaces. Finally, in Chapter \ref{chapt_perfd_spc} we introduce the $\arc_\varpi$-topos $\cX_{\arc_\varpi}$, which we argue is a natural place to do Berkovich geometry, and construct the Berkovich functor (Theorem \ref{intro_const_berko_funct}) assigning a underlying topological space to every $\arc_\varpi$-sheaf. Then, we isolate two natural categories of geometric objects (namely, the categories of perfectoid spaces and $\arc_\varpi$-analytic spaces) as full subcategories of $\cX_{\arc_\varpi}$, and establish a version of the Gerritzen-Grauert theorem for $\arc_\varpi$-sheaves (Theorem \ref{intro_mono_arc_topos}).

\section{Acknowledgments}
First and foremost, I would like to thank my advisor Mattias Jonsson for his support and encouragement. It would be difficult to overstate the positive impact that he has had throughout my years as a graduate student -- not only has his guidance made this thesis immeasurably better, but I have learned so much about life from him. I will always be grateful. I would also like to thank the rest of my doctoral committee for their time and interest in my work, as well as Bhargav Bhatt to whom this thesis owes a great deal of intellectual debt.

I am grateful to my friends at Michigan for enriching my mathematical and social experience, as well as my friends elsewhere for helping me remember everything that life has to offer. And finally, thank you to my family for their unconditional love and support throughout my whole life; and to my wife, Veronica, for being a consistent source of happiness as well as building a life with me we are both proud of.

\newpage

\chapter{Commutative Algebra}\label{chapt_comm_alg}

Throughout this chapter we fix a prime number $p$ and a perfectoid non-archimedean field $K$ together with an object $\varpi \in K$ satisfying $1 > |\varpi^p| \ge |p|$ and  a compatible system of $p$-power roots $\{\varpi^{1/p^n}\}_{n \in \ZZ_{\ge 0}}$. In Section \ref{sect_int_perfd_alg} we recall the definition of integral perfectoid algebras, a generalization of the original definition of perfectoid Banach $K$-algebra, and show stability of integral perfectoid algebras under passing to their $\varpi$-torsion free quotient and integral closures with respect to $\varpi$. In Section \ref{sect_almost_math} we develop the theory of almost mathematics with respect to the ideal $(\varpi^{1/p^\infty}) \subset K_{\le 1}$, taking as a starting point Lurie's derived $\infty$-category $\cD(K_{\le 1})$. Our motivation to do so is to be able to use $\arc_\varpi$-descent for integral perfectoids, established by Bhatt and Scholze \cite[Proposition 8.10]{prisms}, to prove Tate acyclicity for perfectoid Banach $K$-algebras (Theorem \ref{intro_tate_acyclicity_perfd}). In Section \ref{sect_int_closure} we show that the inclusion $\CAlg_{K_{\le 1}}^{\wedge a \tf} \subset \CAlg_{K_{\le 1}}$ of $\varpi$-torsion free $\varpi$-complete almost $K_{\le 1}$-algebras into all $K_{\le 1}$-algebras admits a left adjoint and describe this left adjoint explicitly. In Section \ref{sect_ban_alg} we establish the dictionary (Theorem \ref{intro_dictionary}), which says that there is an equivalence of categories $(-)_{\le 1}: \Ban_K^{\contr} \simeq \CAlg_{K_{\le 1}}^{\wedge a \tf}: [\frac{1}{\varpi}]$. Finally, in Section \ref{sect_perfd_ban_alg} we show that the dictionary induces an equivalence of categories between almost integral perfectoid algebras and perfectoid Banach $K$-algebras.

\section{Integral Perfectoid Algebras}\label{sect_int_perfd_alg}

\subsection{Definitions and basic properties}

\begin{defn}\label{tilting_defn} The tilting functor
\begin{equation*}
	(-)^{\flat}: \{ \text{Rings} \} \longrightarrow \{ \text{Perfect } \FF_p \text{-algebras} \}
\end{equation*}
is defined by
\begin{equation*}
	A^{\flat} := \lim_{a \mapsto a^p} A/p
\end{equation*}
\end{defn}

\begin{lemma}\label{p-Witt_unique_lift} The $p$-typical Witt vector functor
\begin{equation*}
	W(-): \{ \text{Perfect } \FF_p \text{-algebras} \} \longrightarrow \{ p \text{-adically complete } \ZZ_p \text{-algebras} \}
\end{equation*}
can be uniquely characterized as follows: for a perfect $\FF_p$-algebra $R$ there exists an essentially unique $p$-adically complete flat $\ZZ_p$-algebra $W(R)$, such that $W(R)/p = R$.
\end{lemma}

\begin{proof} \cite[Proposition 3.12]{phodge_according_Beilinson}
\end{proof}

\begin{lemma} Let $R$ be a ring, and $I \subset R$ a finitely generated ideal. Then any morphism $f: M_1 \rightarrow M_2$ of $I$-adically complete modules is $I$-adically continuous.
\end{lemma}

\begin{proof} By definition of $I$-adic completeness we have $M_i = \lim M_i/I^n M$, so the collection of open subsets $\{I^n M_i\} \subset M_i$ form a basis of open neighborhoods around $0 \in M_i$. To show that $f: M_1 \rightarrow M_2$ is $I$-adically continuous it suffices to show that for each $n \ge 0$ there exists an $m \ge 0$ such that $f(I^m M_1) \subset I^n M_2$. By the linearity of $f$ over $R$ it follows that the composition $M_1 \rightarrow M_2 \rightarrow M_2/I^n M$ factors as
\begin{cd}
	M_1 \ar[r, "f"] \ar[d] & M_2 \ar[d] \\
	M_1/I^n M_1 \ar[r] & M_2/I^n M_2
\end{cd}
showing that $f(I^n M_1) \subset I^n M_2$ for all $n \ge 0$.
\end{proof}

\begin{lemma} When restricted to $p$-complete algebras, the tilting functor admits a fully faithful left adjoint, given by the $p$-typical Witt vectors
\begin{equation*}
	W(-): \{ \text{Perfect } \FF_p \text{-algebras} \} \longrightarrow \{ p \text{-adically complete } \ZZ_p \text{-algebras} \}
\end{equation*}
The adjunction $W(-) \dashv (-)^{\flat}$ is specified as follows: let $R$ be a perfect $\FF_p$-algebra and $S$ a $p$-adically complete $\ZZ_p$-algebra, then for any map $f: W(R) \rightarrow S$ there exists a unique $\FF_p$-algebra map $g: R \rightarrow S^{\flat}$ making the following diagram commute
\begin{cd}
	W(R) \ar[rr, "f"] \ar[d, "\bmod p", swap] && S \ar[d, "\bmod p"] \\
	R \ar[r, "g"] & S^{\flat} \ar[r] & S/p
\end{cd}
Where $S^{\flat} \rightarrow S/p$ is the projection $S^{\flat} = \lim_{s \mapsto s^p} S/p \rightarrow S/p$, and $g$ is the canonical factorization of $f \bmod p: R \rightarrow S/p$ through $R \rightarrow S^{\flat}$.
\end{lemma}

\begin{proof} \cite[Proposition 3.12]{phodge_according_Beilinson}
\end{proof}

\begin{defn} The Fontaine functor
\begin{equation*}
	\A_{\inf} (-) : \{ p \text{-adically complete } \ZZ_p \text{-algebras} \} \longrightarrow 
	\{ p \text{-adically complete } \ZZ_p \text{-algebras} \}
\end{equation*}
is defined as the composition $\A_{\inf} (-) = W((-)^\flat)$. Hence, for every $p$-complete ring $S$ we obtain the counit of adjunction
\begin{equation*}
	\theta: \A_{\inf} (S) \rightarrow S
\end{equation*}
The counit of adjunction can be characterized as the unique map making the following diagram commute
\begin{cd}
	\A_{\inf} (S) \ar[r] \ar[d, "\bmod p", swap] & S \ar[d, "\bmod p"] \\
	S^{\flat} \ar[r] & S/p
\end{cd}
where the bottom map $S^{\flat} \rightarrow S/p$ is the projection $S^{\flat} = \lim_{s \mapsto s^p} S/p \rightarrow S/p$.
\end{defn}

\begin{defn}\label{sharp_map_definition} Let $A$ be a $\pi$-complete ring, for some element $\pi \in A$ dividing $p$. Then, by \cite[Lemma 3.2(i)]{BMS1} we know that the canonical maps of multiplicative monoids, 
\begin{equation*}
	\lim_{a \mapsto a^p} A \rightarrow \lim_{a \mapsto a^p} A/p \rightarrow \lim_{a \mapsto a^p} A/\pi
\end{equation*}	
are isomorphisms. By projection onto the last coordinate we obtain a map of multiplicative monoids
\begin{equation*}
	\sharp: A^{\flat} = \lim_{a \mapsto a^p} A \rightarrow A \qquad a \mapsto a^{\sharp}
\end{equation*}
which is called the sharp map.
\end{defn}

\begin{lemma} For a perfect $\FF_p$-algebra $R$, there exists a unique multiplicative section
\begin{equation*}
	[-]: R \rightarrow W(R)
\end{equation*}
of the projection map $\bmod p : W(R) \rightarrow R$. Furthermore, for every $f \in W(R)$ there exists a unique $p$-adic expansion
\begin{equation*}
	f = \sum_{i = 0}^{\infty} [a_i]p^i
\end{equation*}
called the Teichmuller expansion of $f$.
\end{lemma}

\begin{proof} \cite[Lecture 2 - Construction 3.6]{bhatt2018lectures}
\end{proof}

\begin{lemma}\label{theta_characterization} For any $p$-complete ring $S$, the counit map $\theta: \A_{\inf} (S) \rightarrow S$ satisfies $\theta([a]) = a^{\sharp}$. Moreover, the map $\theta$ agrees with the map $\theta_1: \A_{\inf} (S) \rightarrow S$ of \cite[Lemma 3.3]{BMS1}.
\end{lemma}

\begin{proof} \cite[Footnote page 8]{purityflat}
\end{proof}

\begin{defn}[Integral Perfectoid]\label{int_perf_def} A ring $S$ is integral perfectoid if
	\begin{enumerate}[(1)]
		\item The ring $S$ is $\pi$-complete for some element $\pi \in S$, such that $\pi^p$ divides $p$.
		\item The counit map $\theta: \A_{\inf}(S) \rightarrow S$ is surjective and its kernel is principal.
	\end{enumerate}
\end{defn}

\begin{example} Any perfect ring of characteristic $p$ is an integral perfectoid ring. Indeed, for any perfect $\FF_p$-algebra $R$ we have that $R$ is complete with respect to $0 \in R$ and the map $\theta: W(R) \rightarrow R$ is generated by $p$.
\end{example}

\begin{rem} Let $A$ be a ring, and $M$ a $J$-adically complete $A$-module for some ideal $J \subset A$. By \cite[Tag 090T]{stacks-project} we know that for any finitely generated ideal $I \subset J \subset A$ the module $M$ is also $I$-adically complete. Therefore, any perfectoid ring $R$ is $p$-adically complete.
\end{rem}

\begin{lemma} A ring $S$ is integral perfectoid if and only if it satisfies the conditions of \cite[Definition 3.5]{BMS1}, that is: $S$ is $\pi$-complete for some element $\pi \in S$ such that $\pi^p$ divides $p$, the Frobenius map $\varphi: S/p \rightarrow S/p$ is surjective, and the kernel $\theta: \A_{\inf} (S) \rightarrow S$ is principal.
\end{lemma}

\begin{proof} If $S$ is integral perfectoid, it suffices to show that $\varphi: S/p \rightarrow S/p$ is surjective. Indeed, by hypothesis we know that the map $\A_{\inf} (S) \rightarrow S$ is surjective and so its reduction mod $p$, the map $S^{\flat} \rightarrow S/p$, is also surjective. Then, the perfectness of $S^{\flat}$ implies that Frobenius $\varphi: S/p \rightarrow S/p$ is also surjective. Conversely, if $S$ satisfies the conditions of \cite[Definition 3.5]{BMS1}, it suffices to show that the map $\theta: \A_{\inf} (S) \rightarrow S$ is surjective, but this follows from \cite[Lemma 3.9(v)]{BMS1}.
\end{proof}

\begin{lemma}\label{compatible_roots_perfectoid} Let $S$ be an integral perfectoid ring which is $\pi$-complete with respect to an element $\pi \in S$ such that $\pi^p$ divides $p$. Then there exist $u,v \in S^{\times}$ such that $u\pi$ and $vp$ admit systems of $p$-power roots. 
\end{lemma}

\begin{proof} \cite[Lemma 3.9]{BMS1}
\end{proof}

%%%%%%%%%%%%%%%%%%%
\subsection{\texorpdfstring{$\delta$}{delta}-rings and perfect prisms}

\begin{defn} A $\delta$-ring is a pair $(R, \delta)$ where $R$ is a commutative $\ZZ_{(p)}$-algebra and $
	\delta: R \rightarrow R$ is a map of sets with $\delta(0) = \delta(1) = 0$, satisfying the following two identities
\begin{align*}
	\delta(xy) =& x^p \delta(y) + y^p \delta(x) + p \delta(x) \delta(y)  \\
	\delta(x + y) =& \delta(x) + \delta(y) + \frac{x^p + y^p - (x+y)^p}{p} =
	 \delta(x) + \delta(y) - (p-1)! \sum_{i = 1}^{p-1} \frac{x^i}{i!} \frac{y^{p-i}}{(p-i)!}
\end{align*}
A morphism of $\delta$-rings $(R, \delta_R) \rightarrow (S, \delta_S)$ consists of a morphism of commutative rings $f:R \rightarrow S$ such that the following diagram of sets commutes
\begin{cd}
	R \ar[r, "\delta_R"] \ar[d, "f", swap] & R \ar[d, "f"] \\
	S \ar[r, "\delta_S"] & S
\end{cd}
\end{defn}

\begin{rem}[$\delta$-structures and Frobenius lifts] For a given $\delta$-ring $(R, \delta)$, we write $\varphi: R \rightarrow R$ for the map defined by $\varphi(x) = x^p + p \delta(x)$; the identities on $\delta$ ensure that $\varphi: R \rightarrow R$ is a ring homomorphism making the following diagram commute
\begin{cd}
	R \ar[r, "\varphi"] \ar[d, "\bmod p", swap] & R \ar[d, "\bmod p"] \\
	R/p \ar[r, "\text{Frob}"] & R/p
\end{cd}
In other words, $\varphi: R \rightarrow R$ lifts Frobenius on $R/p$.

In this paper we will only make use of delta rings $(R, \delta)$ where the underlying ring $R$ is $p$-torsion free. In this situation, any lift $\varphi: R \rightarrow R$ of the Frobenius on $R/p$ comes from a unique $\delta$-structure on $R$; given by the formula $\delta(x) = \frac{\varphi(x) - x^p}{p}$. In other words, if $R$ is $p$-torsion free then a specifying a $\delta$-structure on $R$ is the same as specifying a morphism of rings $\varphi: R \rightarrow R$ that lifts Frobenius on $R/p$.
\end{rem}

\begin{prop}\label{delta_rings_lim_colim} The category of $\delta$-rings admits all limits and colimits, and the forgetful functor
\begin{equation*}
	\{ \delta \text{-rings} \} \longrightarrow \{ \text{Rings} \}
\end{equation*}
preserves limits and colimits.
\end{prop}

\begin{proof} \cite[Remark 2.7]{prisms}
\end{proof}

\begin{lemma}[Completions] Let $A$ be a $\delta$-ring, and $I \subset A$ be a finitely generated ideal containing $p$. Then, the map $\delta: A \rightarrow A$ is $I$-adically continuous: more precisely, for each $n$ there is some $m$ such that for all $x \in A$, one has $\delta(x + I^m) \subset \delta(x) + I^n$.

Moreover, the $I$-adic completion of $A$, denoted by $A^\wedge_{I}$, acquires a unique $\delta$-structure making the following diagram of sets commute
\begin{cd}
	A \ar[r, "\delta"] \ar[d] & A \ar[d] \\
	A^{\wedge}_I \ar[r, "\delta"] & A^{\wedge}_I
\end{cd}
\end{lemma}

\begin{proof} \cite[Lemma 2.17]{prisms}
\end{proof}

\begin{defn} An element $d$ of a $\delta$-ring $A$ is distinguished if $\delta(d)$ is a unit.
\end{defn}

\begin{defn} A $\delta$-ring $A$ is perfect if $\varphi: A \rightarrow A$ is an isomorphism.
\end{defn}

\begin{prop}[Perfect $\delta$-rings]\label{perfect_delta_rings_equiv} The following categories are equivalent:
\begin{enumerate}[(1)]
	\item The category of perfect $p$-complete $\delta$-rings.
	\item The category of $p$-adically complete and $p$-torsion free rings $A$ with $A/p$ being perfect.
	\item The category of perfect $\FF_p$-algebras.
\end{enumerate}
The functor relating $(1)$ and $(2)$ is the forgetful functor; in particular, we learn a posteriori that any ring homomorphism between two perfect $p$-complete $\delta$-rings is automatically a map of $\delta$-rings. The functor relating $(2)$ and $(3)$ are $A \mapsto A/p$ and $R \mapsto W(R)$; in particular, there is a unique $\delta$-structure on $W(R)$ for $R$ perfect of characteristic $p$.
\end{prop}

\begin{proof} \cite[Corollary 2.31]{prisms}
\end{proof}

\begin{example} The category of perfect $p$-complete $\delta$-rings admits an initial object given by $W(\FF_p) = \ZZ_p$. Since Frobenius is the identity map on $\FF_p$, the unique morphism of rings $\varphi: \ZZ_p \rightarrow \ZZ_p$ which lifts Frobenius on $\FF_p$ is the identity. This in turn completely characterizes the $\delta$-structure on $\ZZ_p$, which is given by
\begin{equation*}
	\delta(x) = \frac{x - x^p}{p} \in \ZZ_p
\end{equation*}
Specializing to the case where $x = p$, we learn that $\delta(p) = 1 - p^{p-1} \in \ZZ_p^{\times}$, which shows that $p \in \ZZ_p$ is a distinguished element.
\end{example}

\begin{lemma}[Perfect elements have rank $1$]\label{frob_perfect_element} Fix a $p$-adically complete $\delta$-ring $A$ and some $x \in A$ admitting a $p^n$-th root for all $n \ge 0$, then $\delta(x) = 0$.
\end{lemma}

\begin{proof} \cite[Lemma 2.32]{prisms}
\end{proof}

\begin{example}\label{teichmuller_expansion} Let $A$ be a perfect $p$-complete $\delta$-ring, let us show how to explicitly describe the Frobenius lift $\varphi: A \rightarrow A$ and its $\delta$-structure. Since $A$ is a $p$-complete perfect $\delta$-ring, there exists a unique perfect $\FF_p$-algebra $R$, such that $W(R) = A$. In particular, for any $f \in A$ we can write its Teichmuller expansion
\begin{equation*}
	f = \sum_{i = 0}^{\infty} [a_i] p^i
\end{equation*}
Thus, by the $p$-adic continuity of $\delta: A \rightarrow A$, it suffices to describe the values $\delta(p)$ and $\delta([a_i])$ for all $a_i \in R$. Since $\ZZ_p$ is the initial perfect $p$-complete $\delta$-ring, there exists a structure map $\ZZ_p \rightarrow A$ which forces the identity $\delta(p) = 1 - p^{p-1} \in A^{\times}$. On the other hand, using the multiplicativity of the map $[-]: R \rightarrow W(R)$ it follows that $[a_i] \in W(R)$ admits a $p^n$-root for all $n \ge 0$, as $R$ is perfect, which implies that $\delta([a_i]) = 0$.

Similarly, by the $p$-adic continuity of $\varphi: A \rightarrow A$, it suffices to describe the values $\varphi(p)$ and $\varphi([a_i])$ for all $a_i \in R$. Again, using the structure map $\ZZ_p \rightarrow A$ it follows that $\varphi(p) = p$, and from the identity $\varphi([a_i])= [a_i]^p + p \delta([a_i])$ we learn that $\varphi([a_i]) = [a_i]^p = [a_i^p]$, since $[-]: R \rightarrow W(R)$ is multiplicative and $\delta([a_i]) = 0$. Therefore, we can express the Teichmuller expansion of $\varphi(f)$ as follows
\begin{equation*}
	\varphi(f) = \sum_{i = 0}^{\infty} [a_i^p]p^i
\end{equation*}
\end{example}

\begin{lemma}[Distinguished elements in perfect $\delta$-rings] Let $A$ be a perfect $p$-complete $\delta$-ring, and fix $d \in A$. Denote by $R$ the unique perfect $\FF_p$-algebra $R$ such that $W(R) = A$, and $d = \sum_{i = 0}^{\infty} [a_i]p^i$ the corresponding Teichmuller expansion of $d$. Then, $d$ is distinguished if and only if $a_1 \in R^{\times}$.
\end{lemma}

\begin{proof} \cite[Lemma 2.33]{prisms}
\end{proof}

\begin{lemma}\label{delta_ring_p_torsion} Let $A$ be a perfect $p$-complete $\delta$-ring. Fix a distinguished element $d \in A$. Then
	\begin{enumerate}[(1)]
		\item The element $d \in A$ is a non-zero divisor.
		\item The ring $R = A/d$ has bounded $p^\infty$-torsion; in fact, we have $R[p] = R[p^\infty]$.
	\end{enumerate}
\end{lemma}

\begin{proof} \cite[Lemma 2.34]{prisms}
\end{proof}

\begin{defn}[Perfect Prisms]\label{perfect_prisms_def} Fix a pair $(A, (d))$ comprising of a perfect $\delta$-ring $A$ and an ideal $(d) \subset A$ generated by a distinguished element $d \in A$. We say that $(A, (d))$ is a perfect prism if $A$ is $(p,d)$-complete. A morphism $(A, (d_1)) \rightarrow (B, (d_2))$ of perfect prisms is a map $f: A \rightarrow B$ of perfect $\delta$-rings, such that $f(d_1) \in (d_2)$.
\end{defn}

\begin{lemma}[Rigidity of maps]\label{rig_prisms} If $(A, (d_1)) \rightarrow (B, (d_2))$ is a map of perfect prisms, then the natural map of $A$-modules $(d_1) \rightarrow (d_2)$ induces an isomorphism $(d_1) \otimes_A B \simeq (d_2)$. In particular $d_1 B = (d_2)$.

Conversely, if $(A, (d))$ is a perfect prism, and $A \rightarrow B$ a map of perfect $\delta$-rings, with $B$ being $(p,d)$-complete, then $(B, (d))$ is also a perfect prism.
\end{lemma}

\begin{proof} The fact that the canonical map $(d_1) \rightarrow (d_2)$ of $A$-modules induces an isomorphism $(d_1) \otimes_A B \simeq (d_2)$ follows from \cite[Lemma 3.5]{prisms}. For the converse, since $d \in A$ is distinguished and the map $A \rightarrow B$ is a map of perfect $\delta$-rings, it follows that $d \in B$ is also distinguished. By virtue of $B$ being $p$-complete it follows that $d \in B$ is a non-zero divisor, and then the result follows from \cite[Lemma 3.5]{prisms}.
\end{proof}

\begin{lemma} The following categories are equivalent
\begin{enumerate}[(1)]
	\item The category of perfect prisms, in the sense of Definition \ref{perfect_prisms_def}.
	\item The category of pairs $(A, I)$ comprising of a perfect $\delta$-ring $A$ and an ideal $I \subset A$; such that, $I \subset A$ is a Cartier divisor on $\Spec (A)$, the ring $A$ is derived $(p,I)$-complete, and $p \in I + \varphi(I)A$. A morphism $(A,I) \rightarrow (B,J)$ in this category is a map $f: A \rightarrow B$ of perfect $\delta$-rings such that $f(I) \subset J$.
\end{enumerate}
\end{lemma}

\begin{proof} If $(A,(d))$ is a perfect prism, then the hypothesis that $A$ is $(p,d)$-complete implies that $A$ is derived $(p,d)$-complete \cite[Lemma 091R]{stacks-project}. Moreover, since $d \in A$ is distinguished and $A$ is $(p,d)$-complete it follows that $p \in (d, \varphi(d))$ by \cite[Lemma 2.25]{prisms}. On the other hand, if $(A,I)$ is a pair as in $(2)$, then the ideal $I$ is principal and any generator is a distinguished element; and the ring $A$ is classically $(p,I)$-complete \cite[Lemma 3.8]{prisms}.
\end{proof}

\begin{prop}[Perfectoid rings = perfect prisms]\label{perfect_prism_perfectoid} The following two categories are equivalent
\begin{enumerate}[(1)]
	\item The category of integral perfectoid rings $R$.
	\item The category of perfect prisms $(A, (d))$.
\end{enumerate}
The functors are $R \mapsto (\A_{\inf} (R), \ker(\theta))$ and $(A,(d)) \mapsto A/d$ respectively.
\end{prop}

\begin{proof} \cite[Theorem 3.10]{prisms}
\end{proof}

\begin{rem} Let $R$ be a integral perfectoid ring, let us explain how to obtain an element $\pi \in R$, such that $R$ is $\pi$-complete and $\pi^p$ divides $p$, from the perfect prism $(A,d) = (\A_{\inf}(R), \ker(\theta))$. Via the isomorphism $A \simeq W(R^\flat)$ we can express $d = [a_0] + pu$ for a unit $u \in A$. Letting $\pi \in R$ be $\theta([a_0^{1/p}]) \in R$ it follows that $\pi^p = -p u$ in $R$. Under this choice of $\pi$ we learn that $\pi \in R$ admits compatible $p$-power roots in $R$, and so does $p \in R$ up to a unit multiple. Then, by virtue of \cite[Tag 0319]{stacks-project} we learn that $R$ is $p$-complete if and only if it is $\pi$-complete.

However, if we have an integral perfectoid ring $R$ and an element $\varpi \in R$ such that $R$ is $\varpi$-complete and $\varpi^p$ divides $p$ in $R$, it need not be true that $\varpi$ is equal to the element $\pi \in R$ produced in the above paragraph.
\end{rem}

\subsection{Tilting correspondence}

Let $R$ be a integral perfectoid ring which is $\varpi$-complete with respect to some $\varpi \in R$ where $\varpi^p$ divides $p$. Recall from Lemma \ref{compatible_roots_perfectoid} that there are (non-canonical) $\varpi^\flat, p^\flat \in R^\flat$ such that $(\varpi^\flat)^\sharp, (p^\flat)^\sharp \in R$ are unit multiples of $\varpi, p$ respectively.

\begin{lemma}\label{mod_p_tilt} Let $R$ be an integral perfectoid ring which is $\varpi$-complete with respect to some $\varpi \in R$ where $\varpi^p$ divides $p$. Then, the sharp map $\sharp: R^\flat \rightarrow R$ induces the following ring isomorphisms
\begin{equation*}
	R^\flat/p^\flat \rightarrow R/p \qquad R^\flat/(\varpi^\flat)^p \rightarrow R/\varpi^p
\end{equation*}
\end{lemma}

\begin{proof} Follows from the proof of \cite[Lemma 3.10]{BMS1}, and the fact that the map $\theta: \A_{\inf}(R) \rightarrow R$ is surjective with principal kernel (cf. \cite[2.1.2.2]{purityflat}).
\end{proof}

\begin{lemma}\label{perfectoid_frob_mod_p} Let $R$ be an integral perfectoid ring which is $\varpi$-complete with respect to some $\varpi \in R$ where $\varpi^p$ divides $p$. Then, the $p$-power map $a \mapsto a^p$ induces an isomorphism of rings $R/\varpi \rightarrow R/\varpi^p$.
\end{lemma}

\begin{proof} \cite[Lemma 3.10]{BMS1}
\end{proof}

\begin{lemma}\label{completeness_tilt} Let $R$ be an integral perfectoid ring which is $\varpi$-complete with respect to some $\varpi \in R$ where $\varpi^p$ divides $p$. Then, the tilt $R^\flat$ is $\varpi^\flat$-complete.
\end{lemma}

\begin{proof} It suffices to show that the canonical map $R^\flat \rightarrow \lim R^\flat/(\varpi^\flat)^{p^n}$ is an isomorphism, in what follows we will use $\varphi$ to denote the Frobenius morphism. Notice that we have an isomorphism $\varphi^{n}: R^\flat/\varpi^\flat \rightarrow R^\flat/(\varpi^\flat)^{p^n}$ from Lemma \ref{perfectoid_frob_mod_p}, which we can precompose with the isomorphism $\sharp^{-1}: R/\varpi \rightarrow R^\flat/\varpi^\flat$ from Lemma \ref{mod_p_tilt} to get an isomorphism $g_n: R/\varpi \rightarrow  R^\flat/(\varpi^\flat)^{p^n}$ of rings.

Next, consider the following commutative diagram
	\begin{cd}
		\cdots \ar[r, "\varphi"] & R/\varpi \ar[r, "\varphi"] \ar[d, "g_n"]& R/\varpi \ar[r, "\varphi"] \ar[d, "g_{n-1}"]& \cdots \ar[r, "\varphi"] & R/\varpi \ar[d, "g_0"] \\
		\cdots \ar[r] & R^\flat/(\varpi^\flat)^{p^n} \ar[r] & R^\flat/(\varpi^\flat)^{p^{n-1}} \ar[r] & \cdots \ar[r] & R^\flat/(\varpi^\flat)
	\end{cd}
Since all the vertical maps are isomorphisms we obtain the desired isomorphism $R^\flat \rightarrow \lim R^\flat/(\varpi^\flat)^{p^n}$, proving the claim.
\end{proof}

\begin{prop}\label{prism_tilt_correspondence} Let $R$ be a integral perfectoid ring, with tilt $R^\flat$. Then, the following categories are equivalent
\begin{enumerate}[(1)]
	\item The category of integral perfectoid rings over $R$.
	\item The category of integral perfectoid rings over $R^\flat$.
\end{enumerate}
Where the functor from $(1)$ to $(2)$ is given by the tilting functor. Furthermore, if $R$ is $\varpi$-complete with respect to some $\varpi \in R$ where $\varpi^p$ divides $p$, then the above equivalence restricts to an equivalence between the categories
\begin{enumerate}[(1')]
	\item The category of $\varpi$-complete integral perfectoid rings over $R$.
	\item The category of $\varpi^{\flat}$-complete integral perfectoid rings over $R^\flat$.
\end{enumerate}
\end{prop}

\begin{proof} Let $(A, d)$ be the perfect prism corresponding to $R$, then $A/p \simeq R^\flat$ as $A = W(R^\flat)$. Under the equivalence of Proposition \ref{perfect_prism_perfectoid} the category $(1)$ is equivalent to the category of perfect prisms over $(A,d)$, and the category $(2)$ is equivalent to the category of perfect prisms over $(A, p)$. Furthermore, by Lemma \ref{rig_prisms} we know that the category of perfect prisms over $(A,d)$ is equivalent to the category of perfect $\delta$-rings over $A$, and the category of perfect prisms over $(A,p)$ is equivalent to the category of perfect $\delta$-rings over $A$. This provides the desired equivalence of categories.

Finally, let us identify the equivalence of categories we just describes with the tilting functor when going from $(1)$ to $(2)$. Tracing out the equivalence we see that to a integral perfectoid $R$-algebra $S$ we first associate the perfect $\delta$-ring $W(S^\flat) = \A_{\inf}(S)$ and then we mod out by $p$, giving us the functor $S \mapsto S^\flat$. The equivalence between the categories (1') and (2') then follows from Lemma \ref{completeness_tilt}.
\end{proof}

\begin{prop}\label{product_perfectoid}  Let $R$ be an integral perfectoid ring which is $\varpi$-complete with respect to some $\varpi \in R$ where $\varpi^p$ divides $p$. Let $I$ be a set and $\{B_i\}_{i \in I}$ a collection of (classically) $\varpi$-complete $R$-algebras. Then, $\prod_{i \in I} B_i$ is a $\varpi$-complete integral perfectoid if and only if each $B_i$ is so, and then $(\prod_{i \in I} B_i)^{\flat} = \prod_{i \in I} B_i^{\flat}$.
\end{prop}

\begin{proof} It follows from the definition of $\A_{\inf}(-)$ that it preserves all limits, as $(-)^{\flat}$ and $W(-)$ do so. Hence we have the identity $\A_{\inf}(\prod_{i \in I} B_i) = \prod_{i \in I} \A_{\inf}(B_i)$, proving the result. See also \cite[Proposition 2.1.11(d)]{purityflat}
\end{proof}

%%%%%%%%%%%%%%

\subsection{\texorpdfstring{$\varpi$}{pi}-torsion in integral perfectoid rings}

Let $R$ be a ring, recall that for an element $a \in R$ we denote by $R[a]$ the $a$-torsion of $R$ and by $R[a^\infty]$ the union $\cup_{n \ge 0} R[a^n]$. Similarly, if $a \in R$ admits systems of compatible $p$-power roots, we denote by $R[a^{1/p^\infty}]$ the intersection $\cap_{n \ge 0} R[a^{1/p^n}]$.

\begin{prop}\label{perfectoid_p_torsion} Let $R$ be an integral perfectoid ring, then the sharp map $R^\flat \rightarrow R$ induces an isomorphism $R^\flat[p^\flat] \rightarrow R[p]$ of $\A_{\inf}(R)$-modules.
\end{prop}

\begin{proof} \cite[2.1.2.5]{purityflat}.
\end{proof}

In fact, by Lemma \ref{delta_ring_p_torsion} we know that $R[p] = R[p^\infty]$ and since $R^\flat$ is a perfect ring of characteristic $p$ we have that $R^\flat[p^\flat] = R^\flat[p^{\flat, \infty}]$. In particular, for any $a \in R$ which divides $p$ and admits systems of $p$-power roots, we have that $R[a^n] \subset R[p^\infty] = R[p]$; which in turn implies that the sharp map $R^\flat[a^{\flat, n}] \rightarrow R[a^n]$ is an $\A_{\inf} (R)$-module map for all $n \in \ZZ_{\ge 0}$.

\begin{corollary}\label{perfectoid_torsion} Let $R$ be an integral perfectoid ring which is $\varpi$-complete with respect to some $\varpi \in R$ where $\varpi^p$ divides $p$. Then, the sharp map $\sharp: R^\flat \rightarrow R$ induces an isomorphism $R^\flat[\varpi^{\flat,n}] \rightarrow R[\varpi^n]$ of $\A_{\inf}(R)$-modules, for all $n \in \mathbf{Z}[1/p]$. Moreover, the canonical inclusions induce isomorphisms $R[\varpi^{1/p^\infty}] = R[\varpi] = R[\varpi^\infty]$.
\end{corollary}

\begin{proof} From Proposition \ref{perfectoid_p_torsion} we learn that the sharp map induces an isomorphism $R^\flat[p^\flat] \rightarrow R[p]$ of $\A_{\inf}$-modules. Moreover, since $R^\flat[p^{\flat,\infty}] = R[p^\flat]$ and $R[p^\infty] = R[p]$ we learn that by restricting to the submodules $R^\flat[\varpi^{\flat, n}] \subset R^\flat[p^\flat]$ and $R[\varpi^{n}] \subset R[p]$ we obtain a $\A_{\inf} (R)$-module isomorphism $R^\flat[\varpi^{\flat,n}] \rightarrow R[\varpi^n]$ for all $n \in \mathbf{Z}[1/p]$. For the second part it suffices to show that $R^\flat[\varpi^{\flat, 1/p^\infty}] = R^\flat[\varpi^\flat] = R^\flat[\varpi^{\flat, \infty}]$, but this is clear as $R^\flat$ is a perfect ring of characteristic $p$.
\end{proof}

\begin{prop}\label{pi_torsionfree_perfectoid} Let $R$ be an integral perfectoid ring which is $\varpi$-complete with respect to some $\varpi \in R$ where $\varpi^p$ divides $p$. Denote by $\overline{R}$ the ring $R/R[\varpi^\infty]$ and by $\overline{R^\flat}$ the ring $R^\flat/R^\flat[\varpi^{\flat, \infty}]$. Then, $\overline{R}$ is a $\varpi$-torsion free integral perfectoid ring, with tilt $\overline{R^\flat}$. Moreover, $\overline{R}$ is $\varpi$-complete.
\end{prop}

\begin{proof} The fact that $\overline{R}$ is integral perfectoid with tilt $\overline{R^\flat}$ follows from \cite[2.1.3]{purityflat}. To show that $\overline{R}$ is (classically) $\varpi$-complete it suffices to show that $R[\varpi^\infty]$ is derived $\varpi$-complete (\ref{derived_complete_cat}), but this follows from the identity $R[\varpi^\infty] = R[\varpi^{1/p^\infty}]$.
\end{proof}

\subsection{\texorpdfstring{$p$}{p}-integral closedness of integral perfectoid rings}

\begin{defn}\label{p_int_closed_defn} For an injective morphisms $R \hookrightarrow S$ of rings, we say that $R$ is $p$-integrally closed in $S$ if every $a \in S$ where $a^p \in R$ satisfies $a \in R$. The $p$-integral closure of $R$ in $S$ is constructed as $\cup _{n \ge 0} R_n \subset S$, where $R_0 = R$ and $R_{n+1} \subset S$ is the $R_n$-subalgebra generated by all the $a \in S$ such that $a^p \in R_n$; it is the smallest $p$-integrally closed subring of $S$ containing $R$. Clearly, the $p$-integral closure of $R \hookrightarrow S$ is contained in the integral closure.
\end{defn}

\begin{lemma}\label{p_int_closed_frob} Let $R$ be a ring and $\varpi \in R$ a non-zero divisor which satisfies $\varpi^p$ divides $p$. Then, the map
	\begin{equation*}
		\varphi: R/\varpi \rightarrow R/\varpi^p \qquad a \mapsto a^p
	\end{equation*}
is an is injective if and only if $R \subset R[\frac{1}{\varpi}]$ is $p$-integrally closed.
\end{lemma}

\begin{proof} Assume that the map $\varphi: R/\varpi \rightarrow R/\varpi^p$ is injective, and let $a \in R[\frac{1}{\varpi}]$ be an element which satisfies $a^p \in R$. Then, there exists a $n \ge 0$ such that $\varpi^n a \in R$; since we need to show that $a \in R$ we may assume that $n > 0$. We claim that $\varpi^{n-1} a \in R$, which implies $a \in R$ by induction. Indeed, since $a^p \in R$ we have that $\varpi^{pn} a^p \in R$ is sent to zero under the quotient map $R \rightarrow R/\varpi^p$, which in turn implies that $\varpi^n a \in R$ is sent to zero under the map $R \rightarrow R/\varpi$, by injectivity of $\varphi$. Hence, we can conclude that $\varpi^n a \in \varpi R$, and since $\varpi \in R$ is a non-zero divisor it follows that $\varpi^{n-1} a \in R$.
	
Conversely, assume that $R \subset R[\frac{1}{\varpi}]$ is $p$-integrally closed, we need to show that $\varphi: R/\varpi \rightarrow R/\varpi^p$ is injective. For the sake of contradiction, assume that we have a non-zero $a \in R/\varpi$ such that $0 = a^p \in R/\varpi^p$. Let $\tilde{a} \in R$ be a lift of $a \in R/\varpi$ along the quotient map $R \rightarrow R/\varpi$, the assumption that $0 = a^p \in R/\varpi^p$ implies that $\tilde{a}^p \in \varpi^p R$. From the fact that $\varpi \in R$ is a non-zero divisor we can conclude that there is a unique element $\frac{\tilde{a}^p}{\varpi^p} \in R$, which in turn implies that there is a unique element $\frac{\tilde{a}}{\varpi} \in R$ as $R \subset R[\frac{1}{\varpi}]$ is $p$-root closed and $\frac{\tilde{a}}{\varpi} \in R[\frac{1}{\varpi}]$. Thus, we obtain that $\tilde{a} \in \varpi R$, which contradicts the assumption that $a \in R/\varpi$ is non-zero. 
\end{proof}

\begin{corollary}\label{perfectoid_p_int_closed} Let $R$ be a integral perfectoid ring which is $\varpi$-complete with respect to some $\varpi \in R$ where $\varpi^p$ divides $p$, and denote by $\overline{R} := R/R[\varpi^\infty]$ the $\varpi$-torsion free quotient of $R$. Then, $\overline{R} \subset R[\frac{1}{\varpi}]$ is $p$-integrally closed.
\end{corollary}

\begin{proof} By Proposition \ref{pi_torsionfree_perfectoid} we know that $\overline{R}$ is again a integral perfectoid ring, from which we can deduce that the Frobenius map $\varphi: \overline{R}/\varpi \rightarrow \overline{R}/\varpi^p$ is an isomorphism by Lemma \ref{perfectoid_frob_mod_p}. Thus, the identity $\overline{R}[\frac{1}{\varpi}] = R[\frac{1}{\varpi}]$ implies the result by virtue of Lemma \ref{p_int_closed_frob}.
\end{proof}

\newpage

\section{Almost Mathematics}\label{sect_almost_math}

\subsection{Derived completeness}

For any ring $R$ we will denote by $\Mod_R$ the abelian category of $R$-modules, which comes equipped with a symmetric monoidal structure $(\Mod_R, \otimes_R)$ given by the $R$-linear tensor product of modules. For our purposes it would be convenient to work with $R$-modules which are $I$-complete \cite[Tag 0317]{stacks-project} with respect to some ideal $I \subset R$. However, since the category of $I$-complete $R$-modules is not abelian in general \cite[Tag 07JQ]{stacks-project}, we are forced to take a more sophisticated point of view.

\begin{const}\label{derived_complete_cat} In what follows we will make use of the theory of stable $\infty$-categories \cite[Chapter 1]{lurieHA}, especially the $\infty$-categorical enhancement of the derived category of $\Mod_R$, which we denote by $\cD(R)$ and refer the reader to \cite[Section 1.3.5]{lurieHA} for a definition. The category $\cD(R)$ comes equipped with a $t$-structure $(\cD(R)^{\le 0}, \cD(R)^{\ge 0})$ and the canonical inclusion functor $\Mod_R \hookrightarrow \cD(R)$ identifies $\Mod_R$ with the heart of the $t$-structure on $\cD(R)$ (cf. \cite[Proposition 7.1.1.13]{lurieHA}). We will say that $M \in \cD(R)$ is an $R$-module if it is in the essential image of the canonical fully faithful embedding $\Mod_R \hookrightarrow \cD(R)$, and we say that $M$ is an $R$-complex otherwise.

The stable $\infty$-category $\cD(R)$ comes equipped with a tensor product $\otimes_R^L$ which commutes with colimits independently on each variable, endowing $\cD(R)$ with the structure of a symmetric monoidal $\infty$-category. Since we will be constructing tensor products on various categories, let us explain how to extract the (classical) tensor product on $\Mod_R$ from that on $\cD(R)$. First, notice that the connective objects $\cD(R)^{\le 0} \subset \cD(R)$ are stable under the tensor product $\otimes_R^L$ endowing $\cD(R)^{\le 0}$ with a symmetric monoidal structure, given by $\otimes_R^L$. Now, since the canonical inclusion $\Mod_R \hookrightarrow \cD(R)^{\le 0}$ admits a left adjoint
\begin{equation*}
	\tau^{\ge 0}: \cD(R)^{\le 0} \longrightarrow \Mod_R = \cD(R)^\heartsuit
\end{equation*}
which satisfies the following property: if a morphism $M_1 \rightarrow M_2$ is an isomorphism after applying $\tau^{\ge 0}$ then for any other $N \in \cD(R)^{\le 0}$ the canonical map $M_1 \otimes_R^L N \rightarrow M_2 \otimes_R^L N$ is an isomorphism after applying $\tau^{\le 0}$ \cite[Proposition 2.2.1.8]{lurieHA}. Then, we can endow $\Mod_R$ with an essentially unique symmetric monoidal structure such that the truncation map $\tau^{\ge 0}: \cD(R)^{\le 0} \rightarrow \Mod_R = \cD(R)^\heartsuit$ is symmetric monoidal \cite[Proposition 2.2.1.9]{lurieHA}. In particular, the induced monoidal structure on $\Mod_R$ is given by $H^0(- \otimes_R^L -)$ which is exactly the classical tensor product on $\Mod_R$.

Passing to commutative algebra objects, in the sense of \cite[Section 2.1.3]{lurieHA}, this induces a pair of adjoint functors
\begin{equation*}
	\tau^{\ge 0}: \CAlg(\cD(R)^{\le 0}) \rightarrow \CAlg(\Mod_R) \qquad \CAlg(\Mod_R) \hookrightarrow \CAlg(\cD(R)^{\le 0})
\end{equation*}
with $\tau^{\ge 0}$ being a left adjoint to the canonical inclusion.
\end{const}

\begin{defn} To any object $x \in R$ we can associate the following two full subcategories of $\cD(R)$:
	\begin{enumerate}[(1)]
		\item A complex $M \in \cD(R)$ is said to be $x$-local if the canonical multiplication by $x$ map $x: M \rightarrow M$ is an isomorphism. The full subcategory of $\cD(R)$ spanned by $x$-local objects is denoted by $\cD(R)[x^{-1}]$.
		\item A complex $M \in \cD(R)$ is said to be derived $x$-complete if the canonical map
		\begin{equation*}
			M \rightarrow R\lim_n \Cone(x^n: M \rightarrow M)
		\end{equation*}
		is an isomorphism, where $x^n: M \rightarrow M$ is the multiplication by $x^n$ map. The full subcategory of $\cD(R)$ spanned by derived $x$-complete objects is denoted by $\cD(R)^\wedge_x$.
	\end{enumerate}
	For a comparison between our definition of derived $x$-complete complexes and that of \cite[Section 7.3.1]{luriespectral}, we refer the reader to \cite[Tag 091N]{stacks-project}. More generally, for a finitely generated ideal $I \subset R$ we say that an $R$-complex is derived $I$-complete if it is complete with respect to all $x \in I$ \cite[Corollary 7.3.3.3]{luriespectral}.
\end{defn}

\begin{notation} We will use the terminology ``classically $I$-complete'' to refer to the notion discussed in \cite[Tag 0317]{stacks-project}. We use the adjective ``classically'' to distinguish it from the derived notions we just introduced. Sometimes we will abuse language and use the term $I$-complete instead of ``classically $I$-complete'', as we did in the previous section.
\end{notation}

Given an $R$-module $M$ and a finitely generated ideal $I \subset R$, one may ask how the notions of being classically $I$-complete and derived $I$-complete compare -- in general they are distinct. An $R$-module $M$ is classically $I$-complete if and only if $M$ is derived $I$-adically complete and $\cap I^n M = 0$ \cite[Tag 091T]{stacks-project}. In particular, this implies that the derived $I$-completion functor is not a derived functor of the classical $I$-completion. For a more detailed treatment of derived $I$-complete modules we refer the reader to \cite[Tag 091N]{stacks-project} or to \cite[Section 7.3]{luriespectral}.

\begin{rem}\label{perfectoid_bounded_pi_torsion} Given a ring $R$ and an element $\varpi \in R$, if $R$ has bounded $\varpi$-torsion then we have that the classical $\varpi$-completion of $R$ and the derived $\varpi$-completion of $R$ agree \cite[Tag 0BKF]{stacks-project}. Specializing to the main case of interest to us, that of integral perfectoid rings: let $(A,d)$ be a derived $(p,d)$-complete perfect prism, then a priori we only know that $R := A/d$ is derived $p$-complete, but since $R[p] = R[p^\infty]$ (by Lemma \ref{delta_ring_p_torsion}) it follows that $R$ is classically $p$-adically complete. Now, let $\varpi \in R$ be an element that admits compatible $p$-power roots $\varpi^{1/p^n} \in R$, since $R$ has bounded $\varpi$-torsion \cite[2.1.3.2]{purityflat} it follows that $R$ is derived $\varpi$-complete if and only if $R$ is classically $\varpi$-complete.
\end{rem}

\begin{const}\label{const_complete_monoidal_str} From \cite[Proposition 7.3.1.4]{luriespectral} we conclude that the canonical inclusion functor
\begin{equation*}
	i_* : \cD(R)^\wedge_x \longrightarrow \cD(R)
\end{equation*}
admits a left adjoint
\begin{equation*}
	(-)^{\wedge}_x: \cD(R) \longrightarrow \cD(R)^\wedge_x \qquad \text{called the derived } x\text{-completion functor}
\end{equation*}
and that $\cD(R)^\wedge_x$ is itself a stable $\infty$-category. Furthermore, since the derived $x$-completion functor is compatible with the monoidal structure of $\cD(R)$ -- that is if $M_1 \rightarrow M_2$ is an isomorphism after $x$-completing then for any other $N \in \cD(R)$ the map $M_1 \otimes_R^L N \rightarrow M_2 \otimes_R^L N$ is an isomorphism after $x$-completion \cite[Proposition 7.3.5.1]{luriespectral} -- it follows from \cite[Proposition 2.2.1.9.]{lurieHA} that we can endow $\cD(R)^\wedge_x$ with a unique symmetric monoidal structure making the derived $x$-completion functor symmetric monoidal. In particular, for $M, N \in \cD(R)_x^\wedge$ the derived completed tensor product can be computed as
\begin{equation*}
	M \cotimes_R^L N := (M \otimes_R^L N)^\wedge_x
\end{equation*}
\end{const}

One of the most important structural properties of derived $I$-complete $R$-complexes is that we can check they are zero modulo $I$.

\begin{lemma}[Derived Nakayama]\label{derived_nakayama} Let $I \subset R$ be a finitely generated ideal, and $M$ a derived $I$-complete $R$-complex. Then $M = 0$ if and only if $M \otimes_R^L R/I = 0$.
	
\end{lemma}
\begin{proof} \cite[Tag 0G1U]{stacks-project}
\end{proof}

Next, we introduce the classically complete tensor product and explain its relation to the derived completed tensor product.

\begin{notation}\label{classical_complete_tensor} Given pair of $R$ modules $M_1, M_2 \in \Mod_R$ and $x \in R$, there is a natural notion of a (classically) $x$-completed tensor product $M_1 \cotimes_R M_2$, given by
\begin{equation*}
	M_1 \cotimes_R M_2 := \lim_n (M_1 \otimes_R M_2)/(x^n)
\end{equation*}
where the limit is computed in $\Mod_R$.
\end{notation}

\begin{warning} Despite what the notation suggest it is not generally true that the derived $x$-completed tensor product $\cotimes_R^L$ is the derived functor of the classically $x$-completed tensor product of $R$-modules. Indeed, given two $R$-modules $M_1, M_2$ we know that $N := H^0(M_1 \cotimes_R^L M_2)$ is a derived $x$-complete module, but since $N$ is not necessarily $x$-separated it need not be classically $x$-complete. We adjective ``derived'' on the monoidal structure $(\cD(R)^\wedge_x, \cotimes_R^L)$ is due to the fact that $\cotimes_R^L$ is defined at the level of stable $\infty$-categories and not because it is a derived functor of $\cotimes_R$.
\end{warning}

\begin{defn} Let $M$ be a $R$-module and $x$ an element of $R$, we say that $M$ is derived $x$-complete if it is derived $x$-complete when considered as an object of $\cD(R)$ via the canonical inclusion $\Mod_R \hookrightarrow \cD(R)$. The category of derived $x$-complete modules is the full subcategory of $\Mod_R$ spanned by the derived $x$-complete modules, we denote this category by $\Mod_{x, R}^\wedge$.
\end{defn}

\begin{prop} The category of derived $x$-complete modules $\Mod_{x, R}^\wedge$ has the following properties:
	\begin{enumerate}[(1)]
		\item It is abelian and it has all limits and colimits.
		\item The canonical inclusion functor $\Mod_{x, R}^\wedge \rightarrow \Mod_R$ is exact and it commutes with all limit.
	 	\item Filtered colimits in $\Mod_{x, R}^\wedge$ are not exact in general. In particular, it is not a Grothendieck abelian category.
	\end{enumerate}
\end{prop}

\begin{proof} \cite[Tag 0ARC]{stacks-project}
\end{proof}

It is due to the fact that $\Mod_{x, R}^\wedge$ is not a Grothendieck abelian category that we are forced to work ``fully derived''; that is, by working on the derived $\infty$-category $\cD(R)^\wedge_x$ from the  beginning and then isolating a convenient abelian category $\Mod_{x, R}^\wedge$ inside of it. In fact, it is not generally true that $\cD(R)^\wedge_x$ is the derived category of $\Mod_{x,R}^\wedge$, especially in the case that is of most interest to us -- when $R$ is a integral perfectoid ring. However, if we were to assume that $R$ is noetherian there is more one can say \cite[Section 7.3.7]{luriespectral}.

\begin{lemma}\label{derived_completeness_on_cohomology} Let $R$ be a ring, $x$ an element of $R$ and $M$ an $R$-complex. Then, the following are equivalent
	\begin{enumerate}[(1)]
		\item The $R$-complex $M$ is derived $x$-complete.
		\item All cohomology groups $H^i(M)$ for $i \in \ZZ$ are derived $x$-complete $R$-modules.
	\end{enumerate}
\end{lemma}

\begin{proof} \cite[Theorem 7.3.4.1]{luriespectral}.
\end{proof}

\begin{const}\label{complete_t_structure} Let $R$ be a ring and $x$ an element of $R$. It follows from Lemma \ref{derived_completeness_on_cohomology} and \cite[Proposition 7.3.4.4]{luriespectral} that the following pair of subcategories
\begin{equation*}
	(\cD(R)_{x}^\wedge)^{\ge 0} := \cD(R)^{\ge 0} \cap (\cD(R)_{x}^\wedge) \qquad
	(\cD(R)_{x}^\wedge)^{\le 0} := \cD(R)^{\le 0} \cap (\cD(R)_{x}^\wedge)
\end{equation*}
determine a $t$-structure on $\cD(R)_x^\wedge$, for which we have that $\cD(R)_x^{\wedge \heartsuit} = \Mod_{R, x}^{\wedge}$. In particular, this implies that the inclusion $\cD(R)^\wedge_x \hookrightarrow \cD(R)$ is $t$-exact while the derived $x$-completion functor $(-)^\wedge_x: \cD(R) \rightarrow \cD(R)^\wedge_{x}$ is only right $t$-exact. Hence, it follows that the full-subcategory $\cD(R)^{\wedge, \le 0}_{x}$ is stable under the derived completed tensor product, endowing $\cD(R)^{\wedge, \le 0}_{x}$ with the structure of a symmetric monoidal $\infty$-category. And since the truncation map
\begin{equation*}
	\tau^{\ge 0}: \cD(R)_x^{\wedge, \le 0} \rightarrow \Mod_{R,x}^\wedge
\end{equation*}
is compatible with the monoidal functor \cite[Proposition 2.2.1.8]{lurieHA} it follows that we can endow $\Mod_{R,x}^\wedge$ with the structure of a symmetric monoidal category whose tensor product is given by $H^0(- \cotimes_R^L -)$.
\end{const}

Summarizing the current situation, we have a commutative diagram of the form
\begin{cd}
	\cD(R)^{\le 0} \ar[r, "(-)^\wedge_x"] \ar[d, "\tau^{\ge 0}", swap] & \cD(R)^{\wedge, \le 0}_{x} \ar[d, "\tau^{\ge 0}"] \\
	\Mod_{R} \ar[r, "H^0(-^\wedge_x)", swap] & \Mod_{R, x}^\wedge
\end{cd}
where all the maps are left adjoints to canonical inclusions on the opposite directions. Commutativity of the diagram follows from the fact that the inclusion maps on the opposite directions clearly commute and the uniqueness of adjoint functors. Furthermore all the symmetric monoidal structures of the categories above (except that on $\cD(R)^{\le 0}$) are completely determined by the requirement that all the functors in the diagram are symmetric monoidal.

\begin{const}\label{complete_algebra_objects} Given a symmetric monoidal $\infty$-category we learn from \cite[Section 2.1.3]{lurieHA} that we can consider the category $\CAlg(\cC)$ of commutative algebra objects in $\cC$; and that symmetric monoidal functors send commutative algebra objects to commutative algebra objects. In particular, this implies that we have a commutative diagram
\begin{cd}
	\CAlg(\cD(R)^{\le 0}) \ar[r, "(-)^\wedge_x"] \ar[d, "\tau^{\ge 0}", swap] & \CAlg(\cD(R)^{\wedge, \le 0}_{x}) \ar[d, "\tau^{\ge 0}"] \\
	\CAlg(\Mod_{R}) \ar[r, "H^0(-^\wedge_x)", swap] & \CAlg(\Mod_{R, x}^\wedge)
\end{cd}
where $\CAlg(\Mod_R)$ is the classical category of commutative $R$-algebras. Moreover, given a pair of morphisms $S_1 \leftarrow R \rightarrow S_2$ in any of the commutative algebra object categories described above (eg. $\CAlg(\cD(R)^{\le 0})$) we have that the pushout can be identified with $S_1 \otimes_R^{?} S_2$, where $\otimes_R^{?}$ is the tensor product (eg. $- \otimes_R^L -$) in the underlying module category (eg. $\cD(R)^{\le 0}$) of the category in question (cf. \cite[Proposition 3.2.4.7]{lurieHA}).

Furthermore, the canonical inclusions (eg. $\Mod_{R, x}^\wedge \hookrightarrow \Mod_R$) induce a commutative diagram of fully-faithful functors
\begin{cd}
	\CAlg(\cD(R)^{\le 0}) & \CAlg(\cD(R)^{\wedge, \le 0}_{x}) \ar[l, hook']\\
	\CAlg(\Mod_{R}) \ar[u, hook]  & \CAlg(\Mod_{R, x}^\wedge) \ar[l, hook'] \ar[u, hook]
\end{cd}
which are the right adjoints to the localization functors described above (cf. \cite[Proposition 2.2.1.9]{lurieHA}).
\end{const}

The following results show that in many circumstances the distinction between classical and derived completion of perfectoid rings, and their corresponding perfect prisms, disappears. Let us also recall that in official definition of perfect prism $(A,d)$ taken in \cite{prisms} it is only required that $A$ is derived $(p,d)$-complete, but by \cite[Lemma 3.8]{prisms} we know that this implies that $A$ is classically $(p,d)$-complete.

\begin{lemma}\label{torsion_perfect_delta_ring} Let $A$ be a $p$-complete perfect $\delta$-ring, $x \in S := A/p$ and $[x] \in A$ its Teichmuller lift. Then, 
\begin{equation*}
	A[x^\infty] = A[x] = A[x^{1/p^\infty}]
\end{equation*}
where $A[x]$ denotes the $[x]$-torsion of $A$.
\end{lemma}

\begin{proof} By Proposition \ref{perfect_delta_rings_equiv} we know that $A$ is $p$-torsion free and that $S = A/p$ is perfect, in particular this implies that $A = W(S)$ and that any element $f \in A$ admits a Teichmuller expansion as $f = \sum_{i = 0}^{\infty} [a_i]p^i$ (cf. Example \ref{teichmuller_expansion}). Thus, if $[x]f = 0$ in $A$, as $A$ is $p$-torsion free we learn that $[x][a_i] = [xa_i] = 0$, for all $i$, which is equivalent to requiring that $xa_i = 0$ in $S$. Finally, as $S$ is perfect we can conclude that if $xa_i = 0$ then $x^{1/p^n}a_i = 0$ for all $n \in \ZZ_{\ge 0}$, proving the desired result.
\end{proof}

\begin{lemma}\label{classical_varpi_complete_perfect_prism} Let $S$ be an integral perfectoid ring which is $\varpi$-complete with respect to some $\varpi \in S$ where $\varpi^p$ divides $p$ (cf. \ref{perfectoid_bounded_pi_torsion}). Then, the corresponding perfect prism $(A,d)$ is (classically) $(p, d, [\varpi^{\flat}])$-complete, and $d \in (p, [\varpi^{\flat}])$.
\end{lemma}

\begin{proof} By definition of a perfect prism it follows that $A$ is classically $(p,d)$-complete, and the $\varpi$-completeness of $R$ implies that $R^\flat$ is $\varpi^\flat$-complete (Lemma \ref{completeness_tilt}). And since the Witt vector functor preserves limits we learn that $W(R^\flat)$ is classically $[\varpi^\flat]$-complete. For the claim that $d \in (p, [\varpi^{\flat}])$, it is shown during the proof of \cite[Lemma 3.10]{BMS1} that (up to a unit) $d = p + [\varpi^{\flat}]^p x$ for some $x \in A$, finishing the proof.
\end{proof}

\begin{prop}\label{varpi_completion_perfectoid} Let $S$ be a integral perfectoid such that an element $\varpi \in S$ admits compatible $p$-power roots, with $(A,d)$ as its corresponding $(p,d)$-complete perfect prism. Then,
\begin{enumerate}[(1)]
	\item The derived and classical $\varpi$-completion of $S$ agree. We denote it by $S^{\wedge}$.
	\item The derived and classical $[\varpi^{\flat}]$-completion of $A$ agree. We denote it by $A^{\wedge}$.
	\item $S^{\wedge}$ is an integral perfectoid ring with $(A^{\wedge},d)$ as its corresponding perfect prism.
\end{enumerate}
\end{prop}

\begin{proof} Without loss of generality we may assume that $\varpi \in S$ admits compatible $p$-power roots (\ref{compatible_roots_perfectoid}). From \ref{perfectoid_torsion} it follows that $S$ has bounded $\varpi$-torsion, and so derived and classical $\varpi$-completions of $S$ agree. Similarly, by \ref{torsion_perfect_delta_ring} it follows that $A$ has bounded $[\varpi^{\flat}]$-torsion, showing that classical and derived $\varpi$-completions agree. To conclude, we need to show that $(A^\wedge, d)$ is a perfect prism with $A^\wedge/d = S^\wedge$. Indeed, by functoriality of derived completion it follows that $A^\wedge$ is a perfect $\delta$-ring, and by the structure map $A \rightarrow A^\wedge$ it follows that $d \in A^\wedge$ is a distinguish element -- showing that $(A^\wedge, d)$ is a perfect prism. Finally, since $[\varpi^\flat] = \varpi \bmod d$ it follows that $A^\wedge/d = S^\wedge$ by generalities of derived completion.
\end{proof}

\begin{prop}\label{tensor_integral_perfectoid} Let $R$ be an integral perfectoid which is $\varpi$-complete with respect to some $\varpi \in R$ where $\varpi^p$ divides $p$, with corresponding $(p, [\varpi^{\flat}])$-complete perfect prism $(A,d)$. And let $S_1, S_2$ be (classically) $\varpi$-complete integral perfectoid $R$-algebras. Then,
\begin{enumerate}[(1)]
	\item The derived $\varpi$-complete tensor product $S_1 \cotimes_R^L S_2$ and the classical $\varpi$-complete tensor product $S_1 \cotimes_R S_2$ agree.
	\item The derived $(p, [\varpi^{\flat}])$-complete tensor product $\A_{\inf}(S_1) \cotimes^L_A \A_{\inf} (S_2)$ and the classical $(p, [\varpi^{\flat}])$-complete tensor product $\A_{\inf}(S_1) \cotimes_A \A_{\inf} (S_2)$ agree.
	\item $S_1 \cotimes_R S_2$ is an integral perfectoid ring with corresponding perfect prism $(\A_{\inf}(S_1) \cotimes_A \A_{\inf} (S_2), d)$.
\end{enumerate}
\end{prop}

\begin{proof} We begin by proving $(2)$. By derived Nakayama \cite[Tag 0G1U]{stacks-project} we may check that $\A_{\inf}(S_1) \cotimes^L_A \A_{\inf} (S_2)$ is concentrated in degree zero modulo $p$, so it suffices to show that the derived $\varpi^\flat$-complete tensor product $S_1^\flat \cotimes_{R^\flat}^L S_2^\flat$ is concentrated in degree zero. We learn from \cite[Lemma 3.16]{projectivity_affine_grass} that $S_1^\flat \otimes_{R^\flat}^L S_2^\flat$ is concentrated in degree zero as everything in sight is a perfect ring of characteristic $p$, and since $S_1^\flat \otimes_{R^\flat} S_2^\flat$ is perfect it has bounded $\varpi^\flat$-torsion so the derived $\varpi^\flat$-completion agrees with the classical $\varpi^\flat$-completion, showing that $S_1^\flat \cotimes_{R^\flat}^L S_2^\flat = S_1^\flat \cotimes_{R^\flat} S_2^\flat$. Given that $\A_{\inf}(S_1) \cotimes_A^L \A_{\inf} (S_2)$ is concentrated in degree zero, it can be identified with the derived $(p, [\varpi^{\flat}])$-completion of $\A_{\inf}(S_1) \otimes_A \A_{\inf} (S_2)$, and in turn $\A_{\inf}(S_1) \otimes_A \A_{\inf} (S_2)$ can be identified with the pushout of $\A_{\inf}(S_1) \leftarrow A \rightarrow \A_{\inf}(S_2)$ in the category of $\delta$-rings. By the functoriality of Frobenius we learn that $\A_{\inf}(S_1) \otimes_A \A_{\inf} (S_2)$ is a perfect $\delta$-ring, which implies by \cite[Lemma 2.28]{prisms} that $\A_{\inf} (S_1) \otimes_A \A_{\inf}(S_2)$ is $p$-torsion free, so the derived and classical $p$-completions of $\A_{\inf} (S_1) \otimes_A \A_{\inf}(S_2)$ agree. Then, by \ref{torsion_perfect_delta_ring} we learn that $(\A_{\inf} (S_1) \otimes_A \A_{\inf}(S_2))^{\wedge}_{(p)}$ has bounded $[\varpi^\flat]$-torsion, showing that the classical and derived $[\varpi^\flat]$-completions agree. This completes the proof of $(2)$.

Finally, notice that $((\A_{\inf}(S_1) \otimes_A \A_{\inf} (S_2))^{\wedge}_{(p)}, d)$ is a $(p)$-complete perfect prism as $d$ is a distinguished element -- in particular $d$ is a non-zero divisor. Then, generalities of the derived tensor product and derived completion imply that $(\A_{\inf}(S_1) \otimes_A \A_{\inf} (S_2))^{\wedge}_{(p)}/d = (S_1 \otimes_R^L S_2)^{\wedge}_{(p)}$, proving that $(S_1 \otimes_R^L S_2)^{\wedge}_{(p)}$ is an integral perfectoid ring, in particular, it is concentrated in degree zero and it is classically $p$-complete. Then, Proposition \ref{varpi_completion_perfectoid} completes the proof of $(1)$ and $(3)$.
\end{proof}

\subsection{The almost zero category}

Throughout this section $R$ is an integral perfectoid ring which is $\varpi$-complete with respect to some $\varpi \in R$ where $\varpi^p$ divides $p$. In this section we will introduce the category of $\varpi$-almost zero $R$-modules (resp. $R$-complexes) as full subcategories of $\Mod_R$ (resp. $\cD(R)$) and show that it can be identified with the essential image of the fully faithful functors $\Mod_{(R/\varpi)_{\perfd}} \hookrightarrow \Mod_R$ (resp. $\cD((R/\varpi)_{\perfd}) \hookrightarrow \cD(R)$). Furthermore, we will show that all $\varpi$-almost zero $R$-complexes are automatically $\varpi$-complete giving us a $\varpi$-complete variants of this statements.

\begin{defn} Let $R$ be an integral perfectoid ring which is $\varpi$-complete with respect to some $\varpi \in R$ where $\varpi^p$ divides $p$. Define the ideal $(\varpi)_{\perfd} \subset R$ as the kernel of the map $R \rightarrow (R/\varpi)_{\perfd}$ where $(R/\varpi)_{\perfd}$ is defined as $\colim_{\text{Frob}} R/\varpi$, which makes sense since $R/\varpi$ is of characteristic $p$ by assumption.
\end{defn}

Recall that in the situation of the previous definition there exists a $\varpi^\prime \in R$ which admits compatible $p$-power roots and is a unit multiple of $\varpi$. Thus, there is a natural identification between $(R/\varpi)_{\perfd}$ and $R/(\varpi^{\prime 1/p^\infty})$, which in turn implies that $(\varpi)_{\perfd} = (\varpi^{\prime 1/p^\infty})$ -- we are implicitly using the fact that integral perfectoid rings are reduced \cite[Section 2.1.3]{purityflat}. 

Furthermore, since $R$ and $(R/\varpi)_{\perfd}$ are derived $\varpi$-complete, the map $R \rightarrow (R/\varpi)_{\perfd}$ is surjective we can identify $(\varpi)_{\perfd}$ as the fiber of the map $R \rightarrow (R/\varpi)_{\perfd}$ computed in either $\cD(R)$ or $\cD(R)_{\varpi}^\wedge$ -- as the fully faithful inclusion $\cD(R)_{\varpi}^\wedge \hookrightarrow \cD(R)$ preserves fiber sequences. Hence, we can conclude that $(\varpi)_{\perfd}$ is a derived $\varpi$-complete $R$-module.

\begin{defn} We say that an $R$-module $M$ is $\varpi$-almost zero if: for every element $m \in M$ and $x \in (\varpi)_{\perfd}$ we have that $xm = 0$. Similarly we say that an $R$-complex $M$ is $\varpi$-almost zero if $H^i(M)$ is a $\varpi$-almost zero for all $i \in \ZZ$.

Recall that there is a $t$-exact fully-faithful functor $\cD(R)_{\varpi}^\wedge \hookrightarrow \cD(R)$ (\ref{complete_t_structure}), thus the cohomology groups $H^n(M)$ of a derived $\varpi$-complete $R$-complex $M$ is the same whether we consider it as an object of $\cD(R)_{\varpi}^\wedge$ or $\cD(R)$. Hence, we say that a derived $\varpi$-complete $R$-complex $M$ is almost zero if it is almost zero as an $R$-complex.
\end{defn}

\begin{warning} Since we have not assumed that $R$ is $\varpi$-torsion free it is not generally true that the derived $\varpi$-complete $R$-module $(\varpi)_{\perfd}$ is a flat $R$-module.
\end{warning}

\begin{const}\label{const_almost_zero_idempotent} Consider the following commutative diagram $R$-complexes
\begin{cd}
	(\varpi)_{\perfd} \otimes_R^L (\varpi)_{\perfd} \ar[r] \ar[d] & 
	R \otimes^L_R R \ar[r] \ar[d] & (R/\varpi)_{\perfd} \otimes_R^L (R/\varpi)_{\perfd} \ar[d] \\
	(\varpi)_{\perfd} \ar[r] & R \ar[r] & (R/\varpi)_{\perfd}
\end{cd}
where the maps $R \otimes_R^L R \rightarrow R$ and $(R/\varpi)_{\perfd} \otimes_R^L (R/\varpi)_{\perfd} \rightarrow (R/\varpi)_{\perfd}$ are the multiplication maps coming from the commutative algebra structure of $R$ and $(R/\varpi)_{\perfd}$, and the horizontal maps come from the fiber sequence $(\varpi)_{\perfd} \rightarrow R \rightarrow (R/\varpi)_{\perfd}$ and the fact that $- \otimes_R^L -$ commutes with colimits independently on each variable.

We claim that all the vertical maps and their derived $\varpi$-completed variants are isomorphisms. Indeed, it is clear that $R \otimes_R^L R \rightarrow R$ is an isomorphisms, and since $R$ is derived $\varpi$-complete it is also clear that $R \cotimes_R^L R \rightarrow R$ is an isomorphism. Thus, it suffices to show that either
\begin{equation*}
	(\varpi)_{\perfd} \otimes_R^L (\varpi)_{\perfd} \rightarrow (\varpi)_{\perfd} \qquad \text{ or } \qquad
	(R/\varpi)_{\perfd} \otimes_R^L (R/\varpi)_{\perfd} \rightarrow (R/\varpi)_{\perfd}
\end{equation*}
are isomorphisms. Furthermore, since $(\varpi)_{\perfd}$ and $(R/\varpi)_{\perfd}$ are derived $\varpi$-complete it suffices to show that the completed variants are isomorphisms; we proof this in Lemma \ref{almost_zero_idempotent}.

Once we stablish that the vertical maps and their derived completed variants are isomorphisms, it is clear that the following tensor products are zero
\begin{equation*}
	(R/\varpi)_{\perfd} \otimes_R^L (\varpi)_{\perfd} = 0 \qquad \text{ and } \qquad (R/\varpi)_{\perfd} \cotimes_R^L (\varpi)_{\perfd} = 0
\end{equation*}
Indeed, the $R$-module structure on $(R/\varpi)_{\perfd}$ induces a commutative diagram of the form
\begin{cd}
	(R/\varpi)_{\perfd} \otimes_R^L (\varpi)_{\perfd} \ar[r] \ar[d] & 
	(R/\varpi)_{\perfd} \otimes^L_R R \ar[r] \ar[d, "\simeq"] & (R/\varpi)_{\perfd} \otimes_R^L (R/\varpi)_{\perfd} \ar[d, "\simeq"] \\
	0 \ar[r] & (R/\varpi)_{\perfd} \ar[r] & (R/\varpi)_{\perfd}
\end{cd}
from the fiber sequence $(\varpi)_{\perfd} \rightarrow R \rightarrow (R/\varpi)_{\perfd}$, showing that $(R/\varpi)_{\perfd} \otimes_R^L (\varpi)_{\perfd} = 0$.
\end{const}

\begin{lemma}\label{almost_zero_idempotent} Considering $(\varpi)_{\perfd}$ and $(R/\varpi)_{\perfd}$ as objects in $\cD(R)_\varpi^\wedge$ we get that the natural multiplication maps of Construction \ref{const_almost_zero_idempotent}
	\begin{equation*}
		(\varpi)_{\perfd} \cotimes_R^L (\varpi)_{\perfd} \rightarrow (\varpi)_{\perfd} \qquad (R/\varpi)_{\perfd} \cotimes_R^L (R/\varpi)_{\perfd} \rightarrow (R/\varpi)_{\perfd}
	\end{equation*}
are isomorphism.
\end{lemma}

\begin{proof} By Proposition \ref{tensor_integral_perfectoid} we know that the natural map
	\begin{equation*}
		(R/\varpi)_{\perfd} \cotimes_R^L (R/\varpi)_{\perfd} \rightarrow (R/\varpi)_{\perfd} \cotimes_R (R/\varpi)_{\perfd}
	\end{equation*}
from the derived $\varpi$-complete tensor product to the classical $\varpi$-complete tensor product is an isomorphism. Furthermore, since the map $R \rightarrow (R/\varpi)_{\perfd}$ is surjective, it is clear that the multiplication map $(R/\varpi)_{\perfd} \otimes_R (R/\varpi)_{\perfd} \rightarrow (R/\varpi)_{\perfd}$ is an isomorphism, and since $(R/\varpi)_{\perfd}$ is already $\varpi$-complete we learn that
\begin{equation*}
	(R/\varpi)_{\perfd} \cotimes_R (R/\varpi)_{\perfd} \rightarrow (R/\varpi)_{\perfd}
\end{equation*}
is an isomorphism, proving the desired result. Finally, as sketched in Construction \ref{const_almost_zero_idempotent} it follows that $	(\varpi)_{\perfd} \cotimes_R^L (\varpi)_{\perfd} \rightarrow (\varpi)_{\perfd}$ is also an isomorphism.
\end{proof}

\begin{const}\label{const_almost_zero_ff} Recall that the canonical map of rings $R \rightarrow (R/\varpi)_{\perfd}$ induces a pair of adjoint functors $\cD((R/\varpi)_{\perfd}) \rightleftarrows \cD(R)$; moreover, since every object $M \in \cD((R/\varpi)_{\perfd})$ regarded as an $R$-complex is derived $\varpi$-complete it follows that the map $i_*: \cD((R/\varpi)_{\perfd}) \rightarrow \cD(R)$ factors through the inclusion
\begin{cd}
	\cD((R/\varpi)_{\perfd}) \ar[r, "i_*"] \ar[rd, "i_*^\wedge", swap] & \cD(R) \\
	& \cD(R)^\wedge_{\varpi} \ar[u, hook]
\end{cd}
In particular, this implies that both maps $i_*$ and $i_*^\wedge$ preserve all limits. On the other hand, given an $R$-complex $M$ its clear that the $R$-complex $M \otimes_R^L (R/\varpi)_{\perfd}$ is derived $\varpi$-complete, thus the left adjoint to $i_*$ given by $i^*: - \otimes_R^L (R/\varpi)_{\perfd}$ factors through $\cD(R)^\wedge_{\varpi}$, making the following diagram commute
\begin{cd}
	\cD((R/\varpi)_{\perfd})  & \cD(R) \ar[d, "(-)^\wedge"] \ar[l, "i^*", swap] \\
	& \cD(R)^\wedge_{\varpi} \ar[ul, "i^{\wedge *}"]
\end{cd}
In other words we have the identity $- \cotimes_R^L (R/\varpi)_{\perfd} = - \otimes^L_R (R/\varpi)_{\perfd}$. In particular, this construction shows that the adjunction $(i_*, i^*)$ induces an adjunction $(i_*^\wedge, i^{\wedge *})$.
\end{const}

\begin{prop}\label{almost_zero_fully_faithful} The pair of functors of Construction \ref{const_almost_zero_ff}
\begin{equation*}
	i_*: \cD((R/\varpi)_{\perfd}) \longrightarrow \cD(R) \qquad \text{ and } \qquad 
	i_*^{\wedge}: \cD((R/\varpi)_{\perfd}) \longrightarrow \cD(R)_\varpi^\wedge
\end{equation*}
are fully faithful and preserves all limits and colimits.
\end{prop}

\begin{proof} By the commutativity of the diagrams in Construction \ref{const_almost_zero_ff} if $i_*$ is fully faithful then $i_*^{\wedge}$ is fully faithful, thus we only proof that $i_*$ is fully faithful. By the adjunction $(i_*, i^*)$, given a pair of objects $M,N \in \cD((R/\varpi)_{\perfd})$ we have a natural isomorphism
\begin{equation*}
	R\Hom_R (i_* M, i_* N) \simeq R\Hom_{(R/\varpi)_{\perfd}} (M \otimes_R^L (R/\varpi), N)
\end{equation*}
Hence, it suffices to show that the canonical map $M \otimes_R^L (R/\varpi)_{\perfd} \rightarrow M$ is an isomorphism in $\cD((R/\varpi)_{\perfd})$. From the fact that every $((R/\varpi)_{\perfd})$-complex $M$ can be presented as a colimit of complexes of the form $(R/\varpi)_{\perfd} [n]$ by \cite[Proposition 7.2.4.2]{lurieHA}, and that $-\otimes_R^L (R/\varpi)_{\perfd}$ commutes with colimits independently on each variable we are reduced to the case where $M = (R/\varpi)_{\perfd}$, which was established in Lemma \ref{almost_zero_idempotent}.

Finally, the adjunctions $(i_*, i^*)$ and $(i_*^\wedge, i^{\wedge *})$ make it clear that the functors $i_*$ and $i_*^\wedge$ preserve all limits. Hence, we need to show that $i_*$ and $i^*$ preserve all colimits, but the commutativity of the diagrams in Construction \ref{const_almost_zero_ff} show that it suffices to show that $i_*$ preserves all colimits. To show that $i_*$ preserves all colimits let $M := \colim M_i \in \cD(R)$ be a colimit computed in $\cD(R)$ of $(R/\varpi)_{\perfd}$-complexes $M_i$; it then suffices to show that the canonical map $M \rightarrow M \otimes_R^L (R/\varpi)_{\perfd}$ is an isomorphism. But this is clear as $M_i \rightarrow  M_i \otimes_R^L (R/\varpi)_{\perfd}$ is an isomorphism for all $M_i$ and $-\otimes_R^L (R/\varpi)_{\perfd}$ commutes with all colimits independently on each variable.
\end{proof}

\begin{rem} From the perspective of higher algebra Proposition \ref{almost_zero_fully_faithful} is rather remarkable; it is essentially never the case that for a surjective map $R \twoheadrightarrow S$ the induced functor $\cD(S) \rightarrow \cD(R)$ is fully faithful. A prominent example is the following: given the surjective map $R \otimes_{\CC} R \twoheadrightarrow R$ induced by multiplication of $\CC$-algebras the functor $\cD(R) \rightarrow \cD(R \otimes_{\CC} R)$ is almost never fully faithful; indeed, $R\Hom_R(R, R) = R$ is concentrated in degree zero, while $R\Hom_{R \otimes_{\CC} R} (R, R)$ identifies with the hochschild cohomology of $R$.
\end{rem}

\begin{const}\label{almost_zero_right_adj} Given that the functor $i_*: \cD((R/\varpi)_{\perfd}) \rightarrow \cD(R)$ of presentable $\infty$-categories preserves all colimits we can invoke the adjoint functor theorem \cite[Corollary 5.5.2.9]{lurieHTT} to guarantee the existence of a right adjoint to $i_*$ which we denote by $i^!$. Consider the following diagram
\begin{cd}
	\cD((R/\varpi)_{\perfd}) & \cD(R)  \ar[l, "i^!", swap] \\
	& \cD(R)^\wedge_{\varpi} \ar[u, hook] \ar[lu, "i^{\wedge !}"]
\end{cd}
where $i^{\wedge !}$ is defined as the composition of the other two functors, making the diagram commute. Since the inclusion $\cD(R)^\wedge_\varpi \rightarrow \cD(R)$ is a right adjoint adjoint to the derived completion functor, it follows that $i^{\wedge !}$ is a right adjoint to $i_*^\wedge$.
\end{const}

Recall that the functors $i_*: \cD((R/\varpi)_{\perfd}) \rightarrow \cD(R)$ and the canonical inclusion $\cD(R)^\wedge_\varpi \rightarrow \cD(R)$ are $t$-exact. Hence, the diagrams of Construction \ref{const_almost_zero_ff} induces fully faithful functors of abelian categories
\begin{cd}
	\Mod_{(R/\varpi)_{\perfd}} \ar[r, "i_*"] \ar[rd, "i_*^\wedge", swap] & \Mod_R \\
	& \Mod_{R, \varpi}^\wedge \ar[u, hook]
\end{cd}
As in the derived version this functors admit left adjoints
\begin{cd}
	\Mod_{(R/\varpi)_{\perfd}} & \Mod_R \ar[l, "H^0(i^* -)", swap] \ar[d, "H^0(-^\wedge_\varpi)"]\\
	& \Mod_{R, \varpi}^\wedge \ar[lu, "H^0(i^{\wedge *} - )"]
\end{cd}
However, the functors $i^*, i^{\wedge *}$ and $(-)^\wedge_{\varpi}$ are only right $t$-exact; thus, at the abelian level we are required to truncate this functors to get the left adjoints of the abelian versions of $i_*$ and $i_*^\wedge$.

\begin{prop}\label{almost_zero_as_modules} The categories of $\varpi$-almost zero $R$-complexes and $\varpi$-almost zero $R$-modules admit the following descriptions
\begin{enumerate}[(1)]
	\item The fully faithful functor $i_*: \Mod_{(R/\varpi)} \rightarrow \Mod_{R}$ identifies the category of $\varpi$-almost zero $R$-modules with the essential image of $i_*$.
	\item The fully faithful functor $i_*: \cD((R/\varpi)_{\perfd}) \rightarrow \cD(R)$ identifies the category of $\varpi$-almost zero $R$-complexes with the essential image of $i_*$.
\end{enumerate}
Once this results are established, from the commutativity of the diagrams in Construction \ref{const_almost_zero_ff} it follows that all $\varpi$-almost zero $R$-modules (resp. $R$-complexes) are automatically derived $\varpi$-complete. Hence, we get the following description of the categories of derived $\varpi$-complete $\varpi$-almost zero $R$-modules (resp. $R$-complexes)
\begin{enumerate}[(1)] \setcounter{enumi}{2}
	\item The fully faithful functor $i_*^\wedge: \Mod_{(R/\varpi)} \rightarrow \Mod_{R, \varpi}^\wedge$ identifies the category of derived $\varpi$-complete $\varpi$-almost zero $R$-modules with the essential image of $i_*^\wedge$.
	\item The fully faithful functor $i_*^{\wedge}: \cD((R/\varpi)_{\perfd}) \rightarrow \cD(R)^\wedge_{\varpi}$ identifies the category of derived $\varpi$-complete $\varpi$-almost zero $R$-complexes with the essential image of $i_*^\wedge$.
\end{enumerate}
In particular, this shows that the category of $\varpi$-almost zero $R$-modules (resp. $R$-complexes) is an abelian category (resp. stable $\infty$-category).
\end{prop}

\begin{proof} To proof $(1)$ it suffices to show that given $\varpi$-almost zero $R$-module $M$ that the unit of adjunction $M \rightarrow M \otimes_{R} (R/\varpi)_{\perfd}$ of the pair $(i_*, H^0(i^*-))$ is an isomorphism. Notice that the unit of adjunction $M = M \otimes_R R \rightarrow M \otimes_{R} (R/\varpi)_{\perfd}$ can be obtained by tensoring $M$ with the canonical quotient map $R \rightarrow (R/\varpi)_{\perfd}$, and that any object of the form $(m, x) \in M \otimes_R R$ where $x \in (\varpi)_{\perfd} \subset R$ is zero. Showing that the map $M \rightarrow M \otimes_{R} (R/\varpi)_{\perfd}$ is an isomorphism. Moreover, since this shows that $M$ is in the essential image of $i_*: \cD((R/\varpi)_{\perfd}) \rightarrow \cD(R)$ it also implies that the unit of adjunction $M \rightarrow M \otimes_R^L (R/\varpi)_{\perfd}$ is an isomorphism.

To proof $(2)$ it suffices to show that for any $\varpi$-almost zero $R$-complex $M$ that the unit of adjunction $M \rightarrow M \otimes_R^L (R/\varpi)_{\perfd}$ is an isomorphism. Recall that any bounded object $N$ in $\cD(R)$ can be expressed as successive extensions of its cohomology groups $H^i(N)$; for example, we have the fiber sequence $H^0(N) \rightarrow \tau^{[0,1]} N \rightarrow H^1(N)[-1]$ presenting $\tau^{[0,1]}N $ as an extension of $H^1(N)[-1]$ by $H^0(N)$. Hence, since $H^i(M) \rightarrow H^i(M) \otimes_R^L (R/\varpi)_{\perfd}$ is an isomorphism for all $i \in \ZZ$ it follows that the map $M \rightarrow M \otimes_R^L (R/\varpi)_{\perfd}$ is an isomorphism for all bounded complexes; for example, its easy to see that since $H^0(M)$ and $H^1(M)[-1]$ are in the image of $i_*$ then so is $\tau^{[0,1]} M$. The case for an unbounded $\varpi$-almost zero $M$ $R$-complex follows by approximating $M$ by its truncations $\tau^{[a,b]} M$, which we already showed are in the image of the functor $i_*: \cD((R/\varpi)_{\perfd}) \rightarrow \cD(R)$, and the fact that $i_*$ preserves all limits and colimits (\ref{almost_zero_fully_faithful}).
\end{proof}

\begin{lemma} An $R$-complex $M$ is $\varpi$-almost zero if and only if $(\varpi)_{\perfd} \otimes_R^L M = 0$.
\end{lemma}

\begin{proof} Recall that we have the following fiber sequence of $R$-modules $(\varpi)_{\perfd} \rightarrow R \rightarrow (R/\varpi)_{\perfd}$; tensoring with $M$ we get the fiber sequence
\begin{equation*}
	(\varpi)_{\perfd} \otimes_R^L \rightarrow M \rightarrow M \otimes_R^L (R/\varpi)_{\perfd}
\end{equation*}
From Proposition  \ref{almost_zero_as_modules} we learn that $M$ is $\varpi$-almost zero if and only if the map $M \rightarrow M \otimes_R^L (R/\varpi)_{\perfd}$ is an isomorphism. Hence, $M$ is $\varpi$-almost zero if and only if $M \otimes_R^L (\varpi)_{\perfd} = 0$.
\end{proof}

\subsection{The almost category}

Throughout this section $R$ is an integral perfectoid ring which is $\varpi$-complete with respect to some $\varpi \in R$ where $\varpi^p$ divides $p$. In what follows we will introduce the almost categories $\cD(R)^a \subset \cD(R)$ and $\cD(R)^{\wedge a}_{\varpi} \subset \cD(R)\picomp$ and show that they fit into a commutative diagram
\begin{cd}
	\cD(R)^{\wedge a}_{\varpi} \ar[r, hook] \ar[d, hook] & \cD(R)^a \ar[d, hook] \\
	\cD(R)^\wedge_\varpi \ar[r, hook] & \cD(R)
\end{cd}
given by the canonical inclusions. We will also show that all the canonical inclusions of the diagram above admit left adjoints making the diagram commute
\begin{cd}
	\cD(R) \ar[d, "(-)^\wedge", swap] \ar[r, "(-)^a"]& \cD(R)^a \ar[d, "(-)^\wedge"]\\
	\cD(R)^\wedge_{\varpi} \ar[r, "(-)^{\wedge a}", swap]& \cD(R)_{\varpi}^{\wedge a}
\end{cd}
from which we will induce symmetric monoidal structures on $\cD(R)^a$ and $\cD(R)^{\wedge a}_{\varpi}$ compatible with the localization functors (eg. $(-)^a$ and $(-)^\wedge$) -- as we did in Construction \ref{derived_complete_cat}. However, we warn the reader that without further hypothesis on $R$ (being $\varpi$-torsion free) its not clear how to endow $\cD(R)^a$ and $\cD(R)^{\wedge a}_\varpi$ with a $t$-structure compatible with the localization functors.

\begin{defn}\label{defn_almost_cat} The almost category of $R$-complexes and its $\varpi$-complete variant are defined as follows.
\begin{enumerate}[(1)]
	\item The almost category $\cD(R)^a$ is a full subcategory of $\cD(R)$ spanned by the $R$-complexes $M$ which satisfy the following property: for every $\varpi$-almost zero $R$-complex $N$ the space of maps $R\Hom_R (N, M)$ is contractible. We denote by $j_*: \cD(R)^a \rightarrow \cD(R)$ the canonical inclusion of $\cD(R)^a$ in $\cD(R)$.
	\item The $\varpi$-complete almost category $\cD(R)^{\wedge a}_{\varpi}$ is a full subcategory of $\cD(R)^\wedge_\varpi$ spanned by the derived $\varpi$-complete $R$-complexes $M$ which satisfy the following: for every $\varpi$-almost zero $R$-complex $N$ the space of maps $R\Hom_R (N, M)$ is contractible. We denote by $j_*^\wedge: \cD(R)^{\wedge a}_{\varpi} \rightarrow \cD(R)\picomp$ the canonical inclusion of $\cD(R)^{\wedge a}_{\varpi}$ in $\cD(R)\picomp$.
\end{enumerate}
Intuitively, we are defining the almost category (resp. $\varpi$-complete) as the `orthogonal complement' of the essential image of $i_*: \cD((R/\varpi)_{\perfd}) \rightarrow \cD(R)$ (resp. $i_*^\wedge: \cD((R/\varpi)_{\perfd}) \rightarrow \cD(R)^\wedge_{\varpi}$). We make this perspective precise below.
\end{defn}

Before moving forward, recall that the fully faithful functors from Proposition \ref{almost_zero_fully_faithful}
\begin{equation*}
	i_*: \cD((R/\varpi)_{\perfd}) \rightarrow \cD(R) \qquad \text{ and } \qquad i_*^{\wedge}: \cD((R/\varpi)_{\perfd}) \rightarrow \cD(R)^\wedge_{\varpi}
\end{equation*}
admit left and right adjoints
\begin{align*}
	i^*: \cD(R) \rightarrow \cD((R/\varpi)_{\perfd}) & \qquad i^{\wedge *}: \cD(R)^\wedge_{\varpi} \rightarrow \cD((R/\varpi)_{\perfd}) \qquad \text{ as left adjoints} \\
	i^!: \cD(R) \rightarrow \cD((R/\varpi)_{\perfd}) & \qquad i^{\wedge !}: \cD(R)^\wedge_{\varpi} \rightarrow \cD((R/\varpi)_{\perfd}) \qquad \text{ as right adjoints}
\end{align*}
described in Constructions \ref{const_almost_zero_ff} and \ref{almost_zero_right_adj}.

\begin{const}\label{const_almost_cat_semi_orth_right} The definition of the almost category (\ref{defn_almost_cat}) says exactly that $\cD(R)^a$ (resp. $\cD(R)^{\wedge a}_{\varpi}$) is the right-orthogonal complement (in the sense of \cite[Definition 7.2.1.1]{luriespectral}) of the inclusion $i_*$ (resp. $i_*^{\wedge}$). Moreover, \cite[Proposition 7.2.1.4]{luriespectral} says that since $i_*$ (resp. $i_*^\wedge$) admits a right adjoint given by $i^!$ (resp. $i^{\wedge !}$) the inclusion functors $j_*$ (resp. $j_*^\wedge$) admit a left adjoint which we denote
\begin{equation*}
	(-)^a: \cD(R) \rightarrow \cD(R)^a \qquad \text{ resp. } \qquad (-)^{\wedge a}: \cD(R)\picomp \rightarrow \cD(R)^{\wedge a}_{\varpi} 
\end{equation*}
Even better, \cite[Proposition 7.2.1.4]{luriespectral} implies that $\cD(R)^a$ (resp. $\cD(R)^{\wedge a}_{\varpi}$) is a stable $\infty$-category and that for every $M$ in $\cD(R)$ (resp. $\cD(R)^\wedge_{\varpi}$) there is an essentially unique fiber sequence
\begin{equation*}
	i_* i^! M \longrightarrow M \longrightarrow j_* M^a \qquad \text{ resp. }  \qquad i_*^\wedge i^{\wedge !} M \longrightarrow M \longrightarrow j^\wedge_* (M)^{\wedge a}
\end{equation*}
where the morphisms are given by the (co)unit of the adjunction. Let us summarize the various functors we just introduced with the following diagrams
\begin{cd}
	\cD((R/\varpi)_{\perfd}) \ar[r, shift left = 1ex, "i_*"] & \ar[l, shift left = 1ex, "i^!"]
	\cD(R)_\varpi^\wedge \ar[r, shift left = 1ex, "(-)^a"] & 
	\cD(R)^a \ar[l, shift left = 1ex, "j_*"]
	 &
	\cD((R/\varpi)_{\perfd}) \ar[r, shift left = 1ex, "i_*^\wedge"] & \ar[l, shift left = 1ex, "i^{\wedge !}"]
	\cD(R)_\varpi^\wedge \ar[r, shift left = 1ex, "(-)^{\wedge a}"] & 
	\cD(R)^{\wedge a}_\varpi \ar[l, shift left = 1ex, "j^\wedge_*"]
\end{cd}
\end{const}

\begin{const}\label{const_almost_comp_localization} We are now ready to establish the existence of the commutative diagrams mentioned in the introduction of this section
\begin{cd}
	\cD(R)^{\wedge a}_{\varpi} \ar[r, hook] \ar[d, hook, "j_*^\wedge", swap] & \cD(R)^a \ar[d, hook, "j_*"] &&	
	\cD(R) \ar[d, "(-)^\wedge", swap] \ar[r, "(-)^a"]& \cD(R)^a \ar[d, "(-)^\wedge"]\\
	\cD(R)^\wedge_\varpi \ar[r, hook] & \cD(R) && 
	\cD(R)^\wedge_{\varpi} \ar[r, "(-)^{\wedge a}", swap]& \cD(R)_{\varpi}^{\wedge a}
\end{cd}
On the left hand side diagram the only possibly non-obvious functor is $\cD(R)^{\wedge a}_{\varpi} \rightarrow \cD(R)^a$, but this follows from the fact that all $\varpi$-almost zero $R$-complexes are automatically $\varpi$-complete. We claim that this functor admits a left adjoint $(-)^\wedge: \cD(R)^a \rightarrow \cD(R)^{\wedge a}_{\varpi}$ given by the derived $\varpi$-completion functor; for this it suffices to show that for an object $M \in \cD(R)^a \subset \cD(R)$ its derived $\varpi$-completion $M^\wedge$ is in the full-subcategory $\cD(R)^{\wedge a}_{\varpi} \subset \cD(R)$. Let $N$ be an $\varpi$-almost zero $R$-complex, then by the stability of $\cD(R)^a$ it follows that $R\Hom_R(N, \Kos(M; \varpi^n)) = 0$ for all $n \in \ZZ_{> 0}$, thus we learn that
\begin{equation*}
	R\Hom_R(N, M^\wedge) = R\Hom_R (N, R\lim \Kos(M; \varpi^n)) = R\lim R\Hom_R (N , \Kos(M; \varpi^n)) = 0
\end{equation*} 
showing that $M^\wedge \in \cD(R)^{\wedge a}_{\varpi}$ as desired. It remains to show that the diagram on the right is commutative, but this follows from the fact that all the functors on the right hand side diagram are left adjoints to the functors on the commutative diagram on the left.
\end{const}

Before diving into the next construction let us introduce a couple of subcategories of $\cD(R)$ and $\cD(R)\picomp$.
\begin{enumerate}[(1)]
	\item The subcategory ${}^\perp \cD((R/\varpi)_{\perfd}) \subset \cD(R)$ is the full-subcategory spanned by the $R$-complexes $M$ which satisfy the following: for every $\varpi$-almost zero $R$-complex $N$ the space of maps $R\Hom_R(M, N)$ is contractible.
	\item The subcategory ${}^\perp \cD((R/\varpi)_{\perfd})^\wedge \subset \cD(R)\picomp$ is the full-subcategory spanned by the derived $\varpi$-complete $R$-complexes $M$ which satisfy the following: for every $\varpi$-almost zero $R$-complex $N$ the space of maps $R\Hom_R(M, N)$ is contractible.
\end{enumerate}
In the terminology of \cite[Section 7.2.1]{luriespectral} this subcategories are called the left-orthogonal complement of $\cD((R/\varpi)_{\perfd}) \subset \cD(R)$ (resp. $\cD((R/\varpi)_{\perfd}) \subset \cD(R)\picomp$) -- compare with the definitions in \ref{defn_almost_cat}.

\begin{const}\label{const_almost_cat_semi_orth_left} From \cite[Corollary 7.2.1.7]{luriespectral} we learn that since $i_*$ (resp. $i_*^\wedge$) admits a left adjoint $i^*$ (resp. $i^{\wedge *}$), the fully-faithful functor 
\begin{equation*}
	h_!: {}^\perp \cD((R/\varpi)_{\perfd}) \hookrightarrow \cD(R) \qquad \text{ resp. } \qquad h_!^\wedge: {}^\perp \cD((R/\varpi)_{\perfd})^\wedge \hookrightarrow \cD(R)\picomp
\end{equation*}
admits a right adjoint
\begin{equation*}
	h^*: \cD(R) \rightarrow {}^\perp \cD((R/\varpi)_{\perfd}) \qquad \text{ resp. } \qquad h^{\wedge *}: \cD(R)\picomp \rightarrow {}^\perp \cD((R/\varpi)_{\perfd})^\wedge
\end{equation*}
Even better, \cite[Corollary 7.2.1.7]{luriespectral} implies that ${}^\perp \cD((R/\varpi)_{\perfd})$ (resp. ${}^\perp \cD((R/\varpi)_{\perfd})^\wedge$) is a stable $\infty$-category and that for every $M$ in $\cD(R)$ (resp. $\cD(R)^\wedge_{\varpi}$) there is an essentially unique fiber sequence
\begin{equation*}
	h_! h^* M \rightarrow M \rightarrow i_* i^* M \qquad \text{ resp. } \qquad h_!^\wedge h^{\wedge *} M \rightarrow M \rightarrow i_*^\wedge i^{\wedge *} M \qquad
\end{equation*}
which in turn implies that $h_! h^* M = M \otimes_R^L (\varpi)_{\perfd}$ (resp. $h_!^\wedge h^{\wedge *} M = M \cotimes_R^L (\varpi)_{\perfd} \simeq M \otimes_R^L (\varpi)_{\perfd}$).

Furthermore, by virtue of \cite[Proposition 7.2.1.10]{luriespectral} we learn that by restricting the functor $(-)^a$ (resp. $(-)^{\wedge a}$) to the full-subcategory ${}^\perp \cD((R/\varpi)_{\perfd})$ (resp. ${}^\perp \cD((R/\varpi)_{\perfd})^\wedge$) we get functors
\begin{equation*}
	(-)^a: {}^\perp \cD((R/\varpi)_{\perfd}) \rightarrow \cD(R)^a \qquad \text{ resp. } \qquad (-)^{\wedge a}: {}^\perp \cD((R/\varpi)_{\perfd})^\wedge \rightarrow \cD(R)^{\wedge a}_{\varpi}
\end{equation*}
which induce an equivalence of categories. We claim that we have the identity
\begin{equation*}
	(-)^a \circ h^* \simeq (-)^a \qquad \text{ resp. } \qquad (-)^{\wedge a} \circ h^{\wedge *} \simeq (-)^{\wedge a}
\end{equation*}
Indeed, it suffices to show that $(M \otimes_R^L (R/\varpi)_{\perfd})^a = 0$ (resp. $(M \cotimes_R^L (R/\varpi)_{\perfd})^a = 0$), but this follows from the fact that $M \otimes_R^L (R/\varpi)_{\perfd}$ is $\varpi$-almost zero and the definition of the functors $(-)^a$ (resp. $(-)^{\wedge a}$) from Construction \ref{const_almost_cat_semi_orth_right}. Hence, we can conclude that the functors $(-)^a$ (resp. $(-)^{\wedge a}$) admit left adjoints given by 
\begin{equation*}
	j_! :=  h_! \circ (-)^{a, -1}: \cD(R)^a \rightarrow \cD(R) \qquad \text{ resp. } \qquad j_!^{\wedge}:= h_!^\wedge \circ (-)^{\wedge a, -1}: \cD(R)^{\wedge a}_{\varpi} \rightarrow \cD(R)\picomp
\end{equation*}
where $(-)^{a,-1}$ is the inverse of the functor inducing the equivalence $(-)^a: {}^\perp \cD((R/\varpi)_{\perfd}) \simeq \cD(R)^a$; and $(-)^{\wedge a, -1}$ is defined similarly. We summarize the various functors we have constructed with the following diagram
\begin{cd}
	\cD(R)^a \ar[r, shift left = 1ex, "j_!"] & \ar[l, shift left = 1ex, "(-)^a"]
	\cD(R) \ar[r, shift left = 1ex, "i^*"] & 
	\cD((R/\varpi)_{\perfd})\ar[l, shift left = 1ex, "i_*"]
	&
	\cD(R)^{\wedge a}_\varpi\ar[r, shift left = 1ex, "j_!^\wedge"] & \ar[l, shift left = 1ex, "(-)^{\wedge a}"]
	\cD(R)_\varpi^\wedge \ar[r, shift left = 1ex, "i^*"] & 
	\cD((R/\varpi)_{\perfd})\ar[l, shift left = 1ex, "i_*"]
\end{cd}
which satisfy the following property: for every $M$ in $\cD(R)$ (resp. $\cD(R)^\wedge_\varpi$) there exists an essentially unique fiber sequence
\begin{equation*}
	j_! M^a \rightarrow M \rightarrow i_*i^* M \qquad \text{ resp. } \qquad j_!^\wedge (M)^{\wedge a} \rightarrow M \rightarrow i^{\wedge *} i^\wedge_* M
\end{equation*}
In particular, we have that $j_! M^a = M \otimes_R^L (\varpi)_{\perfd}$ (resp. $j_!^\wedge (M)^{\wedge a} = M \cotimes_R^L (\varpi)_{\perfd} \simeq M \otimes_R^L (\varpi)_{\perfd}$).
\end{const}

The following result provides us with an explicit description $j_* M^a$ (resp. $(j_*^\wedge (M)^{\wedge a})$) in terms of the internal hom to the category $\cD(R)$. 

\begin{const}\label{const_internal_hom} Let us explain how the internal hom of $\cD(R)$ is constructed and explain some subtleties regarding the construction. Fix $M,N \in \cD(R)$, we define a functor
\begin{equation*}
	R\uHom_R (M, N): \cD(R)^{\op} \longrightarrow \Spc \qquad Z \mapsto R\Hom_R (M \otimes_R^L Z, N)
\end{equation*}
It is easy to see that this functor preserves all limits, and since the category $\cD(R)$ is a presentable $\infty$-category it follows that the functor described above is representable by an object $R\uHom_R (M,N) \in \cD(R)$ called the internal hom of $\cD(R)$. The object $R\uHom_R (M,N) \in \cD(R)$ is completely characterized by the fact that there is a natural isomorphism of spaces $R\Hom_R(Z, \uHom_R(M,N)) = R\Hom_R (M \otimes_R^L Z, N)$ -- in other words, it satisfies the tensor-hom adjunction by construction.

Let us emphasize that during the construction of $R\uHom_R (-, -)$ it was critical that the category $\cD(R)$ was a presentable $\infty$-category in order to guarantee that the functor was representable. However, it is generally not true that the category $\cD(R)\picomp$ is presentable -- thus it is not clear how to define an internal hom for $\cD(R)\picomp$ in general. Nevertheless, we claim that for a fixed $M \in \cD(R)\picomp$ the functor
\begin{equation*}
	R\uHom_R ((\varpi)_{\perfd}, M): \cD(R)^{\wedge \op}_{\varpi} \rightarrow \Spc \qquad Z \mapsto R\Hom_R ((\varpi)_{\perfd} \cotimes^L_R Z, M)
\end{equation*}
is representable by the internal hom $R\uHom_R ((\varpi)_{\perfd}, M)$ computed in $\cD(R)$ and that it is derived $\varpi$-complete. Indeed, recall that $(\varpi)_{\perfd} \cotimes_R^L Z \simeq (\varpi)_{\perfd} \otimes_R^L M$, so the internal hom $R\uHom_R ((\varpi)_{\perfd}, M)$ computed in $\cD(R)$ gives us the desired functor when restricted to $\cD(R)\picomp$. Finally, we need to show that $R\uHom_R ((\varpi)_{\perfd}, M)$ is derived $\varpi$-complete whenever $M$ is; recall that a module $M$ is derived $\varpi$-complete if the inverse limit $R\lim (\cdots \rightarrow M \rightarrow M)$ vanishes, with transition maps given by multiplication by $\varpi$. By the identities
\begin{align*}
	& R\lim \Big (  \cdots \rightarrow R\Hom_R ((\varpi)_{\perfd} \cotimes_R^L Z, M) \rightarrow R\Hom_R ((\varpi)_{\perfd} \cotimes_R^L Z, M)  \Big ) \\
	& \simeq R\Hom_R \Big ( (\varpi)_{\perfd} \cotimes_R^L Z, R\lim (\cdots \rightarrow M \rightarrow M) \Big) \simeq R\Hom_R ((\varpi)_{\perfd} \cotimes^L_R Z, 0) \simeq 0
\end{align*}
where $Z$ ranges over all elements in $\cD(R)\picomp$ we can conclude that $R\uHom_R ((\varpi)_{\perfd}, M)$ is derived $\varpi$-complete.
\end{const}

\begin{lemma}\label{almost_unit_internal_hom} Let $M$ be an object of $\cD(R)$ (resp. $\cD(R)\picomp$). Then, the object $j_* M^a \in \cD(R)$ (resp. $j_*^\wedge (M)^{\wedge a} \in \cD(R)\picomp$) admits the following description in terms of the internal homs introduce in the previous construction
\begin{equation*}
	j_* M^a = R\uHom_{R} ((\varpi)_{\perfd}, M) \qquad \text{ resp. } \qquad j_*^\wedge (M)^{\wedge a} = R\uHom_{R} ((\varpi)_{\perfd}, M) 
\end{equation*}
\end{lemma}

\begin{proof} It suffices to show that for all $N$ in $\cD(R)$ (resp. $\cD(R)\picomp$) there exists a natural isomorphism
\begin{align*}
	&R\Hom_R (N, j_* M^a) \simeq R\Hom_R(N, R\uHom_R ((\varpi)_{\perfd}, M)) \\
	\text{resp.} \qquad &R\Hom_R (N, j_*^\wedge (M)^{\wedge a}) \simeq R\Hom_R(N, R\uHom_R ((\varpi)_{\perfd}, M))
\end{align*}
Consider the following sequence of isomorphisms obtained formally by the adjunctions among the functors $j_!$, $(-)^a$, and $j_*$ (resp. their completed variants)
\begin{align*}
	&R\Hom_R (N, j_* M^a) \simeq R\Hom_{R^a} (N^a, M^a) \simeq R\Hom_R(N \otimes_R^L (\varpi)_{\perfd}, M) \\
	\text{resp. } \qquad &R\Hom_R (N, j^\wedge_* (M)^{\wedge a}) \simeq R\Hom_{R^a} (N^{\wedge a}, M^{\wedge a}) \simeq R\Hom_R(N \cotimes_R^L (\varpi)_{\perfd}, M)
\end{align*}
Using the tensor-hom adjunction described in Construction \ref{const_internal_hom} the result follows.
\end{proof}

Finally, we want explain how to endow the categories $\cD(R)^a$ and $\cD(R)^{\wedge a}_{\varpi}$ with symmetric monoidal structures making the functors $(-)^a$ and $(-)^{\wedge a}$ symmetric monoidal. Recall the following definition from \cite[Definition 2.2.1.6]{lurieHA}.

\begin{defn}\label{compatible_with_monoidal_structure} Let $\cC$ be a symmetric monoidal category, and $\cD \subset \cC$ a full subcategory of $\cC$ such that the inclusion $\cD \subset \cC$ admits a left adjoint $L: \cC \rightarrow \cD$. We say that $L$ is compatible with the symmetric monoidal structure of $\cC$ if for every morphisms $X \rightarrow Y$ which becomes an isomorphism upon applying $L$, and any $Z \in \cC$, the induced map $X \otimes Z \rightarrow Y \otimes Z$ becomes an isomorphism upon applying $L$.
\end{defn}

\begin{const}\label{const_almost_monoidal_str} Our goal in this construction is to show that all the functors appearing in the following diagram, which we constructed in \ref{const_almost_comp_localization}, are compatible with the symmetric monoidal structures of the source.
\begin{cd}	
	\cD(R) \ar[d, "(-)^\wedge", swap] \ar[r, "(-)^a"]& \cD(R)^a \ar[d, "(-)^\wedge"] \\
	\cD(R)^\wedge_{\varpi} \ar[r, "(-)^{\wedge a}", swap]& \cD(R)_{\varpi}^{\wedge a}
\end{cd}
Once we establish that all the functors above are compatible with the symmetric monoidal structures of the source, it follows from \cite[Proposition 2.2.1.9]{lurieHA} that the target categories admit unique symmetric monoidal structures making the functors above symmetric monoidal.

As explained in \ref{const_complete_monoidal_str} it follows from \cite[Proposition 7.3.5.1]{luriespectral} that the derived $\varpi$-completion functor is compatible with the symmetric monoidal structure of $\cD(R)$. To show that $(-)^a$ is compatible with the symmetric monoidal structure of $\cD(R)$ recall that a morphism $M \rightarrow N$ becomes an isomorphism after applying $(-)^a$ if and only if $M \otimes_R^L (\varpi)_{\perfd} \rightarrow N \otimes_R^L (\varpi)_{\perfd}$, thus making it clear that $(M \otimes_R^L Z)^a \rightarrow (N \otimes_R^L Z)^a$ is an isomorphism for all $Z \in \cD(R)$. A similar argument shows that the functor $(-)^{\wedge a}$ is compatible with the monoidal structure on $\cD(R)^\wedge_\varpi$.

Recall how the functor $(-)^\wedge: \cD(R)^a \rightarrow \cD(R)^{\wedge a}_{\varpi}$ was defined in \ref{const_almost_comp_localization}: its was obtained by derived $\varpi$-completing an element $M \in \cD(R)^a$ when regarded as an object of $\cD(R)$ via the inclusion $j_*: \cD(R)^a \rightarrow \cD(R)$. Hence, we can conclude that a morphism $M \rightarrow N \in \cD(R)^a$ becomes an isomorphism upon derived $\varpi$-completion if and only if $\Cone(M \rightarrow N)[-1]$ is $\varpi$-local. Therefore, it suffices to show that 
\begin{equation*}
	R\uHom_R((\varpi)_{\perfd}, \Cone(M \rightarrow N)[-1] \otimes_R^L Z)
\end{equation*}
is $\varpi$-local for all $Z \in \cD(R)$, whenever $\Cone(M \rightarrow N)[-1]$ is $\varpi$-local. By the definition of $\varpi$-local it suffices to show that for every $K \in \cD(R)$ which is $\varpi$-nilpotent the mapping space
\begin{equation*}
	R\Hom_R \Big ( K, R\uHom_R((\varpi)_{\perfd}, \Cone(M \rightarrow N)[-1] \otimes_R^L Z) \Big ) \simeq *
\end{equation*}
is contractible. But this follows from the fact that $\Cone(M \rightarrow N)[-1] \otimes_R^L Z$ is $\varpi$-local by \cite[Proposition 7.2.4.9]{luriespectral}, and the fact that $K \otimes^L_R (\varpi)_{\perfd}$ is $\varpi$-nilpotent by \cite[Proposition 7.1.1.12]{luriespectral}, together with the tensor-hom adjunction.

To summarize, we have shown that the categories $\cD(R)^a$ and $\cD(R)^{\wedge a}_{\varpi}$ admit symmetric monoidal structures, denoted by $(- \otimes_R^{L a} -)$ and $(- \cotimes^{L a}_R -)$ respectively, making all the functors on the above diagram symmetric monoidal. Furthermore, by virtue of \cite[Proposition 2.2.1.9]{lurieHA} we have the following pair of commutative diagrams
\begin{cd}
	\CAlg(\cD(R)^{\wedge a}_{\varpi}) \ar[r, hook] \ar[d, hook, "j_*^\wedge", swap] & \CAlg(\cD(R)^a) \ar[d, hook, "j_*"] &&
	\CAlg(\cD(R)) \ar[d, "(-)^\wedge", swap] \ar[r, "(-)^a"]& \CAlg(\cD(R)^a) \ar[d, "(-)^\wedge"]\\
	\CAlg(\cD(R)^\wedge_\varpi) \ar[r, hook] & \CAlg(\cD(R)) && 
	\CAlg(\cD(R)^\wedge_{\varpi}) \ar[r, "(-)^{\wedge a}", swap]& \CAlg(\cD(R)_{\varpi}^{\wedge a})
\end{cd}
where the functors on the right hand side diagram are left adjoint to the functors on the right hand side diagram. Compare this diagrams with those in \ref{const_almost_comp_localization}.
\end{const}

\subsection{The abelian categories}

Throughout this section $R$ is an integral perfectoid ring which is $\varpi$-complete with respect to some $\varpi \in R$ where $\varpi^p$ divides $p$, and such that $R$ is $\varpi$-torsion free. The goal of this section is to show that under the assumption that $R$ is $\varpi$-torsion free most of the result in the previous section admit abelian analogs.

We begin by defining $t$-structures on $\cD(R)^a$ and $\cD(R)^{\wedge a}_{\varpi}$.

\begin{lemma}\label{almost_j!_t_exact} The functors
\begin{equation*}
	j_!(-)^a: \cD(R) \rightarrow  \cD(R) \qquad \text{ and } \qquad j_!^\wedge(-)^{\wedge a}: \cD(R)_\varpi^\wedge \rightarrow  \cD(R)_\varpi^\wedge
\end{equation*}
are $t$-exact, where the $t$-structure on $\cD(R)_\varpi^\wedge$ is the one from Construction \ref{complete_t_structure}.
\end{lemma}

\begin{proof} Since $\varpi$ admits compatible $p$-power roots up to a unit, we may assume without loss of generality that $\varpi$ admits compatible $p$-power roots. Recall from Corollary \ref{perfectoid_torsion} that $R[\varpi] = R[\varpi^{1/p^\infty}] = R[\varpi^\infty]$, which implies that $R$ is $\varpi^q$-torsion free for all $q \in \ZZ[1/p]$; therefore, for any $q \in \ZZ[1/p]$ the multiplication map $- \times \varpi^q: R \longrightarrow (\varpi^q)$ is an isomorphism -- this shows that $(\varpi^q)$ is a free $R$-module.

We claim that $(\varpi)_{\perfd}$ is a flat $R$-module; indeed, it suffices to show that $(\varpi)_{\perfd}$ is a filtered colimit of free $R$-modules, which follows from the following identities
\begin{equation*}
	(\varpi)_{\perfd} \simeq (\varpi^{1/p^\infty}) \simeq \colim ( (\varpi^{1/p^n}) \hookrightarrow (\varpi^{1/p^{n+1}}) \hookrightarrow \cdots)
\end{equation*}
This shows that $j_! (-)^a = - \otimes_R^L (\varpi)_{\perfd}$ is $t$-exact; and the identity $- \otimes_R^L (\varpi)_{\perfd} = - \cotimes^L_R (\varpi)_{\perfd}$ shows that $j_!^\wedge (-)^{\wedge a}$ is also $t$-exact.
\end{proof}

\begin{const}\label{almost_t_structure} Regarding the category $\cD(R)^{a}$ (resp. $\cD(R)^{\wedge a}_{\varpi}$) as a full-subcategory of $\cD(R)$ (resp. $\cD(R)\picomp$) via the inclusions
\begin{equation*}
	j_!: \cD(R)^a \hookrightarrow \cD(R) \qquad \text{ resp. } \qquad j_!^{\wedge}: \cD(R)^{\wedge a}_{\varpi} \hookrightarrow \cD(R)^{\wedge}_{\varpi}
\end{equation*}
We claim that the following subcategories of $\cD(R)^a$ (resp. $\cD(R)^{\wedge a}_{\varpi}$) determine a $t$-structure
\begin{align*}
	(\cD(R)^a)^{\ge 0} := \cD(R)^{\ge 0} \cap \cD(R)^a &\qquad 
	(\cD(R)^a)^{\le 0} := \cD(R)^{\le 0} \cap \cD(R)^a \\
	\text{resp.} \qquad (\cD(R)_{\varpi}^{\wedge a})^{\ge 0} := (\cD(R)_{\varpi}^{\wedge})^{\ge 0} \cap \cD(R)_{\varpi}^{\wedge a} &\qquad 
	(\cD(R)_{\varpi}^{\wedge a})^{\le 0} := (\cD(R)_{\varpi}^{\wedge})^{\le 0} \cap \cD(R)_{\varpi}^{\wedge a}
\end{align*}
where the $t$-structure on $\cD(R)\picomp$ is that of Construction \ref{complete_t_structure}. In order to show that both of this pairs of subcategories determine $t$-structures on $\cD(R)^a$ and $\cD(R)^{\wedge a}_{\varpi}$ it suffices to show that the following pair of morphisms, obtained from the counit of the adjunction, are isomorphisms
\begin{equation*}
	(\tau^{\ge n} M) \otimes_R^L (\varpi)_{\perfd} \rightarrow \tau^{\ge n} M \qquad  (\tau^{\le n} M ) \otimes_R^L (\varpi)_{\perfd} \rightarrow \tau^{\le n} M
\end{equation*}
for any object $M$ in $\cD(R)^a$. From the fiber sequence $\tau^{\le n} M \rightarrow M \rightarrow \tau^{\ge n+1} M$ and the fact that the counit of adjunction $M \otimes^L_R (\varpi)_{\perfd} \rightarrow M$ is an isomorphisms, we conclude that the boundary map must be an isomorphism
\begin{equation*}
	\Big (\tau^{\ge n+1} M \Big ) \Big/^L \Big ( (\tau^{\ge n+1} M) \otimes_R^L (\varpi)_{\perfd} \Big ) \longrightarrow 
	\Big (\tau^{\le n} M \Big ) \Big/^L \Big ( (\tau^{\le n} M) \otimes_R^L (\varpi)_{\perfd} \Big ) \Big[1 \Big]
\end{equation*}
But since the left hand is concentrated in degrees $[n, \infty)$ and the right hand side is concentrated in degrees $(-\infty, n-1]$ it follows that both of them must be zero, proving the result. Let us emphasize that we are using in a critical way the $t$-exactness of $- \otimes_R^L (\varpi)_{\perfd}$ (\ref{almost_j!_t_exact}) to guarantee that the left hand side is concentrated in degrees $[n, \infty)$.
\end{const}

We are now ready to introduce the abelian analogs of the categories $\cD(R)^a$ and $\cD(R)^{\wedge a}_{\varpi}$.

\begin{defn} Recall that if a stable $\infty$-category $\cC$ admits a $t$-structure there exists an abelian subcategory $\cC^\heartsuit \subset \cC$ defined as $\cC^\heartsuit := \cC^{\ge 0} \cap \cC^{\le 0}$. For us, the following two abelian categories will be of interest
\begin{enumerate}[(1)]
	\item The almost category of $R$-module, denoted by $\Mod_R^a$, is the abelian full-subcategory of $\cD(R)^a$ defined as the heart of the $t$-structure on $\cD(R)^a$ from Construction \ref{almost_t_structure}. In other words, $\Mod_R^a := \cD(R)^{a \ge 0} \cap \cD(R)^{a \le 0}$.
	\item The almost category of derived $\varpi$-complete $R$-modules, denoted by $\Mod_{R, \varpi}^{\wedge a}$, is the abelian full-subcategory of $\cD(R)^{\wedge a}_{\varpi}$ defined as the heart of the $t$-structure on $\cD(R)^{\wedge a}_{\varpi}$ from Construction \ref{almost_t_structure}. In other words, $\Mod_{R, \varpi}^{\wedge a} := \cD(R)^{\wedge a \ge 0}_{\varpi} \cap \cD(R)^{\wedge a \le 0}_{\varpi}$.
\end{enumerate}
\end{defn}

Now that the basic definitions are in place, let us explain how to get abelian versions of the functors we introduced in the previous section.

\begin{const}\label{const_abelian_almost_funct} From Lemma \ref{almost_j!_t_exact} and Construction \ref{almost_t_structure} it follows that the pair of adjoint functors
\begin{align*}
	j_!: \cD(R)^a \hookrightarrow \cD(R) &\qquad (-)^a: \cD(R) \rightarrow \cD(R)^a \\
	j_!^\wedge: \cD(R)^{\wedge a}_{\varpi} \hookrightarrow \cD(R)^{\wedge}_{\varpi} &\qquad (-)^{\wedge a}: \cD(R)^\wedge_{\varpi} \rightarrow \cD(R)^{\wedge a}_{\varpi}
\end{align*}
are $t$-exact. Hence, they induce a pair exact adjoint functors of abelian categories
\begin{align*}
	j_!: \Mod_R^a \hookrightarrow \Mod_R &\qquad (-)^a: \Mod_R \rightarrow \Mod_R^a \\
	j_!^\wedge: \Mod_{R, \varpi}^{\wedge a} \hookrightarrow \Mod_{R, \varpi}^{\wedge} &\qquad (-)^a: \Mod_{R, \varpi}^{\wedge} \rightarrow \Mod_{R, \varpi}^{\wedge a}
\end{align*}
with $j_!$ and $j_!^\wedge$ being fully faithful and left adjoint to $(-)^a$ and $(-)^{\wedge a}$ respectively. In particular, we have canonical identifications 
\begin{equation*}
	j_! M^a \simeq M \otimes_R (\varpi)_{\perfd} \qquad j_!^{\wedge} (M)^{\wedge a} \simeq M \otimes_R (\varpi)_{\perfd}
\end{equation*}
for any $M$ in $\Mod_R$ and $\Mod_{R, \varpi}^{\wedge}$ respectively.

On the other hand, the fully faithful functors
\begin{equation*}
	j_*: \cD(R)^a \hookrightarrow \cD(R) \qquad j_*^\wedge: \cD(R)^{\wedge a}_{\varpi} \hookrightarrow \cD(R)^{\wedge}_{\varpi}
\end{equation*}
are only left $t$-exact. In order to see the failure of $t$-exactness recall that $(-)^a$ and $(-)^{\wedge a}$ are $t$-exact, thus it suffices to show that the internal hom $R\uHom_R((\varpi)_{\perfd}, M)$ is a complex concentrated in degrees $[0,1]$ for $M \in \Mod_R$. From the isomorphism $(\varpi)_{\perfd} \simeq \colim (\varpi^{\prime 1/p^n})$ we deduce the identity
\begin{equation*}
	R\uHom_R((\varpi)_{\perfd}, M)  \simeq R\lim R\uHom_R ((\varpi^{\prime 1/p^n}), M)
\end{equation*}
and recall that $\ZZ_{\le 0}$-indexed limits can be expressed as
\begin{equation*}
	R\lim R\uHom_R ((\varpi^{\prime 1/p^n}), M) \simeq \Cone \Big(\prod R\uHom_R ((\varpi^{\prime 1/p^n}), M) \rightarrow \prod R\uHom_R ((\varpi^{\prime 1/p^n}), M) \Big)[-1]
\end{equation*}
Hence, since $R\uHom_R((\varpi^{\prime 1/p^n}), M)$ is concentrated in degree zero, by virtue of $(\varpi^{\prime 1/p^n})$ being a free $R$-module, it follows that $R\uHom_R((\varpi)_{\perfd}, M)$ is an $R$-complex concentrated in degrees $[0,1]$. Therefore at the abelian level we get left exact functors
\begin{equation*}
	\tau^{\le 0} j_*: \Mod_R^a \rightarrow \Mod_R \qquad \tau^{\le 0} j_*^{\wedge}: \Mod_{R, \varpi}^{\wedge a} \rightarrow \Mod_{R, \varpi}^{\wedge}
\end{equation*}
such that $\tau^{\le 0} j_*$ is the right adjoint to $(-)^a$ and $\tau^{\le 0} j_*^{\wedge}$ is the right adjoint to $(-)^{\wedge a}$.
\end{const}

As we did in Lemma \ref{almost_unit_internal_hom} we want to express the objects $\tau^{\le 0} j_* M^a \in \Mod_R$ and $\tau^{\le 0} j^\wedge_* (M)^{\wedge a}$ in terms of the internal hom of $\Mod_R$. 

\begin{const}\label{abelian_internal_hom} Recall how the internal hom of $\Mod_R$ is defined, for a fixed $M,N \in \Mod_R$ we have a functor
\begin{equation*}
	\uHom_R (M, N): \Mod_R^{\op} \rightarrow \Set \qquad Z \mapsto \Hom_R(M \otimes_R Z, N)
\end{equation*}
And since this functor preserves all limits, the presentability of $\Mod_R$ implies that there exists an object $\uHom(M,N) \in \Mod_R$ representing the functor above. Moreover, as in Construction \ref{const_internal_hom} the category $\Mod_{R, \varpi}^{\wedge}$ is not presentable, so it has no intrinsic notion of internal hom. However, we claim that the functor
\begin{equation*}
	\uHom_R((\varpi)_{\perfd}, M): \Mod_{R, \varpi}^{\wedge} \rightarrow \Set \qquad Z \mapsto \Hom_R(H^0((\varpi)_{\perfd} \cotimes^L_R Z), M)
\end{equation*}
is representable by the internal hom $\uHom_R ((\varpi)_{\perfd}, M)$ computed in $\Mod_R$. Indeed, from the identity $(\varpi)_{\perfd} \cotimes_R^L Z \simeq (\varpi)_{\perfd} \otimes_R Z$ it follows that the functor $\uHom((\varpi)_{\perfd}, M): \Mod_{R, \varpi}^{\wedge} \rightarrow \Set$ described above is represented by the internal hom computed in $\Mod_R$. Finally, we need to show that if $M$ is a derived $\varpi$-complete $R$-module, then so is $\uHom((\varpi)_{\perfd}, M)$. Recall that an $R$-module $M$ is derived $\varpi$-complete if and only if $R\lim (\cdots \rightarrow M \rightarrow M) = 0$ with transition maps given by multiplication by $\varpi$. So if $M$ is derived $\varpi$-complete we have the following identities
\begin{align*}
	&R\lim \Big ( \cdots \rightarrow	\uHom_R ((\varpi)_{\perfd}, M) \rightarrow 	\uHom_R ((\varpi)_{\perfd}, M)  \Big) \\
	&\simeq \uHom_R ((\varpi)_{\perfd}, R\lim (\cdots \rightarrow M \rightarrow M)) \simeq 0
\end{align*}
showing that $\uHom_R ((\varpi)_{\perfd}, M)$ is derived $\varpi$-complete.
\end{const}

\begin{lemma}\label{abelian_almost_unit_internal_hom} Let $M$ be an object if $\Mod_R$ (resp. $\Mod_{R, \varpi}^{\wedge}$). Then, the object $\tau^{\le 0} j_* M^a \in \Mod_R$ (resp. $\tau^{\le 0} j_*^\wedge (M)^{\wedge a} \in \Mod_{R, \varpi}^{\wedge}$) admits the following description in terms of the internal homs introduced in the previous construction
\begin{equation*}
	\tau^{\le 0} j_* M^a \simeq \uHom_R ((\varpi)_{\perfd}, M) \qquad \text{ resp. } \qquad \tau^{\le 0} j_*^{\wedge} (M)^{\wedge a} \simeq \uHom_R ((\varpi)_{\perfd}, M) 
\end{equation*}
\end{lemma}

\begin{proof} It suffices to show that for all $N \in \Mod_R$ (resp. $\Mod_{R, \varpi}^\wedge$) there exists a natural isomorphism
\begin{align*}
	&\Hom_R (N, \tau^{\le 0} j_* M^a) \simeq \Hom_R (N, \uHom_R ((\varpi)_{\perfd}, M)) \\
	&\Hom_R (N, \tau^{\le 0} j^{\wedge}_* M^{\wedge a}) \simeq \Hom_R (N, \uHom_R ((\varpi)_{\perfd}, M))
\end{align*}
Consider the following sequence of isomorphisms obtained formally by the adjunctions among the functors $j_!$, $(-)^a$ and $\tau^{\le 0} j_*$ (resp. their completed variants)
\begin{align*}
	&\Hom_R (N, \tau^{\le 0} j_* M^a) \simeq \Hom_{R^a} (N^a, M^a) \simeq \Hom_R (N \otimes_R (\varpi)_{\perfd}, M) \\
	&\Hom_R (N, \tau^{\le 0} j_*^{\wedge} M^{\wedge a}) \simeq \Hom_{R^a} (N^{\wedge a}, M^{\wedge a}) \simeq \Hom_R (H^0(N \cotimes^L_R (\varpi)_{\perfd}), M)
\end{align*}
Using the tensor-hom adjunction and the identity $N \otimes_R (\varpi)_{\perfd} \simeq N \cotimes^L_R (\varpi)_{\perfd}$ the result follows.
\end{proof}

\begin{notation}\label{almost_elements} We will often denote the functor $\tau^{\le 0} j_* (-)^a$ (resp. $\tau^{\le 0} j_*^\wedge (-)^{\wedge a}$) by
\begin{align*}
	&(-)_*: \Mod_R \rightarrow \Mod_R \qquad M \mapsto M_* \simeq \uHom_R ((\varpi)_{\perfd}, M) \\
	\text{resp. } \qquad &(-)_*^\wedge: \Mod_R^\wedge \rightarrow \Mod_R^\wedge \qquad M \mapsto M_*^\wedge \simeq \uHom_R ((\varpi)_{\perfd}, M)
\end{align*}
and we will call $M_*$ (resp. $M^\wedge_*$) the almost element of $M$. Moreover, the unit of the adjunction provides us with a functorial morphisms $M \rightarrow M_*$ (resp. $M \rightarrow M_*^\wedge$) described by $m \mapsto (\varepsilon \mapsto \varepsilon m) \in  \uHom_R ((\varpi)_{\perfd}, M)$, where $\varepsilon \in (\varpi)_{\perfd}$. When $M$ is assumed to be $\varpi$-torsion free the $R$-module $M_*$ (resp. $M^\wedge_*$) can be described as the submodule $M_* \subset M[1/\varpi]$ (resp. $M_*^\wedge \subset M[1/\varpi]$) where $m \in M_*$ (resp. $m \in M_*^\wedge$) if and only if $\varepsilon m \in M$ for all $\varepsilon \in (\varpi)_{\perfd}$.
\end{notation}

\begin{lemma}\label{abelian_j*_fully_faithful} The functors
\begin{equation*}
	\tau^{\le 0} j_*: \Mod_R^a \rightarrow \Mod_R \qquad \tau^{\le 0} j_*^\wedge: \Mod_{R, \varpi}^{\wedge a} \rightarrow \Mod_{R, \varpi}^\wedge 
\end{equation*}
are fully faithful.
\end{lemma}

\begin{proof} It suffices to show that the units of the adjunctions
\begin{equation*}
	\uHom_R ((\varpi)_{\perfd}, M) \longrightarrow \uHom_R ((\varpi)_{\perfd}, \uHom_R ((\varpi)_{\perfd}, M)) \qquad (\varepsilon \mapsto m) \mapsto (\varepsilon^{\prime} \mapsto (\varepsilon \mapsto \varepsilon^{\prime} m))
\end{equation*}
is an isomorphism. By the identities $(\varpi)_{\perfd} \otimes_R (\varpi)_{\perfd} \simeq (\varpi)_{\perfd}$ and $- \cotimes^L_R (\varpi)_{\perfd} \simeq - \otimes_R (\varpi)_{\perfd}$ and the tensor hom adjunction we get an isomorphism
\begin{equation*}
	\uHom_R ((\varpi)_{\perfd}, \uHom_R ((\varpi)_{\perfd}, M)) \rightarrow \uHom_R ((\varpi)_{\perfd}, M) \qquad (\varepsilon^{\prime} \mapsto (\varepsilon \mapsto \varepsilon^{\prime} m)) \mapsto (\varepsilon \varepsilon^{\prime} \mapsto \varepsilon^{\prime} m)
\end{equation*}
proving that the units of the adjunctions are isomorphisms.
\end{proof}

With the basic functorialities established, we now show that the categories $\Mod_R^a$ and $\Mod_{R, \varpi}^{\wedge a}$ inherit symmetric monoidal structures from the categories $\cD(R)^a$ and $\cD(R)^{\wedge a}_{\varpi}$ respectively.

\begin{const}\label{const_ab_almost_monoidal_str} We follow a procedure similar to that of (\ref{derived_complete_cat}) and (\ref{const_complete_monoidal_str}) to obtain the symmetric monoidal structures on $\Mod_R^a$ and $\Mod_{R, \varpi}^{\wedge a}$. First, we need to show that the full-subcategories $\cD(R)^{a \le 0} \subset \cD(R)^a$ and $\cD(R)^{\wedge a \le 0}_{\varpi} \subset \cD(R)^{\wedge a}_{\varpi}$ are stable under the symmetric monoidal structures of $\cD(R)^a$ and $\cD(R)^{\wedge a}_{\varpi}$ respectively. Let $M,N$ be elements of $\cD(R)^{a \le 0}$ (resp. $\cD(R)^{\wedge a \le 0}_{\varpi}$), and recall that the tensor product of $M$ and $N$ in $\cD(R)^{a}$ (resp. $\cD(R)^{\wedge a}_{\varpi}$) admit the following description
\begin{equation*}
	M \otimes_R^{L a} N \simeq (j_! M \otimes_R^L j_! N)^a \qquad \text{ resp. } \qquad M \cotimes_R^{L a} N \simeq ((j_!^\wedge M) \cotimes_R^L (j_!^\wedge N))^{\wedge a}
\end{equation*}
From Lemma \ref{almost_j!_t_exact} we learn that the functors $(-)^a$ and $j_!$ (resp. $(-)^{\wedge a}$ and $j_!^\wedge$) are $t$-exact, and since $M,N$ being connective implies that $M \otimes_R^L N$ (resp. $M \cotimes_R^L N$) is connective, it follows that $M \otimes_R^{L a} N \in \cD(R)^{a \le 0}$ (resp. $M \cotimes_R^{L a} N \in \cD(R)^{\wedge a \le 0}_{\varpi}$). Thus, by \cite[Proposition 2.2.1.1]{lurieHA} it follows that $\cD(R)^{a \le 0}$ and $\cD(R)^{\wedge a \le 0}_{\varpi}$ inherit symmetric monoidal structures from $\cD(R)^a$ and $\cD(R)^{\wedge a}_{\varpi}$ as they are stable under tensor products.
	
Consider the following pair of commutative diagrams
\begin{cd}
	\Mod_{R, \varpi}^{\wedge a} \ar[r, hook] \ar[d, hook] & \cD(R)^{\wedge a \le 0}_{\varpi} \ar[d, hook] &&
	\cD(R)^{a \le 0} \ar[r, "\tau^{\ge 0}"] \ar[d, "(-)^\wedge", swap] & \Mod_{R}^a  \ar[d, "H^0(-^\wedge)"]\\
	\Mod_R^a \ar[r, hook] & \cD(R)^{a \le 0} &&
	\cD(R)^{\wedge a \le 0}_{\varpi} \ar[r, "\tau^{\ge 0}"] & \Mod_{R, \varpi}^{\wedge a}
\end{cd}
where the functors on the left are the canonical inclusions -- it is not hard to see that the canonical inclusion $\cD(R)^{\wedge a}_{\varpi} \subset \cD(R)^a$ is $t$-exact -- and the functors on the right are the left adjoints to the inclusions. By \cite[Proposition 2.2.1.8]{lurieHA} it follows that the functors $\tau^{\ge 0}$ are compatible with the symmetric monoidal structure of the source (\ref{compatible_with_monoidal_structure}), endowing the categories $\Mod_R^a$ and $\Mod_{R, \varpi}^{\wedge a}$ with an essentially unique symmetric monoidal structures making the truncation functors $\tau^{\ge 0}$ symmetric monoidal. We denote the symmetric monoidal structure on $\Mod_R^a$ by $-\otimes_R^a -$, and the symmetric monoidal structure on $\Mod_{R, \varpi}^{\wedge a}$ by $H^0(- \cotimes^{L a}_R -)$ -- we prefer this notation over $- \cotimes_R^a -$ in other to emphasize that we are taking derived $\varpi$-completions (cf. \ref{classical_complete_tensor}). 

Furthermore, we showed in Construction \ref{const_almost_monoidal_str} that the functor $(-)^{\wedge}: \cD(R)^{a \le 0} \rightarrow \cD(R)^{\wedge a \le 0}_{\varpi}$ is compatible with the symmetric monoidal structure of $\cD(R)^{a \le 0}$. It remains to show that $H^0(-^\wedge): \Mod_R^a \rightarrow \Mod_{R, \varpi}^{\wedge a}$ is compatible with the symmetric monoidal structure of $\Mod_R^a$. Recall how the functor $H^0(-^\wedge)$ is defined: for an object $M \in \Mod_R^a \subset \cD(R)^{a \le 0}$, we first perform the derived $\varpi$-completion, regarding $M$ as an object of $\cD(R)^{a \le 0}$, and then we truncate $M^\wedge$. Thus, we need to show that if a morphism $M \rightarrow N$ in $\Mod_R^a$ induces an isomorphism $\tau^{\ge 0} M^\wedge \rightarrow \tau^{\ge 0} N ^\wedge$, then for any $Z \in \Mod_R^a$ the induced map
\begin{equation*}
	\tau^{\ge 0} (M \otimes^{L a}_R Z)^{\wedge} \rightarrow \tau^{\ge 0} (N \otimes^{L a}_R Z)^{\wedge} 
\end{equation*}
is an isomorphism. But this follows from the fact that the derived $\varpi$-completion $(-)^\wedge$ functor and the truncation functor $\tau^{\ge 0}$ are compatible with the symmetric monoidal structures of their sources.

To summarize we have shown that the categories $\Mod_R^a$ and $\Mod_{R, \varpi}^{\wedge a}$ admit symmetric monoidal structures, denoted by $(- \otimes^a_R -)$ and $H^0(- \cotimes^{L a} -)$ respectively, making the truncation functors $\tau^{\ge 0}$ and the derived completion functors $(-)^\wedge$ and $H^0(-^\wedge)$ symmetric monoidal. Furthermore, by \cite[Proposition 2.2.1.9]{lurieHA} it follows that the above commutative diagrams induce commutative diagrams at the level of commutative algebra objects
\begin{cd}
	\CAlg(\Mod_{R, \varpi}^{\wedge a}) \ar[r, hook] \ar[d, hook] & \CAlg(\cD(R)^{\wedge a \le 0}_{\varpi}) \ar[d, hook] &&
	\CAlg(\cD(R)^{a \le 0}) \ar[r, "\tau^{\ge 0}"] \ar[d, "(-)^\wedge", swap] & \CAlg(\Mod_{R}^a)  \ar[d, "H^0(-^\wedge)"]\\
	\CAlg(\Mod_R^a) \ar[r, hook] & \CAlg(\cD(R)^{a \le 0}) &&
	\CAlg(\cD(R)^{\wedge a \le 0}_{\varpi}) \ar[r, "\tau^{\ge 0}"] & \CAlg(\Mod_{R, \varpi}^{\wedge a})
\end{cd}
where the morphisms on the right hand side are the left adjoints to the morphisms on the left hand side.
\end{const}

\begin{prop}\label{abelian_j*_sym_monoidal} The pair of functors
\begin{equation*}
	(-)^a: \Mod_R \rightarrow \Mod_R^a \qquad (-)^{\wedge a}: \Mod_{R, \varpi}^{\wedge} \rightarrow \Mod_{R, \varpi}^{\wedge a}
\end{equation*}
are symmetric monoidal with respect to the symmetric monoidal structures on $\Mod_R^a$ and $\Mod_{R, \varpi}^{\wedge a}$ described in Construction \ref{const_ab_almost_monoidal_str}. As a direct consequence we learn that the right adjoint fully faithful functors
\begin{equation*}
	\tau^{\le 0} j_*: \Mod_R^a \hookrightarrow \Mod_R \qquad \tau^{\le 0} j_*^\wedge: \Mod_{R, \varpi}^{\wedge a} \hookrightarrow \Mod_{R, \varpi}^{\wedge}
\end{equation*}
are lax-symmetric monoidal. In particular, this implies that we get adjoint functors at the level of commutative algebra objects
\begin{align*}
	&(-)^a: \CAlg(\Mod_R) \rightarrow \CAlg(\Mod_R^a) \qquad \tau^{\le 0} j_*: \CAlg(\Mod_R^a) \hookrightarrow \CAlg(\Mod_R) \\
	&(-)^a: \CAlg(\Mod_{R, \varpi}^{\wedge}) \rightarrow \CAlg(\Mod_{R, \varpi}^{\wedge a}) \qquad \tau^{\le 0} j_*^\wedge: \CAlg(\Mod_{R, \varpi}^{\wedge a}) \hookrightarrow \CAlg(\Mod_{R, \varpi}^\wedge)
\end{align*}
with $\tau^{\le 0} j_*$ and $\tau^{\le 0} j_*^\wedge$ being fully faithful functors.
\end{prop}

\begin{proof} Consider the following pair of commutative diagrams
\begin{cd}
	\cD(R)^{\le 0} \ar[r, "\tau^{\ge 0}"] \ar[d, "(-)^a", swap] & \Mod_R \ar[d, "(-)^a"] &&
	 \cD(R)^{\wedge \le 0}_{\varpi} \ar[r, "\tau^{\ge 0}"] \ar[d, "(-)^{\wedge a}", swap] & \Mod_{R, \varpi}^{\wedge} \ar[d, "(-)^{\wedge a}"]\\
	\cD(R)^{a \le 0} \ar[r, "\tau^{\ge 0}"] & \Mod_R^a &&
	\cD(R)^{\wedge a \le 0}_{\varpi} \ar[r, "\tau^{\ge 0}"] & \Mod_{R, \varpi}^{\wedge a}
\end{cd}
We know from Construction \ref{const_almost_monoidal_str} that the functors $(-)^a: \cD(R)^{\le 0} \rightarrow \cD(R)^{a \le 0}$ and $(-)^{\wedge a}: \cD(R)^{\wedge}_{\varpi} \rightarrow \cD(R)^{\wedge a}_{\varpi}$ are symmetric monoidal, and from Construction \ref{const_ab_almost_monoidal_str} we learn that all truncation functors $\tau^{\le 0}$ are symmetric monoidal. This shows that $(-)^a: \Mod_R \rightarrow \Mod_R^a$ and $(-)^{\wedge a}: \Mod_{R, \varpi}^{\wedge} \rightarrow \Mod_{R, \varpi}^{\wedge a}$ are symmetric monoidal, as desired.
\end{proof}

We conclude this section by showing how our constructions compare to those of Gabber-Ramero \cite{gabber_ramero_almost}.

\begin{rem} From Lemma \ref{abelian_j*_fully_faithful} we learn that $\Mod_R^a$ can be regarded as a full-subcategory of $\Mod_R$, together with an exact left adjoint to the inclusion $(-)^a: \Mod_R \rightarrow \Mod_R^a$ -- in particular, for a module $M \in \Mod_R$ we have that $M^a \simeq 0$ if and only if $M_* \simeq 0$. We claim that $M^a \simeq 0$ if and only if $M$ is a $\varpi$-almost zero module. Indeed, if $M$ is $\varpi$-almost zero then $M_* \simeq 0$ by Lemma \ref{abelian_almost_unit_internal_hom}; on the other hand, if $M_* \simeq 0$ then $M$ must be $\varpi$-almost zero again by Lemma \ref{abelian_almost_unit_internal_hom}.

This characterizes the exact functor $(-)^a: \Mod_R \rightarrow \Mod_R^a$ with the following universal property: for any exact functor $F: \Mod_R \rightarrow \cC$ such that $F(M) \simeq 0$ if $M$ is $\varpi$-almost zero, then there exists a unique factorization
\begin{equation*}
	F: \Mod_R \rightarrow \Mod_R^a \rightarrow \cC
\end{equation*}
showing that our abelian category $\Mod_R^a$ enjoys the same universal property as the almost category of \cite{gabber_ramero_almost}, proving their equivalence. Furthermore, in Proposition \ref{abelian_j*_sym_monoidal} we show that the functor $(-)^a: \Mod_R \rightarrow \Mod_R^a$ is symmetric monoidal and that this characterizes the monoidal structure on $\Mod_R^a$; this is the same definition of the monoidal structure on the almost category taken in \cite{gabber_ramero_almost}.
\end{rem}

\newpage

\section{Integral Closures}\label{sect_int_closure}

\subsection{\texorpdfstring{$\varpi$}{pi}-torsion free algebras}

Throughout this section $R$ is an integral perfectoid ring which is $\varpi$-complete with respect to some $\varpi \in R$ where $\varpi^p$ divides $p$, and such that $R$ is $\varpi$-torsion free.

This section we will mostly be concerned with various categories of commutative $R$-algebras, we begin by recalling their definitions and some of the results we have already established in previous sections.

\begin{enumerate}[(1)]
	\item The category $\CAlg_R := \CAlg(\Mod_R)$ of commutative $R$-algebras, which we introduced in \ref{derived_complete_cat}; and its full subcategory $\tau^{\le 0}j_*: \CAlg_R^a \hookrightarrow \CAlg_R$, where $\CAlg_{R}^a := \CAlg(\Mod_R^a)$, of almost commutative $R$-algebras which we introduced in \ref{const_ab_almost_monoidal_str}.
	\item The category $\CAlg_R^\wedge := \CAlg(\Mod_{R, \varpi}^{\wedge})$ of derived $\varpi$-complete commutative $R$-algebras, which we introduced in \ref{complete_algebra_objects}; and its full subcategory $\tau^{\le 0} j_*^\wedge: \CAlg_R^{\wedge a} \hookrightarrow \CAlg_R^{\wedge}$, where $\CAlg_{R}^{\wedge a} := \CAlg(\Mod_{R, \varpi}^{\wedge a})$, of derived $\varpi$-complete almost commutative $R$-algebras which we introduced in \ref{const_ab_almost_monoidal_str}.
\end{enumerate}
This categories have fully faithful functors relating them
\begin{cd}
	\CAlg_R^{\wedge a} \ar[r, hook, "\tau^{\le 0} j_*^\wedge"] \ar[d, hook] & \CAlg_R^{\wedge} \ar[d, hook] \\
	\CAlg_R^{a} \ar[r, hook, "\tau^{\le 0} j_*"] & \CAlg_R
\end{cd}
and each of this functors admit left adjoints that fit into the following commutative diagram
\begin{cd}
	\CAlg_R \ar[r, "H^0(-^\wedge)"] \ar[d, "(-)^a", swap] & \CAlg_R^{\wedge} \ar[d, "(-)^{a}"] \\
	\CAlg_R^a \ar[r,  "H^0(-^\wedge)"] & \CAlg_R^{\wedge a} 
\end{cd}
Let us explain how this is relevant for us in this section. Recall that for any symmetric monoidal category $\cC$ we can consider its category of commutative algebra objects $\CAlg(\cC)$, in the sense of \cite[Section 2.1.3]{lurieHA}, and this category has the following property: for any pair of morphisms $R \leftarrow D \rightarrow S$ in $\CAlg(\cC)$ the pushout can be identified with $R \otimes_D S$ by \cite[Proposition 3.2.4.10]{lurieHA}; where $- \otimes -$ comes from the symmetric monoidal structure on $\cC$. For us, this means that we can compute the various exotic tensor products in $\CAlg_R^\wedge$, $\CAlg_R^a$ and $\CAlg_R^{\wedge a}$ in terms of the classical tensor product of $\CAlg_R$ and the functors connecting them.

The goal of this section is to show that all the categories of algebras we just introduced admit $\varpi$-torsion free analogs, and to construct $\varpi$-torsion free tensor products on them. One of the main advantages of introducing $\varpi$-torsion free versions of this categories is that derived $\varpi$-completion and classical $\varpi$-completion agree -- which will be critical when relating this categories to Banach algebras.

The following lemma shows that being $\varpi$-torsion free interacts well with the (derived) $\varpi$-completion functor.

\begin{lemma}\label{completion_tf} Let $B$ be a ring and $\varpi$ an element of $B$. Then, a $B$-module $M \in \Mod_B$ is $\varpi$-torsion free if and only if its derived $\varpi$-completion $M^\wedge$ is a $\varpi$-torsion free $B$-module.
\end{lemma}

\begin{proof} Recall that the local cohomology of $M$ at $\varpi$ can be identified with $\Cone(M \rightarrow M[1/\varpi])[-1]$ (cf. \cite[Tag 0952]{stacks-project}), which implies that we have the identity $R^0 \Gamma_{\varpi} (M) = M[\varpi^\infty]$ identifying the zeroth cohomology group with the $\varpi$-torsion part of $M$. Moreover, Greenlees-May duality (cf. \cite[Tag 0A6V]{stacks-project}) implies that we have the equality
\begin{equation*}
	R\Gamma_\varpi (M) = R \Gamma_\varpi (M^\wedge)
\end{equation*}
showing that $M[\varpi^\infty] = M^\wedge[\varpi^\infty]$ and proving the desired claim.
\end{proof}

We introduce a couple of the categories of $\varpi$-torsion free algebras.

\begin{defn} Let $B$ be a $\varpi$-torsion free ring, for some $\varpi \in B$.
\begin{enumerate}[(1)]
	\item Let $\CAlg_B^{\tf}$ be the full subcategory of $\CAlg_B$ of commutative $B$-algebras whose underlying module is $\varpi$-torsion free.
	\item Let $\CAlg_B^{\wedge \tf}$ be the full subcategory of $\CAlg_B^{\wedge}$ of commutative $B$-algebras whose underlying module if derived $\varpi$-complete and $\varpi$-torsion free. Since this algebras are $\varpi$-torsion free, being derived $\varpi$-complete is equivalent to being classically $\varpi$-complete.
\end{enumerate}
\end{defn}

\begin{prop}\label{torsion_free_algebras}  Let $B$ be a $\varpi$-torsion free ring, for some $\varpi \in B$. Then, the canonical inclusions
\begin{equation*}
	\CAlg_B^{\tf} \hookrightarrow \CAlg_B \qquad \CAlg_B^{\wedge \tf} \hookrightarrow \CAlg_B^{\wedge}
\end{equation*}
admit left adjoints
\begin{align*}
	&(-)^{\tf}: \CAlg_B \rightarrow \CAlg_B^{\tf} \qquad A \mapsto \overline{A}:= A/A[\varpi^\infty] \\
	&(-)^{\wedge \tf}: \CAlg_B^{\wedge} \rightarrow \CAlg_B^{\wedge \tf} \qquad A \mapsto (\overline{A})^\wedge
\end{align*}
Moreover, every functor in the left hand side diagram admits a left adjoint, making the diagram on the right commute
\begin{cd}
	\CAlg_B^{\wedge \tf} \ar[d, hook] \ar[r, hook] & \CAlg_B^{\wedge}  \ar[d, hook] &&
	\CAlg_B \ar[r, "H^0(-^\wedge)"] \ar[d, "(-)^{\tf}", swap] & \CAlg_B^\wedge \ar[d, "(-)^{\wedge \tf}"] \\
	\CAlg_B^{\tf} \ar[r, hook] & \CAlg_B &&
	\CAlg_B^{\tf} \ar[r, "-^\wedge", swap] & \CAlg_B^{\wedge \tf}
\end{cd}
In particular, this implies that for a pair of morphisms $A \leftarrow D \rightarrow C$ in $\CAlg_B^{\wedge \tf}$ the pushout can be expressed as $(A \otimes_D C)^{\tf \wedge}$ -- the classical $\varpi$-completion of the $\varpi$-torsion free quotient of $A \otimes_D C$.
\end{prop}

\begin{proof} It is clear that the inclusion $\CAlg_B^\wedge \hookrightarrow \CAlg_B$ admits a left adjoint given by $A \mapsto H^0(A^\wedge)$, where $A^\wedge$ denotes the derived $\varpi$-completion of $A$; similarly, by Lemma \ref{completion_tf} it follows that the inclusion $\CAlg_B^{\wedge \tf} \hookrightarrow \CAlg_{B}^{\tf}$ also admits a left adjoint given by $A \mapsto A^\wedge$, where $A^\wedge$ is the classical $\varpi$-completion of $A$ (which agrees with the derived one as $A$ is $\varpi$-torsion free).

Define a functor $\CAlg_B \rightarrow \CAlg_B^{\tf}$ by the formula $A \mapsto \overline{A} := A/A[\varpi^\infty]$ where $\overline{A}$ is obtained by killing the $\varpi$-torsion of $A$ -- it is clear that this map is functorial. To show that this map it a left adjoint to the inclusion $\CAlg_B^{\tf} \hookrightarrow \CAlg_B$ it suffices to notice that for any morphism $A \rightarrow C$, where $C$ is $\varpi$-torsion free, there is an essentially unique factorization as $A \rightarrow A/A[\varpi^\infty] \rightarrow C$, which is clear. On the other hand, we can define a functor $\CAlg_B^{\wedge} \rightarrow \CAlg_B^{\wedge \tf}$ by the rule $A \mapsto (A/A[\varpi^\infty])^\wedge$ where $(-)^\wedge$ is the classical $\varpi$-completion functor which agrees with the derived one -- it is clear that this is a functorial procedure, but let us explain why its necessary to $\varpi$-complete after passing to the $\varpi$-torsion free quotient. Indeed, if $A$ has unbounded $\varpi$-torsion then $A[\varpi^\infty]$ will never by derived $\varpi$-complete (cf. \cite{bhatt2019torsion}), showing that $A/A[\varpi^\infty]$ will never be derived $\varpi$-complete in the unbounded torsion case. Finally, let us show that the map $\CAlg_B^{\wedge} \rightarrow \CAlg_B^{\wedge \tf}$ we just constructed is the left adjoint to the inclusion $\CAlg_B^{\wedge \tf} \rightarrow \CAlg_B^{\wedge}$. Indeed, for any morphism $A \rightarrow C$ in $\CAlg_B^{\wedge}$, where $C$ is $\varpi$-torsion free, will factor as
\begin{equation*}
	A \rightarrow A/A[\varpi^\infty] \rightarrow (A/A[\varpi^\infty])^\wedge \rightarrow C
\end{equation*}
showing that $\CAlg_B^{\wedge} \rightarrow \CAlg_B^{\wedge \tf}$ is indeed the left adjoint to the canonical inclusion in the opposite direction.
\end{proof}

In what follows we will introduce $\varpi$-torsion free categories of almost commutative $R$-algebras, and show that an analogous result to Proposition \ref{torsion_free_algebras} hold in the almost setting.

\begin{defn}\label{defn_comm_alg_almost_tf} Let $R$ is an integral perfectoid ring which is $\varpi$-complete with respect to some $\varpi \in R$ where $\varpi^p$ divides $p$, and such that $R$ is $\varpi$-torsion free. We regard the categories $\CAlg_R^a$ and $\CAlg_R^{\wedge a}$ as subcategories of $\CAlg_R$ and $\CAlg_R^\wedge$ respectively via the functors
\begin{equation*}
	\tau^{\le 0}j_* : \CAlg_R^a \hookrightarrow \CAlg_R \qquad \tau^{\le 0} j_*^\wedge: \CAlg_R^{\wedge a} \hookrightarrow \CAlg_R^\wedge
\end{equation*}
We will be interested in the following categories of almost commutative $R$-algebras
\begin{enumerate}[(1)]
	\item The category $\CAlg_R^{a \tf}$ of $\varpi$-torsion free almost commutative $R$-algebras, defined as $\CAlg_R^{\tf} \cap \CAlg_R^{a}$.
	\item The category $\CAlg_R^{\wedge a \tf}$ of $\varpi$-complete $\varpi$-torsion free commutative $R$-algebras, defined as $\CAlg_R^{\wedge \tf} \cap \CAlg_R^{\wedge a}$.
\end{enumerate}
\end{defn}

\begin{prop}\label{tf_almost_algebras} The canonical inclusion functors
\begin{equation*}
	\CAlg_R^{a \tf} \hookrightarrow \CAlg_R^a \qquad \CAlg_R^{\wedge a \tf} \hookrightarrow \CAlg_R^{\wedge a}
\end{equation*}
admit left adjoints
\begin{align*}
	&(-)^{a \tf}: \CAlg_R^a \rightarrow \CAlg_R^{a \tf} \qquad A \mapsto \Big [ (\tau^{\le 0}j_* A)^{\tf} \Big]^a \\
	&(-)^{\wedge a \tf}: \CAlg_R^{\wedge a} \rightarrow \CAlg_R^{\wedge a \tf} \qquad A \mapsto \Big [ (\tau^{\le 0}j_*^\wedge A)^{\wedge \tf} \Big]^{a} 
\end{align*}
Moreover, every functor in the left hand side diagram admits a left adjoint, making the diagram on the right commute
\begin{cd}
	\CAlg_R^{\wedge a \tf} \ar[r, hook] \ar[d, hook] & \CAlg_R^{\wedge a} \ar[d, hook] &&
	\CAlg_R^a \ar[r, "H^0(-^\wedge)"] \ar[d, "(-)^{a \tf}", swap] & \CAlg_R^{\wedge a} \ar[d, "(-)^{\wedge a \tf}"] \\
	\CAlg_R^{a \tf} \ar[r, hook] & \CAlg_R^a &&
	\CAlg_R^{a \tf} \ar[r, "-^\wedge"]& \CAlg_R^{\wedge a \tf}
\end{cd}
\end{prop}

\begin{proof} During Construction \ref{const_ab_almost_monoidal_str} we showed that the inclusion $\CAlg_R^{\wedge a} \hookrightarrow \CAlg_R^{a}$ admits a left adjoint given by $H^0(-^\wedge)$. To show that the inclusion $\CAlg_R^{a \tf} \hookrightarrow \CAlg_R^{a}$ admits a left adjoint, we realize both of them as subcategories of $\CAlg_R$ via the functor $\tau^{\le 0} j_*: \CAlg_R^a \hookrightarrow \CAlg_R$. Its clear that for any morphism $\tau^{\le 0} j_* (A) \rightarrow \tau^{\le 0} j_* (S)$, where $\tau^{\le 0} j_* (S)$ is $\varpi$-torsion free factors as
\begin{equation*}
	\tau^{\le 0} j_* (A) \longrightarrow \tau^{\le 0} j_* (A)/(\tau^{\le 0} j_* (A)[\varpi^\infty]) \longrightarrow \tau^{\le 0} j_* (S)
\end{equation*}
However, since $\tau^{\le 0} j_* (A)/(\tau^{\le 0} j_* (A)[\varpi^\infty])$ need no be in the essential image of $\tau^{\le 0} j_*$ there is a further factorization
\begin{equation*}
	\tau^{\le 0} j_* (A) \longrightarrow \Big ( \tau^{\le 0} j_* (A)/(\tau^{\le 0} j_* (A)[\varpi^\infty]) \Big )_* \longrightarrow \tau^{\le 0} j_* (S)
\end{equation*}
where $(-)_*$ is the unit of the adjunction, introduced in \ref{almost_elements}. This shows that $(-)^{a \tf}$ is the left adjoint to the inclusion $\CAlg_R^{a \tf} \hookrightarrow \CAlg_R^{a}$. The argument to show that the inclusion $\CAlg_R^{\wedge a \tf} \hookrightarrow \CAlg_R^{\wedge a}$ admits a left adjoint is completely symmetrical to the argument we just described for $\CAlg_R^{a \tf} \hookrightarrow \CAlg_R^{a}$; but in this case we realize the categories $\CAlg_R^{\wedge a \tf} \hookrightarrow \CAlg_R^{\wedge a}$ as full subcategories of $\CAlg_R^{\wedge}$ via the functor $\tau^{\le 0} j_*^\wedge: \CAlg_R^{\wedge a} \hookrightarrow \CAlg_R^\wedge$.

Finally, we show that the inclusion $\CAlg_R^{\wedge a \tf} \hookrightarrow \CAlg_R^{a \tf}$ admits a left adjoint given by $\varpi$-completion. We realize the category $\CAlg_R^{a \tf} \subset \CAlg_R^a$ as a subcategory of $\CAlg_R$ via the inclusion $\tau^{\le 0} j_*$, and the category $\CAlg_R^{\wedge a \tf} \subset \CAlg_R^{\wedge a}$ as a subcategory of $\CAlg_R^{\wedge} \subset \CAlg_R$ via the functor $\tau^{\le 0} j_*^{\wedge}$. Fix an object $A$ in $\CAlg_R^{a \tf} \subset \CAlg_R$, we showed in Construction \ref{const_almost_comp_localization} that $A^\wedge$ will be in contained in $\CAlg_R^{\wedge a} \subset \CAlg_R$, and moreover by Lemma \ref{completion_tf} we learn that $A^\wedge$ is $\varpi$-torsion free. Hence, $A^\wedge$ is contained in $\CAlg_R^{\wedge a \tf} \subset \CAlg_R$, proving that $\varpi$-completion is the left adjoint to the inclusion $\CAlg_R^{\wedge a \tf} \hookrightarrow \CAlg_R^{a \tf}$.
\end{proof}

\begin{prop}\label{almost_funct_tf_alg} Let $R$ is an integral perfectoid ring which is $\varpi$-complete with respect to some $\varpi \in R$ where $\varpi^p$ divides $p$, and such that $R$ is $\varpi$-torsion free. Then, the canonical inclusions
\begin{equation*}
	\tau^{\le 0} j_*: \CAlg_R^{a \tf} \hookrightarrow \CAlg_R^{\tf} \qquad \tau^{\le 0} j_*^\wedge: \CAlg_R^{\wedge a \tf} \hookrightarrow \CAlg_R^{\wedge \tf}
\end{equation*}
introduced in Definition \ref{defn_comm_alg_almost_tf} admit left adjoints
\begin{align*}
	&(-)^a: \CAlg_R^{\tf} \rightarrow \CAlg_R^{a \tf} \qquad A \mapsto A^a\\
	&(-)^{\wedge a}: \CAlg_R^{\wedge \tf} \rightarrow \CAlg_R^{\wedge a \tf} \qquad A \mapsto A^{a}
\end{align*}
In particular, this implies that every functor on the left hand side diagram admits a left adjoint, making the diagram on the right commute
\begin{cd}
	\CAlg_R^{\wedge a \tf} \ar[r, hook]  \ar[d, hook, "\tau^{\le 0} j_*", swap] & \CAlg_R^{a \tf} \ar[d, hook, "\tau^{\le 0} j_*^{\wedge}"] &&
	\CAlg_R^{\tf} \ar[r, "(-)^a"] \ar[d, "-^\wedge", swap] & \CAlg_R^{a \tf} \ar[d, "-^\wedge"]\\
	\CAlg_R^{\wedge \tf} \ar[r, hook] & \CAlg_R^{\tf} &&
	\CAlg_R^{\wedge \tf} \ar[r, "(-)^{a}"] & \CAlg_R^{\wedge a \tf}
\end{cd}
This provides us with a torsion free analog to the diagrams introduced at the  beginning of this section.
\end{prop}

\begin{proof} In order to show that the inclusion $\CAlg_R^{a \tf} \hookrightarrow \CAlg_R^{\tf}$ admits a left adjoint given by $(-)^a$ it suffices to show that for any element $A$ in $\CAlg_R^{a \tf} \subset \CAlg_R^{\tf} \CAlg_R$ the object $A_* := \uHom_R ({(\varpi)_{\perfd}, A})$ in $\CAlg_R^a \subset \CAlg_R$ is $\varpi$-torsion free. Since $A$ is $\varpi$-torsion free then $A_*$ can be identified with the subring of $A[1/\varpi]$ where $a \in A[1/\varpi]$ belongs to $A_*$ if $\varepsilon a \in A$ for all $\varepsilon \in (\varpi)_{\perfd}$, making it clear that $A_*$ is $\varpi$-torsion free.

On the other hand, to show that the inclusion $\CAlg_R^{\wedge a \tf} \hookrightarrow \CAlg_R^{\wedge \tf}$ admits a left adjoint given by $(-)^a$ it suffices to show that for any element $A \in \CAlg_R^{\wedge a \tf} \subset \CAlg_R^{\wedge \tf}$ the object $A_*:= \uHom_R ({(\varpi)_{\perfd}, A})$ is $\varpi$-complete and $\varpi$-torsion free. We already argued in the previous paragraph that if $A$ is $\varpi$-torsion free then so is $A_*$. Moreover, we showed in \ref{abelian_internal_hom} that if $A$ is derived $\varpi$-complete then so is $A_*$; and since $A$ is $\varpi$-torsion free being derived $\varpi$-complete is equivalent to being classically $\varpi$-complete.
\end{proof}

\begin{corollary}\label{pushout_in_comp-tf-a} Let $A \leftarrow D \rightarrow C$ be a pair of morphisms in $\CAlg_R^{\wedge a \tf}$. Then, the pushout can be described by the following sequence of operations
\begin{cd}
	A \otimes_D C \ar[r, "(-)^{\tf}"] & (A \otimes_D C)^{\tf} \ar[r, "-^\wedge"] & (A \otimes_D C)^{\tf, \wedge} \ar[r, "(-)^{a}"] & (A \otimes_D C)^{\tf, \wedge, a}
\end{cd}
\end{corollary}

\begin{proof} Recall that the pushout of $A \leftarrow D \rightarrow C$ computed in $\CAlg_R$ can be canonically identified with $A \otimes_D C$. Therefore, the result follows from the fact that all the following morphisms
\begin{equation*}
	\CAlg_R^{\wedge a \tf} \hookrightarrow \CAlg_R^{\wedge \tf} \hookrightarrow \CAlg_R^{\tf} \hookrightarrow \CAlg_R
\end{equation*}
admit left adjoints, as we have established in this section.
\end{proof}

\subsection{\texorpdfstring{$p$}{p}-integrally closed algebras}

Throughout this section $R$ is an integral perfectoid ring which is $\varpi$-complete with respect to some $\varpi \in R$ where $\varpi^p$ divides $p$, and such that $R$ is $\varpi$-torsion free.

\begin{defn} Recall from \ref{p_int_closed_defn} that a ring $A$ with a non-zero divisor $\varpi \in A$ is said to be $p$-integrally closed in $A[1/\varpi]$ if every $a \in A[1/\varpi]$ where $a^p \in A$ satisfies $a \in A$.

The $p$-integral closure of $A$ in $A[1/\varpi]$, denoted by $A^{\pic}$, is constructed as $\cup _{n \ge 0} A_n \subset A[1/\varpi]$, where $A_0 = A$ and $A_{n+1} \subset A[1/\varpi]$ is the $A_n$-subalgebra generated by all the $a \in A[1/\varpi]$ such that $a^p \in A_n$; it is the smallest $p$-integrally closed subring of $A[1/\varpi]$ containing $A$. Clearly, the $p$-integral closure of $A \hookrightarrow A[1/\varpi]$ is contained in the integral closure.
\end{defn}

\begin{lemma}\label{stab_p_int_closure} Let $A$ be a $\varpi$-torsion free $R$-algebra. Assume that $A$ is $p$-integrally closed with respect to $A \subset A[1/\varpi]$. Then,
\begin{enumerate}[(1)]
	\item The $\varpi$-completion of $A$, denoted by $A^\wedge$, is $p$-integrally closed with respect to $A^\wedge \subset A^\wedge [1/\varpi]$.
	\item The almost elements of $A$, denoted by $A_*$ (\ref{almost_elements}), is $p$-integrally closed with respect to $A_* \subset A[1/\varpi]$.
\end{enumerate}
\end{lemma}

\begin{proof} \cite[Lemma 5.1.1]{bhattlecture_perfectoid}
\end{proof}

\begin{defn} Let $R$ be an integral perfectoid ring which is $\varpi$-complete with respect to some $\varpi$ where $\varpi^p$ divides $p$, and such that $R$ is $\varpi$-torsion free.
\begin{enumerate}[(1)]
	\item The category $\CAlg_R^{\pic} \subset \CAlg_R^{\tf}$ of $R$-algebras $A$ which are $p$-integrally closed in $A[1/\varpi]$.
	\item The category $\CAlg_R^{\wedge \pic} \subset \CAlg_R^{\wedge \tf}$ of $\varpi$-complete $R$-algebras $A$ which are $p$-integrally closed in $A[1/\varpi]$.
	\item Regarding $\CAlg_R^{a \tf}$ as a full-subcategory of $\CAlg_R^{\tf}$ via the functor $\tau^{\le 0} j_*$, we define the almost category of $p$-integrally closed $R$-algebras as $\CAlg_R^{a \pic} := \CAlg_R^{a \tf} \cap \CAlg_R^{\pic}$.
	\item Regarding $\CAlg_R^{\wedge a \tf}$ as a full-subcategory of $\CAlg_R^{\wedge \tf}$ via the functor $\tau^{\le 0} j_*^\wedge$, we define the almost category of $\varpi$-complete $p$-integrally closed $R$-algebras as $\CAlg_R^{\wedge a \pic} := \CAlg_R^{\wedge a \tf} \cap \CAlg_R^{\wedge \pic}$.
\end{enumerate}
\end{defn}

\begin{prop}\label{pic_algebras} The fully faithful functors
\begin{equation*}
	\CAlg_R^{\pic} \hookrightarrow \CAlg_R^{\tf} \qquad \CAlg_R^{\wedge \pic} \hookrightarrow \CAlg_R^{\wedge \tf}
\end{equation*}
admit left adjoints
\begin{align*}
	&(-)^{\pic}: \CAlg_R^{\tf} \rightarrow \CAlg_R^{\pic} && A \mapsto A^{\pic} \\
	&(-)^{\wedge \pic}: \CAlg_R^{\wedge \tf} \rightarrow \CAlg_R^{\wedge \pic} && A \mapsto (A^{\pic})^\wedge
\end{align*}
Moreover, every functor on the left hand side diagram admits a left adjoint, making the diagram on the right commute
\begin{cd}
	\CAlg_R^{\wedge \pic} \ar[r, hook] \ar[d, hook] & \CAlg_R^{\wedge \tf} \ar[d, hook] &&
	\CAlg_R^{\tf} \ar[r, "-^\wedge"] \ar[d, "(-)^{\pic}", swap] & \CAlg_R^{\wedge \tf} \ar[d, "(-)^{\wedge \pic}"] \\
	\CAlg_R^{\pic} \ar[r, hook] & \CAlg_R^{\tf} &&
	\CAlg_R^{\pic} \ar[r, "-^\wedge"] & \CAlg_R^{\wedge \pic}
\end{cd}
\end{prop}

\begin{proof} First, we show that the inclusion $\CAlg_R^{\pic} \rightarrow \CAlg_R^{\tf}$ admits a left adjoint given by $(-)^{\pic}$. Consider a morphism $f: A \rightarrow S$ in $\CAlg_R^{\tf}$ and such that $S$ is $p$-integrally closed in $S[1/\varpi]$. We need to show that the morphism $f: A \rightarrow S$ factorizes uniquely as $A \rightarrow A^{\pic} \rightarrow S$. Let $a$ be an element of $A[1/\varpi]$ such that $a^p \in A$, it follows that $f(a) \in S[1/\varpi]$ is an element of $S$ by hypothesis; making it clear that we have a unique factorization $A \rightarrow A^{\pic} \rightarrow S$.

Next, we show that the inclusion $\CAlg_R^{\wedge \pic} \rightarrow \CAlg_R^{\wedge \tf}$ admits a left adjoint given by $(-)^{\wedge \pic}$. Given a morphism $A \rightarrow S$ in $\CAlg_R^{\wedge \tf}$ and such that $S$ is $p$-integrally closed with respect to $S \subset S[1/\varpi]$, the argument in the preceding paragraph shows that we have an essentially unique factorization $A \rightarrow A^{\pic} \rightarrow S$. However, we cannot guarantee that $A^{\pic}$ is $\varpi$-complete, so we obtain a further factorization $A \rightarrow A^{\pic} \rightarrow (A^{\pic})^\wedge \rightarrow S$. And by Lemma \ref{stab_p_int_closure} we can conclude that $(A^{\pic})^\wedge$ is $p$-integrally closed with respect to $A^{\pic, \wedge} \subset A^{\pic, \wedge}[1/\varpi]$, proving that $(-)^{\wedge \pic}$ is indeed the left adjoint to the inclusion $\CAlg_R^{\wedge \pic} \rightarrow \CAlg_R^{\wedge \tf}$.

Finally, recall that we already showed in \ref{torsion_free_algebras} that the inclusion $\CAlg_R^{\wedge \tf} \hookrightarrow \CAlg_R^{\tf}$ admits a left adjoint given by $\varpi$-completion. And by Lemma \ref{stab_p_int_closure} we conclude that $\CAlg_R^{\wedge \pic} \hookrightarrow \CAlg_R^{\pic}$ also admits a left adjoint given by $\varpi$-completion. 
\end{proof}

\begin{prop}\label{pic_almost_algebras} The fully faithful functors
\begin{equation*}
	\CAlg_R^{a \pic} \hookrightarrow \CAlg_R^{a \tf} \qquad \CAlg_R^{\wedge a \pic} \hookrightarrow \CAlg_R^{\wedge a \tf}
\end{equation*}
admit left adjoints
\begin{align*}
	& (-)^{a \pic}: \CAlg_R^{a \tf} \rightarrow \CAlg_R^{a \pic} && A \mapsto \Big [ (\tau^{\le 0} j_* A)^{\pic} \Big ]^a \\
	& (-)^{\wedge a \pic}: \CAlg_R^{\wedge a \tf} \rightarrow \CAlg_R^{\wedge a \pic} && A \mapsto \Big[ ((\tau^{\le 0} j_*^\wedge A)^{\pic})^a \Big ]^\wedge \simeq \Big[ ((\tau^{\le 0} j_*^\wedge A)^{\pic})^\wedge \Big ]^a
\end{align*}
Moreover, every functor on the left hand side admits a left adjoint, making the diagram on the right commute
\begin{cd}
	\CAlg_R^{\wedge a \pic} \ar[r, hook] \ar[d, hook] & \CAlg_R^{a \pic} \ar[d, hook] &&
	\CAlg_R^{a \tf} \ar[r, "-^\wedge"] \ar[d, "(-)^{a \pic}", swap] & \CAlg_R^{\wedge a \tf}  \ar[d, "(-)^{\wedge a \pic}"]\\
	\CAlg_R^{\wedge a \tf} \ar[r, hook] & \CAlg_R^{a \tf} &&
	\CAlg_R^{a \pic} \ar[r, "-^\wedge"] & \CAlg_R^{\wedge a \pic}
\end{cd}
\end{prop}

\begin{proof} First, let us remark that the isomorphism $\Big[ ((\tau^{\le 0} j_*^\wedge A)^{\pic})^a \Big ]^\wedge \simeq \Big[ ((\tau^{\le 0} j_*^\wedge A)^{\pic})^\wedge \Big ]^a$ follows from the commutativity of the diagrams in Proposition \ref{almost_funct_tf_alg}. To show that the fully faithful functor $\CAlg_R^{a \pic} \hookrightarrow \CAlg_R^{a \tf}$ admits a left adjoint given by $(-)^{a \pic}$, it suffices to show that for any morphism $A \rightarrow S$ in the essential image of $\tau^{\le 0} j_*: \CAlg_R^{a \tf} \hookrightarrow \CAlg_R^{\tf}$ there is an essentially unique factorization $A \rightarrow (A^{\pic})_* \rightarrow S$ whenever $S$ is $p$-integrally closed with respect to $S \subset S[1/\varpi]$. Indeed, we showed in Proposition \ref{pic_algebras} that there is an essentially unique factorization $A \rightarrow A^{\pic} \rightarrow S$; and since $S$ satisfies $S_* \simeq S$ by hypothesis, we have a factorization $A \rightarrow (A^{\pic})_* \rightarrow S$. Then, Lemma \ref{stab_p_int_closure} implies that $(A^{\pic})_* \subset A[1/\varpi]$ is $p$-integrally closed, showing that $(-)^{a \pic}$ is the left adjoint to the inclusion $\CAlg_R^{a \pic} \hookrightarrow \CAlg_R^{a \tf}$.
	
Next, to show that the fully faithful functor $\CAlg_R^{\wedge a \pic} \hookrightarrow \CAlg_R^{\wedge a \tf}$ admits a left adjoint given by $(-)^{\wedge a \pic}$, it suffices to show that for any morphism $A \rightarrow S$ in the essential image of $\tau^{\le 0} j_*^{\wedge}: \CAlg_R^{\wedge a \tf} \rightarrow \CAlg_R^{\wedge \tf}$, where $S$ is assumed to be $p$-integrally closed with respect to $S \subset S[1/\varpi]$, there exists an essentially unique factorization
\begin{equation*}
	A \rightarrow (A^{\pic})^\wedge \rightarrow (A^{\pic, \wedge})_* \rightarrow S
\end{equation*}
Indeed, since $S$ is $\varpi$-complete and $p$-integrally closed it follows from Proposition \ref{pic_algebras} that there exists an essentially unique factorization $A \rightarrow (A^{\pic})^\wedge \rightarrow S$. And since $S$ satisfies $S \simeq S_*$ it follows that we get the desired factorization.

Recall that we already showed in Proposition \ref{tf_almost_algebras} that the inclusion $\CAlg_R^{a \tf} \hookrightarrow \CAlg_R^{\wedge a \tf}$ admits a left adjoint given by $\varpi$-completion. Finally, we need to show that the fully faithful functor $\CAlg_R^{\wedge a \pic} \rightarrow \CAlg_R^{a \pic}$ admits a left adjoint given by $\varpi$-completion. This follows from the fact that $\varpi$-completing preserves being $p$-integrally closed (\ref{stab_p_int_closure}) and the fact that the essential image of $\tau^{\le 0} j_*: \CAlg_R^a \rightarrow \CAlg_R$ is stable under derived $\varpi$-completion, which follows from Construction \ref{const_almost_comp_localization}.
\end{proof}

\begin{prop}\label{almost_funct_pic_algebras} The fully faithful functors
\begin{equation*}
	\tau^{\le 0}j_*: \CAlg_R^{a \pic} \rightarrow \CAlg_R^{\pic} \qquad \tau^{\le 0} j_*^\wedge: \CAlg_R^{\wedge a \pic} \hookrightarrow \CAlg_R^{\wedge \pic}
\end{equation*}
admit left adjoints
\begin{align*}
	& (-)^a: \CAlg_R^{\pic} \rightarrow \CAlg_R^{a \pic} && A \mapsto A^a \\
	& (-)^a: \CAlg_R^{\wedge \pic} \rightarrow \CAlg_R^{\wedge a \pic} && A \mapsto A^a
\end{align*}
Moreover, every functor on the left hand side diagram admits a left adjoint, making the diagram on the right commute
\begin{cd}
	\CAlg_R^{\wedge a \pic} \ar[r, hook, "\tau^{\le 0} j_*"] \ar[d, hook] & \CAlg_R^{\wedge \pic} \ar[d, hook] &&
	\CAlg_R^{\pic} \ar[r, "(-)^a"] \ar[d, "-^\wedge", swap] & \CAlg_R^{a \pic} \ar[d, "-^\wedge"] \\
	\CAlg_R^{a \pic} \ar[r, hook, "\tau^{\le 0} j_*"] & \CAlg_R^{\pic} &&
	\CAlg_R^{\wedge \pic} \ar[r, "(-)^a"] & \CAlg_R^{\wedge a \pic}
\end{cd}
\end{prop}

\begin{proof} We showed in Proposition \ref{pic_algebras} and \ref{pic_almost_algebras} that the inclusions $\CAlg_R^{\wedge \pic} \hookrightarrow \CAlg_R^{\pic}$ and $\CAlg_R^{\wedge a \pic} \hookrightarrow \CAlg_R^{a \pic}$ admit a left adjoint given by $\varpi$-completion. To show that the fully faithful functor $\tau^{\le 0} j_*: \CAlg_R^{a \pic} \hookrightarrow \CAlg_R^{\pic}$ admits a left adjoint given by $(-)^a$ it suffices to show that for any map $A \rightarrow S$ in $\CAlg_R^{\pic}$, where $S$ satisfies $S \simeq S_*$, there exists an essentially unique factorization $A \rightarrow A_* \rightarrow S$ such that $A_*$ is $p$-integrally closed with respect to $A[1/\varpi]$. The essentially unique factorization is clear, and the fact that $A_*$ is $p$-integrally closed follows from Lemma \ref{stab_p_int_closure}.

We proceed similarly to show that the fully faithful functor $\tau^{\le 0} j_*: \CAlg_R^{\wedge a \pic} \hookrightarrow \CAlg_R^{\wedge \pic}$ admits a left adjoint given by $(-)^a$. Let $A \rightarrow S$ be a morphism in $\CAlg_R^{\wedge \pic}$ such that $S$ satisfies $S \simeq S_*$, then we obtain an essentially unique factorization $A \rightarrow A_* \rightarrow S$. Hence, it suffices to show that $A_*$ is $p$-integrally closed with respect to $A_* \subset A[1/\varpi]$ and that $A_*$ is $\varpi$-complete. Lemma \ref{stab_p_int_closure} guarantees that $A_*$ is $p$-integrally closed with respect to $A_* \subset A[1/\varpi]$, and we showed in Construction \ref{abelian_internal_hom} that if $A$ is derived $\varpi$-complete then so is $A_* := \uHom_R ((\varpi)_{\perfd}, A)$.
\end{proof}

\subsection{Integrally closed algebras}

Throughout this section $R$ is an integral perfectoid ring which is $\varpi$-complete with respect to some $\varpi \in R$ where $\varpi^p$ divides $p$, and such that $R$ is $\varpi$-torsion free.

\begin{lemma}\label{stab_ic} Let $A$ be a $\varpi$-torsion free $R$-algebra. Assume that $A$ is integrally closed with respect to $A \subset A[1/\varpi]$. Then,
\begin{enumerate}[(1)]
	\item The $\varpi$-completion of $A$, denoted by $A^\wedge$, is integrally closed with respect to $A^\wedge \subset A^\wedge [1/\varpi]$.
	\item The almost elements of $A$, denoted by $A_*$ (\ref{almost_elements}), is integrally closed with respect to $A_* \subset A[1/\varpi]$.
\end{enumerate}
\end{lemma}

\begin{proof} \cite[Lemma 5.1.2]{bhattlecture_perfectoid}
\end{proof}

\begin{lemma}\label{ic_of_pic_preserves_complete} Let $A$ be a $\varpi$-torsion free $R$-algebra which is $p$-integrally closed with respect to $A \subset A[1/\varpi]$. Then, if $A$ is $\varpi$-complete, so is $A^{\ic}$.
\end{lemma}

\begin{proof} Recall that if $A$ is integrally closed in $A[1/\varpi]$ this means that if $f \in A^{\ic}$ then there exists a $n \in \ZZ_{\ge 0}$ such that $\{f^{\ZZ_{\ge 0}}\} \subset \frac{1}{\varpi^n} A \subset A[\frac{1}{\varpi}]$. Without loss of generality we may assume that $\varpi$ admits compatible $p$-power roots in $A$, thus by the $p$-integral closedness of $A \subset A[\frac{1}{\varpi}]$ we conclude that if $f^{pm} \varpi^n \in A$ then $f^m \varpi^{n/p} \in A$. By repeating this procedure we can conclude that $\{f^{\ZZ_{\ge 0}} \} \subset \frac{1}{\varpi^{1/p^n}} A$ for all $n \in \ZZ_{\ge 0}$.

From the fiber sequence $A \hookrightarrow A^{\ic} \rightarrow A^{\ic}/A$ to show that $A^{\ic}$ is $\varpi$-complete it suffices to show that $A^{\ic}/A$ is derived $\varpi$-complete. Notice that by the work done in the previous paragraph we get that every element of $A^{\ic}/A$ is $\varpi$-torsion, and so it is derived $\varpi$-complete.
\end{proof}

\begin{defn} Let $R$ be an integral perfectoid ring which is $\varpi$-complete with respect to some $\varpi \in R$ where $\varpi^p$ divides $p$, and such that $R$ is $\varpi$-torsion free.
\begin{enumerate}[(1)]
	\item The category $\CAlg_R^{\ic} \subset \CAlg_R^{\pic} \subset \CAlg_R^{\tf}$ of $R$-algebras $A$ which are integrally closed in $A[1/\varpi]$.
	\item The category $\CAlg_R^{\wedge \ic} \subset \CAlg_R^{\wedge \pic} \subset \CAlg_R^{\wedge \tf}$ of $\varpi$-complete $R$-algebras $A$ which are integrally closed in $A[1/\varpi]$.
	\item Regarding $\CAlg_R^{a \pic} \subset \CAlg_R^{a \tf}$ as full subcategories of $\CAlg_R^{\pic} \subset \CAlg_R^{\tf}$ via the functor $H^0 j_*$, we define the almost category of integrally closed $R$-algebras as $\CAlg_R^{a \ic}:= \CAlg_R^{\ic} \cap \CAlg_R^{a \tf}$.
	\item Regarding $\CAlg_R^{\wedge a \pic} \subset \CAlg_R^{\wedge a \tf}$ as full subcategories of $\CAlg_R^{\wedge \pic} \subset \CAlg_R^{\wedge \tf}$ via the functor $H^0 j_*$, we define the almost category of $\varpi$-complete $R$-algebras as $\CAlg_R^{\wedge a \ic} := \CAlg_R^{\wedge a \tf} \cap \CAlg_R^{\ic}$.
\end{enumerate}
\end{defn}

\begin{prop}\label{ic_algebras} The fully faithful functors
\begin{equation*}
	\CAlg_R^{\ic} \hookrightarrow \CAlg_R^{\pic} \qquad \CAlg_R^{\wedge \ic} \hookrightarrow \CAlg_R^{\wedge \pic}
\end{equation*}
admit left adjoints
\begin{align*}
	&(-)^{\ic}: \CAlg_R^{\pic} \rightarrow \CAlg_R^{\ic} && A \mapsto A^{\ic} \\
	&(-)^{\ic}: \CAlg_R^{\wedge \pic} \rightarrow \CAlg_R^{\wedge \ic} && A \mapsto A^{\ic}
\end{align*}
Moreover, every functor on the left hand side diagram admits a left adjoint, making the diagram on the right commute
\begin{cd}
	\CAlg_R^{\wedge \ic} \ar[r, hook] \ar[d, hook] & \CAlg_R^{\ic} \ar[d, hook] &&
	\CAlg_R^{\pic} \ar[r, "-^\wedge"]  \ar[d, "(-)^{\ic}", swap] & \CAlg_R^{\wedge \pic} \ar[d, "(-)^{\ic}"]  \\
	\CAlg_R^{\wedge \pic} \ar[r, hook] & \CAlg_R^{\pic} &&
	\CAlg_R^{\ic} \ar[r, "-^\wedge"] & \CAlg_R^{\wedge \ic}
\end{cd}
\end{prop}

\begin{proof} We already showed in \ref{pic_algebras} that the fully faithful functor $\CAlg_R^{\wedge \pic} \hookrightarrow \CAlg_R^{\pic}$ admits a left adjoint given by $\varpi$-completion. To show that the functor $\CAlg_R^{\wedge \ic} \hookrightarrow \CAlg_R^{\ic}$ admits a left adjoint given by $\varpi$-completion, notice that for any morphism $A \rightarrow C$ in $\CAlg_R^{\ic}$ and such that $C$ is $\varpi$-complete we get an essentially unique factorization $A \rightarrow A^\wedge \rightarrow C$; where $A^\wedge$ is still integrally closed in $A^\wedge[\frac{1}{\varpi}]$ by \ref{stab_ic}.

Next, we show that the fully faithful functor $\CAlg_R^{\pic} \hookrightarrow \CAlg_R^{\ic}$ admits a left adjoint given by $(-)^{\ic}$. Let $\varphi: A \rightarrow C$ be a morphism in $\CAlg_R^{\pic}$ such that $C$ is integrally closed in $C[\frac{1}{\varpi}]$, and let $a \in A[\frac{1}{\varpi}]$ be an element which is the root of a monic polynomial $g \in A[X]$. Then, $\varphi(a) \in C[\frac{1}{\varpi}]$ is the root of the monic polynomial $\varphi(g) \in C[X]$, but since $C \subset C[\frac{1}{\varpi}]$ is integrally closed it follows that $\varphi(a) \in C$. Hence, it follows that the map $\varphi: A \rightarrow C$ factors as $A \rightarrow A^{\ic} \rightarrow C$; proving that $(-)^{\ic}$ is the left adjoint to $\CAlg_R^{\pic} \hookrightarrow \CAlg_R^{\ic}$. Finally, we know from \ref{ic_of_pic_preserves_complete} that applying $(-)^{\ic}$ preserves $\varpi$-completeness, thus it from the discussion above that the fully faithful functor $\CAlg_R^{\wedge \pic} \hookrightarrow \CAlg_R^{\wedge \ic}$ admits a left adjoint given by $(-)^{\ic}$.
\end{proof}

\begin{prop}\label{ic_almost_algebras} The fully faithful functors
\begin{equation*}
	\CAlg_R^{a \ic} \hookrightarrow \CAlg_R^{a \pic} \qquad \CAlg_R^{\wedge a \ic} \hookrightarrow \CAlg_R^{\wedge a \pic}
\end{equation*}
are equivalences with inverse given by
\begin{align*}
	&(-)^{a \ic}: \CAlg_R^{a \pic} \rightarrow \CAlg_R^{a \ic} && A \mapsto (H^0 j_* A)^{\ic, a} \\
	&(-)^{a \ic}: \CAlg_R^{\wedge a \pic} \rightarrow \CAlg_R^{\wedge a \ic} && A \mapsto (H^0 j_* A)^{\ic, a}
\end{align*}
Moreover, every functor on the left hand side diagram admits a left adjoint, making the diagram on the right commute
\begin{cd}
	\CAlg_R^{\wedge a \ic} \ar[r, hook] \ar[d, hook, "\simeq", swap] & \CAlg_R^{a \ic} \ar[d, hook, "\simeq"] &&
	\CAlg_R^{a \pic} \ar[r, "-^\wedge"] \ar[d, "(-)^{a \ic}", swap]& \CAlg_R^{\wedge a \pic} \ar[d, "(-)^{a \ic}"]\\
	\CAlg_R^{\wedge a \pic}  \ar[r, hook] & \CAlg_R^{a \pic} &&
	\CAlg_R^{a \ic} \ar[r, "-^\wedge"] & \CAlg_R^{\wedge a \ic}
\end{cd}
\end{prop}

\begin{proof} To show that the fully faithful functor $\CAlg_R^{a \ic} \hookrightarrow \CAlg_R^{a \pic}$ is an equivalent, with inverse $(-)^{a \ic}$ it suffices to show that if $A \in \CAlg_R^{a \pic}$ then $S := H^0 j_* A$ is integrally closed in $S \subset S[\frac{1}{\varpi}]$. By hypothesis we know that $S$ is $p$-integrally closed in $S[\frac{1}{\varpi}]$, and from the proof of \ref{ic_of_pic_preserves_complete} we can conclude that if $a \in S[\frac{1}{\varpi}]$ is the root of a monic polynomial in $S[X]$, then $\{a^{\ZZ_{\ge 0}} \} \subset \frac{1}{\varpi^{1/p^n}} S$ for all $n \in \ZZ_{\ge 0}$, were we can assume without loss of generality that $\varpi$ admits compatible $p$-power roots. Moreover, as $S$ satisfies $S \simeq S_*$ by definition, it follows that $a \in S$, showing that $S$ is integrally closed in $S[\frac{1}{\varpi}]$. The same argument, coupled with the fact that if $A \in \CAlg_R^{\tf}$ is $\varpi$-complete then so is $A_*$, shows that the fully faithful functor $\CAlg_R^{\wedge a \ic} \hookrightarrow \CAlg_R^{\wedge a \pic}$ is an equivalence.

To conclude, we need to show that the inclusion $\CAlg_R^{\wedge a \ic} \hookrightarrow \CAlg_R^{a \ic}$ admits a left adjoint given by $\varpi$-completion. Let $A \rightarrow C$ be a morphism in the essential image of $H^0j_*: \CAlg_R^{a \ic} \rightarrow \CAlg_R^{\ic}$ such that $C$ is $\varpi$-complete; then there is an essentially unique factorization $A \rightarrow A^\wedge \rightarrow C$, where $A^\wedge$ is integrally closed in $A^\wedge[\frac{1}{\varpi}]$ by \ref{ic_of_pic_preserves_complete}, and satisfies $A^\wedge \simeq (A^\wedge)_*$ by \ref{const_almost_comp_localization}.
\end{proof}

\begin{prop}\label{almost_funct_ic_algebras} The fully faithful functors
\begin{align*}
	&H^0 j_*: \CAlg_R^{a \ic} \hookrightarrow \CAlg_R^{\ic} && H^0j_*: \CAlg_R^{\wedge a \ic} \hookrightarrow \CAlg_R^{\wedge \ic}
\end{align*}
admit left adjoints given by
\begin{align*}
	&(-)^a : \CAlg_R^{\ic} \rightarrow \CAlg_R^{a \ic} && A \mapsto A^a \\
	&(-)^a : \CAlg_R^{\wedge \ic} \rightarrow \CAlg_R^{\wedge a \ic} && A \mapsto A^a
\end{align*}
Moreover, every functor on the left hand side diagram admits a left adjoint, making the diagram on the right commute
\begin{cd}
	\CAlg_R^{\wedge a \ic} \ar[r, hook] \ar[d, hook] & \CAlg_R^{\wedge \ic} \ar[d, hook] &&
	\CAlg_R^{\ic} \ar[r, "(-)^a"] \ar[d, "-^\wedge", swap] & \CAlg_R^{a \ic} \ar[d, "-^\wedge"] \\
	\CAlg_R^{a \ic} \ar[r, hook] & \CAlg_R^{\ic} &&
	\CAlg_R^{\wedge \ic} \ar[r, "(-)^a"] & \CAlg_R^{\wedge a \ic}
\end{cd}
\end{prop}

\begin{proof} To show that the fully faithful functor $H^0 j_*: \CAlg_R^{a \ic} \hookrightarrow \CAlg_R^{\ic}$ admits a left adjoint given by $(-)^a$ it suffices to show that if $A \in \CAlg_R^{\ic}$, then $A_*$ is still integrally closed in $A[\frac{1}{\varpi}]$, which is shown in \ref{stab_ic}. Similarly, to show that the fully faithful functor $H^0 j_*: \CAlg_R^{\wedge a \ic} \hookrightarrow \CAlg_R^{\wedge \ic}$ admits a left adjoint given by $(-)^a$, it suffices to show that if $A \in \CAlg_R^{\wedge \ic}$ then $A_*$ is still integrally closed in $A[\frac{1}{\varpi}]$ and $\varpi$-complete. We have already argued that $A_*$ is integrally closed in $A[\frac{1}{\varpi}]$, and the $\varpi$-completeness of $A_*$ follows from \ref{abelian_internal_hom}.

To conclude, recall we already showed that the fully faithful functors $\CAlg_R^{\wedge \ic} \hookrightarrow \CAlg_R^{\ic}$ and $\CAlg_R^{\wedge a \ic} \hookrightarrow \CAlg_R^{a \ic}$ admit left adjoint given by $\varpi$-completion in \ref{ic_algebras} and \ref{ic_almost_algebras} respectively.
\end{proof}

\subsection{Total integrally closed algebras}

Throughout this section $R$ is an integral perfectoid ring which is $\varpi$-complete with respect to some $\varpi \in R$ where $\varpi^p$ divides $p$, and such that $R$ is $\varpi$-torsion free. In particular, recall that this implies that $\varpi \in R$ admits compatible $p$-power roots up to a unit.

\begin{defn}\label{defn_total_int_closed} A $\varpi$-torsion free $R$-algebra $S$ is said to be total integrally closed if for any $f \in S[1/\varpi]$ which satisfies $\{f^{\ZZ_{\ge 0}}\} \subset \frac{1}{\varpi^m} S \subset S[1/\varpi]$ for some $m \in \ZZ_{\ge 0}$, then $f \in S$.
\end{defn}

\begin{defn}\label{defn_tot_int_closure} For an object $S \in \CAlg_R^{\pic}$ we may define its total integral closure with respect to $S \subset S[\frac{1}{\varpi}]$, denoted by $S^{\tic}$, as the collection of all elements $f \in S[\frac{1}{\varpi}]$ for which there exists an $m \in \ZZ_{\ge 0}$ where $\{f^{\ZZ_{\ge 0}}\} \subset \frac{1}{\varpi^m} S \subset S[\frac{1}{\varpi}]$. One can easily check that $S^{\tic} \subset S[\frac{1}{\varpi}]$ inherits a ring structure from $S[\frac{1}{\varpi}]$.
\end{defn}

\begin{warning} Even though the definition of total integral closure would make sense for an object $S \in \CAlg_R^{\tf}$ the operation described above is not well behaved, in particular the total integral closure of $S \in \CAlg_R^{\tf}$ need not be total integrally closed. Finally, let us remark that when $S$ is noetherian then the total integral closure agree with the integral closure, but in general we may have an strict inclusion $S^{\ic} \subset S^{\tic}$.
\end{warning}

\begin{lemma}\label{stab_tic} Let $A$ be a $\varpi$-torsion free $R$-algebra, and assume that $A$ is total integrally closed with respect to $A \subset A[\frac{1}{\varpi}]$. Then,
\begin{enumerate}[(1)]
	\item The $\varpi$-completion of $A$, denoted by $A^\wedge$, is totally integrally closed with respect to $A^{\wedge} \subset A^{\wedge}[\frac{1}{\varpi}]$.
	\item The almost elements of $A$, denoted by $A_*$ (\ref{almost_elements}), is total integrally closed with respect to $A_* \subset A[\frac{1}{\varpi}]$.
\end{enumerate}
\end{lemma}

\begin{proof} \cite[Lemma 5.1.3]{bhattlecture_perfectoid}
\end{proof}

\begin{lemma}\label{tic_of_pic_preserves_complete} Let $S$ be a $\varpi$-torsion free $p$-integrally closed $R$-algebra. Then, the ring $S^{\tic}$ is total integrally closed. Moreover, if $S$ is $\varpi$-complete then $S^{\tic}$ is $\varpi$-complete.
\end{lemma}

\begin{proof} Recall that $\varpi \in R$ admits compatible $p$-power roots up to a unit, thus we assume without loss of generality that $\varpi \in R$ admits compatible $p$-power roots. We claim that if $f \in S^{\tic}$ then $\{f^{\ZZ_{\ge 0}} \} \subset \frac{1}{\varpi^{1/p^n}} S$ for all $n \in \ZZ_{\ge 0}$. Indeed, by assumption we have that there exists an $m \in \ZZ_{\ge 0}$ such that $f^{\ZZ_{\ge 0}} \varpi^m \in S$; and the $p$-integral closure of $S \subset S[\frac{1}{\varpi}]$ implies that if $f^k \varpi^m \in S$ then $f^k \varpi^{m/p} \in S$. By repeating the procedure we get the claim.
	
To show that $S^{\tic}$ is totally integrally closed it suffices to show that if $\{f^{\ZZ_{\ge 0}} \} \subset \frac{1}{\varpi^m} S^{\tic}$ then $\{f^{\ZZ_{\ge 0}} \} \subset \frac{1}{\varpi^{m+1}} S$, but since $S^{\tic} \subset \frac{1}{\varpi} S$ the result follows. Next, to show that $S^{\tic}$ is $\varpi$-complete it suffices to show that it is derived $\varpi$-complete as it is $\varpi$-torsion free by construction; moreover, using the fiber sequence
\begin{equation*}
	S \hookrightarrow S^{\tic} \rightarrow S^{\tic}/S
\end{equation*}
we are reduced to showing that $S_{\tic}/S$ is derived $\varpi$-complete, since the subcategory $\cD(R)\picomp \subset \cD(R)$ is closed under extensions. Thus, we need to show that the canonical map
\begin{equation*}
	S_{\tic}/S \longrightarrow R\lim_{n} \Kos(S_{\tic}/S, \varpi^n)
\end{equation*}
is an isomorphism. Writing this limit explicitly we get the diagram
\begin{cd}
	\cdots \ar[r, "\times \varpi"] & S_{\tic}/S \ar[d, "\times \varpi^n"] \ar[r, "\times \varpi"] & S_{\tic}/S \ar[d, "\times \varpi^{n-1}"] 
	\ar[r, "\times \varpi"] & \cdots\\
	\cdots \ar[r, "\Id"] & S_{\tic}/S  \ar[r, "\Id"] & S_{\tic}/S \ar[r, "\Id"] & \cdots
\end{cd}
But notice that every element of $S_{\tic}/S$ is killed by $\varpi$, so all vertical maps and all the upper horizontal maps are zero, making it clear that the desired map $S_{\tic}/S \rightarrow R\lim_{n} \Kos(S_{\tic}/S, \varpi^n)$ is an isomorphism.
\end{proof}

\begin{defn} Let $R$ is an integral perfectoid ring which is $\varpi$-complete with respect to some $\varpi \in R$ where $\varpi^p$ divides $p$, and such that $R$ is $\varpi$-torsion free.
\begin{enumerate}[(1)]
	\item The category $\CAlg_R^{\tic} \subset \CAlg_R^{\ic}$ of $\varpi$-torsion free $R$-algebras $A$ which are total integrally closed in $A[\frac{1}{\varpi}]$.
	\item The category $\CAlg_R^{\wedge \tic} \subset \CAlg_R^{\wedge \ic}$ of $\varpi$-torsion free $\varpi$-complete $R$-algebras $A$ which are total integrally closed with respect to $A[\frac{1}{\varpi}]$.
	\item Regarding $\CAlg_R^{a \ic}$ as a full-subcategory of $\CAlg_R^{\ic}$ via the functor $H^0 j_*$, we define the almost category of total integrally closed $R$-algebras as $\CAlg_R^{a \tic} := \CAlg_R^{\tic} \cap \CAlg_R^{a \tf}$.
	\item Regarding $\CAlg_R^{\wedge a \ic}$ as a full-subcategory of $\CAlg_R^{\wedge \ic}$ via the functor $H^0 j_*$, we define the almost category of $\varpi$-complete total integrally closed $R$-algebras as $\CAlg_R^{\wedge a \tic} := \CAlg_R^{\wedge a \ic} \cap \CAlg_R^{\tic}$.
\end{enumerate}
\end{defn}

\begin{prop}\label{tic_algebras} The fully faithful functors
\begin{align*}
	\CAlg_R^{\tic} \hookrightarrow \CAlg_R^{\ic} && \CAlg_R^{\wedge \tic} \hookrightarrow \CAlg_R^{\wedge \ic}
\end{align*}
admit left adjoints
\begin{align*}
	&(-)^{\tic}: \CAlg_R^{\ic} \rightarrow \CAlg_R^{\tic} && A \mapsto A^{\tic} \\
	&(-)^{\tic}: \CAlg_R^{\wedge \ic} \rightarrow \CAlg_R^{\wedge \tic} && A \mapsto A^{\tic}
\end{align*}
Moreover, every functor on the left hand side diagram admits a left adjoint, making the diagram on the right commute
\begin{cd}
	\CAlg_R^{\wedge \tic} \ar[r, hook] \ar[d, hook] & \CAlg_R^{\tic} \ar[d, hook] &&
	\CAlg_R^{\ic} \ar[r, "-^\wedge"] \ar[d, "(-)^{\tic}", swap] & \CAlg_R^{\wedge \ic} \ar[d, "(-)^{\tic}"] \\
	\CAlg_R^{\wedge \ic} \ar[r, hook] & \CAlg_R^{\ic} &&
	\CAlg_R^{\tic} \ar[r, "-^\wedge"] & \CAlg_R^{\wedge \tic}
\end{cd}
\end{prop}

\begin{proof} First, we show that the fully faithful functor $\CAlg_R^{\tic} \rightarrow \CAlg_R^{\ic}$ admits a left adjoint given by $(-)^{\tic}$. Let $\varphi: A \rightarrow C$ be a morphism in $\CAlg_R^{\ic}$ and such that $C$ is total integrally closed with respect to $C \subset C[\frac{1}{\varpi}]$, we need to show that there exists an essentially unique factorization of $\varphi$ as $A \rightarrow A^{\tic} \rightarrow C$. Let $f \in A[\frac{1}{\varpi}]$ such that there exists an $m \in \ZZ_{\ge 0}$ where $\{f^{\ZZ_{ \ge 0}} \} \subset \frac{1}{\varpi^m} A$, then $\varphi(f) \in C[\frac{1}{\varpi}]$ satisfies $\{ \varphi(f)^{\ZZ_{\ge 0}} \} \subset \frac{1}{\varpi^m} C$, but since $C$ is total integrally closed it follows that $f \in C$. Thus, it follows that we have an essentially unique factorization of $\varphi$ as $A \rightarrow A^{\tic} \rightarrow C$, and since $A^{\tic}$ is total integrally closed by \ref{tic_of_pic_preserves_complete} it follows that $(-)^{\tic}$ is the desired left adjoint. Similarly, we show that $(-)^{\tic}$ is left adjoint to the fully faithful functor $\CAlg_R^{\wedge \tic} \hookrightarrow \CAlg_R^{\wedge \ic}$. We need to show that for any morphism $\varphi: A \rightarrow C$ in $\CAlg_R^{\wedge \ic}$ where $C$ is total integrally closed, there exists an essentially unique factorization $A \rightarrow A^{\tic} \rightarrow C$; but this is what we showed above. Thus, it remains to show that $A^{\tic}$ is totally integrally closed and $\varpi$-complete, which follows from \ref{tic_of_pic_preserves_complete}.

Finally, recall that we already showed in \ref{ic_algebras} that $\CAlg_R^{\wedge \ic} \hookrightarrow \CAlg_R^{\ic}$ admits a left adjoint given by $\varpi$-completion. It remains to show that $\CAlg_R^{\wedge \tic} \hookrightarrow \CAlg_R^{\tic}$ admits a left adjoint given by $\varpi$-completion. It is clear that for a morphism $A \rightarrow C$ in $\CAlg_R^{\tic}$ and where $C$ is $\varpi$-complete that there is an essentially unique factorization $A \rightarrow A^\wedge \rightarrow C$, and by \ref{stab_tic} it follows that $A^\wedge$ is total integrally closed with respect to $A^\wedge \subset A^\wedge[\frac{1}{\varpi}]$, proving the claim.
\end{proof}

\begin{prop}\label{tic_almost_algebras} The fully faithful functors
\begin{align*}
	\CAlg_R^{a \tic} \hookrightarrow \CAlg_R^{a \ic} && \CAlg_R^{\wedge a \tic} \hookrightarrow \CAlg_R^{\wedge a \ic}
\end{align*}
are equivalences with inverse given by 
\begin{align*}
	&(-)^{a \tic}: \CAlg_R^{a \ic} \rightarrow \CAlg_R^{a \tic} && A \mapsto (H^0 j_* A)^{\tic, a} \\
	&(-)^{a \tic}: \CAlg_R^{\wedge a \ic} \rightarrow \CAlg_R^{\wedge a \tic} && A \mapsto (H^0 j_* A)^{\tic, a}
\end{align*}
Moreover, every functor on the left hand side diagram admits a left adjoint, making the diagram on the right commute
\begin{cd}
	\CAlg_R^{\wedge a \tic} \ar[r, hook] \ar[d, hook, "\simeq", swap] & \CAlg_R^{a \tic} \ar[d, hook, "\simeq"] &&
	\CAlg_R^{a \ic} \ar[r, "-^\wedge"] \ar[d, "(-)^{a \tic}", swap] & \CAlg_R^{\wedge a \ic} \ar[d, "(-)^{a \tic}"] \\
	\CAlg_R^{\wedge a \ic} \ar[r, hook] & \CAlg_R^{a \ic} &&
	\CAlg_R^{a \tic} \ar[r, "-^\wedge"] & \CAlg_R^{\wedge a \tic}
\end{cd}
\end{prop}

\begin{proof} To show that $\CAlg_R^{a \tic} \hookrightarrow \CAlg_R^{a \ic}$ is an equivalence with inverse $(-)^{a \tic}$, it suffices to show that $S:= H^0 j_* A$ is total integrally closed with respect to $S[\frac{1}{\varpi}]$. By hypothesis we have that $S$ is integrally closed in $S[\frac{1}{\varpi}]$ (in particular $p$-integrally closed), and by the proof of \ref{tic_of_pic_preserves_complete} we can conclude that if $f \in S^{\tic}$ then $\{f^{\ZZ_{\ge 0}} \} \subset \frac{1}{\varpi^{1/p^n}} S$ for all $n \in \ZZ_{\ge 0}$. As $S$ satisfies $S \simeq S_*$ it follows that $f \in S$, showing that $S$ is total integrally closed. The same argument, coupled with the fact that if $A \in \CAlg_R^{\tf}$ is $\varpi$-complete then so is $A_*$, shows that $\CAlg_R^{\wedge a \tic} \hookrightarrow \CAlg_R^{\wedge a \ic}$ is an equivalence.

It remains to show that $\CAlg_R^{\wedge a \tic} \hookrightarrow \CAlg_R^{a \tic}$ admits a left adjoint given by $\varpi$-completion. Let $A \rightarrow C$ be a morphism in the essential image of $H^0 j_*: \CAlg_R^{a \tic} \rightarrow \CAlg_R^{\tic}$ where $C$ is $\varpi$-complete; then there exists an essentially unique factorization $A \rightarrow A^\wedge \rightarrow C$ where $A^\wedge$ is total integrally closed in $A^{\wedge}[\frac{1}{\varpi}]$ by \ref{stab_tic}, and satisfies $A^\wedge \simeq (A^\wedge)_*$ by \ref{const_almost_comp_localization}.
\end{proof}

\begin{prop}\label{almost_funct_tic_algebras} The fully faithful functors
\begin{align*}
	H^0 j_*: \CAlg_R^{a \tic} \hookrightarrow \CAlg_R^{\tic} && H^0 j_*: \CAlg_R^{\wedge a \tic} \hookrightarrow \CAlg_R^{\wedge \tic}
\end{align*}
are equivalences with inverse given by
\begin{align*}
	&(-)^a: \CAlg_R^{\tic} \rightarrow \CAlg_R^{a \tic} && A \mapsto A^a \\
	&(-)^a: \CAlg_R^{\wedge \tic} \rightarrow \CAlg_R^{\wedge a \tic} && A \mapsto A^a
\end{align*}
Moreover, every functor on the left hand side diagram admits a left adjoint, making the diagram on the right commute
\begin{cd}
	\CAlg_R^{\wedge a \tic} \ar[r, hook] \ar[d, hook, "\simeq", swap] & \CAlg_R^{a \tic} \ar[d, hook, "\simeq"] &&
	\CAlg_R^{\tic} \ar[r, "-^\wedge"] \ar[d, "(-)^a", swap] & \CAlg_R^{\wedge \tic} \ar[d, "(-)^a"] \\
	\CAlg_R^{\wedge \tic} \ar[r, hook] & \CAlg_R^{\tic} &&
	\CAlg_R^{a \tic} \ar[r, "-^\wedge"] & \CAlg_R^{\wedge a \tic}
\end{cd}
\end{prop}

\begin{proof} To show that $H^0 j_*: \CAlg_R^{a \tic} \hookrightarrow \CAlg_R^{\tic}$ is an equivalence it suffices to show that for $A \in \CAlg_R^{\tic}$ we have the identity $A \simeq A_*$. Recall from \ref{almost_elements} that $A_* := \uHom_{R}((\varpi)_{\perfd}, A)$ can be identified with the set of elements of $a \in A[\frac{1}{\varpi}]$ such that $\varepsilon a \in A$ for all $\varepsilon \in (\varpi)_{\perfd}$, and if we assume without loss of generality that $\varpi$ admits compatible $p$-power roots $(\varpi)_{\perfd}$ can be identified with $\cup (\varpi^{1/p^n})$. Thus $A_*$ can be identified with the set of elements $a \in A[\frac{1}{\varpi}]$ which satisfy $a \varpi^{1/p^n} \in A$ for all $n \in \ZZ_{\ge 0}$, and since $A$ is total integrally closed in $A[\frac{1}{\varpi}]$ it follows that $A_* \simeq A$ as desired. The exact same argument shows that $H^0 j_*: \CAlg_R^{\wedge a \tic} \hookrightarrow \CAlg_R^{\wedge \tic}$ is an equivalence with the desired inverse.

The fact that the functors $\CAlg_R^{\wedge \tic} \hookrightarrow \CAlg_R^{\tic}$ and $\CAlg_R^{\wedge a \tic} \hookrightarrow \CAlg_R^{a \tic}$ admit left adjoints given by $\varpi$-completion was already established in \ref{tic_algebras} and \ref{tic_almost_algebras} respectively.
\end{proof}

\begin{prop}\label{tic_preserves_perfectoid} Let $S$ be a $\varpi$-complete $\varpi$-torsion free integral perfectoid $R$-algebra. Then, $S^{\tic}$ is a $\varpi$-complete $\varpi$-torsion free integral perfectoid $R$-algebra.
\end{prop}

\begin{proof} First, we show that the Frobenius map $\varphi: S^{\tic}/\varpi^p \rightarrow S^{\tic}/\varpi^p$ is surjective. Let $\overline{x} \in S^{\tic}/\varpi^p$, and $x \in S^{\tic}$ such that $x = \overline{x} \bmod \varpi^p$, then since $S$ is integral perfectoid we know that Frobenius is surjective $\varphi: S/\varpi^p \rightarrow S/\varpi^p$, which in turn implies that there exists $y, z \in S$ such that
\begin{equation*}
	\varpi x = y^p + \varpi^p z \in S
\end{equation*}
where we are implicitly using that $\varpi^{1/p^n} x \in S$ for all $n \in \ZZ_{\ge 0}$. We claim that $y^p/\varpi \in S^{\tic} \subset S[1/\varpi]$. Indeed, since $x + \varpi^{p-1} z = y^p/\varpi \in S[1/\varpi]$, it would suffice to show that $\varpi^{1/p^n} (x + \varpi^{p-1}z) \in S$, but this is clear as $z \in S$ and $\varpi^{1/p^n} x \in S$. And since $S^{\tic} \subset S[1/\varpi]$ is $p$-integrally closed it follows that $y/\varpi^{1/p} \in S^{\tic}$, showing that $\varphi: S^{\tic}/\varpi^p \rightarrow S^{\tic}/\varpi^p$ is surjective.

Next, we need to show that the Frobenius map $\varphi: S^{\tic}/\varpi \rightarrow S^{\tic}/\varpi^p$ is injective. Let $\overline{x} \in S^{\tic}/\varpi$ such that $\varphi(\overline{x}) = 0$, then there exists $x \in S^{\tic}$ such that $x = \overline{x} \bmod \varpi$, and a $y \in S^{\tic}$ such that $x^p = \varpi^p y \in S^{\tic}$. It suffices to show that $x \in \varpi S$, by construction $x^p/\varpi^p \in S^{\tic}$ and by the $p$-integral closedness of $S^{\tic}$ it follows that $x/\varpi \in S^{\tic}$, showing that $x \in \varpi S^{\tic}$ as desired.

Finally, we by Lemma \ref{tic_of_pic_preserves_complete} it follows that $S^{\tic}$ is $\varpi$-complete and it is $\varpi$-torsion free by construction. To show that $S^{\tic}$ is integral perfectoid it suffices to show that the Frobenius map $\varphi: S^{\tic}/\varpi \rightarrow S^{\tic}/\varpi^p$ is surjective by \cite[Lemma 3.10]{BMS1}, but this follows from the fact that Frobenius $\varphi: S^{\tic}/\varpi^p \rightarrow S^{\tic}/\varpi^p$ factors as $S^{\tic}/\varpi^p \twoheadrightarrow S^{\tic}/\varpi \rightarrow S^{\tic}/\varpi^p$, where the first map is the canonical projection.
\end{proof}

\newpage

\section{Banach Algebras}\label{sect_ban_alg}

\subsection{Definitions and basic properties}

Throughout this section all Banach rings are assumed to be non-archimedean. Furthermore, we will assume that all (non-archimedean) Banach rings $R$ have a topological nilpotent unit, that is, there exists a $\varpi \in R^{\times}$ such that $\varpi^n \rightarrow 0$ as $n \rightarrow \infty$.

\begin{defn}\label{defn_banach} A (non-archimedean) Banach ring $R$ is a ring $R$ equipped with a map $|-|: R \rightarrow \RR_{\ge 0}$ satisfying
\begin{enumerate}[(1)]
	\item $|f| = 0$ only if $f = 0$
	\item $|f| = |-f|$
 	\item $|fg| \le |f||g|$
  \item $|f + g| \le \max(|f|, |g|)$
  \item $R$ is complete with respect to the metric $d(f,g) = |f-g|$
\end{enumerate}
A map $|-|: R \rightarrow \RR_{\ge 0}$ satisfying the above conditions is called a (non-archimedean) norm on $R$. A morphism of Banach rings $\varphi: R \rightarrow S$ is a morphisms of rings $R \rightarrow S$ which is bounded, that is, there exists a $C > 0$ such that $|\varphi(f)|_{S} \le C|f|_{R}$ for all $f \in R$. We denote the category of all (non-archimedean) Banach rings by $\Ban$.

An alternative category of interest is the category of $\Ban^{\contr}$ of Banach rings with contractive maps. Recall that a map $\varphi: R \rightarrow S$ of Banach rings is contractive if $|\varphi(f)|_{S} \le |f|_{R}$ for all $f \in R$.
\end{defn}

In what follows it is convenient to sometimes leave the world of Banach algebras and work with general semi-normed and normed rings. We do not review this notions here, but refer the reader to Section 1.1 and 1.2 of \cite{BGR} for details.

\begin{defn} A (non-archimedean) Banach ring $K$ is called a non-archimedean field if the norm on $K$ is multiplicative, that is, we have $|xy| = |x| |y|$ for any pair of elements $x,y \in K$.
\end{defn}

\begin{defn}\label{defn_banach_algebras} Let $B$ be a (non-archimedean) Banach ring.  A (non-archimedean) Banach $B$-algebra $R$ is a Banach ring $R$ together with a contractive structure map $\varphi: B \rightarrow R$ of Banach rings. A morphisms of Banach $B$-algebras $R_1 \rightarrow R_2$ is a morphisms of Banach rings that commute with the structure map from $B$. We denote the category of (non-archimedean) Banach $B$-algebras by $\Ban_B$.

We will also be interested in the category $\Ban_B^{\contr}$ of Banach $B$-algebras with contractive maps. Unlike the situation with $\Ban_B$ it is easy to see that $\Ban_B^{\contr}$ is equivalent to the category $\Ban_{B /}^{\contr}$ of Banach rings $R$ and contractive maps, equipped with a contractive map $B \rightarrow R$.
\end{defn}

\begin{rem} Given a contractive map of Banach rings $\varphi: K \rightarrow R$, if we further assume that $K$ is a non-archimedean field then the map $\varphi: K \rightarrow R$ is an isometry; that is, we have an equality $|\varphi(a)|_R = |a|_K$.
\end{rem}

Next, we would like to show that the norm in the definition of Banach rings can be recovered (up to isomorphism) from the topology on the Banach ring. With this in mind, notice that any Banach ring is a topological Hausdorff ring with continuous multiplication (cf. \cite[Section 1.2]{BGR}), and that bounded morphisms of Banach rings give rise to continuous morphisms of topological Hausdorff rings (cf. \cite[Section 1.1]{BGR})

\begin{prop}\label{banach_to_topological_ff} Let $B$ be a (non-archimedean) Banach ring. Then, the forgetful functor
	\begin{equation*}
		\Ban_B \longrightarrow \{\text{Topological Hausdorff } B \text{-algebras} \}
	\end{equation*}
is fully faithful.
\end{prop}

\begin{proof} It suffices to show that for any continuous map of Banach $B$-algebras $f: R \rightarrow S$ regarded as topological Hausdorff $B$-algebras, the map $f: R \rightarrow S$ is bounded with respect to their norms. Endowing the ring $\im(f) \subset S$ with the subspace topology (which comes from the norm inherited from $S$), it suffices to show that the continuous map $R \rightarrow \im(f)$ is bounded. But since the map $R \rightarrow \im(f)$ is surjective, it is an open map by the open mapping theorem \cite[Theorem 0.1]{henkel2014open}, so in particular it is bounded.

Let us emphasize that in order to invoke \cite[Theorem 0.1]{henkel2014open} its critical that $B$ admits a topological nilpotent unit $\varpi \in B^{\times}$. Furthermore, one needs to show that every element $R \in \Ban_B$ admits a countable fundamental systems of open neighborhoods of $0$, but this is clear as $\{R_{< 1/n}\}$ form such a fundamental system for $R$.
\end{proof}

Of special interest to us will be the subcategory of uniform Banach algebras, which contains all perfectoid Banach algebras.

\begin{defn} A Banach ring $R$ is said to be uniform, if the norm is power-multiplicative, that is $|a^n| = |a|^n$ for all $a \in R$ and $n \ge 1$.
\end{defn}

The following provides a convenient mechanism for proving that various uniform Banach algebras are non-archimedean. We often use it without mention in the sequel.

\begin{lemma}\label{recog_na} Let $A$ be a uniform Banach ring. Then, $A$ is non-archimedean if and only if $|n|_A \le 1$ for all $n \in \ZZ$.
\end{lemma}

\begin{proof} If $A$ is non-archimedean its clear that $|n|_{A} \le 1$ for all $n \in \ZZ$. Conversely, assume that the norm on $A$ is power-multiplicative and that $|n|_{A} \le 1$ for all $n \in\ZZ$. Then, for any pair of objects $a,b \in A$ we have the identity
\begin{align*}
	|a+b| = |(a+b)^n|^{1/n} = \Big | \sum_{j = 0}^n \binom{n}{j} a^j b^{n-j} \Big |^{1/n} \le \Big ( \sum_{j = 0}^n |\binom{n}{j}| |a|^j |b|^{n-j} \Big )^{1/n} \le (n+1)^{1/n} \max(|a|, |b|)
\end{align*}
ans since $\lim_{n \rightarrow \infty} (n+1)^{1/n} = 1$, it follows that $|a+b| \le \max(|a|, |b|)$.
\end{proof}

In general, given a Banach ring $R$ we can always construct a power-multiplicative semi-norm (cf. \cite[Section 1.1]{BGR}) on $R$ from the given norm.

\begin{defn}\label{defn_sp_radius} Let $R$ be a Banach ring, we define the spectral radius $\rho: R \rightarrow \RR_{\ge 0}$ by the formula
\begin{equation*}
	\rho(a) := \lim_{n} |a^n|^{1/n} 
\end{equation*}
\end{defn}

In particular, we know that the limit exists and that it gives rise to a power-multiplicative semi-norm on $R$  which satisfies $\rho(a) \le |a|$ for all $a \in R$ \cite[Section 1.3.2]{BGR}. Furthermore, since it clearly satisfies $\rho(m) \le 1$ for all $m \in \ZZ$ we conclude that the induced semi-norm is non-archimedean \cite[Section 1.2]{BGR}.

\begin{lemma} Let $R$ be a Banach ring, and $|-|_1$ and $|-|_2$ be two equivalent norms on $R$. Then, $\rho_1 = \rho_2$, where $\rho_i$ is the spectral radius coming from the norm $|-|_i$.
\end{lemma}

\begin{proof} If $|-|_1$ and $|-|_2$ are equivalent norms on $R$, then we know that there exists constants $C_1, C_2 >0$ such that
\begin{equation*}
	C_1 |a^n|_1 \le |a^n|_2 \le C_2 |a^n|_1
\end{equation*}
and $\lim_n C_i^{1/n} = 1$ the result follows.
\end{proof}

\begin{lemma}\label{morphisms_unif_ban_contractive} Let $f: R_1 \rightarrow R_2$ be a morphisms of Banach rings where $R_2$ is uniform, in other words $R_2$ are endowed with the spectral radius norm. Then, the morphism $f$ is contractive, that is, we have $|f(a)|_{R_2} \le |a|_{R_1}$.
\end{lemma}

\begin{proof} \cite[Section 1.3.1]{BGR}
\end{proof}

\begin{defn}\label{defn_uniform_banach} Let $\uBan \subset \Ban$ be the full-subcategory spanned by uniform Banach rings, and if $B$ is a Banach algebra we denote by $\uBan_B \subset \Ban_B$ the full subcategory spanned by uniform Banach $B$-algebras.

Recall that as part of the definition of Banach $B$-algebras we have that the structure map $B \rightarrow R$ is contractive. By Lemma \ref{morphisms_unif_ban_contractive} we know that all morphisms that have as target a uniform Banach algebra are contractive. Hence, the category $\uBan_{B/}$ of uniform Banach rings equipped with a bounded map $B \rightarrow R$ is equivalent to $\uBan_B$. Moreover, by Proposition \ref{adjunction_uniform_ban} we conclude that $\uBan_B = \uBan_{B^u}$, where $B^u$ is the uniformization of $B$ defined in Construction \ref{const_uniformization}.
\end{defn}

Next, we introduce the notion of uniformization for general Banach rings and characterized them by a universal property.

\begin{const}\label{const_uniformization} Let $(R, |-|)$ be a Banach ring with norm $|-|$. Then, we claim that there exists a (canonical) contractive map of Banach rings $(R, |-|) \rightarrow (R^u, \rho)$, where $R^u$ is the completion of $R$ with respect to $\rho$, and $R \rightarrow R^u$ is the completion map of their underlying sets. By \cite[Section 1.1.7]{BGR} we learn that $R^u$ is complete as a group, and that the completion map of the underlying sets $R \rightarrow R^u$ lifts to a contractive map of complete groups. Multiplication on $R^u$ is defined by using the density of the image $R \rightarrow R^u$, which also shows that $R \rightarrow R^u$ is a ring map. Finally, the fact that $R^u$ is non-archimedean follows from \ref{recog_na} and the fact that $R \rightarrow R^u$ is contractive with $R$ non-archimedean.
\end{const}

\begin{prop}\label{adjunction_uniform_ban} Let $B$ be a (non-archimedean) Banach ring. The inclusion $\uBan_B \hookrightarrow \Ban_B$ admits a left adjoint
	\begin{equation*}
		(-)^u: \Ban_B \longrightarrow \uBan_B
	\end{equation*}
which we call the uniformization functor.
\end{prop}

\begin{proof} It suffices to show that for any morphisms $R \rightarrow S$ of Banach $B$-algebras, where $S$ is uniform, admits a unique factorization as $R \rightarrow R^u \rightarrow S$, where $R \rightarrow R^u$ is the uniformization map. As abstract groups this follows from Proposition 6 of \cite[Section 1.1.7]{BGR}, and by Construction \ref{const_uniformization} we know that $R \rightarrow R^u$ is a ring map, so it suffices to show that $R^u \rightarrow S$ is a ring map. But since multiplication on $R^u$ is completely determined by the dense image of $R \rightarrow R^u$, it follows that $R^u \rightarrow S$ is a ring map as $R \rightarrow S$ is a ring map.
\end{proof}

We are now ready to introduce the tensor product of Banach $B$-algebras.

\begin{const}\label{const_comp_tensor_ban} Let $B$ be a (non-archimedean) Banach ring and $R, S$ be elements of $\Ban_B$. Then we can endow the ring $R \otimes_B S$ with the semi-norm
\begin{equation*}
	|f| = \inf\{ \max_{i} \{ |r_i||s_i| \}  \text{ for } f = \sum_{i} r_i \otimes s_i \}
\end{equation*}
and denote by $R \cotimes_B S$ the completion of $R \otimes_B S$ with respect to the semi-norm above, making $R \cotimes_B S$ into a Banach ring. Furthermore, $R \cotimes_B S$ can be characterized as the pushout of the following diagram in the category $\Ban_B$
\begin{cd}
	B \ar[r] \ar[d] & R \ar[d] \\
	S \ar[r] & R \cotimes_B S
\end{cd}
Furthermore, we learn that the structure map $B \rightarrow R \cotimes_B S$ is contractive, and the maps $S \rightarrow R \cotimes_B S$ and $R \rightarrow R \cotimes_B S$ are contractive (cf. \cite[Section 3.1]{BGR}); and by construction it is clear that $R \cotimes_B S$ is non-archimedean. In other words, we have shown that the category $\Ban_B$ has finite coproducts, since $R \sqcup S = R \cotimes_B S$ as $B$ is the initial object.

A similar argument shows that in the category $\Ban_B^{\contr}$ of Banach $B$-algebras with contractive maps, for any pair of maps $R \leftarrow D \rightarrow S$ the pushout can be identified with $R \cotimes_D S$ (cf. \cite[Section 3.1]{BGR}).
\end{const}

\begin{warning} We want to emphasize that the proofs in \cite[Section 3.1]{BGR} use in a critical way the fact that the structure map $B \rightarrow R$ and $B \rightarrow S$ are contractive. Therefore, we cannot extend the same argument to show that any pair of maps $R \leftarrow D \rightarrow S$ in $\Ban_B$ have $R \cotimes_D S$ as its pushout, as its not guaranteed that either map $R \leftarrow D \rightarrow S$ is contractive, we only know they are bounded.
\end{warning}

However, the situation is slightly better if we restrict ourselves to uniform Banach algebras.

\begin{const}\label{const_unif_tensor_ban} Let $B$ be a Banach ring, and $R,S$ elements of $\uBan_B$. Then, we can endow $R \cotimes_B S$ from Construction \ref{const_comp_tensor_ban} with the spectral radius semi-norm $\rho$, and we denote its completion with respect to $\rho$ by $R \cotimes^u_B S$. We know that $R \cotimes^u_B S$ is a Banach ring, as it is the completion of a Banach ring, and that the structure map $B \rightarrow R \cotimes^u_B S$ is contractive, lifting $R \cotimes^u_B S$ to an element of $\uBan_B$. Furthermore by the adjunction of Proposition \ref{adjunction_uniform_ban} we learn that $R \cotimes^u_B S$ can be characterized as the pushout of the following diagram in the category $\uBan_B$
\begin{cd}
	B \ar[r] \ar[d] & R \ar[d] \\
	S \ar[r] & R \cotimes^u_B S
\end{cd}
Even better, say we have a pair of maps $R \leftarrow D \rightarrow S$ in $\uBan_B$, we claim that the pushout of $R \leftarrow D \rightarrow S$ can be identified with $R \cotimes_D^u S$. Indeed, any commutative diagram in $\uBan_B$ of the following form
\begin{cd}
	D \ar[r] \ar[d] & R \ar[d] \\
	S \ar[r] & H
\end{cd}
can be lifted to a diagram in $\uBan_D$, from which we obtain a unique map $R \cotimes_D^u S \rightarrow H$ making the diagram commute.
\end{const}

We conclude this section by showing that various natural topologies on $R$ are equivalent.

\begin{lemma}\label{equiv_topologies_seminormed} Let $R$ be a semi-normed ring with topological nilpotent unit $\varpi \in R^\times$. The following systems of open neighborhoods around zero are cofinal
\begin{equation*}
	\{\varpi^n R_{\le 1} \}_{n \in \ZZ_{\ge 0}} \qquad \{R_{\le |\varpi^n|} \}_{n \in \ZZ_{\ge 0}} \qquad \{R_{\le |\varpi|^n} \}_{n \in \ZZ_{\ge 0}}  \qquad \{R_{\le r} \}_{r \in \RR_{> 0}}
\end{equation*}
In other words, they induce equivalent topologies on $R$.
\end{lemma}

\begin{proof} By the non-archimedean triangle inequality it is clear that every subset of $R$ in $\{\varpi^n R_{\le 1}, R_{\le |\varpi^n|}, R_{\le |\varpi|^n}, R_{\le r}\}$ is an open subset of $R$. Moreover, since $\varpi \in R$ is a topological nilpotent unit, it follows that $\{R_{\le |\varpi|^n} \}_{n \in \ZZ_{\ge 0}}$ and $\{R_{\le r} \}_{r \in \RR_{> 0}}$ are cofinal systems of open neighborhoods. Thus, it remains to see that $\{\varpi^n R_{\le 1} \}_{n \in \ZZ_{\ge 0}}$,  $\{R_{\le |\varpi^n|} \}_{n \in \ZZ_{\ge 0}}$ and $\{R_{\le |\varpi|^n} \}_{n \in \ZZ_{\ge 0}}$ are cofinal systems.

For a fixed $n \in \ZZ_{\ge 0}$ we claim that we have inclusions
\begin{equation*}
	\varpi^n R_{\le 1} \subset R_{\le |\varpi^n|} \subset R_{\le |\varpi|^n}
\end{equation*}
Indeed, if $a \in \varpi^n R_{\le 1}$ we learn that there exists a $b \in R_{\le 1}$ such that $\varpi^n b = a$, and by the submultiplicativity of the semi-norm on $R$ we get that $|a| \le |\varpi^n| |b| \le |\varpi^n|$, as desired. The inequality $|\varpi^n| \le |\varpi|^n$ then implies the second inclusion. Finally, we need to show that for each $n \in \ZZ_{\ge 0}$ there exists a $N \gg 0$ such that $R_{\le |\varpi|^N} \subset \varpi^n R_{\le 1}$. Since $\varpi \in R^{\times}$ is topologically nilpotent, there exists a $N \gg 0$ such that $|\varpi|^N |\varpi^{-n}| \le 1$ for any fixed $n$, we claim that any such $N$ would satisfy $R_{\le |\varpi|^N} \subset \varpi^n R_{\le 1}$. It suffices to show that $|a \varpi^{-n}| \le 1$ for any $a \in R_{\le |\varpi|^N}$, which follows by the inequalities $|a \varpi^{-n}| \le |a| |\varpi^{-n}| \le |\varpi|^N |\varpi^{-n}| \le 1$.
\end{proof}

\begin{corollary}\label{banach_is_pi_comp} If $R$ is a Banach ring, then $R_{\le 1}$ is (classically) $\varpi$-complete.
\end{corollary}

\begin{proof} Since $R_{\le 1} \subset R$ is a closed and open subset of $R$, it follows that $R_{\le 1}$ is complete with respect to the norm $|-|$ inherited from $R$; equivalently, $R_{\le 1}$ is complete with respect to the topology induced by $\{ R_{\le r} \}_{r \in (0,1]}$. By definition, being $\varpi$-complete means that the canonical map $R_{\le 1} \rightarrow \lim R_{\le 1}/\varpi^n$ is an isomorphism; this is equivalent to showing that $R_{\le 1}$ is complete with respect to the topology induced by $\{\varpi^n R_{\le 1}\}$. Thus, the result follows from Lemma \ref{equiv_topologies_seminormed}.
\end{proof}

\begin{prop}\label{filtered_colim_banach_contr} Let $I$ be a filtered category, consider a functor $I \rightarrow \Ban_K^{\contr}$ indexing a family of Banach $K$-algebras $\{A_i\}_{i \in I}$ with contractive transition maps $A_i \rightarrow A_j$. Then, $\colim_I A_i$ computed in $\Ban^{\contr}_K$ exists and can be computed as follows: let $A = \colim_I A_i$ be the colimit computed in the category of rings and $\psi_i: A_i \rightarrow A$ the natural maps, endow $A$ with the seminorm
\begin{equation*}
	|a|_A := \inf \{ |a_i|_{A_i} \text{ for all } a_i \in A_i \text{ such that } \psi_i(a_i) = a \}
\end{equation*}
Denote by $A^\wedge$ the completion of $A$ with respect to the given seminorm. Then, $\colim_I A_i$ computed in $\Ban_K^{\contr}$ is isomorphic to $A^\wedge$.

Furthermore, if we have a collection of bounded morphisms $\{ h_i: A_i \rightarrow B\}_{i \in I}$ of Banach $K$-algebras satisfying the following conditions:
\begin{enumerate}[(1)]
	\item The morphisms $\{h_i\}_{i \in I}$ are compatible, that is, the composition $A_j \rightarrow A_i \rightarrow B$ is equal to $h_j$.
	\item There exists a constant $C > 0$ such that $|h_i(a_i)|_{B} \le C |a_i|_{A_i}$ uniformly for all $h_i$.
\end{enumerate}
Then, there exists a unique bounded map $h^{\wedge}: A^{\wedge} \rightarrow B$ such $|h^\wedge (a)|_B \le C|a|_{A^{\wedge}}$ and the composition $A_i \rightarrow A^{\wedge} \rightarrow B$ is equal to $h_i$.
\end{prop}

\begin{proof} Let $S$ be a Banach $K$-algebra, together with a collection of contractive morphisms $s_i: A_i \rightarrow S$ such that $A_j \rightarrow A_i \rightarrow S$ is equal to $s_j$. Since $A$ is defined as $\colim_I A_i$ in the category of rings it is clear that there exists a morphism of rings $s: A \rightarrow S$ such that $s \circ \psi_i = s_i$, we claim that endowing $A$ with the seminorm $|-|_A$ makes the map $s: A \rightarrow S$ contractive. We need to show that for all $a \in A$ we have the inequality $|s(a)|_S \le |a|_A$; indeed, let $\{a_i\}$ be the collection of all elements in $\{A_i\}$ which satisfy $\psi_i(a_i) = a$, then by hypothesis we have that $|s(a)|_S = |s_i(a_i)|_S \le |a_i|_{A_i}$, which in turn implies that $|s(a)|_S \le |a|_A$ as desired. Finally, since $S$ is complete with respect to its norm it follows that the map $A \rightarrow S$ factors as $A \rightarrow A^{\wedge} \rightarrow S$. The fact that $A^\wedge \rightarrow S$ is contractive follows from Proposition $6(ii)$ \cite[Section 1.1.7]{BGR}.

Next, as in the previous paragraph we have an induced map $h: A \rightarrow B$, which we need to show satisfies $|h(a)|_B \le C|a|_A$ for all $a \in A$. Let $\{a_i\}$ be the collection of elements in $\{A_i\}$ which satisfy $\psi_i(a_i) = a$, then by hypothesis we have that $|h(a)|_B \le C|a_i|_{A_i}$, which in turn implies that $|h(a)|_B \le C|a|_A$ as desired. Finally, since $B$ is complete with respect to its norm it follows that the map $A \rightarrow B$ factors as $A \rightarrow A^{\wedge} \rightarrow B$. The fact that $h^{\wedge}: A^{\wedge} \rightarrow B$ satisfies the bound $|h^{\wedge}(a)|_B \le C|a|_{A^{\wedge}}$ follows from Proposition $6(ii)$ \cite[Section 1.1.7]{BGR}.
\end{proof}

\subsection{Powerbounded elements}

Throughout this section all Banach rings are assumed to be non-archimedean. Furthermore, we will assume that all (non-archimedean) Banach rings $R$ have a topological nilpotent unit, that is, there exists a $\varpi \in R^{\times}$ such that $\varpi^n \rightarrow 0$ as $n \rightarrow \infty$.

\begin{defn}\label{defn_powerbounded} Let $R$ be a Banach ring, we denote by $R^\circ \subset R$ the subset of elements $f \in R$ such that 
\begin{equation*}
	\{ f^{\ZZ_{\ge 0}} \} \subset \frac{1}{\varpi^m} R_{\le 1}  \qquad \text{for some } m \in \ZZ_{\ge 0}
\end{equation*}
where $R_{\le 1}$ is the set of elements of $R$ with norm $\le 1$ (cf. Definition \ref{defn_total_int_closed}). One can easily check that $R^\circ \subset R$ inherits a ring structure from $R$, and we call $R^\circ$ the ring of powerbounded elements of $R$. Notice that Lemma \ref{equiv_topologies_seminormed} implies that an element $f \in R$ is powerbounded if and only if $\{|f^n|\}_{n \in \ZZ_{\ge 0}} \subset \RR_{\ge 0}$ is a bounded subset.
\end{defn}

The main goal of this section is to show that the spectral radius is closely related to the ring of powerbounded elements.

\begin{const} For a Banach ring $R$, let $R^{\circ, r} \subset R$ be the subset of elements of $R$, where
\begin{equation*}
	R^{\circ, r} : = \{ a \in R \text{ such that } r^{-n} |a^n| \text{ is bounded } \} \qquad \text{for a fixed } r \in \RR_{> 0}
\end{equation*}
This gives rise to a $\RR_{> 0}$-indexed filtration $\Fil^\circ R$ on $R$
\begin{equation*}
	\cdots \subset R^{\circ, r_1} \subset R^{\circ, r_2} \subset \cdots \subset R \qquad \text{with } \Fil^\circ_{r} R = R^{\circ, r}
\end{equation*}
where $r_1 < r_2$. Finally, let $\widehat{\Fil^\circ} R$ be the `right-continuous' completion of $\Fil^\circ R$, which is defined by the formula
\begin{equation*}
	\widehat{\Fil^\circ_r} R := \bigcap_{r^\prime > r} R^{\circ, r^\prime} \qquad \text{for all } r \in \RR_{\ge 0}
\end{equation*}
By construction, we have the following identity $\inf \{ r | a \in \Fil^\circ_r R \} = \min \{ r | a \in \widehat{\Fil^\circ_r} R \}$.
\end{const}

The following result shows that we can recover Temkin's construction of the graded reduction of $R$ (cf. \cite[Section 3]{temkin2004local})
\begin{equation*}
	\widehat{\text{Gr}^\circ} R := \bigoplus_{r \in \RR_{\ge 0}} \widehat{\Fil^\circ_r} R / \widehat{\Fil^\circ_{<r}} R \qquad \text{where} \qquad  \widehat{\Fil^\circ_{< r}} := \bigcup_{r^{\prime} < r} \widehat{\Fil^\circ_{r^\prime}} R
\end{equation*}
from the $\RR_{\ge 0}$-indexed filtration $\widehat{\Fil^\circ} R$.

\begin{prop}\label{spec_radius_power_bdd} Let $R$ be a Banach ring. Then,
\begin{equation*}
	R_{\rho \le r} = \widehat{\Fil^\circ_r} R
\end{equation*}
where $R_{\rho \le r} \subset R$ is the subset of elements of $R$ which have spectral radius $\le r$.
\end{prop}

\begin{proof} We will first show that $\widehat{\Fil^\circ_r} R \subset R_{\rho \le r}$. By construction, if $a \in \widehat{\Fil^\circ_r} R$ then for each $r^\prime > r$ there exists a $N_{r^\prime} \gg 0$ such that $r^{\prime -1} |a^n|^{1/n} \le N_{r^\prime}^{1/n}$ for all $n \in \ZZ_{\ge 0}$. The limit $\lim N_{r^\prime}^{1/n} = 1$ implies that $\rho(a) \le r^\prime$ for all $r^\prime > r$, showing that $\widehat{\Fil^\circ_r} R \subset R_{\rho \le r}$ as desired.

Now we need to show that $R_{\rho \le r} \subset \widehat{\Fil^\circ_r} R$. Let $a \in R$ be an element of $R$ such that $\rho(a) \le r$, setting $r_n = |a^n|$ we conclude that $\rho(a) = \inf r_n^{1/n}$ by Fekete's Lemma (cf. \cite[Section 1.3.2]{BGR}). Thus, if we show that $a \in R^{\circ, r_n^{1/n}}$ for all $n \in \ZZ_{> 0}$ it would follow that $a \in \widehat{\Fil^\circ_r} R$. For each integer $m$ which satisfies $0 \le m < n$, consider the following subsets of $\RR_{\ge 0}$
\begin{equation*}
	B_m := \{ r_n^{-\frac{m}{n} - k} |a^{m + nk}| \text{ for } k \in \ZZ_{> 0} \}
\end{equation*}
By the inequality $r_n^{-\frac{m}{n} - k} |a^{m+nk}| \le (r_n^{-\frac{m}{n}})|a^m| (r_n^{-k}) |a^{nk}|$ it follows that every subset $B_m \subset \RR_{\ge 0}$ is a bounded subset. Hence, it follows that the set $\cup B_m = \{r_n^{-k/n} |a^k|\}_{k \in \ZZ_{> 0}}$ is bounded as there are only finitely many $B_m$'s, proving that $a \in R^{\circ, r_n^{1/n}}$.
\end{proof}

\begin{rem} Let $R$ be a Banach ring and $|-|_1$ and $|-|_2$ be equivalent norms on $R$. Then, we know that the spectral radius of $|-|_1$ and the spectral radius of $|-|_2$ are the same, showing that the construction of $\widehat{\Fil^\circ_r} R$ only depends on the norm up to equivalence.
\end{rem}

The following example shows that its essential that we complete $\Fil^\circ R$ in order to get the equality $R_{\rho \le r} = \widehat{\Fil^\circ_r} R$.

\begin{example}\label{example_powerbdd_stric_sp_radius} We will construct a Banach algebra $R$ and an element $t \in R$ such that $\rho(t) \le 1$ but $\{t^n\}_{n \in \ZZ_{\ge 0}} \subset R$ is not a bounded subset. Define a function $|-|: \CC_p[T] \rightarrow \RR_{\ge 0}$ by the formula
\begin{equation*}
	|\sum a_i T^i| = \max \{ |a_i| (i + 1)\}
\end{equation*}
It is clear that this defines a non-archimedean norm on $\CC_p[T]$, and we define $R$ to be the completion of $\CC_p[T]$ with respect to $|-|$. Its easy to see that $T \in R$ is not power-bounded, but we claim that $\rho(T) = 1$. Indeed, we have that $\rho(T) = \lim (n + 1)^{1/n} = 1$, showing that the inclusion $R^\circ \subset R_{\rho \le 1}$ is strict. 
\end{example}

We finish this section by showing that various notions of uniformity agree. Recall that in the original definition of perfectoid Banach algebra \cite[Definition 5.1]{scholze2012perfectoid} we find the condition that $R^\circ \subset R$ is open and bounded, we will show that this is equivalent to requiring that $R$ is uniform in our sense. Furthermore, we will show that completing a Banach algebra with respect to the topology $\{\varpi^n R^\circ \}$ is equivalent to uniformizing.

\begin{corollary}\label{equiv_topologies_spec_rad} Let $R$ be a Banach ring. Then, the following systems of open neighborhoods around zero are cofinal
\begin{equation*}
	\{R_{\rho \le r} \}_{r \in \RR_{\ge 0}} \qquad \{\varpi^n R_{\rho \le 1} \}_{n \in \ZZ_{\ge 0}} \qquad \{\varpi^n R^\circ \}_{n \in \ZZ_{\ge 0}}
\end{equation*}
In other words, they induce equivalent topologies on $R$.
\end{corollary}

\begin{proof} Since $\rho$ is a non-archimedean semi-norm on $R$ it follows that every subset in $\{R_{\rho \le r}, \varpi^n R_{\rho \le 1} \}$ is an open subset of $R$. We claim that $R^\circ \subset R$ is an open subset; indeed, since $R_{\le 1} \subset R$ is open it suffices to show that $f R_{\le 1} \subset R$ is an open subset and that $f R_{\le 1} \subset R^\circ$ for all $f \in R^\circ$. But both of this assertions are clear.

It follows from Proposition \ref{equiv_topologies_seminormed} that $\{R_{\rho \le r} \}_{r \in \RR_{\ge 0}}$ and $\{\varpi^n R_{\rho \le 1} \}_{n \in \ZZ_{\ge 0}}$ are cofinal systems of neighborhoods, thus it suffices to show that $ \{\varpi^n R_{\rho \le 1} \}_{n \in \ZZ_{\ge 0}}$ and $\{\varpi^n R^\circ \}_{n \in \ZZ_{\ge 0}}$ are cofinal. From Proposition \ref{spec_radius_power_bdd} we learn that $R^\circ \subset R_{\rho \le 1}$, which in turn implies that $\varpi^n R^\circ \subset \varpi^n R_{\rho \le 1}$ for all $n \in \ZZ_{\ge 0}$. Conversely, if $a \in \varpi R_{\rho \le 1} \subset R_{\rho \le |\varpi|}$ then $a \in R^\circ$ again by Proposition \ref{spec_radius_power_bdd}, which implies that $\varpi^n R_{\rho \le 1} \subset \varpi^{n-1} R^\circ$ as desired.
\end{proof}

\begin{corollary}\label{equiv_uniform_power_bdd} Let $R$ be a Banach ring. Then, the norm on $R$ equivalent to the spectral radius norm if and only if $R^\circ \subset R$ is bounded. In other words, $R$ is uniform up to isomorphism.
\end{corollary}

\begin{proof} If the norm of $R$ is equivalent to the spectral radius then $R^\circ \subset R_{\rho \le 1} \subset \frac{1}{\varpi^{n}} R_{\le 1}$ showing that $R^\circ \subset R$ is bounded. For the converse, by Proposition \ref{banach_to_topological_ff} and Lemma \ref{equiv_topologies_seminormed} we know that $R$ having a norm equivalent to the spectral radius is equivalent to the requirement that $\{\varpi^n R_{\rho \le 1}\}_{n \in \ZZ_{\ge 0}}$ is a fundamental system of open neighborhoods around $0 \in R$. Hence, it suffices to show that $\{\varpi^n R^\circ \}_{n \in \ZZ_{\ge 0}}$ and $\{\varpi^n R_{\le 1} \}_{n \in \ZZ_{\ge 0}}$ are cofinal systems of neighborhoods, since by Corollary \ref{equiv_topologies_spec_rad} we already know that $\{\varpi^n R_{\rho \le 1}\}_{n \in \ZZ_{\ge 0}}$ and $\{\varpi^n R^\circ \}_{n \in \ZZ_{\ge 0}}$ are cofinal. Clearly we have that $R_{\le 1} \subset R^\circ$ which implies that $\varpi^n R_{\le 1} \subset \varpi^n R^\circ$ for all $n \in \ZZ_{\ge 0}$. For the converse, since $R^\circ$ is bounded it follows that there exists a $N \gg 0$ such that $\varpi^N R^\circ \subset R_{\le 1}$ which implies that $\varpi^{N + n} R^\circ \subset \varpi^n R_{\le 1}$.
\end{proof}

\subsection{The dictionary}

The main purpose of this section is to show to what extent we can pass between Banach $K$-algebras to $\varpi$-complete $\varpi$-torsion free $K_{\le 1}$-algebras without losing any information. We begin constructing the functors relating $\CAlg_{K_{\le 1}}^{\wedge \tf}$ with the category $\Ban_K^{\contr}$ of Banach $K$-algebras with contractive maps (cf. Definition \ref{defn_banach_algebras}).

\begin{const}\label{from_tf-comp_to_banach} To construct the functor
\begin{equation*}
	(-) \Big[ \frac{1}{\varpi} \Big]: \CAlg_{K_{\le 1}}^{\wedge \tf} \longrightarrow \Ban_K^{\contr} \qquad A \mapsto A \Big[\frac{1}{\varpi} \Big]
\end{equation*}
it suffices to show that the canonical functor $(-)[1/\varpi]: \CAlg_{K_{\le 1}}^{\wedge \tf} \rightarrow \CAlg_K$ admits a lift to $\Ban_K^{\contr}$. Let $A$ be an object of $\CAlg_{K_{\le 1}}^{\wedge \tf}$, it is clear that $A[\frac{1}{\varpi}]$ will be an (abstract) $K$-algebra. We need to show that we can endow $A[\frac{1}{\varpi}]$ with a norm making $A[\frac{1}{\varpi}]$ into a Banach ring, and such that for all morphisms $A \rightarrow C$ in $\CAlg_{K_{\le 1}}^{\wedge \tf}$ the induced map $A[\frac{1}{\varpi}] \rightarrow C[\frac{1}{\varpi}]$ is contractive with respect to the given norm.

Define a function
\begin{equation*}
	|-|: A \Big[\frac{1}{\varpi} \Big] \longrightarrow \RR_{\ge 0} \qquad \qquad |a| = \inf \{|k|_K \text{ such that } a \in k A \text{ over all } k \in K \} 
\end{equation*}
First, notice that it is clear that the norm induced on $K_{\le 1} [\frac{1}{\varpi}]$ is the same norm as the norm on $K$. We need to show that $|-|$ is a (non-archimedean) norm on $A[\frac{1}{\varpi}]$. To check that $|a| = 0$ implies that $a = 0$ recall that $\varpi \in K$ is a topological nilpotent unit, so this implies that $a \in \bigcap_{n \in \ZZ} (\varpi^n A)$, but since $A$ is classically $\varpi$-complete it follows that $a = 0$. It is also clear that $|a| = |-a|$ since $a \in kA$ if and only if $-a \in kA$. For the non-archimedean triangle inequality and the submultiplicativity of $|-|$ fix $f \in k_f A$ and $g \in k_g A$. Then, $f + g \in \max(k_f, k_g) A$ where the maximum is taken with respect to the norms of $k_f$ and $k_g$ -- this implies that $|f + g| \le \max(|f|, |g|)$. Similarly, we have that $fg \in k_f k_g A$ which implies that $|fg| \le |f||g|$. We have shown that $|-|$ defines a non-archimedean norm on $A[\frac{1}{\varpi}]$.

Next, we show that $A[\frac{1}{\varpi}]$ is Banach with respect to the norm $|-|: A[\frac{1}{\varpi}] \rightarrow \RR_{\ge 0}$ described above. For this, it suffices to show that the following systems of open neighborhoods around zero are cofinal (by \ref{equiv_topologies_seminormed} and the classical $\varpi$-completeness of $A$): the system of neighborhoods $\{\varpi^n A\}$ and the system of neighborhoods $\{\varpi^n A[\frac{1}{\varpi}]_{\le 1}\}$. We clearly have that $\varpi^n A \subset \varpi^n A[\frac{1}{\varpi}]_{\le 1}$, and by construction we have that $\varpi^{n+1} A[\frac{1}{\varpi}]_{\le 1} \subset \varpi^n A$ -- showing that $A[\frac{1}{\varpi}]$ is Banach with respect to the norm constructed above. Finally, we remark that it is clear by construction that for any morphism $A \rightarrow C$ in $\CAlg_{K_{\le 1}}^{\wedge \tf}$ the induced map $A[\frac{1}{\varpi}] \rightarrow C[\frac{1}{\varpi}]$ is contractive with respect to the given norm.
\end{const}

\begin{const}\label{from_banach_to_tf-comp} To construct the functor
\begin{equation*}
	(-)_{\le 1}: \Ban_K^{\contr} \longrightarrow \CAlg_{K_{\le 1}}^{\wedge \tf} \qquad R \mapsto R_{\le 1}
\end{equation*}
it suffices to show that the canonical functor $(-)_{\le 1}: \Ban_K^{\contr} \rightarrow \CAlg_{K_{\le 1}}$ admits a lift to $\CAlg_{K_{\le 1}}^{\wedge \tf}$. We need to check that $R_{\le 1}$ is $\varpi$-complete and $\varpi$-torsion free. It is clear by construction that $R_{\le 1} \subset R$ is $\varpi$-torsion free, and Corollary \ref{banach_is_pi_comp} shows that $R_{\le 1}$ is classically $\varpi$-complete.
\end{const}

The next few results show that if we restrict ourselves to the essential image of $H^0 j_*: \CAlg_{K_{\le 1}}^{\wedge a \tf} \hookrightarrow \CAlg_{K_{\le 1}}^{\wedge \tf}$ the functors we just defined are close to being an equivalence of categories; and if the value group of $K$ is dense, then it is actually an equivalence of categories.

\begin{lemma} The essential image of the functor $(-)_{\le 1}: \Ban_K^{\contr} \rightarrow \CAlg_{K_{\le 1}}^{\wedge \tf}$ is contained in the essential image of $H^0 j_*: \CAlg_{K_{\le 1}}^{\wedge a \tf} \hookrightarrow \CAlg_{K_{\le 1}}^{\wedge \tf}$.
\end{lemma}

\begin{proof} Since $H^0 j_*$ is fully faithful it suffices to show that for any Banach $K$-algebra $A$ the associated $K_{\le 1}$-algebra $A$ satisfies $A_{\le 1} \simeq (A_{\le 1})_*$, or in other words we need to show that $a \in (A_{\le 1})_* \subset A$ satisfies $|a| \le 1$. Indeed, recall that since $A_{\le 1}$ is $\varpi$-torsion free we can identify $(A_{\le 1})_* \subset A$ with the elements $a \in A$ such that $\varepsilon a \in A_{\le 1}$ for all $\varepsilon \in (\varpi)_{\perfd}$. From the fact that $K \rightarrow A$ is an isometry and \cite[Section 1.2.2]{BGR} it follows that $|\varepsilon a| = |\varepsilon| |a|$, which implies that $|a| \le 1/|\varepsilon|$ for all $\varepsilon \in (\varpi)_{\perfd}$. Without loss of generality we may assume that $\varpi \in K$ admits compatible $p$-power roots; since $|\varpi^{1/p^n}| \rightarrow 1$ as $n \rightarrow \infty$ it follows that $|a| \le 1$.
\end{proof}

The following result can be found in \cite[Section 2.3]{andrelemme}.

\begin{prop}\label{equiv_almost_tf_banach} Define the functor $(-)[\frac{1}{\varpi}]: \CAlg_{K_{\le 1}}^{\wedge a \tf} \rightarrow \Ban_K^{\contr}$ as the composition
\begin{cd}
	\CAlg_{K_{\le 1}}^{\wedge a \tf} \ar[rr, hook, "H^0 j_*"] && \CAlg_{K_{\le 1}}^{\wedge \tf} \ar[rr, "(-) \lbrack \frac{1}{\varpi} \rbrack "] && \Ban_{K}^{\contr}
\end{cd}
This functor is an equivalence of categories with inverse given by  $(-)_{\le 1}: \Ban_K^{\contr} \rightarrow \CAlg_{K_{\le 1}}^{\wedge a \tf}$.
\end{prop}

\begin{proof} Let $A$ be an element of $\CAlg_{K_{\le 1}}^{\wedge a \tf}$ regarded as an object of $\CAlg_{K_{\le 1}}^{\wedge \tf}$ via the functor $H^0 j_*$. We claim that $A = (A[\frac{1}{\varpi}])_{\le 1}$. Indeed, by construction we have an inclusion $A \subset (A[\frac{1}{\varpi}])_{\le 1}$, so it suffices to show that if $a \in A[\frac{1}{\varpi}]$ satisfies $|a| \le 1$ then $a \in A$. Recall that by the definition of the norm on $A[1/\varpi]$ it follows that $|a| \le 1$ if and only if $a \in \frac{1}{\varepsilon} A$ for all $\varepsilon$ in the maximal ideal $\frakm = (\varpi)_{\perfd}$ of $K_{\le 1}$, and since $A$ satisfies $A \simeq A_*$ it follows that $a \in A$. 
	
On the other hand, its clear that $(A_{\le 1})[\frac{1}{\varpi}]$ and $A$ are isomorphic as abstract rings, we claim that the identity map $(A_{\le 1})[\frac{1}{\varpi}] \rightarrow A$ is contractive. Indeed, if $a \in kA_{\le 1}$ then there exists an $\alpha \in A_{\le 1}$ such that $|a| = |k| |\alpha|$, which in turn implies that $|a| \le |k|$. Hence, it remains to show that the map $(A_{\le 1})[\frac{1}{\varpi}] \rightarrow A$ is an isometry. First, we need the following auxiliary result result: the value group $|K^{\times}| \subset \RR_{\ge 0}$ is dense. Indeed, since $K$ is a perfectoid field we may assume that the topological nilpotent unit $\varpi \in K^{\times}$ admits compatible $p$-power root, then since $\ZZ[\frac{1}{p}] \subset \RR$ is a dense subset, it follows that $|\varpi|^{\ZZ[1/p]} \subset \RR_{\ge 0}$ is dense, proving that $|K^{\times}| \subset \RR_{\ge 0}$ is dense.

Set $S := (A_{\le 1})[\frac{1}{\varpi}]$, it remains to show that the map $S \rightarrow A$, which is the identity at the level of rings, induces an equality between the norms $|-|_A$ and $|-|_S$. Fix some $a \in A$, and pick a sequence of $\{k_i\} \subset K^{\times}$, with decreasing norms, such that $|k_i| \rightarrow |a|_A$. We claim that $|a|_S \le |k_i|$ for all $\{k_i\}$, we have that $|a/k_i| = |a|/|k_i|$, which implies that $a/k_i \in A_{\le 1}$ and so $a \in k_i A_{\le 1}$, proving the claim. Hence, we have that $|a|_S \le |a|_A$; and since the counit of the adjunction $S \rightarrow A$ is a contractive map, it follows that $|a|_S = |a|_A$ for all $a \in A$.
\end{proof}

\begin{corollary}\label{tensor_dic_int_to_ban} Let $A \leftarrow B \rightarrow C$ be a pair of morphisms in $\CAlg_{K_{\le 1}}^{\wedge a \tf}$. Then, we have a canonical isomorphism in $\Ban_K^{\contr}$.
\begin{equation*}
	(A \otimes_B C)^{\tf, \wedge, a}[1/\varpi] \simeq A[1/\varpi] \cotimes_{B[1/\varpi]} C[1/\varpi]
\end{equation*}
\end{corollary}

\begin{proof} Recall that the pushout of $A \leftarrow B \rightarrow C$ in $\CAlg_{K_{\le 1}}^{\wedge a \tf}$ can be computed as $(A \otimes_B C)^{\tf, \wedge, a}$; and that the pushout of $A[1/\varpi] \leftarrow B[1/\varpi] \rightarrow C[1/\varpi]$ in $\Ban_K^{\contr}$ can be identified with the completed tensor product. Since the functor $(-)[1/\varpi]: \CAlg_{K_{\le 1}}^{\wedge a \tf} \rightarrow \Ban_K^{\contr}$ is an equivalence it preserves pushouts, proving the result.
\end{proof}

\begin{corollary}\label{tensor_dic_ban_to_int} Let $A \leftarrow B \rightarrow C$ be a pair of morphisms in $\Ban_K^{\contr}$. Then, we have a canonical isomorphism in $\Ban_K^{\contr}$
\begin{equation*}
	(A_{\le 1} \otimes_{B_{\le 1}} C_{\le 1})^{\tf, \wedge, a}[1/\varpi] \rightarrow A \cotimes_B C
\end{equation*}
\end{corollary}

\begin{proof} Follows directly from the equivalence of categories \ref{equiv_almost_tf_banach}, the description of $(A_{\le 1} \otimes_{B_{\le 1}} C_{\le 1})^{\tf, \wedge, a}$ as the pushout of $A_{\le 1} \leftarrow B_{\le 1} \rightarrow C_{\le 1}$ in $\CAlg_{K_{\le 1}}^{\wedge a \tf}$, and the description of $A \cotimes_B C$ as the pushout of $A \leftarrow B \rightarrow C$ in $\Ban_K^{\contr}$.
\end{proof}

\begin{corollary}\label{non_zero_comp_tensor_ban} Let $A \leftarrow K \rightarrow B$ in $\Ban_K^{\contr}$. Then, the completed tensor product $A \cotimes_K C$ is non-zero.
\end{corollary}

\begin{proof} Since $A_{\le 1}$ and $B_{\le 1}$ are torsion free $K_{\le 1}$-algebras it follows that they are flat $K_{\le 1}$-algebras (\cite[Lemma 0539]{stacks-project}). Furthermore, as $A$, $B$, $A_{\le 1}/\varpi$ and $B_{\le 1}/\varpi$ are non-zero, we can conclude that $A_{\le 1}$ and $B_{\le 1}$ are faithfully flat $K_{\le 1}$-algebras. Therefore, as being faithfully flat is stable under base-change and composition we conclude that $A_{\le 1} \otimes_{K_{\le 1}} B_{\le 1}$ is faithfully flat over $K_{\le 1}$, in particular it is torsion free and $(A_{\le 1} \otimes_{K_{\le 1}} B_{\le 1})/\varpi$ is non-zero.

Hence, we can conclude that $A_{\le 1} \cotimes_{K_{\le 1}} B_{\le 1}$ is non-zero as it admits a surjective map to $(A_{\le 1} \otimes_{K_{\le 1}} B_{\le 1})/\varpi$, and torsion free by \ref{completion_tf}. This in turn implies that $(A_{\le 1} \cotimes_{K_{\le 1}} B_{\le 1})_*$ (\ref{almost_elements}) is non-zero and torsion free, which implies that $A \cotimes_B C$ is non-zero by \ref{tensor_dic_ban_to_int}.
\end{proof}

\begin{const}\label{const_funct_uBan_tic_dic} Let $A$ be an object of $\CAlg_{K_{\le 1}}^{\wedge \tic} \subset \CAlg_{K_{\le 1}}^{\wedge \text{tf}}$, we claim that the norm on $A[\frac{1}{\varpi}]$ (from Construction \ref{from_tf-comp_to_banach}) is power-multiplicative, and so $A[\frac{1}{\varpi}] \in \uBan_K \subset \Ban_K^{\contr}$. Recall that by the submultiplicativity of non-archimedean norms we have the inequality $|a^2| \le |a|^2$ for all $a \in A[\frac{1}{\varpi}]$, thus we need to show that $|a^2| \ge |a|^2$. Let $k \in K$ be an element such that $a^2/k \in A = A[\frac{1}{\varpi}]_{\le 1}$. By the density of the value group $|K^{\times}| \subset \RR_{\ge 0}$ we may pick a sequence $\{k_1, \dots, k_i, \dots \} \subset K^{\times}$ such that $\{|k_i|\}$ is decreasing and $|k_i| \rightarrow \sqrt{|k|}$. We claim that $a/k_i \in A$ for all $k_i$. Indeed, since by assumption we have that $|k_i|^2 \ge |k|$ it follows that $a^2/k_i^2 \in A$, which in turn implies that $\{a^n/k_i^n\} \subset A[\frac{1}{\varpi}]$ is a bounded set. Hence, by the assumption that $A$ is totally integrally closed we can conclude that $a/k_i \in A$ for all $k_i$. We have shown that for any $k \in K$ such that $|a^2| \le |k|$ we have that $|a| \le \sqrt{|k|}$, proving that $|a|^2 \le |a^2|$. Hence, we get a functor
\begin{equation*}
	(-) \Big[ \frac{1}{\varpi} \Big]: \CAlg_{K_{\le 1}}^{\wedge \tic} \longrightarrow \uBan_K \qquad A \mapsto A \Big[\frac{1}{\varpi} \Big]
\end{equation*}
On the other hand, given an object $R \subset \uBan_K \subset \Ban_K^{\contr}$ we need to show that $R_{\le 1} \subset R$ is total integrally closed. Let $f$ be an element of $R$ such that $\{f^{\ZZ_{\ge 0}} \} \subset \frac{1}{\varpi^n}R_{\le 1}$, then $f \in R^\circ$, but since $R^\circ = R_{\le 1}$ it follows that $R_{\le 1} \subset R$ is total integrally closed. Thus, using that maps between uniform Banach algebras are always contractive we get a functor
\begin{equation*}
	(-)_{\le 1}: \uBan_K \rightarrow \CAlg_{K_{\le 1}}^{\wedge \tic}
\end{equation*}
\end{const}

\begin{prop}\label{uBan_tic_dict} The following pair of functors determine an equivalence of categories
\begin{cd}
	(-)_{\le 1}: \uBan_K \ar[r, shift left = 1.5] & \CAlg_{K_{\le 1}}^{\wedge \tic} \ar[l, shift left = 1.5] : (-)\Big[ \frac{1}{\varpi} \Big]
\end{cd}
Between the category of uniform Banach $K$-algebras and the category of $\varpi$-complete $\varpi$-torsion free total integrally closed $K_{\le 1}$-algebras.
\end{prop}

\begin{proof} Follows directly from the equivalence of categories of \ref{equiv_almost_tf_banach} and Construction \ref{const_funct_uBan_tic_dic}.
\end{proof}

\newpage

\section{Perfectoid Banach Algebras}\label{sect_perfd_ban_alg}

\subsection{Almost perfectoid algebras}

Throughout this section fix an integral perfectoid ring $B$ which is $\varpi$-complete with respect to some $\varpi \in B$ where $\varpi^p$ divides $p$, and such that $B$ is $\varpi$-torsion free. 

\begin{defn}\label{defn_various_perfectoid} We will be concerned with the following three flavors of integral perfectoid $B$-algebras.
\begin{enumerate}[(1)]
	\item The full subcategory of $\CAlg_B^{\wedge}$ spanned by all $\varpi$-complete integral perfectoid $B$-algebras. We denote this category by $\Perfd_B^{\Prism}$.
	\item The full subcategory of $\CAlg_B^{\wedge a}$ spanned by the image of the functor $(-)^a: \Perfd_B^{\Prism} \rightarrow \CAlg_B^{\wedge a}$ (cf. \ref{abelian_j*_sym_monoidal}). We denote this category by $\Perfd_B^{\Prism a}$.
	\item The full subcategory of $\CAlg_B^{\wedge \tic}$ spanned by all the $\varpi$-complete integral perfectoid $B$-algebras $A$ which are total integrally closed with respect to $A[\frac{1}{\varpi}]$. We denote this category by $\Perfd_B^{\Prism \tic}$.
\end{enumerate}
Notice how we are requiring all the perfectoid $B$-algebras to be $\varpi$-complete, which is not automatic from the definition of integral perfectoid ring.
\end{defn}

\begin{prop}\label{equiv_perfd_tic_almost} The functor
\begin{equation*}
	(-)^a: \Perfd_B^{\Prism \tic} \longrightarrow \Perfd_B^{\Prism a}
\end{equation*}
is an equivalence of categories with inverse given by $H^0 j_*: \Perfd_B^{\Prism a} \rightarrow \Perfd_B^{\Prism \tic}$.
\end{prop}

\begin{proof} We need show that the fully faithful functor $H^0 j_*: \Perfd_B^{\Prism a} \rightarrow \CAlg_B^{\wedge}$ has its image contained in $\Perfd_B^{\Prism \tic}$; equivalently we need to show that the functor $(-)_*: \Perfd_B^{\Prism} \rightarrow \CAlg_B^{\wedge}$, which we introduced in \ref{abelian_j*_sym_monoidal} and identified with $\uHom_B ((\varpi)_{\perfd}, -)$ in \ref{abelian_almost_unit_internal_hom}, has its image contained in $\Perfd_B^{\Prism \tic}$. Let $A$ be an object of $\Perfd_B^{\Prism}$; recall from \ref{perfectoid_torsion} that $A$ satisfies $A[\varpi^\infty] = A[\varpi^{1/p^\infty}]$ and denote by $\overline{A}$ the $\varpi$-complete $\varpi$-torsion free integral perfectoid $B$-algebra $A/A[\varpi^\infty]$ (cf. \ref{pi_torsionfree_perfectoid}), we have that the morphism $A \rightarrow \overline{A}$ induces an isomorphism $A_* \rightarrow \overline{A}_*$ since $A[\varpi^{1/p^\infty}]^a = 0$ by definition. Moreover, as $\overline{A}$ is a $\varpi$-complete $\varpi$-torsion free integral perfectoid $B$-algebra it follows from \ref{tic_preserves_perfectoid} that $\overline{A}^{\tic}$ is again an $\varpi$-complete $\varpi$-torsion free integral perfectoid $B$-algebra. And since $\overline{A}$ is $p$-integrally closed (\ref{perfectoid_p_int_closed}) it follows from \ref{ic_almost_algebras} and \ref{tic_almost_algebras} that the morphism $\overline{A} \rightarrow \overline{A}^{\tic}$ induces an isomorphism $\overline{A}_* \rightarrow (\overline{A}^{\tic})_*$. Finally, from \ref{almost_funct_tic_algebras} we learn that the canonical map $\overline{A}^{\tic} \rightarrow (\overline{A}^{\tic})_*$ is an isomorphism. To summarize we have shown that the canonical map $A_* \rightarrow \overline{A}^{\tic} \simeq (\overline{A}^{\tic})_*$ is an isomorphism, showing that we get a fully faithful functor $H^0 j_*: \Perfd_B^{\Prism a} \rightarrow \Perfd_B^{\Prism \tic}$.
	
To conclude, the above argument shows that the the category $\Perfd_B^{\Prism a} \subset \CAlg_B^{\wedge a}$ can actually be realized as a full-subcategory of $\CAlg_B^{\wedge a \tic}$. And recall from \ref{almost_funct_tic_algebras} that we get have an equivalence $H^0 j_*: \CAlg_B^{\wedge a \tic} \rightleftarrows \CAlg_B^{\wedge \tic}: (-)^a$, proving the desired equivalence between $\Perfd_B^{\Prism a}$ and $\Perfd_B^{\Prism \tic}$.
\end{proof}

\begin{corollary}\label{tensor_perfd_tic} Let $A \leftarrow D \rightarrow C$ be a pair of morphisms in $\Perfd_B^{\Prism \tic}$. Then, its pushout computed in $\CAlg_B^{\wedge \tic}$ can be identified with $(A \cotimes_D C)_*$, where $- \cotimes_C -$ denotes the (classically) $\varpi$-complete tensor product. Moreover $(A \cotimes_D C)_*$ is an object of $\Perfd_{B}^{\Prism \tic}$.
\end{corollary}

\begin{proof} Recall from \ref{tensor_integral_perfectoid} that the pushout of $A \leftarrow D \rightarrow C$ computed in $\CAlg_B^{\wedge}$ can be identified with the integral perfectoid $B$-algebra $A \cotimes_D C$. Moreover, we know that the functor $(-)^a: \CAlg_B^\wedge \rightarrow \CAlg_B^{\wedge a}$ is a left adjoint to $H^0 j_*$ by \ref{abelian_j*_sym_monoidal}, so the pushout of $A^a \leftarrow D^a \rightarrow C^a$ in $\CAlg_B^{\wedge a}$ can be identified with $(A \cotimes_C D)^a$, in particular we learn that $(A \cotimes_D C)^a$ is an object of $\Perfd_B^{\Prism a} \subset \CAlg_B^{\wedge a \tic}$. Since the fully faithful functor $\CAlg_B^{\wedge a \tic} \hookrightarrow \CAlg_B^{\wedge a}$ admits a left adjoint, it follows that $(A \cotimes_D C)^a$ can be identified with the pushout of $A^a \leftarrow D^a \rightarrow C^a$ in $\CAlg_B^{\wedge a \tic}$. Under the equivalence $H^0 j_*: \CAlg_B^{\wedge a \tic} \rightarrow \CAlg_B^{\wedge \tic}$ from \ref{almost_funct_tic_algebras} it follows that the pushout of $A \leftarrow B \rightarrow C$ computed in $\CAlg_B^{\wedge \tic}$ can be identified with $(A \cotimes_D C)_*$ as desired. Finally, the fact that $(A \cotimes_D C)_*$ is an object of $\Perfd_B^{\Prism \tic}$ follows from the fact that $(A \cotimes_D C)^a$ was an object of $\Perfd_B^{\Prism a}$ and the equivalence from \ref{equiv_perfd_tic_almost}.
\end{proof}

\begin{prop}\label{integral_tilt_corr} The tilting functor determines an equivalence of categories 
\begin{align*}
	\Perfd_B^{\Prism \tic} \longrightarrow \Perfd_{B^\flat}^{\Prism \tic} && A \mapsto A^\flat
\end{align*}
\end{prop}

\begin{proof} Let $A$ be an object of $\Perfd_B^{\Prism}$. Recall from \ref{completeness_tilt} that if $A$ is $\varpi$-complete then $A^\flat$ is $\varpi^\flat$-complete, which implies by \ref{prism_tilt_correspondence} that the tilting functor $\Perfd_B^{\Prism} \rightarrow \Perfd_{B^\flat}^{\Prism}$ is an equivalence of categories. Moreover, by \ref{perfectoid_torsion} we learn that if $A$ is $\varpi$-torsion free then $A^\flat$ is $\varpi^\flat$-torsion free. Thus, it remains to show that if a $\varpi$-torsion free $A$ is total integrally closed with respect to $A[\frac{1}{\varpi}]$, then $A^\flat$ is total integrally closed with respect to $A^\flat[\frac{1}{\varpi^\flat}]$; but from \ref{equiv_perfd_tic_almost} we learn that it suffices to show that if $A_* \simeq A$ then $A^\flat \simeq (A^\flat)_*$.
	
Recall that for any $A \in \Perfd_B^{\Prism}$ we have a morphism of multiplicative monoids $\sharp: A^\flat \rightarrow A$ which maps $\varpi^\flat$ to a multiplicative unit of $\varpi$, which implies that we get another morphism of multiplicative monoids $\sharp: A^\flat [\frac{1}{\varpi^\flat}] \rightarrow A[\frac{1}{\varpi}]$. Therefore, if we have an object $a \in A^\flat[\frac{1}{\varpi^\flat}]$ which satisfies $a \varepsilon^\flat \in A^\flat$ for all $\varepsilon^\flat \in (\varpi^{\flat, 1/p^\infty})$ then $\varepsilon a^\sharp \in A$ for all $\varepsilon \in (\varpi^{\sharp, 1/p^\infty}) = (\varpi)_{\perfd}$; and since $A$ satisfies $A \simeq A_*$ it follows that follows that $a^\sharp \in A$ which implies that $a \in A^\flat$ showing that $A^\flat \simeq (A^\flat)_*$ as desired.
\end{proof}

\subsection{Definition and basic properties}

\begin{defn}\label{defn_banach_perfectoid} A perfectoid field $K$ is a non-archimedean field for which exists a $\varpi \in K$ which satisfies $1 > |\varpi^p| \ge |p|$, and such that the Frobenius morphism $K_{\le 1}/\varpi^p \rightarrow K_{\le 1}/\varpi^p$ is surjective.

A perfectoid Banach $K$-algebra is a uniform Banach $K$-algebra $R$ such that the Frobenius morphism $\varphi: R_{\le 1}/\varpi^p \rightarrow R_{\le 1}/\varpi^p$ is surjective. We will denote the full-subcategory of $\uBan_K$ spanned by the perfectoid Banach $K$-algebras by $\Perfd_K^{\Ban}$.
\end{defn}

\begin{lemma}\label{perfectoid_field_is_int_perfectoid} If $K$ is a Banach perfectoid field, then $K_{\le 1}$ is a integral perfectoid algebra which is $\varpi$-complete $\varpi$-torsion free and total integrally closed with respect to $K_{\le 1} \subset K$.
\end{lemma}

\begin{proof} Since non-archimedean fields are uniform by definition, it follows from Proposition \ref{uBan_tic_dict} that $K_{\le 1}$ is $\varpi$-complete $\varpi$-torsion free and totally integrally closed with respect to $K_{\le 1} \subset K$. To show that $K_{\le 1}$ is an integral perfectoid algebra, it suffices to show that the induced surjective morphism $K_{\le 1}/\varpi \rightarrow K_{\le 1}/\varpi^p$ given by $a \mapsto a^p$, is injective (cf. \cite[Lemma 3.10]{BMS1}), but this follows from Lemma \ref{p_int_closed_frob}.
\end{proof}

For the rest of the section fix a Banach perfectoid field $K$, and an element $\varpi \in K$ which satisfies $1 > |\varpi^p| \ge |p|$. We now introduce the various flavors of integral perfectoid algebras that will be relevant for us.

\begin{lemma} If $R$ is a perfectoid Banach $K$-algebra, then $R_{\le 1}$ is an object of $\Perfd_{K_{\le 1}}^{\Prism \tic}$.
\end{lemma}

\begin{proof} Since $R$ is uniform by definition, it follows from \ref{uBan_tic_dict} that $R_{\le 1} \in \CAlg_{K_{\le 1}}^{\wedge \tic}$, so it remains to show that $R_{\le 1}$ is an integral perfectoid algebra. By hypothesis we know that Frobenius $R/\varpi^p \rightarrow R/\varpi^p$ is surjective, which means that the induced map $R/\varpi \rightarrow R/\varpi^p$ given by $a \mapsto a^p$ is surjective. Hence, it remains to show that $R/\varpi \rightarrow R/\varpi^p$ is injective (cf. \cite[Lemma 3.10]{BMS1}), but this follows from Lemma \ref{p_int_closed_frob}.
\end{proof}

\begin{prop}\label{equiv_perfd_ban_tic} The functor
\begin{equation*}
	(-)_{\le 1}: \Perfd_K^{\Ban} \longrightarrow \Perfd_{K_{\le 1}}^{\Prism \tic}
\end{equation*}
determines an equivalence of categories with inverse $(-)[\frac{1}{\varpi}]$.
\end{prop}

\begin{proof} From the equivalence of categories $(-)_{\le 1}: \uBan_K \leftrightarrows \CAlg_{K_{\le 1}}^{\wedge \tic}: (-)[\frac{1}{\varpi}]$ established in \ref{uBan_tic_dict}, it remains to show that if $A \in \Perfd_{K_{\le 1}}^{\Prism \tic}$ then $A[\frac{1}{\varpi}]_{\rho}$ is a perfectoid Banach $K$-algebra. From the identity $A[\frac{1}{\varpi}]_{\rho, \le 1} \simeq A$, we learn that it suffices to show that the Frobenius morphism $A/\varpi^p \rightarrow A/\varpi^p$ is surjective. From \cite[Lemma 3.10]{BMS1} we know that the $p$-power map $A/\varpi \rightarrow A/\varpi^p$ is an isomorphism, and since Frobenius can be realized as the composition $A/\varpi^p \twoheadrightarrow A/\varpi \rightarrow A/\varpi^p$ the result follows.
\end{proof}

\begin{corollary}\label{tensor_ban_perfectoid} Let $A \leftarrow B \rightarrow C$ be a pair of morphisms in $\Perfd_K^{\Ban}$. Then, $A \cotimes_B C$ is a perfectoid Banach $K$-algebra, in particular, it is endowed with the spectral radius norm.
\end{corollary}

\begin{proof} From \ref{tensor_integral_perfectoid} we learn that $A_{\le 1} \cotimes_{B_{\le 1}}  C_{\le 1}$ is an object of $\Perfd_{K_{\le 1}}^{\Prism}$, and from \ref{equiv_perfd_tic_almost} we learn that $(A_{\le 1} \cotimes_{B_{\le 1}}  C_{\le 1})^a \in \Perfd_{K_{\le 1}}^{\Prism a}$ is $\varpi$-torsion free. Hence, it follows from \ref{tensor_dic_int_to_ban} that we have an isometric isomorphism
\begin{equation*}
	(A_{\le 1} \cotimes_{B_{\le 1}}  C_{\le 1})^a [\frac{1}{\varpi}] \simeq A \cotimes_B C
\end{equation*}
proving that $A \cotimes_B C$ is a perfectoid Banach $K$-algebra by \ref{equiv_perfd_ban_tic}.
\end{proof}

\subsection{Valuation rings}

For the definitions and basic results we will follow \cite[00I8]{stacks-project}.

\begin{defn}[valuation rings] We follow \cite[00I8]{stacks-project} as our official definition of valuation rings.
\begin{enumerate}[(1)]
	\item Let $K$ be a field. Let $A, B$ be local rings contained in $K$. We say that $B$ dominates $A$ if $A \subset B$ and $\frakm_A = A \cap \frakm_B$.
	\item Let $A$ be a ring. We say that $A$ is a valuation ring if $A$ is a local domain and $A$ is maximal for all the relations of dominations among local rings contained in the fraction field of $A$.
	\item Let $A$ be a valuation ring with fraction field $K$. If $R \subset K$ is a subring of $K$, then we say $A$ is centered on $R$ if $R \subset A$.
\end{enumerate}
\end{defn}

\begin{lemma}\label{properties_value_group} Let $A$ be a valuation ring with fraction field $K$. Set $\Gamma = K^\times/A^\times$, and define $\gamma_1 \le \gamma_2$ in $\Gamma$ if and only if $\gamma_1/\gamma_2 \in \Gamma$ is in the image of the canonical map $A\setminus\{0\} \rightarrow \Gamma$. Then,
\begin{enumerate}[(1)]
	\item The pair $(\Gamma, \le)$ is a totally ordered abelian group, and we call it the value group of the valuation ring $A$.
	\item The induced map $v: A \rightarrow \Gamma \cup \{0\}$ is called the valuation associated to $A$, we sometimes abuse notation and also call the map $v: K \rightarrow \Gamma \cup \{0\}$ by the same name. Furthermore, $v: A \rightarrow \Gamma \cup \{ 0 \}$ and $v: K \rightarrow \Gamma \cup \{ 0\}$ satisfy the following conditions
	\begin{enumerate}[(2.1)]
		\item $v(a) = 1$ if and only if $a \in A^{\times}$, and $v(a) = 0$ if and only if $a = 0$.
		\item $v(ab) = v(a)v(b)$.
		\item $v(a + b) \le \max(v(a), v(b))$.
	\end{enumerate}
	In particular, the valuation on $A$ factors as $v: A \rightarrow \Gamma_{\le 1} \cup \{0\} \rightarrow \Gamma \cup \{0\}$.
	\item We have the following equalities
	\begin{equation*}
		A = \{x \in K | v(x) \le 1 \} \qquad \frakm_A = \{x \in K | v(x) < 1 \} \qquad A^\times = \{x \in K | v(x) = 1 \}
	\end{equation*}
	where we make $\Gamma \cup \{0\}$ a totally ordered commutative monoid by declaring that $0 < \gamma$ for all $\gamma \in \Gamma$.
\end{enumerate}
\end{lemma}
	
\begin{proof} Part $(1)$ follows from \cite[Tag 00ID]{stacks-project}, by writing the group structure on $\Gamma$ multiplicatively. Its customary to transform multiplicative notation on $\Gamma$ to additive notation by the rule $x \mapsto - \log(x)$, in which case $(2)$ follows from \cite[Tag 00IF]{stacks-project}. Finally, $(3)$ follows from \cite[Tag 00IG]{stacks-project}.
\end{proof}

\begin{lemma}\label{prime_ideals_valuation_rings} Let $(\Gamma \cup\{0\}, \le)$ be a totally ordered commutative monoid. An ideal of $\Gamma \cup\{0\}$ is a subset $I \subset \Gamma \cup\{0\}$, such that all elements of $I$ are $\le 1$ and $\gamma \in I$, $\gamma^{\prime} \le \gamma$ implies $\gamma^{\prime} \in I$. We say than an ideal of $\Gamma \cup \{0\}$ is prime if it satisfies the following conditions
	\begin{enumerate}[(1)]
		\item For all $\gamma \in I$, we have $\gamma < 1$.
		\item If $\gamma_1 \gamma_2 \in I$ and $\gamma_1, \gamma_2 \le 1$, then $\gamma_1 \in I$ or $\gamma_2 \in I$.
	\end{enumerate}
Let $R$ be a valuation ring, then the map $v: R \rightarrow K^{\times}/R^{\times} \cup \{0\}$ induces a bijection of ideals. Furthermore, this bijection is inclusion preserving, and maps prime ideals to prime ideals.
\end{lemma}

\begin{proof} After rewriting the group structure on $\Gamma \cup \{0\}$ multiplicatively, it follows from \cite[Tag 00IH]{stacks-project}.
\end{proof}

\begin{lemma}\label{arch_rank_one_valuations} Let $R$ be a valuation ring with fraction field $K$. The value group $\Gamma = K^{\times}/R^{\times}$ is said to be of rank one if $\Gamma \not = 1$ and for every $0 < \gamma < 1$ in $\Gamma$ and any $\gamma^{\prime} \in \Gamma$, there exists a $n \gg 0$ such that $\gamma^{n} < \gamma^{\prime}$. Then, the value group $\Gamma = K^{\times}/R^{\times}$ is of rank one if and only if $R$ has Krull dimension one.
\end{lemma}

Based on this result we will often abuse language and call valuation rings of rank one, what should be properly called valuation rings of Krull dimension one.

\begin{proof} See for example \cite[Theorem 10.7]{matsumura1989commutative} for a proof, we record it here for completeness. If the value group $\Gamma$ has rank one, by \ref{prime_ideals_valuation_rings} it suffices to show that $\Gamma \cup \{0\}$ has exactly two prime ideals, $\{0\}$ and $\frakm = \{x < 1 | x \in \Gamma \cup\{0\} \}$. The assumption that $\Gamma \not = 1$ already implies that $\{0\} \not= \frakm$, now assume that there exists a non-zero prime ideal $\frakp$ of $\Gamma \cup \{0\}$ such that there exists a $\gamma \in \frakm$ but $\gamma \not\in \frakp$. As $\frakp$ is non-zero there exists a $0 \not= \gamma^{\prime} \in \frakp$, but then the assumption that $\Gamma$ has rank one implies that there exists a $n \gg 0$ such that $\gamma^n < \gamma^{\prime}$, which implies that $\gamma \in \frakp$, contradicting our assumption. This proves that $R$ has Krull dimension one.

Conversely, if $R$ has Krull dimension one, pick $0 \not= r \in \frakm_R$ and by the primality of radical ideals in valuation rings we learn that $\sqrt{rR} = \frakm_R$. Therefore, for any $h \in \frakm_R$ we have that $0 \not= h^n \in rR$, proving that the value group $K^{\times}/R^{\times}$ is of rank one.
\end{proof}

\begin{lemma}\label{embed_value_group_reals_rk_one} Let $R$ be a rank one valuation ring with fraction field $K$. Then, there exists a order-preserving multiplicative injective map $K^{\times}/R^{\times} =: \Gamma \hookrightarrow \RR_{> 0}^{\times}$; furthermore, this extends to an order preserving multiplicative injective map $\Gamma \cup \{0\} \hookrightarrow \RR_{> 0}^{\times} \cup \{0\}$. 
	
Conversely, if $R$ is a valuation ring with fraction field $K$, if there is an order-preserving multiplicative injective map $K^{\times}/R^{\times} \hookrightarrow \RR_{> 0}^{\times}$ then $K^{\times}/R^{\times}$ is of rank one.
\end{lemma}

\begin{proof} Follows from \cite[Theorem 10.6]{matsumura1989commutative}.
\end{proof}

\begin{lemma}\label{recog_val_rings} Let $K$ be a field, and $v: K \rightarrow \Gamma \cup \{0\}$ a map satisfying: $v(a) = 0$ if and only if $a = 0$, $v(ab) = v(a)v(b)$ and $v(a + b) \le \max(v(a), v(b))$. Then, $R = \{x \in K| v(x) \le 1 \}$ is a valuation ring.
\end{lemma}

\begin{proof} Since $R \subset K$, and for all $x \in K$ we have that $v(x) \le 1$ or $v(x^{-1}) \le 1$ or both, it follows that $R$ is a valuation ring by \cite[Tag 052K]{stacks-project}.
\end{proof}

\begin{example}\label{non_arch_field_val_rank_one} Let $K$ be a non-archimedean field with (necessarily) multiplicative norm $|-|: K \rightarrow \RR_{\ge 0}^{\times}$, then $K_{\le 1}$ is a rank one valuation ring. Indeed, by Lemma \ref{recog_val_rings} we know that $K_{\le 1}$ is a valuation ring, thus to show that $K_{\le 1}$ is of rank one it suffices to show that we have an order preserving multiplicative injective map $K^{\times}/K_{\le 1}^{\times} \hookrightarrow \RR_{> 0}^{\times}$. We claim this embedding is given by the non-archimedean norm $|-|$, indeed by construction we have that $x \in K^{\times}$ satisfies $|x| = 1$ if and only if $x \in K_{\le 1}^{\times}$, thus the map $|-|: K^{\times} \rightarrow \RR_{> 0}^{\times}$ factors as
\begin{equation*}
	|-|: K^{\times} \rightarrow K^{\times}/K_{\le 1}^{\times} \hookrightarrow \RR_{> 0}^{\times}
\end{equation*}
and the fact that the injective map $K^{\times}/K_{\le 1}^{\times} \hookrightarrow \RR_{> 0}^{\times}$ is order preserving follows from the construction.
\end{example}

\begin{lemma}\label{equiv_extensions_rk_one} Let $f: V \rightarrow W$ be a morphism of rank one valuation rings, and $\frakm_V \subset V$, $\frakm_W \subset W$ their respective maximal ideals. Then, the following are equivalent
\begin{enumerate}[(1)]
	\item The morphism $f: V \rightarrow W$ is faithfully flat.
	\item The induced map $|\Spec(W)| \rightarrow |\Spec(V)|$ of underlying topological spaces is the identity map of Sierpinski spaces.
	\item The morphism $f: V \rightarrow W$ is injective and $f(\frakm_V) \subset \frakm_W$.
\end{enumerate}
If a morphism of rank one valuations $f: V \rightarrow W$ satisfies this equivalent conditions we say that $f$ is an extension of rank one valuation rings.
\end{lemma}

\begin{proof} $(1) \Rightarrow (2)$. Recall that a faithfully flat morphism $\Spec(R) \rightarrow \Spec(S)$ is surjective on its underlying topological spaces \cite[Tag 00HQ]{stacks-project}, and that for any rank one valuation ring $V$ we have the identification $|\Spec(V) | = \{ \eta, s \}$, where the latter is the Sierpinski space. Hence, we conclude that $|\Spec(W)| \rightarrow |\Spec(V)|$ is the identity map as it is a continuous surjective maps of Sierpinski spaces.

$(2) \Rightarrow (3)$. Let $\pi \in \frakm_V$ be a non-zero element, if $f(\pi) \in W$ were zero then then the map $\Spec(W) \rightarrow \Spec(V)$ will factor as $\Spec(W) \rightarrow \Spec(V/\pi) \rightarrow \Spec(V)$ showing that the desired map cannot be surjective. Similarly, if $f(\pi) \in W$ is not an element of the maximal ideal $\frakm_W \subset W$, then the map $\Spec(W) \rightarrow \Spec(V)$ factors as $\Spec(W) \rightarrow \Spec(V[\pi^{-1}]) \rightarrow \Spec(V)$, which implies that the desired map cannot be surjective since $V[\pi^{-1}] = \Frac(V)$.

$(3) \Rightarrow (1)$. From the hypothesis it follows that $f^{-1} (\frakm_W) = \frakm_V$ and $f^{-1} (0) = 0$, showing that the induced morphism $\Spec(W) \rightarrow \Spec(V)$ is surjective, so it suffices to show that $V \rightarrow W$ is flat by \cite[Tag 00HQ]{stacks-project}. But since $W$ is a domain and the map $f: V \rightarrow W$ is injective, it follows that $W$ is torsion-free as a $V$-module, proving that $f$ is flat by \cite[Tag 0539]{stacks-project}.
\end{proof}

\begin{prop} Let $f: V \rightarrow W$ be an extension of rank one valuation rings, with fraction fields $K_V$ and $K_W$ respectively. Then, there is a unique multiplicative order-preserving embedding $K_V^{\times}/V^{\times} \hookrightarrow K_W^{\times}/W^{\times}$ which is compatible with the map $K_V \rightarrow K_W$. Furthermore, if we fix a multiplicative order-preserving embedding $K_V^{\times}/V^{\times} \hookrightarrow \RR_{> 0}^{\times}$ (as in \ref{embed_value_group_reals_rk_one}) and assume that it has dense image, then there is a unique multiplicative order-preserving embedding $K_W^{\times}/W^{\times} \hookrightarrow \RR_{ > 0}^{\times}$ compatible with $K_V^{\times}/V^{\times} \hookrightarrow \RR_{> 0}^{\times}$.
\end{prop}

\begin{proof} From \ref{equiv_extensions_rk_one} we learn that the induced map $K_V^{\times} \hookrightarrow K_W^{\times}$ is injective and $f^{-1}(W^{\times}) = V^{\times}$, which implies that the induced map $K_V^{\times}/V^{\times} \rightarrow K_W^{\times}/W^{\times}$ is multiplicative and injective. Furthermore, from the definition of the order structure on value groups (cf. \ref{properties_value_group}) one can easily check that the induced map $K_V^{\times}/V^{\times} \rightarrow K_W^{\times}/W^{\times}$ is order preserving.

Next, fix a multiplicative order-preserving embedding $|-|_V: K_V^{\times}/V^{\times} \hookrightarrow \RR_{> 0}^{\times}$, and define
\begin{align*}
	|-|_W: K_W^{\times}/W^{\times} \rightarrow \RR_{> 0}^{\times} && |a|_W = \inf\{|\gamma|_V \text{ for all } \gamma \in K_V^{\times}/V^{\times} \text{such that } a \in (K_W^{\times}/W^{\times})_{\le \gamma} \}
\end{align*}
Its clear from the construction and the fact that $|-|_V$ has dense image that $|-|_W$ is the unique order preserving map $K_W^{\times}/W^{\times} \rightarrow \RR_{> 0}^{\times}$ which is compatible with $|-|_V$. To show that $|-|_W$ is multiplicative, notice that by the density of the value group $|-|_V: K^{\times}_V/V^{\times} \hookrightarrow \RR_{> 0}^{\times}$ we can find a increasing sequence of elements $\{\gamma_{a,n}\}, \{\gamma_{b,n}\} \subset K^{\times}_V/V^{\times}$ such that
\begin{align*}
	\lim_{n \rightarrow \infty} |\gamma_{a,n}|_V = |a|_W && \lim_{n \rightarrow \infty} |\gamma_{b,n}|_V = |b|_W
\end{align*}
then the multiplicativity of $|-|_V$ implies that $\lim_{n \rightarrow \infty} |\gamma_{a,n} \gamma_{b,n}| = |a|_W |b|_W$; and as both sequences $\{\gamma_{a,n}\}, \{\gamma_{b,n}\}$ are increasing we learn that $\gamma_{a,n} \gamma_{b, n} \le ab$, implying that $|a|_W |b|_W \le |ab|_W$. On the other hand, fix two objects $a,b \in K_W^{\times}/W^{\times}$, then its clear that if $a \le \gamma_a$ and $b \le \gamma_b$ for $\gamma_a, \gamma_b \in K^{\times}_V/V^{\times}$, then $ab \le \gamma_1 \gamma_2$ showing that $|ab|_W \le |a|_W |b|_W$. This proves that $|-|_W: K_W^{\times}/W^{\times} \rightarrow \RR_{> 0}^{\times}$ is multiplicative. 

It remains to show that $|-|_W$ is injective. Its clear from the construction of the order on $K_W^{\times}/W^{\times}$ that if $a \le b$ and $b \le a$ then $a = b \in K_W^{\times}/W^{\times}$, thus for two distinct elements $a, b \in K_W^{\times}/W^{\times}$ we may assume that $a < b$, and in particular $a/b < 1$. By the multiplicativity of $|-|_W$ it then suffices to show that $|a/b|_W < 1 \in \RR_{> 0}^{\times}$. As $W$ has rank one for any element $\gamma \in (K_V^{\times}/V^{\times})_{< 1}$ there exists a $n \gg 0$ such that $(a/b)^n < \gamma < 1$ which in turn implies that $|a/b|^n_W < 1$, proving the claim.
\end{proof}

\begin{lemma}\label{rk_one_val_tot_int_closed} Let $V$ be a rank one valuation ring with fraction field $K$, and fix a multiplicative order preserving embedding $|-|_V: K^{\times}/V^{\times} \hookrightarrow \RR^{\times}_{> 0}$. Then, for any $\pi \in \frakm \subset V$ we have that $V \subset V[\frac{1}{\pi}] = K$ is $\pi$-total integrally closed (cf. \ref{defn_total_int_closed}).
\end{lemma}

\begin{proof} We need to show that for any $f \in K$ if $\{f^{\ZZ_{\ge 0}} \} \subset \frac{1}{\pi^n}V$ for some $n$, then $f \in V$. Indeed, by assumption we have that $\{f^{\ZZ_{\ge 0}} \pi^n \} \subset V$ which implies that $|f^{\ZZ_{\ge 0}} \pi^n|_V \le 1$ in $\RR_{> 0}^{\times}$. Then, the multiplicativity of $|-|_V$ implies that $|f|_V \le 1$, which shows that $f \in V$ by \ref{properties_value_group}.
\end{proof}

\begin{lemma}\label{completion_rank_one_val} Let $V$ be a rank one valuation ring with fraction field $K$, and fix a multiplicative order preserving embedding $|-|: K^{\times}/V^{\times} \hookrightarrow \RR_{> 0}^{\times}$. Then, for any $\pi \in \frakm_V \subset V$, the canonical map $V \rightarrow V_{\pi}^{\wedge}$ is an extension of rank one valuation rings, where $V_{\pi}^{\wedge}$ denotes the (classical, equivalently derived) $\pi$-completion of $V$.
\end{lemma}

\begin{proof} As $V$ has is a rank one valuation ring we learn that $|\pi^n| \rightarrow 0$ as $n \rightarrow \infty$, thus classical $\pi$-completion, which is defined as $V_{\pi}^{\wedge} := \lim V/\pi^n$, agrees with the completion of $V$ with respect to the map $|-|: V \rightarrow \RR_{> 0}^{\times} \cup \{0\}$. The characterization of $V_{\pi}^{\wedge}$ as the completion of $V$ with respect to $|-|$ together with the multiplicativity of $|-|$, shows that one can further identify $V^{\wedge}_{\pi}$ with the subring of $K_{\pi}^{\wedge}$ of elements $a \in K_{\pi}^{\wedge}$ which satisfy $|a| \le 1$; where $K_{\pi}^{\wedge}$ is the completion of $K$ with respect to the norm $|-|: K \rightarrow \RR_{> 0}^{\times} \cup \{0\}$. Then Lemma \ref{recog_val_rings} show that $V_{\pi}^{\wedge}$ is a valuation ring, and Lemma \ref{arch_rank_one_valuations} show that $V^{\wedge}_{\pi}$ is a rank one valuation ring. Finally, the characterization of $V_{\pi}^{\wedge}$ as the completion of $V$ with respect to $|-|: V \rightarrow \RR_{ > 0}^{\times}$ show that $f:V \rightarrow V_{\pi}^{\wedge}$ is injective and that $f^{-1}(\frakm_{V^{\wedge}_{\pi}}) \subset \frakm_V$, which proves the claim by \ref{equiv_extensions_rk_one}.
\end{proof}

\begin{lemma}\label{properties_tilt_val_rings} Let $R$ be an integral perfectoid ring. Then,
\begin{enumerate}[(1)]
	\item If $R$ is a domain, then $R^{\flat}$ is a domain.
	\item If $R$ is a domain, then the isomorphism of multiplicative monoids $R^{\flat} \rightarrow \lim_{x \mapsto x^p} R$ induces an isomorphism of multiplicative monoids $\Frac(R^{\flat}) \rightarrow \lim_{x \mapsto x^p} \Frac(R)$. In particular, we obtain a multiplicative map $\sharp: \Frac(R^\flat) \rightarrow \Frac(R)$, by projection onto the last coordinate.
	\item If $R$ is a valuation ring, set $\Gamma = \Frac(R)^{\times}/R^{\times}$ and let $v: \Frac(R) \rightarrow \Gamma \cup \{0\}$ be the valuation associated to $R$ (cf. \ref{properties_value_group}). Then, the composition
	\begin{cd}
		v_{\sharp}: \Frac(R^{\flat}) \ar[r, "\sharp"] & \Frac(R) \ar[r, "v"] & \Gamma \cup \{0\}
	\end{cd}
	satisfies the hypothesis of Lemma \ref{recog_val_rings} and $R^{\flat} = \{x \in \Frac(R^{\flat}) | v_{\sharp}(x) \le 1 \}$. In particular $R^\flat$ is a valuation ring.
	\item If $R$ is a valuation ring, the multiplicativity of the maps $\sharp: \Frac(R^{\flat}) \rightarrow \Frac(R)$ and $\sharp: R^{\flat} \rightarrow R$ induce an isomorphism of ordered multiplicative monoids
	\begin{equation*}
		\sharp: \Frac(R^{\flat})^{\times}/R^{\flat \times} \rightarrow \Frac(R)^{\times}/R^{\times}
	\end{equation*}
	with respect to the order defined in Lemma \ref{properties_value_group}. In particular, the Krull dimension of $R^{\flat}$ is the same as the Krull dimension of $R$.
\end{enumerate}
\end{lemma}

\begin{proof} We begin by proving $(1)$. Recall from \ref{sharp_map_definition} that we have an isomorphism of multiplicative monoids $R^{\flat} \rightarrow \lim_{x \mapsto x^p} R$. Assume that we have $\{a_n\}, \{b_n\} \in R^{\flat} = \lim_{x \mapsto x^p} R$ such that $\{a_n\} \{b_n\} = 0$ in $R^{\flat}$ which implies by construction that $a_n b_n = 0$ in $R$ for all $n$. Then, as $R$ is a domain, either $a_0 = 0$ or $b_0= 0$, without loss of generality we may assume that $a_0 = 0$. As we have that $a_n^p = a_{n-1}$, then the domain assumption shows again that $a_n = 0$ for all $n$, proving that $\{a_n\} = 0 \in R^{\flat}$.

For $(2)$, first recall that we have a multiplicative map $\sharp: R^{\flat} \rightarrow R$ which is obtained as the projection onto the last coordinate $R^{\flat} = \lim_{x \mapsto x^p} R \rightarrow R$. As everything in sight is multiplicative, it follows that for any $a \in R^{\flat}$ we have an isomorphism $R^{\flat}[1/a] = \lim_{x \mapsto x^p} R[1/a^{\sharp}]$. Then, since every element $r \in R$ is in the image of $\sharp: R^{\flat} \rightarrow R$ up to a unit (\ref{compatible_roots_perfectoid}), passing to the limit inverting all $a \in R^{\flat}$ shows that we have an isomorphism of multiplicative monoids $\Frac(R^{\flat}) \rightarrow \lim_{x \mapsto x^p} \Frac(R)$.

Next we prove $(3)$. As the maps $v: \Frac(R) \rightarrow \Gamma \cup \{0\}$ and $\sharp: \Frac(R^\flat) \rightarrow \Frac(R)$ are multiplicative maps and $\Frac(R^\flat)$ is a field, in order to show that $v_{\sharp}$ satisfies the hypothesis of Lemma \ref{recog_val_rings} it suffices to show that for $a,b \in \Frac(R^{\flat})$ then $v_{\sharp}(a + b) \le \max(v_{\sharp}(a), v_{\sharp}(b))$. Furthermore, as $a,b \in \Frac(R^\flat)$ we know that there exists an $f \in \Frac(R^\flat)$ such that $af, bf \in R^\flat$, thus by the multiplicativity of $v_{\sharp}$ we may assume that $a,b \in R^\flat$. As $R^\flat$ is perfect, we know that $a,b$ admit compatible $p$-power roots, which we denote by $a^{1/p^n}, b^{1/p^n}$ respectively, then as $(a + b)^{\sharp} = a^{\sharp} + b^{\sharp} \bmod p$ we can conclude by the binomial theorem that
\begin{equation*}
	(a^{1/p^n} + b^{1/p^n})^{p^n \sharp} = (a + b)^{\sharp} = (a^{1/p^n \sharp} + b^{1/p^n \sharp})^{p^n} \bmod p^{n+1}
\end{equation*}
which by the $p$-completeness of $R$ implies that $\lim_{n \rightarrow \infty} (a^{1/p^n \sharp} + b^{1/p^n \sharp})^{p^n} = (a+b)^{\sharp}$ in $R$. Observing that $v(a^{1/p^n \sharp} + b^{1/p^n \sharp})^{p^n} \le \max(v(a^{\sharp}), v(b^{\sharp}))$, and assuming without loss of generality that $v(b^\sharp) \le v(a^\sharp)$, we learn that
\begin{equation*}
	(a^{1/p^n \sharp} + b^{1/p^n \sharp})^{p^n}/(a)^{\sharp} \in R
\end{equation*}
for all $n$, which implies by the $p$-completeness of $R$ that $(a + b)^{\sharp}/a^{\sharp} \in R$. This proves that $v_{\sharp}(a + b) \le \max(v_{\sharp}(a), v_{\sharp}(b))$, showing that $v_{\sharp}$ satisfies the conditions of Lemma \ref{recog_val_rings}. It remains to show that $R^{\flat} = \{x \in \Frac(R^{\flat}) | v_{\sharp}(x) \le 1 \}$. Indeed, as the map $\sharp: \Frac(R^\flat) \rightarrow \Frac(R)$ is compatible with $\sharp: R^{\flat} \rightarrow R$ it follows that $R^{\flat} \subset \{x \in \Frac(R^{\flat}) | v_{\sharp}(x) \le 1 \}$. Now, assume that $x \in \Frac(R^\flat)$ satisfies $v_{\sharp}(x) \le 1$, then $x^{\sharp} \in R \subset \Frac(R)$ and it admits compatible $p$-power roots in $R$, proving that $x \in R^{\flat}$ as desired.

Finally, we proof $(4)$. Its clear that we have a multiplicative map $\sharp: \Frac(R^{\flat})^{\times}/R^{\flat \times} \rightarrow \Frac(R)^{\times}/R^{\times}$, to show that it preserves the order fix $a,b \in \Frac(R^{\flat})^{\times}/R^{\flat \times}$ and let $A,B \in \Frac(R^{\flat})$ be corresponding lifts, then $a \le b$ if and only if $A/B \in R^{\flat}$, which in turn implies that $A^{\sharp}/B^{\sharp} \in R$, proving that $a^{\sharp} \le b^{\sharp} \in \Frac(R)^{\times}/R^{\times}$. Moreover, the fact that every element of $R$ admits compatible $p$-power roots up to a unit (\ref{compatible_roots_perfectoid}), implies that $\sharp: \Frac(R^{\flat})^{\times}/R^{\flat \times} \rightarrow \Frac(R)^{\times}/R^{\times}$ is surjective, it remains to show that if $x^{\sharp} = 1$ then $x = 1$ in $\Frac(R^{\flat})^{\times}/R^{\flat \times}$. Assume that $x^{\sharp} = 1$, and pick a lift $X \in R^{\flat \times} \subset \Frac(R^{\flat})^{\times}$, then the assumption that $x^{\sharp} = 1$ implies that $X^{\sharp} \in R^{\times}$ and thus any $p$-power root of $X$ will be an element of $R^{\times}$, proving the result. The claim about the Krull dimension of $R^{\flat}$ follows from \ref{prime_ideals_valuation_rings}.
\end{proof}

\begin{prop}\label{integral_tilt_corr_val_one} Let $V$ be an integral perfectoid rank one valuation ring. Then, the following categories are equivalent
\begin{enumerate}[(1)]
	\item The category of integral perfectoid rank one valuation rings $W$ over $V$, such that the structure map $V \rightarrow W$ is faithfully flat.
	\item The category of integral perfectoid rank one valuation rings $W^{\flat}$ over $V^{\flat}$, such that the structure map $V^{\flat} \rightarrow W^{\flat}$ is faithfully flat.
\end{enumerate}
where the functor from $(1)$ to $(2)$ is given by the tilting functor. Furthermore, if $V$ is $\varpi$-complete with respect to some $\varpi \in V$ where $\varpi^p$ divided $p$, then the above equivalence restricts to an equivalence between the categories
\begin{enumerate}[(1')]
	\item The category of integral perfectoid rank one valuation rings $W$ over $V$, such that $W$ is $\varpi$-complete and the structure map $V \rightarrow W$ is faithfully flat.
	\item The category of integral perfectoid rank one valuation rings $W^{\flat}$ over $V^{\flat}$, such that $W^{\flat}$ is $\varpi^\flat$-complete and the structure map $V^{\flat} \rightarrow W^{\flat}$ is faithfully flat.
\end{enumerate}
\end{prop}

\begin{proof} The equivalence between the categories $(1)$ and $(2)$ follows from Proposition \ref{prism_tilt_correspondence} and Lemma \ref{properties_tilt_val_rings}. Similarly, the equivalence between the categories (1') and (2') then follow from Lemma \ref{completeness_tilt}.
\end{proof}

\begin{lemma}\label{na_field_rk_one_dict} Let $K$ be a perfectoid non-archimedean field with norm $|-|: K \rightarrow \RR_{\ge 0}^{\times}$ and $\varpi \in K$ such that $1 > |\varpi^p| \ge |p|$. Then, the following categories are equivalent
\begin{enumerate}[(1)]
	\item The category of non-archimedean fields over $L$ over $K$
	\item The category of $\varpi$-complete rank one valuation rings $V$, with faithfully flat structure map $K_{\le 1} \rightarrow V$. In particular, all maps between rank one valuation rings are forced to be faithfully flat.
\end{enumerate}
Where the functor $(1) \rightarrow (2)$ is given by $L \mapsto L_{\le 1}$, with inverse given by $V \mapsto V[\frac{1}{\varpi}]$ (cf. \ref{equiv_almost_tf_banach}).
\end{lemma}

\begin{proof} It was explained in Example \ref{non_arch_field_val_rank_one} that $L_{\le 1}$ is a rank one valuation ring, and the fact that $K_{\le 1} \rightarrow L_{\le 1}$ is faithfully flat follows from \ref{equiv_extensions_rk_one}. On the other hand, it follows from \ref{equiv_extensions_rk_one} that if $K_{\le 1} \rightarrow V$ is faithfully flat then $V[\frac{1}{\varpi}]$ is a field. Finally, since rank one valuation rings are totally integrally closed in their fraction field (\ref{rk_one_val_tot_int_closed}), we learn from the equivalence \ref{uBan_tic_dict} that $V[\frac{1}{\varpi}]$ is uniform so it is a non-archimedean field, $L_{\le 1}$ is $\varpi$-complete and that the functors above describe an equivalence of categories. 
\end{proof}

\subsection{Tilting correspondence}

Let $V$ be an integral perfectoid rank one valuation ring, which is $\varpi$-complete with respect to some $\varpi \in V$ such that $\varpi^p$ divides $p$, and denote by $K$ its fraction field. For example $K$ could be a non-archimedean perfectoid field with $\varpi \in K$ satisfying $1 > |\varpi^p| > |p|$ and $V = K_{\le 1}$ (cf. \ref{non_arch_field_val_rank_one} and \ref{perfectoid_field_is_int_perfectoid}). Fix an multiplicative order preserving injective map of $|-|: K^{\times}/V^{\times} \hookrightarrow \RR_{> 0}^{\times}$, which is equivalent to specifying a non-archimedean norm $|-|: K \rightarrow \RR_{\ge 0}^{\times}$. And denote by $V^{\flat}$ and $K^{\flat}$ the tilt of $V$ and $K$ respectively (cf. \ref{integral_tilt_corr_val_one}), then the order preserving multiplicative isomorphism $\sharp: K^{\flat \times}/V^{\flat \times} \rightarrow K^{\times}/V^{\times}$ determines a unique multiplicative order preserving injective map $|-|: K^{\flat \times}/V^{\flat \times} \hookrightarrow \RR_{> 0}^{\times}$ compatible with the one on $K^{\times}/V^{\times}$, which in turn induces a unique non-archimedean multiplicative norm on $|-|: K^{\flat} \rightarrow \RR_{\ge 0}^{\times}$, making $K^{\flat}$ a perfectoid non-archimedean field. More concretely, if $K$ is a perfectoid non-archimedean field, then the composition
\begin{cd}
	K^{\flat} \ar[r, "\sharp"] & K \ar[r, "|-|"] & \RR_{\ge 0}^{\times}
\end{cd}
determines a multiplicative norm on the perfectoid field $K^{\flat}$, this was checked in the proof of Lemma \ref{properties_tilt_val_rings}(3).

\begin{prop}\label{tilting_corr_banach} Let $K$ be a perfectoid non-archimedean field with norm $|-|: K \rightarrow \RR_{\ge 0}^{\times}$ and $\varpi \in K$ such that $1 > |\varpi^p| > |p|$. Let $K^{\flat}$ be its tilt, which has norm $|-|: K^{\flat} \rightarrow \RR_{\ge 0}^{\times}$ given by $x \mapsto |x^\sharp|$. Then, the composition
\begin{cd}
	\Perfd_{K}^{\Ban} \ar[rr, "(-)_{\le 1}"] && 
	\Perfd_{K_{\le 1}}^{\Prism \tic} \ar[r, "(-)^{\flat}"] & 
	\Perfd_{K^{\flat}_{\le 1}}^{\Prism \tic} \ar[rr, "(-)\lbrack \frac{1}{\varpi^\flat} \rbrack"] && 
	\Perfd_{K^{\flat}}^{\Ban}
\end{cd}
determines an equivalence of categories; we will often abuse notation and denote the composition by $(-)^\flat$. This equivalence has the following properties
\begin{enumerate}[(1)]
	\item It identifies the subcategories of non-archimedean field over $K$, with the category of non-archimedean fields over $K^{\flat}$. Furthermore, if $L$ is a non-archimedean field over $K$ with multiplicative norm $|-|_L$, then $L^{\flat}$ will have a multiplicative norm given by $L^\flat \ni l \mapsto |l^\sharp|_L$ (cf. Lemma \ref{properties_value_group}(3)).
	\item If $A$ is an object of $\Perfd_K^{\Ban}$ with power-multiplicative norm $|-|_A$, then $A^\flat \in \Perfd_{K^{\flat}}^{\Ban}$ will have power multiplicative norm given by $A^\flat \ni a \mapsto |a^\sharp|_A$. A priori the sharp map is only defined on $\sharp: A^{\flat}_{\le 1} \rightarrow A_{\le 1}$, but its multiplicativity implies it naturally extends to
	\begin{equation*}
		\sharp: A^{\flat}_{\le 1}[\frac{1}{\varpi^\flat}] = A^\flat \longrightarrow A = A_{\le 1}[\frac{1}{\varpi}]
	\end{equation*}
	Furthermore, the map $\sharp: A^\flat \rightarrow A$ determines a bijection between multiplicative semi-norms $A \rightarrow \RR_{\ge 0}^{\times}$ and $A^\flat \rightarrow \RR_{\ge 0}^{\times}$ (cf. \ref{defn_berko_spectrum}).
\end{enumerate}
Finally, let us remark that the equivalence $\Perfd_{K}^{\Ban} \simeq \Perfd_{K_{\le 1}}^{\Prism \tic}$, which was deduced from \ref{equiv_almost_tf_banach}, relies critically on the fact that $K$ already has a non-archimedean multiplicative, norm thus its important to fix the norm on $K^\flat$ through the map $\sharp: K^\flat \rightarrow K$.
\end{prop}

\begin{proof} It follows from Proposition \ref{equiv_perfd_ban_tic} and Proposition \ref{integral_tilt_corr} that the composition
\begin{align*}
	(-)^\flat: \Perfd_{K}^{\Ban} \longrightarrow \Perfd_{K^{\flat}}^{\Ban}
\end{align*}
determines an equivalence of categories. For (1) the fact that $(-)^{\flat}$ induces an equivalence between perfectoid non-archimedean fields is a combination of Proposition \ref{integral_tilt_corr} and Lemma \ref{na_field_rk_one_dict}. To complete the proof of $(1)$ we need to check that the norm on $L^\flat$ is given by $L^\flat \ni l \mapsto |l^\sharp|_L$. Indeed, from Lemma \ref{embed_value_group_reals_rk_one} we learn that the map $K^{\flat \times}/K_{\le 1}^{\flat \times} \hookrightarrow \RR_{> 0}^{\times}$ completely determines the map $L^{\flat \times}/L_{\le 1}^{\flat \times} \hookrightarrow \RR_{> 0}^{\times}$, then the claim follows from the isomorphism of multiplicative ordered value groups $\sharp: L^{\flat \times}/L_{\le 1}^{\flat \times} \rightarrow L^{\times}/L_{\le 1}^{\times}$, which is compatible with the isomorphism $\sharp: K^{\flat \times}/K_{\le 1}^{\flat \times} \rightarrow K^{\times}/K_{\le 1}^{\times}$.

In order to proof (2), recall that since $A$ and $A^{\flat}$ are uniform the Berkovich maximum modulus principle \ref{berko_max_modulus} implies that their norms are completely determined by their multiplicative seminorms $A,A^\flat \rightarrow \RR_{\ge 0}$ (\ref{defn_berko_spectrum}). Thus, it remains to show that the map $\sharp: A^\flat \rightarrow A$ determines a bijection between the multiplicative seminorms on $A$ and $A^\flat$. Indeed, if $x: A \rightarrow \RR_{\ge 0}^{\times}$ is the multiplicative semi-norm then $x$ factors as
\begin{equation*}
	x: A \rightarrow \cH(x) \rightarrow \overline{\cH(x)}^\wedge \rightarrow \overline{\cH(x)}^{\wedge \times}/\overline{\cH(x)}^{\wedge \times}_{\le 1} \cup \{0\} \hookrightarrow \RR_{\ge 0}^{\times}
\end{equation*}
where $\cH(x)$ is a non-archimedean field (cf. \ref{defn_completed_residue_field}), $\overline{\cH(x)}^\wedge$ is the completion of its algebraic closure (in particular it is a perfectoid field), the map $\overline{\cH(x)}^\wedge \rightarrow \overline{\cH(x)}^{\wedge \times}/\overline{\cH(x)}^{\wedge \times} \cup \{0\}$ is the canonical map to its value group (\ref{properties_value_group}), and $\overline{\cH(x)}^{\wedge \times}/\overline{\cH(x)}^{\wedge \times} \hookrightarrow \RR_{>0 }^{\times}$ the unique multiplicative ordered map compatible with $K^{\times}/K_{\le 1}^{\times} \hookrightarrow \RR_{> 0}^{\times}$ (\ref{embed_value_group_reals_rk_one}). Tilting this collection of maps gives rise to a multiplicative seminorm
\begin{equation*}
	x^\flat: A^\flat \rightarrow \overline{\cH(x)}^{\wedge \flat} \rightarrow \overline{\cH(x)}^{\wedge \flat \times}/\overline{\cH(x)}^{\wedge \flat \times}_{\le 1} \cup \{0\} \hookrightarrow \RR_{\ge 0}^{\times}
\end{equation*}
And by the functoriality of the sharp map we can conclude that this identifies $x^\flat: A^\flat \rightarrow \RR_{\ge 0}^{\times}$ with $x \circ \sharp$. Then, the equivalence $\Perfd_{K}^{\Ban} \simeq \Perfd_{K^\flat}^{\Ban}$ shows that the sharp map $\sharp: A^\flat \rightarrow A$ determines a bijection between the multiplicative semi-norms on $A$ and $A^\flat$, finishing the proof of (2).
\end{proof}

\begin{prop}\label{tilting_ban_topo} Let $K$ be a perfectoid non-archimedean field with norm $|-|: K \rightarrow \RR_{\ge 0}^{\times}$ and $\varpi \in K$ such that $1 > |\varpi^p| > |p|$. Let $K^{\flat}$ be its tilt, which has norm $|-|: K^{\flat} \rightarrow \RR_{\ge 0}^{\times}$ given by $x \mapsto |x^\sharp|$. Let $A$ and $A^\flat$ be perfectoid Banach algebras over $K$ and $K^\flat$ respectively, then
\begin{enumerate}[(1)]
	\item The sharp map $\sharp: A^\flat \rightarrow A$ determines a homeomorphism $\sharp^*: |\cM(A)| \rightarrow |\cM(A^\flat)|$ (cf. \ref{defn_berko_spectrum}), given explicitly by $x \mapsto x^\flat$ where $x^\flat = x \circ \sharp$.
	\item If $A \rightarrow B$ is a morphism of perfectoid Banach $K$-algebras, and the induced map $|\cM(B)| \rightarrow |\cM(A)|$ has image $U \subset |\cM(A)|$. Then, the corresponding map $A^\flat \rightarrow B^\flat$ induces a map $|\cM(B^\flat)| \rightarrow |\cM(A^\flat)|$ with image $\sharp^{*, -1}(U) \subset |\cM(A^\flat)|$.
\end{enumerate}
\end{prop}

\begin{proof} We already showed in Proposition \ref{tilting_corr_banach}(1) that the map $\sharp: A^\flat \rightarrow A$ determines a bijection $\sharp^*: |\cM(A)| \rightarrow |\cM(A^\flat)|$ given by $x \mapsto x^\flat$, as both topological spaces are compact hausdorff (\ref{berko_sp_is_comp}) to show that $\sharp^*$ is a homeomorphism it suffices to show that $\sharp^*$ is continuous (\ref{monadicity_comp_set}). Recall that the topology on $|\cM(A)|$ is defined as the weakest topology making the all maps $|\cM(A)| \rightarrow \RR_{\ge 0}$, defined as $x \mapsto |g(x)|$ for all $g \in A$, continuous. Thus it suffices to show that the maps $|\cM(A)| \rightarrow \RR_{\ge 0}$ given by $x \mapsto |f(x^\flat)|$ are continuous for all $f \in A^\flat$, but since $|f(x^\flat)| = |f^\sharp(x)|$ by \ref{tilting_corr_banach} the claim follows. This completes the proof of (1).

In order to proof (2), it suffice to show that a multiplicative semi-norm $x: A \rightarrow \RR_{\ge 0}$ factors as $A \rightarrow B \rightarrow \RR_{\ge 0}$ if and only if $x^\flat: A^\flat \rightarrow \RR_{\ge 0}$ factors as $A^\flat \rightarrow B^\flat \rightarrow \RR_{\ge 0}$. Indeed, assume that $x: A \rightarrow \RR_{\ge 0}$ factors through $B$, let $\cH(x)_B$ the the non-archimedean field associated to the multiplicative semi-norm $B \rightarrow \RR_{\ge 0}$ (cf. \ref{defn_completed_residue_field}), and $\overline{\cH(x)}_B^{\wedge}$ the completion of its algebraic closure (in particular, it is a perfectoid field). Then, $x$ admits a factorization as
\begin{equation*}
	x: A \rightarrow B \rightarrow \cH(x)_B \rightarrow \overline{\cH(x)}_B^\wedge \rightarrow \overline{\cH(x)}_B^{\wedge \times}/\overline{\cH(x)}^{\wedge \times}_{B, \le 1} \cup \{0\} \hookrightarrow \RR_{\ge 0}^{\times}
\end{equation*}
Tilting this collection of maps we get the map (cf. the proof of \ref{tilting_corr_banach})
\begin{equation*}
	x^\flat: A^\flat \rightarrow B^\flat \rightarrow \overline{\cH(x)}_{B^\flat}^{\wedge \flat} \rightarrow \overline{\cH(x)}_{B^\flat}^{\wedge \flat \times}/\overline{\cH(x)}^{\wedge \flat \times}_{B^\flat, \le 1} \cup \{0\} \hookrightarrow \RR_{\ge 0}^{\times}
\end{equation*}
showing that if $x: A \rightarrow \RR_{\ge 0}$ factors through $B$, then $x^\flat: A^\flat \rightarrow \RR_{\ge 0}$ factors through $B^\flat$. A completely symmetrical argument shows the converse, completing the proof of (2).
\end{proof}

\newpage

\chapter{The Berkovich Spectrum}\label{chapt_berk_sp}

Throughout this chapter we fix a prime number $p$ and a perfectoid non-archimedean field $K$ together with an object $\varpi \in K$ satisfying $1 > |\varpi^p| \ge |p|$ and  a compatible system of $p$-power roots $\{\varpi^{1/p^n}\}_{n \in \ZZ_{\ge 0}}$. In Section \ref{sect_basic_prop_berko_sp} we introduce the Berkovich spectrum following \cite{berkovich_spectral}, and recall some favorable categorical properties of the category of compact Hausdorff spaces which we will leverage in the next chapter. In Section \ref{sect_rational_domains} we establish the basic properties of rational domains of Banach $K$-algebras without any finiteness assumptions -- subtle distinctions appear between ``algebraic'' and ``topological'' definitions of rational domains, though these disappear after uniformization. In Section \ref{sect_struct_presheaf_perfd} we show that affinoid perfectoid spaces admit a well-behaved structure presheaf, in particular we show that the topological rational domains admit a unique representative by an affinoid perfectoid space, and that completed residue fields of affinoid perfectoid spaces are perfectoid Banach $K$-algebras. Finally, in Section \ref{sect_tate_acyclicity_perfd} we leverage the dictionary (Theorem \ref{intro_dictionary}) and $\arc$-descent of integral perfectoid algebras of Bhatt and Scholze (\cite[Proposition 8.10]{prisms}) to prove Tate acyclicity for perfectoid Banach $K$-algebras (Theorem \ref{intro_tate_acyclicity_perfd}).

\section{Basic Properties}\label{sect_basic_prop_berko_sp}

\subsection{Introducing the spectrum}

Throughout this section fix a non-trivially valued non-archimedean field $K$, and recall that all Banach $K$-algebras are assumed to be non-archimedean. We will be particularly interested in the opposite category of Banach $K$-algebras, which we denote by $\Ban_K^{\op}$ -- to a Banach $K$-algebra $A$ we use the symbol $\cM(A)$ to denote the corresponding object in $\Ban_K^{\op}$, and to a morphism $A \rightarrow B$ in $\Ban_K$ we denote by $\cM(B) \rightarrow \cM(A)$ the corresponding morphism in $\Ban_K^{\op}$. In \cite{berkovich_spectral} Berkovich introduced a functor
\begin{align*}
	|-|: \Ban_K^{\op} \longrightarrow \Comp && \cM(A) \mapsto |\cM(A)|
\end{align*}
which associates to $\cM(A)$ its ``underlying topological space'', which we will proof is in fact a compact hausdorff space. We will denote by $\Comp$ the category of compact hausdorff spaces. In what follows we will review the basic properties of this construction, as covered in \cite[Chapter 1]{berkovich_spectral}.

\begin{defn}\label{defn_berko_spectrum} Let $A$ be a Banach $K$-algebra, define $|\cM(A)|$ to be the set of non-archimedean semi-norms (i.e. satisfying $(2), (3)$ and $(4)$ of \ref{defn_banach})
\begin{align*}
	x: A \longrightarrow \RR_{\ge 0} && f \mapsto |f(x)|
\end{align*}
which satisfy the following additional properties
\begin{enumerate}[(1)]
	\item The map $x: A \rightarrow \RR_{\ge 0}$ is bounded, that is, for all $f \in A$ we have that $|f(x)| \le |f|_{A}$, where $|-|_A$ denotes the norm on $A$.
	\item The map $x: A \rightarrow \RR_{\ge 0}$ is multiplicative, that is, for any pair of objects $f,g \in A$ we have that $|fg(x)| = |f(x)||g(x)|$ and $|1(x)| = 1$.
\end{enumerate}
We will call the semi-norms $A \rightarrow \RR_{\ge 0}$ satisfying conditions $(1)$ and $(2)$ rank one valuations (or multiplicative semi-norms) on $A$. Furthermore, we endow $|\cM(A)|$ with the weakest topology making the maps
\begin{align*}
	|\cM(A)| \longrightarrow \RR_{\ge 0} && x \mapsto |f(x)| \text{ for all } f \in A
\end{align*}
continuous.
\end{defn}

\begin{lemma}\label{bdd_valuations} Let $A$ be a Banach $K$-algebra, and $x: A \rightarrow \RR_{\ge 0}$ be a map satisfying $(2)$ of \ref{defn_berko_spectrum}. Then, $x$ satisfies $(1)$ if and only if there exists a $C > 0$ such that $|f(x)| \le C |f|_A$ for all $f \in A$.
\end{lemma}

\begin{proof} If $x: A \rightarrow \RR_{\ge 0}$ satisfies $(1)$ then its clear that $|f(x)| \le C |f|_A$ with $C = 1$ for all $f \in A$. On the other hand, since $x$ satisfies $(2)$ the inequality $|f^n(x)| \le C |f^n|_A$ implies that $|f(x)| \le C^{1/n} |f|_A$ for all $n \ge 0$, which in turn implies that $|f(x)| \le |f|_{A}$.
\end{proof}

\begin{prop}\label{berko_sp_is_comp} For any non-zero Banach $K$-algebra $A$, the topological space $|\cM(A)|$ is a non-empty compact hausdorff space.
\end{prop}

\begin{proof} The fact that $|\cM(A)|$ is non-empty follows from \cite[Theorem 1.2.1]{berkovich_spectral}. We present an alternative proof of the fact that $|\cM(A)|$ is a compact hausdorff space which we learn from Mattias Jonsson. From the construction of $|\cM(A)|$ as a topological space it is clear that we have an injective morphism
\begin{align*}
	|\cM(A)| \longrightarrow \prod_{f \in A} [0, |f|_A] && x \mapsto \prod_{f \in A} |f(x)|
\end{align*} 
which is an homeomorphism onto its image. Thus, it remains to show that the image of $|\cM(A)| \rightarrow \prod_{f \in A} [0, |f|_A]$ is a closed subset. Indeed, a point $(t_f)_{f \in A} \in \prod_{f \in A} [0, |f|_A]$ lies in the image of $\cM(A)$ if and only if it satisfies the following conditions $t_0 = 0$, $t_1 = 1$, $t_f = t_{-f}$, $t_{f + g} \le t_{f} + t_{g}$ and $t_{fg} = t_f t_g$; and since this conditions define a closed subset of $\prod_{f \in A} [0, |f|_A]$ the result follows.
\end{proof}

\begin{prop}\label{berko_sp_functor} The assignment $\cM(A) \mapsto |\cM(A)|$ determines a functor $\Ban_K^{\op} \rightarrow \Comp$.
\end{prop}

\begin{proof} Since morphisms of Banach $K$-algebras $\varphi: B \rightarrow A$ are bounded, it follows from Lemma \ref{bdd_valuations} that if we have a rank one valuation $A \rightarrow \RR_{\ge 0}$, then the composition $B \rightarrow A \rightarrow \RR_{\ge 0}$ is a rank one valuation on $B$. Hence, we get an induced map of sets $|\cM(A)| \rightarrow |\cM(B)|$, it remains to show that this map is continuous. From the proof of \ref{berko_sp_is_comp} we learn that we can realize $|\cM(B)|$ as a closed subset of $\prod_{f \in B} [0, |f|_B]$, and from Lemma \ref{bdd_valuations} we learn that we have a map of sets
\begin{align*}
	|\cM(A)| \longrightarrow \prod_{f \in B} [0, |f|_B] && x \mapsto |\varphi(f)(x)|
\end{align*}
whose image is contained in $|\cM(B)| \subset \prod_{f \in B} [0, |f|_B]$. It remains to show that the map $|\cM(A)| \rightarrow \prod_{f \in B} [0, |f|_B]$ is continuous, but this is clear from the definition of the topology on $|\cM(A)|$.
\end{proof}

\begin{prop}\label{recog_unit_berko} Let $A$ be a Banach $K$-algebra. An element $f \in A$ is invertible if and only if $|f(x)| \not= 0$ for all $x \in |\cM(A)|$.
\end{prop}

\begin{proof} If $f \in A$ is invertible, then for any $x \in |\cM(A)|$ we have the identity $1 = |ff^{-1} (x)| = |f(x)| |f^{-1}(x)|$ showing that $|f(x)| \not= 0$ for all $x \in |\cM(A)|$. Conversely, if $f \in A$ is not invertible, then $f \in \frakm \subset A$ is an element of a maximal ideal of $A$. As maximal ideals of Banach $K$-algebras are automatically closed (\cite[1.2.4/5]{BGR}) it follows that $\kappa := A/\frakm$ endowed with the quotient norm $A \twoheadrightarrow \kappa$ is a Banach field, which in turn implies that $|\cM(\kappa)|$ is non-empty by \ref{berko_sp_is_comp}. It then follows that for any rank one valuation $x \in |\cM(\kappa)|$ the composition $A \twoheadrightarrow \kappa \rightarrow \RR_{\ge 0}$ will send $f \mapsto 0$, as desired. 
\end{proof}

\begin{defn} Let $\{A_i\}_{i \in I}$ be a collection of Banach $K$-algebras indexed by a set $I$. Then, we define the bounded direct product of $\{A_i\}_{i \in I}$ as
\begin{equation*}
	\prod_{i \in I} A_i := \{ a = (a_i)_{i \in I} \text{ which satisfy } |a| := \sup_{i \in I} |a_i|_{A_i} < \infty \} 
\end{equation*}
with norm $|a| = \sup_{i \in I} |a_i|_{A_i}$. Its easy to see that the bounded direct product $\prod A_i$ is itself a non-archimedean Banach $K$-algebra. Furthermore, the bounded direct product $\prod A_i$ enjoys the following universal property: for a collection of contractive morphisms $\{B \rightarrow A_i\}_{i \in I}$ there in an essentially unique morphism
\begin{equation*}
	B \longrightarrow \prod_{i \in I} A_i
\end{equation*}
factoring the maps $B \rightarrow \prod A_i \rightarrow A_i$; or in other words we have an identity $\Hom_{\Ban_K^{\contr}} (B, \prod A_i) = \prod \Hom_{\Ban_K^{\contr}} (B, A_i)$. However, it is not clear that $\prod A_i$ enjoys the same universal property with respect to a collection of bounded (not necessarily contractive) morphisms $\{B \rightarrow A_i\}_{i \in I}$.
\end{defn}

\begin{lemma}\label{stone_cech_berko_sp} Let $\{K_i\}_{i \in I}$ be a collection of non-archimedean fields indexed by a set $I$. Then, $|\cM(\prod_{I} K_i)|$ is homeomorphic to $\beta(I)$, the Stone-Cech compactification of $I$ as a discrete set. Furthermore, given a subset $J \subset I$ the quotient map $\prod_I K_i \rightarrow \prod_{J} K_i$ induces an injective map
\begin{equation*}
	|\cM(\prod_J K_i)| \hookrightarrow |\cM(\prod_I K_i)|
\end{equation*}
which corresponds to the canonical inclusion $\beta(J) \subset \beta(I)$.
\end{lemma}

\begin{proof} This follows from \cite[Proposition 1.2.3]{berkovich_spectral} and its proof. We refer the reader to \cite[Section 3.2]{lurie_ultracategories} for a discussion on the Stone Cech compactification $\beta(I)$ of a discrete set $I$.
\end{proof}

\begin{defn}\label{defn_completed_residue_field} Let $A$ be a Banach $K$-algebra and $x: A \rightarrow \RR_{\ge 0}$ a rank one valuation, in the sense of \ref{defn_berko_spectrum}. Then, its clear from the definitions that the $\ker(x)$ the kernel of $x$ is equal to a prime ideal $\frakp_x \subset A$. Then, the rank one valuation $x: A \rightarrow \RR_{\ge 0}$ determines a rank one valuation on the domain $x: A/\frakp_x \rightarrow \RR_{\ge 0}$, and by the multiplicativity of $x$ it extends to the fraction field
\begin{eqnarray}
	x: \Frac(A/\frakp_x) \rightarrow \RR_{\ge 0}
\end{eqnarray}
we denote by $\cH(x)$ the completion of $\Frac(A/\frakp_x)$ with respect to $x$, and call it the completed residue field of $A$ at $x \in |\cM(A)|$. In particular, we get that the map $x: A \rightarrow \RR_{\ge 0}$ factor uniquely as $x: A \rightarrow \cH(x) \rightarrow \RR_{\ge 0}$. By construction, it follows that $\cH(x)$ is uniform, and so the map $A \rightarrow \cH(x)$ is contractive, showing that $\cH(x)$ is non-archimedean by \ref{recog_na}.

From \ref{stone_cech_berko_sp} it follows that $|\cM(\cH(x))|$ is a singleton, and from the construction of the map $A \rightarrow \cH(x)$ it follows that the image of $|\cM(\cH(x))| \rightarrow |\cM(A)|$ is exactly $x \in |\cM(A)|$.
\end{defn}

\begin{defn} Let $A$ be a Banach $K$-algebra, the Gelfand transform is the map 
\begin{equation*}
	A \rightarrow \prod_{x \in |\cM(A)|} \cH(x)
\end{equation*}
induced from the collection of contractive morphisms $\{A \rightarrow \cH(x)\}_{x \in |\cM(A)|}$. The induced map of Berkovich spectrum
\begin{equation*}
	|\cM(\prod_{x \in |\cM(A)|} \cH(x))| \longrightarrow |\cM(A)|
\end{equation*}
is surjective, as for any $x \in |\cM(A)|$ the map $A \rightarrow \cH(x)$ factors as $A \rightarrow \prod_{y \in |\cM(A)|} \cH(y) \rightarrow \cH(x)$ which by functoriality of $|\cM(-)|$ produces a map
\begin{equation*}
	|\cM(\cH(x))| \rightarrow |\cM(\prod_{y \in |\cM(A)|} \cH(y))| \rightarrow |\cM(A)|
\end{equation*}
showing surjectivity. Furthermore, the same argument shows that the map $|\cM(\prod \cH(x))| \rightarrow |\cM(A)|$ of compact hausdorff spaces can be identified with the map $\beta(|\cM(A)|^\delta) \rightarrow |\cM(A)|$ from the Stone Cech compactification of $|\cM(A)|$ considered as a discrete set (cf. \ref{monadicity_comp_set}).
\end{defn}

\begin{prop}\label{univ_completed_residue} Let $A$ be a Banach $K$-algebra. Fix a point $x \in |\cM(A)|$, then the induced morphisms $\cM(\cH(x)) \rightarrow \cM(A)$ satisfies the following universal property: for each map $\cM(K_x) \rightarrow \cM(A)$ from a non-archimedean field $K_x$, whose image $|\cM(K_x)| \rightarrow |\cM(A)|$ is equal to $x \in |\cM(A)|$, there exist a unique factorization
\begin{align*}
	A \longrightarrow \cH(x) \longrightarrow K_x && \cM(K_x) \rightarrow \cM(\cH(x)) \rightarrow \cM(A)
\end{align*}
\end{prop}

\begin{proof} For any non-archimedean field $K_x$ we claim that the rank one valuation $K_x \rightarrow \RR_{\ge 0}$ corresponding to the unique point of $\cM(K_x)$ is exactly the norm $|-|_{K_x}: K_x \rightarrow \RR_{\ge 0}$ on $K_x$. Indeed, since $K_x$ is a non-archimedean field it follows that the norm $|-|_{K_x}$ defines a point of $|\cM(K_x)|$, and since $|\cM(K_x)|$ is a singleton the claim follows.

Let $x \in |\cM(A)|$ be the point which is on the image of the map $|\cM(K_x)| \rightarrow |\cM(A)|$, it follows directly from the definitions that the composition $A \rightarrow K_x \rightarrow \RR_{\ge 0}$ of the map $A \rightarrow K_x$ followed by $|-|_{K_x}: K_x \rightarrow \RR_{\ge 0}$, is exactly the rank one valuation $x: A \rightarrow \RR_{\ge 0}$ corresponding to the point $x \in |\cM(A)|$. Therefore, it follows that we get a factorization
\begin{equation*}
	A \rightarrow \Frac(A/\frakp_x) \rightarrow K_x \rightarrow \RR_{\ge 0}
\end{equation*}
and as $K_x$ is complete with respect to $|-|_{K_x}: K_x \rightarrow \RR_{\ge 0}$ it follows that the map $\Frac(A/\frakp_x) \rightarrow K_x$ factors through the completion of $\Frac(A/\frakp_x)$ with respect to $\Frac(A/\frakp_x) \rightarrow K_x \rightarrow \RR_{\ge 0}$, which is exactly $\cH(x)$, proving the result.
\end{proof}

\begin{prop}[Berkovich's maximum modulus principle]\label{berko_max_modulus} Let $A$ be a Banach $K$-algebra. Then, the spectral radius $\rho_A$ of $A$ (cf. \ref{defn_sp_radius}) satisfies
\begin{equation*}
	\rho_A (f) = \max_{x \in |\cM(A)|} |f(x)|
\end{equation*}
for all $f \in A$.
\end{prop}

\begin{proof} \cite[Theorem 1.3.1]{berkovich_spectral}
\end{proof}

\begin{corollary}\label{recog_na_field} Let $A$ be a uniform Banach $K$-algebra which satisfies $|\cM(A)| = \pt$. Then, $A$ is a non-archimedean field.
\end{corollary}

\begin{proof} Since $A$ is uniform it follows from the definition that $\rho_A = |-|_A$, furthermore Berkovich's maximum modulus principle (\ref{berko_max_modulus}) guarantees that $\rho_A(f) = |f(x)|$ for the unique point $x \in |\cM(A)|$, in particular this implies that $\rho_A$ is a multiplicative norm. It remains to show that $A$ is a field. Assume that there exists a non-zero maximal ideal $\frakm \subset A$, then as maximal ideals of Banach rings are closed \cite[Corollary 1.2.4/5]{BGR} it follows that $\kappa := A/\frakm$ is a non-zero Banach $K$-algebra endowed with the quotient norm $A \twoheadrightarrow \kappa$. Finally, from \ref{berko_sp_is_comp} we learn that $|\cM(\kappa)|$ is non-empty and so there exists a rank one valuation $y: A \twoheadrightarrow \kappa \rightarrow \RR_{\ge 0}$ which satisfies $|\frakm(y)| = 0$, and so different from the point $x \in |\cM(A)|$ corresponding to $\rho_A$. We have reached a contradiction.
\end{proof}

\subsection{The category of compact hausdorff spaces}

Our goal in this section is to review some properties of the category $\Comp$, the category of compact hausdorff spaces, which will play an important role for us. We will follow \cite[Appendix A and B]{lurie_ultracategories}.

\begin{defn}[Effective epimorphism]\label{defn_eff_epi} Let $\cC$ be a category that admits fiber products, and suppose we are given a morphism $f: X \rightarrow Y$ in $\cC$. Let $X \times_Y X$ denote the fiber product of $X$ with itself over $Y$, and let $\pi, \pi^{\prime}: X \times_Y X \rightarrow X$ denote the projection onto the two factors. We will say that $f$ is an effective epimorphism if it exhibits $Y$ as a coequalizer of the maps $\pi, \pi^{\prime}: X \times_Y X \rightarrow X$. In other words, $f$ is an effective epimorphism if, for every object $Z \in \cC$, composition with $f$ induces a bijection
\begin{equation*}
	\Hom_{\cC} (Y, Z) = \{ u \in \Hom_{\cC} (X,Z) : u \circ \pi = u \circ \pi^\prime \}
\end{equation*}
\end{defn}

\begin{rem} Let $\cC$ be a category that admits fiber products. Then every effective epimorphism is an epimorphism. Indeed, let $f: X \rightarrow Y$ be an effective epimorphism, we would need to show that for every object $Z \in \cC$ the following map
\begin{equation*}
	\Hom_{\cC} (Y, Z) \longrightarrow \Hom_{\cC} (X, Z)
\end{equation*}
is an injection, but this follows from the definition above.
\end{rem}

\begin{example} In the category of sets every surjective map is an effective epimorphism: if $g: X \rightarrow Y$ is a surjective map of sets, then we can recover $Y$ as the quotient of $X$ by the equivalence relation $R = X \times_Y X = \{ (x, x^\prime): g(x) = g(x^\prime) \}$.
\end{example}

\begin{prop}\label{monadicity_comp_set} The forgetful functor $U: \Comp \rightarrow \Set$ has the following properties
\begin{enumerate}[(1)]
	\item The functor $U$ admits a left adjoint $\beta: \Set \rightarrow \Comp$, given by $I \mapsto \beta(I)$, where $\beta(I)$ denotes the Stone-Cech compactification of the set $I$. In particular, $U$ preserves limits.
	\item The functor $U$ detects isomorphisms. That is, any map $f: X \rightarrow Y$ in $\Comp$ is an isomorphism if and only if $U(f)$ is an isomorphism.
	\item For any pair of maps $X \rightrightarrows Y$ in $\Comp$ which admit a common section $Y \rightarrow X$, there exists a coequalizer $\coeq(X \rightrightarrows Y)$ in $\Comp$, and the canonical map 
	\begin{equation*}
		\coeq(U(X) \rightrightarrows U(Y) ) \rightarrow U(\coeq(X \rightrightarrows Y))
	\end{equation*}
	is an isomorphism. 
\end{enumerate}
\end{prop}

\begin{proof} It is shown in \cite[V.6.2]{categories_maclane} that the left adjoint to $U: \Comp \rightarrow \Set$ is given by $\beta$, proving $(1)$. Statement $(2)$ follows from the fact that $U$ is a monadic functor, which was shown in \cite[VI.9]{categories_maclane}. To show statement $(3)$ recall that $U$ being monadic implies that for a pair of morphisms $X \rightrightarrows Y$ which admit a common section $Y \rightarrow X$ the coequalizer $\coeq(X \rightrightarrows Y)$ exists in $\Comp$ and the morphism $\coeq(U(X) \rightrightarrows U(Y) ) \rightarrow U(\coeq(X \rightrightarrows Y))$ is an isomorphism if $\coeq(U(X) \rightrightarrows U(Y))$ exists in $\Set$ and the map $U(Y) \rightarrow \coeq(U(X) \rightrightarrows U(Y))$ admits a section. It is clear that $\coeq(U(X) \rightrightarrows U(Y))$ exists in $\Set$ as the category of sets admits all (small) colimits and its clear that the map $U(Y) \rightarrow \coeq(U(X) \rightrightarrows U(Y))$ admits a section as it is a surjective map of sets.
\end{proof}

\begin{example} Let $\cC$ be a category that admits fiber products, and let $X \rightarrow Y$ be a morphism. We claim that the pair of projection maps $X\times_Y X \rightrightarrows X$ admit a common section. Indeed the diagonal map $\Delta: X \rightarrow X \times_Y X$ provides a common section.
\end{example}

\begin{prop}\label{epimorphism_comp} Let $f: X \rightarrow Y$ be a morphism in $\Comp$. Then $f$ is an effective epimorphism if and only if $U(f)$ is surjective.
\end{prop}

\begin{proof} If $f$ is an effective epimorphism, by definition we have that the canonical map $\coeq(X \times_Y X \rightrightarrows X) \rightarrow Y$ is an isomorphism. Since the forgetful functor $U: \Comp \rightarrow \Set$ preserves reflective co-equalizers it follows that the canonical map
\begin{equation*}
	\coeq( U(X \times_Y X) \rightrightarrows U(X)) \rightarrow U(Y)
\end{equation*}
is an isomorphism. Moreover, since $U$ also preserves limits we have that $U(X \times_Y X) = U(X) \times_{U(Y)} U(X)$ proving that $U(f): U(X) \rightarrow U(Y)$ is an effective epimorphism in the category $\Set$, and so it is surjective.

Conversely, assume that $U(f): U(X) \rightarrow U(Y)$ is surjective, then it is an effective epimorphism in $\Set$, and so we have that the canonical map
\begin{equation*}
	\coeq( U(X) \times_{U(Y)} U(X) \rightrightarrows U(X) ) \rightarrow U(Y)
\end{equation*}
is an isomorphism, and as the pair of morphisms $X \times_Y X \rightrightarrows X$ admits a common section given by the diagonal $\Delta: X \rightarrow X \times_Y X$ it follows from the monadicity of $U$ that $\coeq(X \times_Y X \rightrightarrows X)$ exists in $\Comp$. We need to show that the canonical map $\coeq(X \times_Y X \rightrightarrows X) \rightarrow Y$ is an isomorphism, but since $U$ detects isomorphisms it suffices to show that
\begin{equation*}
	\coeq( U(X \times_Y X) \rightrightarrows U(X) ) \rightarrow U(Y)
\end{equation*}
is an isomorphism. But this follows from the previous assertion and the fact that $U$ preserves limits.
\end{proof}

\begin{example} Let $X$ be an object of $\Comp$, by adjunction we have a surjective map
\begin{equation*}
	\beta(U(X)) \longrightarrow X
\end{equation*}
from the Stone-Cech compactification of $X$ considered as a discrete set. Then the previous result implies that the canonical map
\begin{equation*}
	\coeq( \beta(U(X)) \times_X \beta(U(X)) \rightrightarrows \beta(U(X)) ) \longrightarrow X
\end{equation*}
is an isomorphism. 

Furthermore, let us mention that $\beta(U(X))$ is a profinite set \cite[Proposition 3.2.3]{lurie_ultracategories}, and that $\beta(U(X)) \times_X \beta(U(X))$ is also a profinite set. Indeed, since $\beta(U(X))$ is a profinite set we have that $\beta(U(X)) = \lim S_i$ where $S_i$ are finite sets, and since limits commute with limits we learn that
\begin{equation*}
	\beta(U(X)) \times_X \beta(U(X)) = \lim ( S_i \times_X S_i )
\end{equation*}
showing that $\beta(U(X)) \times_X \beta(U(X))$ is also a profinite set.
\end{example}

\begin{defn}[Regular Categories]\label{defn_regular_cat} Let $\cC$ be a category. We will say that $\cC$ is regular if the following conditions are satisfied
\begin{enumerate}[(R1)]
	\item The category $\cC$ admits finite limits.
	\item Every morphism $f: X \rightarrow Z$ in $\cC$ can be written as a composition $X\twoheadrightarrow Y \hookrightarrow Z$, where $X \twoheadrightarrow Y$ is an effective epimorphism and $Y \hookrightarrow Z$ is a monomorphism. 
	\item The collection of effective epimorphisms in $\cC$ is closed under pullbacks. That is, if we are given a pullback diagram
	\begin{cd}
		X^{\prime} \ar[r] \ar[d, "f^{\prime}"] & X \ar[d, "f"] \\
		Y^{\prime} \ar[r] & Y
	\end{cd}
	in $\cC$ where $f$ is an effective epimorphism, the morphism $f^{\prime}$ is also an effective epimorphism.
\end{enumerate}
\end{defn}

\begin{corollary} Let $\cC$ be a category which admits pullbacks and satisfies condition (R2) of the above definition. Then, for every morphism $f: X \rightarrow Z$ the factorization $X \twoheadrightarrow Y \hookrightarrow Z$ is functorial and unique up to unique isomorphism. From now on we will write $\im(f)$ instead of $Y$.
\end{corollary}

\begin{proof} Follows from A.1.4 and A.1.5 in \cite[Appendix A]{lurie_ultracategories}.
\end{proof}

\begin{prop}\label{comp_regular_cat} $\Comp$ is a regular category. Moreover, the category $\Comp$ admits all limits.
\end{prop}

\begin{proof} Let $\cI \rightarrow \Comp$ be diagram category sending $i \mapsto X_i$, we know from \cite[Tag 08ZT]{stacks-project} that the category of topological spaces has all limits and that the forgetful functor to sets commutes with limits. Therefore, it suffices to show that the topological space $Z: = \lim_{\cI} X_i$ is a compact Hausdorff space. Recall that $Z$ can be realized as the equalizer of a pair of morphisms $\prod X_i \rightrightarrows \prod X_i$, and that each $\prod X_i$ is a compact Hausdorff space by Tychonoff's theorem \cite[Tag 08ZU]{stacks-project}. Since the forgetful functor $U: \Top \rightarrow \Set$ commutes with all limits it follows that we can identify $U(Z)$ with a subset of $\prod U(X_i)$, and from the proof of \cite[08ZV]{stacks-project} we can deduce that $U(Z) \subset \prod X_i$ is a closed subset, showing that $Z$ is a compact Hausdorff space. This proves that $\Comp$ has all limits, and in particular that (R1) is satisfied.
	
Let $f: X \rightarrow Z$ be a morphisms in $\Comp$, then the map of sets $U(f)$ admits a unique factorization as $U(X) \twoheadrightarrow \im(U(f)) \hookrightarrow U(Z)$ as a surjective map followed by an injective map. We endow $\im(U(f))$ with the subspace topology, and we denote this topological space by $Y$. Since $f$ is a map of compact hausdorff spaces it follows that $\im(U(f)) \subset Z$ is a closed subset, showing that $Y$ is compact hausdorff. Moreover since $X \rightarrow Y$ is surjective it follows from \ref{epimorphism_comp} that it is an effective epimorphism. It is clear that $Y \rightarrow Z$ is a monomorphism, as it is injective. This finishes the proof of (R2).

Finally (R3) is a direct consequence of the fact that surjective maps of sets are closed under base-change and the forgetful functor $U: \Comp \rightarrow \Set$ preserves limits.
\end{proof}

\begin{prop}\label{mono_comp} Let $f: X \rightarrow Y$ be a morphism of compact hausdorff spaces. Then, the following are equivalent
\begin{enumerate}[(1)]
	\item The morphism $f: X \rightarrow Y$ is a monomorphism in $\Comp$.
	\item The morphism of sets $U(f): U(X) \rightarrow U(Y)$ is injective.
	\item The morphism $f: X \rightarrow Y$ determines an isomorphism between $X$ and $\im(f) \subset Y$ with the subspace topology.
	\item Let $g: Z \rightarrow Y$ be any morphism of compact hausdorff spaces such that $\im(g) \subset \im(f) \subset Y$, then there exists a unique factorization $Z \rightarrow X \rightarrow Y$ making the following diagram commute
	\begin{cd}
		Z \ar[r, "g"] \ar[rd, dashed] & Y \\
		& X \ar[u, "f", swap]
	\end{cd}
\end{enumerate}
\end{prop}

\begin{proof} $(1) \Rightarrow (2)$. Follows from the fact that the forgetful functor $U: \Comp \rightarrow \Set$ is a right adjoint and so it preserves monomorphisms. The claim follows from the fact that monomorphisms in $\Set$ are exactly the injective maps of sets.

$(2) \Rightarrow (3)$. From the proof of Proposition \ref{comp_regular_cat} we can conclude that $f$ admits a unique factorization $X \twoheadrightarrow \im(f) \hookrightarrow Y$, where $\im(f) \subset Y$ is endowed with the subspace topology. By assumption we know that $f: X \rightarrow Y$ is injective, so the surjective map $X \twoheadrightarrow \im(f)$ is also injective. The claim then follows from the fact that bijections of compact hausdorff spaces are isomorphisms (\ref{monadicity_comp_set}).

$(3) \Rightarrow (4)$. From the proof of \ref{comp_regular_cat} we know that $g: Z \rightarrow Y$ admits a unique factorization $Z \twoheadrightarrow \im(g) \hookrightarrow Y$, where $\im(g) \subset Y$ is endowed with the subspace topology. From the hypothesis of $(3)$ we know that the map $\im(g) \hookrightarrow Y$ factors uniquely as $\im(g) \rightarrow X \rightarrow Y$, proving the claim.

$(4) \Rightarrow (1)$. Follows directly from the definition of monomorphisms.
\end{proof}

\begin{defn}[Equivalence Relations]\label{defn_equiv_relation} Let $\cC$ be a category which admits finite limits and let $X$ be an object of $\cC$. We say that a monomorphism $R \hookrightarrow X \times X$ is an equivalence relation if, for every object $Y \in \cC$, the image of the induced map
\begin{equation*}
	\Hom_{\cC} (Y, R) \rightarrow \Hom_{\cC} (Y, X \times X) \simeq \Hom_{\cC} (Y, X) \times \Hom_{\cC} (Y, X)
\end{equation*}
is an equivalence relation on the set $\Hom_{\cC} (Y, X)$.
\end{defn}

\begin{example} Let $\cC$ be a category which admits finite limits and let $f: X \rightarrow Y$ be a morphism in $\cC$. Then the fiber product $X \times_Y X$ comes equipped with a canonical monomorphism $X \times_Y X \hookrightarrow X \times X$ which presents $X \times_Y X$ as an equivalence relation on $X$. To see that $X \times_Y X \rightarrow X \times X$ is a monomorphism, recall that this map fits into the pullback diagram
\begin{cd}
	X \times_Y X \ar[r] \ar[d]  & X \times X \ar[d] \\
	Y \ar[r, "\Delta"] & Y \times Y
\end{cd}
and since the diagonal map $\Delta: Y \rightarrow Y \times Y$ is always a monomorphism, and monomorphisms are stable under base change the claim follows.
\end{example}

\begin{defn}[Effective Equivalence Relations]\label{defn_eff_equiv_relation} Let $\cC$ be a category which admits finite limits and let $X$ be an object of $\cC$. We will say that an equivalence relation $R$ on $X$ is effective if there exists an effective epimorphism $f: X \twoheadrightarrow Y$ such that $R = X \times_Y X \hookrightarrow X \times X$.
\end{defn}

\begin{rem}\label{comp_equiv_relation_eff_epi} Let $\cC$ be a category which admits finite limits, let $X$ be an object of $\cC$, and let $R$ be an effective equivalence relation on $X$. Then there exists an effective epimorphism $f: X \twoheadrightarrow Y$ in $\cC$ such that $R = X \times_Y X$. The assumption that $f$ is an effective epimorphism implies that it exhibits $Y$ as the coequalizer of the diagram $R \rightrightarrows X$. In particular, $Y$ is determined (up to unique isomorphism) bu the equivalence relation $R$; we will emphasize this dependence by denoting $Y$ by $X/R$. It follows that the construction $R \mapsto X/R$ induced a bijection
\begin{equation*}
	\{ \text{Effective equivalence relations } R \hookrightarrow X \times X \} / \text{iso} \simeq
	\{ \text{Effective epimorphisms } X \twoheadrightarrow Y \} / \text{iso}
\end{equation*}
The inverse bijection carries an effective epimorphism $f: X \rightarrow Y$ to the equivalence relation $X \times_Y X$.
\end{rem}

\begin{prop}\label{all_equiv_eff_comp} Every equivalence relation on an object $X \in \Comp$ is effective.
\end{prop}

\begin{proof} Let $R \subset X \times X$ be an equivalence relation on $X$, we need to show that $\coeq(R \rightrightarrows X) =: Y$ exists and that it satisfies $R \simeq X \times_Y X$. Recall that the functor $U: \Comp \rightarrow \Set$ preserves equivalence relations, and since all equivalence relations on $\Set$ are effective it follows that $\coeq(U(R) \rightrightarrows U(X))=: S$ exists and it satisfies $U(R) \simeq U(X) \times_S U(X)$. Now, since the pair of maps $R \rightrightarrows X$ admits a common section given by the diagonal $\Delta: X \rightarrow R \subset X \times X$ it follows from \ref{monadicity_comp_set} that $\coeq(R \rightrightarrows X) =: Y$ exists an it satisfies $U(Y) = S$. It remains to show that the canonical map $R \rightarrow X \times_Y X$ is an equivalence. From the identity $U(R) \simeq U(X) \times_{U(Y)} U(X)$ it follows that $R \rightarrow X \times_Y X$ is an equivalence, as bijections of compact hausdorff spaces are homeomorphisms by \ref{monadicity_comp_set}.
\end{proof}

\begin{prop}\label{surj_are_effective_regular_cat} Let $\cC$ be a regular category, let $X$ be an object of $\cC$, and let $R \hookrightarrow X \times X$ be an equivalence relation on $X$. The following conditions are equivalent
\begin{enumerate}[(1)]
	\item The equivalence relation $R$ is effective.
	\item There exists a morphism $f: X \rightarrow Y$ such that $R = X \times_Y X \hookrightarrow X \times X$. Furthermore, the canonical map $X/R \rightarrow \im(f)$ is an isomorphism.
\end{enumerate}
\end{prop}

\begin{proof} \cite[Proposition A.2.5]{lurie_ultracategories}
\end{proof}

\begin{prop}\label{general_coeq_comp} Let $X \rightarrow Y$ and $Z \rightarrow X \times_Y X$ be a surjective maps of compact hausdorff spaces, and define the morphisms $Z \rightrightarrows X$ as the composition $Z \rightarrow X \times_Y X \rightrightarrows X$. Then, the canonical map
\begin{equation*}
	\coeq(Z \rightrightarrows X) \rightarrow Y
\end{equation*}
is an isomorphism of compact hausdorff spaces.
\end{prop}

\begin{proof} Since the morphism $X \rightarrow Y$ is surjective, it follows from \ref{epimorphism_comp} that the canonical map $\coeq(X \times_Y X \rightrightarrows X) \rightarrow Y$ is an isomorphism. In particular, the monomorphism $R:= X \times_Y X \hookrightarrow X \times X$ determines an effective equivalence relation on $X$ uniquely characterizing the identity $X/R \simeq Y$ (cf. \ref{comp_equiv_relation_eff_epi}). From the pair of maps $Z \rightrightarrows X$ we can produce a map $f: Z \rightarrow X \times X$, which admits a unique factorization as $Z \twoheadrightarrow \im(f) \hookrightarrow X \times X$, as $\Comp$ is a regular category (\ref{comp_regular_cat}). Thus, it remains to show that $\im(f) \subset X \times X$ determines an equivalence relation on $X$ equal to $R$. Indeed, by construction we have a surjective map $Z \twoheadrightarrow R$ and a monomorphism $R \hookrightarrow X$ whose composition is $f$, then the uniqueness of the factorization $Z \twoheadrightarrow \im(f) \hookrightarrow X$ shows that $\im(f) = R$ as desired.
\end{proof}

\subsection{Remarks on the topology of the spectrum}

Throughout this subsection $K$ is a non-trivially valued non-archimedean field.

\begin{prop}\label{uniform_berko_sp} Let $A$ be a Banach $K$-algebra, and $A \rightarrow A^u$ the uniformization map introduced in \ref{adjunction_uniform_ban}. Then, the induced map $|\cM(A^u)| \rightarrow |\cM(A)|$ is an isomorphism of compact hausdorff spaces.
\end{prop}

\begin{proof} From \ref{monadicity_comp_set} it suffices to show that the induced map $|\cM(A^u) \rightarrow |\cM(A)|$ of compact hausdorff spaces is a bijection. Hence, we need to show that any rank one valuation $x: A \rightarrow \RR_{\ge 0}$ admits an essentially unique factorization $A \rightarrow A^u \rightarrow \RR_{\ge 0}$. From the definition of completed residue fields (\ref{defn_completed_residue_field}) it follows that the map $x$ factors as $x: A \rightarrow \cH(x) \rightarrow \RR_{\ge 0}$, which in turn implies that the map $x$ factors as
\begin{equation*}
	A \rightarrow A^u \rightarrow \cH(x) \rightarrow \RR_{\ge 0}
\end{equation*}
since non-archimedean fields are uniform and the fact that the inclusion $\uBan_K \hookrightarrow \Ban_K$ admits a left adjoint given by the uniformization functor $(-)^u$ (cf. \ref{adjunction_uniform_ban}).
\end{proof}

\begin{prop}\label{berko_sp_disj_union} Let $\{A_i\}_{i \in I}$ be a finite collection of Banach $K$-algebras. Then, the canonical map
\begin{equation*}
	|\cM(\prod_{i \in I} A_i)| \longrightarrow \bigsqcup_{i \in I} |\cM(A_i)|
\end{equation*}
is an isomorphism of compact hausdorff spaces.
\end{prop}

\begin{proof} First, let us construct a map $|\cM(\prod A_i)| \rightarrow \sqcup |\cM(A_i)|$ which we will later show is a bijection, which will prove by \ref{monadicity_comp_set} that it is an isomorphism of compact hausdorff spaces. Recall that $\cM(\prod A_i)$ can be identified with the coproduct of $\{\cM(A_i)\}_{i \in I}$ in the category $\Ban_K^{\op}$, and that $\sqcup |\cM(A_i)|$ can be identified with the coproduct of $\{|\cM(A_i)|\}_{i \in I}$ in the category $\Comp$. The existence of the functor $|-|: \Ban_K^{\op} \rightarrow \Comp$ guarantees that we get a natural comparison map $|\cM(\prod A_i)| \rightarrow \sqcup |\cM(A_i)|$.

Its clear that the map $|\cM(\prod A_i)| \rightarrow \sqcup |\cM(A_i)|$ is surjective, by construction for each $i \in I$ we get a morphism $|\cM(A_i)| \rightarrow |\cM(\prod A_i)|$ such that when composed with $|\cM(\prod A_i)| \rightarrow \sqcup |\cM(A_i)|$ we get the identity map on $|\cM(A_i)|$. It remains to show injectivity, for this it suffices to show that every rank one valuation $x: \prod A_i \rightarrow \RR_{\ge 0}$ admits a factorization $\prod A_i \rightarrow A_i \rightarrow \RR_{\ge 0}$ through some $A_i$. Let $\frakp_x = \ker(x)$ be the prime ideal of $\prod A_i$ which is the kernel $x$, then the map $x$ admits a factorization as $\prod A_i \rightarrow (\prod A_i)/\frakp_x \rightarrow \RR_{\ge 0}$, but generalities on prime ideals on products of rings guarantee that the quotient map $\prod A_i \rightarrow (\prod A_i)/\frakp_x$ admits a factorization as $\prod A_i \rightarrow A_i \rightarrow (\prod A_i)/\frakp_x$ for some $A_i$.
\end{proof}

Recall from \ref{filtered_colim_banach_contr} that if $I$ is a filtered category and $I \rightarrow \Ban_K^{\contr}$ a functor indexing a collection of Banach $K$-algebras $\{A_i\}_{i \in I}$ with contractive transition maps, then $\colim_I A_i$ computed in $\Ban_K^{\contr}$ exists and it admits an explicit description.

\begin{prop}\label{cofiltered_lim_berko_sp} Let $I$ be a filtered category and $I \rightarrow \Ban_K^{\contr}$ a functor indexing a collection of Banach $K$-algebras $\{A_i\}_{i \in I}$ with contractive transition maps. Then, the canonical map
\begin{equation*}
	|\cM(\colim_I A_i) | \rightarrow \lim_{I^{\text{op}}} |\cM(A_i) |
\end{equation*}
where $\colim_I A_i$ is computed in $\Ban_K^{\contr}$, is an isomorphism of compact hausdorff spaces.
\end{prop}

\begin{proof} We follow our convention and denote by $\cM(B)$ the object of $\Ban_K^{\contr}$ which corresponds to the Banach $K$-algebra $B$. The functor $I \rightarrow \Ban_K^{\contr}$ gives rise to a functor $I^{\op} \rightarrow \Ban_K^{\contr, \op}$, and notice that $\cM(\colim_I A_i)$ can be identified with $\lim_{I^{\op}} \cM(A_i)$ computed in $\Ban_K^{\contr, \op}$. Then, the existence of the functor $|-|: \Ban_K^{\contr, \op} \rightarrow \Comp$ yields a canonical map $|\cM(\colim_I A_i)| \rightarrow \lim_{I^{\op}} |\cM(A_i)|$. By \ref{monadicity_comp_set} it suffices to show that this map is a bijection.
	
We first show injectivity. Consider the collection $\{\psi_i: A_i \rightarrow \colim_I A_i\}$ of natural maps coming from the definition of colimits, let $A := \cup_{i \in I} \im(\psi_i)$ be the union of all the images of $\psi_i$, and recall from \ref{filtered_colim_banach_contr} that $A \subset \colim_I A_i$ is dense. Then, if we have two distinct rank one valuations $v_1, v_2: \colim_I A_i \rightarrow \RR_{\ge 0}$ their restrictions to $A \subset \colim_I A_i$ gives us distinct maps $v_1, v_2: A \rightarrow \RR_{\ge 0}$. Therefore, we can conclude that there exists a $j \in I$ such that if $i > j$ (equivalently, there exists a map $j \rightarrow i$ in $I$) the rank one valuations $v_{1, i}, v_{2, i}: A_i \rightarrow A \rightarrow \RR_{\ge 0}$ will be distinct. In other words, for all $i > j$ the canonical maps $|\cM(\colim_I A_i)| \rightarrow |\cM(A_i)|$ send $v_1, v_2 \in |\cM(\colim_I A_i)|$ to distinct points of $|\cM(A_i)|$. Since the forgetful functor $\Comp \rightarrow \Set$ preserves limits (\ref{monadicity_comp_set}), it follows that $v_1, v_2$ will get mapped to distinct points under the map $|\cM(\colim_I A_i) | \rightarrow \lim_{I^{\text{op}}} |\cM(A_i) |$, proving injectivity.

To show surjectivity, we pick a point $x \in \lim_{I^{\text{op}}} |\cM(A_i) |$, which in turn gives us a collection of points $\{x_i \in |\cM(A_i)|\}$ such that for all map $j \rightarrow i$ in $I$, the induced map $|\cM(A_i)| \rightarrow |\cM(A_j)|$ satisfies $x_i \mapsto x_j$. In other words, we have a collection of rank one valuations $\{v_i: A_i \rightarrow \RR_{\ge 0}\}$ such that for any $i > j$ we have a commutative diagram
\begin{cd}
	A_j \ar[r] \ar[rd, "v_j", swap] & A_i \ar[d, "v_i"]\\
	& \RR_{\ge 0}
\end{cd}
This in turn induces a multiplicative map $v: A \rightarrow \RR_{\ge 0}$, and by endowing $A$ with the seminorm from \ref{filtered_colim_banach_contr} its easy to see that $v: A \rightarrow \RR_{\ge 0}$ satisfies $v(a) \le |a|_A$, making $v$ a rank one valuation on the seminormed ring $A$. Following the same recipe as in \ref{defn_completed_residue_field} we can construct a factorization of $v$ into a contractive map $A \rightarrow \cH(x)$ to a non-archimedean field and a rank one valuation $\cH(x) \rightarrow \RR_{\ge 0}$, then by Proposition $6(ii)$ of \cite[Section 1.1.7]{BGR} we get a factorization $A \rightarrow \colim_I A_i \rightarrow \cH(x)$, proving that we have a rank one valuation $v_A: \colim_I A_i \rightarrow \cH(x) \rightarrow \RR_{\ge 0}$ such that $v_A \circ \psi_i = v_i$. This proves surjectivity of the canonical map $|\cM(\colim_I A_i) | \rightarrow \lim_{I^{\text{op}}} |\cM(A_i) |$.

\end{proof}

\newpage

\section{Rational Domains}\label{sect_rational_domains}

\subsection{Definition and basic properties}

Let $K$ be a non-trivially valued non-archimedean field, and $\varpi \in K^{\times}$ a topological nilpotent unit.

\begin{defn} Let $A$ be a Banach $K$-algebra with norm $|-|_A$. Define $A\langle T_1, \dots, T_n \rangle$ as the set
\begin{equation*}
	A\langle T_1, \dots, T_n \rangle := \Big \{\sum_{\nu \in \ZZ_{\ge 0}^n} a_\nu T^\nu \in A [\![T_1, \dots, T_n]\!] \text{  such that } \lim_{|\nu| \rightarrow \infty} |a_\nu|_A \rightarrow 0 \Big \}
\end{equation*}
Then, Proposition 2 of \cite[Section 1.4.1]{BGR} shows that $A\langle T_1, \dots, T_n \rangle$ inherits an $A$-algebra structure from the inclusion $A\langle T_1, \dots, T_n \rangle \subset A [\![T_1, \dots, T_n]\!]$. Furthermore, Proposition 3 of \cite[Section 1.4.1]{BGR} shows that $A\langle T_1, \dots, T_n \rangle$ is a Banach $A$-algebra (in particular a Banach $K$-algebra) with norm
\begin{align*}
	|-| : A\langle T_1, \dots, T_n \rangle \rightarrow \RR_{\ge 0} && \sum_{\nu \in \ZZ_{\ge 0}^n} a_\nu T^{\nu} \mapsto \max_{\nu \in \ZZ_{\ge 0}^n} |a_\nu|_A
\end{align*}
and that the polynomials $A[T_1, \dots, T_n] \subset A\langle T_1, \dots, T_n \rangle$ are dense. Notice that its clear from the definition that $A \langle T_1 \rangle \langle T_2 \rangle \simeq A \langle T_1, T_2 \rangle$ are isometrically isomorphic.
\end{defn}

\begin{prop}\label{free_banach_algebras} Let $f: A \rightarrow B$ be a morphism of Banach $K$-algebras. Then, for any collection $\{b_1, \dots, b_n\} \subset B^{\circ} \subset B$ of power bounded elements of $B$ (cf. \ref{defn_powerbounded}) there exists a unique morphism of Banach $K$-algebras $\tilde{f}: A \langle T_1, \dots, T_n \rangle \rightarrow B$ sending $T_i \mapsto b_i$ making the following diagram commute
\begin{cd}
	A \ar[r, "f"] \ar[d, hook] & B \\
	A \langle T_1, \dots, T_n \rangle \ar[ru, dashed, "\tilde{f}", swap]
\end{cd}
Furthermore, if $A \rightarrow B$ is contractive and the collection $\{b_1, \dots, b_n\} \subset B_{\le 1} \subset B$, then the induced map $\tilde{f}: A \langle T_1, \dots, T_n \rangle \rightarrow B$ sending $T_i \mapsto b_i$ is contractive.
\end{prop}

\begin{proof} Follows from \cite[Proposition 1.4.3/1]{BGR} and its proof.
\end{proof}

\begin{prop} Let $A \rightarrow B$ be a contractive map of Banach $K$-algebras. Then, there is a natural isometric isomorphism
\begin{align*}
	B \cotimes_A A \langle T_1, \dots, T_n \rangle \rightarrow B \langle T_1, \dots, T_n \rangle && 1 \otimes T_i \mapsto T_i
\end{align*}
In particular, we have a natural isometric isomorphism
\begin{equation*}
	A \langle T_1, \dots, T_n \rangle \cotimes_A A \langle X_1, \dots, X_m \rangle \simeq A \langle T_1, \dots, T_n, X_1, \dots, X_m \rangle 
\end{equation*}
\end{prop}

\begin{proof} Since the maps $A \rightarrow B$ and $A \rightarrow A \langle T_1, \dots, T_n \rangle$ are contractive it follows from \ref{const_comp_tensor_ban} that $B \cotimes_A A \langle T_1, \dots, T_n \rangle$ exists in $\Ban_K$. Recall that the universal property of $B \cotimes_A A \langle T_1, \dots, T_n \rangle$ (cf. \ref{const_comp_tensor_ban}) says that a morphism $B \cotimes_A A \langle T_1, \dots, T_n \rangle \rightarrow Z$ in $\Ban_K$ is equivalent to specifying a pair of bounded $A$-algebra maps $B \rightarrow Z$ and $A \langle T_1, \dots, T_n \rangle \rightarrow Z$, which by \ref{free_banach_algebras} is equivalent to specifying a bounded $A$-algebra map $B \rightarrow Z$ and a collection of power bounded elements $\{z_1, \dots, z_n \} \subset Z$, which is in turn equivalent to specifying a map $B \langle T_1, \dots, T_n \rangle \rightarrow Z$ by \ref{free_banach_algebras}. This implies that the map $B \cotimes_A A \langle T_1, \dots, T_n \rangle \rightarrow B \langle T_1, \dots, T_n \rangle$ induced from the following pair of contractive $A$-algebra maps $A \langle T_1, \dots, T_n \rangle \rightarrow B \langle T_1, \dots, T_n \rangle$ and $B \rightarrow B \langle T_1, \dots, T_n \rangle$ is an contractive isomorphism, and it sends $1 \otimes T_i \mapsto T_i$.

By virtue of Proposition \ref{free_banach_algebras} there is a unique contractive map
\begin{align*}
	B \langle T_1, \dots, T_n \rangle  \rightarrow B \cotimes_A A \langle T_1, \dots, T_n \rangle && T_i \mapsto 1 \otimes T_i
\end{align*}
of Banach $K$-algebras, in particular the uniqueness implies that it must be the inverse of $B \cotimes_A A \langle T_1, \dots, T_n \rangle \rightarrow B \langle T_1, \dots, T_n \rangle$, proving the result.
\end{proof}

\begin{defn}\label{defn_rational_domain} Let $A$ be a Banach $K$-algebra, and $\{f_1, \dots, f_n \} \subset A$ a finite collection of elements that generate the unit ideal of $A$. For a fixed $f_i \in \{f_1, \dots, f_n \}$ let $I_{f_i}$ be the ideal of $A \langle T_1, \dots, T_n \rangle$ generated by
\begin{equation*}
	I_{f_i} := (f_i T_1 - f_1, \dots, f_i T_n - f_n) \subset A \langle T_1, \dots, T_n \rangle
\end{equation*}
We define a rational domain of $A$ to be given by the Banach $K$-algebra
\begin{equation*}
	A \blang \frac{f_1, \dots, f_n}{f_i} \brang := (A \langle T_1, \dots, T_n \rangle / I_{f_i})^{\wedge}
\end{equation*}
together with a morphism given by the composition
\begin{equation*}
	A \hookrightarrow A \langle T_1, \dots, T_n \rangle \twoheadrightarrow A \blang \frac{f_1, \dots, f_n}{f_i} \brang
\end{equation*}
We remark that in general the ideal $I_{f_i} \subset A \langle T_1, \dots, T_n \rangle$ need not be closed, thus we need to pass to the completion of the semi-normed ring $A \langle T_1, \dots, T_n \rangle / I_{f_i}$, endowed with the quotient semi-norm, to obtain a Banach $K$-algebra. Its clear from the construction that the morphism $A \rightarrow A \blang \frac{f_1, \dots, f_n}{f_i} \brang$ is contractive.

Its also easy to see that the rational domain $A \blang \frac{f_1, \dots, f_n}{f_i} \brang$ is non-zero. Indeed, it suffice to show that $I_{f_i} \subset \frakm \subset A \langle T_1, \dots, T_n \rangle$ for some maximal ideal $\frakm$, as maximal ideals of Banach $K$-algebras are automatically closed (\cite[1.2.4/5]{BGR}). But its clear that $1 \not\in I_{f_i}$, proving the claim.
\end{defn}

\begin{lemma}\label{unit_in_rational_domain} Let $A$ be a Banach $K$-algebra, and $A \rightarrow A \blang \frac{f_1, \dots, f_n}{f_i} \brang$ a rational domain of $A$. Then, $f_i \in A \blang \frac{f_1, \dots, f_n}{f_i} \brang$ is a unit.
\end{lemma}

\begin{proof} From the definition of rational domains it suffices to show that there exists a $X \in A \blang \frac{f_1, \dots, f_n}{f_i} \brang$ such that $f_iX -1 \in I_{f_i}$. As $\{f_1, \dots, f_n \} \subset A$ generate the unit ideal, there exists $\{h_1, \dots, h_n \} \subset A$ such that $\sum_j h_j f_j = 1$, we claim that $X = \sum_j h_j T_j$ will satisfy $f_i X - 1 \in I_{f_i}$. Indeed its easy to see that 
\begin{equation*}
	f_i X - 1 = \sum_j f_i h_j T_j - h_j f_j \in I_{f_i}
\end{equation*}
proving the claim.
\end{proof}

\begin{prop}\label{rat_domains_are_mono} Let $A$ be a Banach $K$-algebra, and $A \rightarrow A \blang \frac{f_1, \dots, f_n}{f_i} \brang$ a rational domain of $A$. Passing to the opposite category, the induced map
\begin{equation*}
	\cM \Big ( A \blang \frac{f_1, \dots, f_n}{f_i} \brang \Big ) \longrightarrow \cM(A)
\end{equation*}
is a monomorphism in $\Ban_K^{\op}$.
\end{prop}

\begin{proof} We need to show that for any morphism $h: A \rightarrow B$ of Banach $K$-algebras, if there exists a factorization of $h$ as
\begin{equation*}
	A \rightarrow A \blang \frac{f_1, \dots, f_n}{f_i} \brang \rightarrow B
\end{equation*}
then the factorization is unique. Assume that such a factorization exists, then from \ref{unit_in_rational_domain} we learn that $h(f_i) \in B$ is a unit, and that any $A$-algebra map $A \blang \frac{f_1, \dots, f_n}{f_i} \brang \rightarrow B$ must satisfy $f_j/f_i \mapsto h(f_j)/h(f_i)$. We claim that $h(f_j)/h(f_i) \in B$ is power-bounded. Indeed, by construction  we know that the quotient map $A \langle T_1, \dots, T_n \rangle \rightarrow A \blang \frac{f_1, \dots, f_n}{f_i} \brang$ is contractive and sends $T_j \mapsto f_j/f_i$, then the existence of the map $A \blang \frac{f_1, \dots, f_n}{f_i} \brang \rightarrow B$ guarantees that that $h(f_j)/h(f_i) \in B$ is power bounded. 

Then, from \ref{free_banach_algebras} we learn that there is a unique $A$-algebra morphism of Banach $K$-algebra
\begin{align*}
	A \langle T_1, \dots, T_n \rangle \rightarrow B && T_j \mapsto h(f_j)/h(f_i)
\end{align*}
as $I_{f_i}$ is in the kernel of this map, it follows from the definition of $A \blang \frac{f_1, \dots, f_n}{f_i} \brang$ that there is a unique map $A \blang \frac{f_1, \dots, f_n}{f_i} \brang \rightarrow B$ which factors
\begin{align*}
	A \langle T_1, \dots, T_n \rangle \rightarrow A \blang \frac{f_1, \dots, f_n}{f_i} \brang \rightarrow  B && T_j \mapsto h(f_j)/h(f_i)
\end{align*}
Showing that there is a unique $A$-algebra morphism of Banach $K$-algebras $A \blang \frac{f_1, \dots, f_n}{f_i} \brang \rightarrow  B$ satisfying $f_j/f_i \mapsto h(f_j)/h(f_i)$, proving the result.
\end{proof}

\begin{defn}\label{defn_top_rat_domain} Let $A$ be a Banach $K$-algebra and $\{f_1, \dots, f_n\} \subset A$ a finite collection of elements generating the unit ideal. Define
\begin{equation*}
	|\cM(A)| \Big ( \frac{f_1, \dots, f_n}{f_i} \Big ) := \{x \in |\cM(A)| \text{ such that } |f_j(x)| \le |f_i(x) | \text{ for all } 1 \le j \le n \} 
\end{equation*}
and endow it with the subspace topology $|\cM(A)|\Big ( \frac{f_1, \dots, f_n}{f_i} \Big ) \subset |\cM(A)|$. From the injective map of compact hausdorff spaces $|\cM(A)| \rightarrow \prod_{f \in A} [0, |f|_A]$ (cf. \ref{berko_sp_is_comp}), we learn that $|\cM(A)|\Big ( \frac{f_1, \dots, f_n}{f_i} \Big )$ is a closed subset of $|\cM(A)|$ as it can be realized as the intersection of $|\cM(A)| \subset \prod_{f \in A} [0, |f|_A]$ and the subset of points $(t_f)_{f \in A} \in \prod_{f \in A} [0, |f|_A]$ which satisfy $t_{f_j} \le t_{f_i}$. This shows that
\begin{equation*}
	|\cM(A)|\Big ( \frac{f_1, \dots, f_n}{f_i} \Big ) \hookrightarrow |\cM(A)|
\end{equation*}
is a monomorphism of compact hausdorff spaces (cf. \ref{mono_comp}).

Furthermore, we claim that $|f_i(x)| \not = 0$ for all $x \in |\cM(A)| \Big ( \frac{f_1, \dots, f_n}{f_i} \Big ) \subset |\cM(A)|$. Indeed, if $|f_i(x)| = 0$, then $|f_j(x)| = 0$ for all $1 \le j \le n$, and since $\{f_1, \dots, f_n\} \subset A$ generate the unit ideal this would imply that $|1(x)| = 0$, contradicting the definition of rank one valuations on $A$ (cf. \ref{defn_berko_spectrum}).
\end{defn}

\begin{prop}\label{rational_dom_univ_prop} Let $B$ be a uniform Banach $K$-algebra, and $h: A \rightarrow B$ a (necessarily contractive) morphism of Banach $K$-algebras. If the induced map $|\cM(B)| \rightarrow |\cM(A)|$ of compact hausdorff spaces has its image contained in $|\cM(A)|\Big ( \frac{f_1, \dots, f_n}{f_i} \Big ) \subset |\cM(A)|$, then the morphism $h: A \rightarrow B$ admits a unique factorization
\begin{equation*}
	A \longrightarrow A \blang \frac{f_1, \dots, f_n}{f_i} \brang \longrightarrow B
\end{equation*}
\end{prop}

\begin{proof} From \ref{rat_domains_are_mono} we learn that if $h$ admits a factorization as $A \longrightarrow A \blang \frac{f_1, \dots, f_n}{f_i} \brang \longrightarrow B$ then its must necessarily be unique. From the assumption that the image of $|\cM(B)| \rightarrow |\cM(A)|$ is contained in $|\cM(A)|\Big ( \frac{f_1, \dots, f_n}{f_i} \Big ) \subset |\cM(A)|$ we can deduce that
\begin{align*}
	|h(f_j)(x)| \le |h(f_i)(x)| \not= 0 && \text{for all } x \in |\cM(B)| \text{  and all } 1 \le j \le n
\end{align*}
This implies that $h(f_i) \in B$ is a unit by virtue of \ref{recog_unit_berko}, and that
\begin{align*}
	\Big | \frac{h(f_j)}{h(f_i)} (x) \Big| \le 1 && \text{for all } x \in |\cM(B)| \text{  and all } 1 \le j \le n
\end{align*}
by the multiplicativity of rank one valuations. Berkovich's maximum modulus principle \ref{berko_max_modulus} then implies that $\rho_B (h(f_j)/h(f_i)) \le 1$, which in turn implies that $h(f_j)/h(f_i) \in B$ is power bounded in $B$ by the assumption that $B$ is uniform.

From \ref{free_banach_algebras} we learn that there is a unique $A$-algebra map of Banach $K$-algebras
\begin{align*}
	A \langle T_1, \dots, T_n \rangle \rightarrow B && T_j \mapsto h(f_j)/h(f_i)
\end{align*}
and as the ideal $I_{f_i} \subset A \langle T_1, \dots, T_n \rangle$ (cf. \ref{defn_rational_domain}) is in the kernel of the map, we get an induced map
\begin{equation*}
	A \blang \frac{f_1, \dots, f_n}{f_i} \brang \longrightarrow B
\end{equation*}
as desired.
\end{proof}

\begin{example} Let us give an example showing that the hypothesis that $B$ is uniform in Proposition \ref{rational_dom_univ_prop} cannot be removed. This stands in contrast with what happens in Huber's theory of adic spaces where there is not uniformity assumption (cf. \cite[Proposition 1.3]{huber_rigid_varieties}). Recall the example from \ref{example_powerbdd_stric_sp_radius}: let $A$ be the Banach $\CC_p$-algebra defined as the completion of $\CC_p[T]$ with respect to the norm
\begin{align*}
	|-|: \CC_p[T] \rightarrow \RR_{\ge 0} && |\sum a_i T^i| = \max\{|a_i|(i+1) \}
\end{align*}
and whose spectral radius satisfies $\rho(T) = 1$. This implies by \ref{uniform_berko_sp} that all rank one valuations $x: A \rightarrow \RR_{\ge 0}$ satisfy $|T(x)| \le 1$, in particular the map of compact hausdorff spaces
\begin{equation*}
	|\cM(A)| \Big ( \frac{T}{1} \Big) \rightarrow |\cM(A)|
\end{equation*}
is an isomorphism. To complete the example it suffices to show that the identity map $A \rightarrow A$ does not admit a factorization
\begin{equation*}
	A \rightarrow A \blang \frac{T}{1} \brang \rightarrow A
\end{equation*}
Indeed, its clear from the construction that $T \in A \blang \frac{T}{1} \brang$ has norm $\le 1$, making it impossible to have a bounded map $A \blang \frac{T}{1} \brang \rightarrow A$ satisfying $T \mapsto T$.
\end{example}

\begin{prop}\label{mono_berko_sp_injective} Let $f: \cM(B) \rightarrow \cM(A)$ be a monomorphism in $\Ban_K^{\op}$. Then, the induced map $|f|: | \cM(B) | \rightarrow | \cM(A) |$ in $\Comp$ is injective; equivalently, $|f|$ is a monomorphism in $\Comp$ (cf. \ref{mono_comp}).
\end{prop}

\begin{proof} Assume that we have two points $y_1, y_2 \in |\cM(B)|$ such that $|f|(y_1) = |f|(y_2) = x \in |\cM(A)|$, we need to show that $y_1 = y_2$. Let $\cM(\cH(y_1)) \rightarrow \cM(B)$ and $\cM(\cH(y_2)) \rightarrow \cM(B)$ be the morphisms in $\Ban_K^{\op}$ corresponding to the points $y_1, y_2 \in |\cM(B)|$ as described in \ref{defn_completed_residue_field}. We claim that it suffices to show that there exists a non-archimedean field $\kappa$ together with a pair of morphisms $\cH(y_1) \rightarrow \kappa$ and $\cH(y_2) \rightarrow \kappa$ making the following diagram commute
\begin{cd}
	\cM(\kappa) \ar[r] \ar[d] & \cM(\cH(y_1)) \ar[d, "f"] \\
	\cM(\cH(y_2)) \ar[r, "f"] & \cM(A)
\end{cd}
Indeed, if such a diagram exists then the pair of morphisms
\begin{equation*}
	(\cM(\kappa) \rightarrow \cM(\cH(y_i)) \rightarrow \cM(B))_{i \in \{1, 2\}} \in \Maps_{\Ban_K^{\op}} (\cM(\kappa), \cM(B)) 
\end{equation*}
would be mapped to the same morphism
\begin{equation*}
	(\cM(\kappa) \rightarrow \cM(\cH(y_i)) \rightarrow \cM(A))_{i \in \{1, 2\} } \in \Maps_{\Ban_K^{\op}} (\cM(\kappa), \cM(A)) 
\end{equation*}
under the map $\Maps_{\Ban_K^{\op}} (\cM(\kappa), \cM(B)) \rightarrow \Maps_{\Ban_K^{\op}} (\cM(\kappa), \cM(A))$ induced by $f$. But since $f$ is a monomorphism we deduce that the pair of morphisms $(\cM(\kappa) \rightarrow \cM(\cH(y_i)) \rightarrow \cM(B))_{i \in \{1, 2\}}$ are in fact equal, showing that $y_1 = y_2 \in |\cM(B)|$.

Recall that $x \in |\cM(A)|$ was defined as $x := |f|(y_i)$, denote by $\cM(\cH(x)) \rightarrow \cM(A)$ the morphism in $\Ban_K^{\op}$ corresponding to the point $x \in |\cM(A)|$. From the universal property of completed residue fields (\ref{univ_completed_residue}) we learn that the maps $\cM(\cH(y_i)) \rightarrow \cM(A)$ admit a unique factorization as $\cM(\cH(y_i)) \rightarrow \cM(\cH(x)) \rightarrow \cM(A)$. We claim that it suffices to show that the completed tensor product $\cH(y_1) \cotimes_{\cH(x)} \cH(y_2) \not= 0$. Indeed, we know from \ref{const_comp_tensor_ban} that the following diagram is a pullback diagram in $\Ban_K^{\op}$
\begin{cd}
	\cM(\cH(y_1) \cotimes_{\cH(x)} \cH(y_2)) \ar[r] \ar[d] & \cM(\cH(y_1)) \ar[d] \\
	\cM(\cH(y_2)) \ar[r]  & \cM(\cH(x))
\end{cd}
It is clear that the completed tensor product $\cH(y_1) \cotimes_{\cH(x)} \cH(y_2)$ is defined as the morphisms $\cH(x) \rightarrow \cH(y_i)$ are contractive, since all bounded morphisms between uniform Banach $K$-algebras are contractive (\ref{morphisms_unif_ban_contractive}). Then, if $\cH(y_1) \cotimes_{\cH(x)} \cH(y_2) \not= 0$ we know from \ref{berko_sp_is_comp} that $\cM(\cH(y_1) \cotimes_{\cH(x)} \cH(y_2)) \not = \emptyset$, so by picking a point $z \in |\cM(\cH(y_1) \cotimes_{\cH(x)} \cH(y_2)) |$ we can produce a non-archimedean field $\cH(z)$ together with a morphism $\cM(z) \rightarrow \cM(\cH(y_1) \cotimes_{\cH(x)} \cH(y_2))$, which would provide the desired maps $\cM(z) \rightarrow \cM(\cH(y_i))$.

Finally, since the morphisms $\cH(y_i) \rightarrow \cH(x)$ considered as abstract field are faithfully flat we can conclude that the tensor product $\cH(y_1) \otimes_{\cH(x)} \cH(y_2) \not= 0$. Gruson's Theorem \cite[Theorem 3.2.1(iv)]{gruson_theorie_padic} guarantees that the canonical map
\begin{equation*}
	\cH(y_1) \otimes_{\cH(x)} \cH(y_2) \longrightarrow  \cH(y_1) \cotimes_{\cH(x)} \cH(y_2)
\end{equation*}
is injective. Proving that $\cH(y_1) \cotimes_{\cH(x)} \cH(y_2) \not= 0$ as desired.
\end{proof}

\begin{corollary}\label{topo_img_rat_domain} The the rational domain $A \rightarrow A \blang \frac{f_1, \dots, f_n}{f_i} \brang$ induces an injective morphism
\begin{equation*}
	|f|: \Big | \cM \Big ( A \blang \frac{f_1, \dots, f_n}{f_i} \brang \Big ) \Big | \longrightarrow | \cM(A) |
\end{equation*}
with image equal to $|\cM(A)| \Big ( \frac{f_1, \dots, f_n}{f_i} \Big )$. In particular, $|f|$ determines an homeomorphism $ \Big | \cM \Big ( A \blang \frac{f_1, \dots, f_n}{f_i} \brang \Big ) \Big | \simeq |\cM(A)| \Big ( \frac{f_1, \dots, f_n}{f_i} \Big )$ of compact hausdorff spaces (cf. \ref{mono_comp}).
\end{corollary} 

\begin{proof} We know from \ref{rat_domains_are_mono} that the morphism $f: \cM \Big ( A \blang \frac{f_1, \dots, f_n}{f_i} \brang \Big ) \longrightarrow \cM(A)$ is a monomorphism in $\Ban_K^{\op}$, and therefore $|f|$ is injective by \ref{mono_berko_sp_injective}. It remains to show that for every $x \in |\cM(A)| \Big ( \frac{f_1, \dots, f_n}{f_i} \Big ) \subset |\cM(A)|$, the induced map $A \rightarrow \cH(x)$ admits a factorization as
\begin{equation*}
	A \rightarrow A \blang \frac{f_1, \dots, f_n}{f_i} \brang \rightarrow \cH(x)
\end{equation*}
but this follows from \ref{rational_dom_univ_prop}.
\end{proof}

\begin{prop}\label{tensor_of_rat_domains} Let $A \rightarrow A \blang \frac{f_1, \dots, f_n}{f_i} \brang$ and $A \rightarrow A \blang \frac{g_1, \dots, g_m}{g_j} \brang$ be a pair of rational domains of the Banach $K$-algebra $A$. Then, the canonical map
\begin{equation*}
	A \rightarrow A \blang \frac{f_1, \dots, f_n}{f_i} \brang \cotimes_A A \blang \frac{g_1, \dots, g_m}{g_j} \brang
\end{equation*}
is isometrically isomorphic to the rational domain
\begin{equation*}
	A \rightarrow A \blang \frac{f_1g_1, \dots, f_k g_l, \dots, f_n g_m}{f_i g_j} \brang
\end{equation*}
where the numerator ranges over all pairs $f_k g_l$.
\end{prop}

\begin{proof} As rational domains $A \rightarrow A \blang \frac{f_1, \dots, f_n}{f_i} \brang$ are contractive morphisms, it follows that the tensor product $A \rightarrow A \blang \frac{f_1, \dots, f_n}{f_i} \brang \cotimes_A A \blang \frac{g_1, \dots, g_m}{g_j} \brang$ exists and the morphism is contractive. Furthermore, as all $f_k/f_i$ and $g_l/g_j$ have norms $\le 1$ on their respective rational domains, it follows that there exists a unique contractive map
\begin{align*}
	A \langle T_{1,1}, \dots, T_{k,l}, \dots, T_{n,m} \rangle \rightarrow A \blang \frac{f_1, \dots, f_n}{f_i} \brang \cotimes_A A \blang \frac{g_1, \dots, g_m}{g_j} \brang && T_{k,l} \mapsto \frac{f_k}{f_i} \otimes \frac{g_l}{g_j}
\end{align*}
and as the ideal 
\begin{equation*}
	I_{f_i g_j} := (f_i g_j T_{1,1} - f_1 g_1, \dots, f_i g_j T_{k,l} - f_k g_l, \dots, f_i g_j T_{n,m} - f_n g_m ) \subset A \langle T_{1,1}, \dots, T_{k,l}, \dots, T_{n,m} \rangle 
\end{equation*}
is contained in the kernel of the map, we get an induced contractive map
\begin{align*}
	A \blang \frac{f_1g_1, \dots, f_k g_l, \dots, f_n g_m}{f_i g_j} \brang \rightarrow A \blang \frac{f_1, \dots, f_n}{f_i} \brang \cotimes_A A \blang \frac{g_1, \dots, g_m}{g_j} \brang && \frac{f_k g_l}{f_i g_j} \mapsto \frac{f_k}{f_i} \otimes \frac{g_l}{g_j}
\end{align*}
On the other hand, its clear by construction that all $f_k/f_i$ and $g_l/g_j$ are elements of $A \blang \frac{f_1g_1, \dots, f_k g_l, \dots, f_n g_m}{f_i g_j} \brang$ and have norm $\le 1$. Hence, there exists a unique contractive map
\begin{align*}
	A \langle X_1, \dots, X_n \rangle \cotimes_A A \langle Y_1, \dots, Y_m\rangle \rightarrow A \blang \frac{f_1g_1, \dots, f_k g_l, \dots, f_n g_m}{f_i g_j} \brang && X_k \mapsto \frac{f_k}{f_i}, Y_l \mapsto \frac{g_l}{g_j}
\end{align*}
which contains the ideal $(I_{f_i}, I_{g_j})$ (see \ref{defn_rational_domain} for a definition), inducing a contractive map
\begin{align*}
	A \blang \frac{f_1, \dots, f_n}{f_i} \brang \cotimes_A A \blang \frac{g_1, \dots, g_m}{g_j} \brang \rightarrow A \blang \frac{f_1g_1, \dots, f_k g_l, \dots, f_n g_m}{f_i g_j} \brang && \frac{f_k}{f_i} \otimes \frac{g_l}{g_j} \mapsto \frac{f_k g_l}{f_i g_j}
\end{align*}
proving the desires isometric isomorphism.
\end{proof}

\subsection{Fiber products}

Let $K$ be a non-trivially valued non-archimedean field, and $\varpi \in K^{\times}$ a topological nilpotent unit.

\begin{prop}\label{mono_iso_on_residue_fields} Let $A \rightarrow B$ be a contractive morphism of Banach $K$-algebras inducing a monomorphism $f: \cM(B) \rightarrow \cM(A)$ in $\Ban_K^{\op}$. Fix $y \in |\cM(B)|$ and $x := |f|(y) \in |\cM(A)|$; then, the canonical map, induced from the universal property of completed residue field (\ref{univ_completed_residue}),
\begin{equation*}
	\cH(x) \longrightarrow \cH(y)
\end{equation*}
is an isomorphism of Banach $K$-algebras. Furthermore, the following pair of morphisms which factor $\cH(x) \rightarrow \cH(y)$ (whose construction we describe in the proof)
\begin{align*}
	\cH(x) \rightarrow \cH(x) \cotimes_A B && \cH(x) \cotimes_A B \rightarrow \cH(y)
\end{align*}
are isometric isomorphisms of Banach $K$-algebras.
\end{prop}

\begin{proof} Let us begin by explaining the strategy of the proof. Since the pair of maps $\cH(x) \leftarrow A \rightarrow B$ are contractive, it follows that $\cH(x) \cotimes_A B$ exists and its characterized as the pushout of $\cH(x) \leftarrow A \rightarrow B$ in $\Ban_K$ (\ref{const_comp_tensor_ban}). Hence, the universal property of completed residue field (\ref{univ_completed_residue}) together with the fact that $|f|(y) = x$ yield the following commutative diagram in $\Ban_K^{\op}$
\begin{cd}
	\cM(\cH(y)) \ar[rrd, bend left = 15] \ar[rdd, bend right = 15] \ar[rd, dashed] \\
	& \cM(\cH(x) \ar[r] \ar[d] \cotimes_A B) & \cM(B) \ar[d] \\
	& \cM(\cH(x)) \ar[r] & \cM(A)
\end{cd}
The proof consists of two main steps. First, as $\cH(y)$ is uniform it follows from \ref{adjunction_uniform_ban} that the morphism $\cH(x) \cotimes_A B \rightarrow \cH(y)$ factors through $(\cH(x) \cotimes_A B)^u \rightarrow \cH(y)$ which we will show is an isomorphism of Banach $K$-algebras, and since they are both uniform it is automatically isometric. The next step consists on showing that $\cH(x) \rightarrow \cH(x)\cotimes_A B$ is an isometric isomorphism of Banach $K$-algebras, proving in particular that $\cH(x) \cotimes_A B$ is itself uniform. This will complete the proof of the result.

First, we show that $(\cH(x) \cotimes_A B)^u \rightarrow \cH(y)$ is an isomorphism. Since $\cM(B) \rightarrow \cM(A)$ is a monomorphism in $\Ban_K^{\op}$, and monomorphisms are stable under base-change, it follows that $\cM(\cH(x) \cotimes_A B) \rightarrow \cM(\cH(x))$ is a monomorphism in $\Ban_K^{\op}$. In particular, \ref{mono_berko_sp_injective} implies that $|\cM(\cH(x) \cotimes_A B)| = \pt$, so \ref{recog_na_field} implies that $(\cH(x) \cotimes_A B)^u$ is a non-archimedean field. From the construction of $(\cH(x) \cotimes_A B)^u$ we learn that it enjoys the following universal property among non-archimedean fields: if $\kappa$ is a non-archimedean field together with a morphism $\cM(\kappa) \rightarrow \cM(B)$ such that $|f|(|\cM(\kappa)|) = x \in |\cM(A)|$ then there exists a unique factorization $\cM(\kappa) \rightarrow \cM((\cH(x) \cotimes_A B)^u) \rightarrow \cM(B)$; but this is exactly the same universal property that $\cH(y)$ enjoys (\ref{univ_completed_residue}) as we have the identity $|f|^{-1}(x) = y$, proving that the canonical map $(\cH(x) \cotimes_A B)^u \rightarrow \cH(y)$ is an isometric isomorphism.

It remains to show that the canonical map $\cH(x) \rightarrow \cH(x) \cotimes_A B =: B_x$ is an isometric isomorphism in $\Ban_K^{\op}$. Since monomorphisms are stable under base-change it follows that the canonical map $\cM(B_x) \rightarrow \cM(\cH(x))$ is a monomorphism in $\Ban_K^{\op}$. Then, since the map $\cH(x) \rightarrow B_x$ is contractive (\ref{const_comp_tensor_ban}) we learn that the tensor product $B_x \cotimes_{\cH(x)} B_x$ exists, and generalities on monomorphisms guarantee that the codiagonal map $B_x \cotimes_{\cH(x)} B_x \rightarrow B_x$ is an isometric isomorphism in $\Ban_K$. Furthermore, Gruson's Theorem \cite[Theorem 3.2.1(iv)]{gruson_theorie_padic} tells us that the morphism of abstract rings $B_x \otimes_{\cH(x)} B_x \hookrightarrow B_x \cotimes_{\cH(x)} B_x$ is injective; which in turn implies that the algebraically defined codiagonal map $B_x \otimes_{\cH(x)} B_x \rightarrow B_x$ is also injective. However, since the algebraic codiagonal map $B_x \otimes_{\cH(x)} B_x \rightarrow B_x$ is surjective (i.e. affine schemes are separated), it follows that $B_x \otimes_{\cH(x)} B_x \rightarrow B_x$ is an isomorphism of abstract rings.

To show that the contractive map $\varphi: \cH(x) \rightarrow B_x$ is an isomorphism of abstract rings, it suffices to check so after tensoring with $- \otimes_{\cH(x)} B_x$, as the map $\varphi: \cH(x) \rightarrow B_x$ is faithfully flat. Hence, we need to show that the map $\varphi \otimes \text{id}: \cH(x) \otimes_{\cH(x)} B_x \longrightarrow B_x \otimes_{\cH(x)} B_x$ is an isomorphism. Post-composing $\varphi \otimes \text{id}$ with the codiagonal map
\begin{equation*}
	 \text{id}: \cH(x) \otimes_{\cH(x)} B_x \longrightarrow B_x \otimes_{\cH(x)} B_x \longrightarrow B_x
\end{equation*}
recovers the identity map. This implies that $\varphi \otimes \text{id}$ is bijective, and so we have shown that the contractive map $\cH(x) \rightarrow B_x$ is an isomorphism of abstract rings, in particular $B_x$ is a field. Finally, we need to show that $\varphi: \cH(x) \rightarrow B_x$ is an isometry. Let $z: B_x \rightarrow \RR_{\ge 0}$ be a rank one valuation, which exists by \ref{berko_sp_is_comp}, and for the sake of contradiction assume there exists a $a \in \cH(x)$ such that $\rho_{\cH(x)} (a) > |\varphi(a)|_{B_x}$. Then, by construction we get that $\rho_{\cH(x)} (a) > |z(\varphi(a))|$, giving rise to a rank one valuation $\cH(x) \rightarrow B_x \rightarrow \RR_{\ge 0}$ which is distinct from $\rho_{\cH(x)}$. But this contradicts the fact that $|\cM(\cH(x))| = \pt$ (\ref{stone_cech_berko_sp}), proving that $\cH(x) \rightarrow B_x$ is an isometric isomorphism, as desired.
\end{proof}

\begin{prop}\label{general_fiber_prod_berko_sp} Let $A \rightarrow B$ and $A \rightarrow C$ be contractive morphisms of Banach $K$-algebras. Then, the canonical map of compact hausdorff spaces (whose construction we describe in the proof)
\begin{equation*}
	q: |\cM(B \cotimes_A C) | \longrightarrow |\cM(B)| \times_{|\cM(A)|} |\cM(C)|
\end{equation*}
is surjective.
\end{prop}

\begin{proof} Given that the morphisms $B \leftarrow A \rightarrow C$ are contractive, it follows that its pushout in $\Ban_K$ can be identified with $B \cotimes_A C$; equivalently, the fiber product of $\cM(B) \rightarrow \cM(A) \leftarrow \cM(C)$ computed in $\Ban_K^{\op}$ is $\cM(B \cotimes_A C)$. Then, the characterization of $\cM(B \cotimes_A C)$ as a fiber product together with the existence of the functor $|-|: \Ban_K^{\op} \rightarrow \Comp$ (\ref{berko_sp_functor}) gives rise to the map $q: |\cM(B \cotimes_A C) | \rightarrow |\cM(B)| \times_{|\cM(A)|} |\cM(C)|$.

Recall that since the forgetful functor $\Comp \rightarrow \Set$ preserves limits (\ref{monadicity_comp_set}), we can conclude that specifying a point in $|\cM(B)| \times_{|\cM(A)|} |\cM(C)|$ is equivalent to specifying a pair of point $x_b \in |\cM(B)|$ and $x_c \in |\cM(C)|$ which have the same image $x_a \in |\cM(A)|$. Then, the universal property of completed residue fields (\ref{univ_completed_residue}) guarantee that we have the following commutative diagram in $\Ban_K^{\op}$
\begin{cd}
	\cM(\cH(x_b)) \ar[r] \ar[d] & \cM(\cH(x_a)) \ar[d] & \cM(\cH(x_c)) \ar[l] \ar[d] \\
	\cM(B) \ar[r] & \cM(A) & \cM(C) \ar[l]
\end{cd}
Then, the interaction of fiber product with the functor $|-|: \Ban_K^{\op} \rightarrow \Comp$ gives rise to the following diagram of compact hausdorff spaces
\begin{cd}
	\vert \cM(\cH(x_b) \cotimes_{\cH(x_a)} \cH(x_c)) \vert \ar[r] \ar[d] &
	\vert \cM(\cH(x_b))\vert \times_{ \vert \cM(\cH(x_a)) \vert} \vert \cM(\cH(x_c))\vert = \pt \ar[d] \\
	\vert \cM(B \cotimes_A C) \vert \ar[r, "q"] & \vert \cM(B)\vert \times_{\vert\cM(A)\vert} \vert\cM(C)\vert
\end{cd}
Hence, it remains to show that $\cH(x_b) \cotimes_{\cH(x_a)} \cH(x_c) \not= 0$, as then $|\cM(\cH(x_b) \cotimes_{\cH(x_a)} \cH(x_c))| \not= \emptyset$ by \ref{berko_sp_is_comp}. Recall that Gruson's Theorem \cite[Theorem 3.2.1(iv)]{gruson_theorie_padic} tells us that the map $\cH(x_b) \otimes_{\cH(x_a)} \cH(x_c) \hookrightarrow \cH(x_b) \cotimes_{\cH(x_a)} \cH(x_c)$ is injective, and since $\cH(x_b) \otimes_{\cH(x_a)} \cH(x_c) \not= 0$ by faithful flatness of $\cH(x_a) \rightarrow \cH(x_b)$ the claim follows.
\end{proof}

\begin{corollary}\label{mono_fiber_prod_berko_sp} Let $A \rightarrow B$ and $A \rightarrow C$ be contractive morphisms of Banach $K$-algebras, and assume that the map $\cM(B) \rightarrow \cM(A)$ is a monomorphism in $\Ban_K^{\op}$. Then, the canonical map from \ref{general_fiber_prod_berko_sp}
\begin{equation*}
	q: |\cM(B \cotimes_A C) | \longrightarrow |\cM(B)| \times_{|\cM(A)|} |\cM(C)|
\end{equation*}
is an isomorphism of compact hausdorff spaces.
\end{corollary}

\begin{proof} Since we already proved that $q$ is surjective (\ref{general_fiber_prod_berko_sp}), it suffices to show that $q$ is injective by \ref{monadicity_comp_set}. Since monomorphisms are stable under base-change and the functor $|-|: \Ban_K^{\op} \rightarrow \Comp$ sends monomorphisms in $\Ban_K^{\op}$ to monomorphisms in $\Comp$ by \ref{mono_berko_sp_injective}, it follows that the map $|\cM(B \cotimes_A C)| \rightarrow |\cM(C)|$ is injective (cf. \ref{mono_comp}). Furthermore, as the map $|\cM(B \cotimes_A C)| \rightarrow |\cM(C)|$ factors as
\begin{equation*}
	|\cM(B \cotimes_A C)| \longrightarrow |\cM(B)| \times_{|\cM(A)|} |\cM(C)| \longrightarrow |\cM(C)|
\end{equation*}
we are able to conclude that $q$ is injective as desired.
\end{proof}

\begin{prop}\label{gluing_berko_sp} Let $\{A \rightarrow B_i\}_{i \in I}$ be a finite collection of contractive morphisms of Banach $K$-algebras, such that the induced map of compact hausdorff spaces
\begin{equation*}
	\bigsqcup_{i \in I} |\cM(B_i)| \rightarrow |\cM(A)|
\end{equation*}
is surjective. Then, the canonical map
\begin{equation*}
	\coeq \Big( \vert \cM( \prod_{(i,j) \in I \times I} B_i \cotimes_A B_j) | \rightrightarrows |\cM( \prod_{i \in I} B_i) | \Big ) \longrightarrow |\cM(A)|
\end{equation*}
is an isomorphism.
\end{prop}

\begin{proof} Recall from \ref{berko_sp_disj_union} that we have canonical isomorphisms of compact hausdorff spaces
\begin{align*}
	|\cM( \prod_{i \in I} B_i) | \simeq \bigsqcup_{i \in I} |\cM(B_i)| && |\cM(\prod_{(i,j) \in I \times I} B_i \cotimes_A B_j)| \simeq \bigsqcup_{(i,j) \in I \times I} |\cM(B_i \cotimes_A B_j)|
\end{align*}
Since by assumption the induced map $\sqcup |\cM(B_i)| \rightarrow |\cM(A)|$ is surjective, it follows from \ref{epimorphism_comp} that the canonical map
\begin{equation*}
	\coeq \Big ( \bigsqcup_{(i,j) \in I \times I} |\cM(B_i)| \times_{|\cM(A)|} |\cM(B_j)| \rightrightarrows \bigsqcup_{i \in I} |\cM(B_i)|   \Big ) \rightarrow |\cM(A)|
\end{equation*}
is an isomorphism. Thus, by \ref{general_coeq_comp} it suffices to show that the canonical map
\begin{equation*}
	|\cM(B_i \cotimes_A B_j)| \rightarrow |\cM(B_i)| \times_{|\cM(A)|} |\cM(B_j)|
\end{equation*}
is a surjective map of compact hausdorff spaces, but this follows from \ref{general_fiber_prod_berko_sp}.
\end{proof}

\subsection{Stalks}

Let $K$ be a non-trivially valued non-archimedean field, and $\varpi \in K^{\times}$ a topological nilpotent unit.

\begin{defn}\label{defn_stalks} Let $A$ be a Banach $K$-algebra, and $x$ a point of $|\cM(A)|$. We define the Banach $K$-algebra $\cO_x^{\wedge}$ as
\begin{align*}
	\cO_x^\wedge := \colim_{x \in |\cM(A_V)| \subset |\cM(A)|} A_V && \text{computed in } \Ban_K
\end{align*}
where the colimit ranges over all rational domains $A \rightarrow A_V$ which satisfy $x \in |\cM(A_V)| \subset |\cM(A)|$. Let us explain why the colimit exists in $\Ban_K$: as rational domains $A \rightarrow A_V$ are given by contractive morphisms, it suffices to show by \ref{filtered_colim_banach_contr} that the colimit is filtered. Indeed, if $A \rightarrow A_{V}$ and $A \rightarrow A_W$ are rational domains satisfying $x \in |\cM(A_V)| \cap |\cM(A_W)|$ then $A \rightarrow A_V \cotimes_A A_W$ is a rational domain satisfying $x \in |\cM(A_V \cotimes_A A_W)|$ by \ref{tensor_of_rat_domains} and \ref{mono_fiber_prod_berko_sp}.
\end{defn}

\begin{prop}\label{properties_stalk} Let $A$ be a Banach $K$-algebra, and $x$ a point of $|\cM(A)|$. Then,
\begin{enumerate}[(1)]
	\item The canonical map $\cM(\cO_x^\wedge) \rightarrow \cM(A)$ is a monomorphism in $\Ban_K^{\op}$.
	\item The induced injective map $|\cM(\cO_x^{\wedge})| \rightarrow |\cM(A)|$ identifies $|\cM(\cO_x^{\wedge})|$ with $x \in |\cM(A)|$. In particular, we have that $|\cM(\cO_x^{\wedge})| = \pt$.
	\item Let $B$ be a uniform Banach $K$-algebra and $A \rightarrow B$ a morphism of Banach $K$-algebras. If the induced map $|\cM(B)| \rightarrow |\cM(A)|$ has its image contained in $x \in |\cM(A)|$. Then, the morphism $A \rightarrow B$ admits a unique factorization
	\begin{equation*}
		A \rightarrow \cO_x^\wedge \rightarrow B
	\end{equation*}
	\item The morphism $\cO_x^\wedge \rightarrow \cH(x)$ of Banach $K$-algebras induced from the canonical map $A \rightarrow \cH(x)$ and $(3)$ induces an isomorphism on $\cO_x^{\wedge, u} \rightarrow \cH(x)$. In particular, this implies that if $A \rightarrow B$ is a morphism as in $(3)$ then there exists a unique factorization
	\begin{equation*}
		A \rightarrow \cH(x) \rightarrow B
	\end{equation*}
\end{enumerate}
\end{prop}

\begin{proof} In order to proof $(1)$ recall that every morphism $\cM(A_V) \rightarrow \cM(A)$ from a rational domain is a monomorphism (\ref{rat_domains_are_mono}), and since monomorphisms are stable under limits it follows that $\cM(\cO_x^\wedge) \rightarrow \cM(A)$ is a monomorphism in $\Ban_K^{\op}$. For $(2)$ recall that we have the identity
\begin{equation*}
	|\cM(\cO_x^\wedge)| \simeq \lim |\cM(A_V)| \rightarrow |\cM(A)|
\end{equation*}
from \ref{cofiltered_lim_berko_sp}, where the limit ranges over all rational domains $\cM(A_V) \rightarrow \cM(A)$ which satisfy $x \in |\cM(A_V)| \subset |\cM(A)|$. As the forgetful functor $\Comp \rightarrow \Set$ preserves limits (\ref{monadicity_comp_set}) it follows that $x \in |\cM(\cO_x^\wedge)| \subset |\cM(A)|$, it remains to show that for every other $y \in |\cM(A)|$ there exists a rational domain $\cM(A_W) \rightarrow \cM(A)$ containing $x$ satisfying $y \not\in |\cM(A_W)| \subset |\cM(A)|$. Notice that there exists an element $f \in A$ such that $|f(x)| \not= |f(y)|$, and without loss of generality assume that $|f(x)| < |f(y)|$. Since $K$ is non-trivially valued, it follows that there exists a $C \in |K^{\times}|^{\mathbf{Q}}$ such that $|f(x)| \le C$ and $|f(y)| > C$. Then, we can produce the following rational domains
\begin{align*}
	A \rightarrow A_1 := A \frac{\langle C^{-1} T \rangle }{(f - T)}^{\wedge} && A \rightarrow A_2 := A \frac{\langle C^{-1} T \rangle }{(fT - 1)}^{\wedge}
\end{align*}
were the injective map $|\cM(A_1)| \rightarrow |\cM(A)|$ has image $|\cM(A)| \Big ( \frac{f}{C} \Big )$, and $|\cM(A_2)| \rightarrow |\cM(A)|$ has image $|\cM(A)| \Big ( \frac{C}{f} \Big )$. Showing that $x \in |\cM(A_1)| \subset |\cM(A)|$ and $y \not\in |\cM(A_1)| \subset |\cM(A)|$, finishing the proof of $(2)$.

In order to proof $(3)$ it suffices to show that the factorization exists, as uniqueness follows from the fact that $\cM(\cO_x^\wedge) \rightarrow \cM(A)$ is a monomorphism, which we established in $(1)$. For any rational domain $\cM(A_V) \rightarrow \cM(A)$ whose image contains $x \in |\cM(A)|$, we know from \ref{rational_dom_univ_prop} that there exists a unique factorization $A \rightarrow A_V \rightarrow B$. Then, from the definition of $A \rightarrow \cO_x^{\wedge}$ as a colimit it we obtain the desired factorization $A \rightarrow \cO_x^{\wedge} \rightarrow B$.

Finally, we proof $(4)$. Recall that the canonical map $\cO_x^\wedge \rightarrow (\cO_x^{\wedge, u})$ defines an isomorphism of compact hausdorff spaces $|\cM(\cO_x^{\wedge, u})| \rightarrow |\cM(\cO_x^{\wedge})|$ (\ref{uniform_berko_sp}). Furthermore, it follows from \ref{recog_na_field} that $\cO_x^{\wedge, u}$ is a non-archimedean field, which still satisfies the universal property from $(3)$ by \ref{adjunction_uniform_ban}. In particular, this implies that for any non-archimedean field $\kappa$ together with a morphism $A \rightarrow \kappa$ which induces a map $|\cM(\kappa)| \rightarrow |\cM(A)|$ with image $x \in |\cM(A)|$ then there exists a unique factorization $A \rightarrow \cO_x^{\wedge, u} \rightarrow \kappa$. But this is the same universal property that $A \rightarrow \cH(x)$ enjoys (\ref{univ_completed_residue}), finishing the proof of $(4)$.
\end{proof}

\begin{corollary}\label{banach_fiber_at_pt} Let $A \rightarrow B$ be a contractive morphism of Banach $K$-algebras, and fix a point $x \in |\cM(A)|$. Then, the canonical morphism of compact hausdorff spaces
\begin{equation*}
	|\cM(\cH(x) \cotimes_A B )| \longrightarrow | \cM(\cH(x))|\times_{|\cM(A)|} |\cM(B)| \simeq |f|^{-1} (x)
\end{equation*}
is an isomorphism.
\end{corollary}

\begin{proof} From the construction its clear that the map $A \rightarrow \cO_x^{\wedge}$ is contractive and from \ref{properties_stalk} we know that map $\cM(\cO_x^{\wedge}) \rightarrow \cM(A)$ is a monomorphism in $\Ban_K^{\op}$. Therefore, \ref{mono_fiber_prod_berko_sp} implies that the canonical map
\begin{equation*}
	|\cM(\cO_x^{\wedge} \cotimes_A B)| \longrightarrow | \cM(\cO_x^{\wedge})|\times_{|\cM(A)|} |\cM(B)| \simeq |f|^{-1} (x)
\end{equation*}
is an isomorphism of compact hausdorff spaces. From the maps $\cM(\cH(x)) \rightarrow \cM(\cO_x^{\wedge}) \rightarrow \cM(A)$ we get an induced morphism
\begin{equation*}
	|\cM(\cH(x) \cotimes_A B)| \rightarrow |\cM(\cO_x^{\wedge} \cotimes_A B)|
\end{equation*}
we claim that it is an isomorphism of compact hausdorff spaces. Indeed, it suffices to check that the induced map $(\cO_x^{\wedge} \cotimes_A B)^u \rightarrow (\cH(x) \cotimes_A B)^u$ is an isomorphism by \ref{uniform_berko_sp}. Using the characterization of $- \cotimes_A -$ as pushouts (\ref{const_comp_tensor_ban}) and the fact that $(-)^u: \Ban_K \rightarrow \uBan_K$ is a left adjoint (\ref{adjunction_uniform_ban}) it follows that we have the identity $(\cO_x^{\wedge, u} \cotimes_A B)^u \simeq (\cO_x^{\wedge} \cotimes_A B)^u$, which proves the claim by virtue of \ref{properties_stalk}. The result then follows from the fact that the canonical map $\cO_x^{\wedge} \rightarrow \cH(x)$ induces an isomorphism $|\cM(\cH(x))| \rightarrow |\cM(\cO_x^{\wedge})|$ of compact hausdorff spaces. 
\end{proof}

\newpage

\section{Structure presheaf of Perfectoids}\label{sect_struct_presheaf_perfd}

\subsection{Rational domains}

Let $K$ be a perfectoid non-archimedean field and a topological nilpotent unit $\varpi \in K^{\times}$.

\begin{defn}\label{int_free_completed_alg} Let $A$ be an object of $\CAlg_{K_{\le 1}}^{\wedge \tf}$, define $A[T_1, \dots, T_n]^{\wedge}$ as the $\varpi$-completion of $A[T_1, \dots, T_n]$. This object enjoys the following universal property in $\CAlg_{K_{\le 1}}^{\wedge \tf}$: given a morphism $f: A \rightarrow B$ in $\CAlg_{K_{\le 1}}^{\wedge \tf}$ and a choice of objects $\{b_1, \dots, b_n \} \subset B$, there exists a unique morphism $\tilde{f}: A[T_1, \dots, T_n]^{\wedge} \rightarrow B$ sending $T_i \mapsto b_i$ making the following diagram commute
\begin{cd}
	A \ar[r, "f"] \ar[d, hook] & B \\
	A [T_1, \dots, T_n]^{\wedge} \ar[ru, "\tilde{f}", dashed, swap]
\end{cd}
\end{defn}

\begin{prop}\label{equiv_free_completed_alg_int_ban} Let $A \in \CAlg_{K_{\le 1}}^{\wedge \tf}$ be an object in the essential image of $H^0 j_*: \CAlg_{K_{\le 1}}^{\wedge a \tf} \rightarrow \CAlg_{K_{\le 1}}^{\wedge \tf}$ (cf. \ref{almost_funct_tf_alg}). Then,
\begin{enumerate}[(1)]
	\item The object $A[T_1, \dots, T_n]^{\wedge} \in \CAlg_{K_{\le 1}}^{\wedge \tf}$ is in the essential image of $H^0 j_*$.
	\item Under the equivalence of \ref{equiv_almost_tf_banach}, the object $A[T_1, \dots, T_n]^{\wedge} \in \CAlg_{K_{\le 1}}^{\wedge a \tf}$ corresponds to $A[\frac{1}{\varpi}] \langle T_1, \dots, T_n \rangle \in \Ban_K^{\contr}$ defined in \ref{free_banach_algebras}.
\end{enumerate}
\end{prop}

\begin{proof} In order to show $(1)$ it suffices to show that $(A[T_1, \dots, T_n]^{\wedge})_* = A[T_1, \dots, T_n]^{\wedge}$ (cf. \ref{almost_elements}). Let $f = \sum_{\nu \in \ZZ_{\ge 0}^n} a_\nu T^{\nu}$ be an object of $A[T_1, \dots, T_n]^{\wedge}[\frac{1}{\varpi}]$ and assume that $\varepsilon f \in A[T_1, \dots, T_n]^{\wedge}$ for all $\varepsilon \in (\varpi)_{\perfd}$. Then, $\varepsilon a_\nu \in A$ for all $\varepsilon$, implying that $a_\nu \in A$ by the assumptions on $A$, proving that $f \in A[T_1, \dots, T_n]^{\wedge}$ as desired. Statement $(2)$ follows from the equivalence \ref{equiv_almost_tf_banach} and the fact that $A[T_1, \dots, T_n]^{\wedge} \in \CAlg_{K_{\le 1}}^{\wedge a \tf}$ and $A[\frac{1}{\varpi}] \langle T_1, \dots, T_n \rangle \in \Ban_K^{\contr}$ have the same universal property (cf. \ref{int_free_completed_alg} and \ref{free_banach_algebras}).
\end{proof}

\begin{prop}\label{int_model_rat_domain} Let $A$ be a Banach $K$-algebra, and $\{f_1, \dots, f_n\} \subset A$ a collection of elements generating the unit ideal which have norm $\le 1$. Then, the rational domain $A \blang \frac{f_1, \dots, f_n}{f_i} \brang$ corresponds to
\begin{equation*}
	\Big ( A_{\le 1} [T_1, \dots, T_n]^{\wedge} / (f_i T_1 - f_1, \dots, f_i T_n - f_n) \Big )^{\wedge a \tf} \in \CAlg_{K_{\le 1}}^{\wedge a \tf}
\end{equation*}
under the equivalence \ref{equiv_almost_tf_banach}. Furthermore, as the rational domain $A \blang \frac{f_1, \dots, f_n}{f_i} \brang$ only depends on the ideal $(f_i T_1 - f_1, \dots, f_i T_n - f_n)$ one may always assume that $|f_j|_A \le 1$ for all $j$ by replacing $f_j \mapsto \varpi^n f_j$.
\end{prop}

\begin{proof} Let $h: K \langle T_1, \dots, T_n \rangle \rightarrow A \langle T_1, \dots, T_n \rangle$ be the unique morphism of Banach $K$-algebras satisfying $h(T_j) = f_i T_j - f_j$ (\ref{free_banach_algebras}), and $q: K \langle T_1, \dots, T_n \rangle \rightarrow K$ the unique map satisfying $q(T_j) = 0$. As $h$ and $q$ are contractive maps we learn that the rational domain $A \blang \frac{f_1, \dots, f_n}{f_i} \brang$ can be presented as the pushout of the following diagram
\begin{cd}
	K \langle T_1, \dots, T_n \rangle \ar[r, "h"] \ar[d, "q"] & A \langle T_1, \dots, T_n \rangle \ar[d] \\
	K \ar[r] & A \blang \frac{f_1, \dots, f_n}{f_i} \brang
\end{cd}
On the other hand, under \ref{equiv_free_completed_alg_int_ban} we learn that the morphisms $h$ and $q$ induce morphisms $h_{\le 1}: K_{\le 1} [T_1, \dots, T_n]^{\wedge} \rightarrow A_{\le 1}[T_1, \dots, T_n]^{\wedge}$ and $q_{\le 1}: K_{\le 1} [T_1, \dots, T_n]^{\wedge} \rightarrow K_{\le 1}$ in $\CAlg_{K_{\le 1}}^{\wedge \tf a}$ under the equivalence \ref{equiv_almost_tf_banach}. Then, \ref{pushout_in_comp-tf-a} implies that the following diagram in $\CAlg_{K_{\le 1}}^{\wedge a \tf}$ is a pushout diagram
\begin{cd}
	K_{\le 1} [T_1, \dots, T_n]^{\wedge} \ar[r, "h_{\le 1}"] \ar[d, "q_{\le 1}"] & A_{\le 1}[T_1, \dots, T_n]^{\wedge} \ar[d]\\
	K_{\le 1} \ar[r] & \Big ( A_{\le 1} [T_1, \dots, T_n]^{\wedge} / (f_i T_1 - f_1, \dots, f_i T_n - f_n) \Big )^{\wedge a \tf}
\end{cd}
Finally, the equivalence \ref{equiv_almost_tf_banach} proves the result.
\end{proof}

\begin{defn} Let $A$ be an object of $\CAlg_{K_{\le 1}}^{\wedge \tf}$, define $A[T_1^{1/p^\infty}, \dots, T_n^{1/p^\infty}]$ as the following colimit computed in $\CAlg_{K_{\le 1}}$ (equivalently, in $\CAlg_{K_{\le 1}}^{\tf}$)
\begin{cd}
	A[T_1^{1/p^\infty}, \dots, T_n^{1/p^\infty}] := \colim_{\ZZ_{\ge 0}} \Big ( A[T_1, \dots, T_n] \ar[r, "\varphi"] &  A[T_1, \dots, T_n] \ar[r, "\varphi"] & \cdots  \Big )
\end{cd}
where $\varphi: A[T_1, \dots, T_n] \rightarrow A[T_1, \dots, T_n]$ is the unique $A$-algebra map sending $T_i \mapsto T_i^p$. This object satisfies the following universal property in $\CAlg_{K_{\le 1}}$: given a morphism $f: A \rightarrow B$ in $\CAlg_{K_{\le 1}}$ and a choice of objects $\{b_1, \dots, b_n\} \subset B$ together with a choice of compatible $p$-power roots for each $b_i$, there exists a unique morphism $\tilde{f}: A[T_1^{1/p^\infty}, \dots, T_n^{1/p^\infty}] \rightarrow B$ sending $T_i^{1/p^n} \mapsto b_i^{1/p^n}$ making the following diagram commute
\begin{cd}
	A \ar[r, "f"] \ar[d, hook] & B \\
	A [T_1^{1/p^\infty}, \dots, T_n^{1/p^\infty}] \ar[ru, "\tilde{f}", dashed, swap]
\end{cd}
\end{defn}

\begin{defn}\label{int_free_perfect_alg} Let $A$ be an object of $\CAlg_{K_{\le 1}}^{\wedge \tf}$, define $A[T_1^{1/p^\infty}, \dots, T_n^{1/p^\infty}]^{\wedge}$ as the $\varpi$-completion of $A[T_1^{1/p^\infty}, \dots, T_n^{1/p^\infty}]$. Given that the fully faithful functor $\CAlg_{K_{\le 1}}^{\wedge \tf} \subset \CAlg_{K_{\le 1}}^{\tf}$ admits a left adjoint given by $\varpi$-completion (\ref{torsion_free_algebras}), there is a canonical identification of $A[T_1^{1/p^\infty}, \dots, T_n^{1/p^\infty}]^{\wedge}$ with the following colimit computed in $\CAlg_{K_{\le 1}}^{\wedge \tf}$ (equivalently, in $\CAlg_{K_{\le 1}}^{\wedge}$)
\begin{cd}
	A[T_1^{1/p^\infty}, \dots, T_n^{1/p^\infty}]^{\wedge} := \colim_{\ZZ_{\ge 0}} \Big ( A[T_1, \dots, T_n]^{\wedge} \ar[r, "\varphi"] &  A[T_1, \dots, T_n] \ar[r, "\varphi"]^{\wedge} & \cdots  \Big )
\end{cd}
where $\varphi$ is determined by the rule $T_i \mapsto T_i^p$. Again, by virtue of the fact that $\varpi$-completion is left adjoint to the inclusion $\CAlg_{K_{\le 1}}^{\wedge \tf} \subset \CAlg_{K_{\le 1}}^{\tf}$ it follows that $A[T_1^{1/p^\infty}, \dots, T_n^{1/p^\infty}]^{\wedge}$ admits the following universal property in $\CAlg_{K_{\le 1}}^{\wedge \tf}$: given morphism $f: A \rightarrow B$ in $\CAlg_{K_{\le 1}}^{\wedge \tf}$ and a choice of objects $\{b_1, \dots, b_n\} \subset B$ together with a choice of compatible $p$-power roots for each $b_i$, there exists a unique morphism $\tilde{f}: A[T_1^{1/p^\infty}, \dots, T_n^{1/p^\infty}]^{\wedge} \rightarrow B$ sending $T_i^{1/p^n} \mapsto b_i^{1/p^n}$ making the following diagram commute
\begin{cd}
	A \ar[r, "f"] \ar[d, hook] & B \\
	A [T_1^{1/p^\infty}, \dots, T_n^{1/p^\infty}]^{\wedge} \ar[ru, "\tilde{f}", dashed, swap]
\end{cd}
Finally, let us remark that the derived and classical $\varpi$-completion of $A[T_1^{1/p^\infty}, \dots, T_n^{1/p^\infty}]$ agree as it is a $\varpi$-torsion free algebra.
\end{defn}

\begin{prop}\label{free_alg_perfect_is_perfectoid} Let $A$ be an object of $\Perfd_{K_{\le 1}}^{\Prism a}$ (\ref{defn_various_perfectoid}), and $(\A_{\inf}(A), d)$ its corresponding prefect prism under the equivalence \ref{perfect_prism_perfectoid}. Then,
\begin{enumerate}[(1)]
	\item $A[T_1^{1/p^\infty}, \dots, T_n^{1/p^\infty}]^{\wedge}$ is an object of $\Perfd_{K_{\le 1}}^{\Prism a}$.
	\item The derived and classical $(p, [\varpi^{\flat}])$-completion of $\A_{\inf}(A)[T_1^{1/p^\infty}, \dots, T_n^{1/p^\infty}]$ agree. We denote it by $\A_{\inf}(A)[T_1^{1/p^\infty}, \dots, T_n^{1/p^\infty}]^{\wedge}$.
	\item $\A_{\inf}(A)[T_1^{1/p^\infty}, \dots, T_n^{1/p^\infty}]^{\wedge}$ is a perfect $(p, [\varpi^\flat])$-complete $\delta$-ring, and
	\begin{align*}
		\varphi: \A_{\inf}(A)[T_1^{1/p^\infty}, \dots, T_n^{1/p^\infty}]^{\wedge} \longrightarrow \A_{\inf}(A)[T_1^{1/p^\infty}, \dots, T_n^{1/p^\infty}]^{\wedge} && \varphi(T_i^{1/p^k}) = T_i^{1/p^{k-1}}
	\end{align*}
	its Frobenius lift.
	\item The corresponding perfect prism of $A[T_1^{1/p^\infty}, \dots, T_n^{1/p^\infty}]^{\wedge}$ is $(\A_{\inf}(A)[T_1^{1/p^\infty}, \dots, T_n^{1/p^\infty}]^{\wedge}, d)$.
\end{enumerate}	
\end{prop}

\begin{proof} We begin by proving $(2)$. As $\A_{\inf}(A)$ is $p$-torsion free so is $\A_{\inf}(A)[T_1^{1/p^\infty}, \dots, T_n^{1/p^\infty}]$, showing that the derived and classical $p$-completion agree, we denote by by $\A_{\inf}(A)[T_1^{1/p^\infty}, \dots, T_n^{1/p^\infty}]^{\wedge}_{(p)}$. As $p$-completion preserves $p$-torsion freeness (\ref{completion_tf}), and since $\A_{\inf}(A)[T_1^{1/p^\infty}, \dots, T_n^{1/p^\infty}]^{\wedge}_{(p)}/p = A^{\flat}[T_1^{1/p^\infty}, \dots, T_n^{1/p^\infty}]$ is a perfect algebra of characteristic $p$, we learn from \ref{perfect_delta_rings_equiv} that $\A_{\inf}(A)[T_1^{1/p^\infty}, \dots, T_n^{1/p^\infty}]^{\wedge}_{(p)}$ is a perfect $\delta$-ring. Furthermore, as $T_i^{1/p^k}$ are elements of rank one (\ref{frob_perfect_element}) we can conclude that the Frobenius lift on $\A_{\inf}(A)[T_1^{1/p^\infty}, \dots, T_n^{1/p^\infty}]^{\wedge}_{(p)}$ is given by the Frobenius lift on $\A_{\inf}(A)$ and $T_i^{1/p^k} \mapsto T_i^{1/p^{k-1}}$. As $\A_{\inf}(A)[T_1^{1/p^\infty}, \dots, T_n^{1/p^\infty}]^{\wedge}_{(p)}$ is a perfect $\delta$-ring, it follows that it has bounded $[\varpi^{\flat}]$-torsion, showing that its derived and classical $[\varpi^{\flat}]$-completion agree. This completes the proof of $(2)$, and the description of the Frobenius lift on $\A_{\inf}(A)[T_1^{1/p^\infty}, \dots, T_n^{1/p^\infty}]^{\wedge}_{(p)}$ implies $(3)$.

Finally, as $d$ is a distinguished element it follows that $(\A_{\inf}(A)[T_1^{1/p^\infty}, \dots, T_n^{1/p^\infty}]^{\wedge}, d)$ is a perfect prism, in particular $d$ is a non-zero divisor. And as $\A_{\inf}(A)[T_1^{1/p^\infty}, \dots, T_n^{1/p^\infty}]/d = A[T_1^{1/p^\infty}, \dots, T_n^{1/p^\infty}]$, generalities on derived completions imply that $\A_{\inf}(A)[T_1^{1/p^\infty}, \dots, T_n^{1/p^\infty}]^{\wedge}/d = A[T_1^{1/p^\infty}, \dots, T_n^{1/p^\infty}]^{\wedge}$. Finishing the proof of $(1)$ and $(4)$.
\end{proof}

\begin{defn}\label{free_perfect_banach_alg} Let $A$ be a Banach $K$-algebra with norm $|-|_A$. Define the morphism $\varphi: A \langle T_1, \dots, T_n \rangle \rightarrow A \langle T_1, \dots, T_n \rangle$ as the unique contractive morphism of $A$-algebras satisfying $T_i \mapsto T_i^p$ (cf. \ref{free_banach_algebras}). We define $A \langle T_1^{1/p^\infty}, \dots, T_n^{1/p^{\infty}} \rangle$ as the colimit computed in $\Ban_K^{\contr}$
\begin{cd}
	A \langle T_1^{1/p^\infty}, \dots, T_n^{1/p^{\infty}} \rangle := \colim_{\ZZ_{\ge 0}} \Big ( A \langle T_1, \dots, T_n \rangle \ar[r, "\varphi"] & A \langle T_1, \dots, T_n \rangle \ar[r, "\varphi"] & \cdots \Big )^{\wedge}
\end{cd}
which is guaranteed to exist as all transition maps are contractive (\ref{filtered_colim_banach_contr}). As $\varphi$ is an isometry we learn that the norm on $A \langle T_1^{1/p^\infty}, \dots, T_n^{1/p^{\infty}} \rangle$ is given by
\begin{align*}
	|-|: A \langle T_1^{1/p^\infty}, \dots, T_n^{1/p^{\infty}} \rangle \rightarrow \RR_{\ge 0} && \sum_{\nu \in \ZZ[\frac{1}{p}]_{\ge 0}^n} a_\nu T^{\nu} \mapsto \max_{\nu \in \ZZ[\frac{1}{p}]_{\ge 0}^n} |a_\nu|_A
\end{align*}
We remark that the morphism $A \rightarrow A \langle T_1^{1/p^\infty}, \dots, T_n^{1/p^{\infty}} \rangle$ enjoys the following universal property in $\Ban_K^{\contr}$: given a contractive morphism $A \rightarrow B$ and a collection of objects $\{b_1, \dots, b_n\} \subset B_{\le 1} \subset B$ together with a choice of compatible $p$-power roots $b_i^{1/p^n} \in B_{\le 1}$ for each $b_i$, there exists a unique contractive morphism $\tilde{f}: A \langle T_1^{1/p^\infty}, \dots, T_n^{1/p^{\infty}} \rangle \rightarrow B$ sending $T_i^{1/p^n} \mapsto b_i^{1/p^n}$ making the following diagram commute
\begin{cd}
	A \ar[r, "f"] \ar[d, hook] & B \\
	A \langle T_1^{1/p^\infty}, \dots, T_n^{1/p^{\infty}} \rangle \ar[ru, "\tilde{f}", dashed, swap]
\end{cd}
This universal property is a consequence of the definition of $A \langle T_1^{1/p^\infty}, \dots, T_n^{1/p^{\infty}} \rangle$ together with \ref{free_banach_algebras}.

More generally, given a bounded morphism $A \rightarrow B$ of Banach $K$-algebras, a constant $C > 0$, and a collection of objects $\{b_1, \dots, b_n\} \subset B^{\circ} \subset B$ together with a choice of compatible $p$-power roots $b_i^{1/p^n} \in B_{\le C}$ for each $b_i$, there exists a unique bounded morphism $\tilde{f}:  A \langle T_1^{1/p^\infty}, \dots, T_n^{1/p^{\infty}} \rangle \rightarrow B$ sending $T_i^{1/p^n} \mapsto b_i^{1/p^n}$ making the following diagram commute
\begin{cd}
	A \ar[r, "f"] \ar[d, hook] & B \\
	A \langle T_1^{1/p^\infty}, \dots, T_n^{1/p^{\infty}} \rangle \ar[ru, "\tilde{f}", dashed, swap]
\end{cd}
This universal property is a consequence of the definition of $A \langle T_1^{1/p^\infty}, \dots, T_n^{1/p^{\infty}} \rangle$ together with \ref{free_banach_algebras} and \ref{filtered_colim_banach_contr}.
\end{defn}

\begin{prop}\label{equiv_free_perfect_completed_alg_int_ban} Let $A \in \CAlg_{K_{\le 1}}^{\wedge \tf}$ be an object in the essential image of $H^0 j_*: \CAlg_{K_{\le 1}}^{\wedge a \tf} \rightarrow \CAlg_{K_{\le 1}}^{\wedge \tf}$ (cf. \ref{almost_funct_tf_alg}). Then,
\begin{enumerate}[(1)]
	\item The object $A[T_1^{1/p^\infty}, \dots, T_n^{1/p^\infty}]^{\wedge} \in \CAlg_{K_{\le 1}}^{\wedge \tf}$ is in the essential image of $H^0 j_*$.
	\item Under the equivalence of \ref{equiv_almost_tf_banach}, the object $A[T_1^{1/p^\infty}, \dots, T_n^{1/p^\infty}]^{\wedge} \in \CAlg_{K_{\le 1}}^{\wedge a \tf}$ corresponds to $A[\frac{1}{\varpi}]\langle T_1^{1/p^\infty}, \dots, T_n^{1/p^{\infty}} \rangle \in \Ban_K^{\contr}$.
\end{enumerate}
\end{prop}

\begin{proof} In order to show $(1)$ it suffices to show that $(A[T_1^{1/p^\infty}, \dots, T_n^{1/p^\infty}]^{\wedge})_* = A[T_1^{1/p^\infty}, \dots, T_n^{1/p^\infty}]^{\wedge}$ (cf. \ref{almost_elements}). Let $f = \sum_{\nu \in \ZZ[\frac{1}{p}]_{\ge 0}^n} a_\nu T^{\nu}$ be an object of $A[T_1^{1/p^\infty}, \dots, T_n^{1/p^\infty}]^{\wedge}[\frac{1}{\varpi}]$ and assume that $\varepsilon f \in A[T_1^{1/p^\infty}, \dots, T_n^{1/p^\infty}]^{\wedge}$ for all $\varepsilon \in (\varpi)_{\perfd}$. Then, $\varepsilon a_\nu \in A$ for all $\varepsilon$, implying that $a_\nu \in A$ by the assumptions on $A$, proving that $f \in A[T_1^{1/p^\infty}, \dots, T_n^{1/p^\infty}]^{\wedge}$ as desired. Statement $(2)$ follows from the equivalence \ref{equiv_almost_tf_banach} and the fact that $A[T_1^{1/p^\infty}, \dots, T_n^{1/p^\infty}]^{\wedge} \in \CAlg_{K_{\le 1}}^{\wedge a \tf}$ and $A[\frac{1}{\varpi}] \langle T_1^{1/p^\infty}, \dots, T_n^{1/p^\infty} \rangle \in \Ban_K^{\contr}$ have the same universal property (cf. \ref{int_free_perfect_alg} and \ref{free_perfect_banach_alg}).
\end{proof}

\begin{defn}\label{defn_perfect_rat_domains} Let $A$ be a Banach $K$-algebra, and $\{f_1, \dots, f_n\} \subset A$ a collection of elements generating the unit ideal, together with a choice of compatible $p$-power roots $f_i^{1/p^m} \in A$ for each $f_i$. Let $I_{f_i} \subset A \langle T_1, \dots, T_n \rangle$ be the ideal generated by $(f_i T_1 - f_1, \dots, f_i T_n - f_n)$, and $I_{f_i}^{1/p^m} \subset A \langle T_1^{1/p^\infty}, \dots, T_n^{1/p^\infty} \rangle$ the following ideal
\begin{align*}
	I_{f_i}^{1/p^m} := (f_i^{1/p^m} T_1^{1/p^m} - f_1^{1/p^m}, \dots, f_i^{1/p^m} T_n^{1/p^m} - f_n^{1/p^m}) && I_{f_i}^{1/p^\infty} = \bigcup_{m \in \ZZ_{\ge 0}} I_{f_i}^{1/p^m}
\end{align*}
Notice that the ideal $I_{f_i}^{1/p^\infty}$ admits an exhaustive increasing filtration
\begin{equation*}
	I_{f_i} \subset I_{f_i}^{1/p} \subset \cdots \subset I_{f_i}^{1/p^m} \subset \cdots \subset I_{f_i}^{1/p^\infty}
\end{equation*}
We define the perfected rational domain $A \rightarrow A \blang \frac{f_1^{1/p^\infty}, \dots, f_n^{1/p^\infty}}{f_i^{1/p^\infty}} \brang$ as the contractive morphisms obtained as the composition of the following maps
\begin{equation*}
	A \rightarrow A \langle T_1^{1/p^\infty}, \dots, T_n^{1/p^\infty} \rangle \rightarrow (A \langle T_1^{1/p^\infty}, \dots, T_n^{1/p^\infty} \rangle/I_{f_i}^{1/p^\infty})^{\wedge} =: A \blang \frac{f_1^{1/p^\infty}, \dots, f_n^{1/p^\infty}}{f_i^{1/p^\infty}} \brang
\end{equation*}
Furthermore, its clear by construction that the morphism $A \rightarrow A \blang \frac{f_1^{1/p^\infty}, \dots, f_n^{1/p^\infty}}{f_i^{1/p^\infty}} \brang$ admits a factorization into a pair of contractive maps
\begin{equation*}
	A \rightarrow A \blang \frac{f_1, \dots, f_n}{f_i} \brang \rightarrow A \blang \frac{f_1^{1/p^\infty}, \dots, f_n^{1/p^\infty}}{f_i^{1/p^\infty}} \brang
\end{equation*}
\end{defn}

\begin{lemma}\label{equiv_rat_domain_and_perfected} Let $q: A \blang \frac{f_1, \dots, f_n}{f_i} \brang \rightarrow A \blang \frac{f_1^{1/p^\infty}, \dots, f_n^{1/p^\infty}}{f_i^{1/p^\infty}} \brang$ be the contractive morphism of Banach $K$-algebras introduced in \ref{defn_perfect_rat_domains}. Then, $q$ admits a bounded inverse.
\end{lemma}

\begin{proof} Since $\frac{1}{f_i} \in A \blang \frac{f_1, \dots, f_n}{f_i} \brang$ by \ref{unit_in_rational_domain}, we can conclude that $\Big( \frac{f_j}{f_i} \Big )^{1/p^n} \in A \blang \frac{f_1, \dots, f_n}{f_i} \brang$ for all $n \in \ZZ_{\ge 0}$. Moreover, as $\varpi \in K$ is a topological nilpotent unit we know that there exists a $m > 0$ such that $|\frac{\varpi^m}{f_i}| \le 1$ in $A \blang \frac{f_1, \dots, f_n}{f_i} \brang$, hence we can conclude that $|\Big( \frac{f_j}{f_i} \Big )^{1/p^n}| \le |\varpi^{-m}|$ for all $\Big (\frac{f_j}{f_i} \Big)^{1/p^n} \in A \blang \frac{f_1, \dots, f_n}{f_i} \brang$. This bound, together with the fact that each $\frac{f_j}{f_i}$ is power-bounded in $A \blang \frac{f_1, \dots, f_n}{f_i} \brang$ implies that there exists a bounded map $A \langle T_1^{1/p^\infty}, \dots, T_n^{1/p^\infty} \rangle \rightarrow A \blang \frac{f_1, \dots, f_n}{f_i} \brang$ sending $T_j^{1/p^n} \mapsto \Big (\frac{f_j}{f_i} \Big)^{1/p^n}$ (cf. \ref{free_perfect_banach_alg}); and since $I_{f_i}^{1/p^\infty}$ is in the kernel we get an induced bounded map
\begin{align*}
	q^{-1}: A \blang \frac{f_1^{1/p^\infty}, \dots, f_n^{1/p^\infty}}{f_i^{1/p^\infty}} \brang \rightarrow A \blang \frac{f_1, \dots, f_n}{f_i} \brang && \Big (\frac{f_j}{f_i} \Big)^{1/p^n} \mapsto \Big (\frac{f_j}{f_i} \Big)^{1/p^n}
\end{align*}
Its clear that this morphisms are inverses of each other, proving the result.
\end{proof}

\begin{prop}\label{int_model_perfect_rat_domain} Let $A$ be a uniform Banach $K$-algebra, and $\{f_1, \dots, f_n\} \subset A$ a collection of elements generating the unit ideal which have norm $\le 1$, together with a choice of compatible $p$-power roots $f_i^{1/p^m} \in A$ for each $f_i$. Then, the perfected rational domain $A \blang \frac{f_1^{1/p^\infty}, \dots, f_n^{1/p^\infty}}{f_i^{1/p^\infty}} \brang$ corresponds to
\begin{equation*}
	\Big ( A_{\le 1}[T_1^{1/p^\infty}, \dots, T_n^{1/p^\infty}]^{\wedge} / I_{f_i, \le 1}^{1/p^\infty}   \Big )^{\wedge a \tf} \in \CAlg_{K_{\le 1}}^{\wedge a \tf}
\end{equation*}
under the equivalence \ref{equiv_almost_tf_banach}, and where the ideal $I_{f_i, \le 1}^{1/p^\infty} \subset A_{\le 1}[T_1^{1/p^\infty}, \dots, T_n^{1/p^\infty}]^{\wedge}$ is defined as
\begin{align*}
	I_{f_i, \le 1}^{1/p^m} := (f_i^{1/p^m} T_1^{1/p^m} - f_1^{1/p^m}, \dots, f_i^{1/p^m} T_n^{1/p^m} - f_n^{1/p^m}) && I_{f_i, \le 1}^{1/p^\infty} = \bigcup_{m \in \ZZ_{\ge 0}} I_{f_i, \le 1}^{1/p^m}
\end{align*}
Again, notice that the ideal $I_{f_i, \le 1}^{1/p^\infty}$ admits an exhaustive increasing filtration
\begin{equation*}
	I_{f_i, \le 1} \subset I_{f_i, \le 1}^{1/p} \subset \cdots \subset I_{f_i, \le 1}^{1/p^m} \subset \cdots \subset I_{f_i, \le 1}^{1/p^\infty}
\end{equation*}
Furthermore, as $A_{\le 1}[T_1^{1/p^\infty}, \dots, T_n^{1/p^\infty}]$ is $\varpi$-torsion free its derived and classical $\varpi$-completion agree, the same holds for the ideal $I_{f_i}^{1/p^\infty}$.
\end{prop}

\begin{proof} First, let us remark that since $A$ is uniform, if $|f_i|_A \le 1$ then $|f_i^{1/p^m}|_A \le 1$ for all $m \in \ZZ_{\ge 0}$. For the duration of this proof let $P_k$ be the objects of $\CAlg_{K_{\le 1}}^{\wedge a \tf}$ defined as
\begin{equation*}
	P_k := \Big ( A_{\le 1}[T_1^{1/p^\infty}, \dots, T_n^{1/p^\infty}]^{\wedge} / I_{f_i, \le 1}^{1/p^k}   \Big )^{\wedge a \tf} \in \CAlg_{K_{\le 1}}^{\wedge a \tf}
\end{equation*}
And recall that $P_k$ can be described as a pushout of the following diagram, computed in $\CAlg_{K_{\le 1}}^{\wedge a \tf}$
\begin{cd}
	A_{\le 1} [T_1, \dots, T_n]^{\wedge} \ar[r, "h_k"] \ar[d] & A_{\le 1}[T_1^{1/p^\infty}, \dots, T_n^{1/p^\infty}]^{\wedge} \ar[d] \\
	A_{\le 1} \ar[r] & P_k
\end{cd}
where $h_k(T_j) = f_i^{1/p^k} T_j^{1/p^k} - f_j^{1/p^k}$. Using the characterization of $P_k$ as a pushout of the above diagram and the equivalence \ref{equiv_almost_tf_banach}, we can conclude, by a similar argument to the one employed in \ref{int_model_rat_domain}, that $P_k$ corresponds to the Banach $K$-algebra 
\begin{equation*}
	Q_k := (A\langle T_1^{1/p^\infty}, \dots, T_n^{1/p^{\infty}} \rangle/ I_{f_i}^{1/p^k})^{\wedge}
\end{equation*}
where the ideal $I_{f_i}^{1/p^m} \subset A\langle T_1^{1/p^\infty}, \dots, T_n^{1/p^{\infty}}\rangle$ was introduced in \ref{defn_perfect_rat_domains}.

Finally, notice that we have canonical identifications $(P_k/I_{f_i, \le 1}^{1/p^{k + 1}})^{\wedge a \tf} \simeq P_{k+1}$ and $(Q_k/I_{f_i}^{1/p^{k+1}})^{\wedge} \simeq Q_{k+1}$, providing us with morphisms $P_k \rightarrow P_{k+1}$ and $Q_k \rightarrow Q_{k+1}$ which correspond to each other under the equivalence \ref{equiv_almost_tf_banach}. In particular, this implies that $P_\infty = \colim_k P_k$ (computed in $\CAlg_{K_{\le 1}}^{\wedge a \tf}$) and $Q_\infty = \colim_k Q_k$ (computed in $\Ban_K^{\contr}$) corresponds to each other under the equivalence \ref{equiv_almost_tf_banach}. And by construction we conclude that
\begin{align*}
	P_{\infty} = \Big ( A_{\le 1}[T_1^{1/p^\infty}, \dots, T_n^{1/p^\infty}]^{\wedge} / I_{f_i, \le 1}^{1/p^\infty}   \Big )^{\wedge a \tf} &&
	Q_{\infty} = A \blang \frac{f_1^{1/p^\infty}, \dots, f_n^{1/p^\infty}}{f_i^{1/p^\infty}} \brang
\end{align*}
proving the desired result.
\end{proof}

\begin{lemma}\label{perfect_delta_ring_two_ideals} Let $R$ be an element of $\Perfd_{K_{\le 1}}^{\Prism a}$, and $(A,d)$ its corresponding classically $(p, [\varpi^\flat])$-complete perfect prism.  Let $a,b \in A/p$ and $[a],[b] \in A$ its Teichmuller lifts. Define the ideals of $A$
\begin{align*}
	[I] = \bigcup_{m \in \ZZ_{\ge 0}} \Big ( [a^{1/p^m} - b^{1/p^m}] \Big ) && I = \bigcup_{m \in \ZZ_{\ge 0}} \Big ( [a^{1/p^m}] - [b^{1/p^m}] \Big )
\end{align*}
Then,
\begin{enumerate}[(1)]
	\item The derived and classical $(p, [\varpi^\flat])$-completions of $I$ (resp. $[I]$) agree. We denote it by $I^{\wedge}$ (resp. $[I]^{\wedge})$.
	\item There is an inclusion $[I]^{\wedge} \subset I^{\wedge}$ of ideals of $A$.
	\item $A/[I]^{\wedge} = (A/[I])^{\wedge}$ is a $(p, [\varpi^\flat])$-complete perfect $\delta$-ring.
\end{enumerate}
\end{lemma}

\begin{proof} As $I$ and $[I]$ are ideals of $A$, it follows they are both $p$-torsion free, thus their derived and classical $p$-completions agree, we denote this $p$-completions by $I^{\wedge}_{(p)}$ and $[I]^{\wedge}_{(p)}$. In particular, $I^{\wedge}_{(p)}$ and $[I]^{\wedge}_{(p)}$ are still ideals of $A$, thus they have bounded $[\varpi^\flat]$-torsion (\ref{torsion_perfect_delta_ring}), which in turn implies that their classical and derived $[\varpi^\flat]$-completions agree. Finishing the proof of $(1)$.

We claim that in $A$ we have the following identity
\begin{equation*}
	\lim_{m \rightarrow \infty} \Big ([a^{1/p^m}] - [b^{1/p^m}] \Big)^{p^m} = [a - b]
\end{equation*}
and it evident generalizations to $[a^{1/p^m} - b^{1/p^m}]$. Once we establish this identity, $(2)$ would clearly follow. Indeed, recall we have the identities $[a^{1/p^m}] - [b^{1/p^m}] = [a^{1/p^m} - b^{1/p^m}] \bmod p$, which in turn imply by the binomial theorem (cf. \cite[Lemma 2.0.5]{bhattlecture_perfectoid}) that
\begin{equation*}
	\Big ( [a^{1/p^m}] - [b^{1/p^m}]   \Big)^{p^k} = [a^{1/p^m} - b^{1/p^m}]^{p^k} \bmod p^{k+1}
\end{equation*}
which by $p$-completeness of $A$ implies the desired identities, finishing the proof of $(2)$.

Next we prove $(3)$. Notice that the identity $(A/[I])^{\wedge} = A/[I]^{\wedge}$ is immediate from generalities on derived completions. Define a map of perfect $(p, [\varpi^\flat])$-complete $\delta$-rings $h: A[T^{1/p^\infty}]^{\wedge} \rightarrow A$ where $h(T^{1/p^n}) = [a-b]^{1/p^n}$, then we can identify $(A/[I])^{\wedge}$ as the $(p, [\varpi^\flat])$-complete pushout of the following diagram
\begin{cd}
	A[T^{1/p^\infty}]^{\wedge} \ar[r, "h"] \ar[d, "T^{1/p^n} \mapsto 0"] & A \ar[d] \\
	A \ar[r] & (A/[I])^{\wedge}
\end{cd}
and as everything in sight is a perfect $\delta$-ring, it follows that $(A/[I])^{\wedge}$ is a $(p, [\varpi^\flat])$-complete perfect $\delta$-ring, proving $(3)$.
\end{proof}

\begin{prop}\label{perfect_rat_domain_is_perfectoid} Let $A$ be a perfectoid Banach $K$-algebra, and $\{f_1, \dots, f_n\} \subset A$ be a collection of elements generating the unit ideal of $A$, together with a choice of compatible $p$-power roots $f_i^{1/p^m} \in A$ for each $f_i$. Then, the perfected rational domain $A \blang \frac{f_1^{1/p^\infty}, \dots, f_n^{1/p^\infty}}{f_i^{1/p^\infty}} \brang$ is a perfectoid Banach $K$-algebra.

Since we can always assume that $|f_j| \le 1$ for all $f_j$, if $K_{\le 1}$ corresponds to the perfect prism $(\A_{\inf}(K_{\le 1}), d)$, define the $(p, [\varpi^\flat])$-complete perfect prism $(W,d)$ as
\begin{align*}
	& W := \A_{\inf} (A_{\le 1})[T_1^{1/p^\infty}, \dots, T_n^{1/p^\infty}]^{\wedge} / [\cI]^{\wedge} && \text{where}\\
	& [\cI] := \bigcup_{m \in \ZZ_{\ge 0}} \Big ([f_1^{\flat, 1/p^m} - T_1^{1/p^m} f_i^{\flat, 1/p^m}], \dots, [f_n^{\flat, 1/p^m} - T_n^{1/p^m} f_i^{\flat, 1/p^m}] \Big) && \text{and} \\
	& [\cI]^{\wedge} := [\cI]^{\wedge}_{(p, [\varpi^\flat])}
\end{align*}
where the completion of $[\cI]^{\wedge}_{(p, [\varpi^\flat])}$ is derived and also classical. Then, the perfected rational domain $A \blang \frac{f_1^{1/p^\infty}, \dots, f_n^{1/p^\infty}}{f_i^{1/p^\infty}} \brang$ corresponds to $(W/d)^a \in \Perfd_{K_{\le 1}}^{\Prism a}$ under the equivalence \ref{equiv_perfd_ban_tic}, and the map $A \rightarrow (W/d)^a [\frac{1}{\varpi}]$ is the untilt (\ref{tilting_corr_banach}) of the map
\begin{equation*}
	A^{\flat} \rightarrow A^{\flat} \blang \frac{f_1^{\flat, 1/p^\infty}, \dots, f_n^{\flat, 1/p^\infty}}{f_i^{\flat, 1/p^\infty}} \brang
\end{equation*}
\end{prop}

\begin{proof} Let us begin by recalling that the classical and derived $(p, [\varpi^\flat])$-completion of $P:=\A_{\inf} (A_{\le 1})[T_1^{1/p^\infty}, \dots, T_n^{1/p^\infty}]$ agree (\ref{free_alg_perfect_is_perfectoid}), we denote its $(p, [\varpi^\flat])$-completion by $P^{\wedge}$. The same argument as in Lemma \ref{perfect_delta_ring_two_ideals}(1) shows that the derived and classical $(p, [\varpi^\flat])$-completion of $[\cI]$ agree, and the same holds for the ideals $\cI, (\cI, d), ([\cI], d) \subset P^{\wedge}$, where
\begin{equation*}
	\cI := \bigcup_{m \in \ZZ_{\ge 0}} \Big ([f_1^{\flat, 1/p^m}] - [T_1^{1/p^m} f_i^{\flat, 1/p^m}], \dots, [f_n^{\flat, 1/p^m}] - [T_n^{1/p^m} f_i^{\flat, 1/p^m}] \Big)
\end{equation*}
Then, from Lemma \ref{perfect_delta_ring_two_ideals}(2) we learn that the canonical map $P^{\wedge} \rightarrow P^{\wedge}/\cI^{\wedge} = (P/ \cI)^{\wedge}$ factors as
\begin{align*}
	P^{\wedge} \rightarrow W := P^{\wedge}/[\cI]^{\wedge} \twoheadrightarrow P^{\wedge}/\cI^{\wedge} = (P/\cI)^{\wedge}
\end{align*}
By the same argument as in Lemma \ref{perfect_delta_ring_two_ideals}(3) we learn that $W$ is a perfect $\delta$-ring, and since $d \in \A_{\inf}(K_{\le 1})$ is a distinguished element it follows that $(W, d)$ is a perfect prism. In particular, $W/d$ can be identified with the $K_{\le 1}$-algebra
\begin{align*}
	&W/d = \colim_{m} A_{\le 1}[T_1^{1/p^\infty}, \dots, T_n^{1/p^\infty}]^{\wedge} / [I_{f_i, \le 1}^{1/p^m}]^{\wedge} && \text{where} \\
	&[I_{f_i, \le 1}^{1/p^m}] := \Big ( (f_i^{1/p^m} T_1^{1/p^m} - f_1^{1/p^m})^{\sharp}, \dots, (f_i^{1/p^m} T_n^{1/p^m} - f_n^{1/p^m})^{\sharp} \Big) && [I_{f_i}^{1/p^\infty}] = \bigcup_{m \in \ZZ_{\ge 0}} [I_{f_i, \le 1}^{1/p^m}]
\end{align*}
where all the completions are derived (also classical) with respect to $\varpi$. Reducing the above sequence of maps modulo $d$ gives us maps
\begin{equation*}
	A_{\le 1} \rightarrow A_{\le 1}[T_1^{1/p^\infty}, \dots, T_n^{1/p^\infty}]^{\wedge} \rightarrow W/d \twoheadrightarrow A_{\le 1}[T_1^{1/p^\infty}, \dots, T_n^{1/p^\infty}]^{\wedge}/(I_{f_i, \le 1}^{1/p^\infty})^{\wedge}
\end{equation*}
We are implicitly using the fact that $\Big( \cI_{f_i, \le 1}^{1/p^\infty} \Big)/d = I_{f_i, \le 1}^{1/p^\infty}$, which follows from the identity $[f_j^{\flat, 1/p^m}] - [T_j^{1/p^m} f_i^{\flat, 1/p^m}] = f_j^{1/p^m} - T_j^{1/p^m} f_i^{1/p^m} \bmod d$ -- see \ref{int_model_perfect_rat_domain} for the definition of the ideal $I_{f_i}^{1/p^\infty}$.

It is clear from the construction of $W/d$ that $A \rightarrow (W/d)^a[\frac{1}{\varpi}]$ is the untilt (\ref{tilting_corr_banach}) of the map $A^{\flat} \rightarrow A^{\flat} \blang \frac{f_1^{\flat, 1/p^\infty}, \dots, f_n^{\flat, 1/p^\infty}}{f_i^{\flat, 1/p^\infty}} \brang$, in particular it follows from Proposition \ref{tilting_ban_topo} that the induced map $|\cM((W/d)^a[\frac{1}{\varpi}])| \rightarrow |\cM(A)|$ has image contained in $|\cM(A)|\Big (\frac{f_1, \dots, f_n}{f_i} \Big)$. Then, \ref{rational_dom_univ_prop} and \ref{equiv_rat_domain_and_perfected} imply that the map $A \rightarrow (W/d)^a[\frac{1}{\varpi}]$ factors as
\begin{equation*}
	A \rightarrow A \blang \frac{f_1^{1/p^\infty}, \dots, f_n^{1/p^\infty}}{f_i^{1/p^\infty}} \brang \rightarrow (W/d)^a[\frac{1}{\varpi}]
\end{equation*}
Using this factorization, together with the fact that $A_{\le 1} \rightarrow (W/d)^a$ is an epimorphism in $\Perfd_{K_{\le 1}}^{\Prism a}$ (cf. \ref{rat_domains_are_mono}), we learn that the unit of the adjunction $W/d \rightarrow (W/d)_*$ (cf. \ref{almost_elements}) factors as 
\begin{align*}
	W/d \twoheadrightarrow A_{\le 1}[T_1^{1/p^\infty}, \dots, T_n^{1/p^\infty}]^{\wedge}/(I_{f_i, \le 1}^{1/p^\infty})^{\wedge} \rightarrow (W/d)_*
\end{align*}
passing to the almost category $\Mod_{K_{\le 1}, \varpi}^{\wedge a}$ shows that the identity map $(W/d)^a \rightarrow (W/d)^a$ factors as
\begin{cd}
	& \Big ( A_{\le 1}[T_1^{1/p^\infty}, \dots, T_n^{1/p^\infty}]^{\wedge}/(I_{f_i, \le 1}^{1/p^\infty})^{\wedge} \Big )^a \ar[rd] \\
	(W/d)^a \ar[ru, two heads] \ar[rr] && (W/d)^a
\end{cd}
Where the map $(W/d)^a \twoheadrightarrow \Big ( A_{\le 1}[T_1^{1/p^\infty}, \dots, T_n^{1/p^\infty}]^{\wedge}/(I_{f_i, \le 1}^{1/p^\infty})^{\wedge} \Big )^a$ is an epimorphism in $\Mod_{K_{\le 1}, \varpi}^{\wedge a}$, since left adjoints preserve epimorphisms; and it is also a monomorphism since it factors the identity map $(W/d)^a \rightarrow (W/d)^a$. Thus, since $\Mod_{K_{\le 1}, \varpi}^{\wedge a}$ is an abelian category it follows that
\begin{equation*}
	(W/d)^a \rightarrow \Big ( A_{\le 1}[T_1^{1/p^\infty}, \dots, T_n^{1/p^\infty}]^{\wedge}/(I_{f_i, \le 1}^{1/p^\infty})^{\wedge} \Big )^a
\end{equation*}
is an isomorphism in $\Mod_{K_{\le 1}, \varpi}^{\wedge a}$, proving the result (cf. \ref{int_model_perfect_rat_domain}).
\end{proof}

\subsection{Structure presheaf on affinoid perfectoids}

Let $K$ be a perfectoid non-archimedean field. 

\begin{lemma}[Approximation Lemma]\label{approx_lemma} Let $A$ be a perfectoid Banach $K$-algebra, for any $f \in A$ and any $c \ge 0, \varepsilon > 0$, there exists a $f_{c, \varepsilon} \in A$ admitting compatible $p$-power roots $f_{c, \varepsilon}^{1/p^m} \in A$ such that for all $x \in |\cM(A)|$ we have
\begin{equation*}
	|(f - f_{c, \varepsilon})(x)| \le |\varpi|^{1-\varepsilon} \max (|f(x)|, |\varpi|^c)
\end{equation*}
\end{lemma}

\begin{proof} \cite[Corollary 6.7 (i)]{scholze2012perfectoid}
\end{proof}

\begin{prop} Let $A$ be a perfectoid Banach $K$-algebra, and $\{f_1, \dots, f_n \} \subset A$ a collection of elements generating the unit ideal of $A$. Then, there exists a collection of elements $\{g_1, \dots, g_n\} \subset A$ generating the unit ideal of $A$, admitting compatible $p$-power roots $g_j^{1/p^m} \in A$, and such that we have an equality
\begin{equation*}
	|\cM(A)| \Big (\frac{f_1, \dots, f_n}{f_i} \Big) = |\cM(A)| \Big (\frac{g_1, \dots, g_n}{g_i} \Big)
\end{equation*}
of compact hausdorff subspaces of $|\cM(A)|$.
\end{prop}

\begin{proof} First, we will show that we have an inclusion $|\cM(A)| \Big (\frac{f_1, \dots, f_n}{f_i} \Big) \subset |\cM(A)| \Big (\frac{g_1, \dots, g_n}{g_i} \Big)$ as subsets of $|\cM(A)|$. Indeed, as $|f_i(x)| \not= 0$ for all $x \in |\cM(A)| \Big (\frac{f_1, \dots, f_n}{f_i} \Big)$ (cf. \ref{defn_top_rat_domain}), the existence of the continuous functions $|\cM(A)| \Big (\frac{f_1, \dots, f_n}{f_i} \Big) \rightarrow \RR_{\ge 0}$ defined by $x \mapsto |f_i(x)|$ guarantee that there exists a $c \gg 0 $ such that $|f_i(x)| > |\varpi|^c$ for all $x \in |\cM(A)| \Big (\frac{f_1, \dots, f_n}{f_i} \Big)$ and all $1 \le i \le n$. Then, the non-archimedean triangle inequality implies that $|f_i(x)| = |g_i(x)|$ for all $x \in |\cM(A)| \Big (\frac{f_1, \dots, f_n}{f_i} \Big)$, where $g_i = f_{i, c, \varepsilon}$ (as defined in \ref{approx_lemma}). For the same fixed $c \gg 0$ we set $g_j = f_{j, c, \varepsilon}$, and we claim that $|f_j(x)| \le |f_i(x)|$ implies that $|g_j(x)| \le |g_i(x)|$. If $|f_j(x)| \ge |\varpi|^c$, then the non-archimedean triangle inequality implies that $|f_j(x)| = |g_j(x)|$, in which case the claim is clear. On the other hand if $|f_j(x)| < |\varpi|^c$ and $|f_j(x)| \not= |g_j(x)|$, then the identities 
\begin{equation*}
	|(f_j - g_j)(x)| = \max(|f_j(x)|, |g_j(x)|) < |\varpi|^c < |f_i(x)| = |g_i(x)|
\end{equation*}	
prove the desired claim.

Next, we show that the collection of elements $\{g_1, \dots, g_n\} \subset A$ introduced above generates the unit ideal of $A$. Its easy to see from the definition of $|\cM(A)| \Big (\frac{f_1, \dots, f_n}{f_i} \Big)$ (cf. \ref{defn_top_rat_domain}) that the following map of compact hausdorff spaces
\begin{equation*}
	h: \bigsqcup_{1 \le i \le n} |\cM(A)| \Big (\frac{f_1, \dots, f_n}{f_i} \Big) \longrightarrow |\cM(A)|
\end{equation*}
is surjective. For the sake of contradiction, assume that there exists a maximal ideal $\frakm \subset A$ containing the ideal $(g_1, \dots, g_n)$; then any point $x \in |\cM(A)|$ contained in the image of $|\cM(A/\frakm)| \rightarrow |\cM(A)|$ will satisfy $|g_i(x)| = 0$ for all $g_i$. But this contradicts the surjectivity of the map $h$, as $|g_i(x)| \not= 0$ for all $x \in |\cM(A)| \Big (\frac{f_1, \dots, f_n}{f_i} \Big)$.

Finally, we need to show that we have an inclusion $|\cM(A)| \Big (\frac{g_1, \dots, g_n}{g_i} \Big) \subset |\cM(A)| \Big (\frac{f_1, \dots, f_n}{f_i} \Big)$. We claim that for all $x \in |\cM(A)| \Big (\frac{g_1, \dots, g_n}{g_i} \Big)$ we have the equality $|f_i(x)| = |g_i(x)|$ and $|g_i(x)| \ge |\varpi|^c$. Indeed, if $|f_i(x)| \ge |\varpi|^c$ then the argument in the first paragraphs shows that $|f_i(x)| = |g_i(x)|$, thus for the sake of contradiction we may assume that $|f_i(x)| < |\varpi|^c$ and $|f_i(x)| \not= |g_i(x)|$. Then, we have the identities
\begin{equation*}
	|(f_i - g_i)(x)| = \max( |f_i(x)|, |g_i(x)| ) < |\varpi|^c
\end{equation*}
which implies by hypothesis that $|g_j(x)| < |\varpi|^c$ for all $1 \le j \le n$. However, the surjectivity of $h$ shows that there exists a $j$ such that $x \in |\cM(A)| \Big (\frac{f_1, \dots, f_n}{f_j} \Big)$ which in turn implies that $|g_j(x)| = |f_j(x)| > |\varpi|^c$, contradicting the inequality $|g_j(x)| < |\varpi|^c$. We have shown that $|f_i(x)| = |g_i(x)|$ and $|g_i(x)| \ge |\varpi|^c$ for all $x \in |\cM(A)| \Big (\frac{g_1, \dots, g_n}{g_i} \Big)$. It remains to show that if $|g_j(x)| \le |g_i(x)|$ then $|f_j(x)| \le |f_i(x)|$. If $|f_j(x)| \ge |\varpi|^c$ then the non-archimedean triangle inequality implies that $|f_j(x)| = |g_j(x)|$, in which case the result is clear. Thus, we may assume that $|f_j(x)| < |\varpi|^c$ and $|f_j(x)| \not= |g_j(x)|$, then the identities
\begin{equation*}
	|(f_j - g_j)(x)| = \max(|f_j(x)|, |g_j(x)|) < |\varpi|^c \le |f_i(x)| = |g_i(x)|
\end{equation*}	
proving the result.

\end{proof}

\begin{defn}\label{defn_rational_site} Let $A$ be a Banach $K$-algebra. Define $|\cM(A)|_{\rat}$ as the category whose objects are compact hausdorff subsets $U \subset |\cM(A)|$ for which there exists a collection of elements $\{f_1, \dots, f_n\} \subset A$ generating the unit ideal of $A$ such that
\begin{equation*}
	U = |\cM(A)|\Big (\frac{f_1, \dots, f_n}{f_i} \Big) \subset |\cM(A)|
\end{equation*}
For a pair of object $U,V \in |\cM(A)|_{\rat}$ the morphisms are described by
\begin{equation*}
	\Maps_{|\cM(A)|_{\rat}} (U, V) = 
	\begin{cases}
		* & \text{ if } U \subset V \\
		\emptyset & \text{ otherwise}
	\end{cases}
\end{equation*}
Furthermore, notice that if $V = |\cM(A)| \Big (\frac{g_1, \dots, g_m}{g_j} \Big)$ then $U \cap V \in |\cM(A)|_{\rat}$ as
\begin{equation*}
	U \cap V = |\cM(A)| \Big (\frac{f_1g_1, \dots, f_k g_l , \dots,  g_m}{f_ig_j} \Big)
\end{equation*}
where the numerator ranges over all pairs $f_k g_l$.
\end{defn}

\begin{thm}[Structure presheaf]\label{struct_presheaf_perfectoid} Let $A$ be a perfectoid Banach $K$-algebra, set $X = \cM(A)$ and $(U \rightarrow V) \subset |X|_{\rat}$. Then, there exists a functor
\begin{align*}
	\cO_X(-): |X|_{\rat}^{\op} \longrightarrow \Perfd_K^{\Ban} && U \mapsto \cO_X(U)
\end{align*}
satisfying the following conditions
\begin{enumerate}[(1)]
	\item $\cO_X (X) = A$
	\item There exists a collection of elements $\{f_1, \dots, f_n\} \subset A$ generating the unit ideal, together with a choice of compatible $p$-power roots $f_j^{1/p^m} \in A$ for all $m \in \ZZ_{\ge 0}$, such that
	\begin{equation*}
		\cO_X(U) = A \blang \frac{f_1^{1/p^\infty}, \dots, f_n^{1/p^\infty}}{f_i^{1/p^\infty}} \brang
	\end{equation*}
	Furthermore, the perfectoid Banach $K$-algebra $\cO_X(U)$ is independent of the choice of $\{f_1, \dots, f_n\} \subset A$ and its $p$-power roots.
	\item The induced map $\cM(\cO_X(U)) \rightarrow \cM(\cO_X(V))$ is a monomorphism in $\Ban_K^{\op}$.
	\item The induced map $f: |\cM(\cO_X(U))| \rightarrow |\cM(\cO_X(V))|$ of compact hausdorff spaces is injective and satisfies $\im(f) = U \subset |\cM(\cO_X(V))|$.
	\item For any morphism $f: \cM(B) \rightarrow \cM(\cO_X(V))$ from a uniform Banach $K$-algebra $B$, such that $\im(|f|) \subset U \subset V$, there exists a unique factorization of $\cO_X(V) \rightarrow B$ as
	\begin{equation*}
		\cO_X(V) \rightarrow \cO_X(U) \rightarrow B
	\end{equation*}
	\item For any triple of objects $U, W, Z \in |X|_{\rat}$ satisfying $U \cup W \subset Z$, we have the identity
	\begin{equation*}
		\cO_X(U \cap W) = \cO_X(U) \cotimes_{\cO_X(Z)} \cO_X(W)
	\end{equation*}
\end{enumerate}
Furthermore, under the equivalence $(-)_{\le 1}: \Perfd_{K}^{\Ban} \rightleftarrows \Perfd_{K_{\le 1}}^{\Prism a}: (-)[\frac{1}{\varpi}]$ (\ref{equiv_perfd_ban_tic}) we obtain another functor
\begin{align*}
	\cO_X(-)_{\le 1}: |X|_{\rat}^{\op} \longrightarrow \Perfd_{K_{\le 1}}^{\Prism a} && U \mapsto \cO_X(U)_{\le 1}
\end{align*}
\end{thm}

\begin{proof} We begin by constructing the functor $\cO_X(-)$, which will involve making some choices and then we will show that the functor is independent of the choices. From the approximation lemma (\ref{approx_lemma}) we know that for each $U \in |X|_{\rat}$ there exists a collection of $\{f_1, \dots, f_n\} \subset A$ generating the unit ideal, together with a choice of compatible $p$-power roots $f_j^{1/p^m} \in A$ for all $m \in \ZZ_{\ge 0}$, such that
\begin{align*}
	U = |\cM(A)|\Big (\frac{f_1, \dots, f_n}{f_i} \Big) && \text{and set} && \cO_X(U) = A \blang \frac{f_1^{1/p^\infty}, \dots, f_n^{1/p^\infty}}{f_i^{1/p^\infty}} \brang
\end{align*}
Then, by \ref{perfect_rat_domain_is_perfectoid} we can conclude that $\cO_X(U) \in \Perfd_{K}^{\Ban}$ -- in particular, for $U = X$ we choose $\cO_X(X) = A$, this finishes the proof of $(1)$.

Next, recall that in \ref{equiv_rat_domain_and_perfected} we proved we have an isomorphism of Banach $K$-algebras
\begin{equation*}
	A \blang \frac{f_1, \dots, f_n}{f_i} \brang \simeq A \blang \frac{f_1^{1/p^\infty}, \dots, f_n^{1/p^\infty}}{f_i^{1/p^\infty}} \brang = \cO_X(U)
\end{equation*}
in particular, as the right hand side is a uniform Banach $K$-algebra the isomorphism presents $\cO_X(U)$ as the uniformization of $A \blang \frac{f_1, \dots, f_n}{f_i} \brang$. Assume that we have a morphism $f: \cM(B) \rightarrow X$ such that $\im(|f|) \subset U \subset |X|$, we showed in \ref{rational_dom_univ_prop} that the morphism $A \rightarrow B$ admits an essentially unique factorization
\begin{equation*}
	A \rightarrow A \blang \frac{f_1, \dots, f_n}{f_i} \brang \rightarrow \cO_X(U) \rightarrow B
\end{equation*}
characterizing $\cO_X(U)$ among all uniform Banach $K$-algebras by a universal property. In turn, this implies that we have a functor $\cO_X(-): |X|_{\rat}^{\op} \rightarrow \Perfd_K^{\Ban}$ and that $\cO_X(U)$ is independent of the choice of $\{f_1, \dots, f_n\} \subset A$ and its $p$-power roots. This completed the proof of $(2)$, for $(5)$ it remains to show that $\cM(\cO_X(U)) \rightarrow \cM(\cO_X(V))$ is a monomorphism in $\Ban_K^{\op}$.

In \ref{rat_domains_are_mono} we proved that $\cM \Big (A \blang \frac{f_1, \dots, f_n}{f_i} \brang \Big) \rightarrow X$ is a monomorphism in $\Ban_K^{\op}$, which in turn implies that $\cM ( \cO_X(U) ) \rightarrow X$ is a monomorphism in $\Ban_K^{\op}$ for all $U \in |X|_{\rat}$. Furthermore, if $U \subset V$ the fact that $\cM ( \cO_X(U) ) \rightarrow X$ admits a unique factorization as $\cM(\cO_X(U)) \rightarrow \cM(\cO_X(V)) \rightarrow X$ implies that $\cM(\cO_X(U)) \rightarrow \cM(\cO_X(V))$ is a monomorphism in $\Ban_K^{\op}$, proving $(3)$ and $(5)$.  Statement $(4)$ then follows from \ref{mono_berko_sp_injective} and \ref{topo_img_rat_domain}.

Finally, recall that $(5)$ characterizes the map $\cO_X(Z) \rightarrow \cO_X(U \cap W)$ as the unique morphisms of perfectoid Banach $K$-algebras such that if $f: \cM(B) \rightarrow \cM(\cO_X(Z))$, from a uniform Banach $K$-algebra $B$, and $\im(|f|) \subset U \cap W$, then there exists a unique factorization $\cO_X(Z) \rightarrow \cO_X(U \cap W) \rightarrow B$ of the map $\cO_X(Z) \rightarrow B$. On the other hand, $\cO_X(Z) \rightarrow \cO_X(U) \cotimes_{\cO_X(Z)} \cO_X(W)$ is a morphism of perfectoid Banach $K$-algebras (\ref{tensor_ban_perfectoid}) and we claim that it satisfies the same universal property as $\cO_X(U \cap W)$. Indeed, if $B$ is a uniform Banach $K$-algebra such that $f: \cM(B) \rightarrow \cM(\cO_X(Z))$ satisfies $\im(|f|) \subset U \cap W \subset |Z|$, then the map $\cO_X(Z) \rightarrow B$ admits unique factorization as
\begin{align*}
	\cO_X(Z) \rightarrow \cO_X(U) \rightarrow B && \cO_X(Z) \rightarrow \cO_X(W) \rightarrow B
\end{align*}
which in turn implies that there is an essentially unique factorization of $\cO_X(Z) \rightarrow B$ as $\cO_X(Z) \rightarrow \cO_X(U) \cotimes_{\cO_X(Z)} \cO_X(W) \rightarrow B$, proving that $\cO_X(U \cap W)$ and $\cO_X(U) \cotimes_{\cO_X(Z)} \cO_X(W)$ satisfy the same universal property. This finishes the proof of $(6)$.
\end{proof}

\begin{prop}\label{colimit_of_perfectoids} Let $\cI \rightarrow \Ban_K^{\contr}$ be a functor such that $\cI \ni i \mapsto A_i \in \Perfd_K^{\Ban}$. Then, the $\varpi$-completed colimit $\colim_\cI A_i$ is a perfectoid Banach $K$-algebra.
\end{prop}

\begin{proof} Under the equivalence $(-)[\frac{1}{\varpi}]: \Ban_K^{\contr} \rightleftarrows \CAlg_{K_{\le 1}}^{\wedge a \tf}: (-)_{\le 1}$ (\ref{equiv_almost_tf_banach}), we obtain a functor
\begin{align*}
	\cI \rightarrow \CAlg_{K_{\le 1}}^{\wedge a \tf} && i \mapsto A_{i, \le 1} \in \Perfd_{K_{\le 1}}^{\Prism a}
\end{align*}
so it suffices to show that $\colim_{\cI} A_{i, \le 1}$ is an element of $\Perfd_{K_{\le 1}}^{\Prism a} \subset  \CAlg_{K_{\le 1}}^{\wedge a \tf}$. Recall that we have an inclusion $\CAlg_{K_{\le 1}}^{\wedge a \tf} \subset \CAlg_{K_{\le 1}}^{\wedge}$ which admits a left adjoint given by (cf. \ref{torsion_free_algebras} and \ref{tf_almost_algebras})
\begin{align*}
	(-)^{\wedge a \tf}: \CAlg_{K_{\le 1}}^{\wedge} \rightarrow \CAlg_{K_{\le 1}}^{\wedge a \tf} && R \mapsto (H^0 j_* R)^{\tf, \wedge, a}
\end{align*}
where $(-)^{\tf, \wedge, a}$ correspond to first passing to the $\varpi$-torsion free quotient, followed by $\varpi$-completion and then passing to the almost category. In \ref{equiv_perfd_tic_almost} we showed that when restricted to $\Perfd_{K_{\le 1}}^{\Prism} \subset \CAlg_{K_{\le 1}}^{\wedge}$ the functor $(-)^{\wedge a \tf}$ is identified with $(-)^a: \Perfd_{K_{\le 1}}^{\Prism} \rightarrow \Perfd_{K_{\le 1}}^{\Prism a}$. Hence, using the inclusion $\CAlg_{K_{\le 1}}^{\wedge a \tf} \subset \CAlg_{K_{\le 1}}^{\wedge}$ it suffices to show that the functor
\begin{align*}
	\cI \rightarrow \CAlg_{K_{\le 1}}^{\wedge} && i \mapsto A_{i, \le 1} \in \Perfd_{K_{\le 1}}^{\Prism}
\end{align*}
satisfies that $\colim_{\cI} A_{i, \le 1}$, computed in $\CAlg_{K_{\le 1}}^{\wedge}$, is an element of $\Perfd_{K_{\le 1}}^{\Prism}$.

Using the functoriality of $\A_{\inf}(-)$ we produce a functor $\cI \rightarrow \CAlg_{\A_{\inf}(K_{\le 1})}^{\wedge}$, to the category of $(p, d, [\varpi^\flat])$-complete $\A_{\inf}(K_{\le 1})$-algebras, given by $i \mapsto \A_{\inf}(A_{i, \le 1})$, where $d$ is a distinguished element of $\A_{\inf}(K_{\le 1})$ such that $\A_{\inf}(K_{\le 1})/d = K_{\le 1}$ (cf. \ref{perfect_prism_perfectoid}). Define $W = \colim_{\cI} \A_{\inf}(A_{i, \le 1})$, we claim that $W$ admits the structure of a perfect $\delta$-ring. Indeed, each $\A_{\inf}(A_{i, \le 1})$ admits a functorial structure of a perfect $\delta$-ring, and since the forgetful functor
\begin{equation*}
	\delta\text{-Rings}_{\A_{\inf}(K_{\le 1})}^{\wedge} \rightarrow \CAlg_{\A_{\inf}(K_{\le 1})}^{\wedge}
\end{equation*}
preserves all colimits (\ref{delta_rings_lim_colim}), we can conclude that $W$ admits the structure of a perfect $\delta$-ring compatible with the maps $\A_{\inf}(A_{i, \le 1}) \rightarrow W$. Furthermore, by the rigidity theorem (\ref{rig_prisms}) it follows that $(W,d)$ is a perfect prism, and so $W/d \in \Perfd_{K_{\le 1}}^{\Prism}$.

It remains to show that $W/d = \colim_{\cI} A_{i, \le 1}$, where the colimit is computed in $\CAlg_{K_{\le 1}}^{\wedge}$, but this follows from the following identities
\begin{align*}
	W/d = W \cotimes_{A_{\inf}(K_{\le 1})} K_{\le 1} & \simeq (\colim_{\cI} \A_{\inf}(A_{i, \le 1})) \cotimes_{A_{\inf}(K_{\le 1})} K_{\le 1} \\
	&\simeq \colim_{\cI} ( \A_{\inf}(A_{i, \le 1})\cotimes_{A_{\inf}(K_{\le 1})} K_{\le 1} ) \\
	&\simeq \colim_{\cI} (A_{i, \le 1})
\end{align*}
proving the result, where we are implicitly using the fact that $- \cotimes_{A_{\inf}(K_{\le 1})} K_{\le 1}$ commutes with all colimits.

\end{proof}

\begin{thm}\label{stalks_perfectoid} Let $A$ be a perfectoid Banach $K$-algebra, and set $X = \cM(A)$. Fix a point $x \in |X|$ and define $\cO_{X,x} \in \Ban_K$ as
\begin{align*}
	\cO_{X,x} = \colim_{x \in U} \cO_X(U) && \text{computed in } \Ban_K
\end{align*}
where the colimit ranges over all $x \in U \in |X|_{\rat}$. Then,
\begin{enumerate}[(1)]
	\item The Banach $K$-algebra $\cO_{X,x}$ is a perfectoid Banach $K$-algebra.
	\item The induced map $\cM(\cO_{X,x}) \rightarrow \cM(\cO_X(U))$ is a monomorphism in $\Ban_K^{\op}$.
	\item The induced injective map $|\cM(\cO_{X,x})| \rightarrow U$ identifies $|\cM(\cO_{X,x})|$ with $x \in U \subset |X|$. In particular, we have that $|\cM(\cO_{X,x})| = \pt$.
	\item Let $B$ be a uniform Banach $K$-algebra and $\cO_X(U) \rightarrow B$ a morphism of Banach $K$-algebras. If the induced map $|\cM(B)| \rightarrow |\cM(\cO_X(U))|$ has its image contained in $x \in U$, then the morphism $\cO_X(U) \rightarrow B$ admits a unique factorization
	\begin{equation*}
		\cO_X(U) \rightarrow \cO_{X,x} \rightarrow B
	\end{equation*}
	\item The morphism $\cO_{X, x} \rightarrow \cH(x)$ of Banach $K$-algebras, induced from the canonical map $\cO_X(U) \rightarrow \cH(x)$ and (3), is an isomorphism of Banach $K$-algebras. In particular, $\cH(x)$ is a perfectoid field.
\end{enumerate}
\end{thm}

\begin{proof} Statement (1) follows from Proposition \ref{colimit_of_perfectoids}. For statement (2), a cofinality argument show that $\cO_{X,x} = \colim_{x \in V \subset U} \cO_X(V)$ for any $U \in |X|_{\rat}$ containing $x$; then the fact that $\cM(\cO_X(V)) \rightarrow \cM(\cO_X(U))$ is a monomorphism in $\Ban_K^{\op}$ for all $x \in V \subset U \subset |X|$, as shown in \ref{struct_presheaf_perfectoid}(3), implies that $\cM(\cO_{X,x}) \rightarrow \cM(\cO_X(U))$ is a monomorphism in $\Ban_K^{\op}$, as monomorphisms are stable under limits.

For (3) recall that we have the identity $|\cM(\cO_{X, x})| \simeq \lim_{x \in V \subset U} |\cM(\cO_X(V))| \rightarrow |\cM(\cO_X(U))|$ from \ref{cofiltered_lim_berko_sp}. As the forgetful functor $\Comp \rightarrow \Set$ preserves limits (\ref{monadicity_comp_set}) it follows that $x \in |\cM(\cO_{X,x})| \subset |\cM(\cO_X(U))|$, it remains to show that for every other $y \in |\cM(\cO_X(U))|$ there exists a rational domain $W \in |X|_{\rat}$ such that $W \subset U$ and $|\cM(\cO_{X}(W))| \rightarrow |\cM(\cO_X(U))|$ contains $x$ and $y \not\in |\cM(W)| \subset |\cM(\cO_X(U))|$. By assumption there exists a $f \in \cO_{X}(U)$ such that $|f(x)| \not= |f(y)|$, and without loss of generality assume that $|f(x)| < |f(y)|$. Then, since the value group of $K$ is dense in $\RR_{\ge 0}$, there exists a $C \in K$ such that $|f(x)| \le |C|$ and $|f(y)| > |C|$. Then, the rational domains
\begin{align*}
	\cO_X(U) \rightarrow W_1 := \cO_{X}(U) \frac{\langle C^{-1} T \rangle}{(f - T)}^{\wedge} && \cO_X(U) \rightarrow W_2 := \cO_{X}(U) \frac{\langle C^{-1} T \rangle}{(fT - 1)}^{\wedge}
\end{align*}
induce injective maps $|\cM(W_1)| \rightarrow |\cM(\cO_X(U))|$ and $|\cM(W_2)| \rightarrow |\cM(\cO_X(U))|$, which have image $|\cM(\cO_X(U))|(\frac{f}{C})$ and $|\cM(\cO_X(U))|(\frac{C}{f})$ respectively. Showing that $x \in W_1 \subset U$ and $y \not\in W_1$, finishing the proof of (3). To proof (4), recall that for any $W \in |X|_{\rat}$ satisfying $x \in W \subset U$ we have that the map $\cO_X(U) \rightarrow B$ factors as $\cO_{X}(U) \rightarrow \cO_X(W) \rightarrow B$ by \ref{struct_presheaf_perfectoid}, the claim then follow from the definition of $\cO_{X,x}$.

Finally, we proof (5). Recall that the universal property of $\cO_X(U) \rightarrow \cH(x)$ (cf. \ref{properties_stalk}(4) and \ref{mono_iso_on_residue_fields}) characterizes $\cH(x)$ among all uniform Banach $K$-algebras with a map from $\cO_X(U)$, which is the same universal property we showed that $\cO_X(U) \rightarrow \cO_{X,x}$ has in (4), completing the proof of (5).
\end{proof}

\newpage

\section{Tate acyclicity for Perfectoids}\label{sect_tate_acyclicity_perfd}

\begin{convention} Throughout this section we want to systematically work with categories of sheaves on $\Perfd_{V}^{\Prism}$, unfortunately for technical reasons we are forced to fix an uncountable strong limit cardinal $\kappa$ (i.e. $\kappa$ is uncountable and $\lambda < \kappa$ then $2^{\lambda} < \kappa$) and restrict our attention to the category $\Perfd_{V, < \kappa}^{\Prism}$ where the cardinality of the perfectoid rings is bounded by $\kappa$. This has the benefit that the collection of object in $\Perfd_{V, < \kappa}^{\Prism}$ is a set, which implies that now the category of sheaves $\Shv(\Perfd_{V, < \kappa}^{\Prism})$ is a topos.

Fortunately, none of the constructions we consider depend on the uncountable strong limit cardinal $\kappa$. Thus, in order to unburden notation we will drop the $\kappa$ from our notation, denoting by $\Perfd_{V}^{\Prism}$ what should be denoted by $\Perfd_{V, < \kappa}^{\Prism}$; or denote by $\Sch_{\qcqs}$ the category of $\qcqs$ schemes with cardinality bounded by $\kappa$.
\end{convention}

\subsection{Various flavors of the arc-topology}

\begin{defn}\label{defn_arc} A map of qcqs schemes $f: Y \rightarrow X$ (eg. $X$ and $Y$ could be affine) is an $\arc$-cover if for any valuation ring $V$ of Krull dimension $\le 1$ and a map $\Spec(V) \rightarrow X$, there exists an extension $\Spec(W) \rightarrow \Spec(V)$ of valuation rings of Krull dimension $\le 1$ and a map $\Spec(W) \rightarrow Y$ making the following diagram commute
\begin{cd}
	\Spec(W) \ar[r] \ar[d] & Y \ar[d] \\
	\Spec(V) \ar[r] & X
\end{cd}
Where an extension of valuation ring $V \rightarrow W$ is a faithfully flat map of valuation rings (cf. \ref{equiv_extensions_rk_one}). 

Its easy to check that the collection of $\arc$-cover in the category $\Sch_{\qcqs}$ (of qcqs schemes) satisfy the conditions of \cite[Proposition A.3.2.1]{luriespectral}, giving rise to a finitary Grothendieck topology on $\Sch_{\qcqs}$ which we call the $\arc$-topology.
\end{defn}

\begin{lemma}\label{faithfully_flat_implies_arc_cover} Let $R \rightarrow S$ be a faithfully flat of rings, then it is an $\arc$-cover.
\end{lemma}

\begin{proof} First, notice that since $R \rightarrow S$ is faithfully flat, the induced map of topological spaces $\Spec (S) \rightarrow \Spec (R)$ is surjective. Let $V$ be a valuation ring of Krull dimension $\le 1$, and a map $\Spec(V) \rightarrow \Spec(R)$, then since faithfully flat maps are closed under base-change it follows that $\Spec(S \otimes_R V) =:\Spec(D) \rightarrow \Spec(V)$ is faithfully flat, and in particular surjective. If $V$ has Krull dimension zero, pick a point $y \in \Spec(D)$, then the map $\Spec(\kappa(y)) \rightarrow \Spec(V)$ is an extension of valuation rings, with the desired lifting property. Hence, we may assume that $V$ is a rank one valuation ring. Let $\mathfrak{p}_2 \subset \mathfrak{p}_1$ two prime ideals on $D$ laying over the zero ideal and the maximal ideal, whose existence is guaranteed by the going-down theorem, and at the expense of possibly loosing flatness we replace $D$ by $(D/\mathfrak{p}_2)_{\mathfrak{p}_1}$, and still have a surjective map $\Spec (D) \rightarrow \Spec (V)$. Then by \cite[Tag 00IA]{stacks-project} we can find a valuation ring $V^{\prime}$ such that $D \rightarrow V^\prime$ dominates $D$ (see \cite[Tag 00I8]{stacks-project} for a definition), so we obtain a surjective map $\Spec(V^{\prime}) \rightarrow \Spec(V)$ of valuation rings.
	
Next, pick a non-zero element $\pi \in \frakm \subset V$, then the image of $\pi$ under the map $V \rightarrow V^{\prime}$ is a non-zero non-unit element of $V^{\prime}$. Let $\mathfrak{q}_0 = \cap \pi^n V^{\prime}$ and $\mathfrak{q}_1 = \sqrt{\pi V^\prime}$, which are prime ideals of $V^{\prime}$ (\ref{prime_ideals_valuation_rings}). Set $W = (V^\prime/\mathfrak{q}_0)_{\mathfrak{q}_1}$, so we get a map $V \rightarrow W$ of rank 1 valuation rings, which send $\pi \in V$ to non-zero element of the maximal ideal of $W$. This shows that we have a faithfully flat map $\Spec(W) \rightarrow \Spec(V)$ of rank one valuation rings, with the appropriate lifting properties. The result follows.
\end{proof}

\begin{rem}\label{canonical_arc_covers} Any qcqs scheme $X$ admits an $\arc$-cover $Y \rightarrow X$, where $Y$ is the spectrum of a product of absolutely integrally closed valuation rings of Krull dimension $\le 1$. Indeed, we define an equivalence class of valuation rings of Krull dimension $\le 1$ mapping to $X$ as follows: we say that $\Spec(V_1) \rightarrow X$ and $\Spec(V_2) \rightarrow X$ are in the same equivalence class if there is a third map $\Spec(V) \rightarrow X$ such that the map $\Spec(V_i) \rightarrow X$ factors as $\Spec(V_i) \rightarrow \Spec(V) \rightarrow X$. Denote by $\Val_X$ this equivalence class of valuation rings of Krull dimension $\le 1$ mapping to $X$. Picking a absolutely integrally closed representative element $V_x$ for each class $x \in \Val_X$ we obtain an $\arc$-covering
\begin{equation*}
	\Spec(\prod_{x \in \Val_X} V_x) \rightarrow X 
\end{equation*}
Finally, for the sake of completeness let us show that this cover $\Spec(\prod_{x \in \Val_X} V_x) \rightarrow X$ is compatible with the choice of an uncountable strong limit cardinal $\kappa$. We say that $X$ has cardinality $< \kappa$ if locally on the Zariski topology $X$ is of the form $\Spec(R)$ where $R$ has cardinality $\le \lambda < \kappa$. Then, there are $\le \lambda$ equivalence classes in $\Val_X$, and each equivalence class admits an absolutely integrally closed representative of size $\le \lambda$, and the product $\prod_{x \in \Val_X} V_x$ has cardinality $\le \lambda^\lambda \le 2^{\lambda \times \lambda} \le 2^{2^\lambda} < \kappa$.
\end{rem}

\begin{defn}\label{defn_pi_complete_arc} Let $V$ be a rank one valuation ring (cf. \ref{arch_rank_one_valuations}), and fix (possibly zero) $\varpi \in V$. A map of qcqs $V$-schemes $f: Y \rightarrow X$ is an $\varpi$-complete $\arc$-cover if for any $\varpi$-complete $V$-valuation ring $W_1$ of Krull dimension $\le 1$ and a map $\Spf(W_1) \rightarrow X$, there exists an extension $\Spf(W_2) \rightarrow \Spf(W_1)$ of $\varpi$-complete valuation rings of Krull dimension $\le 1$ and a map $\Spf(W_2) \rightarrow Y$ making the following diagram commute
\begin{cd}
	\Spf(W_2) \ar[r] \ar[d] & Y \ar[d] \\
	\Spf(W_1) \ar[r] & X
\end{cd}
It is easy to check that the collection of $\varpi$-complete $\arc$-covers in the category $\Sch_{V,\qcqs}$ (resp. in the category of $\varpi$-complete qcqs $V$-schemes, denoted by $\Sch_{V, \qcqs, \varpi}^{\wedge}$) satisfy the conditions of \cite[Proposition A.3.2.1]{luriespectral}, giving rise to a finitary Grothendieck topology on $\Sch_{V, \qcqs}$ (resp. $\Sch_{V, \qcqs, \varpi}^{\wedge}$) which we call the $\varpi$-complete $\arc$-topology.
\end{defn}

\begin{lemma}\label{arc_implies_pi_complete_arc} Let $V$ be a rank one valuation ring, and fix (possibly zero) $\varpi \in V$. If a map of qcqs $V$-schemes $f: X \rightarrow Y$ is an $\arc$-cover then it is a $\varpi$-complete $\arc$-cover.
\end{lemma}

\begin{proof} Let $W_1$ be a $\varpi$-complete valuation ring of Krull dimension $\le 1$ and a map $\Spf(W_1) \rightarrow X$. Then, since $f$ is an $\arc$-cover we know that there exists a valuation ring $W_2$ of Krull dimension $\le 1$ (not necessarily $\varpi$-complete) a extension $W_1 \rightarrow W_2$ and a map $\Spec(W_2) \rightarrow Y$ making the following diagram commute
\begin{cd}
	\Spec(W_2) \ar[r] \ar[d] & Y \ar[d] \\
	\Spf(W_1) \ar[r] & X
\end{cd}
There are two situations to consider: when $W_1$ has Krull dimension $1$, and when it has Krull dimension $0$ in which case $W_1$ is a field with $0 = \varpi \in W_1$. If $W_1$ has Krull dimension $1$ then $W_2$ must also have Krull dimension $1$, and by Lemma \ref{completion_rank_one_val} we learn that the $\varpi$-completion map $W_2 \rightarrow W_{2, \varpi}^{\wedge}$ is faithfully flat, showing that $W_1 \rightarrow W_{2, \varpi}^{\wedge}$ is faithfully flat and proving the claim. On the other hand if $W_1$ has Krull dimension $0$, then since $\varpi = 0$ it follows that $\varpi = 0 \in W_2$ showing that $W_2$ is already $\varpi$-complete. This completes the proof of the result.
\end{proof}

\begin{example} Let $R \rightarrow S$ be a $\varpi$-complete $\arc$-cover. Then $R \rightarrow S \times R[\frac{1}{\varpi}]$ is an $\arc$-cover.
\end{example}

\begin{rem}\label{canonical_pi_complete_arc_covers} Following Remark \ref{canonical_arc_covers} we claim that any qcqs $V$-scheme $X$ admits a $\varpi$-complete $\arc$-cover $Y \rightarrow X$, where $Y$ is a product of $\varpi$-complete absolutely integrally closed valuation rings of Krull dimension $\le 1$. Indeed, let $\Val_X$ be the equivalence class of valuation rings of Krull dimension $\le 1$ mapping to $X$, pick an absolutely integrally closed representative element $V_x$ for each class $x \in \Val_X$, and denote by $V_x^{\wedge}$ its $\varpi$-completion. We obtain the following $\varpi$-complete $\arc$-cover
\begin{equation*}
	\Spf(\prod_{x \in \Val_X} V_x^{\wedge}) \rightarrow X
\end{equation*}
Furthermore, if $X$ has cardinality $< \kappa$ for an uncountable strong limit cardinal it follows that each $V_x^{\wedge}$ also has cardinality $< \kappa$, and so does the product $\prod_{x \in \Val_X} V_x^{\wedge}$.
\end{rem}

\begin{rem} Fixing $V = \ZZ_p$ and $\varpi = p \in \ZZ_p$, we obtain from Definition \ref{defn_pi_complete_arc} a notion of the $p$-complete $\arc$-topology on the category of $p$-complete qcqs schemes. This definition differs slightly from the one given in \cite[Definition 8.7]{prisms}, as in the latter all valuation rings are assumed to be of Krull dimension $1$. We claim that if $f: Y \rightarrow X$ is a morphisms of $p$-complete qcqs schemes, then it is an $p$-complete $\arc$-cover in the sense of Definition \ref{defn_pi_complete_arc} if and only if it is a $p$-complete $\arc$-cover in the sense of \cite[Definition 8.7]{prisms}. It suffices to show that if $f: Y \rightarrow X$ is a $p$-complete $\arc$-cover in the sense of \cite[Definition 8.7]{prisms} and $W_1$ has Krull dimension $0$ (in particular $W_1$ is a field with $0 = p \in W_1$) there exists a faithfully flat map $W_1 \rightarrow W_2$ of $p$-complete valuation rings of Krull dimension $\le 1$ and a map $\Spf(W_2) \rightarrow Y$ making the appropriate diagram commute. Indeed, since $W_1$ is a field then $W_1[[t]]$ is a $p$-complete valuation ring of Krull dimension $1$, and the canonical map $W_1 \rightarrow W_1[[t]]$ induces a map
\begin{equation*}
	\Spf(W_1[[t]]) \rightarrow \Spf(W_1) \rightarrow X
\end{equation*}
Then, since $f: Y \rightarrow X$ is a $p$-complete $\arc$-cover in the sense of \cite[Definition 8.7]{prisms} we know that there exists a faithfully flat map $W_1[[t]] \rightarrow W_2$ of $p$-complete valuation rings of Krull dimension $1$, and a map $\Spf(W_2) \rightarrow Y$ making the following diagram commute
\begin{cd}
	\Spf(W_2) \ar[rr] \ar[d] && Y \ar[d] \\
	\Spf(W_1[[t]]) \ar[r] & \Spf(W_1) \ar[r] & X
\end{cd}
The claim follows from the fact that $W_1$ is a field and so $\Spf(W_2) \rightarrow \Spf(W_1)$ is faithfully flat.
\end{rem}

\begin{defn}\label{defn_arc_pi} Let $V$ be a rank one valuation ring (cf. \ref{arch_rank_one_valuations}), and fix a non-zero $\varpi \in V$. A map of qcqs $V$-schemes $f: Y \rightarrow X$ is an $\arc_{\varpi}$-cover if for any $\varpi$-complete $V$-valuation ring $W_1$ of Krull dimension $\le 1$ with a faithfully flat structure map $V \rightarrow W_1$, and a map $\Spf(W_1) \rightarrow X$, there exists an extension $\Spf(W_2) \rightarrow \Spf(W_1)$ of $\varpi$-complete $V$-valuation rings of Krull dimension $\le 1$ and a map $\Spf(W_2) \rightarrow Y$ making the following diagram commute
\begin{cd}
	\Spf(W_2) \ar[r] \ar[d] & Y \ar[d] \\
	\Spf(W_1) \ar[r] & X
\end{cd}
The additional hypothesis that the structure map $V \rightarrow W_1$ is a faithfully flat imposes the condition that $0 \not=\varpi \in W_1, W_2$ (cf. \ref{equiv_extensions_rk_one}), and so it rules out the possibility that $W_1$ or $W_2$ have Krull dimension zero. Its clear that the definition of $\arc_\varpi$-cover does not depend on the choice of $\varpi \in \frakm_V \subset V$, and that if $f: Y \rightarrow X$ is a $\varpi$-complete $\arc$-cover then it is also a $\arc_\varpi$-cover.

It is easy to check that the collection of  $\arc_{\varpi}$-covers in the category $\Sch_{V,\qcqs}$ (resp. $\Sch_{V, \qcqs, \varpi}^{\wedge}$) satisfy the conditions of \cite[Proposition A.3.2.1]{luriespectral}, giving rise to a finitary Grothendieck topology on $\Sch_{V, \qcqs}$ (resp. $\Sch_{V, \qcqs, \varpi}^{\wedge}$) which we call the $\arc_{\varpi}$-topology.
\end{defn}

\begin{example} Let $R \rightarrow S$ be a $\arc_{\varpi}$-cover. Then $R \rightarrow S \times R/\varpi$ is a $\varpi$-complete $\arc$-cover, and $R \rightarrow S \times R/\varpi \times R[\frac{1}{\varpi}]$ is an $\arc$-cover.
\end{example}

\begin{rem}\label{canonical_arc_pi_covers} Following Remark \ref{canonical_arc_covers} we claim that any $V$ scheme $X$ admits an $\arc_{\varpi}$-cover $Y \rightarrow X$, where $Y$ is a product of $\varpi$-complete absolutely integrally closed rank one valuation rings with $\varpi \not= 0$. Indeed, let $\Val_X$ be the equivalence class of valuation rings (of Krull dimension $\le 1$) mapping to $X$, and let $\Val_{X, \varpi} \subset \Val_{X}$ the subset of equivalence classes that admit a representative where the valuation rings satisfies $\varpi \not= 0$. Then the map
\begin{equation*}
	\Spf(\prod_{x \in \Val_{X, \varpi}} V_x^{\wedge}) \rightarrow X
\end{equation*}
is the desired $\arc_\varpi$-cover. Furthermore, if $X$ has cardinality $< \kappa$ for an uncountable strong limit cardinal, then so does $\prod_{x \in \Val_{X, \varpi}} V_x^{\wedge}$.
\end{rem}

\begin{rem} In the definition of $\arc_{\varpi}$-covers given in \cite[Definition 6.14]{arc_topology} it is not required that the rank one valuation rings involved are $\varpi$-complete. However, due to Lemma \ref{completion_rank_one_val} we know that a map $f: X \rightarrow Y$ of $\varpi$-complete qcqs $V$-schemes is an $\arc_{\varpi}$-cover in their sense if and only if it is an $\arc_{\varpi}$-cover in the sense of Definition \ref{defn_arc_pi}.
\end{rem}

Next, we provide a quick review of the sheaf axiom in the setting of $\infty$-categories. For the sake of simplicity we will often restrict ourselves to sheaves of the form $F: \Sch_{V, \qcqs}^{\op} \rightarrow \cC$, where $\cC$ is an $\infty$-category, however most of the following discussion works more generally for sheaves on a finitary site -- we refer the reader to \cite[Section A.3]{luriespectral} for a more detailed discussion.

\begin{notation} We let $\Delta$ be the category of simplices: the objects are finite linearly ordered sets $[n] = \{0 < 1 < \cdots < n\}$ for $n >0$, and the morphisms in $\Delta$ are nondecreasing functions. Similarly, we define the augmented simplicial category $\Delta_+$ as the category obtained from $\Delta$ by formally adding a initial object. Let $\cC$ be an $\infty$-category, a cosimplicial object in $\cC$ is a functor $X^{\bullet}: \Delta \rightarrow \cC$. The limit of this diagram is called the totalization of $X^{\bullet}$ and denote it by $\Tot(X^{\bullet}) \in \cC$, assuming the limit exists.

An important source of examples of cosimplicial objects comes from the Cech nerve construction. For example, fix a rank one valuation ring $V$ and a non-zero element $\varpi \in V$, then for any morphism $R \rightarrow S$ in $\CAlg_{V}$ (resp. $\CAlg_{V, \varpi}^{\wedge}$) we obtain the following cosimplicial object
\begin{cd}[column sep=small, row sep = small]
	& S \ar[r, shift left = 1.5] \ar[r, shift right = 1.5]
	& \ar[l] S \otimes_R S \ar[r, shift left = 3] \ar[r] \ar[r, shift right = 3]
	& \ar[l, shift left = 1.5] \ar[l, shift right = 1.5] S \otimes_R S \otimes_R S \cdots \\
	\text{resp. }  
	& S \ar[r, shift left = 1.5] \ar[r, shift right = 1.5]
	& \ar[l] S \cotimes_R S \ar[r, shift left = 3] \ar[r] \ar[r, shift right = 3]
	& \ar[l, shift left = 1.5] \ar[l, shift right = 1.5] S \cotimes_R S \cotimes_R S \cdots
\end{cd}
which we denote by $S^{\bullet/R}: \Delta \rightarrow \CAlg_V$ (resp. $S^{\bullet/R}: \Delta \rightarrow \CAlg_{V, \varpi}^{\wedge}$).

Finally, let us introduce the notion of partial totalization. For each $n \ge 0$ we denote by $\Tot_n(X^{\bullet})$ the limit of the $n$-truncated cosimplicial object $\Delta_{\le n} \hookrightarrow \Delta \overset{X^{\bullet}} \rightarrow \cC$. Then, by commuting limits with limits we learn that $\Tot (X^{\bullet}) = \lim_n \Tot_n(X^\bullet)$.
\end{notation}

\begin{defn}\label{defn_infty_sheaf} Let $F: \Sch_{V, \qcqs}^{\op} \rightarrow \cC$ (resp. $F:\Sch_{V, \qcqs, \varpi}^{\wedge, \op} \rightarrow \cC$) be a presheaf valued in an $\infty$-category $\cC$. We will say that $F$ satisfies descent for a morphism $Y \rightarrow X$ in $\Sch_{V, \qcqs}$ (resp. $\Sch_{V, \qcqs, \varpi}^{\wedge}$) if it satisfies the $\infty$-categorical sheaf axioms with respect to $Y \rightarrow X$, in other words, if the natural map
\begin{align*}
	F(X) \rightarrow \Tot (F(Y^{\bullet/X})) && \text{where} \qquad Y^{n/X} = Y \times_X \cdots \times_X Y \qquad n \text{ times}
\end{align*}
is an equivalence. If this property holds for all maps $f: Y \rightarrow X$ that are covers in $\tau$, for $\tau \in \{ \arc, \varpi\text{-complete} \arc, \arc_{\varpi} \}$, and further if $F$ carries finite disjoint union to finite products, then we shall say that $F$ satisfies $\tau$-descent or is a $\tau$-sheaf.
\end{defn}

\begin{rem} By \cite[Proposition A.3.3.1]{luriespectral} it follows that the notion of $\tau$-sheaf given above is equivalent to the notion of $\cC$-valued sheaf with respect to $\tau$ given in \cite[Sect A.3.3]{luriespectral}, where the sheaf axiom is formulated in terms of covering sieves as opposed to Cech descent.
\end{rem}

\begin{defn}\label{defn_hypercomplete_sheaf} Endow $\Sch_{V, \qcqs}$ with a Grothendieck topology $\tau$. A hypercovering in $\Sch_{V, \qcqs}$ is an augmented simplicial object $X_{-1} \rightarrow X_{\bullet}$ in $\Sch_{V, \qcqs}$ such that for every $n \ge 0$ the canonical map
\begin{equation*}
	X_n \rightarrow (\cosk_{n-1}(X_{\bullet} \rightarrow X_{-1}))_n
\end{equation*}
is an $\tau$-cover. Here $\cosk_n$ denotes the augmented coskeleton functor, defined as the right Kan extension along the inclusion $\Delta_{+, \le n}^{\op} \hookrightarrow \Delta_{+}^{\op}$.

Let $F: \Sch_{V,\qcqs}^{\op} \rightarrow \cC$ be a presheaf valued in the $\infty$-category $\cC$. We say that $F$ is a hypercomplete $\cC$-valued sheaf if and only if the following properties are satisfied:
\begin{enumerate}[(1)]
	\item $F$ preserved finite products
	\item For every hypercover $U_{\bullet} \rightarrow X$ in $\Sch_{V,\qcqs}$, the canonical morphism
	\begin{equation*}
		F(X) \rightarrow \Tot(U_{\bullet})
	\end{equation*}
	is an isomorphism in $\cC$.
\end{enumerate}
See \cite[Proposition A.5.7.2]{luriespectral} and \cite[Section A.3]{mann_six_funct_rig} for comparisons with alternative definitions of hypercomplete sheaves\footnote{One advantage of the alternative definitions of (hypercomplete) sheaf $F: \cC \rightarrow \cD$ is that there are no hypothesis on the category $\cC$, which will be useful when considering (hypercomplete) sheaves on a basis $\cB \subset \cC$ (cf. the hypothesis on \cite[Section A.3.3]{luriespectral}), for example $\cB \subset \cC$ may not be closed under finite limits.}.

The main advantage of hypercomplete sheaves is that they are determined by its values on a basis $\cB \subset \Sch_{V, \qcqs}$. A basis of $\Sch_{V, \qcqs}$ is a full-subcategory $\cB \subset \Sch_{V, \qcqs}$ such that for every element $X \in \Sch_{V,\qcqs}$, admits a $\tau$-covering $Y \rightarrow X$, with $Y \in \cB$. Then, we learn from \cite[Proposition A.3.11]{mann_six_funct_rig} that the restriction functor
\begin{align*}
	\Shv_{\tau} (\Sch_{V, \qcqs}, \cD)^{\wedge} \rightarrow \Shv_{\tau}(\cB, \cD)^{\wedge} && F \mapsto F|_{\cB}
\end{align*}
determines an equivalence of categories of hypercomplete $\cD$-valued $\tau$-sheaves. The inverse of this equivalence if given right Kan extension along $\cB \subset \Sch_{V, \qcqs}$.
\end{defn}

In what follows we make some light use of the notion of an $\infty$-category $\cC$ compactly generated by cotruncated objects, we refer the reader to \cite[Definition 3.1.4]{cdh_descent} and \cite[Definition 3.5]{arc_topology} for the basic definitions. The most relevant examples of $\infty$-categories compactly generated by cotruncated objects for us are the derived $\infty$-category of coconnective objects $\cD(R)^{\ge 0}$ of a ring $R$ and the $\infty$-category $\Spc_{\le n}$ of $n$-truncated spaces.

\begin{prop}\label{sheaf_implies_hypersheaf} Let $\cD$ be a $\infty$-category generated by cotruncated objects, for example $\cD = \cD(A)^{\ge d}$ or $\cD = \Set$, and let $\cC$ be a $1$-category equipped with a Grothendieck topology $\tau$. If $F: \cC^{\op} \rightarrow \cD$ is a $\tau$-sheaf, then $F$ is a hypercomplete $\tau$-sheaf.
\end{prop}

\begin{proof} By the Yoneda Embedding (cf. \cite[Proposition 5.1.3.2]{lurieHTT}), it suffices to show that for every compact object $D \in \cD$ the composition $h^D \circ F: \cC^{\op} \rightarrow \Spc$, where $h^D := \Maps_{\cD}(D, -): \cD \rightarrow \Spc$, is a hypercomplete $\tau$-sheaf. By the hypothesis that $\cD$ is generated by cotruncated objects, it follows that $h^D \circ F$ has its image contained in $\Spc_{\le n}$, then it follows that $h^D \circ F$ is hypercomplete by \cite[Lemma 6.5.2.9]{lurieHTT}, proving the desired result.
\end{proof}

\begin{example} Since Zariski covers are $\tau$-covers (\ref{faithfully_flat_implies_arc_cover}) for $\tau \in \{ \arc, \varpi\text{-complete} \arc, \arc_{\varpi} \}$, it follows that any hypercomplete $\tau$-sheaf on $\Sch_{V, \qcqs}$ is completely determined by its values on $\CAlg_{V}^{\op} \subset \Sch_{V, \qcqs}$. Even better, if $\tau \in \{\varpi\text{-complete} \arc, \arc_{\varpi} \}$, then we learn from Remarks \ref{canonical_arc_pi_covers} and \ref{canonical_arc_pi_covers} that $\CAlg_{V}$ admits a $\tau$-basis given by $\Perfd_{V}^{\Prism} \subset \CAlg_{V}$.
\end{example}

\begin{prop}\label{BK_formula} Let $F: \Sch_{V, \qcqs}^{\op} \rightarrow \cD(A)$ (resp. $F:\Sch_{V, \qcqs, \varpi}^{\wedge, \op} \rightarrow \cD(A)$) be a presheaf valued in the derived $\infty$-category $\cD(A)$ for some ring $A$. Let $X^{\bullet}: \Delta \rightarrow \cD(A)$ be a cosimplicial object, and $X^{-1} \rightarrow X^{\bullet}$ an augmented cosimplicial object; assume that $F(X^{-1})$ and $F(X^n)$ are contained in $\Mod_A \subset \cD(A)$. Then, the canonical map
\begin{equation*}
	F(X^{-1}) \rightarrow \Tot (F(X^n))
\end{equation*}
is an equivalence, if and only if the following complex
\begin{equation*}
	0 \rightarrow F(X^{-1}) \rightarrow F(X^0) \rightarrow F(X^1) \rightarrow \cdots
\end{equation*}
is acyclic. Where the complex corresponds to the cosimplicial object $F(X^{\bullet}):\Delta \rightarrow \cD(A)$ under the Dold-Kan correspondence.
\end{prop}

\begin{proof} See \cite[Section 1.2.3]{lurieHA} for a review of the simplicial Dold-Kan correspondence, and \cite[Tag 019H]{stacks-project} for the corresponding cosimplicial (and logically equivalent) Dold-Kan correspondence. The claim is then the Bousfield-Kan formula \cite[Chapter 11]{bousfield_kan}, or \cite[Proposition 1.2.4.5]{lurieHA} for a simplicial $\infty$-categorical variant.
\end{proof}

\begin{prop}\label{commute_Tot} Let $\cC$ and $\cD$ be stable $\infty$-categories with endowed with a right-complete $t$-structure, $X^{\bullet}$ a cosimplicial object of $\cC^{\ge 0} \subset \cC$, and $F: \cC \rightarrow \cD$ a exact functor of stable $\infty$-categories. If $F(X^{\bullet}) \subset \cD^{\ge d}$ for some $d \in \ZZ$, then $\Tot(X^{\bullet})$ and $\Tot F(X^{\bullet})$ exist and the canonical map
\begin{equation*}
	F(\Tot(X^{\bullet})) \rightarrow \Tot F(X^{\bullet}) 
\end{equation*}
is an isomorphism in $\cD$.
\end{prop}

\begin{proof} Since $F$ is an exact functor of stable $\infty$-categories, by definition we know that it commutes with finite limits, hence the canonical map $F(\Tot_n(X^{\bullet})) \rightarrow \Tot_n F(X^{\bullet})$ is an isomorphism. Using the fact that $\Tot_n$ and $\Tot$ agree on $H^i(-)$ for $i < n$ (cf. \cite[Proposition 1.2.4.5]{lurieHA}) the result follows. The hypothesis that $X^{\bullet} \subset \cC^{\ge 0}$ and that $F(X^{\bullet}) \subset \cD^{\ge d}$ are there to ensure that we can apply \cite[Proposition 1.2.4.5]{lurieHA} to the problem at hand, while the right completeness allows us to check isomorphisms at the level of cohomology groups.
\end{proof}

\begin{prop}\label{perfection_arc_sheaf} Let $V$ be a ring of characteristic $p$. Then, the functor
\begin{align*}
	F: \Sch_{V, \qcqs}^{\op} \rightarrow \cD(V) && X \mapsto R\Gamma(X_{\perf}, \cO_{X, \perf})
\end{align*}
is a hypercomplete $\arc$-sheaf. In particular, if $A \rightarrow B$ is an $\arc$-cover of perfect $V$-algebras, the following complex is acyclic
\begin{align*}
	0 \rightarrow A \rightarrow B \rightarrow B \otimes_A B \rightarrow B \otimes_A B \otimes_A B \rightarrow \cdots
\end{align*}
\end{prop}

\begin{proof} Combining \cite[Theorem 4.1(a)]{projectivity_affine_grass} with Proposition \ref{sheaf_implies_hypersheaf} we learn that the functor
\begin{align*}
	(-)_{\perf}: \Perf_{V} \rightarrow \cD(V) && A \mapsto A
\end{align*}
is a hypercomplete $v$-sheaf. Then, as the map $B \rightarrow B_{\perf}$ from $B \in \CAlg_{V}$ is a $v$-cover and $B_{\perf} \in \CAlg_{V}$, we learn that $\Perf^{\op}_{V} \subset \Sch_{V, \qcqs}$ determines a basis for the $v$-topology, and so by right Kan extending the functor $(-)_{\perf}$ along $\Perf_{V}^{\op} \hookrightarrow \Sch_{V, \qcqs}$ we obtain a hypercomplete $v$-sheaf
\begin{align*}
	F: \Sch_{V, \qcqs}^{\op} \rightarrow \cD(V) && X \mapsto R\Gamma(X_{\perf}, \cO_{X, \perf})
\end{align*}
(cf. \ref{defn_hypercomplete_sheaf}). Given that the functor $F$ takes values in $\cD(V)^{\ge 0} \subset \cD(V)$ we learn from \cite[Theorem 4.1]{arc_topology} that to check that $F$ is an $\arc$-sheaf it suffices to show that for any absolutely integrally closed $V$-valuation ring $W$ and a prime ideal $\frakp \subset W$ we have an exact sequence
\begin{equation*}
	0 \rightarrow W_{\perf} \rightarrow (W/\frakp)_{\perf} \oplus (W_{\frakp})_{\perf} \rightarrow \kappa(\frakp)_{\perf} \rightarrow 0
\end{equation*}
However, since absolutely integrally closed valuation rings are perfect, we have the identity $W = W_{\perf}$; furthermore, we learn from \cite[Tag 0DCN]{stacks-project} and \cite[Tag 088Y]{stacks-project} that if $W$ is an absolutely integrally closed valuation ring so is $W_{\frakp}$ and $W/\frakp$, so we have the identities
\begin{align*}
	(W/\frakp)_{\perf} \simeq W/\frakp && (W_{\frakp})_{\perf} \simeq W_{\frakp} && \kappa(\frakp)_{\perf} \simeq \kappa(\frakp)
\end{align*}
Hence, it remains to show that the following is an exact sequence
\begin{equation*}
	0 \rightarrow W \rightarrow W/\frakp \oplus W_{\frakp} \rightarrow \kappa(\frakp) \rightarrow 0
\end{equation*}
which was proven in \cite[Lemma 6.3]{projectivity_affine_grass}. Hypercompleteness follows from the fact that $F$ takes values on $\cD(V)^{\ge 0} \subset \cD(V)$ (\ref{sheaf_implies_hypersheaf}).
\end{proof}

Finally, we show that under the equivalence \ref{equiv_almost_tf_banach} a map $\cM(S) \rightarrow \cM(R)$ in $\Ban_K^{\contr}$ is surjective if and only if $R \rightarrow S$ is an $\arc_{\varpi}$-cover.

\begin{prop}\label{arc_pi_cover_for_banach} Let $K$ be a perfectoid non-archimedean field and $\varpi \in K$ such that $1 > |\varpi^p| > |p|$, in particular $K_{\le 1}$ is a perfectoid rank one valuation ring (cf. Example \ref{non_arch_field_val_rank_one} and Lemma \ref{perfectoid_field_is_int_perfectoid}). For a morphism $R \rightarrow S$ in $\Ban_K^{\contr}$ the following are equivalent
\begin{enumerate}[(1)]
	\item The induced map of compact hausdorff spaces $|\cM(S)| \rightarrow |\cM(R)|$ is surjective.
	\item The induced map $\Spf(S_{\le 1}) \rightarrow \Spf(R_{\le 1})$ is an $\arc_{\varpi}$-cover.
\end{enumerate}
\end{prop}

\begin{proof} We begin by proving the following: a map $|\cM(S)| \rightarrow |\cM(R)|$ is surjective if and only if for any $K$-non-archimedean field $L$ and a map $\cM(L) \rightarrow \cM(R)$, there exists a morphism $L \rightarrow F$ of non-archimedean fields and a map $\cM(F) \rightarrow \cM(S)$ making the following diagram commute
\begin{cd}
	\cM(F) \ar[r] \ar[d] & \cM(S) \ar[d] \\
	\cM(L) \ar[r] & \cM(R)
\end{cd}
Indeed, if the condition on non-archimedean fields hold then its clear that the map $|\cM(S)| \rightarrow |\cM(R)|$ is surjective. On the other hand, assume that $|\cM(S)| \rightarrow |\cM(R)|$ is surjective, recall that universal property of the map $\cM(\cH(x)) \rightarrow \cM(R)$ (cf. \ref{univ_completed_residue}): for any map $\cM(P) \rightarrow \cM(R)$ from a non-archimedean field $P$ with image $x \in |\cM(R)|$, then the map $\cM(P) \rightarrow \cM(R)$ admits a unique factorization as $\cM(P) \rightarrow \cM(\cH(x)) \rightarrow \cM(R)$. Fix $x \in |\cM(R)|$ and $y \in |\cM(S)|$ such that $y \mapsto x$, then by the universal property of completed residue fields we learn that we have a commutative diagram
\begin{cd}
	\cM(\cH(y)) \ar[r] \ar[d] & \cM(S) \ar[d] \\
	\cM(\cH(x)) \ar[r] & \cM(R)
\end{cd}
Without loss of generality assume that $|\cM(L)| \rightarrow |\cM(R)|$ has image $x \in |\cM(R)|$, then it factors through $\cM(L) \rightarrow \cM(\cH(x))$. By \ref{non_zero_comp_tensor_ban} we learn that $Z:= L \cotimes_{\cH(x)} \cH(y)$ is non-zero, and so $|\cM(Z)|$ is non-empty by \ref{berko_sp_is_comp}, so we fix $z \in |\cM(Z)|$ and we obtain a commutative diagram
\begin{cd}
	\cM(\cH(z)) \ar[r] \ar[rd] & \cM(Z) \ar[r] \ar[d] & \cM(\cH(y)) \ar[r] \ar[d] & \cM(S) \ar[d] \\
	&\cM(L) \ar[r] & \cM(\cH(x)) \ar[r] & \cM(R)
\end{cd}
proving the desired implication. The statement of the proposition then follows from Lemma \ref{na_field_rk_one_dict}.
\end{proof}

\subsection{Perfectoidization and \texorpdfstring{$\arc_{\varpi}$}{arc pi}-descent}

Fix a perfectoid non-archimedean field $K$ and let $(A,d)$ the perfect prism corresponding to $K_{\le 1}$ with frobenius lift $\varphi_A: A \rightarrow A$.

\begin{defn} In \cite[Section 7.2]{prisms} Bhatt and Scholze introduced a cohomology theory called derived prismatic cohomology
\begin{align*}
	\Prism_{-}: \CAlg_{K_{\le 1}, p}^{\wedge} \longrightarrow \cD(A)_{(p, d)}^{\wedge} && R \mapsto \Prism_R
\end{align*}
which associates to any derived $p$-complete $K_{\le 1}$-algebra a derived $(p, d)$-complete commutative algebra object in $\cD(A)^{\wedge}_{(p,d)}$ equipped with a $\varphi_A$-semilinear endomorphism $\varphi: \Prism_R \rightarrow \Prism_R$. For details and various properties of this construction we refer the reader to \cite[Section 7.2]{prisms}.
\end{defn}

\begin{rem} Notice that unlike Bhatt and Scholze we denote the derived prismatic cohomology of $R \in \CAlg_{K_{\le 1},p}^{\wedge}$ by $\Prism_R$ as opposed to $\Prism_{R/A}$, this is due to the fact that since we are working over a perfect prism $(A,d)$ derived prismatic cohomology does not depend on the base $A$. Indeed, for any $R \in \CAlg_{K_{\le 1}, p}^{\wedge}$ the canonical map of sites $(R/A)_{\Prism} \rightarrow (R)_{\Prism}$ is an equivalence (cf. \cite[Lemma 4.8]{prisms}) between the absolute prismatic site of $R$ and the prismatic site of $R/A$ -- in particular we learn that we have a canonical isomorphism $R\Gamma (R_{\Prism}, \cO_{\Prism}) \simeq  R\Gamma ((R/A)_{\Prism}, \cO_{\Prism})$. If we further assume that $R$ is smooth we know from \cite[Construction 7.6]{prisms} that there is a canonical equivalence $\Prism_{R/A} \simeq R\Gamma ((R/A)_{\Prism}, \cO_{\Prism})$, where $\Prism_{R/A}$ is the derived prismatic cohomology of $(R/A)$, showing in particular that the derived prismatic cohomology of a smooth $R \in \CAlg_{K_{\le 1}, p}^{\wedge}$ is independent of the base $A$, and we choose to denote it by $\Prism_R$ to emphasize this independence. Finally, as the derived prismatic cohomology of a general $R \in \CAlg_{K_{\le 1}, p}^{\wedge}$ is defined as the left Kan extension from the subcategory of polynomial algebras, it follows that $\Prism_{R/A}$ is independent of the base $A$, and so we choose to denote it by $\Prism_R$.
\end{rem}

\begin{defn} Fix a $p$-complete $K_{\le 1}$-algebra $R$.
\begin{enumerate}[(1)]
	\item The perfection $\Prism_{R, \perf}$ of $\Prism_R$ is defined as
	\begin{align*}
		\Prism_{R, \perf} := \colim_{\varphi_R} \Big (\Prism_{R} \rightarrow \varphi_{A*} \Prism_R \rightarrow \varphi_{A*}^{2} \Prism_R \rightarrow \cdots  \Big )^{\wedge}_{(p,d)} \in \cD(A)^{\wedge}_{(p,d)}
	\end{align*}
	the derived $(p,d)$-complete $E_{\infty}$-$A$-algebra equipped with a $\varphi_A$-semilinear automorphism induced by $\varphi_R$.
	\item The perfectoidization $R_{\perfd}$ of $R$ is defined as
	\begin{align*}
		R_{\perfd} := \Prism_{R, \perf} \otimes_A^L A/d \in \cD(K_{\le 1})^{\wedge}_{p}
	\end{align*}
	This is a derived $p$-complete $E_{\infty}$-$K_{\le 1}$-algebra.
\end{enumerate}
\end{defn}

\begin{prop}\label{properties_perfd} Fix a $p$-complete $K_{\le 1}$-algebra $R$. Then,
\begin{enumerate}[(1)]
	\item Assume that $0 = p \in R$, then $R_{\perfd}$ coincides with the usual direct limit perfection of $R_{\perf}$ of $R$.
	\item Both $\Prism_{R, \perf}$ and $R_{\perfd}$ lie in $\cD^{\ge 0}(A)$ and $\cD^{\ge 0}(K_{\le 1})$ respectively.
	\item For any choice of topology on $(R)_{\Prism}^{\perf}$ between the flat and chaotic topology, there exists a canonical map
	\begin{align*}
		\Prism_{R, \perf} \rightarrow R\Gamma(R_{\Prism}^{\perf}, \cO_{\Prism})
	\end{align*}
	which is an equivalence. In particular, the perfectoidization
	\begin{align*}
		R_{\perfd} \simeq R\Gamma (R_{\Prism}^{\perf}, \overline{\cO}_{\Prism}) \simeq \lim_{R \rightarrow S} S
	\end{align*}
	is the derived limit of $S$ over all maps from $R$ to a perfectoid ring $S$, and does not depend on the choice of base prism $(A, d)$.
	\item Assume that $R$ is semiperfectoid, that is, there exists a surjective map $P \twoheadrightarrow R$ from an integral perfectoid $K_{\le 1}$-algebra. Then, $\Prism_{R, \perf}$ is concentrated in degree zero, and $(\Prism_{R, \perf}, d)$ is the initial object of $R_{\Prism}^{\perf}$.
	\item In general, $\Prism_{-, \perf}$ and $(-)_{\perfd}$ are defined for derived $p$-complete simplicial commutative $K_{\le 1}$-algebras $R$. In those cases, $\Prism_{R, \perf}$ and $R_{\perfd}$ only depend on $\pi_0(R)$.
\end{enumerate}
\end{prop}

\begin{proof} (1) is \cite[Example 8.3]{prisms}, (2) is \cite[Lemma 8.4]{prisms}, and (3) is \cite[Proposition 8.5]{prisms}. To see (4) recall that in \cite[Lemma 7.2]{prisms} it is proven that if $R$ is semiperfectoid then $R_{\Prism}$ admits an initial object $(\Prism_R^{\init}, d)$, which by definition implies that $\Prism_R^{\init}$ is concentrated in degree zero. Then, \cite[Lemma 3.9]{prisms} shows that
\begin{equation*}
	\Prism_{R, \perf}^{\init} := \colim_{\varphi_R} \Big (\Prism_{R}^{\init} \rightarrow \varphi_{A*} \Prism_R^{\init} \rightarrow \varphi_{A*}^{2} \Prism_R^{\init} \rightarrow \cdots  \Big )^{\wedge}_{(p,d)}
\end{equation*}
is concentrated in degree zero and $(\Prism_{R, \perf}^{\init}, d)$ is the initial object of $R_{\Prism}^{\perf}$; the result then follows from the isomorphism $\Prism_{R, \perf} \rightarrow \Prism_{R, \perf}^{\init} = R\Gamma(R_{\Prism}^{\perf}, \cO_{\Prism})$ coming from (3). Finally, (5) follows from the fact the equivalence of sites $R_{\Prism}^{\perf} = (\pi_0(R))_{\Prism}^{\perf}$ and (3).
\end{proof}

Since the subcategory $\Perfd_{K_{\le 1}}^{\Prism} \subset \CAlg_{K_{\le 1}}^{\wedge}$ is stable under pullbacks and coproducts, it follows that $\Perfd_{K_{\le 1}}^{\Prism}$ inherits the $\varpi$-complete $\arc$-topology from $\CAlg_{K_{\le 1}}^{\wedge}$.

\begin{prop}\label{pi_complete_arc_descent_for_perfectoids} The functor
\begin{align*}
	F: \Perfd_{K_{\le 1}}^{\Prism} \rightarrow \cD(K_{\le 1})^{\wedge}_{\varpi} && R \mapsto R
\end{align*}
is a hypercomplete $\varpi$-complete $\arc$-sheaf. In particular, if $R \rightarrow S$ is a $\varpi$-complete $\arc$-cover then the following complex is acyclic
\begin{equation*}
	0 \rightarrow R \rightarrow S \rightarrow S \cotimes_R S \rightarrow S \cotimes_R S \cotimes_R S \rightarrow \cdots
\end{equation*}
\end{prop}

\begin{proof} We follow \cite[Proposition 8.10]{prisms}. We claim that it suffices to show that $F^{\flat}: \Perfd_{K_{\le 1}^{\flat}}^{\Prism} \rightarrow \cD(K_{\le 1}^{\flat})$ given by $R^\flat \mapsto R^\flat$ is a $\varpi^\flat$-complete $\arc$-sheaf. Indeed, its clear that to show that $F$ is a $\varpi$-complete $\arc$-sheaf it suffices to show that for any $\varpi$-complete $\arc$-cover $R \rightarrow S$, the canonical map
\begin{align*}
	R \rightarrow \Tot(S^{\bullet/R})^{\wedge}_{\varpi} && \text{in } \cD(K_{\le 1})^{\wedge}_{\varpi}
\end{align*}
is an equivalence. Since the functor $ - \otimes_{A} A/d: \cD(A)^{\wedge}_{([\varpi^\flat], p)} \rightarrow \cD(K_{\le 1})^{\wedge}_{\varpi}$ satisfies the hypothesis of \ref{commute_Tot} for all objects $\A_{\inf}(R)$ and $\A_{\inf}(S^{n/R})$ it follows that it suffices to check that the canonical map
\begin{align*}
	\A_{\inf}(R) \rightarrow \Tot(\A_{\inf}(S^{\bullet/R}))^{\wedge}_{([\varpi^\flat, p])} && \text{in } \cD(A)^{\wedge}_{([\varpi^\flat], p)}
\end{align*}
is an equivalence. The same argument as above together with the derived Nakayama lemma we conclude that it suffices to show that the canonical map
\begin{align*}
	R^{\flat} \rightarrow \Tot(S^{\flat \bullet/R^\flat})^{\wedge}_{\varpi^\flat} && \text{in } \cD(K_{\le 1}^\flat)^{\wedge}_{\varpi^\flat}
\end{align*}
is an equivalence. Since $R^\flat \rightarrow S^\flat$ is a $\varpi^\flat$-complete $\arc$-cover if and only if $R \rightarrow S$ is a $\varpi$-complete $\arc$-cover, by the tilting correspondence, we have reduced to showing that $F^{\flat}: \Perfd_{K_{\le 1}^{\flat}}^{\Prism} \rightarrow \cD(K_{\le 1}^{\flat})$ is a $\varpi^\flat$-complete $\arc$-sheaf.

For the rest of the proof we assume that $R$, $S$ and $K_{\le 1}$ are of characteristic $p$. Recall that if $R \rightarrow S$ is a $\varpi$-complete $\arc$-cover in $\Perfd_{K_{\le 1}}^{\Prism}$ then the map $R \rightarrow S \times R[\frac{1}{\varpi}]$ is an $\arc$-cover in $\Perfd_{K_{\le 1}}^{\Prism}$. And in Proposition \ref{perfection_arc_sheaf} we showed that the canonical map
\begin{align*}
	R \rightarrow \Tot((S \times R[\frac{1}{\varpi}])^{\bullet/R}) && \text{in } \cD(K_{\le 1})
\end{align*}
is an equivalence. Then, as the derived $\varpi$-completion functor $(-)_{\varpi}^{\wedge}: \cD(K_{\le 1}) \rightarrow \cD(K_{\le 1})^{\wedge}_{\varpi}$ agrees with the classical $\varpi$-completion and satisfies the hypothesis of \ref{commute_Tot} for all objects $R$ and $(S \times R[\frac{1}{\varpi}])^{n/R}$ it follows that the canonical map
\begin{align*}
	R \rightarrow \Tot \Big ((S \times R[\frac{1}{\varpi}])^{\bullet/R} \Big )^{\wedge}_{\varpi} && \text{in } \cD(K_{\le 1})^{\wedge}_{\varpi}
\end{align*}
is an equivalence. Finally, we notice that $(S \times R[\frac{1}{\varpi}])^{\wedge}_{\varpi} \simeq S$, and since the derived $\varpi$-completion functor $(-)_{\varpi}^{\wedge}: \cD(K_{\le 1}) \rightarrow \cD(K_{\le 1})^{\wedge}_{\varpi}$ is symmetric monoidal, it follows that $\Big ((S \times R[\frac{1}{\varpi}])^{\bullet/R} \Big )^{\wedge}_{\varpi} \simeq (S^{\bullet/R})^{\wedge}_{\varpi}$, proving the result. The fact that $F$ is a hypercomplete sheaf follows from the fact that it takes values on $\cD(K_{\le 1})^{\ge 0} \subset \cD(K_{\le 1})$ and \ref{sheaf_implies_hypersheaf}.
\end{proof}

\begin{thm}\label{arc_pi_descent_for_perfectoids} The functor
\begin{align*}
	F: \Perfd_{K_{\le 1}}^{\Prism} \rightarrow \cD(K_{\le 1})^{\wedge a}_{\varpi} && R \mapsto R^a
\end{align*}
is a hypercomplete $\arc_{\varpi}$-sheaf. In particular, if $R \rightarrow S$ is a $\arc_{\varpi}$-cover then the following complex is almost acyclic
\begin{align*}
	0 \rightarrow R^a \rightarrow S^a \rightarrow S^a \cotimes_{R^a}^a S^a \rightarrow S^a \cotimes_{R^a}^a S^a \cotimes_{R^a}^a S^a \rightarrow \cdots
\end{align*}
\end{thm}

\begin{proof} To show that $F$ is a $\arc_\varpi$-sheaf it suffices to show that for any $\arc_{\varpi}$-cover $R \rightarrow S$ the canonical map
\begin{align*}
	R \rightarrow \Tot(S^{\bullet/R})^{\wedge a}_{\varpi} && \text{in } \cD(K_{\le 1})_{\varpi}^{\wedge a}
\end{align*}
is an equivalence. We claim that if $R \rightarrow S$ is an $\arc_{\varpi}$-cover then $R \rightarrow S \times (R/\varpi)_{\perf}$ is a $\varpi$-complete $\arc$-cover; indeed, its clear that if $R \rightarrow S$ is an $\arc_{\varpi}$-cover then $R \rightarrow S \times R/\varpi$ is a $\varpi$-complete $\arc$-cover, but since $\varpi$-complete $\arc$-covers can be tested by $\varpi$-complete absolutely integrally closed valuation rings $V$ of Krull dimension $\le 1$, any map $S \times R/\varpi \rightarrow V$ will factor as
\begin{equation*}
	S \times R/\varpi \rightarrow S \times (R/\varpi)_{\perf} \rightarrow V
\end{equation*}
proving the desired claim. Therefore, we learn from Proposition \ref{pi_complete_arc_descent_for_perfectoids} that the canonical map
\begin{align*}
	R \rightarrow \Tot\Big ( (S \times (R/\varpi)_{\perf})^{\bullet/R}   \Big )^{\wedge}_{\varpi} && \text{in } \cD(K_{\le 1})^{\wedge}_{\varpi}
\end{align*}
is an equivalence. Then, since the functor $(-)^a: \cD(K_{\le 1})^{\wedge}_{\varpi} \rightarrow \cD(K_{\le 1})^{\wedge a}_{\varpi}$ preserves all limits it follows that the canonical map 
\begin{align*}
	R \rightarrow \Tot\Big ( (S \times (R/\varpi)_{\perf})^{\bullet/R}   \Big )^{\wedge a}_{\varpi} && \text{in } \cD(K_{\le 1})^{\wedge a}_{\varpi}
\end{align*}
is an equivalence. Finally, we notice that $(S \times (R/\varpi)_{\perf})^a \simeq S$ , and since the functor $(-)^a: \cD(K_{\le 1})^{\wedge}_{\varpi} \rightarrow \cD(K_{\le 1})^{\wedge a}_{\varpi}$ is symmetric monoidal it follows that $\Big ( (S \times (R/\varpi)_{\perf})^{\bullet/R}   \Big )^{\wedge a}_{\varpi} \simeq (S^{\bullet/R})^{\wedge a}_{\varpi}$, proving the result. The fact that $F$ is a hypercomplete sheaf follows from the fact that it takes values on $\cD(K_{\le 1})^{a, \ge 0} \subset \cD(K_{\le 1})^a$ and \ref{sheaf_implies_hypersheaf}.
\end{proof}

\begin{corollary}\label{arc_pi_descent_for_bananach_perfectoids} Let $R \rightarrow S$ be a morphism of $\Perfd_{K}^{\Ban}$, such that the induced map of compact Hausdorff spaces $|\cM(S)| \rightarrow |\cM(R)|$ is surjective. Then, the following complex is exact and admissible
\begin{align*}
	0 \rightarrow R \rightarrow S \rightarrow S \cotimes_R S \rightarrow S \cotimes_R S \cotimes_R S \rightarrow \cdots
\end{align*}
\end{corollary}

\begin{proof} Once exactness is established admissibility is automatic by the open mapping theorem, thus it suffices to show that the complex is exact. From Proposition \ref{arc_pi_cover_for_banach} we learn that the map $R_{\le 1} \rightarrow S_{\le 1}$ is an $\arc_{\varpi}$-cover in $\Perfd_{K_{\le 1}}^{\Prism}$, and so in particular the following complex is almost acyclic (equivalently, almost exact) by Proposition \ref{arc_pi_descent_for_perfectoids} 
\begin{align*}
	0 \rightarrow R_{\le 1} \rightarrow S_{\le 1} \rightarrow S_{\le 1} \cotimes_{R_{\le 1}}^a S_{\le 1} \rightarrow S_{\le 1} \cotimes_{R_{\le 1}}^a S_{\le 1} \cotimes_{R_{\le 1}}^a S_{\le 1} \rightarrow \cdots && \text{in } \cD(K_{\le 1})^{\wedge a}_{\varpi}
\end{align*}
where we are implicitly using Propositions \ref{equiv_perfd_ban_tic} and \ref{equiv_perfd_tic_almost} to conclude that $R_{\le 1} = R_{\le 1}^a$ and $S_{\le 1} = S_{\le 1}^a$, and the fact that the subcategory $\Perfd_{K_{\le 1}}^{\Prism} \subset \CAlg_{K_{\le 1}}^{\wedge}$ is clossed under tensor products (\ref{tensor_integral_perfectoid}). Then, since the functor $(-)[\frac{1}{\varpi}]: \cD(K_{\le 1})^a \rightarrow \cD(K)$ is exact it follows that the following complex is acyclic
\begin{align*}
	0 \rightarrow R \rightarrow S \rightarrow (S_{\le 1} \cotimes_{R_{\le 1}}^a S_{\le 1})[\frac{1}{\varpi}] \rightarrow (S_{\le 1} \cotimes_{R_{\le 1}}^a S_{\le 1} \cotimes_{R_{\le 1}}^a S_{\le 1})[\frac{1}{\varpi}] \rightarrow \cdots && \text{in } \cD(K)
\end{align*}
and by Corollary \ref{tensor_dic_ban_to_int} we learn that we have a canonical identification
\begin{equation*}
	\Big ( S_{\le 1} \cotimes_{R_{\le 1}}^a S_{\le 1} \Big )[\frac{1}{\varpi}] \simeq S \cotimes_R S
\end{equation*}
proving the desired result.
\end{proof}

\begin{corollary}[Tate's Acyclicity]\label{tate_acyclicity} Let $X = \cM(R)$ be an object of $\Perfd_{K}^{\Ban, \op}$, and $\{U_i\}_{i \in I}$ a finite collection of elements of $|X|_{\rat}$ (cf. Definition \ref{defn_rational_site}) such that the induced map $\sqcup U_i \rightarrow |X|$ is surjective. Then, the following complex is exact and admissible
\begin{align*}
	0 \rightarrow \cO_X(X) \rightarrow \prod_{i \in I} \cO_X(U_i) \rightarrow \prod_{i,j \in I} \cO_X(U_i \cap U_j) \rightarrow \prod_{i,j,k \in I} \cO_X(U_i \cap U_j \cap U_k) \rightarrow \cdots
\end{align*}
Furthermore, the complex
\begin{align*}
	0 \rightarrow \cO_X(X)_{\le 1} \rightarrow \prod_{i \in I} \cO_X(U_i)_{\le 1} \rightarrow \prod_{i,j \in I} \cO_X(U_i \cap U_j)_{\le 1} \rightarrow \prod_{i,j,k \in I} \cO_X(U_i \cap U_j \cap U_k)_{\le 1} \rightarrow \cdots
\end{align*}
is almost acyclic. The functors $\cO_{X}(-)$ and $\cO_X(-)_{\le 1}$ are defined in Theorem \ref{struct_presheaf_perfectoid}.
\end{corollary}

\begin{proof} The exactness and admissibility of the first complex (involving the functor $\cO_{X}(-)$) is a direct consequence of Corollary \ref{arc_pi_descent_for_bananach_perfectoids} and Theorem \ref{struct_presheaf_perfectoid}. On the other hand, the almost acyclicity of the second complex (involving the functor $\cO_{X}(-)_{\le 1}$) is a direct consequence of Proposition \ref{arc_pi_descent_for_perfectoids} and the fact that given a morphism of perfectoid Banach $K$-algebras $R \rightarrow S$ we have a canonical isomorphism $S_{\le 1} \cotimes^a_{R_{\le 1}} S_{\le 1} \simeq (S \cotimes_R S)_{\le 1}$ by Corollary \ref{tensor_dic_ban_to_int}.
\end{proof}

\begin{prop}\label{perfectoidization_pi_complete_arc_descent} The functor $(-)_{\perfd}:  \CAlg_{K_{\le 1}, \varpi}^{\wedge} \rightarrow \cD(K_{\le 1})^{\wedge}_{\varpi}$ satisfies the following properties:
\begin{enumerate}[(1)]
	\item The functor $(-)_{\perfd}$ is a hypercomplete $\varpi$-complete $\arc$-sheaf.
	\item For any $S \in \CAlg_{K_{\le 1}, \varpi}^{\wedge}$, there is a canonical identification $S_{\perfd} = R\Gamma_{\varpi-\arc}(S, \cO)$.
	\item For any pair of morphisms $S_1 \leftarrow S_3 \rightarrow S_2$ in $\CAlg_{K_{\le 1}, \varpi}^{\wedge}$, the canonically induced map
	\begin{align*}
		(S_{1, \perfd}) \cotimes_{S_{3, \perfd}}^L (S_{2, \perfd}) \rightarrow (S_1 \cotimes_{S_3}^L S_2)_{\perfd}
	\end{align*}
	is an equivalence. In order to make sense of $(S_1 \cotimes_{S_3}^L S_2)_{\perfd}$ recall Proposition \ref{properties_perfd}(5).
	\item Let $S$ be an object of $\CAlg_{K_{\le 1}, \varpi}^{\wedge}$, and assume that $S_{\perfd}$ is connective (equivalently, by Proposition \ref{properties_perfd}(2), concentrated in degree $0$). Then $S_{\perfd}$ is an object of $\Perfd_{K_{\le 1}}^{\Prism}$ and $S \rightarrow S_{\perfd}$ is the universal map from $S$ to a perfectoid ring.
\end{enumerate}
\end{prop}

\begin{proof} Recall that if $R$ is an object of $\Perfd_{K_{\le 1}}^{\Prism}$ then the canonical map $R \rightarrow R_{\perfd}$ is an isomorphism (cf. Proposition \ref{properties_perfd}(3)). By virtue of Proposition \ref{pi_complete_arc_descent_for_perfectoids} we learn that when restricting the functor $(-)_{\perfd}$ to the subcategory $\Perfd_{K_{\le 1}}^{\Prism} \subset \CAlg_{K_{\le 1}, \varpi}^{\wedge}$ we obtain a hypercomplete $\varpi$-complete $\arc$-sheaf
\begin{align*}
	(-)_{\perfd}: \Perfd_{K_{\le 1}}^{\Prism} \rightarrow \cD(K_{\le 1})^{\wedge}_{\varpi} && R \mapsto R
\end{align*}
and Proposition \ref{properties_perfd}(3) shows that $(-)_{\perfd}$ can be realized as the right Kan extension of $(-)_{\perfd}$ along the inclusion $\Perfd_{K_{\le 1}}^{\Prism} \subset \cD(K_{\le 1})^{\wedge}_{\varpi}$. Furthermore, recall that since $\Perfd_{K_{\le 1}}^{\Prism}$ forms a basis in $\varpi$-complete $\arc$-topology (Remark \ref{canonical_pi_complete_arc_covers}), the canonical map
\begin{align*}
	\Shv_{\varpi-\arc}(\CAlg_{K_{\le 1}, \varpi}^{\wedge, \op}, \cD(K_{\le 1}))^{\wedge} \rightarrow \Shv_{\varpi-\arc}(\Perfd_{K_{\le 1}}^{\Prism, \op}, \cD(K_{\le 1}))^{\wedge} && F \mapsto F|_{\Perfd_{K_{\le 1}}^{\Prism}}
\end{align*}
determines an equivalence of categories of hypercomplete sheaves (cf. \ref{defn_hypercomplete_sheaf}) for the $\varpi$-complete $\arc$ topology; with the inverse given by right Kan extension along the inclusion $\Perfd_{K_{\le 1}}^{\Prism} \subset \cD(K_{\le 1})^{\wedge}_{\varpi}$. This proves that $(-)_{\perfd}$ must be a hypercomplete $\varpi$-complete $\arc$-sheaf, finishing the proof of (1). 

The argument for (2) is similar, by Proposition \ref{pi_complete_arc_descent_for_perfectoids} we learn that when restrict the functor $R\Gamma_{\varpi-\arc}(-, \cO)$ to the subcategory $\Perfd_{K_{\le 1}}^{\Prism} \subset \CAlg_{K_{\le 1}, \varpi}^{\wedge}$ we obtain a hypercomplete $\varpi$-complete $\arc$-sheaf
\begin{align*}
	R\Gamma_{\varpi-\arc}(-, \cO): \Perfd_{K_{\le 1}}^{\Prism} \rightarrow \cD(K_{\le 1})^{\wedge}_{\varpi} && R \mapsto R
\end{align*}
Then, since $\Perfd_{K_{\le 1}}^{\Prism}$ forms a basis in $\varpi$-complete $\arc$-topology (Remark \ref{canonical_pi_complete_arc_covers}), it follows from the definitions that $R\Gamma_{\varpi-\arc}(-, \cO): \CAlg_{K_{\le 1}, \varpi}^{\wedge} \rightarrow \cD(K_{\le 1})^{\wedge}_{\varpi}$ is the right Kan extension of $R\Gamma_{\varpi-\arc}(-, \cO)$ along the inclusion $\Perfd_{K_{\le 1}}^{\Prism} \subset \CAlg_{K_{\le 1}, \varpi}^{\wedge}$. The rest of the argument follows as in (1).

Part (3) follows from the description of $(-)_{\perfd}$ in terms of derived prismatic cohomology \cite[Proposition 8.13]{prisms}, and (4) is \cite[Corollary 8.14.]{prisms}.
\end{proof}

\begin{prop}\label{perfectoidization_arc_pi_descent} Define the functor $(-)_{\perfd}^a: \CAlg_{K_{\le 1}, \varpi}^{\wedge} \rightarrow \cD(K_{\le 1})^{\wedge a}_{\varpi}$ as the following composition
\begin{align*}
	(-)_{\perfd}^a: \CAlg_{K_{\le 1}, \varpi}^{\wedge} \overset{(-)_{\perfd}}{\longrightarrow} \cD(K_{\le 1})^{\wedge}_{\varpi} \overset{(-)^{a}}{\longrightarrow} \cD(K_{\le 1})^{\wedge a}_{\varpi}
\end{align*}
The functor $(-)_{\perfd}^a$ satisfies the following properties.
\begin{enumerate}[(1)]
	\item For any $S \in \CAlg_{K_{\le 1}}^{\wedge}$, there is a canonical identitifaction $S_{\perfd}^a = \lim_{S \rightarrow P} P$, where the limit ranges over all maps $S \rightarrow P$ where $P \in \Perfd_{K_{\le 1}}^{\Prism a}$, and is computed in $\cD(K_{\le 1})^{\wedge a}_{\varpi}$.
	\item The functor $(-)_{\perfd}^a$ is a hypercomplete $\arc_\varpi$-sheaf.
	\item For any $S \in \CAlg_{K_{\le 1}}^{\wedge}$, there is a canonical identification $S_{\perfd}^a = R\Gamma_{\arc_\varpi}(S, \cO^a)$.
	\item For any pair of morphisms $S_1 \leftarrow S_3 \rightarrow S_2$ in $\CAlg_{K_{\le 1}, \varpi}^{\wedge}$, the canonically induced map
	\begin{align*}
		(S_{1, \perfd}^a) \cotimes_{S_{3, \perfd}^a}^{L, a} (S_{2, \perfd}^a) \rightarrow (S_1 \cotimes_{S_3}^L S_2)_{\perfd}^a
	\end{align*}
	is an equivalence. In order to make sense of $(S_1 \cotimes_{S_3}^L S_2)_{\perfd}$ recall Proposition \ref{properties_perfd}(5).
\end{enumerate}
\end{prop}

\begin{proof} For the proof of (1), recall from Proposition \ref{properties_perfd}(3) that there is a canonical identification $S_{\perfd} = \lim_{S \rightarrow R} R$ where the limit ranges over all maps $S \rightarrow R$ where $R$ is a perfectoid ring. Then, the fact that $(-)^a$ is both a left and right adjoint shows that $S_{\perfd}^a = \lim_{S \rightarrow R} R^a$, which proves the result. For $(2)$ notice that by $(1)$ we already know that $(-)_{\perfd}^a$ preserves finite products, thus to show that it is an $\arc_\varpi$-sheaf it suffice to show that for any $\arc_\varpi$-cover $S \rightarrow R$ the canonical map
\begin{align*}
	S_{\perfd}^a \rightarrow \Tot (R^{\bullet/S})^{a}_{\perfd}
\end{align*}
is an isomorphism. To show this, recall that if $S \rightarrow R$ is an $\arc_\varpi$-cover then $S \rightarrow R \times (S/\varpi)_{\perf}$ is a $\varpi$-complete $\arc$-cover, where $(S/\varpi)_{\perf}$ is the classical direct limit perfection. Thus, from Proposition \ref{perfectoidization_pi_complete_arc_descent}(1) it follows that the map canonical map 
\begin{align*}
	S_{\perfd}^a \rightarrow \Tot \Big((R \times (S/\varpi)_{\perf})^{\bullet/S} \Big)^{a}_{\perfd}
\end{align*}
is an isomorphism. Then, by the compatibility of $(-)_{\perfd}$ and $(-)^a$ with tensor products and finite products, and the fact that $(S/\varpi)_{\perf} \simeq ((S/\varpi)_{\perf})_{\perfd}$ we learn that
\begin{align*}
	\Big((R \times (S/\varpi)_{\perf})^{\bullet/S} \Big)^{a}_{\perfd} \simeq (R^{\bullet/S})^{a}_{\perfd} && \text{as} \qquad (S/\varpi)_{\perf}^a \simeq 0
\end{align*}
Finally, hypercompleteness follows from the fact that $(S)_{\perfd}^a \in \cD(K_{\le 1})^{\ge 0}$ by part (1) and Proposition \ref{sheaf_implies_hypersheaf}, completing the proof of (2).

In order to prove (3), recall from Theorem \ref{arc_pi_descent_for_perfectoids} that when we restrict the functor $R\Gamma_{\arc_\varpi}(-, \cO^a)$ to the subcategory $\Perfd_{K_{\le 1}}^{\Prism a} \subset \CAlg_{K_{\le 1}, \varpi}^{\wedge}$ we obtain a hypercomplete $\arc_\varpi$-sheaf
\begin{align*}
	R\Gamma_{\arc_\varpi}(-, \cO^a): \Perfd_{K_{\le 1}}^{\Prism a} \rightarrow \cD(K_{\le 1})^{\wedge a}_{\varpi} && R \mapsto R = R^a
\end{align*}
We claim that the category $\Perfd_{K_{\le 1}}^{\Prism a} \subset \CAlg_{K_{\le 1}, \varpi}^{\wedge}$ is a basis for the $\arc_\varpi$-topology. Indeed, we already showed that $\Perfd_{K_{\le 1}}^{\Prism}$ is a basis for the $\arc_\varpi$-topology (Remark \ref{canonical_arc_pi_covers}), thus it remains to show that for any map $P \rightarrow V$ in $\Perfd_{K_{\le 1}}^{\Prism}$ where $V$ is a rank one valuation ring with faithful flat structure map $K_{\le 1} \rightarrow V$, there exists a unique factorization as $P \rightarrow P^a \rightarrow V$, but this follows from the fact that $V \in \Perfd_{K_{\le 1}}^{\Prism a}$. Therefore, it follows from the definitions that $R\Gamma_{\arc_\varpi}(-, \cO^a): \CAlg_{K_{\le 1}, \varpi}^{\wedge} \rightarrow \cD(K_{\le 1})^{\wedge a}_{\varpi}$ is the right Kan extension of $R\Gamma_{\arc_\varpi}(-, \cO^a)$ along the inclusion $\Perfd_{K_{\le 1}}^{\Prism a} \subset  \CAlg_{K_{\le 1}, \varpi}^{\wedge}$, but this is the same description given of $(-)_{\perfd}^a$ in part (1). This completes the proof of (3).

Statement (4) follows from the analogous statement in Proposition \ref{perfectoidization_pi_complete_arc_descent} and the fact that the functor $(-)^a: \cD(K_{\le 1})^{\wedge}_{\varpi} \rightarrow \cD(K_{\le 1})^{\wedge a}_{\varpi}$ is symmetric monoidal.
\end{proof}

\begin{prop}\label{perfectoidization_banach} Define the functor $(-)_{\le 1, \perfd}^a: \Ban_K^{\contr} \rightarrow \cD(K_{\le 1})^{\wedge a}$ as the following composition
\begin{align*}
	(-)_{\le 1, \perfd}^a: \Ban_{K}^{\contr} \overset{(-)_{\le 1}}{\longrightarrow} \CAlg_{K_{\le 1}, \varpi}^{\wedge} \overset{(-)_{\perfd}^a}{\longrightarrow} \cD(K_{\le 1})^{\wedge a}_{\varpi}
\end{align*}
The functor $(-)_{\le 1, \perfd}^a$ satisfies the following properties
\begin{enumerate}[(1)]
	\item For any $A \in \Ban_K^{\contr}$, there exists a canonical identification $A_{\le 1, \perfd}^a = \lim_{A \rightarrow P} P_{\le 1}$ where the limit ranges over all maps $A \rightarrow P$ where $P \in \Perfd_K^{\Ban}$, and is computed in $\cD(K_{\le 1})^{\wedge a}_{\varpi}$.
	\item The functor $(-)_{\le 1, \perfd}^a$ is a hypercomplete $\arc_\varpi$-sheaf.
	\item For any $A \in \Ban_K^{\contr}$, there is a canonical identification $A_{\le 1, \perfd}^a = R\Gamma_{\arc_\varpi}(A, \cO_{\le 1}^a)$.
	\item For any pair of morphisms $A_1 \leftarrow A_3 \rightarrow A_2$ in $\Ban_K^{\contr}$, the canonically induced map
	\begin{align*}
		(A_{1, \le 1, \perfd}^a) \cotimes_{A_{3, \le 1, \perfd}^a}^{L, a} (A_{2, \le 1, \perfd}^a) \rightarrow (A_1 \cotimes_{A_3} A_2)_{\le 1, \perfd}^a
	\end{align*}
	is an equivalence.
\end{enumerate}
\end{prop}

\begin{proof} Part (1) follows directly from Proposition \ref{perfectoidization_arc_pi_descent}(1), together with the dictionary Proposition \ref{equiv_almost_tf_banach} and Proposition \ref{equiv_perfd_ban_tic}. For the proof of (2), recall that we showed in Proposition \ref{perfectoidization_arc_pi_descent}(2) that if $A \rightarrow B$ is an $\arc_\varpi$-cover then the canonical map
\begin{align*}
	(A_{\le 1})_{\perfd}^a \rightarrow \Tot(B_{\le 1}^{\bullet/A_{\le 1}})^a_{\perfd}
\end{align*}
is an isomorphism, thus to show that $(-)_{\le 1, \perfd}^a$ is an $\arc_\varpi$-sheaf it remains to show that there is a canonical isomorphism $(B_{\le 1}^{\bullet/A_{\le 1}})^a_{\perfd} \simeq (B^{\bullet/A})_{\le 1, \perfd}^a$. Recall from the Proposition \ref{equiv_almost_tf_banach} that there is a canonical isomorphism $(B_{\le 1}^{\bullet/A_{\le 1}})^{\tf, \wedge, a} \simeq (B^{\bullet/A})_{\le 1}$ and by Proposition \ref{perfectoidization_arc_pi_descent}(1) it is clear that the canonical map $(B_{\le 1}^{\bullet/A_{\le 1}}) \rightarrow (B_{\le 1}^{\bullet/A_{\le 1}})^{\tf, \wedge, a}$ induces an isomorphism
\begin{align*}
	(B_{\le 1}^{\bullet/A_{\le 1}})_{\perfd}^a \overset{\simeq}{\longrightarrow} \Big( (B_{\le 1}^{\bullet/A_{\le 1}})^{\tf, \wedge, a} \Big)_{\perfd}^a
\end{align*}
Finally, hypercompleteness follows from the fact that $(A)_{\le 1, \perfd}^a \in \cD(K_{\le 1})^{\ge 0}$ by part (1) and Proposition \ref{sheaf_implies_hypersheaf}, completing the proof of (2).

In order to prove (3), recall from the proof of Corollary \ref{arc_pi_descent_for_bananach_perfectoids} that when we restrict the functor $R\Gamma_{\arc_\varpi}(-, \cO_{\le 1}^a)$ to the subcategory $\Perfd_{K}^{\Ban} \subset \Ban_K^{\contr}$ we obtain a hypercomplete $\arc_\varpi$-sheaf
\begin{align*}
	R\Gamma_{\arc_\varpi}(-, \cO_{\le 1}^a): \Perfd_K^{\Ban} \rightarrow \cD(K_{\le 1})^{\wedge a}_{\varpi} && R \mapsto R_{\le 1} = R_{\le 1}^a
\end{align*}
From the dictionary (Propositions \ref{equiv_almost_tf_banach} and \ref{equiv_perfd_ban_tic}) and the proof of Proposition \ref{perfectoidization_arc_pi_descent}(3) we learn that $\Perfd_{K}^{\Ban} \subset \Ban_K^{\contr}$ is a basis for the $\arc_\varpi$-topology. Therefore, it follows from the definition of $R\Gamma_{\arc_\varpi}(-, \cO_{\le 1}^a): \Ban_K^{\contr} \rightarrow \cD(K_{\le 1})^{\wedge a}_{\varpi}$ is the right Kan extension of $R\Gamma_{\arc_\varpi}(-, \cO_{\le 1}^a)$ along the inclusion $\Perfd_K^{\Ban} \subset \Ban_K^{\contr}$, but this is the same description given of $(-)_{\le 1, \perfd}^a$ given in part (1). This completes the proof of (3).

Finally, we prove (4). From Proposition \ref{perfectoidization_arc_pi_descent}(4) we have that the canonical map
\begin{align*}
	(A_{1, \le 1, \perfd}^a) \cotimes_{A_{3, \le 1, \perfd}^a}^{L, a} (A_{2, \le 1, \perfd}^a) \rightarrow (A_{1, \le 1} \cotimes^L_{A_{3, \le 1}} A_{2, \le 1})_{\perfd}^a
\end{align*}
is an isomorphism, and by Proposition \ref{properties_perfd}(5) we also know that we have a canonical identification $(A_{1, \le 1} \cotimes^L_{A_{3, \le 1}} A_{2, \le 1})_{\perfd}^a \simeq (\pi_0(A_{1, \le 1} \cotimes^L_{A_{3, \le 1}} A_{2, \le 1}))_{\perfd}^a$. The result then follows from the fact that the canonical map $\pi_0(A_{1, \le 1} \cotimes^L_{A_{3, \le 1}} A_{2, \le 1}) \rightarrow \pi_0(A_{1, \le 1} \cotimes^L_{A_{3, \le 1}} A_{2, \le 1})^{\tf, \wedge, a}$ becomes an isomorphism after applying $(-)_{\perfd}^a$ as showed in part (2), together with the identity $(A_1 \cotimes_{A_3} A_2)_{\le 1} \simeq (\pi_0(A_1 \cotimes_{A_3}^L A_2))^{\tf, \wedge, a}$.
\end{proof}

\newpage

\chapter{Perfectoid Spaces}\label{chapt_perfd_spc}

Throughout this chapter we fix a prime number $p$ and a perfectoid non-archimedean field $K$ together with an object $\varpi \in K$ satisfying $1 > |\varpi^p| \ge |p|$ and  a compatible system of $p$-power roots $\{\varpi^{1/p^n}\}_{n \in \ZZ_{\ge 0}}$. In Section \ref{sect_coh_topoi} we recall the basic properties of coherent topoi, as it will form the categorical bedrock for the definitions of the $\arc_\varpi$-topos $\cX_{\arc_\varpi}$ and the category of condensed sets $\Cond$. In Section \ref{sect_berko_funct} we introduce the main players of this work, namely the $\arc_\varpi$-topos $\cX_{\arc_\varpi}$ and the Berkovich functor $|-|: \cX_{\arc_\varpi} \rightarrow \Cond$ which extends the Berkovich spectrum functor $|-|: \Ban_K^{\contr, \op} \rightarrow \Comp$. Furthermore, we establish the main results for the $\arc_\varpi$-topos like our version of the Gerritzen-Grauert theorem (Theorem \ref{intro_mono_arc_topos}) for $\arc_\varpi$-sheaves. Finally, in Section \ref{sect_analytic_spaces} we isolate the categories of perfectoid spaces\footnote{Generalizing the notion of affinoid perfectoid space.} and $\arc_\varpi$-analytic spaces as full subcategories of $\cX_{\arc_\varpi}$, and we specialize our results from Section \ref{sect_berko_funct} to these settings where the statements become slightly more concrete. Moreover, we show the under some mild hypothesis\footnote{Analogous to local compactness in topology.} $\arc_\varpi$-analytic spaces admit a well-behaved theory of open subsets.

\section{Digression: Coherent Topoi}\label{sect_coh_topoi}

\subsection{Sheaves on a (finitary) site}

\begin{defn}
Let $\cC$ be a small category, define its presheaf category as
\begin{equation*}
	\PreShv(\cC) := \text{Funct} (\cC^{\op}, \Set)
\end{equation*}
Furthermore, there is a functor
\begin{equation*}
	\Yo : \cC \longrightarrow \PreShv(\cC) \qquad c \mapsto \Yo(c) := \Maps_{\cC} (- , c)
\end{equation*}
called the Yoneda embedding.
\end{defn}

\begin{prop}[Properties of PreShv]\label{prop_presheaves} Let $\cC$ be a (small) category and $\PreShv(\cC)$ its presheaf category as described above.
\begin{enumerate}[(1)]
	\item The category $\PreShv(\cC)$ has all (small) limits and colimits. Moreover, limits and colimits are computed objectwise: let $\{X_i\}_{\cI}$ be a (small) $\cI$-indexed diagram in $\PreShv(\cC)$, then for every $c \in \cC$ we have
	\begin{equation*}
		(\colim_{\cI} X_i ) (c) = \colim_{\cI} X_i(c) \qquad \text{and} \qquad (\lim_{\cI} X_i ) (c) = \lim_{\cI} X_i(c)
	\end{equation*}
	\item The Yoneda embedding $\Yo: \cC \rightarrow \PreShv(\cC)$ is fully faithful and it preserves all (small) limits.
	\item (Yoneda Lemma) There exists a natural bijection
	\begin{equation*}
		\Maps_{\PreShv(\cC)} (\Yo(c), F) \simeq F(c) \qquad (\Yo(c) \rightarrow F) \mapsto \Big (\text{Id} \in \Yo (c)(c) \mapsto F(c) \Big)
	\end{equation*}
	Where we make sense of $\text{Id} \in \Yo(c)(c)$ via the identification $\Yo(c)(c) = \Maps_{\cC} (c, c)$.
	\item Every object of $\PreShv(\cC)$ can be described as a colimit of a diagram in $\cC$. More concretely, let $X \in \PreShv(\cC)$ and $c \in \cC$, then
	\begin{equation*}
		\colim_{\Yo(c) \rightarrow X} \Yo(c) = X
	\end{equation*}
	\item (Universal Property of PreShv) Let $\cD$ be a category with all (small) colimits and a functor $F: \cC \rightarrow \cD$. Then, there exists an essetially unique colimit preserving functor $L: \PreShv(\cC) \rightarrow \cD$ making the following diagram commute
	\begin{cd}
		\cC \ar[r, "\Yo"] \ar[rd, "F", swap] & \PreShv(\cC) \ar[d, "L"] \\
		& \cD
	\end{cd}
\end{enumerate}
\end{prop}

\begin{proof} Omitted.
\end{proof}

In the following example we will show that the Yoneda embedding functor $\Yo$ does not preserve most colimits, even in situations where the colimit has concrete geometric meaning.

\begin{example}
Denote by $\Aff = \CAlg^{\op}$ the category of affine schemes and let $\Spec (R) \in \Aff$. Recall that for any affine scheme $\Spec (R)$ its sheaf cohomology satisfies $H^0(\Spec(R), \cO_R) = R$, this can be reformulated as follows: let $\{U_i \rightarrow \Spec(R) \}_{i \in I}$ be a finite Zariski cover, where $U_i = \Spec(R[f_i^{-1}])$, then,
\begin{equation*}
	\coeq \Big(\sqcup_{i,j \in I} (U_{i} \times_{\Spec (R)} U_{j}) \rightrightarrows  \sqcup_{i \in I} U_{i} \Big) \simeq \Spec (R) \qquad \text{in } \Aff
\end{equation*}
However, the above equality no longer holds true after we apply the Yoneda embedding functor $\Yo$. Denote by $W$ the coequalizer of $\Big (\sqcup_{i,j \in I} (U_{i} \times_{\Spec (R)} U_{j}) \rightrightarrows  \sqcup_{i \in I} U_{i} \Big )$ in $\PreShv(\Aff)$, we want to show that the canonical map $W \rightarrow \Spec (R)$ is generally not an isomorphism. It suffices to show that the the identity map $\Spec (R) \rightarrow \Spec (R)$ does not factor though $W \rightarrow \Spec (R)$. Indeed, by definition of $W$ we have that the canonical map $\sqcup_{i \in I} U_{i} \rightarrow \Spec (R)$ induces a surjective map of sets
\begin{equation*}
	\sqcup_{i \in I} U_{i} (\Spec (R)) \twoheadrightarrow W(\Spec (R))
\end{equation*}
Therefore, the identity map $\Spec (R) \rightarrow \Spec (R)$ will factor through $W$ if and only if it factors though some $U_{i} \rightarrow \Spec (R)$, which is impossible unless $f_i \in R^{\times}$ for some $f_i$.
\end{example}

\begin{defn}\label{defn_shv} Let $\cC$ be a (small) category which admits finite limits and $\tau$ a (finitary) Grothendieck topology on $\cC$. The sheaf category of $(\cC, \tau)$, which we denote by $\Shv_{\tau} (\cC)$, is defined as the full subcategory
\begin{equation*}
	\Shv_\tau (\cC) \subset \PreShv(\cC)
\end{equation*}
such that if $F \in \Shv_{\tau} (\cC)$ and $\{U_i \rightarrow X\}_{i \in I}$ is a covering in $\tau$, where $I$ is a finite set, then
\begin{equation*}
	F(X) \simeq \eq \Big (\prod_{i \in I} F(U_i) \rightrightarrows \prod_{i,j \in I} F(U_i \times_X U_j) \Big)
\end{equation*}
In particular, if $U \times_X V \simeq \emptyset$, that is $U,V$ are disjoint in $X$, then $F(U \sqcup V) \simeq F(U) \times F(V)$ (cf. \ref{defn_infty_sheaf}).

The hypothesis that $I$ is a finite set in the above discussion is to guarantee that the topology $\tau$ is finitary -- that is, that for every covering $\{U_j \rightarrow X\}_{j \in J}$ there exists a finite subset $I \subset J$ such that $\{U_j \rightarrow X\}_{j \in I}$ is a covering. Finitary topologies guarantee that the category $\Shv_{\tau}(\cC)$ is better behaved. The $\arc$, $\varpi$-complete $\arc$, and $\arc_{\varpi}$ topologies introduced previously are all finitary.
\end{defn}

\begin{prop}[Properties of Shv]\label{prop_sheaves} Let $\cC$ be a (small) category which admits finite limits and $\Shv_{\tau} (\cC)$ its sheaf category as described above.
\begin{enumerate}[(1)]
	\item The inclusion functor $\Shv_{\tau} (\cC) \hookrightarrow \PreShv(\cC)$ admits a left adjoint $\PreShv(\cC) \rightarrow \Shv_{\tau} (\cC)$ called the sheafification functor. Furthermore, the sheafification functor $\PreShv(\cC) \rightarrow \Shv_{\tau} (\cC)$ preserves finite limits.
	\item The category $\Shv_{\tau} (\cC)$ has all (small) limits and colimits. Moreover, limits are be computed objectwise, while colimits can be described as the sheafification of the objectwise colimit. In particular, if $\{c_i \rightarrow c\}_{i \in I}$ is a covering in $(\cC, \tau)$ then
	\begin{equation*}
		\coeq \Big ( \sqcup_{i, j \in I} \Yo_{\tau} (c_i \times_{c} c_j) \rightrightarrows \sqcup_{i \in I}
		\Yo_{\tau} (c_i) \Big ) \simeq \Yo_{\tau} (c)
	\end{equation*}
	\item The sheafified Yoneda functor $\Yo_{\tau}: \cC \rightarrow \Shv_{\tau}(\cC)$ is defined as the composition of the Yoneda embedding followed by the sheafification functor
	\begin{cd}
		\Yo_{\tau}: \cC \ar[r, "\Yo"] & \PreShv(\cC) \ar[r] & \Shv_{\tau}(\cC)
	\end{cd}
	The functor $\Yo_{\tau}$ preserves all finite limits, but it need not be fully faithful. 
	\item (Universal Property of Shv) Let $\cD$ be a category with all (small) colimits, and a functor $F: \cC \rightarrow \cD$ such that for every covering $\{c_i \rightarrow c\}_{i \in I}$ in $(\cC, \tau)$ we have an isomorphism
	\begin{equation*}
		\coeq \Big ( \sqcup_{i,j \in I} F(c_i \times_c c_j) \rightrightarrows \sqcup_{i \in I} F(c_i) \Big ) \simeq F(c)
	\end{equation*}
	Then, there exists an essentially unique colimit preserving functor $L: \Shv_{\tau} (\cC) \rightarrow \cD$ making the following diagram commute
	\begin{cd}
		\cC \ar[r, "\Yo_{\tau}"] \ar[rd, "F", swap] & \Shv_{\tau}(\cC) \ar[d, "L"] \\
		& \cD
	\end{cd}
\end{enumerate}
\end{prop}

\begin{proof}
	Omitted.
\end{proof}

\begin{defn}[Exact Categories]\label{defn_exact_cat} Let $\cC$ be a category. We say that $\cC$ is exact if it satisfies the following axioms
\begin{enumerate}[(1)]
	\item The categorty $\cC$ admits finite limits.
	\item Every equivalence relation on an object $X$ is effective. (cf. Definition \ref{defn_eff_equiv_relation})
	\item The collection of effective epimorphisms in $\cC$ (cf. Definition \ref{defn_eff_epi}) is closed under pullback. That is, if we are given a pullback diagram
	\begin{cd}
		X^{\prime} \ar[r] \ar[d, "f^{\prime}"] & X \ar[d, "f"] \\
		Y^{\prime} \ar[r] & Y
	\end{cd}
	in $\cC$ where $f$ is an effective epimorphism, the morphism $f^{\prime}$ is also an effective epimorphism.
\end{enumerate}
\end{defn}

\begin{example} The categories of sets is exact.
\end{example}

\begin{prop} Let $\cC$ be a category. If $\cC$ is exact, then it is a regular category (in the sense of Definition \ref{defn_regular_cat}).
\end{prop}

\begin{proof} \cite[Proposition A.2.8]{lurie_ultracategories}
\end{proof}

\subsection{Quasicompact Quasiseparated objects}

\begin{defn}[Disjoint Coproduct]\label{defn_disj_coprod} Let $\cC$ be a category which admits fiber products, and let $X,Y \in \cC$ be objects which admits a coproducts $X \sqcup Y$. We will say that $X \sqcup Y$ is a disjoint coproduct of $X$ and $Y$ if the following pair of conditions is satisfied
\begin{enumerate}[(1)]
	\item Each of the maps $X \rightarrow (X \sqcup Y) \leftarrow Y$ is a monomorphism.
	\item The fiber product $X \times_{X \sqcup Y} Y$ is the initial object of $\cC$.
\end{enumerate}
\end{defn}

\begin{defn}[Grothendieck Topos]\label{defn_topos} A (Grothendieck) topos is a category $\cX$ satisfying the following axioms:
\begin{enumerate}[(1)]
	\item The category $\cX$ is exact (Definition \ref{defn_exact_cat}).
	\item The category $\cX$ admits small coproducts, and coproducts in $\cX$ are disjoint (Definition \ref{defn_disj_coprod}).
	\item The formation of small coproduct in $\cX$ is compatible with pullbacks. That is, for every collection of objects $\{X_i\}_{i \in I}$ having coproduct $X = \bigsqcup_{i \in I} X_i$ and every morphism $f: Y \rightarrow X$, the projection maps $\{X_i \times_X Y \rightarrow Y\}_{i \in I}$ exhibit $Y$ as a coproduct of the objects $\{X_i \times_X Y\}_{i \in I}$.
	\item The category $\cX$ is locally small, and there exists a small subcategory $\cX_0$ which generates $\cX$ in the sense that every object $X \in \cX$ admits an effective epimorphism $\bigsqcup_{i \in I} U_i \rightarrow X$, where each $U_i$ belongs to $\cX_0$.
\end{enumerate}
\end{defn}

\begin{example} The category $\Set$ is a topos.
\end{example}

\begin{prop}\label{shv_is_topos} Let $\cC$ be a small category equipped with a Grothendieck topology $\tau$. Then the category of sheaves $\Shv_{\tau} (\cC)$ is a topos. Furthermore, $\Shv_{\tau} (\cC)$ is generated (in the sense of Definition \ref{defn_topos}(4)) by the essential image of $\Yo_{\tau}: \cC \rightarrow \Shv_{\tau}(\cC)$.
\end{prop}

\begin{proof} It follows from \cite[Corollary C.1.7]{lurie_ultracategories} that $\Shv_{\tau} (\cC)$ is a topos, the fact that it is generated from the essential image of $\Yo_{\tau}: \cC \rightarrow \Shv_{\tau}(\cC)$ is a combination of Proposition \ref{prop_presheaves}(4) and Proposition \ref{prop_sheaves}(1).
\end{proof}

\begin{lemma}\label{morphism_of_shv} Let $\cC$ be a small category endowed with a topology $\tau$, and write $\cX$ for the topos $\Shv_{\tau}(\cC)$. Then
\begin{enumerate}[(1)]
	\item A morphism $F \rightarrow G$ in $\cX$ is a monomorphism if for every object $U \in \cC$ the map $F(U) \rightarrow G(U)$ is injective.
	\item A morphism $F \rightarrow G$ in $\cX$ is an epimorphism if for every object $U \in \cC$ and every section $s \in G(U)$ there exists a covering $\{U_i \rightarrow U\}$ such that for each $i$ the restriction $s|_{U_i}$ is in the image of $F(U_i) \rightarrow G(U_i)$.
	\item A morphism $F \rightarrow G$ in $\cX$ is an isomorphism if and only if it is a monomorphism and a epimorphism.
\end{enumerate}
\end{lemma}

\begin{proof} Ommited.
\end{proof}

\begin{lemma}\label{epi_equal_eff_epi_topos} Let $\cX$ be a topos, then
\begin{enumerate}[(1)]
	\item A map $F \rightarrow G$ in $\cX$ is an epimorphism if and only if it is an effective epimorphism, that is, the canonical map
	\begin{align*}
		\coeq(F \times_G F \rightrightarrows F) \rightarrow G
	\end{align*}
	is an isomorphism.
	\item A map $f: F \rightarrow G$ in $\cX$ admits a functorial and unique factorization as
	\begin{align*}
		F \twoheadrightarrow \im(f) \hookrightarrow G
	\end{align*}
	where $F \twoheadrightarrow \im(f)$ is an effective epimorphism and $\im(f) \hookrightarrow G$ is a monomorphism.
	\item Given a map $f: F \rightarrow G$ in $\cX$, the canonical map $F \times_{\im(f)} F \rightarrow F \times_G F$ is an equivalence. In particular, the canonical map $\coeq(F \times_G F \rightrightarrows F) \rightarrow G$ admits a unique and functorial factorization as
	\begin{align*}
		\coeq(F \times_G F \rightrightarrows F) \rightarrow \im(f) \hookrightarrow G
	\end{align*}
	and the map $\coeq(F \times_G F \rightrightarrows F) \rightarrow \im(f)$ is an isomorphism.
\end{enumerate}
\end{lemma}

\begin{proof} Ommited.
\end{proof}

\begin{defn}\label{defn_qcqs} Let $\cX$ be a topos. 
	
\begin{enumerate}[(1)]
	\item \textit{Quasicompact}: We say that an object $X \in \cX$ is quasicompact, if for for every collection of morphisms $\{U_i \rightarrow X \}_{i \in I}$ such that $\bigsqcup_{i \in I} U_i \rightarrow X$ is an effective epimorphism, there exists a finite subset $I_0 \subset I$ such that $\bigsqcup_{i \in I_0} U_i \rightarrow X$ is an effective epimorphism. Moreover, a morphism $Y \rightarrow X$ in $\cX$ is said to be quasicompact if for every morphism $U \rightarrow X$, where $U$ is quasicompact, the fiber product $Y \times_X U$ is quasicompact.
	\item \textit{Quasiseparated}: We say that an object $X \in \cX$ is quasiseparated if for every pair of morphisms $U \rightarrow X \leftarrow V$, where $U,V$ are quasicompact, the fiber product $U \times_X V$ is quasicompact. Moreover, a morphism $X \rightarrow Y$ is quasiseparated if the diagonal morphism $\Delta: X \rightarrow X \times_Y X$ is quasicompact.
	\item \textit{Quasicompact Quasiseparated (qcqs)}: Finally, we say that an object $X \in \cX$ is qcqs if it is quasicompact and quasiseparated, and we say that a morphism $X \rightarrow Y$ is qcqs if it is quasicompact and quasiseparated. Quasicompact quasiseparated objects are also called coherent in the literature. 
\end{enumerate}
\end{defn}

\begin{example} Any mononorphism in a topos is quasiseparated.
\end{example}

\begin{prop}\label{properties_qcqs_morphism} Let $\cX$ be a topos
\begin{enumerate}[(1)]
	\item Every isomorphism in $\cX$ is a qcqs morphism; that is, it is quasicompact and quasiseparated. The composition of quasicompact (resp. quasiseparated) morphisms is quasicompact (resp. quasiseparated).
	\item Let $f: X \rightarrow Y$ and $g: Y^{\prime} \rightarrow Y$ be morphisms in $\cX$, and $f^{\prime}: X^{\prime} := X \times_Y Y^{\prime} \rightarrow Y$ the basechange of $f$ along $g$. If $f$ is quasicompact (resp. quasiseparated) then $f^{\prime}$ is quasicompact (resp. quasiseparated).
	\item Let $f: X \rightarrow Y$ and $g: Y \rightarrow Z$ be morphisms in $\cX$. If $g \circ f$ is quasiseparated, then $f$ is quasiseparated.
	\item Let $f: X \rightarrow Y$ and $g: Y \rightarrow Z$ be morphisms in $\cX$, and $g$ is quasiseparated. If $g \circ f$ is quasicompact (resp. qcqs), then $f$ is quasicompact (resp. qcqs).
\end{enumerate}
\end{prop}

\begin{proof} \cite[Exp. VI, Proposition 1.8]{SGA4}
\end{proof}

\begin{lemma}\label{alt_defn_qc} Let $\cX$ be a topos and $X$ be an object of $\cX$. Suppose that there exists a finite set $I$ and a collection of morphisms $\{U_i \rightarrow X\}_{i \in I}$ where each $U_i$ is quasicompact and such that the induced map $\bigsqcup_{i \in I} U_i \rightarrow X$ is an epimorphism. Then, $X$ is quasicompact.
\end{lemma}

\begin{proof} \cite[Proposition C.5.2]{lurie_ultracategories}
\end{proof}

\begin{defn}[Coherent Topos]\label{defn_coherent_topos} Let $\cX$ be a topos. We say that $\cX$ is coherent if there exists a (small) collection of objects $\cU$ satisfying the following conditions
\begin{enumerate}[(1)]
	\item The collection $\cU$ generates $\cX$: that is, every object $X \in \cX$ admits a collection of morphisms $\{U_i \rightarrow X\}_{i \in I}$ such that $\bigsqcup_{i \in I} U_i \rightarrow X$ is an epimorphism, and where each $U_i$ belongs to $\cU$.
	\item The collection $\cU$ is closed under finite products. In particular, it contains the final object of $\cX$.
	\item Every object of $\cU$ is quasicompact and quasiseparated.
\end{enumerate}
\end{defn}

\begin{prop}\label{qcqs_check_local} Let $\cX$ be a coherent topos, and $\cU$ a subcategory of $\cX$ satisfying the conditions of Definition \ref{defn_coherent_topos}. Then:
\begin{enumerate}[(1)]
	\item A morphism $X \rightarrow Y$ in $\cX$ is quasicompact if and only if for every morphism $U \rightarrow Y$ where $U \in \cU$, the fiber product $X \times_Y U$ is quasicompact.
	\item An object $X \in \cX$ is qcqs if and only if it is quasicompact and, for every pair of morphisms $U \rightarrow X \leftarrow V$ where $U,V \in \cU$, the fiber product $U \times_X V$ is quasicompact.
	\item Let $X \rightarrow Y$ be a morphism in $\cX$, and $U \rightarrow Y$ an epimorphism ($U$ is not necessarily contained in $\cU$). If the induced map $X \times_Y U \rightarrow U$ is quasicompact, then so is $X \rightarrow Y$.
\end{enumerate}	
\end{prop}

\begin{proof} Part (1) and (2) are \cite[Corollary A.2.1.4]{luriespectral}, and Part (3) is \cite[Corollary A.2.1.5]{luriespectral}.
\end{proof}

\begin{prop}\label{comp_relative_absolute_qcqs} Let $\cX$ be a coherent topos, $\cU$ a subcategory of $\cX$ satisfying the conditions of Definition \ref{defn_coherent_topos}, and $\pt \in \cX$ the final object if $\cX$. Then,
\begin{enumerate}[(1)]
	\item An object $X \in \cX$ is quasicompact if and only if the canonical map $X \rightarrow \pt$ is quasicompact.
	\item An object $X \in \cX$ is quasiseparated if and only if the canonical map $X \rightarrow \pt$ is quasiseparated.
\end{enumerate}
\end{prop}

\begin{proof} First we proof (1). If $X$ is quasicompact, from the assumption that $\pt$ is quasiseparated we learn that for any quasicompact object $Y \in \cX$ the fiber product $X \times Y$ is quasicompact, proving that the morphism $X \rightarrow \pt$ is quasicompact. On the other hand, if the morphism $X \rightarrow \pt$ is quasicompact, since $\pt$ is quasicompact it follows that $X \times_{\pt} \pt \simeq X$ is quasicompact, finishing the proof of (1).

Next we proof (2). Assume that the morphism $X \rightarrow \pt$ is quasiseparated, then by definition we have that the diagonal map $\Delta: X \rightarrow X \times X$ is quasicompact. Given a pair of morphism $f: U \rightarrow X \leftarrow V :g$ where $U,V$ are quasicompact, it follows that $U \times V$ is quasicompact (since $\pt$ is quasiseparated), and so from the following pullback diagram
\begin{cd}
	U \times_X V \ar[r] \ar[d] & U \times V \ar[d, "f \times g"] \\
	X \ar[r, "\Delta"] & X \times X
\end{cd}
we learn that $U \times_X V$ is quasicompact, proving that if $X \rightarrow \pt$ is quasiseparated then $X$ is quasiseparated. On the other hand, we need to show that if $X$ is quasiseparated then the diagonal map $\Delta: X \rightarrow X \times X$ is quasicompact, by Proposition \ref{qcqs_check_local} it suffices to check that for and any morphism $U \rightarrow X \times X$ where $U \in \cU$, the fiber product $U \times_{X \times X} X$ is quasicompact. In order to proof this, first notice that any morphism $U \rightarrow X \times X$ admits a canonical factorization as $U \overset{\Delta}{\longrightarrow} U \times U \overset{f \times g}{\longrightarrow} X \times X$, and that we have the following commutative diagram, where each commuting square is a pullback square
\begin{cd}
	U \times_{X \times X} \ar[r] \ar[d] X & U \ar[d, "\Delta"] \\
	U \times_X U \ar[r] \ar[d] & U \times U \ar[d] \\
	X \ar[r, "\Delta"] & X \times X
\end{cd}
By the hypothesis that $X$ is quasiseparated it follows that $U \times_X U$ is quasicompact, and by the hypothesis on $\cU$ from Definition \ref{defn_coherent_topos} it follows that $U \times U$ is qcqs. Then, since $U \times U$ is quasiseparated it follows that $U \times_{X \times X} X$ is quasicompact, finishing the proof of (2).
\end{proof}

\begin{prop}\label{prop_qcqs_obj} Let $\cX$ be a coherent topos, and $\cX_{\qcqs} \subset \cX$ the full-subcategory of $\qcqs$ objects of $\cX$. Then,
\begin{enumerate}[(1)]
	\item The full subcategory $\cX_{\qcqs} \subset \cX$ is closed under finite limits.
	\item The full subcategory $\cX_{\qcqs} \subset \cX$ is closed under finite coproducts.
	\item Supposed we have an epimorphism $X \rightarrow Y$, if $X$ is qcqs and the fiber product $X \times_Y X$ is quasicompact, then $Y$ is qcqs.
\end{enumerate}
\end{prop}

\begin{proof} Part (1) is \cite[Proposition C.5.9]{lurie_ultracategories}, Part (2) is \cite[Proposition C.5.12]{lurie_ultracategories}, and finally Part (3) is \cite[Proposition C.5.13]{lurie_ultracategories}.
\end{proof}

\begin{prop}\label{finitary_implies_coh} Let $\cC$ be a (small) category which admits finite limits, in particular it admits a final object, endowed with a (finitary) Grothendieck topology $\tau$. Then,
\begin{enumerate}[(1)]
	\item The topos $\Shv_{\tau} (\cC)$ is coherent.
	\item Let $X$ be any object of $\cC$, then $\Yo_{\tau} (X)$ is a qcqs object of the topos $\Shv_{\tau}(\cC)$, where $\Yo_{\tau}$ is the sheafified Yoneda functor $\Yo_{\tau}: \cC \rightarrow \Shv_{\tau}(\cC)$.
\end{enumerate}
\end{prop}

\begin{proof} This is \cite[Proposition C.6.3]{lurie_ultracategories}, or \cite[Proposition A.3.1.3]{luriespectral} for an $\infty$-categorical version. We already proved in Proposition \ref{shv_is_topos} that the category $\Shv_{\tau}(\cC)$ is a topos.
\end{proof}

\begin{prop}\label{saturate_qs} Let $\cX$ be a coherent topos, and $X$ an object of $\cX$. Then, $X$ is quasiseparated if and only if there exists a filtered category $\cI$ and a functor $\cI \rightarrow \cX$ where $i \mapsto j$ is sent to a monomorphism $U_i \hookrightarrow U_j$ between qcqs objects and such that
\begin{align*}
	\colim_{\cI} U_i \simeq X
\end{align*}
Once the conditions described above of the functor $\cI \rightarrow \cX$ are satisfied it is immediate from the construction that the induced maps $U_i \hookrightarrow X$ is a monomorphism. This result implies that quasiseparated objects are closed under fiber products.
\end{prop}

\begin{proof} First, assume that $X$ is quasiseparated. Since $\cX$ is a coherent topos by assumption, we know that there exists a collection of morphisms $\{V_i \rightarrow X\}_{i \in I}$ from qcqs objects such that the induced map $p: \sqcup_{i \in I} V_i \rightarrow X$ is an epimorphisms. For each finite subset $J \subset I$ let us explain how to associate to it a qcqs object $U_J$ together with a monomorphism $U_J \hookrightarrow X$: by virtue of working on a topos the map $p_J: \sqcup_{i \in J} V_i \rightarrow X$ admits a unique factorization as
\begin{align*}
	\sqcup_{i \in J} V_i \twoheadrightarrow \im(p_J) \hookrightarrow X && \text{set} \qquad U_J = \im(p_J)
\end{align*}
To show that $U_J$ is qcqs notice that since it $\sqcup_{i \in J} V_i \twoheadrightarrow U_J$ is an epimorphism, it follows from Lemma \ref{alt_defn_qc} that $U_J$ is quasicompact, and since it $X$ is quasiseparated and $U_J \hookrightarrow X$ a monomorphism it follows that $U_J$ is quasiseparated, proving that $U_J$ is qcqs. Moreover, notice how by construction we have that if $J_1 \subset J_2$ then the monomorphism $U_{J_1} \rightarrow X$ admits a unique factorization as $U_{J_1} \rightarrow U_{J_2} \rightarrow X$, showing that the induced map $U_{J_1} \rightarrow U_{J_2}$ is also a monomorphism.

Let $\cI$ be the category of all finite subsets of $I$, with a morphism $J_1 \rightarrow J_2$ if we have an inclusion $J_1 \subset J_2$ in $I$. Then, the construction above can be summarized into the existence of a functor $\cI \rightarrow X$ where $J \mapsto U_J$. It remains to show that the canonical map $\colim_{\cI} U_J \rightarrow X$ is an equivalence. Since filtered colimits preserve monomorphisms it follows that the induced map $\colim_{\cI} U_J \rightarrow X$ is a monomorphism, thus it remains to show that it is an epimorphism, but this follows from the fact that $p: \sqcup_{i \in I} V_i \rightarrow X$ is a epimorphism and it factors as
\begin{align}
	\sqcup_{i \in I} V_i \rightarrow \colim_{\cI} U_J \rightarrow X
\end{align}
by construction, finishing the proof of the first implication.

On the other hand, assume that there exists a functor $\cI \rightarrow \cX$ satisfying the desired hypothesis. In order to show that $X$ is quasiseparated we need to show that for any pair of morphisms $Y_1 \rightarrow X \leftarrow Y_2$ from quasicompact objects of $\cX$, the fiber product $Y_1 \times_X Y_2$ is quasicompact. Recall that since $\cX$ is a coherent topos it can be realized as $\Shv_{\text{coh}}(\cX_{\qcqs})$ the category of sheaves on $\cX_{\qcqs}$ with respect to the coherent topology (\cite[Proposition C.6.4]{lurie_ultracategories}); thus, for any morphism $Z \rightarrow X$ from a qcqs object there exists a $i \in \cI$ such that there exists a unique factorization $Z \rightarrow U_i \hookrightarrow X$. We claim that for any morphism $Y \rightarrow X$, where $Y$ is only assumed to be quasicompact, there exists a $i \in \cI$ such that there exists a unique factorization $Y \rightarrow U_i \hookrightarrow X$. Indeed, since $\cX$ is coherent pick an epimorphism $Z \twoheadrightarrow Y$ from a qcqs object $Z$. Then, from the work done above we know that the map $Z \rightarrow X$ admits unique factorization as $Z \rightarrow U_i \hookrightarrow X$, so the following identity
\begin{align*}
	\coeq(Z \times_Y Z \rightrightarrows Z) \overset{\simeq}{\longrightarrow} Y
\end{align*}
together with the fact that $U_i \hookrightarrow X$ is a monomorphism, implies that there exists a unique factorization $Y \rightarrow U_i \hookrightarrow X$, as desired. To prove that $Y_1 \times_X Y_2$ is quasicompact, by the work done above we know that there exists a monomorphism $U_i \hookrightarrow X$ from a qcqs object such that $Y_j \rightarrow X$ factors as $Y_j \rightarrow U_i \hookrightarrow X$, thus we have the identity $Y_1 \times_X Y_2 \simeq Y_1 \times_{U_i} Y_2$, and since $U_i$ is qcqs the result follows.
\end{proof}

\newpage

\newpage

\section{The Berkovich Functor}\label{sect_berko_funct}

\subsection{Formalities on condensed sets}

\begin{defn}\label{defn_cond_set} Fix an uncountable strong limit cardinal $\kappa$, and consider the category of compact Hausdorff spaces of cardinality $< \kappa$, which we denote by $\Comp$, and let $\eff$ be the collection of all finite families $\{U_i \rightarrow X \}_{i \in I}$ of jointly surjective maps. Then,
	\begin{enumerate}[(1)]
		\item The collection $\eff$ determines a finitary Grothendieck topology on $\Comp$.
		\item The category of condensed sets is defined as
		\begin{equation*}
			\Cond : = \Shv_{\eff} (\Comp)
		\end{equation*}
		We warn the reader that what we call a condensed set here is called a $\kappa$-condensed set in \cite{condensedlectures}. We refer the reader to \cite[Appendix to Lecture II]{condensedlectures} for a definition of condensed sets which do not depend on a cut-off cardinal $\kappa$ -- we have decided to fix an cut-off cardinal in order to simplify exposition.
		\item The sheafified Yoneda embedding $\Yo_{\eff}: \Comp \rightarrow \Shv_{\eff} (\Comp)$ is fully faithful. This is a direct consequence of Proposition \ref{monadicity_comp_set}(3).	
\end{enumerate}
\end{defn}

\begin{rem} The fully faithful functor $\ProFin \hookrightarrow \Comp$ induces a functor $\Shv_{\eff}(\Comp) \rightarrow \Shv_{\eff}(\ProFin)$, which is an equivalence of categories by virtue of the fact that $\ProFin \subset \Comp$ forms a basis under the topology $\eff$.
\end{rem}

\begin{notation}\label{cond_underlying_set} Let $X$ be a condensed set, then by definition $X$ is a functor $X: \Comp^{\op} \rightarrow \Set$; when evaluated on a point $* \in \Comp$ we obtain a set $X(*)$ which we will call the underlying set of the condensed set $X$. In particular, when $X$ is a compact Hausdorff space the set $X(*)$ is exactly the underlying set of the compact hausdorff space $X$.
\end{notation}

\begin{prop}\label{epi_condensed_sets} The category $\Cond$ has the following basic properties:
\begin{enumerate}[(1)]
	\item The category $\Cond$ is a coherent topos (cf. Definition \ref{defn_coherent_topos}).
	\item The sheafified Yoneda embedding $\Yo_{\eff}: \Comp \rightarrow \Cond$ defines an equivalence of categories between $\Comp$ and the category of qcqs condensed sets (cf. Definition \ref{defn_qcqs}), which we denote by $\Cond_{\qcqs}$.
	\item Let $F \rightarrow G$ be a morphism of condensed sets, and assume that $G$ is qcqs and $F$ is quasicompact. Then, the morphism $F \rightarrow G$ is an epimorphism in $\Comp$ if and only if the induced map of sets $F(*) \rightarrow G(*)$ is surjective.
	\item Let $F \rightarrow G$ be a morphism of condensed sets, and assume that both $G$ and $F$ are qcqs. Then, the morphism $F \rightarrow G$ is an isomorphism if and only if $F(*) \overset{\simeq}{\longrightarrow} G(*)$.
\end{enumerate}
\end{prop}

\begin{proof} Proposition \ref{finitary_implies_coh} already implies (1), and shows that the category of compact Hausdorff spaces $\Comp$ is a subcategory of $\Cond_{\qcqs}$. To finish the proof of (2) we need to show that if $X \in \Cond_{\qcqs}$ then $X$ lies in the essential image of $\Yo_{\eff}: \Comp \hookrightarrow \Cond$. By assumption on $X$, there exists a compact hausdorff space $S$ together with an epimorphism $S \rightarrow X$ of condensed sets such that the fiber product $R:= S \times_X S \subset S \times S$ is a quasicompact subcondensed set of the compact Hausdorff space $S \times S$. By virtue of Proposition \ref{all_equiv_eff_comp}, it remains to show that $R \subset S \times S$ is a compact hausdorff space. As $R$ is quasicompact it follows that there exists a compact hausdorff space $T$ together with an epimorphism $T \rightarrow R$, then since $\Cond$ is a topos it follows that $\coeq(T \times_R T \rightrightarrows T) \rightarrow R$ is an equivalence (cf. Lemma \ref{epi_equal_eff_epi_topos}), and moreover since $R \subset S \times S$ is a monomorphism we have that the following arrows are equivalences
\begin{equation*}
	\coeq(T \times_R T \rightrightarrows T) \overset{\simeq}{\longrightarrow} \coeq(T \times_{S \times S} T \rightrightarrows T) \overset{\simeq}{\longrightarrow} R
\end{equation*}
Using the presentation of $R$ as the coequalizer of two compact hausdorff spaces, namely $\coeq(T \times_{S \times S} T \rightrightarrows T)$,  it follows that $R$ is a compact hausdorff space, finishing the proof of (2).

Next we proof (3). Pick an epimorphism $T \rightarrow G$ from a compact hausdorff space, then the basechange $T \times_G F$ is also quasicompact as $F$ is quasicompact and $G$ is quasiseparated, and pick another compact Hausdorff space $S$ together with an epimorphism $S \rightarrow T \times_G F$. If $F \rightarrow G$ is an epimorphism, then set $T = \pt$ and by the stability of (effective) epimorphisms on a topos we learn that $S \not= \emptyset$, proving that the induced map $F(*) \rightarrow G(*)$ is surjective. On the other hand, if the map $F(*) \rightarrow G(*)$ is surjective, then the induced map $T \times_G F(*) \rightarrow T(*)$ is surjective which in turn implies that the induced map $S(*) \rightarrow T(*)$ is a surjective map of sets; but since $S$ and $T$ are compact hausdorff spaces it follows that $S \rightarrow T$ is an epimorphism (by Proposition \ref{monadicity_comp_set}), and thus so is the composition $S \rightarrow T \rightarrow G$, proving that $F \rightarrow G$ is also an epimorphism as the map $S \rightarrow G$ factors as $S \rightarrow F \rightarrow G$.

Part (4) follows directly from Part (1) and Proposition \ref{monadicity_comp_set}; however, we provide an alternative proof which will generalize to the setting of $\arc_{\varpi}$-sheaves. If $F \rightarrow G$ is an isomorphism it is clear that the induced map $F(*) \rightarrow G(*)$ is an isomorphism of sets. On the other hand, we know from part (3) that the induced map $F \rightarrow G$ is an epimorphism. Hence, since we are working on a topos it suffices to show that $F \rightarrow G$ is a monomorphism. Recall that a map $F \rightarrow G$ is a monomorphism if and only if the diagonal map $\Delta: F \rightarrow F \times_G F$ is an isomorphism, and also recall that for any morphism $F \rightarrow G$ the diagonal map $\Delta: F \rightarrow F \times_G F$ is a monomorphism. Thus it remains to show that the diagonal map $F \rightarrow F \times_G F$ is an epimorphism, but this follows from the fact that $F \times_G F$ is qcqs, the hypothesis that $F(*) \simeq G(*)$ and part (3). This completes the proof.
\end{proof}

\begin{prop}\label{mono_condensed_sets} Let $F \rightarrow G$ be a morphism in $\Cond$, and assume that $F$ is qcqs and $G$ is quasiseparated. Then, the following are equivalent
\begin{enumerate}[(1)]
	\item The morphism $F \hookrightarrow G$ is a monomorphism.
	\item The induced map of sets $F(*) \hookrightarrow G(*)$ is injective.
	\item For any morphism $Z \rightarrow G$ in $\Cond$ such that $\im(Z(*) \rightarrow G(*)) \subset \im(F(*) \rightarrow G(*))$, there exists a unique morphism $Z \rightarrow F$ making the following diagram commute
	\begin{cd}
		& Z \ar[d] \ar[dashed, ld] \\
		F \ar[r] & G
	\end{cd}
\end{enumerate}
\end{prop}

\begin{proof} We begin by proving the equivalence $(1) \simeq (2)$. The implication $(1) \Rightarrow (2)$ is trivial: if $F \rightarrow G$ is a monomorphism, then $F(*) \rightarrow G(*)$ is injective by Lemma \ref{morphism_of_shv}. For the implication $(2) \Rightarrow (1)$, assume that $F(*) \rightarrow G(*)$ is injective and recall from Proposition \ref{saturate_qs} that there exists a collection of qcqs subobjects $\{G_i \subset G\}$ such that $\colim G_i \simeq G$. Thus, since monomorphisms are closed under filered colimits it suffices to show that each of the induced maps $F_i := F \times_G G_i \rightarrow G_i$ are monomorphisms. From the hypothesis we learn that for every $i$ the induced map $F_i(*) \rightarrow G_i(*)$ is injective, and since $G$ is quasiseparated it follows that $F_i$ is quasicompact, from the fact that the morphism $F_i \hookrightarrow F$ is a monomorphism it follows that $F_i$ is quasiseparated as $F$ is, proving that for every $i$ the object $F_i$ is qcqs. To conclude, recall that the sheafified Yoneda embedding $\Yo_{\eff}: \Comp \hookrightarrow \Cond$ preserves monomorphisms, and that since $F_i$ and $G_i$ are qcqs they lie in the essential image of $\Yo_{\eff}$ (by Proposition \ref{epi_condensed_sets}(2)), then the result follows from Proposition \ref{mono_comp}.

It remains to show the equivalence $(3) \simeq (1) + (2)$. To prove the implication $(3) \Rightarrow (1)$, we begin by noticing that if $\im(Z(*) \rightarrow G(*)) \not\subset \im(F(*) \rightarrow G(*))$ then there is no morphism $Z \rightarrow F$ that would make the desired diagram commute, hence the fact that $F \rightarrow G$ is a monomorphism follows from the uniqueness of the lift. It remains to show that $(1) + (2) \Rightarrow (3)$. From the presentation $\Cond = \Shv_{\eff} (\Comp)$ and Propositions \ref{prop_presheaves}(1) and \ref{prop_sheaves}(1) it follows that $Z$ can be realized as $\colim_{i \in \cI} Q_i \simeq Z$, where each $Q_i$ is a compact hausdorff space, thus we may assume that $Z$ is a compact hausdorff space. Since we assume that $F \hookrightarrow G$ is a monomorphism, it follows that $F \times_G Z \rightarrow Z$ is also a monomorphism, hence it remains to show that $F \times_G Z \rightarrow Z$ is an epimorphism as then it would be an isomorphism. As $G$ is quasiseparated and $F$ and $Z$ are qcqs it follows that $F \times_G Z$ is quasicompact, thus to show it is an epimorphism, we learn from Proposition \ref{epi_condensed_sets} that we only need the map $F \times_G Z(*) \rightarrow Z(*)$ to be surjective, but this is clear from the hypothesis.
\end{proof}

The following result will be useful when constructing the Berkovich functor next section.

\begin{lemma}\label{yo_eff_preserves_general_coeq} Let $X \rightarrow Y$ and $Z \rightarrow X \times_Y X$ be surjective maps of compact hausdorff spaces, and define the morphisms $Z \rightrightarrows X$ as the composition $Z \rightarrow X \times_Y X \rightrightarrows X$. Then, the canonical map
\begin{align*}
	\coeq(\Yo_{\eff}(Z) \rightrightarrows \Yo_{\eff}(X)) \rightarrow \Yo_{\eff}(Y)
\end{align*}
is an isomorphism of condensed sets. For the sake of completeness, recall that in Propostion \ref{general_coeq_comp} we proved that the same statement holds when the coequalizer is computed in the category $\Comp$.
\end{lemma}

\begin{proof} If the morphism $Z \rightarrow X \times_Y X$ is the identity map, then the result follows directly from Proposition \ref{prop_sheaves}(2). To prove the general case we follow a strategy similar to the one we followed in \ref{general_coeq_comp}. We claim that the canonical map
\begin{align*}
	\Yo_{\eff}(Z) \rightarrow \Yo_{\eff}(X \times X) \simeq \Yo_{\eff}(X) \times \Yo_{\eff}(X)
\end{align*}
which is induced from the pair of morphisms $Z \rightrightarrows X$, admits a unique factorization as
\begin{align*}
	\Yo_{\eff}(Z) \twoheadrightarrow \Yo_{\eff}(X \times_Y X) \hookrightarrow \Yo_{\eff}(X \times X)
\end{align*}
where the first morphism is a monomorphism and the second morphism is an epimorphism. Indeed, the factorization exists as it already exists in $\Comp$, thus it remains to show that $\Yo_{\eff}$ preserves monomorphisms and epimorphism. Since $\Yo_{\eff}$ preserves finite limit, it follows that it preserves monomorphisms; and the fact that it preserves epimorphisms follows from Proposition \ref{epi_condensed_sets}(3).

To conclude, notice that the work above implies that equivalence relation imposed by $\Yo_{\eff}(Z) \rightrightarrows \Yo_{\eff}(X)$ on $\Yo_{\eff}(X)$ will be the same as the equivalence relation imposed by $\Yo_{\eff}(X \times_Y X) \rightrightarrows \Yo_{\eff}(X)$, proving that we have the isomorphisms
\begin{align*}
	\coeq(\Yo_{\eff}(Z) \rightrightarrows \Yo_{\eff}(X)) \simeq \coeq(\Yo_{\eff}(X \times_Y X) \rightrightarrows \Yo_{\eff}(X)) \simeq \Yo_{\eff}(Y)
\end{align*}
Where in the first isomorphism we are implicitly using the fact that $\Cond$ is a topos.
\end{proof}

\subsection{The \texorpdfstring{$\arc_{\varpi}$}{arc pi}-topos}\label{section_arc_pi_topos}

Fix a perfectoid non-archimedean field $K$.

\begin{const}\label{const_arc_pi_topos} Recall that in Definition \ref{defn_arc_pi} we defined the finitary Grothendieck topology called the $\arc_{\varpi}$ topology on the category $\Sch_{K_{\le 1}, \qcqs}$. We define the $\arc_{\varpi}$-topos as
\begin{equation*}
	\cX_{\arc_{\varpi}} := \Shv_{\arc_{\varpi}}(\Sch_{K_{\le 1}, \qcqs})
\end{equation*}
the category of $\Set$-valued sheaves on the site $(\Sch_{K_{\le 1}, \qcqs}, \arc_{\varpi})$; it follows from Proposition \ref{finitary_implies_coh} that $\cX_{\arc_{\varpi}}$ is a coherent topos. However, recall that the same topos can be realized as the category of sheaves on different sites; for our purposes it will be convenient to have various realizations of the $\arc_{\varpi}$-topos at hand, we proceed by describing the different sites that realize the $\arc_\varpi$-topos.
\begin{enumerate}[(1)]
	\item Consider the full subcategory of affine schemes $\CAlg_{K_{\le 1}}^{\op} =: \Aff_{K_{\le 1}} \subset \Sch_{K_{\le 1}, \qcqs}$, we say that a morphism $X \rightarrow Y$ is an $\arc_\varpi$ cover in $\Aff_{K_{\le 1}}$ if when considered as a morphism in $\Sch_{K_{\le 1}, \qcqs}$ it is an $\arc_{\varpi}$-cover. This determines a finitary Grothendieck topology $\arc_{\varpi}$ on $\Aff_{K_{\le 1}}$, and since $\Aff_{K_{\le 1}} \subset \Sch_{K_{\le 1}, \qcqs}$ forms a basis for the $\arc_{\varpi}$-topology (Lemma \ref{faithfully_flat_implies_arc_cover}) it follows that the canonical map $\Aff_{K_{\le 1}} \hookrightarrow \Sch_{K_{\le 1}, \qcqs}$ induces an equivalence of categories
	\begin{align*}
		\Shv_{\arc_{\varpi}} (\Sch_{K_{\le 1}, \qcqs}) \overset{\simeq}{\longrightarrow} \Shv_{\arc_\varpi} (\Aff_{K_{\le 1}})
	\end{align*}
	\item Under the equivalence $(-)_{\le 1}: \Ban_K^{\contr} \simeq \CAlg_{K_{\le 1}}^{\wedge a \tf}: (-)[\frac{1}{\varpi}]$ of Proposition \ref{equiv_almost_tf_banach} we can consider the category $\Ban_K^{\contr, \op}$ as a full subcategory of $\CAlg_{K_{\le 1}}^{\op} \subset \Sch_{K_{\le 1}, \qcqs}$; we say that a morphism $\cM(A) \rightarrow \cM(B)$ in $\Ban_{K_{\le 1}}^{\contr, \op}$ is an $\arc_{\varpi}$-cover if the induced map of compact hausdorff spaces $|\cM(A)| \rightarrow |\cM(B)|$ is surjective, which is equivalent to the requirement that $\Spec(A_{\le 1}) \rightarrow \Spec(B_{\le 1})$ is an $\arc_\varpi$-cover (Proposition \ref{arc_pi_cover_for_banach}). We claim that $\CAlg_{K_{\le 1}}^{\wedge a \tf} \subset \CAlg_{K_{\le 1}}$ is a basis for the $\arc_\varpi$-topology. Indeed, for any element $R \in \CAlg_{K_{\le 1}}$, $V$ a $\varpi$-complete rank one valuation ring with faithfully flat structure map $K_{\le 1} \rightarrow V$ (as in Definition \ref{defn_arc_pi}), and any map $R \rightarrow V$, then there exists an essentially unique factorization as $R \rightarrow R^{\wedge a \tf} \rightarrow V$, proving the claim that $\CAlg_{K_{\le 1}}^{\wedge a \tf} \subset \CAlg_{K_{\le 1}}$ is a basis for the $\arc_\varpi$-topology. Thus, it follows that the canonical map $\Ban_{K}^{\contr, \op} \hookrightarrow \Aff_{K_{\le 1}}$ determines an equivalence of categories
	\begin{align*}
		\Shv_{\arc_{\varpi}}(\Aff_{K_{\le 1}}) \overset{\simeq}{\longrightarrow} \Shv_{\arc_{\varpi}}(\Ban_K^{\contr, \op}) \simeq \Shv_{\arc_{\varpi}} (\CAlg_{K_{\le 1}}^{\wedge a \tf, \op})
	\end{align*}
	\item Recall that the equivalence $(-)_{\le 1}: \Ban_K^{\contr} \simeq \CAlg_{K_{\le 1}}^{\wedge a \tf}: (-)[\frac{1}{\varpi}]$ of Proposition \ref{equiv_almost_tf_banach} induces an equivalence $(-)_{\le 1}: \Perfd_{K_{\le 1}}^{\Ban} \simeq \Perfd_{K_{\le 1}}^{\Prism a}: (-)[\frac{1}{\varpi}]$ by Proposition \ref{equiv_perfd_ban_tic}. We say that a morphism $\cM(A) \rightarrow \cM(B)$ in $\Perfd_{K}^{\Ban}$ is an $\arc_{\varpi}$-cover if it the induced map of compact Hausdorff spaces $|\cM(A)| \rightarrow |\cM(B)|$ is surjective, which is equivalent to the requirement that $\Spec(A_{\le 1}) \rightarrow \Spec(B_{\le 1})$ is an $\arc_{\varpi}$-cover (Proposition \ref{arc_pi_cover_for_banach}). We claim that $\Perf_{K_{\le 1}}^{\Prism a} \subset \CAlg_{K_{\le 1}}$ is a basis for the $\arc_{\varpi}$-topology; we already showed in Remark \ref{canonical_arc_pi_covers} that for any $R \in \CAlg_{K_{\le 1}}^{\wedge a \tf}$ there exists a $P \in \Perfd_{K_{\le 1}}^{\Prism}$ such that $R \rightarrow P$ is an $\arc_{\varpi}$-cover, it remains to show that the map $R \rightarrow P^a$ is also an $\arc_{\varpi}$-cover. Indeed, let $V$ a $\varpi$-complete rank one valuation ring with faithfully flat structure map $K_{\le 1} \rightarrow V$ (as in Definition \ref{defn_arc_pi}), and any map $P \rightarrow V$, then there exists an essentially unique factorization as $P \rightarrow P^a \rightarrow V$, proving the claim that $\Perfd_{K_{\le 1}}^{\Prism a} \subset \CAlg_{K_{\le 1}}^{\wedge a \tf}$ forms a basis for the $\arc_\varpi$-topology. Thus, it follows that the canonical map $\Perfd_{K_{\le 1}}^{\Prism a} \hookrightarrow \CAlg_{K_{\le 1}}^{\wedge a \tf}$ determines an equivalence of categories
	\begin{align*}
		\Shv_{\arc_\varpi} (\Ban_{K}^{\contr, \op}) \simeq \Shv_{\arc_\varpi} (\CAlg_{K_{\le 1}}^{\wedge a \tf, \op}) \overset{\simeq}{\longrightarrow} \Shv_{\arc_\varpi}(\Perfd_{K_{\le 1}}^{\Prism a, \op}) \simeq \Shv_{\arc_\varpi}(\Perfd_{K}^{\Ban, \op})
	\end{align*}
\end{enumerate}
Finally, let us remaind the reader about our convention that the category $\Sch_{K_{\le 1}, \qcqs}$ is the category of all qcqs schemes over $K_{\le 1}$ with cardinality $< \kappa$, where $\kappa$ is an uncountable strong limit cardinal. As all the other categories we have considered are subcategories of $\Sch_{K_{\le 1}, \qcqs}$, it follows that we are also implicitly imposing a cardinality bound of $\kappa$ in all other categories.
\end{const}

\begin{prop}\label{basic_arc_pi_properties} The category $\cX_{\arc_\varpi}$ has the following properties:
\begin{enumerate}[(1)]
	\item The category $\cX_{\arc_{\varpi}}$ is a coherent topos.
	\item The sheafified Yoneda functor $\Yo_{\arc_{\varpi}}: \Sch_{K_{\le 1}, \qcqs} \rightarrow \Shv_{\arc_{\varpi}}(\Sch_{K_{\le 1}, \qcqs}) \simeq \cX_{\arc_\varpi}$ has its image contained in $\cX_{\arc_{\varpi}, \qcqs}$, we will sometimes write $X_{\arc_\varpi}$ for $\Yo_{\arc_\varpi}(X)$ where $X \in \Sch_{K_{\le 1}, \qcqs}$. Furthermore, the composition 
	\begin{align*}
		\Perfd_{K_{\le 1}}^{\Prism a, \op} \hookrightarrow \Sch_{K_{\le 1}, \qcqs} \overset{\Yo_{\arc_\varpi}}{\longrightarrow} \cX_{\arc_\varpi}
	\end{align*}
	is fully faithfull.
	\item The sheafified Yoneda functor $\Yo_{\arc_\varpi}: \Ban_K^{\contr, \op} \rightarrow \Shv_{\arc_\varpi}(\Ban_K^{\contr, \op}) \simeq \cX_{\arc_\varpi}$ has its image contained in $\cX_{\arc_\varpi, \qcqs}$, we will sometimes write $X_{\arc_\varpi}$ for $\Yo_{\arc_\varpi}(X)$ where $X \in \Ban_K^{\contr, \op}$. Furthermore, the composition
	\begin{align*}
		\Perfd_{K}^{\Ban, \op} \hookrightarrow \Ban_K^{\contr, \op} \overset{\Yo_{\arc_\varpi}}{\longrightarrow} \cX_{\arc_\varpi}
	\end{align*}
	is fully faithfull.
\end{enumerate}
For the sake of completeness let us mention that the functors $\Yo_{\arc_\varpi}: \Sch_{K_{\le 1}, \qcqs} \rightarrow \cX_{\arc_\varpi}$ and $\Yo_{\arc_\varpi}: \Ban_K^{\contr, \op} \rightarrow \cX_{\arc_\varpi}$ are never fully faithful, as seen in Example \ref{banach_yoneda_arc_pi}.
\end{prop}

\begin{proof} Since the $\arc_\varpi$ topology on $\Sch_{K_{\le 1}, \qcqs}$ is finitary it follows from Proposition \ref{finitary_implies_coh} that $\cX_{\arc_\varpi}$ is a coherent topos, proving (1). Again by Proposition \ref{finitary_implies_coh} it follows that the image of the sheafified Yoneda functors $\Yo_{\arc_\varpi}$ have their image contained in $\cX_{\arc_\varpi, \qcqs}$, proving the first part of (2) and (3). Finally, the fact that the functors $\Perfd_{K_{\le 1}}^{\Prism a, \op} \rightarrow \cX_{\arc_\varpi}$ and $\Perfd_{K}^{\Ban, \op} \rightarrow \cX_{\arc_\varpi}$ are fully faithful is a concequence of Propositions \ref{arc_pi_descent_for_perfectoids} and \ref{arc_pi_descent_for_bananach_perfectoids} respectively.
\end{proof}

\begin{example}\label{banach_yoneda_arc_pi} Realizing the topos $\cX_{\arc_\varpi}$ as $\Shv_{\arc_\varpi}(\Ban_K^{\contr, \op})$ we obtain the sheafified Yoneneda functor $\Yo_{\arc_\varpi}: \Ban_{K}^{\contr, \op} \rightarrow \cX_{\arc_\varpi}$, let us provide a concrete description of this functor. Recall that $\Yo_{\arc_\varpi}$ is realized as the composition of the functors
\begin{align*}
	\Yo_{\arc_\varpi}: \Ban_K^{\contr, \op} \overset{\Yo}{\longrightarrow} \PreShv(\Ban_K^{\contr, \op}) \overset{L_{\arc_\varpi}}{\longrightarrow} \Shv_{\arc_\varpi}(\Ban_K^{\contr, \op}) = \cX_{\arc_\varpi}
\end{align*}
where $L_{\arc_\varpi}$ is the sheafification functor with respect to the $\arc_\varpi$-topology; and where $\Yo (\cM(A))$ is given by the functor $\Maps_{\Ban_K^{\contr}} (A, -): \Ban_K^{\contr} \rightarrow \Set$ for any Banach $K$-algebra $A$. By virtue of Corollary \ref{arc_pi_descent_for_bananach_perfectoids} we learn that the restriction 
\begin{align*}
	\Maps_{\Ban_{K}^{\contr}}(A, -)|_{\Perfd_{K}^{\Ban}}: \Perfd_{K}^{\Ban} \longrightarrow \Set
\end{align*}
is already an $\arc_\varpi$-sheaf. Thus, since we can also realize $\cX_{\arc_\varpi}$ as $\Shv_{\arc_\varpi}(\Perfd_K^{\Ban, \op})$ we learn that $\cM(A)_{\arc_\varpi}$ is given by the functor $\Maps_{\Ban_{K}^{\contr}}(A, -)|_{\Perfd_{K}^{\Ban}}$. In particular, this argument shows that if $A \rightarrow A^u$ is the uniformization map, then the induced map $\cM(A^u)_{\arc_\varpi} \rightarrow \cM(A)_{\arc_\varpi}$ is an isomorphism in $\cX_{\arc_\varpi}$, showing that the sheafified Yoneda functor $\Yo_{\arc_\varpi}: \Ban_K^{\contr, \op} \rightarrow \cX_{\arc_\varpi}$ is never fully faithfull.

However, despite the fact that $\Yo_{\arc_\varpi}$ is only naturally defined for the category of Banach $K$-algebras with contractive maps, since $\Yo_{\arc_\varpi}$ identifies a Banach $K$-algebra with its uniformization (and bounded morphism between uniform Banach algebras are contractive) we can define a functor $\Ban_K^{\op} \rightarrow \cX_{\arc_\varpi}$ as the composition
\begin{align*}
	\Ban_K^{\op} \overset{(-)^u}{\longrightarrow} \Ban_K^{\contr, \op} \overset{\Yo_{\arc_\varpi}}{\longrightarrow} \cX_{\arc_\varpi}
\end{align*}
which naturally extends the map $\Yo_{\arc_\varpi}$ along the inclusion $\Ban_{K}^{\contr} \subset \Ban_K$.
\end{example}

\begin{const}[The Berkovich Functor]\label{berko_funct}  Recall that Berkovich defined a functor $|-|: \Ban_K^{\contr, \op} \rightarrow \Comp$ sending $\cM(A) \mapsto |\cM(A)|$. From Proposition \ref{gluing_berko_sp} we learn that if $\cM(B) \rightarrow \cM(A)$ is a morphism in $\Ban_K^{\contr, \op}$, which is an $\arc_{\varpi}$-cover (in the sense of Construction \ref{const_arc_pi_topos}(2)), then the canonical map 
\begin{align*}
	\coeq(|\cM(B) \times_{\cM(A)} \cM(B)| \rightrightarrows |\cM(B)|) \rightarrow |\cM(A)|
\end{align*}	
is an isomorphism of compact hausdorff spaces. Furthermore, combining Lemma \ref{yo_eff_preserves_general_coeq} and Proposition \ref{general_fiber_prod_berko_sp} we learn that the functor $\Yo_{\eff}(-)$ preserves the above coequalizer, that is, the canonical map
\begin{align*}
	\coeq(\Yo_{\eff} |\cM(B) \times_{\cM(A)} \cM(B)| \rightrightarrows \Yo_{\eff}|\cM(B)|) \rightarrow \Yo_{\eff}|\cM(A)|
\end{align*}
is an isomorphism of condensed sets. Finally, realizing $\cX_{\arc_{\varpi}}$ as $\Shv_{\arc_{\varpi}}(\Ban_K^{\contr, \op})$, we learn from  Proposition \ref{prop_sheaves}(4) that there is an essentially unique colimit preserving functor $|-|: \cX_{\arc_{\varpi}} \rightarrow \Cond$ making the following diagram commute
\begin{cd}
	\Ban_K^{\contr, \op} \ar[r, " \vert - \vert"] \ar[d, "\Yo_{\arc_{\varpi}}", swap] & \Comp \ar[d, hook, "\Yo_{\eff}"]\\
	\cX_{\arc_{\varpi}} \ar[r, "\vert - \vert"] & \Cond
\end{cd}
We call the resulting colimit preserving functor $|-|: \cX_{\arc_{\varpi}} \rightarrow \Cond$ the Berkovich functor. Implicit in the construction of the Berkovich functor is the fact that the map $|-|: \Ban_{K}^{\contr, \op} \rightarrow \Comp$ respects the bounds $< \kappa$ with respect the (implicit) uncountable strong limit cardinal $\kappa$. 
\end{const}

\begin{prop}\label{berko_funct_stability} The Berkovich functor $|-|: \cX_{\arc_\varpi} \rightarrow \Cond$ has the following stability properties
\begin{enumerate}[(1)]
	\item Let $F \rightarrow G$ be a morphism in $\cX_{\arc_\varpi}$. If $F \twoheadrightarrow G$ is an epimorphism, then $|F| \twoheadrightarrow |G|$ is an epimorphism of condensed sets.
	\item Let $F_1 \rightarrow F_2 \leftarrow F_3$ be a pair of morphisms in $\cX_{\arc_\varpi}$. Then, the canonical map of condensed sets
	\begin{align*}
		|F_1 \times_{F_2} F_3| \twoheadrightarrow |F_1| \times_{|F_2|} |F_3|
	\end{align*}
	is an epimorphism.
	\item Let $F \rightarrow G$ be a morphism in $\cX_{\arc_\varpi}$. If $F \hookrightarrow G$ is a monomorphism, then $|F| \hookrightarrow |G|$ is a monomorphism of condensed sets.
	\item Let $F_1 \rightarrow F_2 \leftarrow F_3$ be a pair of morphisms in $\cX_{\arc_\varpi}$, and assume that $F_1 \hookrightarrow F_2$ is a monomorphism. Then, the canonical map of condensed sets
	\begin{align*}
		|F_1 \times_{F_2} F_3| \overset{\simeq}{\longrightarrow} |F_1| \times_{|F_2|} |F_3|
	\end{align*}
	is an isomorphism.
	\item If $F$ is a quasicompact object of $\cX_{\arc_\varpi}$, then $|F|$ is a quasicompact condensed set.
	\item If $F$ is a qcqs object of $\cX_{\arc_\varpi}$, then $|F|$ is a qcqs condensed set (equivalently, a compact hausdorff space).
	\item If $F$ is a quasiseparated object of $\cX_{\arc_\varpi}$, then $|F|$ is a quasiseparated condensed set.
\end{enumerate}
\end{prop}

\begin{proof}  First we prove (1). If $F \rightarrow G$ is an epimorphism, then by virtue of the fact we are working on a topos it follows that the canonical map $\coeq(F \times_G F \rightrightarrows F) \overset{\simeq}{\rightarrow} G$ is an isomorphism, and since the Berkovich functor $|-|$ is colimit preserving it follows that the canonical map $\coeq(|F \times_G F| \rightrightarrows |F|) \overset{\simeq}{\rightarrow} |G|$ is an isomorphism. Thus, since $\Cond$ is a topos, it follows that $|F| \rightarrow |G|$ is an epimorphism.

In order to prove (2), let $p: X =: F_1 \sqcup F_2 \sqcup F_3 \twoheadrightarrow F_2$ be the canonically induced morphism, which is also an epimorphism. We claim that the induced map $|X \times_{F_2} X| \rightarrow |X| \times_{|F_2|} |X|$ is an epimorphism. Indeed, since $X \rightarrow F_2$ is an epimorphism and $\cX_{\arc_\varpi}$ is a topos we have that the canonical map $\coeq(X \times_{F_2} X \rightrightarrows X) \rightarrow F_2$ is an isomorphism, and since the Berkovich functor is colimit preserving it follows that the canonical map of condensed sets $\coeq(|X \times_{F_2} X| \rightrightarrows |X|) \rightarrow |F_2|$ is also an isomorphism. Thus, by virtue of the fact that $\Cond$ is a topos it follows that the canonical map $|X \times_{F} X| \rightarrow |X| \times |X|$ factors uniquely as
\begin{align*}
	|X \times_{F_2} X| \twoheadrightarrow |X| \times_{|F_2|} |X| \hookrightarrow |X| \times |X|
\end{align*}
proving that the desired map $|X \times_{F_2} X| \rightarrow |X| \times_{|F_2|} |X|$ is an epimorphism. The result then follows from the fact that if $|X \times_{F_2} X| \rightarrow |X| \times_{|F_2|} |X|$ is an epimorphism then so is $|F_1 \times_{F_2} F_3| \rightarrow |F_1| \times_{|F_2|} |F_3|$.

Next we prove (3). To show that $|F| \rightarrow |G|$ is a monomorphism of condensed sets, it suffices to show that the diagonal map $\Delta: |F| \rightarrow |F| \times_{|G|} |F|$ is an isomorphism, and since $\Delta$ is always a monomorphism it remains to show that it is an epimorphism. Indeed, since $F \rightarrow G$ is a monomorphism, it follows that the canonical map $F \rightarrow F \times_G F$ is an isomorphism and then the result follows from part (2).

To prove (4) recall that a morphism in a topos is an isomorphism if and only if it is a monomorphism and a epimorphism, thus by part (2) it remains to show that the induced map $|F_1 \times_{F_2} F_3| \rightarrow |F_1| \times_{|F_2|} |F_3|$ is a monomorphism. Since monomorphisms are stable under basechange and the Berkovich functor preserves monomorphisms, it follows that the induced map $|F_1 \times_{F_2} F_3| \hookrightarrow |F_3|$ is a monomorphism, and since it factors as
\begin{align*}
	|F_1 \times_{F_2} F_3| \rightarrow |F_1| \times_{|F_2|} |F_3| \rightarrow |F_3|
\end{align*}
it follows that $|F_1 \times_{F_2} F_3| \rightarrow |F_1| \times_{|F_2|} |F_3|$ is also a monomorphism, proving the claim.

In order to prove $(5)$, first recall that we can realize the category $\cX_{\arc_\varpi}$ as $\Shv_{\arc_\varpi}(\Ban_{K}^{\op, \contr})$ and that since $F$ is quasicompact there exists a finite collection of morphisms $\{\cM(A_i)_{\arc_\varpi} \rightarrow F\}_{i \in I}$ such that the induced map $\sqcup \cM(A_i)_{\arc_\varpi} \rightarrow F$ is an epimorphism. Then by part (1) it follows that $\sqcup |\cM(A_i)_{\arc_\varpi}| \rightarrow |F|$ is an epimorphism, and since $\sqcup |\cM(A_i)_{\arc_\varpi}|$ is a quasicompact condensed set, the result follows from Lemma \ref{alt_defn_qc}.

For the proof of $(6)$ we follow the notation of the proof of $(5)$, and set $X := \sqcup \cM(A_i)_{\arc_\varpi} \rightarrow F$. By virtue of Proposition \ref{prop_qcqs_obj} and Lemma \ref{alt_defn_qc}, in order to show that $|F|$ is qcqs it suffices to show the canonical map $|X \times_F X| \rightarrow |X| \times_{|F|} |X|$ is an epimorphism of condensed sets, but this follows from part $(2)$, completing the proof of $(6)$.

Statement (7) is a direct consequence of Proposition \ref{saturate_qs} and the fact that the Berkovich functor is colimit preserving and preserves monomorphisms and qcqs objects.
\end{proof}

The following result provides a description of $|F|(*)$, for any $F \in \cX_{\arc_\varpi}$, in terms of algebraic data. This result is analogous to the one proven in \cite[Proposition 12.7]{diamonds_scholze} for small $v$-stacks.

\begin{prop}[Points]\label{pts_arc_pi_shv} Let $F$ be an object of $\cX_{\arc_\varpi}$ and realize $\cX_{\arc_\varpi}$ as $\Shv_{\arc_\varpi}(\Perfd_K^{\Ban, \op})$. We will abuse notation and denote by $\cM(A) \in \cX_{\arc_\varpi}$ the image of $\cM(A) \in \Perfd_K^{\Ban, \op}$ under the fully faithful functor $\Perfd_{K}^{\Ban, \op} \hookrightarrow \cX_{\arc_\varpi}$. Then, the following hold
\begin{enumerate}[(1)]
	\item For each $x \in |F|(*)$ there exists a perfectoid non-archimedean field $L/K$ and a morphism $\cM(L) \rightarrow F$ such that the induced map $\pt = |\cM(L)| \rightarrow |F|$ is exactly $x: \pt \rightarrow |F|$.
	\item Let $\cM(L_1) \rightarrow F \leftarrow \cM(L_2)$ be a pair of morphism in $\cX_{\arc_\varpi}$ from perfectoid non-archimedean fields over $K$, having the same image on the induced morphism of sets $|\cM(L_1)|(*) \rightarrow |F|(*) \leftarrow |\cM(L_2)|(*)$. Then, there exists a third perfectoid non-archimedean field $\cM(L_3)$ and a pair of morphisms $\cM(L_1) \leftarrow \cM(L_3) \rightarrow \cM(L_2)$ making the following diagram commute
	\begin{cd}
		\cM(L_3) \ar[r] \ar[d] & \cM(L_2) \ar[d] \\
		\cM(L_1) \ar[r] & F
	\end{cd}
	By part (1), this is equivalent to the requirement that $\cM(L_1) \times_F \cM(L_2) \not= \emptyset$.
\end{enumerate}
\end{prop}

\begin{proof} We begin by proving (1). For any $F \in \cX_{\arc_\varpi}$ there exists a (possibly infinite) collection of morphisms $\{\cM(A_i) \rightarrow F\}_{i \in I}$ from perfectoid Banach algebras such that the induced map $\sqcup_{i \in I} \cM(A_i) \rightarrow F$ is an epimorphism, which in turn implies that the induced map of sets $\sqcup_{i \in I} |\cM(A_i)|(*) \rightarrow |F|(*)$ is surjective, by a combination of Proposition \ref{berko_funct_stability} and the fact that an epimorphism of condensed sets $X \rightarrow Y$ must induce a surjective map $X(*) \rightarrow Y(*)$. Thus, for each $x \in |F|(*)$ there exists a point $y \in |\cM(A_i)|(*)$, for some $i \in I$, such that $x \mapsto y$ under the map $|\cM(A_i)|(*) \rightarrow |F|(*)$. The result then follow from the fact that for any $y \in |\cM(A_i)|(*)$ the completed residue field $\cH(y)$ of $A_i$ is perfectoid (cf. Theorem \ref{stalks_perfectoid}).

By the work done in part (1), in order to proof (2) it remains to show that $\cM(L_2) \times_F \cM(L_1) \not= \emptyset$. Indeed, from the hypothesis it is clear that $|\cM(L_1)| \times_{|F|} |\cM(L_2)|(*) \not= \emptyset$, thus since Proposition \ref{berko_funct_stability} guarantees that the the canonical map $|\cM(L_1) \times_{F} \cM(L_2)| \rightarrow |\cM(L_1)| \times_{|F|} |\cM(L_2)|$ is an epimorphism of condensed sets. This proves that $\cM(L_1) \times_{F} \cM(L_2) \not= \emptyset$, since the Berkovich functor is colimit preserving and so it preserves the initial object.
\end{proof}

We show that in certain situations we can check whether a morphism $F \rightarrow G$ is an epimorphism by only checking that associated map of "topological spaces" $|F| \rightarrow |G|$ is surjective at the level of its underlying set. This is similar to \cite[Lemma 12.11]{diamonds_scholze} (compare with \cite[Lemma 4.21]{arc_topology}), which was proven in the setting of $v$-sheaves. 

\begin{prop}[Epimorphisms]\label{epi_arc_topos} Let $F \rightarrow G$ be a morphisms in $\cX_{\arc_{\varpi}}$, and assume that $G$ is qcqs and $F$ is quasicompact. Then, the following are equivalent
\begin{enumerate}[(1)]
	\item The morphism $F \rightarrow G$ is an epimorphism.
	\item Realizing $\cX_{\arc_{\varpi}}$ as $\Shv_{\arc_\varpi}(\Sch_{K_{\le 1, \qcqs}})$, the morphism $F \rightarrow G$ has the $\arc_\varpi$-lifting property: for each $\varpi$-complete rank one valuation ring $V$ with a faithfully flat structure map $K_{\le 1} \rightarrow V$ and a section $g \in G(V)$, there exists a faithfully flat extension of $\varpi$-complete rank one valuation rings $V \rightarrow W$ and a section $f \in F(W)$ lifting the image of $g$ in $G(V) \rightarrow G(W)$.
	\item[(2')] Realizing $\cX_{\arc_{\varpi}}$ as $\Shv_{\arc_{\varpi}}(\Perfd_K^{\Ban, \op})$, the morphism $F \rightarrow G$ has the $\arc_{\varpi}$-lifting property: for each perfectoid non-archimedean field $L/K$ and a section $g \in G(L)$, there exists a morphism $L \rightarrow W$ of perfectoid non-archimedean fields and a section $f \in F(W)$ lifting the image of $g$ in $G(L) \rightarrow G(W)$.
	\item The induced map of condensed sets $|F| \rightarrow |G|$ is an epimorphism, which by Proposition \ref{epi_condensed_sets}(3) is equivalent to the requirement that $|F|(*) \rightarrow |G|(*)$ is a surjective map of sets.
\end{enumerate}
\end{prop}
		
\begin{proof} It follows from Lemma \ref{morphism_of_shv} that $(1) \Rightarrow (2)$, and from Proposition \ref{berko_funct_stability} that $(1) \Rightarrow (3)$. The equivalence between $(2)$ and $(2')$ is a direct consequence of Proposition \ref{equiv_almost_tf_banach} and Lemma \ref{na_field_rk_one_dict}. We begin by proving that $(2') \Rightarrow (1)$. We will abuse notation and denote the fully faithful functor $\Perfd_{K}^{\Ban, \op} \hookrightarrow \cX_{\arc_\varpi}$ as $\cM(A) \mapsto \cM(A)$. Let $B \rightarrow A$ be a morphism in $\Perfd_{K}^{\Ban}$, and $\cM(A) \rightarrow \cM(B)$ the induced morphism in $\cX_{\arc_\varpi}$; if the morphism $\cM(A) \rightarrow \cM(B)$ has the $\arc_\varpi$-lifiting property, then by construction $\cM(A) \rightarrow \cM(B)$ is an epimorphism in $\cX_{\arc_\varpi}$. More generally, since $G$ is quasicompact there exists a $A \in \Perfd_{K}^{\Ban}$ together with an epimorphism $\cM(A) \twoheadrightarrow G$, and since $F$ is quasicompact and $G$ is quasiseparated it follows that the fiber product $F \times_G \cM(A)$ is again quasicompact, and so there exists another $B \in \Perfd_{K}^{\Ban}$ together with an epimorphism $\cM(B) \twoheadrightarrow F \times_G \cM(A)$. Then, since $(1) \Rightarrow (2)$ and the $\arc_\varpi$-lifting property is stable under composition and basechange it follows that the composition $\cM(B) \rightarrow \cM(A)$ has the $\arc_\varpi$-lifiting property and so it is an epimorphism. To conclude, since epimorphisms are stable under composition it follows that $\cM(B) \rightarrow G$ is an epimorphism, and since this map factors as $\cM(B) \rightarrow F \rightarrow G$ it follows that the original map $F \rightarrow G$ is an epimorphism in $\cX_{\arc_\varpi}$.

Finally, it remains to show that $(3) \Rightarrow (2')$, thus we assume that $|F|(*) \rightarrow |G|(*)$ is a surjective map of sets. We will follow the notation of the previous paragraph. By Proposition \ref{pts_arc_pi_shv}, it suffices to show that for any perfectoid non-archimedean field $L/K$ together with a morphism $\cM(L) \rightarrow G$ the fiber product $F \times_G \cM(L) \not= \emptyset$. From the hypothesis that $|F| \rightarrow |G|$ is an epimorphism, it follows that $|F|\times_{|G|} |\cM(L)| \not= \emptyset$, then from Proposition \ref{berko_funct_stability} we learn that the induced map
\begin{align*}
	|F \times_G \cM(L)| \rightarrow |F|\times_{|G|} |\cM(L)|
\end{align*}
is an epimorphism, showing that $|F \times_G \cM(L)|\not= \emptyset$ and so $F \times_G \cM(L) \not= \emptyset$, again by virtue of Proposition \ref{pts_arc_pi_shv}.
\end{proof}

Parallel to the theory of completed residue fields of Berkovich, we stablish that $\arc_\varpi$-sheaves also have a well behaved theory of completed residue fields.

\begin{prop}[Residue Fields]\label{res_fields_arc_topos} Let $F$ be a quasiseparated object of $\cX_{\arc_\varpi}$, then the following hold.
\begin{enumerate}[(1)]
	\item Let $X$ be a quasicompact object of $\cX_{\arc_\varpi}$. Then, for any morphism $p: X \rightarrow F$ the resulting object $\im(p)$ is qcqs.
	\item For any $x \in |F|(*)$ there exists an essentially unique monomorphism $F_x \hookrightarrow F$ from a qcqs object $F_x \in \cX_{\arc_\varpi}$ such that the morphisms $F_x \hookrightarrow F$ gets mapped to $x: \pt \hookrightarrow |F|$ under the Berkovich functor. We call the resulting monomorphism $F_x \hookrightarrow F$ the completed residue field of $F$ at $x \in |F|(*)$.
	\item Let $F_x \hookrightarrow F$ be the completed residue field of $F$ at $x \in |F|(*)$. Then, for any morphism $Z \rightarrow F$ from a quasicompact object $Z$, which gets mapped to $x: |Z| = \pt \rightarrow |F|$ under the Berkovich functor, there exists a unique factorization of $Z \rightarrow F$ as
	\begin{align*}
		Z \rightarrow F_x \hookrightarrow F
	\end{align*}
	In particular, if $Z = \cM(L)_{\arc_\varpi}$ where $L$ is a non-archimedean field, there exists a unique factorization of $\cM(L)_{\arc_\varpi} \rightarrow F$ as $\cM(L)_{\arc_\varpi} \rightarrow F_x \rightarrow F$.
	\item Assume that $F$ is of the form $\cM(A)_{\arc_\varpi}$ for some Banach $K$-algebra $A$. For a fixed $x \in |F|(*)$, the resulting morphism $F_x \hookrightarrow F$ is given by $\cM(\cH(x))_{\arc_\varpi} \hookrightarrow \cM(A)_{\arc_\varpi}$, where $\cH(x)$ is the completed residue field of $A$ at $x \in |\cM(A)|(*) = |\cM(A)_{\arc_\varpi}|(*)$.
\end{enumerate}
\end{prop}

\begin{proof} To prove (1), recall that since $\cX_{\arc_\varpi}$ is a topos, the morphisms $X \rightarrow F$ admits an essentially unique factorization as $X \twoheadrightarrow \im(p) \hookrightarrow F$. Then, since $X$ is quasicompact and the map $X \twoheadrightarrow \im(p)$ is an epimorphism it follows that $\im(p)$ is quasicompact by Lemma \ref{alt_defn_qc}; and since every monomorphism is quasiseparated, it follows from Proposition \ref{comp_relative_absolute_qcqs} and the fact that quasiseparated morphisms are closed under composition that $\im(p)$ is also quasiseparated.
	
Next we prove (2), we begin by constructing the desired monomorphism $F_x \hookrightarrow F$. From Proposition \ref{pts_arc_pi_shv} we know that for each $x \in |F|(*)$ there exists a perfectoid non-archimedean field $L/K$ together with a morphism $g: \cM(L) \rightarrow F$ such that the induced map $\pt = |\cM(L)| \rightarrow |F|$ is exactly the map $x: \pt \rightarrow |F|$. Then, we define the monomorphism $F_x \hookrightarrow F$ as the map induced by the canonical factorization of the map $g$
\begin{align*}
	g: \cM(L) \twoheadrightarrow F_x \hookrightarrow F && \text{where} \qquad F_x := \im(g)
\end{align*}
To show that the morphism $F_x \rightarrow F$ gets mapped $x: \pt \rightarrow F$ under the Berkovich functor it suffices to show that the induced map $|\cM(L)| \rightarrow |F_x|$ is an isomorphism. Since $|\cM(L)| = \pt$ is it is clear that $|\cM(L)| \rightarrow |F_x|$ is a monomorphism, and since the Berkovich functor preserves epimorphism it follows that $|\cM(L)| \rightarrow |F_x|$ is an isomorphism. It remains to show that the monomorphism $F_x \hookrightarrow F$ is unique. Let $Z \hookrightarrow F$ be a monomorphism from a qcqs object $Z \in \cX_{\arc_\varpi}$ such that the induced map $|Z| \hookrightarrow |F|$ is given by $x: \pt \rightarrow |F|$, we need to show that there is an isomorphism $F_x \simeq Z$ which commutes with the monomorphisms towards $F$. Hence, it suffices to show that the fiber product $F_x \times_F Z \not= \emptyset$ and that the projection maps $F_x \times_F Z \rightarrow F_x$ and $F_x \times_F Z \rightarrow Z$ are isomorphisms. By Proposition \ref{pts_arc_pi_shv} it follows that $F_x \times_F Z \not= \emptyset$, and by virtue of Proposition \ref{epi_arc_topos} it follows that both projection maps $F_x \times_F Z \rightarrow F_x$ and $F_x \times_F Z \rightarrow Z$ are epimorphisms. Then, the result follows as monomorphisms are stable under basechange.

In order to proof (3), recall from the discussion in the previous paragraph that the map $f: Z \rightarrow F$ admits a essentially unique factorization as
\begin{align*}
	f: Z \twoheadrightarrow Y \hookrightarrow F && \text{where} \qquad Y := \im(f)
\end{align*}
such that $Y$ is a qcqs object and the induced morphism $|Y| \rightarrow |F|$ is exactly $x: \pt \hookrightarrow |F|$. Then, the uniqueness of completed residue fields, as proven in part (2), shows that there exists an isomorphism $Y \simeq F_x$ respecting the monomorphism towards $F$.

For the proof of (4) we make critical use of Proposition \ref{properties_stalk}. Recall that for each $x \in |\cM(A)|(*)$ there exists a monomorphism $\cM(\cO_{A,x}^{\wedge}) \hookrightarrow \cM(A)$ in the category $\Ban_K^{\contr, \op}$ such that under the induced map $|\cM(\cO_{A, x}^{\wedge})| \rightarrow |\cM(A)|$ is exactly the map $x: \pt \rightarrow |\cM(A)|$. Then, since the sheafified Yoneda functor $\Yo_{\arc_\varpi}: \Ban_{K}^{\contr, \op} \rightarrow \cX_{\arc_\varpi}$ preserves finite limits it follows that it preserves monomorphisms, thus we have an induced monomorphism $\cM(\cO_{A, x}^{\wedge})_{\arc_\varpi} \rightarrow \cM(A)_{\arc_\varpi}$, which gets mapped to $x: \pt \rightarrow |\cM(A)_{\arc_\varpi}|$ under the Berkovich functor. Furthermore, by Proposition \ref{properties_stalk} we learn that the uniformization map is given by $\cO_{A, x}^{\wedge} \rightarrow \cO_{A, x}^{u} \simeq \cH(x)$, thus it remains to show that the induced map $\cM(\cH(x))_{\arc_\varpi} \rightarrow \cM(\cO_{A, x}^{\wedge})_{\arc_\varpi}$ is an isomorphism. For this, recall that under the identity $\cX_{\arc_\varpi} \simeq \Shv_{\arc_\varpi}(\Perfd_K^{\Ban, \op})$ it follows that the functor $\cM(\cO_{A, x}^{\wedge})_{\arc_\varpi}: \Perfd_K^{\Ban} \rightarrow \Set$ is given by the restriction of the functor
\begin{align*}
	\Maps_{\Ban_K^{\contr}}(\cO_{A, x}^{\wedge}, -): \Ban_K^{\contr} \rightarrow \Set
\end{align*}
to the full-subcategory $\Perfd_K^{\Ban} \subset \Ban_K^{\contr}$. Hence, the result follows from the fact that all perfectoid Banach $K$-algebras $P$ are uniform and so any map $\cO_{A, x}^{\wedge} \rightarrow P$ factors as $\cO_{A, x}^{\wedge} \rightarrow \cO_{A,x}^{u} \simeq \cH(x) \rightarrow P$, proving that the induced map $\cM(\cH(x))_{\arc_\varpi} \rightarrow \cM(\cO_{A, x}^{\wedge})_{\arc_\varpi}$ is an isomorphism.
\end{proof}

Similar to \cite[Lemma 12.5]{diamonds_scholze} (compare also with \cite[Lemma 4.21]{arc_topology}), we show that being an isomorphism of $\arc_\varpi$-sheaves is equivalent to being an $\arc_\varpi$-equivalence. Furthermore, we provide a more topological characterization of this condition using our theory of completed residue fields.

\begin{prop}[Isomorphisms]\label{iso_arc_topos} Let $F \rightarrow G$ be a morphism in $\cX_{\arc_{\varpi}}$, and assume that both $G$ and $F$ are qcqs. Then, the following are equivalent
\begin{enumerate}[(1)]
	\item The morphism $F \rightarrow G$ is an isomorphism.
	\item Realizing $\cX_{\arc_\varpi}$ as $\Shv_{\arc_\varpi}(\Sch_{K_{\le 1}, \qcqs})$, the morphism $F \rightarrow G$ is an $\arc_\varpi$-equivalence: there exists a cofinal collection of perfectoid rank one valuation rings with faithfully flat structure map $K_{\le 1} \rightarrow V$ such that the induced map $F(V) \rightarrow G(V)$ is a bijection for all such $V$.
	\item[(2')] Realizing $\cX_{\arc_\varpi}$ as $\Shv_{\arc_\varpi}(\Perfd_{K}^{\Ban, \op})$, the morphism $F \rightarrow G$ is an $\arc_\varpi$-equivalence: there exists a cofinal collection of perfectoid non-archimedean fields $L/K$ such that the induced map $F(L) \rightarrow G(L)$ is a bijection for all such $L$.
	\item The induced map $|F|(*) \rightarrow |G|(*)$ is a bijection of sets, and for each $x \in |F|(*) \simeq |G|(*) \ni y$ the induced map of completed residue fields $F_x \rightarrow G_y$ is an $\arc_\varpi$-equivalence.
\end{enumerate}
\end{prop}

\begin{proof} The equivalence between $(2)$ and $(2')$ follows from Proposition \ref{equiv_almost_tf_banach} and Lemma \ref{na_field_rk_one_dict}, while the implication $(1) \Rightarrow (2) \simeq (2')$ is obvious. Next, let us prove the implication $(2') \Rightarrow (1)$. Recall that since $\cX_{\arc_\varpi}$ is a topos it suffices to show that the map $F \rightarrow G$ is both an epimorphism and a monomorphism, and it is a direct consequence of Proposition \ref{epi_arc_topos} that the map $F \rightarrow G$ is an epimorphism. To show that $F \rightarrow G$ is also a monomorphism, recall that $F \rightarrow G$ is a monomorphism if and only if the diagonal map $\Delta: F \rightarrow F \times_G F$ is an isomorphism. In general for any morphism $X \rightarrow Y$ in a topos the diagonal map $X \rightarrow X \times_Y X$ is a monomorphism, so it remains to show that $\Delta: F \rightarrow F \times_G F$ is an epimorphism, but this follows from the fact that qcqs objects are stable under fiber products, the hypothesis and Proposition \ref{epi_arc_topos}.

Finally, we prove the equivalence of $(3)$ with the other conditions. We begin by noticing that if the induced map $|F|(*) \rightarrow |G|(*)$ is a bijection, Proposition \ref{res_fields_arc_topos} shows that for each $x \in |F|(*) \simeq |G|(*) \ni y$ there exists an induced map $F_x \rightarrow G_y$. For the proof of $(1) \Rightarrow (3)$, notice that since $F \rightarrow G$ is an isomorphism so is $|F| \rightarrow |G|$ and so the induced map $|F|(*) \rightarrow |G|(*)$ is a bijection, and the uniqueness of completed residue fields of Proposition \ref{res_fields_arc_topos} guarantee that the induced map $F_x \rightarrow G_y$ is an isomorphism, and so in particular an $\arc_\varpi$-equivalence. It remains to show that $(3) \Rightarrow (2')$. Recall we have already shown that $(2') \Rightarrow (1)$, so the hypothesis that $F_x \rightarrow G_y$ is an $\arc_\varpi$-equivalence implies that for all perfectoid non-archimedean fields $L/K$ the induced map $F_x(L) \rightarrow G_y(L)$ is a bijection. Let $L/K$ be a perfectoid non-archimedean field and consider a map $\cM(L) \rightarrow G$ which induces the map $x: \pt \rightarrow |G|$ under the Berkovich functor, it remains to show that there exists a unique morphism $\cM(L) \rightarrow F$ making the following diagram commute
\begin{cd}
	& F \ar[d] \\
	\cM(L) \ar[ru, dashed] \ar[r] & G
\end{cd}
Indeed, from Proposition \ref{res_fields_arc_topos} it follows that the map $\cM(L) \rightarrow G$ admits a unique factorization as $\cM(L) \twoheadrightarrow G_y \hookrightarrow G$; thus since $F_x \simeq G_y$ it follows that there exists a unique morphism $\cM(L) \rightarrow F_x$ factoring the map $\cM(L) \rightarrow G_y$, giving us a lift $\cM(L) \rightarrow F_x \hookrightarrow F$. To show that the lift $\cM(L) \rightarrow F$ is unique, notice that any lift $\cM(L) \rightarrow F$ must satisfy that it gets mapped to $x: \pt \rightarrow |F|$ under the Berkovich functor, and so it must admit a unique factorization as $\cM(L) \rightarrow F_x \hookrightarrow F$. Thus, since we already proved that lifts to $\cM(L) \rightarrow F_x$ are unique, it follows that lifts to $\cM(L) \rightarrow F$ are unique, completing the proof.
\end{proof}

Inspired by the classification of affinoid analytic domains in classical rigid geometry, we show that in the setting of $\arc_\varpi$-sheaves that monomorphisms are exactly the analytic domains of an $\arc_\varpi$-sheaf. The analogous result in the setting of $v$-sheaves can be found in \cite[Proposition 12.15]{diamonds_scholze}.

\begin{prop}[Monomorphisms]\label{mono_arc_topos} Let $F \rightarrow G$ be a morphism in $\cX_{\arc_\varpi}$, and assume that $F$ is qcqs and $G$ is quasiseparated. Then, the following are equivalent
\begin{enumerate}[(1)]
	\item The morphism $F \hookrightarrow G$ is a monomorphism.
	\item The induced map $|F|(*) \rightarrow |G|(*)$ is an injective map of sets, and for each $x \in |F|(*) \simeq \im(|F|(*) \hookrightarrow |G|(*)) \ni y$ the induced map $F_x \rightarrow G_y$ of completed residue fields is an $\arc_\varpi$-equivalence.
	\item The morphism $F \rightarrow G$ is an analytic domain: for any object $Z \in \cX_{\arc_\varpi}$ and any morphism $Z \rightarrow G$ satisfying $\im(|Z|(*) \rightarrow |G|(*)) \subset \im(|F|(*) \rightarrow |G|(*))$, there exists a unique morphism $Z \rightarrow F$ making the following diagram commute
	\begin{cd}
		& Z \ar[d] \ar[ld, dashed] \\
		F \ar[r] & G
	\end{cd}
\end{enumerate}
\end{prop}

\begin{proof} We will first prove the equivalence $(1) \simeq (2)$. For the implication $(1) \Rightarrow (2)$, recall that Proposition \ref{berko_funct_stability} that if $F \hookrightarrow G$ is a monomorphism then so is $|F| \hookrightarrow |G|$ and so the induced map of sets $|F|(*) \hookrightarrow |G|(*)$ is also injective. It remains to show that for each $x \in |F|(*)$ and $y \in |G|(*)$ such that $x \mapsto y$, the induced map $F_x \rightarrow G_y$ is an isomorphism. However, since $F_x \rightarrow F$ is uniquely characterized as the monomorphism from a qcqs object that gets mapped to $x: \pt \rightarrow |F|$ under the Berkovich functor (cf. Proposition \ref{res_fields_arc_topos}), and the composition $F_x \rightarrow F \rightarrow G$ gets mapped to $y: \pt \rightarrow |G|$ under the Berkovich functor it follows from the uniqueness of completed residue fields that $F_x \rightarrow G_y$ is an isomorphism, and so an $\arc_\varpi$-equivalence. For the implication $(2) \Rightarrow (1)$, recall that since $\cX_{\arc_\varpi}$ is a topos the morphism $p: F \rightarrow G$ admits a unique factorization as
\begin{align*}
	F \twoheadrightarrow Z \hookrightarrow G && \text{where} \qquad Z = \im(p)
\end{align*}
Thus, it remains to show that the map $F \rightarrow Z$ is an isomorphism. Let us begin by showing that the induced map of sets $|F|(*) \rightarrow |Z|(*)$ is a bijection; from the hypothesis we know that the composition $|F|(*) \rightarrow |Z|(*) \rightarrow |G|(*)$ is an injection, so the induced map $|F|(*) \rightarrow |Z|(*)$ is injective. On the other hand since $F \rightarrow Z$ is an epimorphism by construction, and the Berkovich functor preserves epimorphisms (cf. Proposition \ref{berko_funct_stability}) it follows that the induced map $|F|(*) \rightarrow |Z|(*)$ is surjective, and therefore a bijection. By virtue of Proposition \ref{iso_arc_topos}, to show that $F \rightarrow Z$ is an isomorphism suffices to show that for $x \in |F|(*) \simeq |Z|(*) \ni z$ the induced map of completed residue fields $F_x \rightarrow Z_z$ is an $\arc_\varpi$-equivalence. By the uniqueness of completed residue fields (cf. Proposition \ref{res_fields_arc_topos}) and the fact that $Z \hookrightarrow G$ is a monomorphism, it follows that the induced map of completed residue fields $Z_z \rightarrow G_y$ is an isomorphism, then the hypothesis that $F_x \overset{\simeq}{\longrightarrow} G_y$ implies that the map $F_x \rightarrow Z_z$ is an isomorphism. This completed the proof of $(1) \simeq (2)$.

It remains to show that $(3) \simeq (1) + (2)$. To prove the implication $(3) \Rightarrow (1)$, first notice that if $\im(|Z|(*) \rightarrow |G|(*)) \not\subset \im(|F|(*) \rightarrow |G|(*))$ then it follows from Proposition \ref{pts_arc_pi_shv} that there is no map $Z \rightarrow F$ making the desired diagram commute. The implication $(3) \Rightarrow (1)$ then follows from the uniqueness of the lift $Z \rightarrow F$. Finally, we prove that $(1) + (2) \Rightarrow (3)$; uniqueness of the lift is clear as the map $F \rightarrow G$ is a monomorphism. Now recall that any object $Z \in \cX_{\arc_\varpi}$ can be presented as $\colim_{i \in \cI} \cM(P_i) \simeq Z$ for a collection of perfectoid Banach $K$-algebras $P_i$ (cf. Propositions \ref{prop_presheaves} and \ref{prop_sheaves}), thus we may assume that $Z$ is of the form $\cM(P)$ for some perfectoid Banach $K$-algebra. Hence, it remains to show that the induced map $F \times_G Z \rightarrow Z$ is an isomorphism. By construction it is clear that $|F| \times_{|G|} |Z|(*) \rightarrow |Z|(*)$ is surjective, and by the surjectivity of $|F \times_G Z|(*) \rightarrow |F| \times_{|G|} |Z|(*)$ (cf. Proposition \ref{berko_funct_stability}) it follows that the induced map $|F \times_G Z|(*) \rightarrow |Z|(*)$ is surjective. Then, as $G$ is quasiseparated and $Z, F$ are qcqs it follows that $F \times_G Z$ is quasicompact, proving that the map $F \times_G Z \rightarrow Z$ is an epimorphism by virtue of Proposition \ref{epi_arc_topos}; and as monomorphisms are stable under basechange it follows that $F \times_G Z \rightarrow Z$ is a monomorphism, and therefore an isomorphism. This completes the proof of the implication $(1) + (2) \Rightarrow (3)$

\end{proof}

\begin{defn}\label{defn_subobjects} Let $\cX$ be a coherent topos, and $X$ and object of $\cX$. We define $\Sub(X)_{\qcqs}$ as the category whose objects are monomorphisms $Y \hookrightarrow X$ from a qcqs object $Y \in \cX$, and morphisms are maps $Y_1 \rightarrow Y_2$ making the following diagram commute
\begin{cd}
	Y_1 \ar[rd, hook] \ar[rr] && Y_2 \ar[dl, hook] \\
	& X
\end{cd}
In particular we see that any map $Y_1 \rightarrow Y_2$ in $\Sub(X)_{\qcqs}$ must be a monomorphism. We call $\Sub(X)_{\qcqs}$ the category of qcqs subobjects of $X$. If $X$ is an $\arc_\varpi$-sheaf we can also consider the full-subcategory $\Sub(X)_{\Aff} \subset \Sub(X)_{\qcqs}$ (resp. $\Sub(X)_{\Perfd}$) spanned by objects of the form $Y = \cM(A)_{\arc_\varpi}$ (resp. $Y = \cM(A)$, where $A$ is a perfectoid Banach $K$-algebra).
\end{defn}

\begin{prop}[qcqs Subobjects]\label{subobjects_arc_topos} Let $F$ be a quasiseparated object of $\cX_{\arc_\varpi}$. Then, the Berkovich functor $|-|: \cX_{\arc_\varpi} \rightarrow \Cond$ induces an equivalence of categories
\begin{align*}
	\Sub(F)_{\qcqs} \overset{\simeq}{\longrightarrow} \Sub(|F|)_{\qcqs} && (Y \hookrightarrow F) \mapsto (|Y| \hookrightarrow |F|)
\end{align*}
Furthermore, by Proposition \ref{berko_funct_stability}(4) we learn that for any pair of morphisms $Y_1 \hookrightarrow F \hookleftarrow Y_2$, there is a canonical isomorphism $|Y_1 \times_F Y_2| \simeq |Y_1| \times_{|F|} |Y_2| = |Y_1| \cap |Y_2|$.
\end{prop}

\begin{proof} First, let us recall some properties of quasiseparated objects that will allow us to reduce to the case where $F$ is qcqs. Since $F$ is quasiseparated, we know that there exists a filtered collection of monomorphisms $\{F_i \hookrightarrow F\}$ from qcqs objects $F_i$ such that $\colim F_i \rightarrow F$ is an isomorphism, thus since the Berkovich functor preserves monomorphisms and colimits it follows that there for any morphism $Q \rightarrow |F|$ from a compact Hausdorff space, there exists a monomorphism $F_i \hookrightarrow F$ such that the map $Q \rightarrow |F|$ admits a unique lift $Q \rightarrow |F_i|$. Similarly, by \cite[Theorem C.6.5]{lurie_ultracategories} we know we can realize $\cX_{\arc_\varpi} = \Shv_{\text{coh}}(\cX_{\arc_\varpi, \qcqs})$, and so for any morphism $Z \rightarrow F$ from a qcqs object $Z$ there exists a monomorphism $F_i \hookrightarrow F$ such that the map $Z \rightarrow F$ admits a unique lift $Z \rightarrow F_i$.

Let us show that for any monomorphism $\tilde{Y} \hookrightarrow |F|$ from a compact hausdorff spaces, there exists a monomorphism $Y \hookrightarrow F$ from a qcqs object $Y$ in $\cX_{\arc_\varpi}$ such that $|Y \hookrightarrow F| = \tilde{Y} \hookrightarrow |F|$. By the work in the previous paragraph we may assume that $F$ itself is qcqs. As $F$ is quasicompact we know that there exists a finite collection of morphism $\{\cM(A_i) \rightarrow F\}_{i \in I}$ such tha $A_i$ is a perfectoid Banach $K$-algebra and the induced map $\sqcup \cM(A_i) \rightarrow F$ is an epimorphism. Then, by Proposition \ref{berko_funct_stability}(1) we know that the induced map $\sqcup |\cM(A_i)| \rightarrow |F|$ is an epimorphism, and so in particular the induced map of sets $\sqcup |\cM(A_i)|(*) \rightarrow |F|(*)$ is surjective; so for each $\tilde{x} \in \tilde{Y}(*) \subset |F|(*)$ we can pick a lift $x \in \sqcup |\cM(A_i)|(*) \rightarrow |F|(*)$ and so by Proposition \ref{pts_arc_pi_shv} we can find a perfectoid non-archimedean field $\cM(L_x)$ together with a morphism $\cM(L_x) \rightarrow \sqcup \cM(A_i)$ such that under the Berkovich functor it gets mapped to $x: \pt \rightarrow \sqcup |\cM(A_i)|$. We have shown that for each element $\tilde{x}_j \in \tilde{Y}(*) \subset |F|(*)$ we can find a morphism $\cM(L_{x_j}) \rightarrow \sqcup \cM(A_i)$ such the induced map $|\cM(L_{x_j})| \rightarrow \sqcup |\cM(A_i)|$ lifts the morphism $\tilde{x}_j: \pt \rightarrow |F|$ along the map $\sqcup |\cM(A_i)| \rightarrow |F|$. We denote this (possibily infinite) collection of morphisms as $\{\cM(L_{x_j}) \rightarrow \sqcup \cM(A_i) \}_{j \in J}$, which is equivalent to providing a collection of morphism of perfectoid Banach $K$-algebras $\{\prod A_i \rightarrow L_{x_j} \}_{j \in J}$, which in turn induce a map $\prod_{i \in I} A_i \rightarrow \prod_{j \in J} L_{x_j}$ -- recall that by Proposition \ref{product_perfectoid} and \ref{equiv_perfd_ban_tic} we know that $\prod_{j \in J} L_{x_j}$ is a perfectoid Banach $K$-algebra. Hence, we get an induced map 
\begin{align*}
	p: \cM(\prod_{j \in J} L_{x_j}) \rightarrow \sqcup_{i \in I} \cM(A_i) \rightarrow F 
\end{align*}
We claim that $\im(|\cM(\prod_{j \in J} L_x)|(*) \rightarrow |F|(*)) = \tilde{Y}(*)$; indeed by virtue of Lemma \ref{stone_cech_berko_sp} we know that $|\cM(\prod_{j \in J} L_{x_j})|$ is homeomorphic to the Stone-Cech compactification of $J$, denoted by $\beta(J)$, it is clear from the construction that $\im(|\cM(\prod_{j \in J} L_x)|(*) \rightarrow |F|(*))$ is equal to the closure of $\tilde{Y}(*)$ in $|F|$ but since $\tilde{Y}(*) \subset |F|(*)$ is a compact hausdorff subset the claim follows.

To conclude the proof, we are back to the general situation where $F$ is only assumed to be quasiseparated. We define the monomorphism $Y \hookrightarrow F$ as $Y = \im(p)$, already showing that $Y$ is a qcqs object of $\cX_{\arc_\varpi}$. By construction we have that $\cM(\prod_{j \in J} L_{x_j}) \twoheadrightarrow Y$ is an epimorphism, which by virtue of the fact that the Berkovich functor preserves epimorphism we learn that the induced map $|\cM(\prod_{j \in J} L_{x_j})|(*) \twoheadrightarrow |Y|(*)$ is surjective, showing that $|Y|(*) = \tilde{Y}(*) \subset |F|(*)$. Then, Proposition \ref{mono_condensed_sets} shows that there is a unique isomorphism $|Y| \simeq \tilde{Y}$ respecting the monomorphisms towards $|F|$. To conclude, we we learn from Proposition \ref{mono_arc_topos} that the $Y$ we constructed is unique up to unique isomorphism, and by a combination of Proposition \ref{mono_arc_topos} and \ref{mono_condensed_sets} the claim about the equivalence of categories follows.
\end{proof}

\subsection{Separated morphisms}

Fix a perfectoid non-archimedean field $K$.

\begin{prop}[Zariski Closed Subsets]\label{zariski_closed_perfd} Let $S$ be a perfectoid Banach $K$-algebra, and $I \subset S$ a closed ideal. Then, there exists a perfectoid Banach $K$-algebra $R$ and a surjection $S \twoheadrightarrow R$ such that the induced map $\cM(R) \rightarrow \cM(S)$ identifies with $\cM(S/I)_{\arc_\varpi} \rightarrow \cM(S)$.
\end{prop}

\begin{proof} Let $S_{\le 1}$ be the object corresponding to $S$ in $\Perfd_{K_{\le 1}}^{\Prism a}$, and $I_{\le 1} := S_{\le 1} \cap I$ the associated closed ideal of $S_{\le 1}$. We know from \cite[Theorem 7.4]{prisms} that the map $(S_{\le 1}/I_{\le 1}) \rightarrow (S_{\le 1}/I_{\le 1})_{\perfd}$ is surjective and that $(S_{\le 1}/I_{\le 1})_{\perfd}$ is the universal integral perfectoid algebra equipped with a map $(S_{\le 1}/I_{\le 1}) \rightarrow (S_{\le 1}/I_{\le 1})_{\perfd}$. Set $R_{\le 1} := (S_{\le 1}/I_{\le 1})_{\perfd}^a$, so $R_{\le 1}$ is the universal object of $\Perfd_{K_{\le 1}}^{\Prism a}$ equipped with a map $(S_{\le 1}/I_{\le 1}) \rightarrow R_{\le 1}$, which in turn implies that $R$ is the universal perfectoid Banach $K$-algebra equipped with a map $(S/I) \rightarrow R$ -- then Example \ref{banach_yoneda_arc_pi} shows that we have an identification $\cM(R) \simeq \cM(S/I)_{\arc_\varpi}$. It remains to show that the resulting map $S \rightarrow R$ of perfectoid Banach $K$-algebras is surjective. Indeed, by construction we have that $R = R_{\le 1}[\frac{1}{\varpi}] = (S_{\le 1}/I_{\le 1})_{\perfd}[\frac{1}{\varpi}]$, so its clear that the induced map $S \rightarrow R$ is surjective, since localization is exact. 
\end{proof}

\begin{defn}\label{defn_separated} Let $Y \rightarrow X$ be a morphism in $\cX_{\arc_\varpi}$. Then,
\begin{enumerate}[(1)]
	\item We say that $Y \rightarrow X$ is affine if for every morphisms $\cM(A)_{\arc_\varpi} \rightarrow X$, where $A$ is a Banach $K$-algebra, the fiber product $Y \times_X \cM(A)_{\arc_\varpi}$ is represented by some Banach $K$-algebra $B$, in other words we have an identification $\cM(B)_{\arc_\varpi} = Y \times_X \cM(A)_{\arc_\varpi}$.
	\item We say that $Y \rightarrow X$ is a closed immersion if it is affine, and for every morphisms $\cM(P) \rightarrow X$, where $P$ is a perfectoid Banach $K$-algebra, the induced morphism $Y \times_X \cM(P) \rightarrow \cM(P)$ is represented by a surjective map $P \twoheadrightarrow R$ of perfectoid Banach $K$-algebras. In other words, there exists a perfectoid Banach $K$-algebra $R$ and an isomorphism $\cM(R) = Y \times_X \cM(P)$ such that the induced map $P \rightarrow R$ is surjective.
	\item We say that $Y \rightarrow X$ is separated if the diagonal map $\Delta: Y \rightarrow Y \times_X Y$ is a closed immersion. In particular, we say that $X$ is separated if the map $X \rightarrow \cM(K)$ is separated.
\end{enumerate}
\end{defn}

\begin{example}\label{affine_morphism_banach} Any morphism $X \rightarrow Y$ in $\cX_{\arc_\varpi}$, where $X = \cM(A)_{\arc_\varpi}$ and $Y = \cM(B)_{\arc_\varpi}$, is an affine morphism. Indeed, recall that the sheafification functor $\PreShv(\Ban_K^{\contr, \op}) \rightarrow \Shv_{\arc_\varpi}(\Ban_K^{\contr, \op})$ admits a calculus of fractions, thus there exists a contractive morphism $B \rightarrow A^{\prime}$ and an $\arc_\varpi$-equivalence $A^{\prime} \rightarrow A$ such that the the induced map $X \rightarrow Y$ is given by $\cM(A^{\prime})_{\arc_\varpi} \rightarrow \cM(B)_{\arc_\varpi}$. Hence, for any pair of morphisms $\cM(A)_{\arc_\varpi} \rightarrow \cM(B)_{\arc_\varpi} \leftarrow \cM(C)_{\arc_\varpi}$ we can find a pair of contractive maps $A^{\prime} \leftarrow B \rightarrow C^{\prime}$ which induce the maps $\cM(A)_{\arc_\varpi} \rightarrow \cM(B)_{\arc_\varpi} \leftarrow \cM(C)_{\arc_\varpi}$ after applying the functor $\cM(-)_{\arc_\varpi}$. The claim that morphisms of the form $\cM(A)_{\arc_\varpi} \rightarrow \cM(B)_{\arc_\varpi}$ are affine follows.
\end{example}

\begin{example}\label{closed_immersions_banach} Let us show that if $A \rightarrow B$ is a surjective contractive map of Banach $K$-algebras then the induced map $\cM(B)_{\arc_\varpi} \rightarrow \cM(A)_{\arc_\varpi}$ is a closed immersions. Identify $B = A/I$ for some closed ideal $I \subset A$, then for any perfectoid Banach $K$-algebra $P$ and a map $A \rightarrow P$, we know from Proposition \ref{zariski_closed_perfd} that the induced map $\cM(P/I)_{\arc_\varpi} \rightarrow \cM(P)$ can be identified with $\cM(R) \rightarrow \cM(P)$, where $R$ is a perfectoid Banach $K$-algebra and where the corresponding map $P \rightarrow R$ surjective. The claim then follows from the fact that the sheafified Yoneda functor $\Yo_{\arc_\varpi}: \Ban_{K}^{\contr, \op} \rightarrow \cX_{\arc_\varpi}$ preserves finite limits. 
\end{example}

\begin{example}\label{separated_banach} If $A$ if a Banach $K$-algebra, we claim that $\cM(A)_{\arc_\varpi}$ is a separated object of $\cX_{\arc_\varpi}$. By Example \ref{closed_immersions_banach} it suffices to show that the multiplication map $A \cotimes_K A \rightarrow A$ is surjective; but this follows from the fact that $A \otimes_K A \rightarrow A$ is surjective, and since $A$ is Banach it follows that there is a factorization as $A \otimes_K A \rightarrow A \cotimes_K A \rightarrow A$, showing that the induced map $A \cotimes_K A \rightarrow A$ is surjective.
\end{example}

\begin{lemma}\label{stab_separated_morphisms} The collection of affine morphisms, closed immersions and separated morphisms in $\cX_{\arc_\varpi}$ are stable under composition and basechange.
\end{lemma}

\begin{proof} It is clear from the definitions that affine morphisms and closed immersions are closed under composition and basechange. To show that separated morphisms are closed under composition, assume that $X \rightarrow Y$ and $Y \rightarrow Z$ are separated morphisms in $\cX_{\arc_\varpi}$, and consider the following cartesian diagram
\begin{cd}
	X \times_Y X \ar[r] \ar[d] & X \times_Z X \ar[d] \\
	Y \ar[r, "\Delta"] & Y \times_{Z} Y
\end{cd}
Since $Y \rightarrow Z$ is separated it follows that $\Delta: Y \rightarrow Y \times_Z Y$ is a closed immersion, and since closed immersions are stable under basechange it follows that $X \times_Y X \rightarrow X \times_Z X$. As the diagonal map $X \rightarrow X \times_Z X$ factors as $X \rightarrow X \times_Y X \rightarrow X \times_Z X$ and closed immersions are stable under composition, it follows that $X \rightarrow Z$ is a separated morphism. In order to show that separated morphisms are closed under basechange, assume that $X \rightarrow Y$ is separated and let $Z \rightarrow Y$ be any morphism in $\cX_{\arc_\varpi}$, we need to show that $X \times_Y Z \rightarrow Z$ is separated, equivalently we need to show that the diagonal map $X \times_Y Z \rightarrow (X \times_Y Z) \times_Z (X \times_Y Z)$ is a closed immersion. Using the identity $(X \times_Y Z) \times_Z (X \times_Y Z) = (X \times_Y X) \times_Y Z$ we learn that the following morphism is a closed immersion $X \times_Y Z \rightarrow (X \times_Y Z) \times_Z (X \times_Y Z)$ as closed immersions are closed under basechange.
\end{proof}

\begin{prop}\label{properties_separated} Let $X \rightarrow Y$ be a morphisms in $\cX_{\arc_\varpi}$. Then,
\begin{enumerate}[(1)]
	\item If $X \rightarrow Y$ is a closed immersion, then $X \rightarrow Y$ is a monomorphisms.
	\item If $X$ and $Y$ are separated, then the morphism $X \rightarrow Y$ is separated.
	\item If $Y$ is separated and $X = \cM(A)_{\arc_\varpi}$ for some Banach $K$-algebra $A$, then $X \rightarrow Y$ is an affine morphism. Furthermore, if we have a pair of morphisms $\cM(P_1) \rightarrow Y \leftarrow \cM(P_2)$, where $P_1, P_2$ are perfectoid Banach $K$-algebras, then the fiber product $\cM(P_1) \times_{Y} \cM(P_2)$ is represented by a perfectoid Banach $K$-algebra. In particular, this shows that if $Y$ is separated, then it is quasiseparated.
	\item If $Y$ is separated and $X \rightarrow Y$ is a monomorphism, then $X$ is separated.
	\item Separated objects of $\cX_{\arc_\varpi}$ are stable under fiber product. 
\end{enumerate}
\end{prop}

\begin{proof} In order to prove (1), recall that in a coherent topos a morphism $F \rightarrow G$ is a monomorphism if and only if it for some epimorphism $H \twoheadrightarrow F$ the basechange $H \times_G F \rightarrow H$ is a monomorphism. Therefore, we fix a collection of morphisms $\{\cM(A_i) \rightarrow Y\}_{i \in I}$ such that the induced map $\sqcup_{i \in I} \cM(A_i) \rightarrow X$ is an epimorphism and where each $A_i$ is a perfectoid Banach $K$-algebra, thus it remains to show that the induced map $X \times_Y (\sqcup_{i \in I} \cM(A_i)) \rightarrow (\sqcup_{i \in I} \cM(A_i))$ is a monomorphism. But this follows from the fact that each map $X \times_Y \cM(A_i) \rightarrow \cM(A_i)$ is a monomorphism as $X \rightarrow Y$ is a closed immersion and the fact that sheaves are stable (among presheaves) under arbitraty coproducts.

For the proof of (2), recall that the following commutative diagram is cartesian
\begin{cd}
	X \times_Y X \ar[r, hook] \ar[d] & X \times X \ar[d] \\
	Y \ar[r, hook] & Y \times Y
\end{cd}
Therefore, since monomorphisms are stable under basechange it follows that $X \times_Y X \rightarrow X \times X$ is a monomorphism. It remains to show that for any morphism $\cM(A) \rightarrow X \times_Y X$, where $A$ is a perfectoid Banach $K$-algebra, the basechange $X \times_{X \times_Y X} \cM(A) \rightarrow \cM(A)$ is represented by a surjective map $A \rightarrow B$ of perfectoid Banach $K$-algebras. From the fact that $X \times_Y X \rightarrow X \times X$ is a monomorphism we learn that the canonical map $(X \times_Y X) \times_{X \times X} \cM(A) \rightarrow \cM(A)$ is an isomorphism, and so the claim about $X \times_{X \times_Y X} \cM(A) \rightarrow \cM(A)$ being represented by a surjective map $A \rightarrow B$ follows from the fact that $X$ is separated.

For (3), let $f: \cM(A)_{\arc_\varpi} \rightarrow Y \leftarrow \cM(B)_{\arc_\varpi}: g$ be a pair of morphism, where $A,B$ are Banach $K$-algebras, we need to show that $\cM(A)_{\arc_\varpi} \times_Y \cM(B)_{\arc_\varpi}$ is represented by a Banach $K$-algebra. Since $Y$ is separated it follows that the basechange of $Y \rightarrow Y \times Y$ along the map $f \times g: \cM(A)_{\arc_\varpi} \times \cM(B)_{\arc_\varpi}$, which is given by $\cM(A)_{\arc_\varpi} \times_Y \cM(B)_{\arc_\varpi} \rightarrow \cM(A)_{\arc_\varpi} \times \cM(B)_{\arc_\varpi}$, is a closed embedding, and so $\cM(A)_{\arc_\varpi} \times_Y \cM(B)_{\arc_\varpi}$ is represented by a Banach $K$-algebras. By the same argument, setting $P_1 = A$ and $P_2 = B$, we learn that the map $\cM(P_1) \times_Y \cM(P_2) \rightarrow \cM(P_1) \times \cM(P_2)$ is a closed immersion, showing that $\cM(P_1) \times_Y \cM(P_2)$ is represented by a perfectoid Banach $K$-algebra. 

In order to prove (4), recall that since $X \rightarrow Y$ is a monomorphism we know that the diagonal map $X \rightarrow X \times_Y X$ is an isomorphism. Then, the result follows from the fact that $Y \rightarrow Y \times Y$ is a closed immersion, closed immersions are stable under basechange, and the fact that the basechange of $Y \rightarrow Y \times Y$ along $X \times X \rightarrow Y \times Y$ is given by $X = X \times_Y X \rightarrow X \times X$. While (5) is a consequence of the fact that separated morphisms are stable under basechange (Lemma \ref{stab_separated_morphisms}) and that morphisms between separated objects are separated.
\end{proof}

\newpage

\section{Analytic Geometry}\label{sect_analytic_spaces}

\subsection{Perfectoid Spaces}

Fix a perfectoid non-archimedean field $K$.

\begin{defn}\label{defn_perfectoid_space} An object $X \in \cX_{\arc_\varpi}$ is called a perfectoid space if there exists a collection of monomorphisms $\{\cM(A_i) \hookrightarrow X\}_{i \in I}$, such that the induced map $\sqcup_{i \in I} \cM(A_i) \rightarrow X$ is an epimorphism and where each $A_i$ is a perfectoid Banach $K$-algebra. We denote the full-subcategory of $\cX_{\arc_\varpi}$ spanned all perfectoid spaces as $\PerfdSpc_{K}$, and we say that a perfectoid space is quasicompact (resp. quasiseparated, separated) if it is so as an object of $\cX_{\arc_\varpi}$ (cf. Definitions \ref{defn_qcqs} and \ref{defn_separated}). Similarly, we say that a morphism of perfectoid spaces $X \rightarrow Y$ is an epimorphism (resp. monomorphism, isomorphism, quasicompact, quasiseparated, separated) if it is so as a morphism of $\cX_{\arc_\varpi}$.

In particular, we say that $X$ is an affinoid perfectoid space if there exists a perfectoid Banach $K$-algebra $A$ such that $X$ is represented by $\cM(A)$.
\end{defn}

\begin{thm}\label{perfd_spc_maps} The category $\PerfdSpc_{K}$ of perfectoid spaces has the following properties
\begin{enumerate}[(1)]
	\item Let $X$ be a perfectoid space, for each $x \in |X|(*)$ there exists a unique monomorphism $\cM(\cH(x)) \hookrightarrow X$ from perfectoid non-archimedean field, which gets mapped to $x: \pt \rightarrow |X|$ under the Berkovich functor. In particular, it satisfies the condition of Proposition \ref{res_fields_arc_topos}(2), making the map $\cM(\cH(x)) \hookrightarrow X$ a completed residue field in $\cX_{\arc_\varpi}$.
	\item Let $Y \rightarrow X$ be a morphism of perfectoid spaces, and assume that $Y$ is quasicompact and $X$ is qcqs. Then, $Y \rightarrow X$ is an epimorphism if and only if the induced map $|Y|(*) \rightarrow |X|(*)$ is a surjective map of sets.
	\item Let $Y \rightarrow X$ be a morphism of perfectoid spaces, and assume that $Y$ and $X$ are qcqs. Then, $Y \rightarrow X$ is an isomorphism if and the induced map $|Y|(*) \rightarrow |X|(*)$ is bijective, and for each $y \in |Y|(*) \simeq |X|(*) \ni x$ the induced map of completed residue fields $\cM(\cH(y)) \rightarrow \cM(\cH(x))$ is an $\arc_\varpi$-equivalence (equivalently, an isomorphism).
\end{enumerate}
\end{thm}

\begin{proof} In order to prove $(1)$ fix a collection of monomorphisms $\{\cM(A_i) \hookrightarrow X\}_{i \in I}$, such that the induced map $\sqcup_{i \in I} \cM(A_i) \rightarrow X$ is an epimorphism and where each $A_i$ is a perfectoid Banach $K$-algebra. Then, the induced map $\sqcup_{i \in I} |\cM(A_i)|(*) \rightarrow |X|(*)$ is surjective, and so for any $x \in |X|(*)$ we can pick a lift $\tilde{x} \in \sqcup_{i \in I} |\cM(A_i)|(*)$ along the map $\sqcup_{i \in I} |\cM(A_i)| \rightarrow |X|$. By construction, there exists some $i \in I$ such that $\tilde{x} \in |\cM(A_i)|(*)$, and let $\cM(\cH(x)) \rightarrow \cM(A_i)$ be the completed residue field of $A_i$ at $\tilde{x}$, which by virtue of Theorem \ref{stalks_perfectoid} we know is a monomorphism and $\cH(x)$ is a non-archimedean perfectoid field. Then, the composition $\cM(\cH(x)) \hookrightarrow \cM(A_i) \hookrightarrow X$ is a monomorphism, and by construction it gets mapped to $x: \pt \rightarrow |X|$ under the Berkovich functor. Uniqueness then follows from Proposition \ref{res_fields_arc_topos}(2).

Statement (2) is a direct consequence of Proposition \ref{epi_arc_topos}, and statement (3) is a combination of part (1) and Proposition \ref{iso_arc_topos}.
\end{proof}

\begin{thm}\label{perfd_spc_mono} Let $Y \rightarrow X$ be a morphism of perfectoid spaces, and assume that $Y$ is qcqs and $X$ is quasiseparated. Then, the following are equivalent
\begin{enumerate}[(1)]
	\item The morphism $Y \hookrightarrow X$ is a monomorphism.
	\item The induced map $|Y|(*) \rightarrow |X|(*)$ is an injective map of sets, and for each $y \in |Y|(*) \simeq \im(|Y|(*) \rightarrow |X|(*)) \ni x$ the induced map $\cM(\cH(y)) \rightarrow \cM(\cH(x))$ of completed residue fields is an $\arc_\varpi$-equivalence (equivalently, an isomorphism).
	\item The morphism $Y \rightarrow X$ is an analytic domain: for any object $Z \in \cX_{\arc_\varpi}$ and any morphism $Z \rightarrow X$ satisfying $\im(|Z|(*) \rightarrow |X|(*)) \subset \im(|Y|(*) \rightarrow |X|(*))$, there exists a unique morphism $Z \rightarrow Y$ making the following diagram commute
	\begin{cd}
		& Z \ar[ld, dashed] \ar[d] \\
		Y \ar[r] & X
	\end{cd}
\end{enumerate}
\end{thm}

\begin{proof} This is a direct consequence of Proposition \ref{mono_arc_topos}, and Theorem \ref{perfd_spc_maps}(1) for the identification of the completed residue fields. However, for the sake of completeness let us include a more direct proof in the case where $Y = \cM(R)$ and $X = \cM(S)$ for a pair of perfectoid Banach $K$-algebras $R,S \in \Perfd_K^{\Ban}$. Recall that as in Proposition \ref{mono_arc_topos} it suffices to check condition (3) for object $Z$ of the form $\cM(T)$ for some perfectoid Banach $K$-algebra $T$. The implication $(2) \Rightarrow (3)$ is the most difficult, and relies on the fact that isomorphism of perfectoid Banach $K$-algebras are the same as $\arc_\varpi$-equivalences (cf. Proposition \ref{iso_arc_topos}).

We begin by showing that $(1) \Rightarrow (2)$. The fact that the induced map $|\cM(R)|(*) \rightarrow |\cM(S)|(*)$ is injective follows from Proposition \ref{mono_berko_sp_injective}, and the isomorphism of completed residue fields is Proposition \ref{mono_iso_on_residue_fields}. Next, we show that $(2) \Rightarrow (3)$. We claim that $\cM(R) \rightarrow \cM(S)$ being injective and defining isomorphisms on completed residue fields imply the following: for every perfectoid non-archimedean field $L/K$ we have an injective map
\begin{equation*}
	\Hom_{\Perfd_K^{\Ban, \op}} (\cM(L), \cM(R)) \hookrightarrow \Hom_{\Perfd_K^{\Ban, \op}} (\cM(L), \cM(S))
\end{equation*}
and the image corresponds to maps $S \rightarrow L$ such that the induced map $\pt = |\cM(L)| \rightarrow |\cM(S)|$ is contained in $\im(|\cM(R)|(*) \rightarrow |\cM(S)|(*))$. In order to proof injectivity, assume we have two morphisms $f_1, f_2: R \rightarrow L$, such that when pre-composed with $g: S \rightarrow R$ we gave $f_1 \circ g = f_2 \circ g$. Injectivity of $|\cM(R)|(*) \rightarrow |\cM(S)|(*)$ shows that both maps $f_1, f_2: \cM(L) \rightarrow \cM(R)$ get send to $x: \pt \rightarrow |\cM(R)|$ unde the Berkovich functor, and by virtue of Theorem \ref{stalks_perfectoid} $f_i$ will factor uniquely as
\begin{equation*}
	f_i = h_i \circ v_R: R \rightarrow \cH(x) \rightarrow L
\end{equation*}
thus it suffices to show that $h_1 = h_2$. From the hypothesis that the map $\cM(R) \rightarrow \cM(S)$ induces an isomorphism on completed residue fields, we get that we have an injection
\begin{equation*}
	\Hom_{\Perfd_K^{\Ban, \op}} (\cM(L), \cM(\cH(x))) \hookrightarrow \Hom_{\Perfd_K^{\Ban, \op}} (\cM(L), \cM(S) )
\end{equation*}
via the map $v_R \circ g = S \rightarrow R \rightarrow \cH(x)$. As the maps $h_i: \cH(x) \rightarrow L$ will become the same when pre-composed with $v_R\circ g$, we learn that $h_1 = h_2$. Furthermore, if $|\cM(L)| \rightarrow |\cM(S)|$ is in the image of $|\cM(R)|(*) \rightarrow |\cM(S)|(*)$, the map $S \rightarrow L$ admits a lift to a map $R \rightarrow L$; indeed, $S \rightarrow L$ will factor as $S \rightarrow \cH(x) \rightarrow L$, and by the hypothesis there is a canonical map $R \rightarrow \cH(x)$, showing the desired claim.

It remains to show that $\cM(R) \rightarrow \cM(S)$ is an analytic domain, that is, it satisfies the universal property stated in part (3). As in Proposition \ref{mono_arc_topos} it suffices to check condition (3) for object $Z$ of the form $\cM(T)$ for some perfectoid Banach $K$-algebra $T$. Let $S \rightarrow T$ be a morphism in $\Perfd_K^{\Ban}$ such that which satisfies the hypothesis of $(3)$, and consider the following pullback diagram
\begin{cd}
	\cM (T \cotimes_S R) \ar[r] \ar[d] & \cM(T) \ar[d] \\
	\cM(R) \ar[r]  & \cM(S)
\end{cd}
For every non-archimedean perfectoid field $L$, by the work done in the previous paragraph we have that $\cM(R)(L) \rightarrow \cM(S)(L)$ is an injection, which in turn implies that $\cM (T \cotimes_S R) (L) \rightarrow \cM(T)(L)$ is an injection. Furthermore, since $\im(\cM(T)(L) \rightarrow \cM(S)(L)) \subset \im(\cM(R)(L) \rightarrow \cM(S)(L))$ from the hypothesis and the work done in the previous paragraph, it follows that $\cM(T \cotimes_S R)(L) \rightarrow \cM(T)(L)$ is a bijection for all non-archimedean perfectoid fields $L$. This implies that $\cM(T \cotimes_S R) \rightarrow \cM(T)$ is an $\arc_\varpi$-equivalence, and therefore an isomorphism by Proposition \ref{iso_arc_topos}.

Finally, we need to show that $(3) \Rightarrow (1)$. Its clear that if $\im(|\cM(T)|(*) \rightarrow |\cM(S)|(*)) \not\subset \im(|\cM(R)|(*) \rightarrow |\cM(S)|(*))$ then there is no morphism $\cM(T) \rightarrow \cM(R)$ making the desired diagram commute. Then, the fact that $\cM(R) \rightarrow \cM(S)$ is a monomorphism in $\Perfd_K^{\Ban, \op}$ follows from the uniqueness of the lift.
\end{proof}

\begin{example} Let $A \in \Perfd_{K}^{\Ban}$ be a perfectoid Banach $K$-algebra, and $X = \cM(A)$ it associated affinoid perfectoid space. Then, the following are examples of analytic domains
\begin{enumerate}[(1)]
	\item For each $x \in |X|(*)$ there exists a perfectoid non-archimedean field $\cH(x)$ (cf. Definition \ref{defn_completed_residue_field})
	together with a morphism $\cM(\cH(x)) \rightarrow X$. By Theorems \ref{stalks_perfectoid} and \ref{perfd_spc_mono} we learn that the induced map $\cM(\cH(x)) \rightarrow X$ is an analytic domain.
	\item Recall from Theorem \ref{struct_presheaf_perfectoid} the definition of $|X|_{\rat}$ and the structure sheaf functor $\cO_{X}(-): |X|_{\rat}^{\op} \rightarrow \Perfd_K^{\Ban}$. Then, since $\cM(\cO_{X}(U)) \rightarrow X$ is a monomorphism for every $U \in |X|_{\rat}$ it follows from Theorem \ref{perfd_spc_mono} that $\cM(\cO_{X}(U)) \rightarrow X$ is an analytic domain.
	\item For any closed ideal $I \in A$, the induced map $\cM(A/I)_{\arc_\varpi} \rightarrow \cM(A)$ is a monomorphism of affinoid perfectoid spaces by virtue of Propositions \ref{zariski_closed_perfd} and \ref{properties_separated}(1).
\end{enumerate}
\end{example}

\begin{prop}\label{perfd_spc_separated} The category $\PerfdSpc_K^{\sep}$ of separated perfectoid spaces, has the following properties
\begin{enumerate}[(1)]
	\item  The category $\PerfdSpc_{K}^{\sep}$ is closed under fiber products; that is, if we have a pair of morphisms $X \rightarrow Y \leftarrow Z$ of separated perfectoid spaces then the fiber product $X \times_Y Z$, computed in the category $\cX_{\arc_\varpi}$, is a separated perfectoid space.
	\item Let $f: X \rightarrow Y$ be any morphism of separated perfectoid spaces, and $\{\cM(A_i) \hookrightarrow Y\}_{i \in I}$ a collection of monomorphisms from affinoid perfectoid spaces such that the induced map $\sqcup_{i \in I} \cM(A_i) \rightarrow Y$ is an epimorphism. Then, there exists a collection of monomorphisms $\{\cM(B_j) \hookrightarrow X\}_{j \in J}$ from affinoid perfectoid spaces such that $\sqcup_{j \in J} \cM(B_j) \rightarrow X$ is an epimorphism, and such that for each $j \in J$ there exists a $i \in I$ satisfying 
	\begin{align*}
		|f|(*) \Big (|\cM(B_j)|(*) \Big) \subset |\cM(A_i)|(*) \subset |Y|(*)
	\end{align*}
\end{enumerate}
\end{prop}

\begin{proof} We begin by proving (1). From Proposition \ref{properties_separated}(5) we know that separated objects of $\cX_{\arc_\varpi}$ are stable under fiber product, thus it remains to show that $X \times_Y Z$ is a perfectoid space. Pick a collection of monomorphisms $\{\cM(A_i) \hookrightarrow X\}_{i \in I}$ from affinoid perfectoid spaces such that the induced map $\sqcup_{i \in I} \cM(A_i) \rightarrow X$ is an epimorphism, then since $\cX_{\arc_\varpi}$ is a topos it follows that $\{\cM(A_i)\times_{Y} Z \hookrightarrow X \times_Y Z\}_{i \in I}$ is a collection of monomorphisms which after taking coproducts it induces an epimorphism. By picking a similar collection of maps $\{\cM(B_j) \hookrightarrow Z\}_{j \in J}$, we learn that there is a collection of monomorphisms $\{\cM(A_i) \times_{Y} \cM(B_j) \hookrightarrow X \times_Y Z\}_{(i,j) \in I \times J}$ induce an epimorphism after taking coproducts, and since $Y$ is separated it follows from Proposition \ref{properties_separated}(3) that $\cM(A_i) \times_{Y} \cM(B_j)$ is an affinoid perfectoid space, completing the proof.

For the proof of (2), we follow the notation of the proposition statement. Since separated perfectoid spaces are closed under fiber products, for each $X \times_{Y} \cM(A_i)$ there exists a collection of monomorphisms $\{\cM(B_j) \hookrightarrow X \times_{Y} \cM(A_i)\}_{j \in J_i}$ inducing an epimorphism after taking coproducts, which in turn implies that the totality of this monomorphisms $\{\cM(B_j) \hookrightarrow X\}_{j \in J}$, where $J = \cup_{i \in I} J_i$, induce an epimorphism $\sqcup_{j \in J} \cM(B_j) \rightarrow X$ as desired. The claim that $|f|(*) \Big (|\cM(B_j)|(*) \Big) \subset |\cM(A_i)|(*)$ follows from the construction and Proposition \ref{berko_funct_stability}(4).
\end{proof}

\begin{rem} Let us use the technology developed so far to make contact with Berkovich's theory of $K$-analytic spaces as developed in \cite[Section 1]{berkovichetale}. Let $X$ be a separated perfectoid space, then we learn from Proposition \ref{perfd_spc_separated} that the category $\Sub(X)_{\Perfd}$ (cf. Definition \ref{defn_subobjects}) is closed under fiber products. Thus, it satisfies most of the conditions for a net of compact Hausdorff spaces on a topological space, in the sense of Berkovich\footnote{With the exception that for each $x \in |X|(*)$ there need not exists a finite collection of objects $\{Y_i \hookrightarrow X\}$ in $\Sub(X)_{\Perfd}$ whose union contains a neighborhood of $x$. Later we will introduce the notion of locally compact perfectoid space, which is meant to fill in this gap.}.
	
Furthermore, by Theorem \ref{perfd_spc_mono} we learn that if $|\cM(A_i)|(*) \subset |\cM(A_j)|(*)$ are objects of $\Sub(X)_{\Perfd}$ then there exist a unique monomorphism $\cM(A_i) \rightarrow \cM(A_j)$ which commutes with the map towards $X$, giving rise to something analogous to atlas in the sense of \cite[Definition 1.2.3]{berkovichetale}. We claim that the functor $\Sub(X)_{\Perfd} \rightarrow \cX_{\arc_\varpi}$ defined by $(\cM(A) \hookrightarrow X) \mapsto \cM(A)$ satisfies $\colim_{\Sub(X)_{\Perfd}} \cM(A) = X$. Indeed, let $\Delta^{\op} \rightarrow \cX_{\arc_\varpi}$ be the Cech nerve of the maps $\sqcup_{\Sub(X)_{\Perfd}} \cM(A) \rightarrow X$, by construction we have that $\colim_{\Delta^{\op}} (\sqcup_{\Sub(X)_{\Perfd}} \cM(A_i)^{\bullet/X}) \simeq X$, and since the natural morphism $\Delta^{\op} \rightarrow \cT$ is cofinal, the claim follows.

Finally, let us argue that any morphism $f: X \rightarrow Y$ of separated perfectoid spaces comes from a ``strong morphism'' in the sense of Berkovich \cite[Definition 1.2.7]{berkovichetale}. Indeed, by Proposition \ref{perfd_spc_separated}(2) we know that for each morphism $\cM(A) \hookrightarrow X$ in $\Sub(X)_{\Perfd}$ there exists an $\arc_\varpi$-cover made up of a finite collection of monomorphisms $\{\cM(A_i) \hookrightarrow \cM(A)\}$ such that for each $\cM(A_i)$ there exists a monomorphism $\cM(B_i) \hookrightarrow Y$ in $\Sub(Y)_{\Perfd}$ such that $|f|(*) \Big ( |\cM(A_i)|(*) \Big) \subset |\cM(B_i)|(*)$. By virtue of Theorem \ref{perfd_spc_mono} we know that if $|f|(*) \Big ( |\cM(A_i)|(*) \Big) \subset |\cM(B_i)|(*)$, then there exists an essentially unique morphism $\cM(A_i) \rightarrow \cM(B_i)$ making the following diagram commute
\begin{cd}
	\cM(A_i) \ar[r] \ar[d, hook] & \cM(B_i) \ar[d, hook] \\
	X \ar[r] & Y
\end{cd}
\end{rem}

\subsection{\texorpdfstring{$\arc_\varpi$}{arc pi}-Analytic Spaces}

Fix a perfectoid non-archimedean field $K$.

\begin{defn}\label{defn_analytic_space} An object $X \in \cX_{\arc_\varpi}$ is called an $\arc_\varpi$-analytic space if there exists a collection of monomorphisms $\{\cM(A_i)_{\arc_\varpi} \hookrightarrow X\}_{i \in I}$, such that the induced map $\sqcup_{i \in I} \cM(A_i)_{\arc_\varpi} \rightarrow X$ is an epimorphism and where each $A_i$ is a Banach $K$-algebra. We denote the full-subcategory of $\cX_{\arc_\varpi}$ spanned all $\arc_\varpi$-analytic spaces as $\AnSpc_K$, and we say that a $\arc_\varpi$-analytic space is quasicompact (resp. quasiseparated, separated) if it is so as an object of $\cX_{\arc_\varpi}$ (cf. Definition \ref{defn_qcqs} and \ref{defn_separated}). Similarly, we say that a morphism of $\arc_\varpi$-analytic spaces $X \rightarrow Y$ is an epimorphism (resp. monomorphism, isomorphism, quasicompact, quasiseparated, separated) if it is so as a morphism of $\cX_{\arc_\varpi}$.

In particular, we say that $X$ is an affinoid $\arc_\varpi$-analytic space if there exists a Banach $K$-algebra $A$ such that $X$ is represented by $\cM(A)_{\arc_\varpi}$.
\end{defn}

\begin{thm}\label{analytic_spc_maps} The category $\AnSpc_{K}$ of $\arc_\varpi$-analytic spaces has the following properties
\begin{enumerate}[(1)]
	\item Let $X$ be a $\arc_\varpi$-analytic space, for each $x \in |X|(*)$ there exists a unique monomorphism $\cM(\cH(x))_{\arc_\varpi} \hookrightarrow X$ where $\cH(x)$ is a non-archimedean field, which gets mapped to $x: \pt \rightarrow |X|$ under the Berkovich functor. In particular, it satisfies the condition of Proposition \ref{res_fields_arc_topos}(2), making the map $\cM(\cH(x))_{\arc_\varpi} \hookrightarrow X$ a completed residue field in $\cX_{\arc_\varpi}$.
	\item Let $Y \rightarrow X$ be a morphism of $\arc_\varpi$-analytic spaces, and assume that $Y$ is quasicompact and $X$ is qcqs. Then, $Y \rightarrow X$ is an epimorphism if and only if the induced map $|Y|(*) \rightarrow |X|(*)$ is a surjective map of sets.
	\item Let $Y \rightarrow X$ be a morphism of $\arc_\varpi$-analytic spaces, and assume that $Y$ and $X$ are qcqs. Then, $Y \rightarrow X$ is an isomorphism if and the induced map $|Y|(*) \rightarrow |X|(*)$ is bijective, and for each $y \in |Y|(*) \simeq |X|(*) \ni x$ the induced map of completed residue fields $\cM(\cH(y))_{\arc_\varpi} \rightarrow \cM(\cH(x))_{\arc_\varpi}$ is an $\arc_\varpi$-equivalence (equivalently, an isomorphism).
\end{enumerate}
\end{thm}

\begin{proof} In order to prove (1) fix a collection of monomorphism $\{\cM(A_i)_{\arc_\varpi} \hookrightarrow X \}_{i \in I}$, sucht hat the induced map $\sqcup_{i \in I} \cM(A_i)_{\arc_\varpi} \rightarrow X$ is an epimorphism, and where each $A_i$ is a Banach $K$-algebra. Then, the induced map $\sqcup |\cM(A_i)_{\arc_\varpi}|(*) \rightarrow |X|(*)$ is surjective, and let $\tilde{x} \in \sqcup |\cM(A_i)|(*)$ be a lift of $x \in |X|(*)$ along the map $\sqcup |\cM(A_i)_{\arc_\varpi}|(*) \rightarrow |X|(*)$. By construction, there exists some $i \in I$ such that $\tilde{x} \in |\cM(A_i)_{\arc_\varpi}|(*)$, and since $|\cM(A_i)| = |\cM(A_i)_{\arc_\varpi}$ we can find a map $\cM(\cH(x)) \rightarrow \cM(A_i)$ from the completed residue field $\cH(x)$ of $A_i$ at $\tilde{x}$ (cf. Definition \ref{defn_completed_residue_field}), in the category $\Ban_K^{\op}$. Then, by a combination of Example \ref{banach_yoneda_arc_pi} and Proposition \ref{properties_stalk} we learn that the induced map $\cM(\cH(x))_{\arc_\varpi} \rightarrow \cM(A_i)_{\arc_\varpi}$ is a monomorphism. This in turn implies that the composition $\cM(\cH(x))_{\arc_\varpi} \rightarrow X$ is a monomorphism which gets mapped to $x: \pt \rightarrow X$ under the Berkovich functor. Uniqueness follows from Proposition \ref{res_fields_arc_topos}(2).

Statement (2) is a direct consequence of Proposition \ref{epi_arc_topos}, and statement (3) is a combination of part (1) and Proposition \ref{iso_arc_topos}.
\end{proof}

\begin{thm}\label{analytic_spc_mono} Let $Y \rightarrow X$ be a morphism of $\arc_\varpi$-analytic spaces, and assume that $Y$ is qcqs and $X$ is quasiseparated. Then, the following are equivalent
\begin{enumerate}[(1)]
	\item The morphism $Y \hookrightarrow X$ is a monomorphism.
	\item The induced map $|Y|(*) \rightarrow |X|(*)$ is an injective map of sets, and for each $y \in |Y|(*) \simeq \im(|Y|(*) \rightarrow |X|(*)) \ni x$ the induced map $\cM(\cH(y))_{\arc_\varpi} \rightarrow \cM(\cH(x))_{\arc_\varpi}$ of completed residue fields is an $\arc_\varpi$-equivalence (equivalently, an isomorphism).
	\item The morphism $Y \rightarrow X$ is an analytic domain: for any object $Z \in \cX_{\arc_\varpi}$ and any morphism $Z \rightarrow X$ satisfying $\im(|Z|(*) \rightarrow |X|(*)) \subset \im(|Y|(*) \rightarrow |X|(*))$, there exists a unique morphism $Z \rightarrow Y$ making the following diagram commute
	\begin{cd}
		& Z \ar[ld, dashed] \ar[d] \\
		Y \ar[r] & X
	\end{cd}
\end{enumerate}
\end{thm}

\begin{proof} This is a direct consequence of Proposition \ref{mono_arc_topos}, and Theorem \ref{analytic_spc_maps}(1) for the identification of completed residue fields.
\end{proof}

\begin{example} Let $A$ be a Banach $K$-algebra, and $X = \cM(A)_{\arc_\varpi}$ its associated affinoid $\arc_\varpi$-analytic space. Then, the following are examples of analytic domains
\begin{enumerate}[(1)]
	\item For each $x \in |X|(*)$ there exists a non-archimedean field $\cH(x)$ (cf. Definition \ref{defn_completed_residue_field})
	together with a morphism $\cM(\cH(x))_{\arc_\varpi} \rightarrow X$. By Theorems \ref{analytic_spc_maps} and \ref{analytic_spc_mono} we learn that the induced map $\cM(\cH(x))_{\arc_\varpi} \rightarrow X$ is an analytic domain.
	\item For any subset $\cM(A) \supset V := \{ x \in \cM(A) \text{ such that } |f_i(x)| \le |g(x)| \}$, where $\{g, f_1, \dots, f_n\} \subset A$ generate the unit ideal, there exists a Banach $K$-algebra $B$ and a monomorphism in $\cM(B) \rightarrow \cM(A)$ in $\Ban_K^{\op, \contr}$ with image exactly $V \subset \cM(A)$ (cf. Proposition \ref{rat_domains_are_mono} and Corollary \ref{topo_img_rat_domain}). Then, since the functor $\Yo_{\arc_\varpi}: \Ban_K^{\contr, \op} \rightarrow \cX_{\arc_\varpi}$ preserves monomorphisms and associated compact hausdorff space it follows that $\cM(B)_{\arc_\varpi} \rightarrow \cM(A)_{\arc_\varpi}$ is a monomorphism in $\cX_{\arc_\varpi}$ with image exactly $V \subset \cM(A)_{\arc_\varpi} = \cM(A)$.
	\item For any closed ideal $I \in A$, the induced map $\cM(A/I)_{\arc_\varpi} \rightarrow \cM(A)_{\arc_\varpi}$ is a monomorphism of affinoid $\arc_\varpi$-analytic spaces by virtue of Example \ref{closed_immersions_banach} and Proposition \ref{properties_separated}(1).
\end{enumerate}
\end{example}

\begin{prop}\label{an_spc_separated} The category $\AnSpc_K^{\sep}$ of separated $\arc_\varpi$-analytic spaces, has the following properties
\begin{enumerate}[(1)]
	\item  The category $\AnSpc_K^{\sep}$ is closed under fiber products; that is, if we have a pair of morphisms $X \rightarrow Y \leftarrow Z$ of separated $\arc_\varpi$-analytic spaces then the fiber product $X \times_Y Z$, computed in the category $\cX_{\arc_\varpi}$, is a separated $\arc_\varpi$-analytic space.
	\item Let $f: X \rightarrow Y$ be any morphism of separated $\arc_\varpi$-analytic spaces, and $\{\cM(A_i)_{\arc_\varpi} \hookrightarrow Y\}_{i \in I}$ a collection of monomorphisms from affinoid $\arc_\varpi$-analytic spaces such that the induced map $\sqcup_{i \in I} \cM(A_i)_{\arc_\varpi} \rightarrow Y$ is an epimorphism. Then, there exists a collection of monomorphisms $\{\cM(B_j)_{\arc_\varpi} \hookrightarrow X\}_{j \in J}$ from affinoid $\arc_\varpi$-analytic spaces such that $\sqcup_{j \in J} \cM(B_j)_{\arc_\varpi} \rightarrow X$ is an epimorphism, and such that for each $j \in J$ there exists a $i \in I$ satisfying 
	\begin{align*}
		|f|(*) \Big (|\cM(B_j)_{\arc_\varpi}|(*) \Big) \subset |\cM(A_i)_{\arc_\varpi}|(*) \subset |Y|(*)
	\end{align*}
\end{enumerate}
\end{prop}

\begin{proof} We begin by proving (1). From Proposition \ref{properties_separated}(5) we know that separated objects of $\cX_{\arc_\varpi}$ are stable under fiber product, thus it remains to show that $X \times_Y Z$ is an $\arc_\varpi$-analytic space. Pick a collection of monomorphisms $\{ \cM(A_i)_{\arc_\varpi} \hookrightarrow X\}_{i \in I}$ from affinoid $\arc_\varpi$-analytic spaces such that their coproduct induces an epimorphism, then since $\cX_{\arc_\varpi}$ is a topos it follows that $\{\cM(A)_{\arc_\varpi} \times_Y Z \hookrightarrow X \times_Y Z\}_{i \in I}$ is a collection of monomorphism which induce an epimorphism after taking coproducts. By picking a similar collection of monomorphisms $\{\cM(B_j)_{\arc_\varpi} \hookrightarrow Z\}_{j \in J}$ learn that there is a collection of monomorphisms $\{\cM(A_i)_{\arc_\varpi} \times_Y \cM(B_j)_{\arc_\varpi} \hookrightarrow X \times_Y Z\}_{(i,j) \in I \times J}$ which induce an epimorphism after taking coproducts, and since $Y$ is separated it follows that $\cM(A_i)_{\arc_\varpi} \times_Y \cM(B_j)_{\arc_\varpi}$ is an affinoid $\arc_\varpi$-analytic space, completing the proof of (1).
	
For the proof of (2) we follow the notation from the proposition statement. Since separated $\arc_\varpi$-analytic spaces are closed under fiber products, for each $X \times_Y \cM(A_i)_{\arc_\varpi}$ there exists a collection of monomorphisms $\{\cM(B_j)_{\arc_\varpi} \hookrightarrow X \times_Y \cM(A_i)_{\arc_\varpi} \}_{j \in J_i}$, inducing an epimorphism after taking coproducts, which in turn implies that the totality of this monomorphisms $\{\cM(B_j) \hookrightarrow X\}_{j \in J}$, where $J = \cup_{i \in I} J_i$, induce an epimorphism $\sqcup_{j \in J} \cM(B_j)_{\arc_\varpi} \rightarrow X$ as desired. The claim that $|f|(*) \Big (|\cM(B_j)_{\arc_\varpi}|(*) \Big) \subset |\cM(A_i)_{\arc_\varpi}|(*)$ follows from the construction and Proposition \ref{berko_funct_stability}(4).
\end{proof}

\begin{rem}\label{an_spc_atlas_strong_morphism} Let us use the technology developed so far to make contact with Berkovich's theory of $K$-analytic spaces as developed in \cite[Section 1]{berkovichetale}. Let $X$ be a separated $\arc_\varpi$-analytic space, then we learn from Proposition \ref{an_spc_separated} that the category $\Sub(X)_{\Aff}$ (cf. Definition \ref{defn_subobjects}) is closed under fiber products. Thus, it satisfies most of the conditions for a net of compact Hausdorff spaces on a topological space, in the sense of Berkovich\footnote{With the exception that for each $x \in |X|(*)$ there need not exists a finite collection of objects $\{Y_i \hookrightarrow X\}$ in $\Sub(X)_{\Aff}$ whose union contains a neighborhood of $x$. Later we will introduce the notion of locally compact perfectoid space, which is meant to fill in this gap.}.
	
Furthermore, by Theorem \ref{analytic_spc_mono} we learn that if $|\cM(A_i)_{\arc_\varpi}|(*) \subset |\cM(A_j)_{\arc_\varpi}|(*)$ are objects of $\Sub(X)_{\Aff}$ then there exist a unique monomorphism $\cM(A_i)_{\arc_\varpi} \rightarrow \cM(A_j)_{\arc_\varpi}$ which commutes with the map towards $X$, giving rise to something analogous to atlas in the sense of \cite[Definition 1.2.3]{berkovichetale}. We claim that the functor $\Sub(X)_{\Aff} \rightarrow \cX_{\arc_\varpi}$ defined by $(\cM(A)_{\arc_\varpi} \hookrightarrow X) \mapsto \cM(A)_{\arc_\varpi}$ satisfies $\colim_{\Sub(X)_{\Aff}} \cM(A)_{\arc_\varpi} = X$. Indeed, let $\Delta^{\op} \rightarrow \cX_{\arc_\varpi}$ be the Cech nerve of the maps $\sqcup_{\Sub(X)_{\Aff}} \cM(A)_{\arc_\varpi} \rightarrow X$, by construction we have that $\colim_{\Delta^{\op}} (\sqcup_{\Sub(X)_{\Aff}} \cM(A_i)_{\arc_\varpi}^{\bullet/X}) \simeq X$, and since the natural morphism $\Delta^{\op} \rightarrow \cT$ is cofinal, the claim follows.

Finally, let us argue that any morphism $f: X \rightarrow Y$ of separated $\arc_\varpi$-analytic spaces comes from a ``strong morphism'' in the sense of Berkovich \cite[Definition 1.2.7]{berkovichetale}. Indeed, by Proposition \ref{an_spc_separated}(2) we know that for each morphism $\cM(A)_{\arc_\varpi} \hookrightarrow X$ in $\Sub(X)_{\Aff}$ there exists an $\arc_\varpi$-cover made up of a finite collection of monomorphisms $\{\cM(A_i)_{\arc_\varpi} \hookrightarrow \cM(A)_{\arc_\varpi}\}$ such that for each $\cM(A_i)_{\arc_\varpi}$ there exists a monomorphism $\cM(B_i)_{\arc_\varpi} \hookrightarrow Y$ in $\Sub(Y)_{\Aff}$ such that $|f|(*) \Big ( |\cM(A_i)_{\arc_\varpi}|(*) \Big) \subset |\cM(B_i)_{\arc_\varpi}|(*)$. By virtue of Theorem \ref{analytic_spc_mono} we know that if $|f|(*) \Big ( |\cM(A_i)_{\arc_\varpi}|(*) \Big) \subset |\cM(B_i)_{\arc_\varpi}|(*)$, then there exists an essentially unique morphism $\cM(A_i)_{\arc_\varpi} \rightarrow \cM(B_i)_{\arc_\varpi}$ making the following diagram commute
\begin{cd}
	\cM(A_i)_{\arc_\varpi} \ar[r] \ar[d, hook] & \cM(B_i)_{\arc_\varpi} \ar[d, hook] \\
	X \ar[r] & Y
\end{cd}
\end{rem}

\begin{const}[The Generic Fiber Functor]\label{const_generic_fiber_funct} Recall from Construction \ref{const_arc_pi_topos} that the $\arc_\varpi$-topos is defined as $\cX_{\arc_\varpi} := \Shv_{\arc_\varpi}(\Sch_{K_{\le 1}, \qcqs})$, thus we immediately get a sheafified Yoneda functor
\begin{align*}
	\Yo_{\arc_\varpi}: \Sch_{K_{\le 1}, \qcqs} \longrightarrow \cX_{\arc_\varpi} && X \mapsto X_{\eta, \arc_\varpi}
\end{align*}
which we call the generic fiber functor. In order to justify its name, let us compute provide an explicit description of $X_{\eta, \arc_\varpi}$. In the case where $X = \Spec(A)$ for some $K_{\le 1}$-algebra $A$, from Proposition \ref{iso_arc_topos} we know that $\Yo_{\arc_\varpi}$ sends $\arc_\varpi$-equivalences to isomorphism, and therefore the canonical map $\Spec(A^{\wedge a \tf}) \rightarrow \Spec(A)$ becomes an isomorphism under the generic fiber functor. Furthermore, by realizing $\cX_{\arc_\varpi}$ as $\Shv_{\arc_\varpi}(\Ban_K^{\contr})$ we obtain a canonical isomorphism
\begin{align*}
	\cM \Big(A^{\wedge a \tf} \Big[\frac{1}{\varpi} \Big] \Big)_{\arc_\varpi} \simeq \Spec(A^{\wedge a \tf})_{\eta, \arc_\varpi} \overset{\simeq}{\longrightarrow} \Spec(A)_{\eta, \arc_\varpi}
\end{align*}
providing an explicit description of $\Spec(A)_{\eta, \arc_\varpi}$ and its associated compact Hausdorff space $|\Spec(A)_{\eta, \arc_\varpi}|$. In the case where $X$ is a general object of $\Sch_{K_{\le 1}, \qcqs}$, recall that we can always present $X$ as a coequalizer diagram of affine schemes, where all transition maps are open immersions
\begin{align*}
	\coeq \Big(\sqcup_{j \in I} \Spec(B_j) \rightrightarrows \sqcup_{i \in I} \Spec(A_i) \Big) \overset{\simeq}{\longrightarrow} X
\end{align*}
where the index sets $I$ and $J$ are both finite. Then, since faithfully flat maps are $\arc_\varpi$-covers (Lemma \ref{faithfully_flat_implies_arc_cover}) it follows that $\Yo_{\arc_\varpi}$ preserves the above coequalizer giving us an isomorphism
\begin{align*}
	\coeq \Big(\sqcup_{j \in I} \Spec(B_j)_{\eta, \arc_\varpi} \rightrightarrows \sqcup_{i \in I} \Spec(A_i)_{\eta, \arc_\varpi} \Big) \overset{\simeq}{\longrightarrow} X_{\eta, \arc_\varpi}
\end{align*}
in the category $\cX_{\arc_\varpi}$.
\end{const}

\begin{prop}\label{properties_gen_fib_funct} The generic fiber functor $\Yo_{\arc_\varpi}: \Sch_{K_{\le 1}, \qcqs} \rightarrow \cX_{\arc_\varpi}$ satisfies the following properties
\begin{enumerate}[(1)]
	\item If $X$ is qcqs $K_{\le 1}$-scheme, then $X_{\eta, \arc_\varpi}$ is a qcqs $\arc_\varpi$-analytic space.
	\item If $Y \rightarrow X$ is a closed immersion of qcqs $K_{\le 1}$-schemes, then $Y_{\eta, \arc_\varpi} \rightarrow X_{\eta, \arc_\varpi}$ is a closed immersion in $\cX_{\arc_\varpi}$ (cf. Definition \ref{defn_separated}).
	\item If $X$ is a quasicompact separated $K_{\le 1}$-scheme, then $X_{\eta, \arc_\varpi}$ is a quasicompact separated $\arc_\varpi$-analytic space.
\end{enumerate} 
\end{prop}

\begin{proof} First we prove (1). Since the $\arc_\varpi$-topology is a finitary Grothendieck topology on $\Sch_{K_{\le 1}, \qcqs}$ it follows from Proposition \ref{finitary_implies_coh} that $X_{\eta, \arc_\varpi}$ is a qcqs object of $\cX_{\arc_\varpi}$. To show that $X_{\eta, \arc_\varpi}$ is an $\arc_\varpi$-analytic space, recall that $\Yo_{\arc_\varpi}$ preserves finite limits, monomorphisms, and sends $\arc_\varpi$-covers to epimorphisms, so since any open immersion immersion of schemes is a monomorphism, we conclude that if $\{\Spec(A_i) \rightarrow X\}_{i \in I}$ is a finite Zariski cover of $X$, then the induced maps $\{\Spec(A_i)_{\eta, \arc_\varpi} \rightarrow X_{\eta, \arc_\varpi}\}_{i \in I}$ are monomorphisms and the induced map $\sqcup_{i \in I} \Spec(A_i)_{\eta, \arc_\varpi} \rightarrow X_{\eta, \arc_\varpi}$ is an epimorphism in $\cX_{\arc_\varpi}$. The claim then follows from the identity $\Spec(A_i)_{\eta, \arc_\varpi} \simeq \cM(A^{\wedge a \tf}[\frac{1}{\varpi}])$.

For the prove of (2), recall that a morphism $Y \rightarrow X$ is a closed immersion if and only if for any map $\Spec(A) \rightarrow X$, the induced map $\Spec(B) = Y \times_X \Spec(A) \rightarrow \Spec(A)$ is induced from a surjective map $A \rightarrow B$ of rings. Thus, since $\Yo_{\arc_\varpi}$ preserves finite limits it follows that $Y_{\eta, \arc_\varpi} \rightarrow X_{\eta, \arc_\varpi}$ is affine, and the fact it is a closed immersion then follows from the fact that if $A \rightarrow B$ is surjective, then so is $A^{\wedge}[\frac{1}{\varpi}] \rightarrow B^{\wedge}[\frac{1}{\varpi}]$.

In order to prove (3), we already know from part (1) that $X_{\eta, \arc_\varpi}$ will be a qcqs $\arc_\varpi$-analytic space. Since $X$ is separated we know that the diagonal map $\Delta: X \rightarrow X \times X$ is a closed immersion of qcqs schemes, and so by part (2) and the fact that $\Yo_{\arc_\varpi}$ preserves finite limits it follows that $\Delta: X_{\eta, \arc_\varpi} \rightarrow X_{\eta, \arc_\varpi} \times X_{\eta, \arc_\varpi}$ is a closed immersion, completing the proof.
\end{proof}

\subsection{Open condensed subsets}

\begin{defn} A topological space $X$ is said to be compactly generated (cg) if a subset $U \subset X$ is open if and only if for every continuous map $f: K \rightarrow X$, from a compact Hausdorff space, the inverse image $f^{-1}(U) \subset K$ is open in $K$. We denote the full subcategory of $\Top$ spanned by all compactly generated spaces by $\Top_{\cg}$.

Moreover, we say that $X$ is a compactly generated weak hausdorff space (cgwh) if $X$ is compactly generated and for each continuous map $f: K \rightarrow X$ from a compact hausdorff space, the image $f(K) \subset X$ is a compact hausdorff space under the subspace topology. We denote the category the full subcategory of $\Top$ spanned by all compactly generated weak hausdorff spaces by $\Top_{\cgwh}$. 
	
We refer the reader to \cite{strickland-cgwh} for background on these categories.
\end{defn}

\begin{example} The following are examples of compactly generated weak hausdorff spaces
\begin{enumerate}[(1)]
	\item Compact Hausdorff spaces
	\item Open subsets of compact hausdorff spaces
	\item Locally compact Hausdorff spaces
\end{enumerate}
\end{example}

\begin{const} Define a functor from the category $\Top_{\cgwh}$ to condensed sets as
\begin{align*}
	\underline{(-)}: \Top_{\cgwh} \rightarrow \Cond && X \mapsto \underline{X} := \Maps_{\Top}(-, X)|_{\Comp}: \Comp^{\op} \rightarrow \Set
\end{align*}
It is a consequence of \cite[Proposition 1.7 and Proposition 2.15]{condensedlectures} that this functor is fully faithful, and furthermore we learn from \cite[Theorem 2.16]{condensedlectures} that the essential image of this functor is contained in the subcategory of quasiseparated condensed sets. Its clear from the construction that the functor $\underline{(-)}$ preserves all limits, and in particular fiber products.

Recall in our definition of $\Comp$ (and therefore in $\Cond := \Shv_{\eff} (\Cond)$) there is an implicit uncountable strong limit cardinal $\kappa$ bounding the cardinality of the objects at hand ($< \kappa$), in the above construction this cardinality bound extends to $\Top_{\cgwh}$ -- where we only consider topological spaces with cardinality $< \kappa$.
\end{const}

\begin{defn}\label{defn_open_immersion_cond} Let $X$ be a qcqs condensed set (equivalently, a compact Hausdorff space), then we say that $f: U \rightarrow X$ is an open immersion of there exists an open subset $g: V \hookrightarrow X$ (as topological spaces) such that $\underline{g} = f: U = \underline{V} \rightarrow X$. In particular, since $\underline{(-)}$ preserves all limits, if we have a morphism $K \rightarrow X$ of qcqs condensed sets then the fiber product $U \times_X K \rightarrow K$ will be an open immersion.
	
More generally, we say that a morphism $Y \rightarrow X$ of condensed sets is an open immersion if for every morphism $K \rightarrow X$ from a qcqs condensed set the fiber product $Y \times_X K \rightarrow K$ is an open immersion.
\end{defn}

\begin{rem}\label{subspace_top_condensed} In what follows we will often want to check whether a given a monomorphism $U \hookrightarrow X$, where $X$ is a compact hausdorff space and $U(*) \subset X(*)$ is an open subset of $X$, is in fact an open immersion. The condition that there exists an open subset $V \rightarrow X$ of topological spaces such that $U = \underline{V} \rightarrow X$ can be rephrased as follows: for every morphism $K \rightarrow X$ of compact hausdorff spaces such that $\im(K(*) \rightarrow X(*)) \subset U(*)$ there exists a unique lift $K \rightarrow U$ making the following diagram commute
\begin{cd}
	& K \ar[d] \ar[ld, dashed] \\
	U \ar[r, hook] & X
\end{cd}
\end{rem}

\begin{prop}\label{prop_open_immersion_condensed} Open immersions of condensed sets have the following properties
\begin{enumerate}[(1)]
	\item If $f: U \rightarrow X$ is an open immersion of condensed sets, then it is a monomorphism.
	\item If $f: U \hookrightarrow X$ is an open immersion of condensed sets, and $Z \rightarrow X$ is any morphism such that $\im(Z(*) \rightarrow X(*)) \subset U(*)$, then there exists a unique morphism $Z \rightarrow U$ making the following diagram commute
	\begin{cd}
		& Z \ar[d] \ar[ld, dashed] \\
		U \ar[r, hook] & X
	\end{cd}
	In particular, this implies that $U = \colim_{K \rightarrow X} K \rightarrow X$, where the colimit ranges over all morphisms $K \rightarrow X$, from a compact Hausdorff space $K$, such that $\im(K(*) \rightarrow X(*)) \subset U(*)$.
	\item If $f_1: U_1 \rightarrow X$ and $f_2: U_2 \rightarrow X$ are open immersions, then $\im(f_1 \sqcup f_2) \rightarrow X$ is an open immersion of condensed sets. In particular, the collection of open immersions on $X$ is filtered, that is for any pair of open immersions $f_1: U_1 \rightarrow X$ and $f_2: U_2 \rightarrow X$ there exists an open immersion $g: V \rightarrow X$ such that $f_1$ and $f_2$ factor over $g$.
	\item Let $f: Z \rightarrow Y$ and $g: Y \rightarrow X$ be a pair of monomorphisms of condensed sets. If $g \circ f$ is an open immersion then so is $f$.
	\item Let $\{f_i: U_i \rightarrow X\}$ be a filtered collection of open immersions on $X$, that is for any pair of open immersions $f_i: U_i \rightarrow X$ and $f_j: U_j \rightarrow X$ there exists an $f_k: U_k \rightarrow X$ such that $f_i$ and $f_j$ factor over $f_k$. Then, $\colim U_i \rightarrow X$ is an open immersion.
	\item Open immersions are closed under composition and base-change.
\end{enumerate}
\end{prop}

\begin{proof} \textit{proof of (1):} Recall that since $\Cond$ is a coherent topos we can check whether $f: U \rightarrow X$ is a monomorphisms after basechange on $X$ by an epimorphism. Pick a collection of maps $\{K_i \rightarrow X\}$ from compact Hausdorff spaces $K_i$ such that the induced map $\sqcup K_i \rightarrow X$ is an epimorphism, then by definition we get that each induced map $U \times_X K_i \rightarrow K_i$ is a monomorphism, which in turn implies that $\sqcup K_i \times_X U \rightarrow \sqcup K_i$ is a monomorphism, proving the claim.

\textit{proof of (2):} The uniqueness of the lift is automatic since $U \rightarrow X$ is a monomorphism by part (1), and since every condensed set can be presented as a colimit of compact hausdorff spaces we may assume that $Z$ is a compact Hausdorff space. Then, by the hypothesis that $\im(Z(*) \rightarrow X(*)) \subset U(*)$ it follows that the basechange $U \times_X Z \rightarrow Z$ is an open subset of $Z$ where $U \times_X Z(*) = Z(*)$, proving that $Z \simeq U \times_X Z$ and thus providing the desired lift.

\textit{proof of (3):} Since monomorphisms and epimorphisms are stable under basechange in a topos, together with the characterization of $\im(f_1 \sqcup f_2)$ as the essentially unique object factoring $f_1 \sqcup f_2: U_1 \sqcup U_2 \twoheadrightarrow \im(f_1 \sqcup f_2) \hookrightarrow X$ shows that images of morphisms in a topos are stable under basechange as well. Therefore, we may assume that $X$ is a compact hausdorff space. By construction we have that $\im(f_1 \sqcup f_2)(*) \subset X(*)$ is an open subset of $X$, thus it remains to show that $\im(f_1 \sqcup f_2) \rightarrow X$ has the lifting property of Remark \ref{subspace_top_condensed}. Since $X$ is a compact hausdorff space there is a filtered collection of monomorphisms $\{f_{i,k}: Z_{i,k} \hookrightarrow X\}$ from compact Hausdorff spaces $Z_{i,k}$ such that $f_i: U_i = \colim Z_{i, k} \hookrightarrow K$. Then, for each pair $Z_{1, k_1}, Z_{2, k_2}$ we have a canonical factorization $Z_{1, k_1} \sqcup Z_{2, k_2} \twoheadrightarrow \im(f_{1, k_1} \sqcup f_{2, k_2}) \hookrightarrow X$, in particular we have that $\im(f_{1, k_1} \sqcup f_{2, k_2}) \hookrightarrow X$ is a monomorphism of compact Hausdorff spaces and so it satisfies the lifting property of Remark \ref{subspace_top_condensed}. Then, since filtered colimits preserve epimorphisms and monomorphisms we learn that $\im(f_1 \sqcup f_2) = \colim \im(f_{1, k_1} \sqcup f_{2, k_2}) \rightarrow X$, showing that $\im(f_1 \sqcup f_2) \rightarrow X$ has the desired lifting property, completing the proof. 

\textit{proof of (4):} Recall that since $g: Y \rightarrow X$ is a monomorphism, for any map $W \rightarrow Y$ of condensed sets, the basechange of $g: Y \rightarrow X$ along $W \rightarrow Y \rightarrow X$ induced an isomorphism $W \times_{X} Y \rightarrow Y$. Thus, for any morphism from a compact Hausdorff space $K \rightarrow Y$ the basechange of $g: Y \rightarrow X$ along $K \rightarrow Y \rightarrow X$ induces an isomorphism $K \times_Y X \rightarrow K$. The claim then follows directly from the definitions.

\textit{proof of (5):} Since open immersions are detected after basechange and filtered colimits commute with basechange we may assume that $X$ is a compact hausdorff space. Since unions of open subsets of a compact hausdorff space is an open subset, we learn that $(\colim U_i)(*) \rightarrow X(*)$ is an open subset of $X$, and since filtered colimts of monomorphisms are monomorphisms it remains to show that $\colim U_i \rightarrow X$ has the lifting property of Remark \ref{subspace_top_condensed}. By part (2) we know that open immersions have the lifting property, and since filtered colimit of sheaves (computed in the category of presheaves) are again sheaves, it follows that $\colim U_i \rightarrow X$ has the desired lifting property, proving the claim.

\textit{proof of (6):} Its clear from the definition that open immersions are closed under basechange. Let $f: Z \rightarrow Y$ and $g: Y \rightarrow X$ be a pair of open immersions, we need to show that $g \circ f$ is an open immersion, and given that open immersions are closed under basechange we may assume that $X$ is a compact Hausdorff space. Then, by hypothesis that $g: Y \rightarrow X$ is an open immersion we know that there exist an open subset $U_Y \rightarrow X$ of topological spaces such that $Y = \underline{U_Y} \rightarrow X$; and since open immersions have the lifting property of part (2) we can conclude that there exists an open subset $U_Z \rightarrow U_Y$ such that $Z = \underline{U_Z} \rightarrow \underline{U_Y} = Y$. Then, since open subsets $U_Y$ is also an open subset of $X$ the result follows.
\end{proof}

\begin{defn}\label{defn_interior_cond} Let $Y \hookrightarrow X$ be a monomorphisms of condensed sets. Then, we say that $x \in Y(*) \subset X(*)$ is in the interior of $Y \hookrightarrow X$ if there exists a morphism $U \rightarrow Y$ such that $x \in \im(U(*) \rightarrow Y(*))$ and the composition $U \rightarrow Y \rightarrow X$ is an open immersion. In particular, we learn from Proposition \ref{prop_open_immersion_condensed} that $U \rightarrow Y$ is an open immersion.

Furthermore, for a monomorphism $Y \hookrightarrow X$ we define the interior of the monomorphism $\Int(Y/X)$ as
\begin{align*}
	\Int(Y/X) := \colim_{U \rightarrow Y} U
\end{align*}
where the colimit ranges over all maps $U \rightarrow Y$ whose composition $U \rightarrow Y \rightarrow X$ is an open immersion. Again by Proposition \ref{prop_open_immersion_condensed} we learn that the canonically induced maps $\Int(Y/X) \rightarrow Y$ and $\Int(Y/X) \rightarrow X$ are open immersions.
\end{defn}

\begin{defn}\label{defn_locally_compact_cond} A condensed set $X$ is locally compact if for every $x \in X(*)$ there exists a compact Hausdorff space $K$ and a monomorphism $K \hookrightarrow X$ such that $x \in \Int(K/X)(*) \subset X(*)$.
\end{defn}

\begin{example} All compact Hausdorff spaces are locally compact condensed sets.
\end{example}

\begin{prop}\label{properties_loc_compact_cond} The collection of locally compact condensed sets have the following properties
\begin{enumerate}[(1)]
	\item If $X$ is a locally compact condensed set and $U \hookrightarrow X$ an open immersion. Then $U$ is a locally compact condensed set.
	\item If $X$ is a locally compact quasiseparated condensed set, and $\{V_i \hookrightarrow X\}_{i \in I}$ be a collection of monomorphisms such that the induced map $\sqcup_{i \in I} \Int(V_i/X)(*) \rightarrow X(*)$ is surjective. Then, $\sqcup_{i \in I} V_i \rightarrow X$ is an epimorphism of condensed sets.
\end{enumerate}
\end{prop}

\begin{proof} \textit{proof of (1):} We need to show that for every $x \in U(*) \subset X(*)$ there exists a monomorphism $K \rightarrow U$ from a compact Hausdorff space $K$ such that $x \in \Int(K/U)(*) \subset U(*)$. By the assumption that $X$ is locally compact we can find a monomorphism $E \rightarrow X$ from a compact Hausdorff space such that $x \in \Int(E/X)(*) \subset X(*)$. Consider the pullback diagram
	\begin{cd}
		U_{\Int(E/X)} \ar[r] \ar[d] & \Int(E/X) \ar[d] \\
		U_E \ar[r] \ar[d] & E \ar[d] \\
		U \ar[r] & X
	\end{cd}
Since open immersions are closed under basechange and composition, and the fact that $\Int(E/X) \hookrightarrow E$ is an open immersion, we know that $U_{\Int(E/X)} \hookrightarrow E$ is an open immersion and $x \in U_{\Int(E/X)}(*) \subset E(*) \subset X(*)$, thus by Urysohn's Lemma we can find a compact Hausdorff space $K$ together with a monomorphism $K \hookrightarrow U_{\Int(E/X)}$ such that $x \in K(*) \subset U_{\Int(E/X)}(*)$ and $x \in \Int(K/U_{\Int(E/K)})(*)$. Thus, we obtain maps
\begin{align}
	\Int(K/U_{\Int(E/K)}) \hookrightarrow K \hookrightarrow U_{\Int(E/X)} \hookrightarrow U
\end{align}
Then, since $\Int(K/U_{\Int(E/K)}) \rightarrow U$ is an open immersion, the claim that $U$ is locally compact follows.

\textit{proof of (2):} It suffices to show that $\sqcup_{i \in I} \Int(V_i/X) \rightarrow X$ is an epimorphism, thus we may assume that all $V_i \hookrightarrow X$ are open immersions. To show that $\sqcup_{i \in I} V_i \rightarrow X$ is an epimorphism, it suffices to show that for every morphism $E \rightarrow X$, from an extremally disconnected set $E$, there exists a lift $E \rightarrow \sqcup_{i \in I} V_i$ making the following diagram commute
\begin{cd}
	& E \ar[d] \ar[ld, dashed] \\
	\sqcup_{i \in I} V_i \ar[r] & X
\end{cd}
Recall that extremally disconnected sets are the projective objects in the category of compact Hausdorff spaces. By the compactness of $E$ and the fact that open immersions are closed under basechange, we know that there exists a finite collection $I_0 \subset I$ such that $\im(E(*) \rightarrow X(*)) \subset \cup_{i \in I_0} V_i(*)$. Then, since each $V_i$ are locally compact condensed sets by part (1), for each $x \in \im(E(*) \rightarrow X(*))$ there exists a compact Hausdorff space $K_x$ together with a monomorphism $K_x \hookrightarrow V_i$ such that $x \in \Int(K_x/V_i)(*)$, and again by the compactness of $E$ there are finitely many $\{K_x \hookrightarrow X\}_{x \in J}$ such that $\im(E(*) \rightarrow X(*)) \subset \im(\sqcup_{x \in J} K_x(*) \rightarrow X(*))$. Set $K = \im(\sqcup_{x \in J} K_x \rightarrow X)$, since $X$ is quasiseparated we know that $K$ is a compact Hausdorff space satisfying $\im(E(*) \rightarrow X(*)) \subset K(*)$, thus by Proposition \ref{mono_condensed_sets} we know that the map $E \rightarrow X$ admits an essentially unique lift $E \rightarrow K$. Hence, by the projectivity of extremally disconnected sets the map $E \rightarrow K$ admits a lift to $E \rightarrow \sqcup_{x \in J} K_x$, and so a lift to $E \rightarrow \sqcup_{i \in I} V_i$ as desired.
\end{proof}

\subsection{Open analytic domains}

\begin{defn}\label{defn_open_immersion_arc} Let $f: U \hookrightarrow X$ be a monomorphism in $\cX_{\arc_\varpi}$, then $f: U \hookrightarrow X$ is an open immersion if the induced map $|U| \hookrightarrow |X|$ is an open immersion of condensed sets (cf. Definition \ref{defn_open_immersion_cond}).
\end{defn}

\begin{prop}\label{prop_open_immersion_arc} Open immersions of $\arc_\varpi$-sheaves have the following properties
\begin{enumerate}[(1)]
	\item If $f_1: U_1 \rightarrow X$ and $f_2: U_2 \rightarrow X$ are open immersions, then $\im(f_1 \sqcup f_2) \rightarrow X$ is an open immersion of $\arc_\varpi$-sheaves. In particular, the collection of open immersions on $X$ is filtered, that is for any pair of open immersions $f_1: U_1 \rightarrow X$ and $f_2: U_2 \rightarrow X$ there exists an open immersion $g: V \rightarrow X$ such that $f_1$ and $f_2$ factor over $g$.
	\item Let $f: Z \rightarrow Y$ and $g: Y \rightarrow X$ be a pair of monomorphisms of $\arc_\varpi$-sheaves. If $g \circ f$ is an open immersion then so is $f$.
	\item Let $\{f_i: U_i \rightarrow X\}$ be a filtered collection of open immersions on $X$, that is for any pair of open immersions $f_i: U_i \rightarrow X$ and $f_j: U_j \rightarrow X$ there exists an $f_k: U_k \rightarrow X$ such that $f_i$ and $f_j$ factor over $f_k$. Then, $\colim U_i \rightarrow X$ is an open immersion.
	\item Open immersions are closed under composition and base-change.
\end{enumerate}
\end{prop}

\begin{proof} \textit{proof of (1):} By Proposition \ref{berko_funct_stability} we know that the Berkovich functor preserves epimorphisms and monomorphism, and so it also preserves images, thus the result follow from Proposition \ref{prop_open_immersion_condensed}(3).

\textit{proof of (2):} Follows directly from \ref{prop_open_immersion_condensed}(4), and the fact that the Berkovich functor preserves monomorphisms.

\textit{proof of (3):} Follows from the fact that the Berkovich functor preserves colimits and \ref{prop_open_immersion_condensed}(5)

\textit{proof of (4):} Stability of open immersions under basechange follow from Proposition \ref{berko_funct_stability}(4) and \ref{prop_open_immersion_condensed}(6), stability under composition also follows from tha latter proposition.
\end{proof}

\begin{prop}[Open Subobjects]\label{open_subobject_arc} Let $X$ be a quasiseparated $\arc_\varpi$-sheaf, then
\begin{enumerate}[(1)]
	\item For every open immersion of condensed sets $\tilde{U} \hookrightarrow |X|$, there exists an open immersion $U \hookrightarrow X$ of $\arc_\varpi$-sheaves, such that $\tilde{U} = |U| \hookrightarrow |X|$. Furthermore, $U \hookrightarrow X$ can be realized as $\colim_{Y \hookrightarrow X} Y \rightarrow X$, where the colimit ranges over all monomorphisms $Y \hookrightarrow X$ from a qcqs $\arc_\varpi$-sheaf such that $|Y|(*) \subset \tilde{U}(*) \subset |X|(*)$.
	\item If $U \rightarrow X$ is an open immersion of $\arc_\varpi$-sheaves then for any morphism $Z \rightarrow X$ such that $\im(|Z|(*) \rightarrow |X|(*)) \subset |U|(*)$, there exists a unique morphism $Z \rightarrow U$ making the following diagram commute
	\begin{cd}
	& Z \ar[d] \ar[ld, dashed] \\
	U \ar[r, hook] & X
	\end{cd}
	In particular, this shows that $U \hookrightarrow X$ can be realized as $\colim_{Y \rightarrow X} Y \rightarrow X$ where the colimit ranges over all morphisms $Y \rightarrow X$ from a qcqs $\arc_\varpi$-sheaf $Y$ satisfying $\im(|Y|(*) \rightarrow |X|(*)) \subset |U|(*)$.
	\item For an object $X$ in either category $\cX_{\arc_\varpi}$ or $\Cond$, define $\Open(X)$ as the collection of open immersions $U \hookrightarrow X$, where the morphisms are maps $U_1 \hookrightarrow U_2$ of $\arc_\varpi$-sheaves making the following diagram commute
	\begin{cd}
		U_1 \ar[rr, hook] \ar[rd, hook] && U_2 \ar[ld, hook] \\
		& X
	\end{cd}
	Then, if $X$ is an $\arc_\varpi$-sheaf the Berkovich functor induces an equivalence of categories
	\begin{align*}
		\Open(X) \overset{\simeq}{\longrightarrow} \Open(|X|)
	\end{align*}
\end{enumerate}
\end{prop}

\begin{proof}

\textit{proof of (1):} Recall from Proposition \ref{defn_open_immersion_cond}(2) that an open immersion of condensed sets $\tilde{U} \rightarrow X$ can be realized as $\colim_{Q \rightarrow |X|} Q \rightarrow |X|$ where the colimit ranges over all morphisms $Q \rightarrow |X|$ from a compact Hausdorff space $Q$ such that $\im(Q(*) \rightarrow |X|(*)) \subset \tilde{U}(*)$. Furthermore, since $|X|$ is quasiseparated $\tilde{U}$ can be realized as $\colim_{Q \hookrightarrow |X|} Q \rightarrow |X|$ where the colimit ranges over all monomorphisms $Q \hookrightarrow |X|$. By Proposition \ref{subobjects_arc_topos} we know that for each $Q \hookrightarrow |X|$ there exists a unique monomorphism $Y \hookrightarrow X$ such that $Q = |Y| \rightarrow X$, and define $U = \colim_{Y \hookrightarrow X} Y \rightarrow X$, where the colimit ranges over all monomorphisms $Y \hookrightarrow X$ from a qcqs $\arc_\varpi$-sheaf $Y$, such that $|Y|(*) \subset \tilde{U}(*)$. Since the Berkovich functor preserves colimit it follows that $\tilde{U} = |U| \hookrightarrow |X|$ and so $U \hookrightarrow X$ is an open immersion as desired.

\textit{proof of (2):} Set $\im(Z \rightarrow X) =: V \hookrightarrow X$, then by the description of $U \hookrightarrow X$ from part (1), we know that there exists an essentially unique lift $V \rightarrow U$, and so $Z \rightarrow X$ also admits a lift $Z \rightarrow U$.

\textit{proof of (3):} This is a combination of part (1) and (2), and Proposition \ref{prop_open_immersion_condensed}(2).
\end{proof}

\begin{prop}\label{open_of_spc_is_spc} Let $X$ be a $\arc_\varpi$-analytic space and $U \hookrightarrow X$ an open immersion, then $U$ is an $\arc_\varpi$-analytic space. Similarly, if $X$ is a perfectoid space, then $U$ is a perfectoid space.
\end{prop}

\begin{proof} We claim that for any Banach $K$-algebra $A$ the compact Hausdorff $|\cM(A)|$ space admits a basis of neighborhoods which are of the form
\begin{align*}
	|\cM(A)|\Big (\frac{f_1, \dots, f_n}{r} \Big) &&  |\cM(A)|\Big (\frac{r}{f_1, \dots, f_n} \Big) && (\text{cf. Definition \ref{defn_top_rat_domain}})
\end{align*}
where $r \in |K^*|$ and $\{f_1, \dots, f_n\} \subset A$. Indeed, from the definition of the Berkovich spectrum of $A$, that the compact Hausdorff space $|\cM(A)|$ admits an injective map
\begin{equation*}
	\cM(A) \longrightarrow \prod_{f \in A} [0, |f|]
\end{equation*}
where $\prod_{f \in A} [0, |f|]$ is endowed with the product topology, making it the claim above clear. Furthermore, for both subsets $|\cM(A)|\Big (\frac{f_1, \dots, f_n}{r} \Big)$ and $ |\cM(A)|\Big (\frac{r}{f_1, \dots, f_n} \Big)$ there exists Banach $K$-algebras $B_1$ and $B_2$ respectively, together with contractive morphisms $A \rightarrow B_1$ and $A \rightarrow B_2$ such that the induced maps $\cM(B_1) \hookrightarrow \cM(A)$ and $\cM(B_2) \hookrightarrow \cM(A)$ are monomorphisms and have as image $|\cM(A)|\Big (\frac{f_1, \dots, f_n}{r} \Big)$ and $ |\cM(A)|\Big (\frac{r}{f_1, \dots, f_n} \Big)$ respectively (cf. Proposition \ref{topo_img_rat_domain}). Moreover, if $A$ is a perfectoid Banach $K$-algebra we may choose $B_1$ and $B_2$ to be perfectoid Banach $K$-algebras (cf. Theorem \ref{struct_presheaf_perfectoid}).

Let us first proof the result when $X = \cM(A)_{\arc_\varpi}$, set $V = \im(\sqcup_{i \in I} V_i \rightarrow X)$ where the maps $V_i \rightarrow X$ are rational domains of the above form satisfying $|V_i|(*) \subset |U|(*) \subset |X|(*)$, then by the work done in the previous paragraph we know that $|V|(*) = |U|(*) \subset |X|(*)$ and by construction $V$ is an $\arc_\varpi$-analytic space (resp. a perfectoid space if $A$ is perfectoid). We claim that $V = U$ and by Proposition \ref{open_subobject_arc}(3) it suffices to show that $|V| = |U|$. Indeed, we need to show that for any morphism $Q \rightarrow |X|$ from a compact Hausdorff space $Q$ satisfying $\im(Q(*) \rightarrow |X|(*)) \subset |V|(*)$ there exists a unique lift $Q \rightarrow |V|$. Since $Q$ is compact and rational domains form a basis of neighborhood of $|X|$ we can find a finite subset $I_0 \subset I$ such that $Q(*) \subset \im(\sqcup_{i \in I_0} |V_i|(*) \rightarrow |X|(*))$, setting $W = \im(\sqcup_{i \in I_0} V_i \rightarrow X)$ we can conclude by Proposition \ref{mono_condensed_sets} that there exists a unique lift of $Q \rightarrow |X|$ to $Q \rightarrow |W|$ and so also a lift to $Q \rightarrow |V|$, completing the proof.

For the general case, we have an open immersion $U \hookrightarrow X$, and pick a collection of monomorphisms $\{\cM(A_i)_{\arc_\varpi} \hookrightarrow X\}_{i \in I}$ such that the induced map $\sqcup_{i \in I} \cM(A_i) \rightarrow X$ is an epimorphism. Then, the basechange of $U \times_X \cM(A_i)_{\arc_\varpi} =: U_i \hookrightarrow \cM(A)_{\arc_\varpi}$ is an open immersion, and so an $\arc_\varpi$-analytic space (resp. a perfectoid space), and so it admits a collection of monomorphisms $\{\cM(B_j)_{\arc_\varpi} \hookrightarrow U_i\}_{j \in J_i}$ such that the induced map $\sqcup_{j \in J_i} \cM(B_i)_{\arc_\varpi} \rightarrow U_i$ is an epimorphism. Then, the collection of maps $\{\cM(B_j)_{\arc_\varpi} \hookrightarrow U\}_{j \in J}$, where $J = \cup_{i \in I} J_i$, are monomorphisms and induces an epimorphism after taking coproduct, showing that $U$ is an $\arc_\varpi$-analytic space (resp. a perfectoid space).
\end{proof}

\begin{example}\label{non_vanishing_subsets} For a Banach $K$-algebra, the subset
\begin{align*}
	|\cM(A)|_{f \not= 0} :=  \{x \in |\cM(A)| \text{ such that } |f|(x) \not= 0\} \subset |\cM(A)|
\end{align*}
is an open subset of $|\cM(A)|$, so by Propositions \ref{open_subobject_arc} and \ref{open_of_spc_is_spc} we know that there exists an $\arc_\varpi$-analytic space $U$ together with a monomorphism $U \hookrightarrow \cM(A)_{\arc_\varpi}$ such that $U = |\cM(A)|_{f \not= 0} \rightarrow |\cM(A)_{\arc_\varpi}| = |\cM(A)|$. 
\end{example}

\begin{defn} Let $Y \hookrightarrow X$ be a monomorphism of $\arc_\varpi$-sheaves. We say that $x \in |Y|(*) \subset |X|(*)$ is in the interior of $Y \hookrightarrow X$ if there exists a morphism $U \rightarrow Y$ such that $x \in \im(|U|(*) \rightarrow |Y|(*))$ and the composition $U \rightarrow Y \rightarrow X$ is an open immersion. In particular, we learn from Proposition \ref{prop_open_immersion_arc}(2) that $U \hookrightarrow Y$ is an open immersion.

Furthermore, for a monomorphism $Y \hookrightarrow X$ we define the interior of the monomorphism $\Int(Y/X)$ as
\begin{align*}
	\Int(Y/X) := \colim_{U \rightarrow Y} U
\end{align*}
where the colimit ranges over all maps $U \rightarrow Y$ whose composition $U \rightarrow Y \rightarrow X$ is an open immersion. Again, by Proposition \ref{prop_open_immersion_arc} we learn that the canonically induced maps $\Int(Y/X) \rightarrow Y$ and $\Int(Y/X) \rightarrow X$ are open immersions.
\end{defn}

\begin{lemma}\label{iso_interior_arc} Let $X$ be a quasiseparated $\arc_\varpi$-sheaf and $Y \hookrightarrow X$ a monomorphism of $\arc_\varpi$-sheaves. Then, the canonical map
\begin{align*}
	|\Int(Y/X)| \overset{\simeq}{\longrightarrow} \Int(|Y|/|X|)
\end{align*}
is an isomorphism of condensed sets.
\end{lemma}

\begin{proof} By Proposition \ref{prop_open_immersion_condensed}(2) it suffices to show that $|\Int(Y/X)|(*) = \Int(|Y|/|X|)(*) \subset |X|(*)$, and its clear by the construction that $|\Int(Y/X)|(*) \subset \Int(|Y|/|X|)(*)$, thus it remains to show the reverse inclusion. By Proposition \ref{open_subobject_arc}(3) and the fact that $Y$ and $X$ are quasiseparated we know that there exists a mononomorphism $U \hookrightarrow Y$ whose composition $U \hookrightarrow Y \hookrightarrow X$ is an open immersion, which satisfies $\Int(|Y|/|X|) = |U| \rightarrow |Y| \rightarrow |X|$, showing that $|\Int(Y/X)|(*) \supset \Int(|Y|/|X|)(*)$ and completing the proof.
\end{proof}

Following \cite[Remark 1.2.16]{berkovichetale} we make the following definition, though we prefer the more descriptive name ``locally compact'' over ``good''.

\begin{defn}[Locally Compact]\label{defn_locally_compact_arc} A quasiseparated $\arc_\varpi$-sheaf $X$ is locally compact if for every point $x \in |X|(*)$ there exists a monomorphism $V_x \hookrightarrow X$ from a qcqs $\arc_\varpi$-sheaf such that $x \in |\Int(V_x/X)|(*)$. In particular, we say that a quasiseparated $\arc_\varpi$-analytic space (resp. a perfectoid space) $X$ is locally compact if it is locally compact as an $\arc_\varpi$-sheaf.

Furthermore, by Lemma \ref{iso_interior_arc} we know that $|\Int(V_x/X)| = \Int(|V_x|/|X|)$ and so if $X$ is a quasiseparated locally compact $\arc_\varpi$-sheaf, then $|X|$ is a quasiseparated locally compact condensed set (cf. Definition \ref{defn_locally_compact_cond}).
\end{defn}

\begin{example} If $X$ is a qcqs $\arc_\varpi$-sheaf then it is locally compact.
\end{example}

\begin{prop}\label{properties_loc_compact_arc} The collection of quasiseparated locally compact $\arc_\varpi$-sheaves have the following properties
\begin{enumerate}[(1)]
	\item Let $X$ be a quasiseparated $\arc_\varpi$-sheaf, then $X$ is locally compact if and only if $|X|$ is locally compact.
	\item If $X$ is quasiseparated locally compact $\arc_\varpi$-sheaf and $U \hookrightarrow X$ is an open immersion of $\arc_\varpi$-sheaves, then $U$ is a quasiseparated locally compact $\arc_\varpi$-sheaf.
	\item If $X$ is a quasiseparated locally compact $\arc_\varpi$-sheaf, and $\{V_i \hookrightarrow X\}_{i \in I}$ a collection of monomorphisms such that the induced map $\sqcup_{i \in I} |\Int(V_i/X)|(*) \rightarrow |X|(*)$ is surjective, then $\sqcup_{i \in I} V_i \rightarrow X$ is an epimorphism of $\arc_\varpi$-sheaves. 
\end{enumerate}
\end{prop}

\begin{proof} \textit{proof of (1):} Recall that by Proposition \ref{berko_funct_stability}(6) is $X$ is a quasiseparated $\arc_\varpi$-sheaf then $|X|$ is quasiseparated condensed set. The claim then follows from a combination of Propositions \ref{subobjects_arc_topos} and \ref{open_subobject_arc}(3).

\textit{proof of (2):} This is a direct consequence of part (1) and Proposition \ref{properties_loc_compact_cond}(1).

\textit{proof of (3):} It suffices to show that $\sqcup_{i \in I} \Int(V_i/X) \rightarrow X$ is an epimorphism, thus we may assume that every $V_i \hookrightarrow X$ is an open immersion. To show that $\sqcup_{i \in I} V_i \rightarrow X$ is an epimorphism, we need to show that for every perfectoid Banach $K$-algebra $A$ and any morphism $\cM(A) \rightarrow X$, there exists an $\arc_\varpi$-cover $\sqcup_{h \in H} \cM(A_j) \rightarrow \cM(A)$, where each $A_h$ is perfectoid and $H$ is a finite set, and a map $\sqcup_{h \in H} \cM(A_h) \rightarrow \sqcup_{i \in I} V_i$ making the following diagram commute
\begin{cd}
	\sqcup_{h \in H} \cM(A_h) \ar[r] \ar[d] & \cM(A) \ar[d] \\
	\sqcup_{i \in I} V_i \ar[r] & X
\end{cd}
From the hypothesis we know that $\sqcup_{i \in I} |V_i|(*) \rightarrow |X|(*)$ is surjective and part (2) we know that each $V_i$ is itself a quasiseparated locally compact $\arc_\varpi$-sheaf, so for each $x \in \im(|\cM(A)|(*) \rightarrow |X|(*))$ we can find a monomorphism $Z_x \hookrightarrow V_i \hookrightarrow X$ from a qcqs $\arc_\varpi$-sheaf $Z_x$, such that $x$ is in the interior of $Z_x \hookrightarrow X$, and so by the compactness of $|\cM(A)|$ (and Proposition \ref{berko_funct_stability}(4)) we can find a finite collection $J$ such that $ \im(|\cM(A)|(*) \rightarrow |X|(*)) \subset \im(\sqcup_{x \in J} |Z_x|(*) \rightarrow |X|(*))$. We claim that the fiber product $\sqcup_{x \in J} Z_x \times_X \cM(A) \rightarrow \cM(A)$ is a epimorphism of qcqs $\arc_\varpi$-sheaves. Indeed, for each $x \in J$ we know that the fiber product $\cM(A) \times_X Z_x$ is quasicompact by virtue of the fact that $X$ is quasiseparated, and its quasiseparated since the canonical map $\cM(A) \times_X Z_x \hookrightarrow \cM(A)$ is a monomorphism, and so we can conclude that $\sqcup_{x \in J} Z_x \times_X \cM(A)$ is a qcqs object; which by construction and Proposition \ref{berko_funct_stability}(4), we know it induces a surjective map $\sqcup_{x \in J} |Z_x \times_X \cM(A)|(*) \rightarrow |\cM(A)|(*)$, which in turn implies that $\sqcup_{x \in J} Z_x \times_X \cM(A) \rightarrow \cM(A)$ is an epimorphism by Proposition \ref{epi_arc_topos}. To summarize, we have constructed a commutative diagram of the following form
\begin{cd}
	\sqcup_{x \in J} Z_x \times_X \cM(A) \ar[r] \ar[d] & \cM(A) \ar[d] \\
	\sqcup_{x \in J} Z_x \ar[r] & X
\end{cd}
where the top arrow is an epimorphism. Then, since $\sqcup_{x \in J} Z_x \times_X \cM(A)$ is quasicompact we can find a finite collection of morphisms $\{\cM(A_h) \rightarrow \sqcup_{x \in J} Z_x \times_X \cM(A) \}_{h \in H}$ such that the induced map $\sqcup_{h \in H} \cM(A_h) \rightarrow \sqcup_{x \in J} Z_x \times_X \cM(A)$ is an epimorphism, which in turn implies that $\sqcup_{h \in H} \cM(A_h) \rightarrow \cM(A)$ is an epimorphism, proving the claim.	
\end{proof}

\begin{example} Let $\{f_1, \dots, f_n\}_{i \in I} \subset A$ be a collection of objects of $A$ generating the unit ideal, and let $|\cM(A)|_{f_i \not= 0} \subset |\cM(A)|$ be the open subset of $|\cM(A)|$ introduced in Example \ref{non_vanishing_subsets}, and let $U_{f_i} \hookrightarrow \cM(A)_{\arc_\varpi}$ the monomorphism of $\arc_\varpi$-sheaves associated to the open subset (cf. Proposition \ref{open_subobject_arc}(3)). Then, the induced map $\sqcup_{i \in I} |U_{f_i}|(*) \rightarrow |\cM(A)_{\arc_\varpi}|(*)$ is surjective, and so by Proposition \ref{properties_loc_compact_arc} we learn that $\sqcup U_{f_i} \rightarrow |\cM(A)|_{\arc_\varpi}$ is an epimorphism, which in turn implies that we have the following identity
\begin{align*}
	\coeq \Big ( \sqcup_{(i,j) \in I^2} U_{f_i} \times_{\cM(A)_{\arc_\varpi}} U_{f_j} \rightrightarrows \sqcup_{i \in I} U_{f_i}  \Big) \overset{\simeq}{\longrightarrow} \cM(A)_{\arc_\varpi}
\end{align*}
Finally, recall that by Proposition \ref{berko_funct_stability}(4) we have that we have the identity
\begin{align*}
	|U_{f_i} \times_{\cM(A)_{\arc_\varpi}} U_{f_j}| = |\cM(A)|_{f_i \not= 0} \cap |\cM(A)|_{f_j \not= 0}
\end{align*}
Showing that $\arc_\varpi$-analytic spaces can be covered by ``Zariski open covers''.
\end{example}

\newpage

\phantomsection
\addcontentsline{toc}{chapter}{Bibliography}
\bibliographystyle{alpha}
%\bibliography{Ref}

\end{document}